\documentclass[12pt,a4paper]{amsart}
\usepackage{amssymb}
\usepackage[all,cmtip]{xy}
\usepackage{mathrsfs}

\calclayout

\newcommand{\+}{\protect\nobreakdash-}
\renewcommand{\:}{\colon}

\newcommand{\rarrow}{\longrightarrow}
\newcommand{\larrow}{\longleftarrow}
\newcommand{\ot}{\otimes}
\newcommand{\os}{\mathbin{\lozenge}}
\newcommand{\ocn}{\odot}
\newcommand{\Ocn}{\circledcirc}

\newcommand{\lrarrow}{\mskip.5\thinmuskip\relbar\joinrel\relbar\joinrel
 \rightarrow\mskip.5\thinmuskip\relax}
\newcommand{\llarrow}{\mskip.5\thinmuskip\leftarrow\joinrel\relbar
 \joinrel\relbar\mskip.5\thinmuskip\relax}

\newcommand{\ct}{\mathbin{\text{\textcircled
 {\smaller\smaller\smaller\smaller\smaller T}}}}
\newcommand{\bt}{\mathbin{\text{\smaller$\boxtimes$}}}
\newcommand{\subbt}{\mathbin{\text{\smaller$\scriptstyle\boxtimes$}}}

\newcommand{\oc}{\mathbin{\text{\smaller$\square$}}}
\newcommand{\bu}{{\text{\smaller\smaller$\scriptstyle\bullet$}}}
\newcommand{\subbu}{{\text{\smaller\smaller\smaller
     $\scriptstyle\bullet$}}}
\newcommand{\tim}{\rightthreetimes}

\newcommand{\ilim}
 {\mathop{\text{\normalfont``$\varinjlim$''\!\!}}\nolimits}

\DeclareMathOperator{\Mor}{Mor}
\DeclareMathOperator{\Hom}{Hom}
\DeclareMathOperator{\Ext}{Ext}
\DeclareMathOperator{\Sym}{Sym}
\DeclareMathOperator{\Spec}{Spec}
\DeclareMathOperator{\Mat}{Mat}
\DeclareMathOperator{\GL}{GL}
\DeclareMathOperator{\Spi}{Spi}
\DeclareMathOperator{\Proj}{Proj}
\DeclareMathOperator{\Tor}{Tor}
\DeclareMathOperator{\SemiTor}{SemiTor}
\DeclareMathOperator{\coker}{coker}

\DeclareMathOperator{\cHom}{\mathcal H\mathit{om}}
\DeclareMathOperator{\fHom}{\mathfrak{Hom}}
\DeclareMathOperator{\cSym}{\mathcal S\mskip-.5\thinmuskip\mathit{ym}}

\newcommand{\Id}{\mathrm{Id}}
\newcommand{\id}{\mathrm{id}}

\newcommand{\Sets}{\mathsf{Sets}}
\newcommand{\Sch}{\mathsf{Sch}}
\newcommand{\CSch}{\mathsf{CSch}}
\newcommand{\AffSch}{\mathsf{AffSch}}
\newcommand{\CRings}{\mathsf{CRings}}
\newcommand{\Ind}{\mathsf{Ind}}

\newcommand{\sop}{{\mathsf{op}}}
\newcommand{\co}{{\mathsf{co}}}

\newcommand{\si}{{\mathsf{si}}}
\newcommand{\abs}{{\mathsf{abs}}}
\newcommand{\bb}{{\mathsf{b}}}
\newcommand{\inj}{{\mathsf{inj}}}
\newcommand{\proj}{{\mathsf{proj}}}
\newcommand{\bco}{{\mathsf{bco}}}

\newcommand{\tors}{{\operatorname{\mathsf{--tors}}}}
\newcommand{\qcoh}{{\operatorname{\mathsf{--qcoh}}}}
\newcommand{\tcoh}{{\operatorname{\mathsf{-tcoh}}}}
\newcommand{\syst}{{\operatorname{\mathsf{--syst}}}}
\newcommand{\pro}{{\operatorname{\mathsf{--pro}}}}
\renewcommand{\flat}{{\operatorname{\mathsf{--flat}}}}
\newcommand{\discr}{{\operatorname{\mathsf{--discr}}}}
\newcommand{\modl}{{\operatorname{\mathsf{--mod}}}}
\newcommand{\vect}{{\operatorname{\mathsf{--vect}}}}
\newcommand{\comodl}{{\operatorname{\mathsf{--comod}}}}
\newcommand{\comodr}{{\operatorname{\mathsf{comod--}}}}
\newcommand{\simodl}{{\operatorname{\mathsf{--simod}}}}
\newcommand{\simodr}{{\operatorname{\mathsf{simod--}}}}
\newcommand{\comodrinj}{{\operatorname{\mathsf{comod_{inj}--}}}}
\newcommand{\contra}{{\operatorname{\mathsf{--contra}}}}
\newcommand{\separ}{{\operatorname{\mathsf{--separ}}}}
\newcommand{\sicntr}{{\operatorname{\mathsf{--sicntr}}}}

\newcommand{\qc}{{\operatorname{-qc}}}
\newcommand{\dinj}{{\operatorname{\mathsf{-inj}}}}
\newcommand{\dproj}{{\operatorname{\mathsf{-proj}}}}

\newcommand{\fA}{{\mathfrak A}}
\newcommand{\fB}{{\mathfrak B}}
\newcommand{\fC}{{\mathfrak C}}
\newcommand{\fF}{{\mathfrak F}}
\newcommand{\fG}{{\mathfrak G}}
\newcommand{\fH}{{\mathfrak H}}
\newcommand{\fI}{{\mathfrak I}}
\newcommand{\fJ}{{\mathfrak J}}
\newcommand{\fM}{{\mathfrak M}}
\newcommand{\fN}{{\mathfrak N}}
\newcommand{\fO}{{\mathfrak O}}
\newcommand{\fP}{{\mathfrak P}}
\newcommand{\fQ}{{\mathfrak Q}}
\newcommand{\fR}{{\mathfrak R}}
\newcommand{\fS}{{\mathfrak S}}
\newcommand{\fT}{{\mathfrak T}}

\newcommand{\U}{{\mathfrak U}}
\newcommand{\W}{{\mathfrak W}}
\newcommand{\X}{{\mathfrak X}}
\newcommand{\Y}{{\mathfrak Y}}
\newcommand{\Z}{{\mathfrak Z}}

\newcommand{\p}{{\mathfrak p}}

\newcommand{\cA}{{\mathcal A}}
\newcommand{\cB}{{\mathcal B}}
\newcommand{\D}{{\mathcal D}}
\newcommand{\C}{{\mathcal C}}
\newcommand{\E}{{\mathcal E}}
\newcommand{\F}{{\mathcal F}}
\newcommand{\G}{{\mathcal G}}
\newcommand{\cH}{{\mathcal H}}
\newcommand{\I}{{\mathcal I}}
\newcommand{\J}{{\mathcal J}}
\newcommand{\K}{{\mathcal K}}
\newcommand{\cL}{{\mathcal L}}
\newcommand{\M}{{\mathcal M}}
\newcommand{\N}{{\mathcal N}}
\newcommand{\cO}{{\mathcal O}}
\newcommand{\cP}{{\mathcal P}}
\newcommand{\cQ}{{\mathcal Q}}

\newcommand{\cS}{{\mathcal S}}
\newcommand{\cT}{{\mathcal T}}

\newcommand{\rC}{{\mathscr C}}
\newcommand{\rD}{{\mathscr D}}
\newcommand{\rE}{{\mathscr E}}

\newcommand{\rJ}{{\mathscr J}}
\newcommand{\rK}{{\mathscr K}}
\newcommand{\rL}{{\mathscr L}}
\newcommand{\rM}{{\mathscr M}}
\newcommand{\rN}{{\mathscr N}}

\newcommand{\bPhi}{\boldsymbol{\Phi}}

\newcommand{\bY}{{\boldsymbol{\mathfrak Y}}} 
\newcommand{\bW}{{\boldsymbol{\mathfrak W}}}
\newcommand{\bV}{{\boldsymbol{\mathfrak V}}}

\newcommand{\bfA}{{\boldsymbol{\mathfrak A}}}
\newcommand{\bfF}{{\boldsymbol{\mathfrak F}}}
\newcommand{\bfG}{{\boldsymbol{\mathfrak G}}}
\newcommand{\bfH}{{\boldsymbol{\mathfrak H}}}
\newcommand{\bfP}{{\boldsymbol{\mathfrak P}}}
\newcommand{\bfQ}{{\boldsymbol{\mathfrak Q}}}
\newcommand{\bfR}{{\boldsymbol{\mathfrak R}}}
\newcommand{\bfS}{{\boldsymbol{\mathfrak S}}}

\newcommand{\brF}{{\pmb{\mathscr F}}}
\newcommand{\brG}{{\pmb{\mathscr G}}}

\newcommand{\brE}{{\pmb{\mathscr E}}}

\newcommand{\bJ}{{\pmb{\mathscr J}}}
\newcommand{\bK}{{\pmb{\mathscr K}}}
\newcommand{\bL}{{\pmb{\mathscr L}}}
\newcommand{\bM}{{\pmb{\mathscr M}}}
\newcommand{\bN}{{\pmb{\mathscr N}}} 
\newcommand{\bS}{{\pmb{\mathscr S}}}

\newcommand{\bnG}{{\mathbf G}}
\newcommand{\bnL}{{\mathbf L}}
\newcommand{\bnM}{{\mathbf M}}
\newcommand{\bnN}{{\mathbf N}}
\newcommand{\bnS}{{\mathbf S}}
\newcommand{\bnT}{{\mathbf T}}
\newcommand{\bnY}{{\mathbf Y}}
\newcommand{\bnW}{{\mathbf W}}
\newcommand{\bnV}{{\mathbf V}}
\newcommand{\bnU}{{\mathbf U}}

\newcommand{\bcF}{{\boldsymbol{\mathcal F}}}
\newcommand{\bcG}{{\boldsymbol{\mathcal G}}}
\newcommand{\bcH}{{\boldsymbol{\mathcal H}}}
\newcommand{\bcJ}{{\boldsymbol{\mathcal J}}}
\newcommand{\bcK}{{\boldsymbol{\mathcal K}}}

\newcommand{\bcN}{{\boldsymbol{\mathcal N}}}
\newcommand{\bcQ}{{\boldsymbol{\mathcal Q}}}

\newcommand{\sA}{{\mathsf A}}
\newcommand{\sB}{{\mathsf B}}
\newcommand{\sC}{{\mathsf C}}
\newcommand{\sD}{{\mathsf D}}
\newcommand{\sE}{{\mathsf E}}
\newcommand{\sF}{{\mathsf F}}
\newcommand{\sK}{{\mathsf K}}
\newcommand{\sS}{{\mathsf S}}

\newcommand{\boD}{{\mathbb D}}
\newcommand{\boL}{{\mathbb L}}
\newcommand{\boM}{{\mathbb M}}
\newcommand{\boN}{{\mathbb N}}
\newcommand{\boZ}{{\mathbb Z}}
\newcommand{\boQ}{{\mathbb Q}}
\newcommand{\boR}{{\mathbb R}}
\newcommand{\kk}{\Bbbk}

\theoremstyle{plain}
\newtheorem{thm}{Theorem}[section]
\newtheorem{prop}[thm]{Proposition}
\newtheorem{lem}[thm]{Lemma}
\newtheorem{cor}[thm]{Corollary}
\theoremstyle{definition}
\newtheorem{qst}[thm]{Question}

\newtheorem{rem}[thm]{Remark}
\newtheorem{rems}[thm]{Remarks}
\newtheorem{ex}[thm]{Example}
\newtheorem{exs}[thm]{Examples}

\newcommand{\Section}[1]{\bigskip\section{#1}\medskip}
\begin{document}

\title{Quasi-coherent torsion sheaves, the semiderived category,
and the semitensor product
\\[15pt] \smaller\smaller\smaller\normalfont\itshape
Semi-infinite Algebraic Geometry of Quasi-Coherent Sheaves \\
on Ind-Schemes}

\author{Leonid Positselski}

\address{Institute of Mathematics of the Czech Academy of Sciences \\
\v Zitn\'a~25, 115~67 Praha~1 (Czech Republic); and
\newline\indent Laboratory of Algebra and Number Theory \\
Institute for Information Transmission Problems \\
Moscow 127051 (Russia)}

\email{positselski@math.cas.cz}

\begin{abstract}
 We construct the semi-infinite tensor structure on the semiderived
category of quasi-coherent torsion sheaves on an ind-scheme endowed
with a flat affine morphism into an ind-Noetherian ind-scheme with
a dualizing complex.
 The semitensor product is ``a mixture of'' the cotensor product
along the base and the derived tensor product along the fibers.
 The inverse image of the dualizing complex is the unit object.
 This construction is a partial realization of the Semi-infinite
Algebraic Geometry program, as outlined in the introduction
to~\cite{Pfp}.
\end{abstract}

\maketitle

\tableofcontents

\section*{Introduction}
\medskip

\setcounter{subsection}{-1}
\subsection{{}}
 The aim of this paper is to extend the Semi-infinite Homological
Algebra, as developed in the author's monograph~\cite{Psemi}, to
the realm of Algebraic Geometry.
 According to the philosophy elaborated in the preface to~\cite{Psemi},
semi-infinite homological theories are assigned to ``semi-infinite
algebrac and geometric objects''.
 A detailed explanation of what should be understood by
a ``semi-infinite algebraic variety'' was suggested in the introduction
to the author's paper~\cite{Pfp} (see also
the presentation~\cite{Psli2}).
 In the present work, we develop a part of the Semi-infinite Algebraic
Geometry program along the lines of~\cite{Pfp,Psli2}.

 To be more precise, following the approach of~\cite{Psemi}, one has to
distinguish between the semi-infinite \emph{homology} and
\emph{cohomology} theories.
 The semi-infinite homology groups (SemiTor) are assigned to a pair of
\emph{semimodules}, which means roughly ``comodules along a half of
the variables and modules along the other half''.
 The semi-infinite cohomology groups (SemiExt) are assigned to one
semimodule and one \emph{semicontramodule}; the latter means
``a contramodule along a half of the variables and a module along
the other half''.

 In the context of algebraic geometry, quasi-coherent sheaves on
nonaffine schemes and quasi-coherent torsion sheaves on ind-schemes
play the role of the ``comodules along some of the variables'',
while for contramodules one has to consider
the \emph{contraherent cosheaves}~\cite{Pcosh}.
 In the present paper, we restrict ourselves to the semi-infinite
homology, the SemiTor; so contraherent cosheaves do not appear.

 According to the philosophy of~\cite{Psemi}, the key technical
concept for semi-infinite homological algebra is the \emph{semiderived
category}.
 This means ``the coderived category along a half of the variables
mixed with the derived category along the other half''.
 One is supposed to take the coderived category along the coalgebra
variables and the derived category along the ring/algebra variables.

 How does one interpret this prescription in the context of algebraic
geometry, or more specifically, e.~g., quasi-coherent sheaves on
algebraic varieties?
 Is one supposed to take the coderived or the derived category of
quasi-coherent sheaves, or how does it depend on the nature of
the variety at hand?
 This paper grew out of the author's attempts to find an answer to
this question, which spanned more than a decade since about~2009.
 The paper~\cite{Pfp} and the presentation~\cite{Psli2} were the first
formulations of the conclusions I~had arrived at.

\subsection{{}}
 So, what is semi-infinite geometry?
 Before attempting to answer this question, let us discuss
\emph{semi-infinite set theory} first.

 Let $S$ be a set.
 A \emph{semi-infinite structure} on $S$ is the datum of a set of
\emph{semi-infinite subsets} $S^+$ in $S$ such that, for any two
semi-infinite subsets $S^+{}'$ and $S^+{}''\subset S$,
the set-theoretic difference $S^+{}'\setminus S^+{}''$ is
a \emph{finite set}.
 Furthermore, if a subset $S^+\subset S$ is semi-infinite and
$S^+{}'\subset S$ is a subset for which both the sets
$S^+{}'\setminus S^+$ and $S^+\setminus S^+{}'$ are finite,
then $S^+{}'$ should be also a semi-infinite subset.

 So, in order to specify a semi-infinite structure on $S$, it suffices
to point out one particular semi-infinite subset in~$S$; then it
becomes clear what the other semi-infinite subsets are.
 In the \emph{standard semi-infinite structure} on the set of integers
$\boZ$, the subset of positive integers is a semi-infinite subset.

 Given a semi-infinite structure on $S$, the \emph{dual} semi-infinite
structure is formed by the set of all subsets $S^-=S\setminus S^+$,
where $S^+\subset S$ are the semi-infinite subsets.
 The subsets $S^-\subset S$ are called \emph{co-semi-infinite}.
 Every infinite set has the \emph{trivial} semi-infinite structure,
in which the semi-infinite subsets are precisely the finite subsets,
and the \emph{cotrivial} semi-infinite structure, in which
the co-semi-infinite subsets are precisely the finite ones.
 On a finite set, the trivial and cotrivial semi-infinite structures
coincide, and there are no other semi-infinite structures; but any
infinite set admits infinitely many semi-infinite structures.

 The datum of a semi-infinite structure on $S$ is equivalent to
the datum of a compact, Hausdorff topology on the set
$S_\infty=S\sqcup\{-\infty,+\infty\}$ for which the induced topology on
the subset $S\subset S_\infty$ is discrete.
 A subset $S_+\subset S$ contains a semi-infinite subset
if and only if $S_+\sqcup\{+\infty\}$ is a neighborhood of
the point~$+\infty$ in $S_\infty$, while a subset $S_-\subset S$
contains a co-semi-infinite subset if and only if $S_-\sqcup\{-\infty\}$
is a neighborhood of the point~$-\infty\in S_\infty$.
 In other words, a semi-infinite structure on a set is the same thing
as the datum of a two-point compactification.

 Notice that there is a notion of ``relative cardinality'',
a well-defined integer, for a pair of semi-infinite subsets in~$S$.
 Given two semi-infinite subsets $S^+{}'$ and $S^+{}''\subset S$,
put ``$|S^+{}'|-|S^+{}''|$''${}=|\widetilde S^+\setminus S^+{}''|-
|\widetilde S^+\setminus S^+{}'|=|S^+{}'\setminus\overline S^+|-
|S^+{}''\setminus\overline S^+|\in\boZ$, where $\widetilde S^+$ and
$\overline S^+$ are any semi-infinite subsets in $S$ such that
$\overline S\subset S^+{}'\cap S^+{}''$ and
$S^+{}'\cup S^+{}''\subset\widetilde S^+$.
 So it makes sense to say that ``there are more elements in $S^+{}'$
than in $S^+{}''$\,'' (and how many more, precisely), even though none
of the two sets may be contained in the other one, and their
(infinite) cardinalities are the same.

\subsection{{}}
 The notion of a \emph{locally linearly compact} topological vector
space (otherwise known as a \emph{Tate vector space}) is the central
concept of \emph{semi-infinite linear algebra}.

 Here is the formal definition: a complete, separated topological
vector space $W$ over a field~$\kk$ is said to be \emph{linearly
compact} if open vector subspaces of finite codimension form a base
of neighborhoods of zero in~$W$.
 Equivalently, this means that $W$ is isomorphic to the projective
limit of a (directed, if one wishes) projective system of discrete
finite-dimensional vector spaces, endowed with the projective limit
topology; or the product of discrete one-dimensional vector spaces,
endowed with the product topology.
 A topological vector space $V$ is said to be \emph{locally linearly
compact} if it has a linearly compact open subspace.

 Informally, one can say a locally linearly compact vector space $V$ is
a topological vector space whose topological basis is indexed by
a set $S$ with a natural semi-infinite structure.
 A topological basis of a linearly compact open subspace $W\subset V$
is a semi-infinite subset in $S$, while a basis of the discrete
quotient space $V/W$ is the complementary co-semi-infinite subset
in~$S$.

 To a set $S$ with a semi-infinite structure, one can assign
the locally linearly compact topological vector space
$V=\bigoplus_{t\in S\setminus S^+}\kk t\oplus \prod_{s\in S^+}\kk s$,
where $S^+\subset S$ is a semi-infinite subset.
 Here the topology is discrete on the first direct summand and linearly
compact on the second one.
 Obviously, the (topological) vector space $V$ does not depend on
the choice of a particular semi-infinite subset $S^+$ within
the given semi-infinite structure on the set~$S$.

 Similarly to the relative cardinality of a pair of semi-infinite
subsets, one can speak about the ``relative dimension'' (a well-defined
integer) for a pair of linearly compact open subspaces $W'$ and
$W''\subset V$ in a given locally linearly compact topological vector
space~$V$.
 There is also the notion of a ``relative determinant''
(a functorially defined one-dimensional $\kk$\+vector space)
for $W'$ and~$W''$.

 The topological $\kk$\+vector space of formal Laurent power series
$\kk((t))$ in one variable~$t$ with the coefficients in~$\kk$ is
the thematic example of a locally linearly compact topological
vector space.
 The subspace of formal Taylor power series $\kk[[t]]\subset\kk((t))$
is a linearly compact open subspace.

\subsection{{}}
 Now we can offer the following very rough and imprecise definition of
a \emph{semi-infinite geometric object}, or a ``semi-infinitely
structured space''.
 A semi-infinitely structured space is an (infinite-dimensional) space
$Y$ with local or global coordinates $(x_s)_{s\in S}$ indexed by
a set~$S$ with a natural semi-infinite structure.

 It seems to makes sense to require that, for every point $p\in Y$,
the subset $S(p)\subset S$ of all indices $s\in S$ for which
$x_s(p)\ne0$ be contained in a semi-infinite subset, $S(p)\subset S^+$,
in the set~$S$.
 So all the coordinates may vanish simulaneously, but at most
a semi-infinite subset of the coordinates may be simultaneously nonzero
at any given point on~$Y$.

 Then one can say that a closed subvariety $Y^+\subset Y$ is
a \emph{semi-infinite subvariety} if in local coordinates it can be
defined by a system of equations prescribing a fixed value to every
coordinate~$x_s$ with $s\in S^-$, for some co-semi-infinite
subset $S^-\subset S$.
 One can think of a ``semi-infinite homology theory'', in which
semi-infinite subvarieties would be cycles of dimension (homological
degree) $\infty/2+n$, \,$n\in\boZ$.
 Here one postulates $|S^+|=\infty/2$ for one fixed semi-infinite subset
$S^+\subset S$, and then has $|S^+{}'|\in\infty/2+\boZ$ for every other
semi-infinite subset $S^+{}'\subset S$.

 One can think of a ``semi-infinite intersection theory'', in which
subvarieties of finite codimension in $Y$ would form a graded ring,
with respect to the (properly understood) intersection, and
semi-infinite subvarieties would form a $\boZ$\+graded (or
a ``$(\infty/2+\nobreak\boZ)$\+graded'') module over this graded ring.
 The intersection of a semi-infinite subvariety with a subvariety
of finite codimension would produce another semi-infinite subvariety,
interpreted as their product.

 This is the kind of geometric speculation which inspired the present
work.
 Moving closer to the setting in this paper, one can consider
a particular case when the coordinates $x_s$ with $s\in S^-$ can be
somehow globally separated and grouped together, producing
a fibration $\pi\:Y\rarrow X$.
 So the local coordinates on $X$ are indexed by a co-semi-infinite
subset of the varibles, while the local coordinates on the fibers
are indexed by a semi-infinite subset (thus the fibers $Y_q=
\pi^{-1}(q)$, \,$q\in X$, are ``semi-infinite subvarieties'' in $Y$
in the above sense).

\subsection{{}}
 The cornerstone technical, homological principle of our approach to
semi-infinite algebra and geometry is that one is supposed to work
with the \emph{semiderived category}.
 This means a mixture of the derived category along a semi-infinite
subset of the variables and the coderived (or contraderived) category
along a co-semi-infinite subset of the variables.

 One reason why this is important is because the delicate choice of
an exotic derived category to work with dictates the derived functors
which are naturally produced or well-defined.
 One notices this when one starts working systematically with doubly
unbounded complexes.
 The left derived tensor product is well-behaved on the derived 
category, while the right derived cotensor product behaves well
as a functor on the coderived category.
 In the context of algebraic geometry, this means that the left derived
functor of $*$\+restriction onto a closed subscheme is well-defined
on the derived category, while the right derived functor of
$!$\+restriction (or, in our notation, $+$\+restriction) onto
a locally closed subscheme is well-behaved as a functor between
the coderived categories.

 How does one know along which coordinates the derived category needs
to be taken, and along which ones the coderived category?
 One answer to this question is that, given a fibration, there is
a natural way to define a semiderived category that is a mixture of
the coderived category along the base and the derived category along
the fibers.
 Another heuristic is that, given an infinite set of coordinates which
are allowed to be nonzero all simultaneously, one should take
the derived category along these; but if the condition is imposed that
only a finite subset of the coordinates may differ from zero
at any given point, one should take the coderived category.

\subsection{{}}
 So, what is the coderived category?
 There are several alternative answers to this question offered in
the contemporary literature.
 The simplest definition says the the coderived category of
an abelian or exact category $\sE$ is the \emph{homotopy category of
unbounded complexes of injective objects} in~$\sE$.
 It is presumed that there are enough injective objects in~$\sE$.
 This approach was initiated by Krause~\cite{Kra} and developed by
Becker~\cite{Bec} (see also~\cite{Neem2}, \cite{Sto}, and~\cite{PS5}).

 In the present author's work, the coderived categories first appeared
in connection with derived nonhomogeneous Koszul duality~\cite{Pkoszul}
and were subsequently used for the purposes of semi-infinite
homological algebra in~\cite{Psemi} (see also~\cite{Pfp,Pps,PS2}).
 Our approach emphasizes a construction of the coderived category as
the \emph{quotient category} of the homotopy category by what we call
the full subcategory of \emph{coacyclic} complexes.
 The coacyclic complexes are defined as the ones that can be obtained
from the \emph{totalizations of short exact sequences of complexes}
using the operations of cone and infinite coproduct.
 Any coacyclic complex (in an abelian/exact category with exact
coproducts) is acyclic, but the converse is not generally  true.

 The example of the abelian category $X\qcoh$ of quasi-coherent sheaves
on a semi-separated Noetherian scheme $X$ is instructive.
 For this class of schemes, both the derived category $\sD(X\qcoh)$ and
the coderived category $\sD^\co(X\qcoh)$ are perfectly well-behaved.
 In fact, both of them are compactly generated.
 The compact objects in $\sD(X\qcoh)$ are the \emph{perfect complexes},
while the compact objects in $\sD^\co(X\qcoh)$ are all the bounded
complexes of coherent sheaves.
 In this connection, the coderived category $\sD^\co(X\qcoh)$ has been
called the category of ``ind-coherent sheaves'' in~\cite{Gai}.
 Both the derived and the coderived category can be properly considered
for much wider classes of algebro-geometric objects, but the natural
directions for generalization differ: while the derived category
$\sD(X\qcoh)$ makes perfect sense for an arbitrary quasi-compact
semi-separated scheme~$X$, the natural generality for the coderived
category is that of an ind-Noetherian ind-scheme (or ind-stack)~$\X$.

\subsection{{}}
 As we have mentioned above, one important way to think of
the distinction between the derived and the coderived category in
the algebraic geometry context is in connection with
the \emph{left derived inverse image functor} vs.\
the \emph{extraordinary inverse image functor}.
 This observation seems to be due to Gaitsgory~\cite{Gai}.

 Given a morphism of (good enough) schemes $f\:Y\rarrow X$,
the direct image functor $f_*\:Y\qcoh\rarrow X\qcoh$ has finite
homological dimension, so its right derived functor is equally
well-defined on the unbounded derived and the coderived categories,
$\boR f_*\:\sD(Y\qcoh)\rarrow\sD(X\qcoh)$ and $\boR f_*\:
\sD^\co(Y\qcoh)\rarrow\sD^\co(X\qcoh)$.
 The inverse image functor $f^*\:X\qcoh\rarrow Y\qcoh$, however, has
infinite homological dimension in general.
 As infinite left resolutions are problematic in the coderived
categories, the left derived inverse image $\boL f^*\:\sD(X\qcoh)
\rarrow\sD(Y\qcoh)$ is well-defined on the convenional unbouded
derived categories, but \emph{not} on the coderived categories
(unless the morphism~$f$ has finite Tor-dimension).
 Whenever it exists, the left derived inverse image $\boL f^*$ is
left adjoint to~$\boR f_*$.

 The functor $\boR f_*$ preserves coproducts both in the unbounded
derived and the coderived categories; so, by the Brown representability
theorem, it has a right adjoint functor in both the contexts.
 This right adjoint functor, which we, following the terminology
of~\cite{Pcosh}, call the \emph{extraordinary inverse image functor
in the sense of Neeman} (with the reference to~\cite{Neem-bb}) and
denote by~$f^!$, may be not as important for algebraic geometry as
the \emph{extraordinary inverse image functor in the sense of Deligne},
which we denote by~$f^+$ (it is denoted by~$f^!$ in~\cite{Hart}
and Deligne's appendix to~\cite{Hart}).

 The functor~$f^+$ is defined by the conditions that $(fg)^+\simeq
g^+f^+$ for any pair of composable morphisms $f$ and~$g$; one has
$f^+=f^!$ for a proper morphism~$f$; and one has $j^+=j^*$ for
an open immersion~$j$.
 The functor $f^+\:\sD^\co(X\qcoh)\rarrow\sD^\co(Y\qcoh)$ is
well-defined on the coderived categories~\cite[Section~5.16]{Pcosh},
but \emph{not} on the unbounded derived categories.
 In particular, even for locally closed immersions~$f$, it is
\emph{impossible} to define a functor $f^+\:\sD(X\qcoh)\rarrow
\sD(Y\qcoh)$ in such a way that one would have $j^+=j^*$ for open
immersions, while $i^+=\boR i^!$ would be the functor of right derived
restriction with supports for closed immersions~$i$, and $(fg)^+
\simeq g^+f^+$ for all composable pairs of
morphisms~\cite[Example~6.5]{Neem-bb}.
 The punchline: \emph{the unbounded derived category works well with
the left derived inverse image~$\boL f^*$; the coderived category
words well with the extraordinary inverse image~$f^+$}.

\subsection{{}}
 What is the semiderived category?
 To answer this question properly, it is better to step back and ask
the most basic question: \emph{What is the derived category?}
 The derived category is usually defined as the result of localizing
the homotopy category of complexes by the class of quasi-isomorphisms.
 What are the quasi-isomorphisms?
 Let $R$ be an associative algebra over a field~$\kk$, and let
$f\:C^\bu\rarrow D^\bu$ be a morphism of complexes of $A$\+modules.
 What does it mean that $f$~is a quasi-isomorphism?
 Here is the answer which we suggest: let us apply the forgetful
functor and view~$f$ as a morphism of complexes of $\kk$\+vector
spaces.
 The map~$f$ is a quasi-isomorphism of complexes of $R$\+modules
\emph{if and only if} it is a homotopy equivalence of complexes
of $\kk$\+vector spaces.
 In other words, a complex of $R$\+modules is acyclic if and only if
its underlying complex of $\kk$\+vector spaces is contractible.

 The previous paragraph implies that one should think of the derived
category in the context of a forgetful functor.
 But one does not have to go all the way and forget the whole action
of the algebra $R$, staying only with vector spaces.
 One can forget the action of a half of the variables in $R$, while
keeping the other half.
 Let $A\rarrow R$ be a ring homomorphism.
 One defines the \emph{$R/A$\+semiderived category} of $R$\+modules
$\sD^\si_A(R\modl)$ as the quotient category of the homogopy category
of (unbounded complexes of) $R$\+modules by the thick subcategory of
those complexes which are \emph{coacyclic as complexes of $A$\+modules}.
 This definition of the semiderived category can be found
in~\cite[Section~5]{Pfp} (see~\cite[Section~2.3]{Psemi} for
the original definition of the \emph{semiderived category of
semimodules}).
 Thus the semiderived category is a mixture of the ``coderived category
along~$A$'' and the ``derived category in the direction of $R$
relative to~$A$''.
 This is the way to mix the coderived category with the derived category
alluded to several paragraphs above.

 In the context of algebraic geometry, the forgetful functor
$R\modl\rarrow A\modl$ is interpreted geometrically as
the direct image functor.
 To have the direct image functor exact and faithful (as needed for
the definition of a semiderived category), it is simplest to assume
the geometric morphism in question to be \emph{affine}.
 Thus one can speak of the \emph{semiderived category of quasi-coherent
sheaves} $\sD_X^\si(\bnY\qcoh)$ for an affine morphism of schemes
$\pi\:\bnY\rarrow X$.
 More specifically, following the discussion above, one may want
to restrict generality to Noetherian schemes $X$, and then expand it
to ind-Noetherian ind-schemes~$\X$.
 Then one assumes $\pi\:\bY\rarrow\X$ to be a morphism of ind-schemes
with (possibly infinite-dimensional) affine scheme fibers, and
considers the semiderived category of what we call \emph{quasi-coherent
torsion sheaves on\/~$\bY$}.

\subsection{{}}
 What can one do with the semiderived category of quasi-coherent
torsion sheaves?
 Our suggestion is to construct a tensor category structure on it.
 We start with defining the \emph{cotensor product} operation on
the coderived category $\sD^\co(\X\tors)$ of quasi-coherent torsion
sheaves on~$\X$, and proceed further to construct the \emph{semitensor
product} on the $\bY/\X$\+semiderived category $\sD_\X^\si(\bY\tors)$
of quasi-coherent torsion sheaves on~$\bY$.
 For this purpose, we need to choose a \emph{dualizing complex}
on~$\X$.

 The notion of a dualizing complex of quasi-coherent sheaves on
an ind-scheme is itself a perfect illustration of the importance of
the coderived categories.
 Notice that there is some affinity between the dualizing complexes
and the extraordinary inverse images in the sense of Deligne: for
a morphism of schemes $f\:Y\rarrow X$, given a dualizing complex
$\D^\bu_X$ on $X$, the complex $f^+(\D^\bu_X)$ is a dualizing complex
on~$Y$.
 A dualizing complex on a scheme, however, is bounded; so one does not
feel the distinction between the derived and the coderived category
in connection with it (in fact, there is no difference between
the derived and the coderived category for bounded below complexes).
 A dualizing complex on an ind-scheme, on the other hand, is usually
\emph{not} bounded below; so it is important that a dualizing complex
of quasi-coherent torsion sheaves $\rD^\bu_\X$ on an ind-scheme $\X$ has
to be viewed as an object of the coderived category $\sD^\co(\X\tors)$.
 We explain in Section~\ref{tate-space-subsecn}(7) that, for about
the simplest example of an infinite-dimensional ind-scheme $\X$
(the ``discrete affine space over a field''), the dualizing complex
$\rD^\bu_\X$ on $\X$ is an acyclic complex!

\subsection{{}}
 Let us have a more detailed discussion of the cotensor and semitensor
products, whose definitions are the main objectives of this paper.
 The functor $\Tor^\boZ_1$, known classically as the \emph{torsion
product} of abelian groups~\cite{McL,Nun,Kee}, defines a tensor
structure on the category of torsion abelian groups with $\boQ/\boZ$ as
the unit object.
 The torsion product of $p$-primary abelian groups, for a fixed
prime number~$p$, is very similar to the cotensor product of
comodules over the coalgebra $\rC$ dual to the topological algebra
$\kk[[z]]$ of formal Taylor power series in one variable over a field.

 The first aim of this paper is to extend these constructions to
complexes of quasi-coherent torsion sheaves on an ind-Noetherian
ind-scheme $\X$ with a dualizing complex.
 The dualizing complex $\rD_\X^\bu$ is the unit object of this
tensor category structure, which is defined on the coderived
category~\cite{Pkoszul} \,$\sD^\co(\X\tors)$ of quasi-coherent
torsion sheaves on~$\X$.
 The case of a Noetherian scheme $X$ with a dualizing complex
$\D_X^\bu$ was considered by Murfet in~\cite[Proposition~B.6]{Mur}
and in our paper~\cite[Section~B.2.5]{EP}, and the case of
an ind-affine ind-Noetherian (or ind-coherent) $\aleph_0$\+ind-scheme
$\X$ with a dualizing complex, in~\cite[Section~D.3]{Pcosh}.

 Furthermore, pursuing the Semi-infinite Algebraic Geometry program
as outlined in the introduction to the paper~\cite{Pfp}
and in the presentation~\cite{Psli2}, we consider a flat affine
morphism of ind-schemes $\pi\:\bY\rarrow\X$.
 The fibers of the morphism~$\pi$ are, generally speaking,
infinite-dimensional affine schemes.
 Then the semiderived category~\cite{Psemi,Pfp}
\,$\sD^\si_\X(\bY\tors)$ of the abelian category of quasi-coherent
torsion sheaves on $\bY$ relative to the direct image functor
$\pi_*\:\bY\tors\rarrow\X\tors$ carries the operation of
\emph{semitensor product}, making it a tensor category whose unit
object is the inverse image $\pi^*\rD_\X^\bu$ of the dualizing complex
on $\X$ to the ind-scheme~$\bY$.
 The case of an affine Noetherian (or coherent) scheme $\Spec A$
in the role of $\X$ and an affine scheme $\Spec R$ in the role
of $\bY$ was considered in~\cite[Section~6]{Pfp}.

 The cotensor product $\oc_{\rD_\X^\bu}$ of complexes of quasi-coherent
torsion sheaves on $\X$ is similar to a right derived functor, in that
(under mild assumptions on~$\rD_\X^\bu$) the cotensor product of two
bounded complexes is a complex bounded from below.
 The semitensor product $\os_{\pi^*\rD_\X^\bu}$ of complexes of
quasi-coherent torsion sheaves of $\bY$ resembles a double-sided
derived functor, as the semitensor product of two bounded complexes
is, generally speaking, a doubly unbounded complex.
 Thus our semitensor product construction can be viewed as a version
of \emph{semi-infinite homology theory} for quasi-coherent sheaves;
and indeed, it reduces to a particular case of the semi-infinite
homology functor $\SemiTor$ from the book~\cite{Psemi} in the case of
an ind-zero-dimensional (ind-Artinian) ind-scheme $\X$ of ind-finite
type over a field.

\subsection{{}}
 From the perspective of algebraic geometry, the most natural way to
think about the conventional left derived functor of tensor product
of quasi-coherent sheaves on a scheme may be the following one.
 Let $Y$ be a scheme over a field~$\kk$, and let $\M^\bu$ and $\N^\bu$
be two complexes of quasi-coherent sheaves over it.
 Consider the Cartesian product $Y\times_\kk Y$, and consider
the external tensor product $\M^\bu\bt_\kk\N^\bu$ of the complexes
$\M^\bu$ and $\N^\bu$; this is a complex of quasi-coherent sheaves
on $Y\times_\kk Y$.
 Let $\Delta_Y\:Y\rarrow Y\times_\kk Y$ be the diagonal morphism.
 Then the derived tensor product of $\M^\bu$ and $\N^\bu$ on $Y$ can
be defined as the left derived restriction of the external tensor
product to the diagonal, $\M^\bu\ot_{\cO_Y}^\boL\N^\bu=
\boL\Delta_Y^*(\M^\bu\bt_\kk\N^\bu)$.

 In this spirit, if one is interested in alternative tensor product
operations on quasi-coherent sheaves, the most natural approach may be
to keep the external tensor product unchanged, but tweak
the restriction to the diagonal.
 Following the discussion above, one is supposed to take $\boR\Delta^!$
(or~$\Delta^+$) for the cotensor product and some ``combination of
$\boL\Delta^*$ with $\boR\Delta^!$\,'' for the semitensor product.

 It was shown already in~\cite[end of Section~B.2.5]{EP} that
Murfet's tensor structure~\cite{Mur} (which we call the \emph{cotensor
product}) on the coderived category of quasi-coherent sheaves on
a separated scheme $X$ of finite type over a field~$\kk$ can be
computed as $\M^\bu\oc_{\D_X^\bu}\N^\bu\simeq
\boR\Delta_X^!(\M^\bu\bt_\kk\N^\bu)$ for any two complexes of
quasi-coherent sheaves $\M^\bu$ and $\N^\bu$ on~$X$.
 This presumes the choice of the dualizing complex
$\D_X^\bu=p^+(\kk)$ on the scheme $X$, where $p\:X\rarrow\Spec\kk$ is
the morphism of finite type (this is called the \emph{rigid
dualizing complex} in~\cite{YZ,Yek,Sha}).
 In this paper, we extend this computation to ind-separated
ind-schemes $\X$ of ind-finite type, and further to
the relative/semi-infinite context of an affine morphism
$\pi\:\bY\rarrow\X$.

\subsection{{}}
 The first three sections of this paper develop the basic language of
ind-schemes, quasi-coherent torsion sheaves, and pro-quasi-coherent
pro-sheaves.
 The pro-quasi-coherent pro-sheaves are used in the subsequent sections
as a key technical tool for our constructions involving quasi-coherent
torsion sheaves.
 In the next three sections, we consider an ind-Noetherian ind-scheme
$\X$, the coderived category $\sD^\co(\X\tors)$, and the cotensor
product operation~$\oc_{\rD_\X^\bu}$ on it.
 In the last five sections, we study the relative (or properly
semi-infinite) situation: a flat affine morphism $\pi\:\bY\rarrow\X$,
the semiderived category $\sD_\X^\si(\bY\tors)$, and the semitensor
product functor~$\os_{\pi^*\rD_\X^\bu}$.

 We discuss the basics of the theory of (strict ind-quasi-compact
ind-quasi-separated) ind-schemes in Section~\ref{ind-schemes-secn}.
 The additive/exact/abelian categories of module objects over
ind-schemes (namely, the quasi-coherent torsion sheaves and
the pro-quasi-coherent pro-sheaves) are defined and discussed in
Sections~\ref{torsion-sheaves-secn}\+-\ref{flat-pro-sheaves-secn}.
 In particular, we show that quasi-coherent torsion sheaves on
a reasonable ind-scheme (in the sense of~\cite{BD2}) form
a Grothendieck category.

 The equivalence between the coderived category of quasi-coherent
torsion sheaves and the derived category of flat pro-quasi-coherent
pro-sheaves on an ind-semi-separated ind-Noetherian ind-scheme with
a dualizing complex is constructed in
Section~\ref{dualizing-on-ind-Noetherian-secn}.
 The cotensor product functor over such an ind-scheme is defined and
many (particularly ind-Artinian) examples of it are considered in
Section~\ref{cotensor-secn}.
 In particular, we establish the comparisons with the derived cotensor
product of comodules over a cocommutative coalgebra over a field and
the torsion product of torsion modules over a Dedekind domain.

 The particular case of an ind-separated ind-scheme $\X$ of ind-finite
type over a field~$\kk$ is considered in
Section~\ref{ind-finite-type-secn}.
 In this setting, we compute the cotensor product of complexes of
quasi-coherent torsion sheaves as the right derived
$!$\+restriction of the external tensor product to the diagonal
closed immersion $\Delta_\X\:\X\rarrow\X\times_\kk\X$.

 In the relative context of a flat affine morphism $\pi\:\bY\rarrow\X$
into an ind-semi-separated ind-Noetherian ind-scheme with a dualizing
complex, the equivalence between the semiderived category of
quasi-coherent torsion sheaves on $\bY$ relative to $\X$ and
the derived category of $\X$\+flat pro-quasi-coherent pro-sheaves
on $\bY$ is constructed in Section~\ref{X-flat-on-Y-secn}.
 In particular, the semiderived category $\sD_\X^\si(\bY\tors)$ itself
is introduced in Section~\ref{semiderived-subsecn}.
 It is explained in Remark~\ref{semiderived-compact-generators}
that the semiderived category is compactly generated.
 The functor of semitensor product of complexes of quasi-coherent
torsion sheaves over $\bY$ relative to $\X$ is defined in
Section~\ref{semitensor-secn}.
 The special case of a flat affine morphism $\bY\rarrow\X$ with
an ind-zero-dimensional ind-scheme of ind-finite type $\X$ over
a field~$\kk$ is considered in Section~\ref{ind-artinian-base-subsecn},
and the comparison with the functor SemiTor from
the book~\cite{Psemi} is discussed.

 In Section~\ref{flat-affine-over-ind-finite-type-secn}
we compute the semiderived product as a combination of two kinds of
derived restrictions to the diagonal.
 Let $\X$ be an ind-separated ind-scheme of ind-finite type over
a field~$\kk$, and let $\pi\:\bY\rarrow\X$ be a flat affine morphism.
 The diagonal map $\Delta_\bY\:\bY\rarrow\bY\times_\kk\bY$ factorizes
naturally into two ``partial diagonals'' $\bY\overset\delta\rarrow
\bY\times_\X\bY\overset\eta\rarrow\bY\times_\kk\bY$;
both~$\delta$ and~$\eta$ are closed immersions.
 Let $\rD_\X^\bu$ be a rigid dualizing complex on~$\X$.
 For any two complexes $\bM^\bu$ and $\bN^\bu$ of quasi-coherent
torsion sheaves on $\bY$, we construct a natural isomorphism
$\bM^\bu\os_{\pi^*\rD_\X^\bu}\bN^\bu\simeq\boL\delta^*\,\boR\eta^!
(\bM^\bu\bt_\kk\bN^\bu)$ in $\sD_\X^\si(\bY\tors)$.
 In Section~\ref{weakly-smooth-postcomposition-secn} we show that both
the semiderived category and the semitensor product operation on it are
preserved when a flat affine morphism $\pi\:\bY\rarrow\X$ is replaced
with a flat affine morphism $\pi'\:\bY\rarrow\X'$ such that
$\pi=\tau\pi'$, where $\tau\:\X'\rarrow\X$ is a smooth (or ``weakly
smooth'') affine morphism of finite type between ind-semi-separated
ind-Noetherian ind-schemes.

 Several simple, but geometrically nontrivial examples of flat
affine morphisms of ind-schemes $\pi\:\bY\rarrow\X$ with
an ind-separated ind-scheme $\X$ of ind-finite type over a field
are considered in Section~\ref{geometric-examples-secn}.
 In particular, in Section~\ref{tate-space-subsecn}, we discuss
the example of the flat affine morphism of ind-affine ind-schemes
$\bY\rarrow\X$ corresponding to a surjective open linear map of
locally linearly compact topological vector spaces $V\rarrow V/W$
with a discrete target $V/W$ and a linearly compact kernel $W$,
such as $\kk((t))\rarrow\kk((t))/\kk[[t]]$.
 Here $\kk((t))$ is the topological vector space of formal Laurent
power series in a variable~$t$ over a field~$\kk$, and
$\kk[[t]]\subset\kk((t))$ is the open vector subspace
of Taylor power series.
 It follows from the results of
Section~\ref{weakly-smooth-postcomposition-secn} that the semiderived
category $\sD_\X^\si(\bY\tors)$ in this example does \emph{not} depend
on the choice of a linearly compact open subspace $W\subset V$, but is
entirely determined by the locally linearly compact vector space~$V$.
 The semitensor product operation on $\sD_\X^\si(\bY\tors)$ also
does not depend on~$W$ up to a dimensional cohomological shift and
a determinantal twist.
 A generalization of this example to the ind-schemes of formal loops
taking values in affine algebraic groups is very briefly discussed
in Section~\ref{loop-group-subsecn}. 

 This paper does not strive for maximal natural generality.
 Instead, our aim is to demonstrate certain phenomena in a suitable
context where they manifest themselves in a fully nontrivial way.
 If one is interested in generalizations, one of the first ideas
would be to replace ind-schemes with ind-stacks.
 Another and perhaps even more important direction for possible
generalization is to replace an affine morphism $\pi\:\bY\rarrow\X$
with a quasi-compact, semi-separated one.
 We attempt to make the first step in this direction in the appendix,
where a definition of the semiderived category $\sD_\X^\si(\bY\tors)$
for a nonaffine morphism~$\pi$ is worked out.
 
\subsection*{Acknowledgement}
 Boris Feigin shared his vision of semi-infinite geometry and topology
with his students at the Independent University of Moscow, including me,
around 1994--95.
 These conversations became a major source of inspiration for me
ever since.
 I~am grateful to Sergey Arkhipov for numerous illuminating discussions
of semi-infinite algebra and geometry, particularly in late '00s and
early '10s, which influenced my thinking on the subject and eventually
this work.
 I~also wish to thank Michal Hrbek for helpful conversations, and
Liran Shaul, Tsutomu Nakamura, Alexei Pirkovskii, and Sergey Arkhipov 
for suggesting useful references.
 The author is supported by the GA\v CR project 20-13778S and
research plan RVO:~67985840.

\Section{Ind-Schemes and their Morphisms}  \label{ind-schemes-secn}

\subsection{Ind-objects}  \label{ind-objects-subsecn}
 Let $\sK$ be a small category.
 Consider the category $\Sets^{\sK^\sop}$ of contravariant functors
from $\sK$ to the category of sets.
 The category $\Ind(\sK)$ of \emph{ind-objects in\/~$\sK$} can be
defined as the full subcategory in $\Sets^{\sK^\sop}$ consisting of
the (filtered) direct limits of representable functors
$\Mor_\sK({-},K)$, \ $K\in\sK$.

 Explicitly, this means that the objects of $\Ind(\sK)$ are represented
by inductive systems $(K_\gamma\in\sK)_{\gamma\in\Gamma}$ indexed by
directed posets~$\Gamma$.
 For any $\gamma'$, $\gamma''\in\Gamma$ there exists $\gamma\in\Gamma$
such that $\gamma'\le\gamma$, \,$\gamma''\le\gamma$; and for every
$\beta<\gamma\in\Gamma$ the transition morphism $K_\beta\rarrow
K_\gamma$ in $\sK$ is given in such a way that the triangle diagrams
$K_\alpha\rarrow K_\beta\rarrow K_\gamma$ are commutative for all
$\alpha<\beta<\gamma\in\Gamma$.
 The object of $\Ind(\sK)$ represented by an inductive system
$(K_\gamma)_{\gamma\in\Gamma}$ is denoted by
$\ilim_{\gamma\in\Gamma}K_\gamma\in\Ind(\sK)$.
 The set of morphisms in $\Ind(\sK)$ between the two objects
represented by two inductive systems $(K_\gamma)_{\gamma\in\Gamma}$
and $(L_\delta)_{\delta\in\Delta}$ is computed by the formula
\begin{equation} \label{ind-objects-mor}
 \Mor_{\Ind(\sK)}(\ilim_{\gamma\in\Gamma}K_\gamma,\,
 \ilim_{\delta\in\Delta}L_\delta)\,=\,
 \varprojlim\nolimits_{\gamma\in\Gamma}^{\Sets}\,
 \varinjlim\nolimits_{\delta\in\Delta}^{\Sets}\,
 \Mor_\sK(K_\gamma,L_\delta),
\end{equation}
where the inductive and projective limits in the right-hand side
are taken in the category of sets.

 One can (and we will) consider $\sK$ as a full subcategory in
$\Ind(\sK)$, embedded by the functor assigning to an object $M\in\sK$
the related inductive system $(M_\epsilon)_{\epsilon\in\mathrm{E}}$
indexed by the singleton poset $\mathrm{E}=\{*\}$ with $M_*=M$.
 Then the formula~\eqref{ind-objects-mor} essentially means that,
firstly, $\ilim_{\gamma\in\Gamma}K_\gamma=
\varinjlim_{\gamma\in\Gamma}^{\Ind(\sK)}K_\gamma$, and secondly,
$\Mor_{\Ind(\sK)}(K,\ilim_{\delta\in\Delta}L_\delta)=
\varinjlim_{\delta\in\Delta}^{\Sets}\Mor_\sK(K,L_\delta)$ for all
$K\in\sK$ and $\ilim_{\gamma\in\Gamma}K_\gamma$,
$\ilim_{\delta\in\Delta}L_\delta\in\Ind(\sK)$.

 One can drop the assumption that the category $\sK$ be small and,
for any category $\sK$, define the category of ind-objects
$\Ind(\sK)$ by saying that the objects of $\Ind(\sK)$ correspond to
directed inductive systems in $\sK$ and the morphisms are given by
the formula~\eqref{ind-objects-mor}.

 Let $\sK$ be a category with fibered products (i.~e., every pair of
morphisms with the same target $K\rarrow M$ and $L\rarrow M$ has
a pullback in~$\sK$).
 Then the category $\Ind(\sK)$ also has fibered products, which can be
constructed as follows.

 Let $f\:\ilim_{\gamma\in\Gamma}K_\gamma\rarrow
\ilim_{\epsilon\in\mathrm{E}}M_\epsilon$ and
$g\:\ilim_{\delta\in\Delta}L_\delta\rarrow
\ilim_{\epsilon\in\mathrm{E}}M_\epsilon$ be a pair of morphisms
with the same target in $\Ind(\sK)$.
 Denote by $\Xi$ the set of all quintuples
$(\gamma',\delta',\epsilon',f_{\gamma'\epsilon'},g_{\delta'\epsilon'})$
such that $\gamma'\in\Gamma$, $\delta'\in\Delta$, $\epsilon'\in
\mathrm{E}$, while $f_{\gamma'\epsilon'}\:K_{\gamma'}
\rarrow M_{\epsilon'}$ and $g_{\delta'\epsilon'}\:L_{\delta'}\rarrow
M_{\epsilon'}$ are morphisms in $\sK$ forming two commutative square
diagrams with the morphisms $f$, $g$ and the natural morphisms
$K_{\gamma'}\rarrow\ilim_{\gamma\in\Gamma}K_\gamma$, \
$L_{\delta'}\rarrow\ilim_{\delta\in\Delta}L_\delta$, \
$M_{\epsilon'}\rarrow\ilim_{\epsilon\in\mathrm{E}}M_\epsilon$
in $\Ind(\sK)$.

 The set $\Xi$ is directed in the natural partial order, and
the projection maps $\Xi\rarrow\Gamma$, $\Xi\rarrow\Delta$,
$\Xi\rarrow\mathrm{E}$ are cofinal maps of directed posets.
 Put $K_\xi=K_{\gamma'}$, \ $L_\xi=L_{\delta'}$, \ $M_\xi=M_{\epsilon'}$
for $\xi=(\gamma',\delta',\epsilon',
f_{\gamma'\epsilon'},g_{\delta'\epsilon'})\in\Xi$.
 Then our pair of morphisms with the same target $(f,g)$ in $\Ind(\sK)$
can be obtained by applying the functor $\ilim_{\xi\in\Xi}$ to
the $\Xi$\+indexed inductive system of pairs of morphisms with
the same target $f'_\xi\:K_\xi\rarrow M_\xi$ and $g'_\xi\:L_\xi
\rarrow M_\xi$ in~$\sK$ (where $f'_\xi=f_{\gamma'\epsilon'}$
and $g'_\xi=g_{\delta'\epsilon'}$).

 Finally, let $N_\xi\in\sK$ denote the fibered product of the pair of
morphisms $f'_\xi$ and~$g'_\xi$.
 Then $\ilim_{\xi\in\Xi}N_\xi\in\Ind(\sK)$ is the fibered product of
the pair of morphisms $f$ and~$g$.
 In other words, for any directed poset $\Xi$,
the functor $\ilim_{\xi\in\Xi}\:\sK^\Xi\rarrow\Ind(\sK)$
(acting from the category $\sK^\Xi$ of $\Xi$\+indexed inductive
systems in $\sK$ to the category of ind-objects) preserves fibered
products, as one can easily see.

 In fact, stronger assertions hold.
 First of all, for any category $\sK$, the category $\Ind(\sK)$ can be
equivalently defined using inductive systems indexed by filtered
categories rather than filtered posets~\cite[Theorem~1.5]{AR}.
 Furthermore, the full subcategory $\Ind(\sK)$ is closed under
(filtered) direct limits in $\Sets^{\sK^\sop}$
\,\cite[Theorem~6.1.8]{KS} (so direct limits exist in $\Ind(\sK)$).
 Returning to the particular case of a category $\sK$ with fibered
products, one observes that $\Ind(\sK)$ is closed under fibered
products in $\Sets^{\sK^\sop}$ in this case.
 Since finite limits commute with (filtered) direct limits in
$\Sets^{\sK^\sop}$, it follows that fibered products commute with
direct limits in $\Ind(\sK)$.

\subsection{Ind-schemes}  \label{ind-schemes-subsecn}
 Ind-schemes are the main object of study in this paper.
 Contrary to what the name seems to suggest, in the advanced
contemporary point of view ind-schemes are \emph{not} defined as
ind-objects in the category of schemes (or in any other category).
 Rather, one considers the category of affine schemes $\AffSch$
(so $\AffSch\simeq\CRings^\sop$, where $\CRings$ is the category of
commutative rings).
 Then the ind-schemes are those contravariant functors $\AffSch^\sop
\rarrow\Sets$ (``presheaves of sets'') on the category of affine
schemes which can be obtained as direct limits of functors
representable by schemes~\cite[Section~7.11]{BD2},
\cite[Section~1]{Rich}.
 So the category of ind-schemes is a certain full subcategory in
the category of presheaves of sets on $\AffSch$.

 Let $\Sch$ denote the category of schemes, and let $\CSch\subset\Sch$
be the full subcategory of concentrated (that is, quasi-compact
quasi-separated) schemes.
 Then, for any $X\in\Sch$, the functor $\Mor_{\Sch}({-},X)$ is actually
a \emph{sheaf} of sets in the Zariski topology on $\Sch$.
 In particular, the restrictions of $\Mor_{\Sch}({-},X)$ to $\CSch$
and $\AffSch$ are sheaves in the Zariski topology on the categories
of concentrated schemes and affine schemes, respectively.
 Furthermore, the filtered direct limit functors in the categories
of presheaves and Zariski sheaves of sets on $\CSch$ agree, due to
the finite nature of the Zariski topology on a concentrated scheme.
 The same applies to the prescheaves and Zariski sheaves on $\AffSch$.
 For these reasons, one can consider ind-schemes as Zariski sheaves
of sets on $\AffSch$, which are the same thing as Zariski sheaves of
sets on $\CSch$.
 These arguments explain that ind-schemes can be equivalently viewed as
functors $\CSch^\sop\rarrow\Sets$ instead of $\AffSch^\sop\rarrow\Sets$.

 In view of the previous paragraph, one can develop the following
simplified, na\"\i ve approach to the definition of
\emph{ind-concentrated ind-schemes}, which we will use.
 An (\emph{ind-concentrated}) \emph{ind-scheme} $\X$ is defined as
an ind-object in the category of concentrated schemes $\CSch$.
 In the rest of this paper, all ``schemes'' will be presumed
concentrated, and all ``ind-schemes'' will be ind-concentrated.
 An ind-scheme is said to be \emph{strict} if it can be represented by
an inductive system of closed immersions, i.~e.,
$\X=\ilim_{\gamma\in\Gamma}X_\gamma$, where the morphism of schemes
$X_\gamma\rarrow X_\delta$ in the inductive system is a closed
immersion for every $\gamma<\delta\in\Gamma$.

 An \emph{$\aleph_0$\+ind-scheme} $\X$ is an ind-scheme which can be 
represented by a countable inductive system of schemes, i.~e.,
$\X=\ilim_{\gamma\in\Gamma}X_\gamma$, where the poset $\Gamma$ is
has (at most) countable cartinality.
 It is not difficult to see that any strict $\aleph_0$\+ind-scheme $\X$
can be represented by a sequence of closed immersions indexed by
the natural numbers, $\X=\ilim\,(X_0\to X_1\to X_2\to\dotsb)$, where
the transition morphisms $X_n\rarrow X_{n+1}$ are closed immersions
of schemes.
 The fact that closed immersions are monomorphisms in $\Sch$ is
helpful to keep in mind here (and elsewhere below).

 Following the discussion in Section~\ref{ind-objects-subsecn},
fibered products exist in the category of (ind-concentrated) ind-schemes
(because they exist in $\CSch$, as the full subcategory $\CSch$ is
closed under fibered products in $\Sch$).
 Furthermore, the full subcategory of strict ind-schemes is closed under
fibered products in the category of ind-schemes.
 Moreover, looking into the above construction of the fibered product
in the case of three ind-objects $\ilim_{\gamma\in\Gamma}K_\gamma$,
\,$\ilim_{\delta\in\Delta}L_\delta$, and $\ilim_{\epsilon\in\mathrm{E}}
M_\epsilon$ represented by inductive systems of monomorphisms, one can
see that the morphisms $f_{\gamma'\epsilon'}$ and $g_{\delta'\epsilon'}$
are determined by the indices $\gamma$', $\delta'$, and $\epsilon'$;
so the set $\Xi$ is at most countable whenever the sets
$\Gamma$, $\Delta$, and $\mathrm{E}$ are.
 Hence the full subcategory of strict $\aleph_0$\+ind-schemes
is closed under fibered products in the category of strict ind-schemes.

 It is also useful to observe that, considering $\CSch$ as a full
subcategory in strict ind-schemes, the fibered product of any two
schemes over a strict ind-scheme is a scheme.
 Indeed, given a strict ind-scheme $\X=\ilim_{\gamma\in\Gamma}X_\gamma$
with closed immersions $X_\gamma\rarrow X_\delta$ for $\gamma<\delta
\in\Gamma$, for any two schemes $Y$ and $Z$ and morphisms
$Y\rarrow\X\larrow Z$, the fibered product $Y\times_\X Z$ is isomorphic
to $Y\times_{X_\gamma}Z$, where $\gamma\in\Gamma$ is any index for
which both the morphisms $Y\rarrow\X$ and $Z\rarrow\X$ factorize
through $X_\gamma\rarrow\X$.

\begin{rem} \label{strict-representation-uniqueness}
 The definition of a strict ind-scheme raises an obvious question.
 Let $\X\in\Ind(\CSch)$ be an ind-scheme which can be represented by
an inductive system of closed immersions in $\CSch$.
 In what sense is such a representation unique if it exists?
 Suppose that $\X\simeq\ilim_{\gamma\in\Gamma}Y_\gamma\simeq
\ilim_{\delta\in\Delta}Z_\delta$, where $(Y_\gamma)_{\gamma\in\Gamma}$
and $(Z_\delta)_{\delta\in\Delta}$ are inductive systems of closed
immersions of schemes.
 In what sense are these two ``strict'' representations of $\X$
equivalent?

 A more advanced approach is to consider the poset $\mathrm{E}=
\Gamma\times\Delta$ with the product order, and write
$\X\simeq\ilim_{(\gamma,\delta)\in\mathrm{E}}W_{\gamma,\delta}$,
where $W_{\gamma,\delta}=Y_\gamma\times_{\X}Z_\delta$.
 Still one may be interested in having a straightforward answer to
the following straightforward question.
 Suppose that, for some $\gamma\in\Gamma$ and $\delta\in\Delta$,
the morphism $Y_\gamma\rarrow\X$ factorizes as $Y_\gamma\rarrow
Z_\delta\rarrow\X$.
 How do we know that $Y_\gamma\rarrow Z_\delta$ is a closed immersion?

 The next two lemmas (in which the schemes do not need to be
concentrated) provide an explanation in a natural generality.
\end{rem}

\begin{lem} \label{closed-immersion-two-morphisms-lemma}
\textup{(a)} Let $X\rarrow Y\rarrow Z$ be morphisms of schemes
such that the composition $X\rarrow Z$ is a locally closed immersion.
 Then the morphism $X\rarrow Y$ is a locally closed immersion. \par
\textup{(b)} Let $X\rarrow Y\rarrow Z$ be morphisms of schemes
such that the composition $X\rarrow Y\rarrow Z$ is a closed immersion
and the morphism $Y\rarrow Z$ is separated.
 Then $X\rarrow Y$ is a closed immersion. \par
\textup{(c)} In particular, if the composition $X\rarrow Y\rarrow Z$
is a closed immersion and $Y$ is a separated scheme, then the morphism
$X\rarrow Y$ is a closed immersion.
\end{lem}

\begin{proof}
 Parts~(a) and~(b) are~\cite[Tag~07RK]{SP}.
 The morphism $X\rarrow Y$ is the composition $X=X\times_YY\rarrow
X\times_ZY\rarrow Y$, where the morphism $X\times_ZY\rarrow Y$ is
a base change of the morphism $X\rarrow Z$
(by the morphism $Y\rarrow Z$).
 If $X\rarrow Z$ is a locally closed immersion, then $X\times_ZY\rarrow
Y$ is a locally closed immersion; and if $X\rarrow Z$ is a closed
immersion, then $X\times_ZY\rarrow Y$ is a closed
immersion~\cite[Tag~01JU]{SP}.
 The morphism $X\times_YY\rarrow X\times_ZY$ is a locally closed
immersion; and moreover if the morphism $Y\rarrow Z$ is separated,
then $X\times_YY\rarrow X\times_ZY$ is a closed immersion,
by~\cite[Tag~01KR]{SP}.
 Finally, a composition of locally closed immersions is a locally
closed immersion, and a composion of closed immersions is a closed
immersion~\cite[Tag~02V0]{SP}.
 Part~(c) follows from~(b), since any scheme morphism from a separated
scheme is separated~\cite[Tag~01KV]{SP}.
\end{proof}

\begin{lem} \label{closed-immersion-three-morphisms-lemma}
 Let $X\rarrow Y\rarrow Z\rarrow W$ be morphisms of schemes such that
the compositions $X\rarrow Z$ and $Y\rarrow W$ are closed immersions.
 Then the morphism $X\rarrow Y$ is a closed immersion.
\end{lem}

\begin{proof}
 By Lemma~\ref{closed-immersion-two-morphisms-lemma}(a) applied to
the pair of morphisms $Y\rarrow Z\rarrow W$, the morphism $Y\rarrow Z$
is a locally closed immersion.
 Any locally closed immersion is a separated morphism; so
Lemma~\ref{closed-immersion-two-morphisms-lemma}(b) is applicable
to the pair of morphisms $X\rarrow Y\rarrow Z$, proving that
$X\rarrow Y$ is a closed immersion.
\end{proof}

 Returning to the question posed in
Remark~\ref{strict-representation-uniqueness}, one can choose
an index $\epsilon\in\Gamma$ such that the morphism $Z_\delta\rarrow\X$
factorizes as $Z_\delta\rarrow Y_\epsilon\rarrow\X$, and then one
can choose an index $\eta\in\Delta$ such that the morphism
$Y_\epsilon\rarrow\X$ factorizes as $Y_\epsilon\rarrow Z_\eta\rarrow\X$.
 The triangle diagrams $Y_\gamma\rarrow Z_\delta\rarrow Y_\epsilon$
and $Z_\delta\rarrow Y_\epsilon\rarrow Z_\eta$ are commutative, because
the transition morphisms in both the inductive systems
are monomorphisms.
 By Lemma~\ref{closed-immersion-three-morphisms-lemma}, it follows
that $Y_\gamma\rarrow Z_\delta$ is a closed immersion.

\subsection{Morphisms of ind-schemes}
\label{morphisms-of-ind-schemes-subsecn}
 \emph{In the sequel, we will presume all our ind-schemes to be
(ind-concentrated and) strict.}
 Let us start from the observation that any morphism of ind-schemes
can be obtained by applying the functor $\ilim\,\,$ to some morphism
of inductive systems of (closed immersions of) schemes.

 Indeed, let $f\:\ilim_{\delta\in\Delta}Y_\delta=\Y\rarrow\X=
\ilim_{\gamma\in\Gamma}X_\gamma$ be a morphism of ind-schemes
represented by inductive systems of closed immersions
$(Y_\delta)_{\delta\in\Delta}$ and $(X_\gamma)_{\gamma\in\Gamma}$.
 Arguing as in Section~\ref{ind-objects-subsecn} and keeping in
mind that closed immersions are monomorphisms in $\CSch$,
one can consider the directed poset $\Xi$ of all pairs
$(\delta\in\Delta,\,\gamma\in\Gamma)$ such that the composition
$Y_\delta\rarrow\Y\rarrow\X$ factorizes through the morphism
$X_\gamma\rarrow\X$.
 Then one has $f=\ilim_{\xi\in\Xi}f_\xi$, where $f_\xi\:Y_\xi\rarrow
X_\xi$ is the related morphism between $Y_\xi=Y_\delta$ and
$X_\xi=X_\gamma$.

 Alternatively, consider the product of two posets
$\mathrm{E}=\Delta\times\Gamma$ with the product order, and
for every $\epsilon=(\delta,\gamma)\in\mathrm{E}$ put
$Y_\epsilon=Y_\delta\times_\X X_\gamma$ and $X_\epsilon=X_\gamma$.
 Then $f=\ilim_{\epsilon\in\mathrm{E}}f_\epsilon$, where
$f_\epsilon\:Y_\epsilon\rarrow X_\epsilon$.

 The following representation of morphisms of schemes may be even
more useful.
 Following~\cite[Lemma~1.7 and Definition~1.8]{Rich}, one says that
a morphism of ind-schemes $f\:\Y\rarrow\X$ is ``representable by
schemes'' if, for every scheme $T$ and a morphism of ind-schemes
$T\rarrow\X$, the fibered product $\Y\times_\X T$ is a scheme.
 In this case, let $\X=\ilim_{\gamma\in\Gamma}X_\gamma$ be
a representation of $\X$ by an inductive system of closed immersions
of schemes.
 Put $Y_\gamma=\Y\times_\X X_\gamma$.
 Then $\Y=\ilim_{\gamma\in\Gamma}Y_\gamma$ is a representation of $\Y$
by an inductive system of closed immersions of schemes, and
the morphism $f$ is represented as $f=\ilim_{\gamma\in\Gamma}f_\gamma$,
where $f_\gamma\:Y_\gamma\rarrow X_\gamma$.

 A morphism of ind-schemes $f\:\Y\rarrow\X$ is said to be
\emph{affine} if, for any scheme $T$ and a morphism of ind-schemes
$T\rarrow\X$, the ind-scheme $\Y\times_\X T$ is a scheme \emph{and}
the morphism of schemes $\Y\times_\X T\rarrow T$ is affine.
 Equivalently (in view of~\cite[Lemma~1.7]{Rich}), this means that
$\Y\times_\X T$ is an affine scheme whenever $T$ is an affine scheme.
 Let $\X=\ilim_{\gamma\in\Gamma}X_\gamma$ be an ind-scheme represented
by an inductive system of closed immersions of schemes.
 Then a morphism of ind-schemes $f\:\Y\rarrow\X$ is affine if and only
if, for every $\gamma\in\Gamma$, the fibered product
$\Y\times_\X X_\gamma$ is a scheme and the morphism of schemes
$\Y\times_\X X_\gamma\rarrow X_\gamma$ is affine.

 A morphism of ind-schemes $\Z\rarrow\X$ is said to be a \emph{closed
immersion} if, for every scheme $T$ and a morphism of ind-schemes
$T\rarrow\X$, the ind-scheme $\Z\times_\X T$ is a scheme \emph{and}
the morphism $\Z\times_\X T\rarrow T$ is a closed immersion of schemes.
 Obviously, any closed immersion of ind-schemes is an affine morphism.

 In particular, given a scheme $Z$ and an ind-scheme $\X$, a morphism
of ind-schemes $Z\rarrow\X$ is said to be a closed immersion if, for
every scheme $T$ and a morphism of ind-schemes $T\rarrow\X$,
the morphism of schemes $Z\times_\X T\rarrow T$ is a closed immersion.
 In this case, one says that the morphism $Z\rarrow\X$ makes $Z$
a \emph{closed subscheme} in~$\X$.

 Let $\X=\ilim_{\gamma\in\Gamma}X_\gamma$ be an ind-scheme represented
by an inductive system of closed immersions of schemes $X_\gamma\rarrow
X_\delta$, \ $\gamma<\delta\in\Gamma$.
 Let $i\:Z\rarrow\X$ be a morphism into $\X$ from a scheme $Z$, and let
$\gamma\in\Gamma$ be an index such that $i$~factorizes as
$Z\rarrow X_\gamma\rarrow\X$.
 Then the morphism $i$~is a closed immersion if and only if
the morphism $Z\rarrow X_\gamma$ is a closed immersion of schemes.
 This observation provides another way to answer the question from
Remark~\ref{strict-representation-uniqueness}.

 A morphism of ind-schemes $f\:\Y\rarrow\X$ is said to be \emph{flat}
if, for any scheme $T$ and a morphism of ind-schemes $T\rarrow\X$,
the fibered product $\Y\times_\X T$ is a scheme \emph{and}
the morphism of schemes $\Y\times_\X T\rarrow T$ is flat.
 Let $\X=\ilim_{\gamma\in\Gamma}X_\gamma$ be an ind-scheme represented
by an inductive system of closed immersions of schemes.
 Then a morphism $f\:\Y\rarrow\X$ is flat if and only if, for every
$\gamma\in\Gamma$, the fibered product $\Y\times_\X X_\gamma$
is a scheme and the morphism of schemes $\Y\times_\X X_\gamma
\rarrow X_\gamma$ is flat.

\subsection{Ind-affine examples}
 An ind-scheme is said to be \emph{ind-affine} if it can be represented
by an inductive system of affine schemes.
 It follows from Lemma~\ref{closed-immersion-two-morphisms-lemma}(c)
(with an affine scheme $Y$) that any closed subscheme of a strict
ind-affine scheme is affine.
 Thus any (strict) ind-affine ind-scheme can be represented by
an inductive system of closed immersions of affine schemes.

\begin{exs} \label{Z-hat-ind-scheme}
 (1)~Consider the directed poset of all positive integers $\boZ_{>0}$
in the divisibility order.
 To every $n\in\boZ_{>0}$, assign the affine scheme
$X_n=\Spec \boZ/n\boZ$.
 Whenever $m$ divides~$n$, there is a unique (surjective)
ring homomorphism $\boZ/n\boZ\rarrow\boZ/m\boZ$, and accordingly
a unique morphism (closed immersion) of affine schemes $X_m\rarrow X_n$.
 The inductive system of schemes $(X_n)_{n\in\boZ_{>0}}$ represents
a strict ind-affine $\aleph_0$\+ind-scheme $\X$, which we will denote
by $\X=\Spi\widehat\boZ$.
 Here $\widehat\boZ=\prod_p\boZ_p$ is the profinite completion of
the ring of integers, or equivalently, the product of the rings of
$p$-adic integers $\boZ_p$ taken over all the prime numbers~$p$.

\smallskip
 (2)~Choose a prime number~$p$, and consider the directed poset of all
nonnegative integers $\boZ_{\ge0}$ in the usual linear order.
 To every $r\in\boZ_{\ge0}$, assign the affine scheme $X_r=\Spec
\boZ/p^r\boZ$.
 Whenever $r\le s$, there is a unique (surjective) ring homomorphism
$\boZ/p^s\boZ\rarrow\boZ/p^r\boZ$, and accordingly a unique morphism
(closed immersion) of affine schemes $X_r\rarrow X_s$.
 The inductive system of schemes $(X_r)_{r\in\boZ_{\ge0}}$ represents
a strict ind-affine $\aleph_0$\+ind-scheme $\X$, which we will denote
by $\X=\Spi\boZ_p$.

 The ind-scheme $\Spi\widehat\boZ$ is the coproduct (``disjoint union'') 
of the ind-schemes $\Spi\boZ_p$ in the category of ind-schemes, taken
over all the prime numbers~$p$.
\end{exs}

\begin{exs} \label{coalgebra-ind-scheme}
 (1)~Pick a field~$\kk$, and consider the directed poset of all
nonnegative integers $\boZ_{\ge0}$ in the usual linear order.
 To every $r\in\boZ_{\ge0}$, assign the affine scheme $X_r=\Spec
\kk[x]/x^r\kk[x]$.
 Whenever $r\le s$, there is a unique (surjective) homomorphism of
$\kk$\+algebras $\kk[x]/x^s\kk[x]\rarrow\kk[x]/x^r\kk[x]$ taking
the coset $x+x^s\kk[x]$ to the coset $x+x^r\kk[x]$.
 Let $X_r\rarrow X_s$ be the related closed immersion of affine schemes.
 The inductive system of schemes $(X_r)_{r\in\boZ_{\ge0}}$ represents
a strict ind-affine $\aleph_0$\+ind-scheme $\X$, which we will denote
by $\X=\Spi\kk[[x]]$.
 This ind-scheme comes endowed with a morphism of ind-schemes
$\Spi\kk[[x]]\rarrow\Spec\kk$.

 Example~\ref{coalgebra-ind-scheme}(1) is a close analogue of
Example~\ref{Z-hat-ind-scheme}(2).

\smallskip
 (2)~Let $\rC$ be a coassociative, cocommutative, counital coalgebra
over a field~$\kk$.
 Any coassociative coalgebra over a field is the union of its
finite-dimensional subcoalgebras (which form a directed poset by
inclusion).
 All the subcoalgebras of $\rC$ are also coassociative, cocommutative,
and counital.
 Let $\Gamma$ denote the poset of all finite-dimensional subcoalgebras
of $\rC$ in the inclusion order.
 For a finite-dimensional subcoalgebra $\E\subset\rC$ (so $\E\in
\Gamma$), the dual vector space $\E^*$ is an associative, commutative,
and unital finite-dimensional $\kk$\+algebra.

 For every pair of finite-dimensional subcoalgebras $\E'$,
$\E''\subset\rC$ such that $\E'\subset\E''$, the dual map $\E''{}^*
\rarrow\E'{}^*$ to the inclusion $\E'\rarrow\E''$ is a surjective
homomorphism of commutative rings.
 Consider the related closed immersion of affine schemes
$\Spec\E'{}^*\rarrow\Spec\E''{}^*$.
 The inductive system of schemes $(\Spec\E^*)_{\E\in\Gamma}$ represents
a strict ind-affine ind-scheme $\X$, which we will denote by
$\X=\Spi\rC^*$.
 This ind-scheme comes endowed wih a morphism of ind-schemes
$\Spi\rC^*\rarrow\Spec\kk$.

 Here the notation $\rC^*$ stands for the associative, commutative, and
unital topological $\kk$\+algebra $\rC^*=\varprojlim_{\E\in\Gamma}\E^*$,
with the topology of projective limit of discrete finite-dimensional
vector spaces/algebras~$\E^*$.
 Example~\ref{coalgebra-ind-scheme}(1) is the particular case of
Example~\ref{coalgebra-ind-scheme}(2) corresponding to the choice
of the coalgebra $\rC$ such that $\rC^*=\kk[[x]]$ (with the $x$\+adic
topology on~$\kk[[x]])$.
 The coalgebra $\rC$ is a $\kk$\+vector space with the basis $\{1^*,
x^*,x^2{}^*,\dotsc,x^n{}^*,\dotsc\}$, \ $n\in\boZ_{\ge0}$, with
the counit map $\rC\ni1^*\longmapsto 1\in\kk$, \ $x^n{}^*\longmapsto0$
for $n>0$, and the comultiplication map $\rC\ni x^n{}^*\longmapsto
\sum_{p+q=n} x^p{}^*\ot x^q{}^*\in\rC\ot_\kk\rC$.
\end{exs}

\begin{exs} \label{topological-ring-ind-scheme}
 This example is taken from~\cite[Example~7.11.2(i)]{BD2}.

 (1)~Let $\fR$ be an associative, commutative, unital ring endowed with
a complete, separated topology with a base of neighborhoods of zero
consisting of open ideals.
 For every open ideal $\fI\subset\fR$, consider the quotient ring
$\fR/\fI$ and the affine scheme $\Spec\fR/\fI$.
 Whenever $\fI$, $\fJ\subset\fR$ are open ideals such that
$\fJ\subset\fI$, there exists a unique (surjective) morphism of
commutative rings $\fR/\fJ\rarrow\fR/\fI$ making the triangle diagram
$\fR\rarrow\fR/\fJ\rarrow\fR/\fI$ commutative.
 Let $\Spec\fR/\fI\rarrow\Spec\fR/\fJ$ be the related closed immersion
of affine schemes.
 The inductive system of schemes $(\Spec\fR/\fI)_{\fI\subset\fR}$
indexed by the directed poset of open ideals in $\fR$ with the reverse
inclusion order represents a strict ind-affine ind-scheme $\X$,
which we will denote by $\Spi\fR$.
 For any base of neighborhoods of zero $B$ consisting of open ideals
in $\fR$, one has $\Spi\fR=\ilim_{\fI\in B}\Spec\fR/\fI$ (where $B$ is
viewed as a directed poset in the reverse inclusion order).

\smallskip
 (2)~In particular, assume in the context of~(1) that $\fR$ has
a \emph{countable} base of neighborhoods of zero.
 Then one can choose a countable base of neighborhoods of zero $B$
consisting of open ideals in $\fR$, hence $\Spi\fR=\ilim_{\fI\in B}
\Spi\fR/\fI$ is an ind-affine $\aleph_0$\+ind-scheme.
 The functor $\fR\longmapsto\Spi\fR$ establishes an anti-equivalence
between the category of topological commutative rings with a countable
neighborhood of zero consisting of open ideals and the category of
ind-affine $\aleph_0$\+ind-schemes.

\smallskip
 In the more general context of~(1), the assignment
$\fR\longmapsto\Spi\fR$ is a \emph{fully faithful} contravariant
functor from the category of topological commutative rings where open
ideals form a base of neighborhoods of zero to the category of
ind-affine ind-schemes.
 Its essential image consists of all the ind-affine ind-schemes
$\ilim_{\gamma\in\Gamma}\Spec R_\gamma$ for which the projection map
$\varprojlim_{\gamma\in\Gamma}R_\gamma\rarrow R_\delta$ is surjective
for all $\delta\in\Gamma$.

 Examples~\ref{Z-hat-ind-scheme} and~\ref{coalgebra-ind-scheme}(1)
are particular cases of Example~\ref{topological-ring-ind-scheme}(2).
 In particular, the ind-affine ind-scheme $\Spi\widehat\boZ$ in
Example~\ref{Z-hat-ind-scheme}(1) corresponds to the topological
ring $\fR=\widehat\boZ$ with its profinite topology, while
the ind-affine ind-scheme $\Spi\boZ_p$ in
Example~\ref{Z-hat-ind-scheme}(2) corresponds to the topological
ring $\fR=\widehat\boZ_p$ with its $p$\+adic topology.
 The ind-affine ind-scheme $\Spi\kk[[x]]$ in
Example~\ref{coalgebra-ind-scheme}(1) corresponds to the topological
ring or $\kk$\+algebra $\fR=\kk[[x]]$ with its $x$\+adic topology.

 Example~\ref{coalgebra-ind-scheme}(2) is a particular case of
Example~\ref{topological-ring-ind-scheme}(1) corresponding to
the topological ring $\fR=\rC^*$ with its profinite-dimensional
(linearly compact) topology.
 The ind-affine ind-scheme $\Spec\rC^*$ in
Example~\ref{coalgebra-ind-scheme}(2) corresponds to a topological
ring with a countable base of neighborhoods of zero (i.~e., $\rC^*$
has a countable base of neighborhoods of zero) if and only if
$\Spi\rC^*$ is an $\aleph_0$\+ind-scheme, and if and only if
the underlying vector space of the coalgebra $\rC$ has at most
countable dimension over~$\kk$.
\end{exs}

 For a further discussion of specific examples of ind-affine
ind-schemes, see Section~\ref{tate-space-subsecn} below.

\Section{Quasi-Coherent Torsion Sheaves} \label{torsion-sheaves-secn}

 In this section and below, as mentioned in
Sections~\ref{ind-schemes-subsecn}\+-%
\ref{morphisms-of-ind-schemes-subsecn}, all the \emph{schemes} are
concentrated, and all the \emph{ind-schemes} are ind-concentrated
and strict.

\subsection{Reasonable ind-schemes} \label{reasonable-subsecn}
 Here we largely follow~\cite[Section~7.11.1]{BD2}.

 Let $X$ be a concentrated scheme with the structure sheaf $\cO_X$.
 Notice that any closed subscheme of a concentrated scheme is
concentrated.
 Let $Z\subset X$ be a closed subscheme, and let $\I_{Z,X}\subset\cO_X$
denote the quasi-coherent sheaf of ideals in $\cO_X$ corresponding
to~$Z$.
 So $\I_{Z,X}$ is the kernel of the natural surjective morphism of
quasi-coherent sheaves $\cO_X\rarrow k_*\cO_Z$, where $k\:Z\rarrow X$
is the closed immersion.
 The closed subscheme $Z\subset X$ is said to be \emph{reasonable} if
$\I_{Z,X}$ is generated (as a quasi-coherent sheaf on~$X$) by a finite
set of local sections.

 Let $Y\subset X$ be a closed subscheme such that $Z\subset Y\subset X$.
 Then $Z$ is also a closed subscheme in~$Y$.
 Denote the closed immersion morphism by $i\:Y\rarrow X$.
 Then the natural surjective morphism $\cO_X\rarrow i_*\cO_Y$
of quasi-coherent sheaves on~$X$ restricts to a surjective morphism
$\I_{Z,X}\rarrow i_*\I_{Z,Y}$.
 It follows that $Z$ is a reasonable closed subscheme in $Y$ whenever
$Z$ is a reasonable closed subcheme in~$X$.

 Part~(a) of the following lemma is more general.
 
\begin{lem} \label{base-change-composition-reasonable}
\textup{(a)} Let $Y\rarrow X$ be a morphism of schemes and $Z\subset X$
be a reasonable closed subscheme.
 Then $Z\times_XY$ is a reasonable closed subscheme in~$Y$. \par
\textup{(b)} Let $Y$ and $Z$ be closed subschemes in a scheme $X$
such that $Z\subset Y\subset X$.
 Assume that $Z$ is a reasonable closed subscheme in $Y$, and $Y$ is
a reasonable closed subscheme in~$X$.
 Then $Z$ is a reasonable closed subscheme in~$X$.
\end{lem}

\begin{proof}
 Both the assertions are essentially local (as the schemes are presumed
to be concentrated) and reduce to the affine case.
 Part~(a): let $R\rarrow S$ be a homomorphism of commutative rings and
$I\subset R$ be a finitely generated ideal.
 Then the claim is that the kernel of the surjective map
$S\rarrow R/I\ot_RS$ is a finitely generated ideal in~$S$.
 Part~(b): let $R\rarrow S\rarrow T$ be surjective homomorphisms of
commutative rings such that the kernel of $R\rarrow S$ is a finitely
generated ideal in $R$ and the kernel of $S\rarrow T$ is a finitely
generated ideal in~$T$.
 Then the claim is that the kernel of $R\rarrow T$ is a finitely
generated ideal in~$R$. 
\end{proof}

 Let $\X=\ilim_{\gamma\in\Gamma}X_\gamma$ be an ind-scheme represented
by an inductive system of closed immersions of schemes.
 A closed subscheme $Z\subset\X$ is said to be \emph{reasonable} if,
for every closed subscheme $Y\subset\X$ such that $Z\subset Y$,
the closed subcheme $Z$ in $Y$ is reasonable.
 A closed subscheme $Z\subset\X$ is reasonable if and only if, for
every index $\gamma\in\Gamma$ such that $Z\subset X_\gamma$,
the closed subcheme $Z$ in $X_\gamma$ is reasonable.

 An ind-scheme $\X$ is said to be \emph{reasonable} if it is a filtered
direct limit of its reasonable closed subschemes.
 Equivalently, $\X$ is reasonable if and only if there exists
a (filtered) inductive system of closed immersions of schemes
$(X_\gamma)_{\gamma\in\Gamma}$ such that, for every $\gamma<\delta\in
\Gamma$, the closed subscheme $X_\gamma$ in $X_\delta$ is reasonable.
 Clearly, in the latter case, $X_\gamma$ is a reasonable closed
subscheme in $\X$ for every $\gamma\in\Gamma$.

 Let $\Y\rarrow\X$ be a morphism of ind-schemes which is
``representable by schemes'' in the sense of
Section~\ref{morphisms-of-ind-schemes-subsecn}.
 Then it follows from Lemma~\ref{base-change-composition-reasonable}(a)
that, for any reasonable closed subscheme $Z\subset\X$, the fibered
product $Z\times_\X\Y$ is a reasonable closed subscheme in~$\Y$.
 Therefore, the ind-scheme $\Y$ is reasonable if the ind-scheme $\X$ is.

\subsection{Quasi-coherent sheaves and functors}
\label{qcoh-sheaves-and-functors-subsecn}
 Given a scheme $X$, we denote by $X\qcoh$ the abelian (Grothendieck)
category of quasi-coherent sheaves on~$X$.
 For every morphism of (concentrated) schemes $f\:Y\rarrow X$, we
have the direct and inverse image functors $f_*\:Y\qcoh\rarrow X\qcoh$
and $f^*\:X\qcoh\rarrow Y\qcoh$; the functor $f^*$ is left adjoint
to~$f_*$.
 The functors $f_*$ and $f^*$ are the restrictions (to the full
subcategories of quasi-coherent sheaves) of the similar functors
acting between the ambient abelian categories of sheaves of
$\cO_X$\+modules and sheaves of $\cO_Y$\+modules.

 The category $X\qcoh$ is a tensor subcategory of the (associative,
commutative, and unital) tensor category of sheaves of $\cO_X$\+modules,
with respect to the tensor product functor ${-}\ot_{\cO_X}{-}$.
 The structure sheaf $\cO_X$ is the unit object.
 The inverse image $f^*\:X\qcoh\rarrow Y\qcoh$ is a tensor functor.
 The next lemma is very well-known; it is called the ``projection
formula''.

\begin{lem}  \label{projection-formula}
 Let $f\:Y\rarrow X$ be a morphism of (concentrated) schemes, $\M$ be
a quasi-coherent sheaf on $X$, and $\N$ be a quasi-coherent sheaf
on~$Y$.
 Then there is a natural morphism of quasi-coherent sheaves on~$X$
$$
 \M\ot_{\cO_X}f_*\N\lrarrow f_*(f^*\M\ot_{\cO_Y}\N),
$$
which is an isomorphism if the morphism~$f$ is affine.
\end{lem}

\begin{proof}
 The morphism in question is adjoint to the morphism
$f^*(\M\ot_{\cO_X}f_*\N)\simeq f^*\M\ot_{\cO_Y}f^*f_*\N\rarrow
f^*\M\ot_{\cO_Y}\N$ induced by the adjunction morphism $f^*f_*\N
\rarrow\N$ in $Y\qcoh$.
 The second assertion is local in $X$, so it reduces to the case of
affine schemes, for which it means the following.
 Let $R\rarrow S$ be a commutative ring homomorphism, $M$ be
an $R$\+module, and $N$ be an $S$\+module.
 Then there is a natural isomorphism of $R$\+modules $M\ot_RN\simeq
(S\ot_RM)\ot_SN$.
\end{proof}

 For any two sheaves of $\cO_X$\+modules $\M$ and $\N$, we denote by
$\cHom_{\cO_X}(\M,\N)$ the sheaf of $\cO_X$\+modules with the modules
of sections $\cHom_{\cO_X}(\M,\N)(U)=\Hom_{\cO_U}(\M|_U,\N|_U)$ for
all open subschemes $U\subset X$.
 The sheaf of $\cO_X$\+modules $\Hom_{\cO_X}(\M,\N)$ is quasi-coherent
whenever the sheaf $\N$ is quasi-coherent and the sheaf $\M$ is
locally (i.~e., in restriction to a small enough Zariski neighborhood
of every point of~$X$) the cokernel of a morphism between finite direct
sums of copies of the structure sheaf~$\cO$.

 Let $Z\subset X$ be a closed subscheme with the closed immersion
morphism $i\:Z\rarrow X$.
 For any sheaf of $\cO_X$\+modules $\M$, denote by $i^!\M$ the sheaf
of $\cO_Z$\+modules defined by the property that $i_*i^!\M=
\cHom_{\cO_X}(i_*\cO_Z,\M)$ is the subsheaf of $\M$ consisting of all
the local sections annihilated by~$\I_{Z,X}$.
 The sheaf of $\cO_Z$\+modules $i^!\M$ is quasi-coherent whenever
the sheaf of $\cO_X$\+modules $\M$ is \emph{and} the closed subscheme
$Z\subset X$ is reasonable.
 In this case, the functor $i^!\:X\qcoh\rarrow Z\qcoh$ is right
adjoint to the direct image functor $i_*\:Z\qcoh\rarrow X\qcoh$.

\begin{lem}  \label{reasonable-base-change}
 Let $f\:Y\rarrow X$ be a morphism of (concentrated) schemes and
$Z\subset X$ be a reasonable closed subscheme with the closed
immersion $i\:Z\rarrow X$.
 Consider the pullback diagram
$$
\xymatrix{
 Z\times_XY \ar[r]^-k \ar[d]^-g & Y \ar[d]^-f \\
 Z \ar[r]^-i & X
}
$$
 Then there are natural isomorphisms \par
\textup{(a)} $i^!f_*\simeq g_*k^!$ of functors $Y\qcoh\rarrow Z\qcoh$;
\par
\textup{(b)} $f^*i_*\simeq k_*g^*$ of functors $Z\qcoh\rarrow Y\qcoh$.
\end{lem}

\begin{proof}
 Parts~(a) and~(b) are adjoint to each other, so it suffices to check
any one of them.
 Both the assertions (particularly clearly~(b))
are essentially local and reduce to the case of affine schemes, for
which they mean the following.
 Let $R\rarrow S$ be a commutative ring homomorphism and
$I\subset R$ be a finitely generated ideal.
 Then (a)~for any $S$\+module $N$, there is a natural isomorphism of
$R/I$\+modules $\Hom_R(R/I,N)\simeq\Hom_S(S/IS,N)$; and
(b)~for any $R/I$\+module $M$, there is a natural isomorphism of
$S$\+modules $S\ot_RM\simeq S/IS\ot_{R/I}M$.

 The assumption of the closed subscheme $Z\subset X$ being reasonable
is only needed in part~(a).
 For a further generalization of part~(b), see
Lemma~\ref{affine-flat-base-change}(a) below.
\end{proof}

 More generally, for an arbitrary closed immersion of schemes
$i\:Z\rarrow X$, the direct image functor $i_*\:Z\qcoh\rarrow X\qcoh$
has a right adjoint (whose existence follows already from the facts
that $Z\qcoh$ is a Grothendieck abelian category and $i_*$~is an exact
functor preserving coproducts).
 This functor, which could be properly called
the ``quasi-coherent~$i^!$'', can be constructed by applying
the \emph{coherator} functor~\cite[Sections~B.12--B.14]{TT} to
the ``$\cO$\+module~$i^!$'' mentioned above.
 In fact, a more specific description is
available~\cite[Tag~01R0]{SP}.

 In the sequel, the notation~$i^!$ for a closed immersion~$i$ will
always stand for the ``quasi-coherent~$i^!$''.
 For the closed immersion of a reasonable closed subscheme
$i\:Z\rarrow X$, the ``$\cO$\+module~$i^!$'' and
the ``quasi-coherent~$i^!$'' agree.

 Notice that, for any closed immersion~$i$, the direct image
functor~$i_*$ is fully faithful; so the adjunction morphism
$\N\rarrow i^!i_*\N$ is an isomorphism for all $\N\in Z\qcoh$.
 The adjunction morphism $i_*i^!\M\rarrow\M$ is a monomorphism for
all $\M\in X\qcoh$ (cf.\
Lemma~\ref{closed-under-subquotients-characterized} below).

\subsection{Quasi-coherent torsion sheaves}
\label{qcoh-torsion-sheaves-subsecn}
 We follow~\cite[Sections~7.11.3--4]{BD2}.
 Let $\X$ be a reasonable ind-scheme.
 A \emph{quasi-coherent torsion sheaf} $\rM$ on $\X$ (called
an ``$\cO^!$\+module'' in~\cite{BD2}) is the following set of data:
\begin{enumerate}
\renewcommand{\theenumi}{\roman{enumi}}
\item to every reasonable closed subscheme $Y\subset\X$,
a quasi-coherent sheaf $\rM_{(Y)}$ on $Y$ is assigned;
\item to every pair of reasonable closed subschemes $Y$, $Z\subset\X$,
\ $Z\subset Y$ with the closed immersion morphism $i_{ZY}\:Z\rarrow Y$,
a morphism $i_{ZY}{}_*\rM_{(Z)}\rarrow\rM_{(Y)}$ of quasi-coherent
sheaves on $Y$ is assigned;
\item such that the corresponding morphism $\rM_{(Z)}\rarrow
i_{ZY}^!\rM_{(Y)}$ of quasi-coherent sheaves on $Z$ is an isomorphism;
\item and, for every triple of reasonable closed subschemes
$Y$, $Z$, $W\subset\X$, \ $W\subset Z\subset Y$, the triangle diagram
$i_{WY}{}_*\rM_{(W)}\rarrow i_{ZY}{}_*\rM_{(Z)}\rarrow\rM_{(Y)}$ is
commutative in $Y\qcoh$.
\end{enumerate}

 Let $\X=\ilim_{\gamma\in\Gamma}X_\gamma$ be a representation of $\X$
by an inductive system of reasonable closed subschemes.
 Then, in order to construct a quasi-coherent torsion sheaf $\rM$ on
$\X$, it suffices to specify the quasi-coherent sheaves
$\rM_{(X_\gamma)}\in X_\gamma\qcoh$ for every $\gamma\in\Gamma$
and the morphisms $i_{X_\gamma X_\delta}{}_*\rM_{(X_\gamma)}
\rarrow\rM_{(X_\delta)}$ for every $\gamma<\delta\in\Gamma$
satisfying conditions~(iii\+-iv) for $W=X_\beta$, $Z=X_\gamma$,
$Y=X_\delta$, \ $\beta<\gamma<\delta\in\Gamma$.
 The quasi-coherent sheaves $\rM_{(Y)}$ for all the other reasonable
closed subschemes $Y\subset\X$ and the related morphisms~(ii) can
then be uniquely recovered so that conditions~(iii\+-iv) are satisfied
for all reasonable closed subschemes in~$\X$.

 \emph{Morphisms of quasi-coherent torsion sheaves} $f\:\rM\rarrow\rN$
on $\X$ are defined in the obvious way.
 We denote the category of quasi-coherent torsion sheaves on $\X$ by
$\X\tors$.
 In the rest of Section~\ref{torsion-sheaves-secn}, our main aim is
to prove the following theorem.

\begin{thm} \label{torsion-sheaves-abelian}
 For any reasonable strict ind-concentrated ind-scheme\/ $\X$,
the category of quasi-coherent torsion sheaves\/ $\X\tors$ is
a Grothendieck abelian category.
\end{thm}

 The proof of Theorem~\ref{torsion-sheaves-abelian} will be given at
the end of Section~\ref{Gamma-systems-subsecn}.

\subsection{Ind-affine examples} \label{torsion-ind-affine-subsecn}
 (1)~Let $\X=\Spi\widehat\boZ$ be the ind-affine ind-scheme from
Example~\ref{Z-hat-ind-scheme}(1).
 Then $\X$ is a reasonable ind-scheme, and the category $\X\tors$ is
equivalent to the category of torsion abelian groups.

 Indeed, the closed subschemes of $\Spi\widehat\boZ$ are precisely
the schemes $X_n=\Spec\boZ/n\boZ$.
 All of them are reasonable.
 Let $M$ be a torsion abelian group; then the corresponding
quasi-coherent torsion sheaf $\rM$ on $\Spi\widehat\boZ$ is defined by
the rule that $\rM_{(X_n)}\in X_n\qcoh$ is the quasi-coherent sheaf
corresponding to the $\boZ/n\boZ$\+module $M_{(n)}\subset M$ of
all the elements annihilated by~$n$ in~$M$.
 The morphisms of quasi-coherent sheaves $i_{ZY}{}_*\rM_{(Z)}
\rarrow\rM_{(Y)}$ correspond to the inclusion maps of abelian groups
$M_{(m)}\rarrow M_{(n)}$ for $m$ dividing~$n$.
 Conversely, given a quasi-coherent torsion sheaf $\rM$ on
$\Spi\widehat\boZ$, the related torsion abelian group $M$ is
the direct limit $M=\varinjlim_{n\in\boZ_{>0}}\rM_{(X_n)}(X_n)$.

\smallskip
 (2)~Let $\X=\Spi\boZ_p$ be the ind-affine ind-scheme from
Example~\ref{Z-hat-ind-scheme}(2).
 Then $\X$ is a reasonable ind-scheme, and the category $\X\tors$ is
equivalent to the category of $p$\+primary (torsion) abelian groups.

 Indeed, the closed subschemes of $\Spi\widehat\boZ$ are precisely
the schemes $X_r=\Spec\boZ/p^r\boZ$.
 All of them are reasonable.
 Let $M$ be a $p$\+primary abelian group; then the corresponding
quasi-coherent torsion sheaf $\rM$ on $\Spi\boZ_p$ is defined by
the rule that $\rM_{(X_r)}\in X_r\qcoh$ is the quasi-coherent sheaf
corresponding to the $\boZ/p^r\boZ$\+module $M_{(p^r)}\subset M$ of
all the elements annihilated by~$p^r$ in~$M$.
 Conversely, given a quasi-coherent torsion sheaf $\rM$ on
$\Spi\boZ_p$, the related $p$\+primary abelian group $M$ is
the direct limit $M=\varinjlim_{r\in\boZ_{\ge0}}\rM_{(X_r)}(X_r)$.

\smallskip
 (3)~Let $\X=\Spi\kk[[x]]$ be the ind-affine ind-scheme from
Example~\ref{coalgebra-ind-scheme}(1).
 Then, similarly to~(2), one can describe the closed subschemes in
$\X$ and see that all of them are reasonable.
 The category $\X\tors$ is equivalent to the category of $x$\+primary
torsion $\kk[x]$\+modules, i.~e., the full subcategory in the abelian
category of all $\kk[x]$\+modules consisting of all
the $\kk[x]$\+modules $M$ such that for every element $b\in M$ there
exists an integer $r\ge1$ for which $x^rb=0$ in~$M$.

\smallskip
 (4)~Let $\X=\Spi\rC^*$ be the ind-affine ind-scheme from
Example~\ref{coalgebra-ind-scheme}(2).
 Then $\X$ is a reasonable ind-scheme, and the category $\X\tors$ is
equivalent to the category of comodules over the coalgebra~$\rC$.

 Indeed, the closed subchemes of $\Spi\rC^*$ are precisely
the schemes $X_\E=\Spec\E^*$, where $\E\subset\rC$ are
the finite-dimensional subcoalgebras.
 All the closed subschemes in $\X$ are reasonable.
 For any $\rC$\+comodule $M$, denote by $M_{(\E)}\subset M$
the maximal $\rC$\+subcomodule of $M$ whose $\rC$\+comodule structure
comes from an $\E$\+comodule structure.
 Simply put, $M_{(\E)}$ is the kernel of the composition
$M\rarrow\rC\ot_\kk M\rarrow\rC/\E\ot_\kk M$ of the coaction map
$M\rarrow\rC\ot_\kk M$ with the map $\rC\ot_\kk M\rarrow\rC/\E\ot_\kk M$
induced by the natural surjection $\rC\rarrow\rC/\E$.

 Notice that the category of $\E$\+comodules is naturally equivalent
to the category of $\E^*$\+modules.
 The quasi-coherent torsion sheaf $\rM$ on $\Spi\rC^*$ corresponding
to a $\rC$\+comodule $M$ is defined by the rule that
$\rM_{(X_\E)}\in X_\E\qcoh$ is the quasi-coherent sheaf corresponding
to the $\E^*$\+module~$M_{(\E)}$.
 The morphisms of quasi-coherent sheaves $i_{X_{\E'}X_{E''}}{}_*
\rM_{(X_{\E'})}\rarrow\rM_{(X_{\E''})}$ for $\E'\subset\E''$ correspond
to the inclusion maps $M_{(\E')}\rarrow M_{(\E'')}$.
 Conversely, given a quasi-coherent torsion sheaf $\rM$ on $\Spi\rC^*$,
the related $\rC$\+comodule $M$ is the direct limit
$M=\varinjlim_{\E\subset\rC}\rM_{(X_{\E})}(X_{\E})$.

\smallskip
 (5)~Let $\fR$ be a complete, separated topological commutative ring
where open ideals form a base of neighborhoods of zero, as in
Example~\ref{topological-ring-ind-scheme}(1).
 We will say that an open ideal $\fI\subset\fR$ is \emph{reasonable}
(with respect to the given topology on~$\fR$) if, for any open ideal
$\fJ\subset\fR$ such that $\fJ\subset\fI$, the kernel of the natural
surjective ring homomorphism $\fR/\fJ\rarrow\fR/\fI$ is a finitely
generated ideal in the discrete ring~$\fR/\fJ$.

 A topological ring $\fR$ is said to be \emph{reasonable} if reasonable
open ideals form a base of neighborhoods of zero in~$\fR$.
 Equivalently, $\fR$ is reasonable if and only if there exists
a base of neighborhoods of zero $B$ consisting of open ideals in $\fR$
such that for all $\fJ\subset\fI\in B$ the kernel of the natural
surjective ring homomorphism $\fR/\fJ\rarrow\fR/\fI$ is a finitely
generated ideal in~$\fR/\fJ$.

 Let $\X=\Spi\fR$ be the ind-affine ind-scheme from
Example~\ref{topological-ring-ind-scheme}(1).
 Then $\X$ is a reasonable ind-scheme if and only if $\fR$ is
a reasonable topological ring.

\smallskip
 (6)~Let $\fR$ be a reasonable topological commutative ring; so
$\Spi\fR$ is a reasonable ind-scheme.
 Then the category $\X\tors$ is equivalent to the category $\fR\discr$
of \emph{discrete\/ $\fR$\+modules}.
 This means the full subcategory $\fR\discr\subset\fR\modl$
in the category of abelian category of $\fR$\+modules $\fR\modl$
consisting of all the $\fR$\+modules $M$ such that for every
element $b\in M$ the annihilator of~$b$ is an open ideal in~$\fR$.

 Indeed, the closed subschemes of $\Spi\fR$ are precisely the schemes
$X_\fI=\Spec\fR/\fI$, where $\fI\subset\fR$ ranges over the open ideals.
 Let us say that $\fI\subset\fR$ is a \emph{reasonable open ideal} if
the closed subscheme $X_\fI\subset\X$ is reasonable.
 This means that, for every open ideal $\fJ\subset\fI\subset\fR$,
the kernel of the map $\fR/\fJ\rarrow\fR/\fI$ is a finitely
generated ideal.

 For any discrete $\fR$\+module $M$, denote by $M_{(\fI)}\subset M$
the submodule of all elements annihilated by~$\fI$.
 Then the quasi-coherent torsion sheaf $\rM$ on $\Spi\fR$ corresponding
to $M$ is defined by the rule that, for any reasonable open ideal
$\fI\subset\fR$, the quasi-coherent torsion sheaf $\rM_{(X_\fI)}\in
X_\fI\qcoh$ corresponds to the $\fR/\fI$\+module $M_{(\fI)}$.
 The morphisms of quasi-coherent sheaves $i_{X_\fI X_\fJ}{}_*
\rM_{(X_\fI)}\rarrow\rM_{(X_\fJ)}$ for reasonable open ideals
$\fJ\subset\fI\subset\fR$ correspond to the inclusion maps
$M_{(\fJ)}\rarrow M_{(\fI)}$.
 Conversely, given a quasi-coherent torsion sheaf $\rM$ on $\Spi\fR$,
the related discrete $\fR$\+module $M$ is the direct limit
$M=\varinjlim_{\fI\subset\fR}\rM_{(X_\fI)}(X_\fI)$ taken over
the directed poset of reasonable open ideals $\fI\subset\fR$
(in the reverse inclusion order).

\smallskip
 The above examples explain the terminology ``quasi-coherent torsion
sheaves''.
 One can also observe that, while in every one of the examples~(1--5)
the category $\X\tors$ is indeed abelian as
Theorem~\ref{torsion-sheaves-abelian} claims, the forgetful functors
$\X\tors\rarrow Z\qcoh$ assigning to a quasi-coherent torsion sheaf
$\rM$ on $\X$ the quasi-coherent sheaf $\rM_{(Z)}\in Z\qcoh$ for
a reasonable closed subscheme $Z\subset\X$ are \emph{not} exact.
 In fact, the functors $\rM\longmapsto\rM_{(Z)}\:\X\tors\rarrow Z\qcoh$
are left, but not necessarily right exact (see
Section~\ref{torsion-inverse-images-subsecn} below 
for a further discussion).
 This explains why Theorem~\ref{torsion-sheaves-abelian} is nontrivial
and its proof is not straightforward.

\subsection{Direct limits} \label{torsion-direct-limits-subsecn}
 Recall that the functor of global sections of quasi-coherent sheaves
over a concentrated scheme preserves (filtered) direct limits.
 It follows that so does the direct image functor~$f_*$ for a morphism
of concentrated schemes $f\:Y\rarrow X$.
 The inverse image functor~$f^*$, being a left adjoint, obviously
preserves direct limits.
 Furthermore, for any reasonable closed subscheme $Z\subset X$ with
the closed immersion morphism $i\:Z\rarrow X$, the functor $i^!\:
X\qcoh\rarrow Z\qcoh$ preserves direct limits.

 Let $\X$ be a reasonable ind-scheme and
$(\rM_\theta)_{\theta\in\Theta}$ be an inductive system of
quasi-coherent torsion sheaves on $\X$, indexed by a directed
poset~$\Theta$.
 For every reasonable closed subscheme $Z\subset\X$, put
$\rM_{(Z)}=\varinjlim_{\theta\in\Theta}(\rM_\theta)_{(Z)}$
(where the direct limit is taken in the category of quasi-coherent
sheaves on~$Z$).
 Then the collection of quasi-coherent sheaves $\rM_{(Z)}$ with
the obvious maps $i_{ZY}{}_*\rM_{(Z)}\rarrow\rM_{(Y)}$ for
$Z\subset Y\subset\X$ is a quasi-coherent torsion sheaf $\rM$ on~$\X$
(as one can see from the previous paragraph).
 One has $\rM=\varinjlim_{\theta\in\Theta}\rM_\theta$ in
the category $\X\tors$.

\subsection{Direct images} \label{torsion-direct-images-subsecn}
 Let $\X$ be a reasonable ind-scheme and $f\:\Y\rarrow\X$ be a morphism
of ind-schemes which is ``representable by schemes''.
 According to Section~\ref{reasonable-subsecn}, the ind-scheme $\Y$ is
also reasonable.

 Let $\rN$ be a quasi-coherent torsion sheaf on~$\Y$.
 For every reasonable closed subscheme $Z\subset\X$, put
$\rM_{(Z)}=f_Z{}_*(\rN_{(W)})\in Z\qcoh$, where $f_Z$~is the morphism
$W=Z\times_\X\Y\rarrow Z$ and $f_Z{}_*\:W\qcoh\rarrow Z\qcoh$ is
the direct image functor of quasi-coherent sheaves.
 Then it is clear from Lemma~\ref{reasonable-base-change}(a) that
the collection of quasi-coherent sheaves $\rM_{(Z)}$ with the natural
maps $i_{Z'Z''}{}_*\rM_{(Z')}\rarrow\rM_{(Z'')}$ for $Z'\subset Z''
\subset\X$ is a quasi-coherent torsion sheaf $\rM$ on~$\X$.
 
 Put $f_*\rN=\rM$.
 This construction defines the functor of \emph{direct image of
quasi-coherent torsion sheaves} $f_*\:\Y\tors\rarrow\X\tors$.

\subsection{$\Gamma$-systems} \label{Gamma-systems-subsecn}
 Let $\X=\ilim_{\gamma\in\Gamma}X_\gamma$ be a reasonable ind-scheme
represented by an inductive system of closed immersions of reasonable
closed subschemes.
 The definition of what we will call a \emph{$\Gamma$\+system} on $\X$
(which is a shorthand for ``$(X_\gamma)_{\gamma\in\Gamma}$\+system
of quasi-coherent sheaves'') is obtained from the definition of
a quasi-coherent torsion sheaf in
Section~\ref{qcoh-torsion-sheaves-subsecn} by restricting
the reasonable subschemes under consideration to those belonging
to the inductive system $(X_\gamma)_{\gamma\in\Gamma}$ \emph{and}
dropping the condition~(iii).

 In other words, a $\Gamma$\+system $\boM$ on $\X$ is the following
set of data:
\begin{enumerate}
\renewcommand{\theenumi}{\roman{enumi}}
\item to every index $\gamma\in\Gamma$, a quasi-coherent sheaf
$\boM_{(\gamma)}$ on $X_\gamma$ is assigned;
\item to every pair of indices $\gamma<\delta\in\Gamma$ with
the related transition morphism $i_{\gamma\delta}\:X_\gamma\rarrow
X_\delta$, a morphism $i_{\gamma\delta}{}_*\boM_{(\gamma)}\rarrow
\boM_{(\delta)}$ of quasi-coherent sheaves on $X_\delta$, or
equivalently, a morphism $\boM_{(\gamma)}\rarrow i_{\gamma\delta}^!
\boM_{(\delta)}$ of quasi-coherent sheaves on $X_\gamma$, is assigned;
\setcounter{enumi}{3}
\item such that for every triple of indices $\beta<\gamma<\delta$,
the triangle diagram $i_{\beta\delta}{}_*\boM_{(\beta)}\rarrow
i_{\gamma\delta}{}_*\boM_{(\gamma)}\rarrow\boM_\delta$ is commutative
in $X_\delta\qcoh$.
\end{enumerate}

 \emph{Morphisms of $\Gamma$\+systems} $f\:\boM\rarrow\boN$ on $\X$
are defined in the obvious way.
 We denote the category of $\Gamma$\+systems on $\X$ by
$(\X,\Gamma)\syst$.

 For every $\gamma\in\Gamma$, denote by $i_\gamma\:X_\gamma\rarrow\X$
the natural closed immersion.
 As a particular case of the construction of
Section~\ref{torsion-direct-images-subsecn}, we have the direct image
functor $i_\gamma{}_*\:X_\gamma\qcoh=X_\gamma\tors\rarrow\X\tors$.
 For every $\Gamma$\+system $\boM$ on $\X$, we put
$$
 \boM^+=\varinjlim\nolimits_{\gamma\in\Gamma}
 i_\gamma{}_*\boM_{(\gamma)} \,\in\, \X\tors,
$$
where $\boM_{(\gamma)}\in X_\gamma\qcoh$, \
$i_\gamma{}_*\boM_{(\gamma)}\in\X\tors$, and the direct limit is taken
in the category $\X\tors$.
 Notice that, according to Section~\ref{torsion-direct-limits-subsecn},
all (filtered) direct limits exist in $\X\tors$.

 The quasi-coherent torsion sheaf $\boM^+$ on $\X$ can be described more
explicitly as follows.
 For any reasonable closed subscheme $Z\subset\X$, one has
$$
 (\boM^+)_{(Z)}=\varinjlim\nolimits_{\gamma\in\Gamma: Z\subset X_\gamma}
 i_{Z,\gamma}^!\boM_{(\gamma)},
$$
where the direct limit in $Z\qcoh$ is taken over the cofinal subset
of all $\gamma\in\Gamma$ such that $Z\subset X_\gamma$, and
$i_{Z,\gamma}\:Z\rarrow X_\gamma$ is the closed immersion.

 Conversely, given a quasi-coherent torsion sheaf $\rM$ on $\X$,
the rule $\boM_{(\gamma)}=\rM_{(X_\gamma)}$ defines a $\Gamma$\+system
$\boM$ on $\X$, which we will denote by $\rM|_\Gamma=\boM$.

\begin{lem} \label{torsion-sheaves-Gamma-systems-adjunction}
 The functor\/ $\boM\longmapsto\boM^+\:(\X,\Gamma)\syst\rarrow
\X\tors$ is left adjoint to the functor $\rM\longmapsto\rM|_\Gamma\:
\X\tors\rarrow(\X,\Gamma)\syst$.
\end{lem}

\begin{proof}
 Let $\boM$ be a $\Gamma$\+system and $\rN$ be a quasi-coherent torsion
sheaf on~$\X$.
 Then the abelian group of morphisms $\boM^+\rarrow\rN$ in $\X\tors$ is
isomorphic to the group of all compatible collections of morphisms
$i_{Z,\gamma}^!\boM_{(\gamma)}\rarrow\rN_{(Z)}$ in $Z\qcoh$, defined
for all reasonable closed subschemes $Z\subset\X$ and indices
$\gamma\in\Gamma$ such that $Z\subset X_\gamma$.
 On the other hand, the abelian groups of morphisms $\boM\rarrow
\rN|_\Gamma$ in $(\X,\Gamma)\syst$ is isomorphic to the group of all
compatible collections of morphisms $\boM_{(\gamma)}\rarrow
\rN_{(X_\gamma)}$ in $X_\gamma\qcoh$, defined for all $\gamma\in\Gamma$.
 These are equivalent sets of data, as the morphism
$i_{Z,\gamma}^!\boM_{(\gamma)}\rarrow\rN_{(Z)}$ is uniquely recoverable
by applying the functor $i_{Z,\gamma}^!\:X_\gamma\qcoh \rarrow Z\qcoh$
to the morphism $\boM_{(\gamma)}\rarrow\rN_{(X_\gamma)}$.
\end{proof}

\begin{prop} \label{Gamma-systems-abelian}
 The category $(\X,\Gamma)\syst$ of\/ $\Gamma$\+systems on\/ $\X$ is
a Grothendieck abelian category.
\end{prop}

\begin{proof}
 The assertion that $(\X,\Gamma)\syst$ is an abelian category with
exact direct limit functors is straightforward.
 Moreover, the forgetful functor $(\X,\Gamma)\syst\rarrow X_\gamma
\qcoh$ taking a $\Gamma$\+system $\boM$ to the quasi-coherent sheaf
$\boM_{(\gamma)}$ preserves the kernels, cokernels, and direct limits.
 To show that the category $(\X,\Gamma)\syst$ has a set of generators,
choose for every $\gamma\in\Gamma$ a set of generators $\sS_\gamma
\subset X_\gamma\qcoh$ in the Grothendieck category of quasi-coherent
sheaves on~$X_\gamma$.
 For every quasi-coherent sheaf $\K$ on $X_\gamma$, define
the $\Gamma$\+system $\boM(\gamma,\K)$ by the rules
$\boM(\gamma,\K)_{(\delta)}=i_{\gamma\delta}{}_*\K$ for $\gamma\le
\delta\in\Gamma$ and $\boM(\gamma,\K)_{(\delta)}=0$ for
$\gamma\not\le\delta\in\Gamma$.
 Then all the objects of the form $\boM(\gamma,S)$ with
$\gamma\in\Gamma$ and $S\in\sS_\gamma$ form a set of generators of
the category $(\X,\Gamma)\syst$.
\end{proof}

\begin{lem} \label{torsion-sheaves-as-Giraud-subcat}
 The functor $\rM\longmapsto\rM|_\Gamma\:\X\tors\rarrow(\X,\Gamma)\syst$
is fully faithful.
 The endofunctor\/ $\boM\longmapsto\boM^+|_\Gamma\:(\X,\Gamma)\syst
\rarrow(\X,\Gamma)\syst$ is left exact.
\end{lem}

\begin{proof}
 The first assertion holds essentially because one can equivalently
define a quasi-coherent torsion sheaf on $\X$ as a collection of
quasi-coherent sheaves on the schemes $X_\gamma$, \,$\gamma\in\Gamma$,
endowed with the usual maps and satisfying the usual conditions, as
per the discussion in Section~\ref{qcoh-torsion-sheaves-subsecn}.
 To check the second assertion, one computes that
$$
 (\boM^+|_\Gamma)_{(\gamma)}=
 \varinjlim\nolimits_{\delta\in\Gamma:\gamma\le\delta}
 i^!_{\gamma\delta}\boM_{(\delta)}
$$
and recalls that the functor $i^!_{\gamma\delta}$ is left exact
(since it is a right adjoint).
 It is important here that the forgetful functor assigning to
a $\Gamma$\+system $\boN$ the collection of quasi-coherent sheaves
$\boN_{(\gamma)}$ (viewed as an object of the Cartesian product of
the categories $X_\gamma\qcoh$) is exact and faithful.
\end{proof}

\begin{prop} \label{Giraud-subcategory-abstract-prop}
 Let\/ $\sB$ be an abelian category and\/ $\sA\subset\sB$ be a full
subcategory whose inclusion functor $G\:\sA\rarrow\sB$ has a left
adjoint functor $F\:\sB\rarrow\sA$.
 Assume that the composition $GF\:\sB\rarrow\sB$ is a left
exact functor.
 Then the category\/ $\sA$ is abelian and the functor $F$ is exact.
 If\/ $\sB$ is a Grothendieck category, then so is\/~$\sA$.
\end{prop}

\begin{proof}
 This well-known result describes a familiar setting which occurs,
e.~g., when $\sA$ is a sheaf category and $\sB$ is the related
presheaf category (so $G$ is the inclusion of the sheaves into
the presheaves and $F$ is the sheafification), or when an arbitrary
Grothendieck category $\sA$ is represented as a localization of
a module category $\sB$ via the Gabriel--Popescu theorem.
 The claim is that, in the context of the proposition, the functor
$F$ represents $\sA$ as a quotient category of $\sB$ by its Serre
subcategory of all objects annihilated by~$F$.
 Moreover, if $\sA$ has coproducts, then the subcategory of objects
annihilated by $F$ is closed under coproducts; and it follows that
the direct limits are exact in $\sA$ whenever they are exact in
$\sB$, and a set of generators exists in $\sA$ whenever such a set
exists in~$\sB$.
 The full subcategory $\sA\subset\sB$ is called a \emph{Giraud
subcategory}~\cite[Section~X.1]{St}; notice that the inclusion functor
$G\:\sA\rarrow\sB$ is left exact, but \emph{not} exact in general.
\end{proof}

\begin{proof}[{Proof of Theorem~\ref{torsion-sheaves-abelian}}]
 Put $\sA=\X\tors$ and $\sB=(\X,\Gamma)\syst$.
 Furthermore, put $G(\rM)=\rM|_\Gamma$ and $F(\boM)=\boM^+$.
 Then the category $\sB$ is Grothendieck by
Proposition~\ref{Gamma-systems-abelian}, the functor $F$ is left
adjoint to $G$ by Lemma~\ref{torsion-sheaves-Gamma-systems-adjunction},
the functor $G$ is fully faithful by
Lemma~\ref{torsion-sheaves-as-Giraud-subcat}, and
the functor $GF$ is left exact by the same
Lemma~\ref{torsion-sheaves-as-Giraud-subcat}.
 Thus Proposition~\ref{Giraud-subcategory-abstract-prop} is applicable,
implying that $\sA$ is a Grothendieck category.
\end{proof}

\begin{qst}
 Is the reasonableness assumption needed for the validity of
Theorem~\ref{torsion-sheaves-abelian}\,?
 Is the category of quasi-coherent torsion sheaves on an arbitrary
(strict) ind-scheme abelian?
 Notice that the category of discrete $\fR$\+modules, as defined in
Section~\ref{torsion-ind-affine-subsecn}(6), is a Grothendieck
abelian category for any topological ring~$\fR$.
\end{qst}

\subsection{Inverse images} \label{torsion-inverse-images-subsecn}
 Let $\X=\ilim_{\gamma\in\Gamma}X_\gamma$ be a reasonable ind-scheme
represented by an inductive system of closed immersions of reasonable
closed subschemes.
 Let $f\:\Y\rarrow\X$ be a morphism of ind-schemes which is
``representable by schemes''.
 Put $Y_\gamma=\Y\times_\X X_\gamma$; then $Y_\gamma$ are reasonable
closed subschemes in $\Y$ and $\Y=\ilim_{\gamma\in\Gamma}Y_\gamma$.

 Let $\boM=(\boM_{(\gamma)}\in X_\gamma\qcoh)_{\gamma\in\Gamma}$
be a $\Gamma$\+system on~$\X$.
 For every $\gamma\in\Gamma$, put $\boN_{(\gamma)}=f_\gamma^*
\boM_{(\gamma)}\in Y_\gamma\qcoh$, where $f_\gamma\:Y_\gamma\rarrow
X_\gamma$.
 Then it is clear from
Lemma~\ref{reasonable-base-change}(b) that the collection of
quasi-coherent cosheaves $\boN_{(\gamma)}\in Y_\gamma\qcoh$ has
a natural structure of a $\Gamma$\+system $\boN$ on~$\Y$.
 We put $f^*\boM=\boN$.

 The functor of inverse image of quasi-coherent torsion sheaves
$f^*\:\X\tors\rarrow\Y\tors$ is defined by the rule
$$
 f^*(\rM)=(f^*(\rM|_\Gamma))^+.
$$
 One check directly or deduce from
Lemma~\ref{torsion-direct-inverse-adjunction}(b) that this construction
of the functor $f^*\:\X\tors\rarrow\Y\tors$ does not depend on
the choice of a representation of a reasonable ind-scheme $\X$ by
an inductive system of closed immersions of reasonable
closed subschemes $(X_\gamma)_{\gamma\in\Gamma}$.
 (See Remark~\ref{simplified-flat-torsion-inverse-image} below for
a simpler construction of the functor~$f^*$ in the case of a flat
morphism~$f$.)

 Let us also define the functor of direct image of $\Gamma$\+systems
$f_*\:(\Y,\Gamma)\syst\rarrow(\X,\Gamma)\syst$.
 Put $f_*\boN=\boM$, where $\boN_{(\gamma)}=f_*\boM_{(\gamma)}$ for all
$\gamma\in\Gamma$.
 It is clear that the direct images of torsion sheaves and
$\Gamma$\+systems agree in the sense that
$(f_*\rN)|_\Gamma\simeq f_*(\rN|_\Gamma)$ for all $\rN\in\Y\tors$
(cf.\ Section~\ref{torsion-direct-images-subsecn}).

\begin{lem} \label{torsion-direct-inverse-adjunction}
\textup{(a)} The functor $f^*\:(\X,\Gamma)\syst\rarrow(\Y,\Gamma)\syst$
is left adjoint to the functor $f_*\:(\Y,\Gamma)\syst\rarrow
(\X,\Gamma)\syst$. \par
\textup{(b)} The functor $f^*\:\X\tors\rarrow\Y\tors$ is left adjoint
to the functor $f_*\:\Y\tors\rarrow\X\tors$.
\end{lem}

\begin{proof}
 Part~(a): let $\boM$ be a $\Gamma$\+system on $\X$ and $\boN$ be
a $\Gamma$\+system on~$\Y$.
 Then the group of morphisms $\boM\rarrow f_*\boN$ in $(\X,\Gamma)\syst$
is isomorphic to the group of all compatible collections of
morphisms $\boM_{(\gamma)}\rarrow f_\gamma{}_*\boN_{(\gamma)}$ in
$X_\gamma\qcoh$, defined for all $\gamma\in\Gamma$, while
the group of morphism $f^*\boM\rarrow\boN$ in $(\Y,\Gamma)\syst$
is isomorphic to the group of all compatible collections of morphisms
$f_\gamma^*\boM_{(\gamma)}\rarrow\boN_{(\gamma)}$ in $Y_\gamma\qcoh$,
defined for all $\gamma\in\Gamma$.
 In view of the adjunction of the functors $f_\gamma{}_*\:Y_\gamma\qcoh
\rarrow X_\gamma\qcoh$ and $f_\gamma^*\:X_\gamma\qcoh\rarrow
Y_\gamma\qcoh$, these are two equivalent sets of data.

 Part~(b): let $\rM$ be a quasi-coherent torsion sheaf on $\X$ and $\rN$
be a quasi-coherent torsion sheaf on~$\Y$.
 Then we have
\begin{multline*}
 \Hom_{\X\tors}(\rM,f_*\rN)\simeq\Hom_{(\X,\Gamma)\syst}(\rM|_\Gamma,
 (f_*\rN)|_\Gamma) \\
 \simeq\Hom_{(\X,\Gamma)\syst}(\rM|_\Gamma,f_*(\rN|_\Gamma))
 \simeq\Hom_{(\Y,\Gamma)\syst}(f^*(\rM|_\Gamma),\rN|_\Gamma) \\
 \simeq\Hom_{\Y\tors}(f^*(\rM|_\Gamma)^+,\rN)=
 \Hom_{\Y\tors}(f^*\rM,\rN),
\end{multline*}
where the first isomorphism holds by
Lemma~\ref{torsion-sheaves-as-Giraud-subcat}, the middle one by
part~(a), and the next one by
Lemma~\ref{torsion-sheaves-Gamma-systems-adjunction}.
\end{proof}

\begin{lem} \label{torsion-inverse-image-commutes-with-plus}
 The functors~$({-})^+$ commute with inverse images; in other words,
for any\/ $\Gamma$\+system\/ $\boM$ on\/ $\X$ there is a natural
isomorphism of quasi-coherent torsion sheaves $f^*(\boM^+)\simeq
(f^*\boM)^+$ on\/~$\Y$.
\end{lem}

\begin{proof}
 In view of the adjunctions of
Lemma~\ref{torsion-direct-inverse-adjunction}, the desired
natural isomorphism is adjoint to the natural isomorphism of
$(f_*\rN)|_\Gamma\simeq(f_*\rN|_\Gamma)$ of $\Gamma$\+systems on $\X$
for a quasi-coherent torsion sheaf $\rN$ on~$\Y$.
\end{proof}

 Let $i\:\Z\rarrow\X$ be a closed immersion of ind-schemes.
 Assume that $\X$ is a reasonable ind-scheme; then so is~$\Z$.
 The functor $i^!\:\X\tors\rarrow\Z\tors$ is defined by the rule
$(i^!\rM)_{(W)}=k^!(\rM_{(Y)})$ for all $\rM\in\X\tors$, where
$W\subset\Z$ is an arbitrary reasonable closed subscheme and
$Y\subset\X$ is a reasonable closed subscheme such that
the composition $W\rarrow\Z\overset i\rarrow\X$ factorizes as
$W\overset k\rarrow Y\rarrow\X$.

 Let $\X=\ilim_{\gamma\in\Gamma}X_\gamma$ be a representation of $\X$
by an inductive system of closed immersions of reasonable closed
subschemes.
 Put $Z_\gamma=\Z\times_\X X_\gamma$; then $Z_\gamma$ are reasonable
closed subschemes in $\Z$ and $\Z=\ilim_{\gamma\in\Gamma}Z_\gamma$.
 The functor $i^!\:\X\tors\rarrow\Z\tors$ can be described in these
terms by the rule $(i^!\rM)_{(Z_\gamma)}=i_\gamma^!\rM_{(X_\gamma)}$,
where $i_\gamma$~denotes the closed immersion of schemes
$i_\gamma\:Z_\gamma\rarrow X_\gamma$.

\begin{lem}
 The functor $i^!\:\X\tors\rarrow\Z\tors$ is right adjoint to
the direct image functor $i_*\:\Z\tors\rarrow\X\tors$.
\end{lem}

\begin{proof}
 Similar to the proof of
Lemma~\ref{torsion-direct-inverse-adjunction}(a).
\end{proof}

 In particular, let $Z\subset\X$ be a reasonable closed subscheme
with the closed immersion morphism $i\:Z\rarrow\X$.
 Then, by the definition, one has $i^!\rM=\rM_{(Z)}\in Z\qcoh$
for every $\rM\in\X\tors$.

 Notice that the functor $i^!\:\X\tors\rarrow Z\qcoh$ preserves
direct limits when $Z$ is a reasonable closed subscheme in~$\X$
(as it is clear from the discussion
in Section~\ref{torsion-direct-limits-subsecn}).
 For nonreasonable closed subschemes $Z\subset\X$, this is not
true in general.

\subsection{Injective quasi-coherent torsion sheaves}
 Let $\X=\ilim_{\gamma\in\Gamma}X_\gamma$ be a reasonable ind-scheme
represented by an inductive system of closed immersions of reasonable
closed subschemes.
 Let $i_\gamma\:X_\gamma\rarrow\X$ denote the natural closed immersions.

 According to Theorem~\ref{torsion-sheaves-abelian}, the category
of quasi-coherent torsion sheaves $\X\tors$ is a Grothendieck abelian
category; so it has enough injective objects.
 Let us describe these injectives.
 
\begin{lem} \label{torsion-monomorphisms-characterized}
\textup{(a)} A morphism $f\:\rM\rarrow\rN$ in the abelian category\/ 
$\X\tors$ is a monomorphism if and only if, for every $\gamma\in\Gamma$,
the morphism $i_\gamma^!f\:i_\gamma^!\rM\rarrow i_\gamma^!\rN$ is
a monomorphism in the abelian category $X_\gamma\qcoh$. \par
\textup{(b)} A morphism $f\:\rM\rarrow\rN$ in the abelian category\/ 
$\X\tors$ is an epimorphism whenever, for every $\gamma\in\Gamma$,
the morphism $i_\gamma^!f\:i_\gamma^!\rM\rarrow i_\gamma^!\rN$ is
an epimorphism in the abelian category $X_\gamma\qcoh$.
\end{lem}

\begin{proof}
 Part~(a): for any closed immersion of ind-schemes $i\:\Z\rarrow\X$,
the functor $i^!\:\X\tors\rarrow\Z\tors$ is left exact, since it has
a left adjoint functor $i_*\:\Z\tors\rarrow\X\tors$.
 In particular, for a closed subscheme $Z$ in $\X$ with the closed
immersion morphism $i\:Z\rarrow\X$, the functor $i^!\:\X\tors\rarrow
Z\qcoh$ is left exact.
 This proves the ``only if'' assertion.

 To prove the ``if'', assume that the morphism $i_\gamma^!f$ is
a monomorphism in $X_\gamma\qcoh$ for every $\gamma\in\Gamma$.
 Let $\rK$ be the kernel of~$f$ in $\X\tors$.
 Then $i_\gamma^!\rK$ is the kernel of the morphism $i_\gamma^!f$
in $X_\gamma\qcoh$, since the functor~$i_\gamma^!$ is left exact.
 So we have $\rK_{(X_\gamma)}=i_\gamma^!\rK=0$ for all
$\gamma\in\Gamma$, and it follows immediately that $\rK=0$.

 Part~(b): assume that the morphism $i_\gamma^!f$ is
an epimorphism in $X_\gamma\qcoh$ for every $\gamma\in\Gamma$.
 This means that $f|_\Gamma\:\rM|_\Gamma\rarrow\rN|_\Gamma$ is
an epimorphism of $\Gamma$\+systems on~$\X$.
 Since the functor $({-})^+\:(\X,\Gamma)\syst\rarrow\X\tors$ is (right)
exact, it follows that $(f_\Gamma)^+\:(\rM|_\Gamma)^+\rarrow
(\rN|_\Gamma)^+$ is an epimorphism in $\X\tors$.
 It remains to recall that the adjunction $(f_\Gamma)^+\rarrow f$ is
an isomorphism.
\end{proof}

\begin{lem} \label{closed-under-subquotients-characterized}
 Let\/ $\sA$, $\sB$ be an abelian categories and $F\:\sA\rarrow\sB$ be
a fully faithful exact functor which has a right adjoint functor
$H\:\sB\rarrow\sA$.
 Then the essential image $F(\sA)\subset\sB$ is a full subcategory
closed under subobjects and quotients in\/ $\sB$ if and only if
the adjunction morphism $FH(B)\rarrow B$ is a monomorphism in\/ $\sB$
for every $B\in\sB$.
\end{lem}

\begin{proof}
 Notice that the essential image of a fully faithful exact functor
between abelian categories is always a full subcategory closed under
kernels and cokernels.
 Hence $f(\sA)$ is closed under subobjects in $\sB$ \emph{if and only
if} it is closed under quotients.
 Now we can proceed with a proof of the lemma.

 ``If'': let $A\in F(\sA)$ be an object and $A\rarrow B$ be
an epimorphism in~$\sB$.
 Then the adjunction morphism $FH(A)\rarrow A$ is an isomorphism,
and it follows from commutativity of the obvious diagram that
the adjunction morphism $FH(B)\rarrow B$ is an epimorphism.
 Since the morphism $FH(B)\rarrow B$ is a monomorphism by assumption,
it is an isomorphism.
 Thus $B\in F(\sA)$.
 
 ``Only if'': let $B\in\sB$ be an object and $c\:FH(B)\rarrow B$ be
the adjunction morphism.
 Let $C\in\sB$ be the image of~$c$.
 Then $C$ is a quotient object of an object from $F(\sA)$; by
assumption, it follows that $C\in F(\sA)$.
 Given an object $B\in\sB$, the object $FH(B)\in F(\sA)$ together with
the morphism~$c$ is characterized by the universal property that
any morphism into $B$ from an object of $F(\sA)$  factorizes uniquely
through~$c$.
 Now the monomorphism $C\rarrow B$ has the same universal property;
hence the epimorphism $FH(B)\rarrow C$ is an isomorphism and $c$~is
a monomorphism.
\end{proof}

\begin{lem} \label{injectives-in-Grothendieck-characterized}
 Let\/ $\sA$ be a Grothendieck abelian category with a set of
generators\/ $\sS\subset\sA$.
 Then an object $J\in\sA$ is injective if and only if any morphism
into $J$ from a subobject of any object $S\in\sS$ can be extended
to a morphism $S\rarrow J$ in\/~$\sA$.
\end{lem}

\begin{proof}
 This is a categorical version of the Baer criterion of injectivity
of modules, provable in the same way using the Zorn lemma.
\end{proof}

\begin{lem} \label{closed-subschemes-torsion-subcategories}
\textup{(a)} For any closed subscheme $Z\subset\X$ with the closed
immersion morphism $i\:Z\rarrow\X$, the direct image functor
$i_*\:Z\qcoh\rarrow\X\tors$ is exact and fully faithful.
 Its essential image is closed under subobjects and quotients in
the abelian category\/ $\X\tors$. \par
\textup{(b)} For every\/ $\gamma\in\Gamma$, choose a set of generators\/
$\sS_\gamma\subset X_\gamma\qcoh$ of the abelian category of
quasi-coherent sheaves on~$X_\gamma$.
 Then the set\/ $\sS\subset\X\tors$ of all quasi-coherent torsion
sheaves of the form $i_\gamma{}_*\cS$, where\/ $\gamma\in\Gamma$ and
$\cS\in\sS_\gamma$, is a set of generators of the abelian category\/
$\X\tors$.
\end{lem}

\begin{proof}
 Part~(a): more generally, let $i\:\Z\rarrow\X$ be a closed immersion
of (reasonable) ind-schemes.
 Then, following the discussion in
Section~\ref{torsion-inverse-images-subsecn}, the direct image functor
$i_*\:\Z\tors\rarrow\X\tors$ has adjoints on both sides, $i^*$
and~$i^!$; so $i_*$~is an exact functor.

 Furthermore, let $\rN$ be a quasi-coherent torsion sheaf on~$\Z$.
 Then, following the construction in
Section~\ref{torsion-direct-images-subsecn}, the quasi-coherent
torsion sheaf $i_*\rN$ on $\X$ is defined by the rule
$(i_*\rN)_{(Y)}=i_Y{}_*(\rN_{(W)})$ for all reasonable closed
subschemes $Y\subset\X$, where $i_Y\:W=Y\times_\X\Z\rarrow Y$.
 The quasi-coherent torsion sheaf $i^!i_*\rN$ on $\Z$ is described by
the rule $(i^!i_*\rN)_{(W')}=k^!i_Y{}_*(\rN_{(W)})$ for all reasonable
closed subschemes $W'\subset W$, where $k\:W'\rarrow Y$ is
the composition $W'\rarrow W\rarrow Y$ of the closed immersion
$W'\rarrow W$ and the morphism $i_Y\:W\rarrow Y$ (which is also
a closed immersion, by the definition of a closed immersion of
ind-schemes).
 Clearly, the adjunction morphism $\rN\rarrow i^!i_*\rN$ is
an isomorphism in $\Z\tors$ (because the adjunction morphism
$\rN_{(W)}\rarrow i_Y^!i_Y{}_*\rN_{(W)}$ is an isomorphism in $W\qcoh$
for every~$Y$; cf.\ Section~\ref{qcoh-sheaves-and-functors-subsecn}).
 It follows that the direct image functor~$i_*$ is fully faithful.

 Finally, let $\rM$ be a quasi-coherent torsion sheaf on~$X$.
 Then the quasi-coherent torsion sheaf $i^!\rM$ on $\Z$ is defined as
spelled out in Section~\ref{torsion-inverse-images-subsecn}.
 Hence the quasi-coherent torsion sheaf $i_*i^!\rM$ on $\X$ is
described by the rule $(i_*i^!\rM)_{(Y)}=i_Y{}_*i_Y^!\rM_{(Y)}$
for all reasonable closed subschemes $Y\subset\X$, where $i_Y\:W=
Y\times_\X\Z\rarrow Y$.
 The adjunction morphisms $i_Y{}_*i_Y^!\rM_{(Y)}\rarrow\rM_{(Y)}$ are
monomorphisms in $Y\qcoh$ for all~$Y$ (since $i_Y$~is a closed
immersion of schemes).
 By Lemma~\ref{torsion-monomorphisms-characterized}(a), it follows that
the adjunction morphism $i_*i^!\rM\rarrow\rM$ is a monomorphism in
$\X\tors$.
 It remains to apply Lemma~\ref{closed-under-subquotients-characterized}
in order to conclude that the essential image of the functor
$i_*\:\Z\tors\rarrow\X\tors$ is a full subcategory closed under
subobjects and quotients.

 Part~(b): a set of generators of the abelian category
$(\X,\Gamma)\syst$ was constructed in the proof of
Proposition~\ref{Gamma-systems-abelian}.
 The subset $\sS\subset\X\tors$ is the image of this set of generators
under the functor $\boM\longmapsto\boM^+\:(X,\Gamma)\syst\rarrow
\X\tors$.
 This functor is essentially surjective on objects, exact, and preserves
coproducts (being a left adjoint); hence the image of any set of
generators under this functor is a set of generators.
 (See the arguments in Section~\ref{Gamma-systems-subsecn} for
the details.)

 Alternatively, one can notice that, for every quasi-coherent torsion
sheaf $\rM\in\X\tors$, the natural morphism
$$
 \coprod\nolimits_{\gamma\in\Gamma}i_\gamma{}_*i_\gamma^!\rM
 \lrarrow\rM
$$
is an epimorphism in $\X\tors$ (by
Lemma~\ref{torsion-monomorphisms-characterized}(b)).
 Then the assertion easily follows.
\end{proof}

\begin{prop} \label{torsion-injectives-characterized}
\textup{(a)} For any closed subscheme $Z\subset\X$ with the closed
immersion morphism $i\:Z\rarrow\X$, the functor $i^!\:\X\tors\rarrow
Z\qcoh$ takes injective objects to injective objects. \par
\textup{(b)} A quasi-coherent torsion sheaf $\rJ\in\X\tors$ is
an injective object in\/ $\X\tors$ if and only if, for every\/
$\gamma\in\Gamma$, the quasi-coherent sheaf $i_\gamma^!\rJ\in
X_\gamma\qcoh$ is an injective object in $X_\gamma\qcoh$.
\end{prop}

\begin{proof}
 Part~(a): more generally, for any closed immersion of (reasonable)
ind-schemes $i\:\Z\rarrow\X$, the functor $i_*\:\Z\tors\rarrow\X\tors$
is exact, as explained in the proof of
Lemma~\ref{closed-subschemes-torsion-subcategories}(a).
 The functor $i^!\:\X\tors\rarrow\Z\tors$ is right adjoint to~$i_*$;
so it takes injectives to injectives.
 Part~(b): the ``only if'' assertion is provided by part~(a).
 To prove the ``if'', one can apply
Lemma~\ref{injectives-in-Grothendieck-characterized}
to the set of generators of the Grothendieck category $\X\tors$
provided by Lemma~\ref{closed-subschemes-torsion-subcategories}(b).
 This shows that a quasi-coherent torsion sheaf $\rJ$ on $\X$ is
injective whenever, for any $\M\in X_\gamma\qcoh$ and a subobject
$\rK\subset i_\gamma{}_*\M$, \ $\rK\in\X\tors$, any morphism
$\rK\rarrow\rJ$ can be extended to a morphism $i_\gamma{}_*\M
\rarrow\rJ$ in $\X\tors$.
 By Lemma~\ref{closed-subschemes-torsion-subcategories}(a), there is
a quasi-coherent subsheaf $\N\subset\M$ on $X_\gamma$ such that
$\rK=i_\gamma{}_*\N$.
 Now it suffices to extend a given morphism $\N\rarrow i_\gamma^!\rJ$ to
a morphism $\M\rarrow i_\gamma^!\rJ$ in $X_\gamma\qcoh$.
\end{proof}

\Section{Flat Pro-Quasi-Coherent Pro-Sheaves}
\label{flat-pro-sheaves-secn}

 In this section we continue to follow~\cite[Sections~7.11.3\+-4]{BD2}.

\subsection{Pro-quasi-coherent pro-sheaves}  \label{pro-sheaves-subsecn}
 Let $\X$ be an ind-scheme.
 A \emph{pro-quasi-coherent pro-sheaf} $\fP$ on $\X$ (called
an ``$\cO^p$\+module'' in~\cite{BD2}) is the following set of data:
\begin{enumerate}
\renewcommand{\theenumi}{\roman{enumi}}
\item to every closed subscheme $Y\subset\X$, a quasi-coherent sheaf
$\fP^{(Y)}$ on $Y$ is assigned;
\item to every pair of closed subschemes $Y$, $Z\subset\X$, \
$Z\subset Y$ with the closed immersion morphism $i_{ZY}\:Z\rarrow Y$,
a morphism $\fP^{(Y)}\rarrow i_{ZY}{}_*\fP^{(Z)}$ of quasi-coherent
sheaves on $Y$ is assigned;
\item such that the corresponding morphism $i_{ZY}^*\fP^{(Y)}\rarrow
\fP^{(Z)}$ of quasi-coherent sheaves on $Z$ is an isomorphism;
\item and, for every triple of closed subschemes $Y$, $Z$, $W\subset\X$,
\ $W\subset Z\subset Y$, the triangle diagram $\fP^{(Y)}\rarrow
i_{ZY}{}_*\fP^{(Z)}\rarrow i_{WY}{}_*\fP^{(W)}$ is commutative in
$Y\qcoh$.
\end{enumerate}

 Let $\X=\ilim_{\gamma\in\Gamma}X_\gamma$ be a representation of $\X$
by a inductive system of closed immersions of schemes.
 Then, in order to construct a pro-quasi-coherent pro-sheaf $\fP$ on
$\X$, it suffices to specify the quasi-coherent sheaves
$\fP^{(X_\gamma)}\in X_\gamma\qcoh$ for every $\gamma\in\Gamma$ and
the morphisms $\fP^{(X_\delta)}\rarrow i_{X_\gamma X_\delta}{}_*
\fP^{(X_\gamma)}$ for every $\gamma<\delta\in\Gamma$ satisfying
conditions~(iii\+-iv) for $W=X_\beta$, $Z=X_\gamma$, $Y=X_\delta$,
\ $\beta<\gamma<\delta\in\Gamma$.
 The quasi-coherent sheaves $\fP^{(Y)}$ for all the other closed
subsechemes $Y\subset\X$ and the related morphisms~(ii) can then be
uniquely recovered so that conditions~(iii\+-iv) are satisfied for
all closed subschemes in~$\X$.

 \emph{Morphisms of pro-quasi-coherent pro-sheaves} $f\:\fP\rarrow\fQ$
on $\X$ are defined in the obvious way.
 We denote the additive category of quasi-coherent torsion sheaves
on $\X$ by $\X\pro$.
 The following example shows that the category $\X\pro$ is usually
\emph{not} abelian, and generally not homologically well-behaved.
 In this paper, we will be interested in certain (better behaved)
full subcategories in $\X\pro$.

\begin{ex} \label{pro-sheaves-not-abelian-ex}
 Let $\X=\Spi\boZ_p$ be the ind-affine ind-scheme from
Example~\ref{Z-hat-ind-scheme}(2).
 Then the category $\X\pro$ is equivalent to the category of
$p$\+adically separated and complete abelian groups
(cf.\ Section~\ref{torsion-ind-affine-subsecn}(2)).
 Here an abelian group $P$ is said to be \emph{$p$\+adically
separated and complete} if its natural map to its $p$\+adic
completion $P\rarrow\varprojlim_{r\ge0}P/p^rP$ is an isomorphism.
 The equivalence of categories assigns to every $p$\+adically
separated and complete abelian group $P$ the pro-quasi-coherent
pro-sheaf $\fP$ with the quasi-coherent sheaf $\fP^{(X_r)}$
corresponding to the $\boZ/p^r\boZ$\+module $P/p^rP$ (for
the closed subscheme $X_r=\Spec\boZ/p^r\boZ\subset\Spi\boZ_p=\X$).
 Conversely, to every pro-quasi-coherent pro-sheaf $\fP$ on
$\Spi\boZ_p$, the $p$\+adically separated and complete abelian group
$P=\varprojlim_{r\ge0}\fP^{(X_r)}(X_r)$ is assigned.

 The category of $p$\+adically separated and complete abelian groups
(known also as \emph{separated $p$\+contramodules}) is
\emph{not} abelian~\cite[Example~2.7(1)]{Pcta}.
 So the category $(\Spi\boZ_p)\pro$ is not abelian.
 Similarly, the category $(\Spi\kk[[x]])\pro$ (for the ind-affine
ind-scheme $\X=\Spi\kk[[x]]$ from
Example~\ref{coalgebra-ind-scheme}(1)) is \emph{not} abelian, either.
\end{ex}

 On the other hand, for any ind-scheme $\X$, the category $\X\pro$ has
a natural (associative, commutative, and unital) tensor category
structure.
 The tensor product $\fP\ot^\X\fQ\in\X\pro$ of two pro-quasi-coherent
pro-sheaves $\fP$ and $\fQ\in\X\pro$ is defined by the rule
$(\fP\ot^\X\fQ)^{(Z)}=\fP^{(Z)}\ot_{\cO_Z}\fQ^{(Z)}\in Z\qcoh$ for all
closed subschemes $Z\subset\X$.
 As the inverse images of quasi-coherent sheaves preserve their tensor
products, the construction of the structure (iso)morphism~(ii\+-iii)
for $\fP\ot^\X\fQ$ is obvious.
 The unit object of this tensor structure is the ``pro-structure
pro-sheaf'' $\fO_\X\in\X\pro$, defined by the rule
$(\fO_\X)^{(Z)}=\cO_Z$ for all closed subschemes $Z\subset\X$.

 The aim of the next Section~\ref{pro-in-torsion-action-subsecn} is
to construct a structure of module category over $\X\pro$ on
the category of quasi-coherent torsion sheaves $\X\tors$.

\subsection{Action of pro-sheaves in torsion sheaves}
\label{pro-in-torsion-action-subsecn}
 Let $\X=\ilim_{\gamma\in\Gamma}X_\gamma$ be a reasonable ind-scheme
represented by an inductive system of closed immersions of reasonable
closed subschemes.
 We start with considering the category $(\X,\Gamma)\syst$ of
$\Gamma$\+systems on $\X$ (as defined in
Section~\ref{Gamma-systems-subsecn}) and constructing a structure of
module category over $\X\pro$ on $(\X,\Gamma)\syst$.

 Let $\fP\in\X\pro$ be a pro-quasi-coherent pro-sheaf
and $\boM\in(\X,\Gamma)\syst$ be a $\Gamma$\+system on~$\X$.
 The $\Gamma$\+system $\fP\ot_\X\boM$ on $\X$ is defined by the rule
$$
 (\fP\ot_\X\boM)_{(\gamma)}=
 \fP^{(X_\gamma)}\ot_{\cO_{X_\gamma}}\boM_{(\gamma)}.
$$
 The structure morphism $i_{\gamma\delta}{}_*
((\fP\ot_\X\boM)_{(\gamma)})\rarrow(\fP\ot_\X\boM)_{(\delta)}$ from
Section~\ref{Gamma-systems-subsecn}, item~(ii), is constructed as
the composition
\begin{multline*}
 i_{\gamma\delta}{}_*
 (\fP^{(X_\gamma)}\ot_{\cO_{X_\gamma}}\boM_{(\gamma)})
 \,\simeq\, i_{\gamma\delta}{}_*(i_{\gamma\delta}^*\fP^{(X_\delta)}
 \ot_{\cO_{X_\gamma}}\boM_{(\gamma)}) \\ \,\simeq\,
 \fP^{(X_\delta)}\ot_{\cO_{X_\delta}}i_{\gamma\delta}{}_*\boM_{(\gamma)}
 \lrarrow\fP^{(X_\delta)}\ot_{\cO_{X_\delta}}\boM_{(\delta)}
\end{multline*}
of the isomorphism $i_{\gamma\delta}{}_*(\fP^{(X_\gamma)}
\ot_{\cO_{X_\gamma}}\boM_{(\gamma)})\simeq i_{\gamma\delta}{}_*
(i_{\gamma\delta}^*\fP^{(X_\delta)}\ot_{\cO_{X_\gamma}}\boM_{(\gamma)})$
induced by the structure isomorphism $\fP^{(X_\gamma)}\simeq
i_{\gamma\delta}^*\fP^{(X_\delta)}$, the ``projection formula''
isomorphism $i_{\gamma\delta}{}_*(i_{\gamma\delta}^*\fP^{(X_\delta)}
\ot_{\cO_{X_\gamma}}\boM_{(\gamma)}) \simeq \fP^{(X_\delta)}
\ot_{\cO_{X_\delta}}i_{\gamma\delta}{}_*\boM_{(\gamma)}$, and
the morphism $\fP^{(X_\delta)}\ot_{\cO_{X_\delta}}
i_{\gamma\delta}{}_*\boM_{(\gamma)}\rarrow\fP^{(X_\delta)}
\ot_{\cO_{X_\delta}}\boM_{(\delta)}$ induced by the structure morphism
$i_{\gamma\delta}{}_*\boM_{(\gamma)}\rarrow\boM_{(\delta)}$.

 The tensor product functor
$$
 \ot_\X\:\X\pro\times(\X,\Gamma)\syst\lrarrow(\X,\Gamma)\syst
$$
endows the category of $\Gamma$\+systems on $\X$ with the structure
of an associative, unital \emph{module category} over the tensor
category of pro-quasi-coherent pro-sheaves $\X\pro$.

 The following lemma plays a key role.

\begin{lem} \label{tensor-null-Gamma-systems}
 Let\/ $\fP\in\X\pro$ be a pro-quasi-coherent prosheaf and\/ $\boM$
be a\/ $\Gamma$\+system on\/ $\X$ whose associated quasi-coherent
torsion sheaf vanishes, $\boM^+=0$.
 Then the quasi-coherent torsion sheaf associated with
the\/ $\Gamma$\+system $\fP\ot_\X\boM$ also vanishes,
$(\fP\ot_\X\boM)^+=0$.
\end{lem}

\begin{proof}
 The proof is straightforward.
\end{proof}

 Recall from the proof of
Proposition~\ref{Giraud-subcategory-abstract-prop} that the functor
$\boM\longmapsto\boM^+\:(\X,\Gamma)\syst\allowbreak\rarrow\X\tors$
represents the category of quasi-coherent torsion sheaves $\X\tors$ as
the abelian quotient category of the category of $\Gamma$\+systems
$(\X,\Gamma)\syst$ by the Serre subcategory of all $\Gamma$\+systems
annihilated by this functor.
 In view of Lemma~\ref{tensor-null-Gamma-systems}, it follows that,
for any $\fP\in\X\pro$, the tensor product functor
$\fP\ot_\X{-}\,\:(\X,\Gamma)\syst\rarrow(\X,\Gamma)\syst$
descends uniquely along the functor $\boM\rarrow\boM^+$, leading to
a tensor product functor $\X\tors\rarrow\X\tors$, which we denote by
the same symbol $\fP\ot_\X{-}\,$.

 Explicitly, for any $\fP\in\X\pro$ and $\rM\in\X\tors$ we put
$$
 \fP\ot_\X\rM=(\fP\ot_\X\rM|_\Gamma)^+=
 \varinjlim\nolimits_{\gamma\in\Gamma}i_\gamma{}_*
 (\fP^{(X_\gamma)}\ot_{\cO_{X_\gamma}}\rM_{(X_\gamma)})\in\X\tors,
$$
where $i_\gamma\:X_\gamma\rarrow\X$ is the closed immersion morphism
and $i_\gamma{}_*\:X_\gamma\qcoh\rarrow\X\tors$ is the direct image
functor.
 Clearly, this construction of the tensor product of
a pro-quasi-coherent pro-sheaf and a quasi-coherent torsion sheaf
does not depend on the choice of a representation of a reasonable
ind-scheme $\X$ by an inductive system $(X_\gamma)_{\gamma\in\Gamma}$
of closed immersions of reasonable closed subschemes.

 The resulting tensor product functor
$$
 \ot_\X\:\X\pro\times\X\tors\lrarrow\X\tors
$$
endows the category of quasi-coherent torsion sheaves $\X\tors$ with
the structure of an associative, unital module category over the tensor
category of pro-quasi-coherent pro-sheaves $\X\pro$.

 In the sequel, we will sometimes switch the two arguments
of the functor~$\ot_\X$, writing
$\ot_\X\:\X\tors\times\X\pro\rarrow\X\tors$.

\subsection{Inverse and direct images}
\label{pro-sheaves-inverse-direct-subsecn}
 Let $f\:\Y\rarrow\X$ be a morphism of ind-schemes.
 The functor of inverse image of pro-quasi-coherent pro-sheaves
$f^*\:\X\pro\rarrow\Y\pro$ is defined by the rule $(f^*\fP)^{(W)}=
g^*(\fP^{(Z)})$ for all $\fP\in\X\pro$, where $W\subset\Y$ in
an arbitrary closed subscheme and $Z\subset\X$ is a closed subscheme
such that the composition $W\rarrow\Y\overset f\rarrow\X$ factorizes
as $W\overset g\rarrow Z\rarrow\X$.

 The inverse image functor
$$
 f^*\:\X\pro\lrarrow\Y\pro
$$
is a tensor functor between the tensor categories $\X\pro$ and
$\Y\pro$, taking the unit object $\fO_\X\in\X\pro$ to the unit object
$\fO_\Y\in\Y\pro$.

 In particular, if $Y\subset\X$ is a closed subscheme with the closed
immersion morphism $i\:Y\rarrow\X$, then one has $i^*\fP=\fP^{(Y)}$
in $Y\qcoh$ for every $\fP\in\X\pro$.

 Part~(a) of the following lemma is a generalization of
Lemma~\ref{reasonable-base-change}(b) (with the roles of the schemes
$Y$ and $Z$ switched).

\begin{lem} \label{affine-flat-base-change}
 Let $f\:Y\rarrow X$ and $h\:Z\rarrow X$ be morphisms of
(concentrated) schemes.
 Consider the pullback diagram
$$
\xymatrix{
 Z\times_XY \ar[r]^-k \ar[d]^-g & Y \ar[d]^-f \\
 Z \ar[r]^-h & X
}
$$
 Assume that either \par
\textup{(a)} the morphism~$f$ is affine, or \par
\textup{(b)} the morphism~$h$ is flat. \par
\noindent Then there is a natural isomorphism $h^*f_*\simeq g_*k^*$
of functors $Y\qcoh\rarrow Z\qcoh$.
\end{lem}

\begin{proof}
 Part~(a) is~\cite[Tag~02KG]{SP}.
 The assertion is local in $X$ and $Z$, so it reduces to
the case of affine schemes, for which it means the following.
 Let $R\rarrow S$ and $R\rarrow T$ be homomorphisms of commutative
rings.
 Then, for any $S$\+module $N$, there is a natural isomorphism of
$T$\+modules $T\ot_RN\simeq(T\ot_RS)\ot_SN$.
 Part~(b) is a particular case of~\cite[Tag~02KH]{SP}.
\end{proof}

 Let $f\:\Y\rarrow\X$ be an affine morphism of ind-schemes
(as defined in Section~\ref{morphisms-of-ind-schemes-subsecn}).
 Let $\fQ$ be a pro-quasi-coherent pro-sheaf on~$\Y$.
 For every closed subscheme $Z\subset\X$, put $\fP^{(Z)}=
f_Z{}_*(\fQ^{(W)})\in Z\qcoh$, where $f_Z$ is the affine morphism
of schemes $W=Z\times_\X\Y\rarrow Z$.
 Then it is clear from Lemma~\ref{affine-flat-base-change}(a)
that the collection of quasi-coherent sheaves $\fP^{(Z)}$ with
the natural maps $\fP^{(Z'')}\rarrow i_{Z'Z''}{}_*\fP^{(Z')}$
for $Z'\subset Z''\subset\X$ is a pro-quasi-coherent pro-sheaf
$\fP$ on~$\X$.

 Put $f_*\fQ=\fP$.
 This construction defines the functor of direct image of
pro-quasi-coherent pro-sheaves $f_*\:\Y\pro\rarrow\X\pro$ with
respect to an affine morphism of ind-schemes $f\:\Y\rarrow\X$.
 The functor~$f_*$ is right adjoint to the inverse image
functor $f^*\:\X\pro\rarrow\Y\pro$ (as one can show similarly to
the proof of Lemma~\ref{torsion-direct-inverse-adjunction}(a)).

 Furthermore, for any affine morphism of ind-schemes $f\:\Y\rarrow\X$,
the following projection formula isomorphism holds naturally
\begin{equation} \label{pro-sheaves-projection-formula-eqn}
 f_*(f^*\fP\ot^\Y\fQ)\simeq\fP\ot^\X f_*\fQ
\end{equation}
for all $\fP\in\X\pro$ and $\fQ\in\Y\pro$.

\begin{lem} \label{inverse-images-preserve-torsion-pro-tensor}
 Let $f\:\Y\rarrow\X$ be a morphism of reasonable ind-schemes which is
``representable by schemes''.
 Let\/ $\fP$ a pro-quasi-coherent pro-sheaf on\/ $\X$ and $\rM$
be a quasi-coherent torsion sheaf on\/~$\X$.
 Then there is a natural isomorphism
$$
 f^*(\fP\ot_\X\rM)\simeq f^*\fP\ot_\Y f^*\rM
$$
of quasi-coherent torsion sheaves on\/~$\Y$.
\end{lem}

\begin{proof}
 The isomorphism $f^*(\fP\ot_\X\boM)\simeq f^*\fP\ot_\Y f^*\boM$
of $\Gamma$\+systems on $\Y$ for any $\Gamma$\+system $\boM$ on $\X$
is obvious from the definitions.
 Now we have
\begin{multline*}
 f^*(\fP\ot_\X\rM)=f^*((\fP\ot_\X\rM|_\Gamma)^+)\simeq
 (f^*(\fP\ot_\X\rM|_\Gamma))^+ \\ \simeq
 (f^*\fP\ot_\Y f^*(\rM|_\Gamma))^+ \simeq
 f^*\fP\ot_\Y (f^*(\rM|_\Gamma))^+ \\ \simeq
 f^*\fP\ot_\Y f^*((\rM|_\Gamma)^+)\simeq
 f^*\fP\ot_\Y f^*\rM
\end{multline*}
by the definition of the functors $\ot_\X\:\X\pro\times\X\tors
\rarrow\X\tors$ and $\ot_\Y\:\Y\pro\times\Y\tors\rarrow\Y\tors$,
and by Lemma~\ref{torsion-inverse-image-commutes-with-plus}.
 The point is that both the inverse image and the tensor product
functors in question commute with the functors~$({-})^+$.
\end{proof}

 For two versions of projection formula related to
Lemma~\ref{inverse-images-preserve-torsion-pro-tensor},
see Lemmas~\ref{torsion-pro-projection-formula-lemma}
and~\ref{torsion-pro-projection-formula-second} below.

\subsection{Flat pro-quasi-coherent pro-sheaves}
\label{flat-pro-sheaves-subsecn}
 For any scheme $X$, let us denote by $X\flat$ the full subcategory of
flat quasi-coherent sheaves in $X\qcoh$.
 Then, for any morphism of schemes $k\:Z\rarrow X$, the inverse
images functor $k^*\:X\qcoh\rarrow Z\qcoh$ takes $X\flat$ into
$Z\flat$.
 Furthermore, for any short exact sequence $0\rarrow\F\rarrow\G
\rarrow\cH\rarrow0$ in $X\qcoh$ with $\F$, $\G$, $\cH\in X\flat$,
the short sequence $0\rarrow k^*\F\rarrow k^*\G\rarrow k^*\cH\rarrow0$
is exact in $Z\qcoh$.

 Let $\X=\ilim_{\gamma\in\Gamma}X_\gamma$ be an ind-scheme represented
by an inductive system of closed immersions of schemes.
 A pro-quasi-coherent pro-sheaf $\fF$ on $\X$ is said to be
\emph{flat} if the quasi-coherent sheaf $\fF^{(Z)}$ on $Z$ is flat
for every closed subscheme $Z\subset\X$.
 It is clear from the previous paragraph that $\fF$ is flat whenever
the quasi-coherent sheaf $\fF^{(X_\gamma)}$ on $X_\gamma$ is flat for
every $\gamma\in\Gamma$.
 We denote the full subcategory of flat pro-quasi-coherent pro-sheaves
by $\X\flat\subset\X\pro$.

 Let $0\rarrow\fF\rarrow\fG\rarrow\fH\rarrow0$ be a short sequence
of flat pro-quasi-coherent pro-sheaves on~$\X$.
 We say that this is an (\emph{admissible}) short \emph{exact}
sequence in $\X\flat$ if, for every closed subscheme $Z\subset\X$,
the sequence of quasi-coherent sheaves $0\rarrow\fF^{(Z)}\rarrow
\fG^{(Z)}\rarrow\fH^{(Z)}\rarrow0$ is exact in the abelian category
$Z\qcoh$.
 It suffices to check this condition for the closed subschemes
$Z=X_\gamma$, \,$\gamma\in\Gamma$, belonging to any chosen
representation of $\X$ by an inductive system of closed immersions
of schemes.

 An \emph{exact category} (in the sense of Quillen) is an additive
category endowed with a class of \emph{admissible short exact
sequences} (also called \emph{conflations}) satisfying natural axioms.
 For a reference, see~\cite{Bueh} or~\cite[Appendix~A]{Partin}.

\begin{prop} \label{flat-pro-sheaves-exact-category}
 The category\/ $\X\flat$ of flat pro-quasi-coherent pro-sheaves on\/
$\X$, endowed with the class of admissible short exact sequences
defined above, is an exact category in the sense of Quillen.
\end{prop}

\begin{proof}
 First of all, in the particular case when $\X=X$ is a scheme,
the full subcategory $X\flat\subset X\qcoh$, endowed with the class
of all short sequences that are exact in the abealian category
$X\qcoh$, is an exact category, since $X\flat$ is closed under
extensions in $X\qcoh$ (see, e.~g., \cite[Lemma~10.20]{Bueh}
or~\cite[Example~A.5(3)(a)]{Partin}).
 Moreover, the restriction functor $k^*\:X\flat\rarrow Z\flat$ is
exact (i.~e., takes admissible short exact sequences to admissible
short exact sequences) for any morphism of schemes $k\:Z\rarrow X$.

 To prove the assertion for an ind-scheme $\X$, consider the category
of ``generalized pro-quasi-coherent pro-sheaves'' on $\X$, defined
by items~(i\+-iii) of Section~\ref{pro-sheaves-subsecn} (with
the transitivity condition~(iv) dropped).
 The definitions of a flat generalized pro-quasi-coherent pro-sheaf
and a short exact sequence of flat generalized pro-quasi-coherent
pro-sheaves are similar to the above.
 Then the category of flat generalized pro-quasi-coherent pro-sheaves
is exact, since it can be obtained by the construction
of~\cite[Example~A.5(4)]{Partin} applied to the following pair
of exact functors
$$
 \prod\nolimits_{Z\subset\X}Z\flat\lrarrow
 \prod\nolimits_{Z\subset Y\subset\X}Z\flat\times
 \prod\nolimits_{Z\subset Y\subset\X}Z\flat\llarrow
 \prod\nolimits_{Z\subset Y\subset\X}Z\flat.
$$
 Here $Z$ and $Y$ range over the closed subschemes in~$\X$.
 All the categories in the diagram are exact, with termwise
exact structures on the Cartesian products.
 The first component $\prod_{Z\subset\X}Z\flat\rarrow
\prod_{Z\subset Y\subset\X}Z\flat$ of the leftmost functor assigns
to a collection of flat quasi-coherent sheaves $(\F_Z)_{Z\subset\X}$
the collection of flat quasi-coherent sheaves
$(\F_Z)_{Z\subset Y\subset\X}$.
 The second component $\prod_{Z\subset\X}Z\flat\rarrow
\prod_{Z\subset Y\subset\X}Z\flat$ of the leftmost functor assigns
to a collection $(\F_Z)_{Z\subset\X}$ the collection of flat
quasi-coherent sheaves $(i_{ZY}^*\F_Y)_{Z\subset Y\in\X}$.
 The rightmost functor is the diagonal one.

 Finally, the full subcategory of flat pro-quasi-coherent pro-sheaves
is closed under admissible subobjects and admissible quotients in
the ambient exact category of flat generalized pro-quasi-coherent
pro-sheaves.
 So this full subcategory inherits an exact category structure
by~\cite[Example~A.5(3)(b)]{Partin}).
\end{proof}

 Let $f\:\Y\rarrow\X$ be a morphism of ind-schemes.
 Then the inverse image functor $f^*\:\X\pro\rarrow\Y\pro$ takes
the full subcategory $\X\flat\subset\X\pro$ into the full subcategory
$\Y\flat\subset\Y\pro$.
 The functor
$$
 f^*\:\X\flat\lrarrow\Y\flat
$$
is exact with respect to the exact category structures of $\X\flat$
and $\Y\flat$ (defined above).

 Let $f\:\Y\rarrow\X$ be a flat, affine morphism of ind-schemes
(as defined in Section~\ref{morphisms-of-ind-schemes-subsecn}).
 Then the direct image functor $f_*\:\Y\pro\rarrow\X\pro$ takes
the full subcategory $\Y\flat\subset\Y\pro$ into the full subcategory
$\X\flat\subset\X\pro$.
 The functor
$$
 f_*\:\Y\flat\rarrow\X\flat
$$
is exact with respect to the exact category structures on $\Y\flat$
and $\X\flat$.

 The tensor product of two flat pro-quasi-coherent pro-sheaves is flat.
 So the tensor product functor $\ot^\X\:\X\pro\times\X\pro\rarrow\X\pro$
restricts to a functor
$$
 \ot^\X\:\X\flat\times\X\flat\lrarrow\X\flat,
$$
defining a tensor category structure on $\X\flat$.
 Notice that the ``pro-structure pro-sheaf'' $\fO_\X$ (which is
the unit of the tensor structure on $\X\pro$) belongs to $\X\flat$.
 The tensor product $\ot^\X$ of flat pro-quasi-coherent pro-sheaves
on $\X$ is an exact functor with respect to the exact category
structure of $\X\flat$.

 Assume that the ind-scheme $\X$ is reasonable.
 Then, restricting the tensor product functor $\X\pro\times\X\tors
\rarrow\X\tors$ to flat pro-quasi-coherent pro-sheaves, one obtains
a tensor product functor
$$
 \ot_\X\:\X\flat\times\X\tors\rarrow\X\tors,
$$
which is exact with respect to the exact category structure
of $\X\flat$ and the abelian exact structure on $\X\tors$.

\begin{lem} \label{shriek-star-tensor}
 Let $i\:Z\rarrow X$ be a closed immersion of schemes and $\F$, $\M$
be quasi-coherent sheaves on~$X$.
 Then there is a natural morphism of quasi-coherent sheaves on~$Z$
$$
 i^*\F\ot_{\cO_Z}i^!\M\lrarrow i^!(\F\ot_{\cO_X}\M),
$$
which is an isomorphism whenever $i(Z)$ is a reasonable closed subscheme
in $X$ and $\F$ is a flat quasi-coherent sheaf on~$X$.
\end{lem}

\begin{proof}
 When $i(Z)$ is a reasonable closed subscheme in $X$, both
the assertions are local in~$X$.
 So they reduce to the case of affine schemes, for which they mean
the following.
 Let $R\rarrow S$ be a surjective homomorphism of commutative rings
and $F$, $M$ be $R$\+modules.
 Then there is a natural homomorphism of $S$\+modules
$$
 (S\ot_RF)\ot_S\Hom_R(S,M)=F\ot_R\Hom_R(S,M)\lrarrow
 \Hom_R(S,\>F\ot_RM)
$$
which is an isomorphism whenever $S$ is a finitely presented $R$\+module
and $F$ is a flat $R$\+module.

 To construct the desired morphism in the general case, let $\cL$ be
an arbitrary quasi-coherent sheaf on $Z$, and let
$\cL\rarrow i^*\F\ot_{\cO_Z}i^!\M$ be a morphism in $Z\qcoh$.
 Applying~$i_*$, we produce a morphism $i_*\cL\rarrow
i_*(i^*\F\ot_{\cO_Z}i^!\M)\simeq\F\ot_{\cO_X}i_*i^!\M$ in $X\qcoh$,
where the isomorphism holds by Lemma~\ref{projection-formula}.
 Composing with the morphism $\F\ot_{\cO_X}i_*i^!\M\rarrow
\F\ot_{\cO_X}\M$ induced by the adjunction morphism $i_*i^!\M\rarrow\M$,
we obtain a morphism $i_*\cL\rarrow\F\ot_{\cO_X}\M$ in $X\qcoh$, which
corresponds by adjunction to a morphism
$\cL\rarrow i^!(\F\ot_{\cO_X}\M)$ in $Z\qcoh$.
\end{proof}

\begin{prop} \label{flat-torsion-tensor-prop}
 Let\/ $\X$ be a reasonable ind-scheme, and let $Z\subset\X$ be
a reasonable closed subscheme with the closed immersion morphism
$i\:Z\rarrow X$.
 Let\/ $\fF$ be flat pro-quasi-coherent pro-sheaf on\/ $\X$ and $\rM$ be
a quasi-coherent torsion sheaf on\/~$\X$.
 Then there is a natural isomorphism $i^!(\fF\ot_\X\rM)\simeq
i^*\fF\ot_{\cO_Z}i^!\rM=\fF^{(Z)}\ot_{\cO_Z}\rM_{(Z)}$ in $Z\qcoh$.
\end{prop}

\begin{proof}
 It follows from Lemma~\ref{shriek-star-tensor} that the rule
$\rN_{(Z)}=\fF^{(Z)}\ot\rM_{(Z)}$ defines a quasi-coherent torsion
sheaf $\rN$ on~$\X$.
 Now it is clear that $\fF\ot_\X\rM\simeq\rN$.
\end{proof}

\begin{exs} \label{countable-flat-pro-sheaves-flat-contramodules}
 (1)~This is a generalization of
Example~\ref{pro-sheaves-not-abelian-ex}.
 Let $\fR$ be complete, separated topological commutative ring with
a countable base of neighborhoods of zero consisting of open ideals,
and let $\X=\Spi\fR$ be the related ind-affine $\aleph_0$\+ind-scheme,
as in Example~\ref{topological-ring-ind-scheme}(2).
 Then the category $\X\pro$ is equivalent to the category of
\emph{separated\/ $\fR$\+contramodules}, as defined
in~\cite[Section~1.2]{Pweak}, \cite[Sections~1.2 and~5]{PR},
or~\cite[Section~6.2]{PS}, and discussed in~\cite[Section~D.1]{Pcosh}.
 The category of separated $\fR$\+contramodules $\fR\separ$ is a full
subcategory in the abelian category of $\fR$\+contramodules
$\fR\contra$; so we have $\Spi\fR\pro\simeq\fR\separ\subset\fR\contra$.

 The equivalence assigns to a (separated) $\fR$\+contramodule $\fC$
the pro-quasi-coherent pro-sheaf $\fP$ with the components
$\fP^{(X_\fI)}=\fC/(\fI\tim\fC)$, and conversely, to
a pro-quasi-coherent pro-sheaf $\fP$ the separated $\fR$\+contramodule
$\fC=\varprojlim_{\fI\subset\fR}\fP^{(X_\fI)}$ is assigned
(where $\fI$ ranges over the open ideals in~$\fR$, and the projective
limit is taken in the category of $\fR$\+contramodules, which agrees
with the projective limit in the category of abelian groups).
 It is clear from~\cite[Lemma~D.1.3]{Pcosh} (see
also~\cite[Lemma~6.3(a,c)]{PR}) that this is an equivalence of
categories.

 The generality level in~\cite[Section~D.1]{Pcosh} is that of
complete, separated topological associative rings with a countable
base of neighborhoods of zero consisting of open two-sided ideals.
 In~\cite[Section~6]{PR}, only a (countable) base of open \emph{right} 
ideals is assumed, which makes the exposition more complicated.

\smallskip
 (2)~The equivalence of categories described in~(1) restricts to
an equivalence between their full subcategories of flat
pro-quasi-coherent pro-sheaves and \emph{flat\/ $\fR$\+contramodules}.
 Notice that any flat $\fR$\+contramodule is
separated~\cite[Corollary~D.1.7]{Pcosh} (see
also~\cite[Corollary~6.15]{PR}).
 The full subcategory of flat contramodules $\fR\flat$ (\emph{unlike}
the larger full subcategory of separated ones) is closed under
extensions in the abelian category of $\fR$\+contramodules $\fR\contra$
\,\cite[Lemma~D.1.5]{Pcosh} (see also~\cite[Corollary~6.13
or Corollary~7.1(b)]{PR}), so it inherits an exact category structure
from the abelian exact category structure on $\fR\contra$.
 The equivalence $\Spi\fR\flat\simeq\fR\flat$ is an equivalence of
exact categories~\cite[Lemma~D.1.4]{Pcosh} (see
also~\cite[Lemma~6.7 or~6.10]{PR}).

\smallskip
 (3)~In the context of~(1), assume that $\fR$ is a reasonable
topological ring in the sense of
Section~\ref{torsion-ind-affine-subsecn}(5).
 Then, according to Section~\ref{torsion-ind-affine-subsecn}(6),
the category of quasi-coherent torsion sheaves $\Spi\fR\tors$
is equivalent to the category of discrete $\fR$\+modules $\fR\discr$.
 This equivalence of categories, together with the equivalence
$\Spi\fR\pro\simeq\fR\separ$ from~(1), transforms the tensor product
functor $\ot_\X\:\Spi\fR\pro\times\Spi\fR\tors\rarrow\Spi\fR\tors$
into the functor of \emph{contratensor product} $\ocn_\fR\:\fR\contra
\times\fR\discr \rarrow\fR\discr$ (restricted to separated
$\fR$\+contramodules).

 We refer to~\cite[Section~D.2]{Pcosh}, \cite[Definition~5.4]{PR},
or~\cite[Section~7.2]{PS} for the definition of the contratensor
product of a discrete module and a contramodule over a topological ring.
 In our context, the contratensor product takes values in
discrete $\fR$\+modules (rather than just abelian groups) because
the ring $\fR$ is commutative.

\smallskip
 (4)~The \emph{contramodule tensor product} functor
$\ot^\fR\:\fR\contra\times\fR\contra\rarrow\fR\contra$ on the category
of contramodules over a commutative topological ring $\fR$ was defined
in~\cite[Section~1.6]{Pweak}.
 According to~\cite[Lemma~D.3.1]{Pcosh} (which presumes a countable
base of neighborhoods of zero in~$\fR$), this functor restricts to
an exact functor $\ot^\fR\:\fR\flat\times\fR\flat\rarrow\fR\flat$,
which agrees with the functor of tensor product of flat
pro-quasi-coherent pro-sheaves $\ot^\X\:\X\flat\times\X\flat\rarrow
\X\flat$ under the equivalence of (exact) categories
$\X\flat\simeq\fR\flat$.
\end{exs}

\subsection{Coproducts and colimits}
\label{colimits-of-pro-sheaves-subsecn}
 Let $\X$ be an ind-scheme.
 The coproduct of any family of objects exists in the category
$\X\pro$.
 Specifically, let $(\fP_\theta)_{\theta\in\Theta}$ be a family of
pro-quasi-coherent pro-sheaves on~$X$, indexed by a set~$\Theta$.
 For every closed subscheme $Z\subset\X$, put $\fQ^{(Z)}=
\bigoplus_{\theta\in\Theta}\fP_\theta^{(Z)}$ (where the direct sum
is taken in the category of quasi-coherent sheaves on~$Z$).
 Then the collection of quasi-coherent sheaves $\fQ^{(Z)}$ with
the obvious isomorphisms $i_{ZY}^*\fQ^{(Y)}\simeq\fQ^{(Z)}$ for
$Z\subset Y\subset\X$ is a pro-quasi-coherent pro-sheaf $\fQ$ on~$\X$.
 One has $\fQ=\coprod_{\theta\in\Theta}\fP_\theta$ in the category
$\X\pro$.

 More generally, the colimit of any diagram of objects, indexed by
a small category, exists in $\X\pro$.
 Let $(\fP_\theta)_{\theta\in\Theta}$ be an inductive system of
pro-quasi-coherent pro-sheaves on~$X$, indexed by a small
category~$\Theta$.
 For every closed subscheme $Z\subset\X$, put $\fQ^{(Z)}=
\varinjlim_{\theta\in\Theta}\fP_\theta^{(Z)}$ (where the colimit
is taken in the category of quasi-coherent sheaves $Z\qcoh$).
 Notice that the inverse image functor with respect to any morphism
of schemes preserves all colimits of quasi-coherent sheaves.
 So the collection of quasi-coherent sheaves $\fQ^{(Z)}$ with
the induced isomorphisms $i_{ZY}^*\fQ^{(Y)}\simeq\fQ^{(Z)}$ for
$Z\subset Y\subset\X$ is a pro-quasi-coherent pro-sheaf $\fQ$ on~$\X$.
 One has $\fQ=\varinjlim_{\theta\in\Theta}\fP_\theta$ in $\X\pro$.

 Both the tensor product functor $\ot^\X\:\X\pro\times\X\pro\rarrow
\X\pro$ (on any ind-scheme~$\X$) and the action functor $\ot_\X\:
\X\pro\times\X\tors\rarrow\X\tors$ (on a reasonable ind-scheme~$\X$)
preserve coproducts in both the categories.
 In fact, they even preserve all colimits (as the action functor
for $\Gamma$\+systems $\ot_\X\:\X\pro\times(\X,\Gamma)\syst\rarrow
(\X,\Gamma)\syst$ can be easily seen to preserve colimits).

 The full subcategory $\X\flat$ is closed under coproducts in
$\X\pro$.
 In fact, it is closed under all direct limits (of inductive systems
indexed by directed posets), because flatness of quasi-coherent
sheaves is preserved by such direct limits.
 So all coproducts and direct limits exist in $\X\flat$.
 The direct limit (in particular, coproduct) functors are exact in
the exact category $\X\flat$.

 The inverse image functor $f^*\:\X\pro\rarrow\Y\pro$ with respect
to any morphism of ind-schemes $f\:\Y\rarrow\X$ preserves all colimits.
 So does the direct image functor $f_*\:\Y\pro\rarrow\X\pro$ with
respect to an affine morphism of ind-schemes~$f$.
 In particular, both the functors preserve coproducts.

\subsection{Pro-quasi-coherent commutative algebras}
\label{pro-qcoh-algebras-subsecn}
 Let $X$ be a scheme.
 A \emph{quasi-coherent sheaf of algebras} (or a \emph{quasi-coherent
algebra} for brevity) $\cA$ over $X$ is an (associative and unital)
algebra object in the tensor category $X\qcoh$.
 In other words, $\cA$ is a quasi-coherent sheaf on $X$ endowed
with morphisms of quasi-coherent sheaves $\cO_X\rarrow\cA$ and
$\cA\ot_{\cO_X}\cA\rarrow\cA$ satisfying the usual associativity
and unitality equations.
 The underlying sheaf of abelian groups of a quasi-coherent algebra
has a natural structure of a sheaf of rings.

 Let $f\:Y\rarrow X$ be a morphism of (concentrated) schemes.
 Then the inverse image $f^*\:X\qcoh\rarrow Y\qcoh$ is a tensor
functor, so the inverse image $f^*\cA$ of any quasi-coherent sheaf
of algebras $\cA$ on $X$ is a quasi-coherent sheaf of algebras
on~$Y$.
 
 Furthermore, for any quasi-coherent sheaves $\M$ and $\N$ on $Y$,
there is a natural morphism of quasi-coherent sheaves
$f_*\M\ot_{\cO_X}f_*\N\rarrow f_*(\M\ot_{\cO_Y}\nobreak\N)$ on $X$,
corresponding by adjunction to the morphism of quasi-coherent sheaves
$f^*(f_*\M\ot_{\cO_X}\nobreak f_*\N)\simeq
f^*f_*\M\ot_{\cO_Y}f^*f_*\N\rarrow\M\ot_{\cO_Y}\N$ on $Y$ induced by
the adjunction morphisms $f^*f_*\M\rarrow\M$ and $f^*f_*\N\rarrow\N$.
 In particular, for any quasi-coherent sheaf of algebras $\cB$ on $Y$,
the composition $f_*\cB\ot_{\cO_X}f_*\cB\rarrow f_*(\cB\ot_{\cO_Y}\cB)
\rarrow f_*\cB$ endows the direct image $f_*\cB$
of the quasi-coherent sheaf $\cB$ with the structure of
a quasi-coherent sheaf of algebras on~$X$.

 A \emph{quasi-coherent commutative algebra} (or a \emph{quasi-coherent
sheaf of commutative algebras}) on $X$ is a commutative (associative,
and unital) algebra object in $X\qcoh$.
 Clearly, the inverse and direct images preserve commutativity of
quasi-coherent algebras.
 The next lemma is quite standard and well-known.

\begin{lem} \label{affine-morphisms-qcoh-algebras}
 For any scheme $X$, there is a natural anti-equivalence between
the category of schemes $Y$ endowed with an affine morphism
$Y\rarrow X$ and the category of quasi-coherent commutative
algebras on~$X$.
 To an affine morphism $f\:Y\rarrow X$, the quasi-coherent commutative
algebra $f_*\cO_Y$ on $X$ is assigned.
\end{lem}

\begin{proof}
 The inverse functor assigns to a quasi-coherent sheaf of algebras
$\cA$ on $X$ the scheme $Y$ covered by the following affine schemes.
 For any affine open subscheme $U\subset X$, the affine scheme
$\Spec\cA(U)$ is an open subscheme in~$Y$; in fact,
one has $\Spec\cA(U)=U\times_XY$, so $\Spec\cA(U)$ is the preimage
of $U$ in $Y$ under~$f$.
 The map of the rings of sections $\cO_X(U)\rarrow\cA(U)$ induced by
the unit morphism $\cO_X\rarrow\cA$ of the quasi-coherent sheaf of
algebras $\cA$ induces the morphism of schemes $\Spec\cA(U)\rarrow U$.
 For any pair of affine open subschemes $U$ and $V$ in $X$ such that
$V\subset U$, the restriction map of the rings of sections
$\cA(U)\rarrow\cA(V)$ of the sheaf of rings $\cA$ on $X$ induces
an open immersion $\Spec\cA(V)\rarrow\Spec\cA(U)$; this is the map
$V\times_XY\rarrow U\times_XY$ induced by the open immersion
$V\rarrow U$.
 The key observation is that $V\times_U\Spec\cA(U)\simeq\Spec\cA(V)$
(as it should be), since $\cO_X(V)\ot_{\cO_X(U)}\cA(U)\simeq\cA(V)$.
 We leave further details to the reader.
\end{proof}

 We will denote the affine scheme over $X$ corresponding to
a quasi-coherent algebra $\cA$ under the construction of
Lemma~\ref{affine-morphisms-qcoh-algebras} by $Y=\Spec_X\cA$.
 Notice that an affine morphism of schemes $f\:Y\rarrow X$ is flat
if and only if the quasi-coherent sheaf $\cA=f_*\cO_Y$ on $X$ is flat.

\begin{lem} \label{qcoh-algebra-pullback}
 Let $h\:Z\rarrow X$ be a morphism of schemes and $\cA$ be
a quasi-coherent commutative algebra over~$X$.
 Then there is a natural isomorphism of schemes\/ $\Spec_Zh^*\cA
\simeq Z\times_X\Spec_X\cA$.
 In other words, there is a natural pullback diagram of schemes
$$
\xymatrix{
 \Spec_Z h^*\cA \ar[rr]^-k \ar[d]^-g && \Spec_X\cA \ar[d]^-f \\
 Z \ar[rr]^-h && X
}
$$
where $f$ and~$g$ are the structure morphisms of the schemes\/
$\Spec_X\cA$ and\/ $\Spec_Z h^*\cA$ affine over the schemes
$X$ and $Z$, respectively.
\end{lem}

\begin{proof}
 The question is essentially local in $X$ and $Z$, so it reduces to
the case of affine schemes, for which it is obvious from
the constructions.
\end{proof}

 Let $\cA$ be a quasi-coherent algebra on a scheme~$X$.
 Then a \emph{quasi-coherent module} (or a \emph{quasi-coherent sheaf
of modules}) over $\cA$ is a module object over the algebra object
$\cA$ in the tensor category $X\qcoh$.
 In other words, a quasi-coherent module $\M$ over $\cA$ is
a quasi-coherent sheaf $\M\in X\qcoh$ endowed with a morphism of
quasi-coherent sheaves $\cA\ot_{\cO_X}\M\rarrow\M$ satisfying the usual
associativity and unitality equations.
 The underlying sheaf of abelian groups of a quasi-coherent module $\M$
over $\cA$ is a sheaf of modules over the underlying sheaf of rings
of~$\cA$.

 Let $f\:Y\rarrow X$ be a morphism of schemes.
 Then, for any quasi-coherent module $\M$ over a quasi-coherent
algebra $\cA$ on $X$, the inverse image $f^*\M$ is a quasi-coherent
module over the quasi-coherent algebra $f^*\cA$ on~$Y$.
 For any quasi-coherent module $\N$ over a quasi-coherent
algebra $\cB$ on $Y$, the direct image $f_*\N$ is a quasi-coherent
module over the quasi-coherent algebra~$f_*\cB$ on~$X$.
 The constructions are similar to the above constructions for algebras.

 Furthermore, let $i\:Z\rarrow X$ be a closed immersion of schemes,
and let $\M$ be a quasi-coherent module over a quasi-coherent
algebra $\cA$ on~$X$.
 Then the composition $i^*\cA\ot_{\cO_Z}i^!\M\rarrow
i^!(\cA\ot_{\cO_X}\M)\rarrow i^!\M$ of the morphism
$i^*\cA\ot_{\cO_Z}i^!\M\rarrow i^!(\cA\ot_{\cO_X}\nobreak\M)$ from
Lemma~\ref{shriek-star-tensor} and the morphism
$i^!(\cA\ot_{\cO_X}\M)\rarrow i^!\M$ induced by the structure morphism
$\cA\ot_{\cO_X}\M\rarrow\M$ endows the quasi-coherent sheaf $i^!\M$
with the structure of a quasi-coherent module over the quasi-coherent
algebra $i^*\cA$ on~$Z$.

\begin{lem} \label{qcoh-modules-via-qcoh-algebras-described}
 Let $f\:Y\rarrow X$ be an affine morphism of schemes.
 Then the category of quasi-coherent sheaves on $Y$ is equivalent to
the category of quasi-coherent modules over the quasi-coherent algebra
$f_*\cO_Y$ on~$X$.
 To a quasi-coherent sheaf\/ $\N$ on $Y$, the quasi-coherent
$f_*\cO_Y$\+module $f_*\N\in X\qcoh$ is assigned.
\end{lem}

\begin{proof}
 Put $\cA=f_*\cO_Y$.
 The inverse functor assigns to a quasi-coherent $\cA$\+module $\M$
on $X$ the following quasi-coherent sheaf $\N$ on~$Y$.
 For any affine open subscheme $U\subset X$, the restriction of $\N$ to
the affine open subscheme $U\times_XY=\Spec\cA(U)$ in $Y$ is
the quasi-coherent sheaf on $\Spec\cA(U)$ corresponding to
the $\cA(U)$\+module~$\M(U)$.
 We omit further details.
\end{proof}

 Let $\X$ be an ind-scheme.
 A \emph{pro-quasi-coherent pro-sheaf of algebras} (or
a \emph{pro-quasi-coherent algebra}) $\fA$ over $\X$ is algebra
object in the tensor category $\X\pro$.
 In other words, $\fA$ is a pro-quasi-coherent pro-sheaf on $\X$
endowed with morphisms of pro-quasi-coherent pro-sheaves
$\fO_\X\rarrow\fA$ and $\fA\ot_\X\fA\rarrow\fA$ satisfying the usual
(associativity and unitality, and if mentioned, commutativity)
equations.

 It is clear from the construction of the tensor product of
pro-quasi-coherent pro-sheaves in Section~\ref{pro-sheaves-subsecn}
that the datum of a pro-quasi-coherent algebra $\fA$ on $\X$ is
equivalent to the datum of a quasi-coherent algebra $\fA^{(Z)}$ on
every closed subscheme $Z\subset\X$ together with a compatible
system of isomorphisms of quasi-coherent algebras $\fA^{(Z')}\simeq
i_{Z'Z''}^*\fA^{(Z'')}$ for every pair of closed subschemes $Z'
\subset Z''\subset\X$ with the closed immersion
$i_{Z'Z''}\:Z''\rarrow Z''$.

 Let $f\:\Y\rarrow\X$ be a morphism of ind-schemes and $\fA$ be
a pro-quasi-coherent algebra on~$\X$.
 Then the inverse image $f^*\:\X\pro\rarrow\Y\pro$ is a tensor
functor, hence $f^*\fA$ is a pro-quasi-coherent algebra on~$\Y$.

 Let $f:\Y\rarrow\X$ be an affine morphism of ind-schemes, and let
$\fB$ be a pro-quasi-coherent algebra on~$\Y$.
 Then an adjunction argument similar to the above one for schemes
shows how to construct an algebra structure on the pro-quasi-coherent
pro-sheaf $f_*\fB$ on~$\X$.

\begin{prop} \label{ind-schemes-affine-morphisms-pro-qcoh-algebras}
 For any ind-scheme\/ $\X$, there is a natural anti-equivalence between
the category of ind-schemes\/ $\Y$ endowed with an affine morphism\/
$\Y\rarrow\X$ and the category of pro-quasi-coherent commutative
algebras on\/~$\X$.
 To an affine morphism $f\:\Y\rarrow\X$, the pro-quasi-coherent
commutative algebra $f_*\fO_\Y$ on\/ $\X$ is assigned.
\end{prop}

\begin{proof}
 Let $\X=\ilim_{\gamma\in\Gamma}X_\gamma$ be an inductive system of
closed immersions of schemes representing\/~$\X$.
 The inverse functor to the one mentioned in the proposition assigns
to a pro-quasi-coherent algebra $\fA$ on $\X$ the ind-scheme $\Y$
represented by the inductive system of closed immersions of schemes
$\Y=\ilim_{\gamma\in\Gamma}Y_\gamma$, where $Y_\gamma=\Spec_{X_\gamma}
\fA^{(X_\gamma)}$ (see Lemma~\ref{affine-morphisms-qcoh-algebras}).
 The transition morphisms $Y_\gamma\rarrow Y_\delta$ for
$\gamma<\delta\in\Gamma$ are provided by
Lemma~\ref{qcoh-algebra-pullback}, and it is also clear from
Lemma~\ref{qcoh-algebra-pullback} that the morphism
of ind-schemes $\Y\rarrow\X$ constructed in this way is affine.
\end{proof}

 We will denote the affine ind-scheme over $\X$ corresponding to
a pro-quasi-coherent algebra $\fA$ under the construction of
Proposition~\ref{ind-schemes-affine-morphisms-pro-qcoh-algebras}
by $\Y=\Spi_\X\fA$.
 Notice that an affine morphism of ind-schemes $f\:\Y\rarrow\X$ is
flat if and only the pro-quasi-coherent pro-sheaf $\fA=f_*\fO_Y$
on $\X$ is flat.

 Let $\fA$ be a pro-quasi-coherent algebra over an ind-scheme~$\X$.
 Then a \emph{pro-quasi-coherent module} (or a \emph{pro-quasi-coherent
pro-sheaf of modules}) over $\fA$ is a module over the algebra
object $\fA$ in the tensor category $\X\pro$.
 In other words, a pro-quasi-coherent module $\fM$ over $\fA$ is
a pro-quasi-coherent pro-sheaf $\fM\in\X\pro$ endowed with a morphism
of pro-quasi-coherent pro-sheaves $\fA\ot^\X\fM\rarrow\fM$ satisfying
the usual associativity and unitality equations.

 The datum of a pro-quasi-coherent module $\fM$ over
a given pro-quasi-coherent algebra $\fA$ on $\X$ is equivalent to
the datum of a quasi-coherent module $\fM^{(Z)}$ over
the quasi-coherent algebra $\fA^{(Z)}$ on every closed subscheme
$Z\subset\X$ together with a compatible system of isomorphisms of
quasi-coherent modules $\fM^{(Z')}\simeq i_{Z'Z''}^*\fM^{(Z'')}$ for
every pair of closed subschemes $Z'\subset Z''$.

 Let $\X$ be a reasonable ind-scheme and $\fA$ be a pro-quasi-coherent
algebra over~$\X$.
 Then a \emph{quasi-coherent torsion module} (or a \emph{quasi-coherent
torsion sheaf of modules}) over $\fA$ is a module object in the module
category $\X\tors$ over the algebra object $\fA$ in the tensor category
$\X\pro$.
 In other words, a quasi-coherent torsion module $\rM$ over $\fA$ is
a quasi-coherent torsion sheaf $\rM\in\X\tors$ endowed with a morphism
of quasi-coherent torsion sheaves $\fA\ot_\X\rM\rarrow\rM$ satisfying
the usual associativity and unitality equations.

 Let $\X=\ilim_{\gamma\in\Gamma}X_\gamma$ be a representation of $\X$
by an inductive system of closed immersions of reasonable closed
subschemes.
 Then, for any quasi-coherent sheaf $\rM$ on $\X$, one has
\begin{multline*}
 \Hom_{\X\tors}(\fA\ot_\X\rM,\>\rM)\simeq
 \Hom_{\X\tors}((\fA\ot_\X\rM|_\Gamma)^+,\>\rM) \\ \simeq
 \Hom_{(\X,\Gamma)\syst}(\fA\ot_\X\rM|_\Gamma,\>\rM|_\Gamma).
\end{multline*}
 Hence the datum of a quasi-coherent torsion module $\rM$ over a given
quasi-coherent algebra $\fA$ on $\X$ is equivalent to the datum of
a quasi-coherent module $\rM_{(X_\gamma)}$ over the quasi-coherent
algebra $\fA^{(X_\gamma)}$ on the scheme $X_\gamma$ for every
$\gamma\in\Gamma$, together with a compatible system of
quasi-isomorphisms of quasi-coherent modules $\rM_{(X_\gamma)}
\simeq i_{\gamma\delta}^!\rM_{(X_\delta)}$ for every $\gamma<\delta
\in\Gamma$ (where $i_{\gamma\delta}\:X_\gamma\rarrow X_\delta$ is
the closed immersion).

 Let $f\:\Y\rarrow\X$ be a morphism of ind-schemes.
 Let $\fA$ be a pro-quasi-coherent algebra on $\X$ and $\fM$ be
a pro-quasi-coherent module over~$\fA$.
 Then the inverse image $f^*\fM$ is a pro-quasi-coherent module over
the pro-quasi-coherent algebra $f^*\fM$ on~$\Y$.
 Furthermore, assume that  $f$~is ``representable by schemes'' and
$\X$ is a reasonable ind-scheme, and let $\rM$ be a quasi-coherent 
torsion module over~$\fA$.
 Then it is clear from
Lemma~\ref{inverse-images-preserve-torsion-pro-tensor} that
the inverse image $f^*\rM$ is a quasi-coherent torsion module
over the pro-quasi-coherent algebra~$f^*\fA$.

 Let $f\:\Y\rarrow\X$ be an affine morphism of ind-schemes.
 Let $\fB$ be a pro-quasi-coherent algebra on $\Y$ and $\fN$ be
a pro-quasi-coherent module over~$\fB$.
 Then the direct image $f_*\fN$ is a pro-quasi-coherent module over
the pro-quasi-coherent algebra $f_*\fB$ on~$\X$.
 Furthermore, assume that $\X$ is a reasonable ind-scheme, and let
$\rN$ be a quasi-coherent torsion modules over~$\fB$.
 Then the direct image $f_*\rN$ is a quasi-coherent torsion module
over the pro-quasi-coherent algebra~$f_*\fB$.

\begin{prop} \label{pro-torsion-modules-via-pro-algebras-described}
 Let\/ $f\:\Y\rarrow\X$ be an affine morphism of ind-schemes.
 Then \par
\textup{(a)} the category of pro-quasi-coherent pro-sheaves on\/ $\Y$
is equivalent to the category of pro-quasi-coherent modules over
the pro-quasi-coherent algebra $f_*\fO_\Y$ on\/~$\X$; \par
\textup{(b)} if the ind-scheme\/ $\X$ is reasonable, then the category
of quasi-coherent torsion sheaves on\/ $\Y$ is equivalent to
the category of quasi-coherent torsion modules over
the pro-quasi-coherent algebra $f_*\fO_\Y$ on\/~$\X$.
\end{prop}

\begin{proof}
 In part~(a), to a pro-quasi-coherent pro-sheaf $\fN$ on $\Y$,
the pro-quasi-coherent $f_*\fO_\Y$\+module $f_*\fN$ is assigned.
 To construct the inverse functor, let $\fM$ be a pro-quasi-coherent
module over $f_*\fO_\Y$ on $\X$, and let $Z\subset\X$ be a closed
subscheme; put $W=Z\times_\X\Y$, and denote by $f_Z\:W\rarrow Z$
the natural morphism.
 Then, for the corresponding pro-quasi-coherent pro-sheaf $\fN$ on $\Y$,
the quasi-coherent sheaf $\fN^{(W)}\in W\qcoh$ corresponds to
the quasi-coherent module $\fM^{(Z)}$ over the quasi-coherent
algebra $f_Z{}_*\cO_W\simeq (f_*\fO_\Y)^{(Z)}$ on $Z$ under
the equivalence of categories from
Lemma~\ref{qcoh-modules-via-qcoh-algebras-described}.
 One can use Lemma~\ref{affine-flat-base-change}(a) to construct
the compatibility isomorphisms related to pairs of closed subschemes
$Z'\subset Z''\subset\X$.

 In part~(b), to a quasi-coherent torsion sheaf $\rN$ on $Y$,
the quasi-coherent torsion module $f_*\rN$ over the pro-quasi-coherent
algebra $f_*\fO_\Y$ is assigned.
 To construct the inverse functor, let $\rM$ be a quasi-coherent torsion
module over the quasi-coherent algebra $f_*\fO_Y$ on $\X$, and let
$Z\subset\X$ be a reasonable closed subscheme.
 We keep the same notation $f_Z\:W\rarrow Z$ as above.
 Then, for the corresponding quasi-coherent torsion sheaf $\rN$ on $\Y$,
the quasi-coherent sheaf $\rN_{(W)}\in W\qcoh$ corresponds to
the quasi-coherent module $\rM_{(Z)}$ over the quasi-coherent algebra
$f_Z{}_*\cO_W\simeq (f_*\fO_\Y)^{(Z)}$ on~$Z$.
 One can use Lemma~\ref{reasonable-base-change}(a) to construct
the compatibility isomorphisms related to pairs of reasonable
closed subschemes $Z'\subset Z''\subset\X$.
\end{proof}

\Section{Dualizing Complexes on Ind-Noetherian Ind-Schemes}
\label{dualizing-on-ind-Noetherian-secn}

 For any additive category $\sA$, we denote by $\sC(\sA)$ the category
of complexes in $\sA$ and by $\sK(\sA)$ the homotopy category of
(complexes in)~$\sA$.
 The notation $\sC^+(\sA)$, \ $\sC^-(\sA)$, and $\sC^\bb(\sA)\subset
\sC(\sA)$ stands for the categories of bounded below, bounded
above, and bounded (on both sides) complexes in~$\sA$, as usual; and
similarly for $\sK^+(\sA)$, \ $\sK^-(\sA)$, and $\sK^\bb(\sA)\subset
\sK(\sA)$. 
 For an abelian (or exact) category $\sA$, the full subcategory of
injective objects in $\sA$ is denoted by $\sA_\inj\subset\sA$.

\subsection{Ind-Noetherian ind-schemes}
 An ind-scheme is said to be \emph{ind-Noetherian} if it can be
represented by an inductive system of Noetherian schemes.
 It follows from Lemma~\ref{closed-immersion-two-morphisms-lemma}(a)
that any closed subscheme of a (strict) ind-Noetherian ind-scheme is
Noetherian (since any locally closed subscheme of a Noetherian scheme
is Noetherian).
 Thus any ind-Noetherian ind-scheme can be represented by an inductive
system of closed immersions of Noetherian schemes.

 Let $S$ be a Noetherian scheme and $\X\rarrow S$ be a morphism of
ind-schemes.
 One says that $\X$ is an ind-scheme of \emph{ind-finite type} over
$S$ if $\X$ can be represented by an inductive system of schemes of
finite type over~$S$.
 Similarly to the previous paragraph, any closed subscheme of
an ind-scheme of ind-finite type is a scheme of finite type
(over the Noetherian base scheme~$S$).
 Thus any ind-scheme of ind-finite type can be represented by
an inductive system of closed immersions of schemes of finite type.

 Let $\kk$~be a Noetherian commutative ring.
 Speaking about schemes of finite type and ind-schemes of ind-finite
type over $\Spec\kk$, we will say simply ``over~$\kk$'' instead of
``over $\Spec\kk$'', for brevity.

\begin{exs} \label{ind-Artinian-ind-schemes}
 (1)~A scheme $X$ is said to be \emph{Artinian} if it has a finite open
covering by spectra of Artinian rings; equivalently, this means that
$X$ is the spectum of an Artinian ring.
 So any Artinian scheme is affine and Noetherian.

 An ind-scheme is said to be \emph{ind-Artinian} if it can be
represented by an inductive system of Artinian schemes.
 Any closed subscheme of an ind-Artinian ind-scheme is Artinian;
so an ind-Artinian ind-scheme can be represented by an inductive system
of closed immersions of Artinian schemes.
 Any ind-Artinian ind-scheme is ind-affine and ind-Noetherian.

\smallskip
 (2)~Let $\kk$~be a field.
 Notice that any Artinian scheme of finite type over~$\kk$ is finite
over~$\kk$ (in other words, any finitely generated commutative
Artinian $\kk$\+algebra is finite-dimensional over~$\kk$).

 The construction of Example~\ref{coalgebra-ind-scheme}(2)
establishes an equivalence between the category of ind-Artinian
ind-schemes of ind-finite type over~$\kk$ and the category of
cocommutative coalgebras over~$\kk$.
 To a (coassociative, counital) cocommutative coalgebra $\rC$
over~$\kk$, the ind-Artinian ind-scheme of ind-finite type $\Spi\rC^*$
is assigned.

\smallskip
 (3)~The ind-schemes $\Spi\widehat\boZ$ and $\Spi\boZ_p$ from
Examples~\ref{Z-hat-ind-scheme} are ind-Artinian.

 Quite generally, the category of ind-Artinian ind-schemes is
anti-equivalent to the category of \emph{pro-Artinian topological
commutative rings} in the sense of~\cite[Section~1.1]{Pweak}.
 The functor $\fR\longmapsto\Spi\fR$ from
Example~\ref{topological-ring-ind-scheme}(1) restricts to
the desired anti-equivalence.
 It is important here that, for any directed projective system of
$(R_\gamma)_{\gamma\in\Gamma}$ of Artinian commutative rings
$R_\gamma$ and surjective maps between them, the projection map
$\fR=\varprojlim_{\gamma\in\Gamma}R_\gamma\rarrow R_\delta$ is
surjective for all $\delta\in\Gamma$ \,\cite[Corollary~A.2.1]{Pweak}.
\end{exs}

\subsection{Definition of a dualizing complex}
\label{dualizing-definition-subsecn}
 Notice that any closed subscheme of a Noetherian scheme is reasonable
(in the sense of Section~\ref{reasonable-subsecn}).
 Hence any ind-Noetherian ind-scheme is reasonable, and any closed
subscheme of an ind-Noetherian ind-scheme is reasonable.

 Let $X$ be a scheme.
 Recall the notation $\cHom_{\cO_X}({-},{-})$ for the internal Hom of
sheaves of $\cO_X$\+modules (see
Section~\ref{qcoh-sheaves-and-functors-subsecn}).
 The \emph{quasi-coherent internal Hom} of quasi-coherent sheaves on
$X$ is defined as follows.
 For any two quasi-coherent sheaves $\M$ and $\N\in X\qcoh$,
the quasi-coherent sheaf $\cHom_{X\qc}(\M,\N)\in X\qcoh$ is the object
for which a natural isomorphism of the abelian groups of morphisms
$$
 \Hom_{X\qcoh}(\cL,\cHom_{X\qc}(\M,\N))\simeq
 \Hom_{X\qcoh}(\cL\ot_{\cO_X}\M,\>\N)
$$
holds for all $\cL\in X\qcoh$.
 The quasi-coherent sheaf $\cHom_{X\qc}(\M,\N)$ can be obtained by
applying the \emph{coherator} functor~\cite[Sections~B.12--B.14]{TT}
to the sheaf of $\cO_X$\+modules $\cHom_{\cO_X}(\M,\N)$.
 In particular, one has $\cHom_{X\qc}(\M,\N)=\cHom_{\cO_X}(\M,\N)$
whenever the sheaf of $\cO_X$\+modules $\cHom_{\cO_X}(\M,\N)$ is
quasi-coherent (e.~g., this holds when the scheme $X$ is Noetherian
and the sheaf $\M$ is coherent).

 For any two complexes $\M^\bu$ and $\N^\bu$ of quasi-coherent sheaves
on $X$, the complex $\cHom_{X\qc}(\M^\bu,\N^\bu)$ of quasi-coherent
sheaves on $X$ is constructed by totalizing the bicomplex of
quasi-coherent sheaves with the components $\cHom_{X\qc}(\M^p,\N^q)$,
\ $p$, $q\in\boZ$, by taking infinite products in the Grothendieck
category of quasi-coherent sheaves $X\qcoh$ along the diagonals of
the bicomplex.

\begin{lem} \label{qcoh-internal-Hom-projection}
 Let $f\:Y\rarrow X$ be a morphism of schemes, $\M$ be a quasi-coherent
sheaf on $X$, and $\N$ be a quasi-coherent sheaf on~$Y$.
 Then there is a natural isomorphism of quasi-coherent sheaves on~$X$
$$
 f_*\cHom_{Y\qc}(f^*\M,\N)\simeq
 \cHom_{X\qc}(\M,f_*\N).
$$
\end{lem}

\begin{proof}
 Let $\cL$ be an arbitrary quasi-coherent sheaf on~$X$.
 Then we have
\begin{multline*}
 \Hom_X(\cL,f_*\cHom_{Y\qc}(f^*\M,\N))\simeq
 \Hom_Y(f^*\cL,\cHom_{Y\qc}(f^*\M,\N)) \\ \simeq
 \Hom_Y(f^*\cL\ot_{\cO_Y}f^*\M,\>\N)\simeq
 \Hom_Y(f^*(\cL\ot_{\cO_X}\M),\>\N) \\ \simeq
 \Hom_X(\cL\ot_{\cO_X}\M,\>f_*\N)\simeq
 \Hom_X(\cL,\cHom_{X\qc}(\M,f_*\N)),
\end{multline*}
where $\Hom_X({-},{-})$ is a shorthand notation for the abelian group
$\Hom_{X\qcoh}({-},{-})$, and similarly for~$Y$.
\end{proof}

\begin{lem} \label{qcoh-internal-Hom-of-complexes-projection}
 Let $f\:Y\rarrow X$ be a morphism of schemes, $\M^\bu$ be
a complex of quasi-coherent sheaves on $X$, and $\N^\bu$ be a complex
of quasi-coherent sheaves on~$Y$.
 Then there is a natural isomorphism of complexes of quasi-coherent
sheaves on~$X$
$$
 f_*\cHom_{Y\qc}(f^*\M^\bu,\N^\bu)\simeq
 \cHom_{X\qc}(\M^\bu,f_*\N^\bu).
$$
\end{lem}

\begin{proof}
 Follows from Lemma~\ref{qcoh-internal-Hom-projection} and the fact
that the direct image functor~$f_*$, being a right adjoint, preserves
infinite products of quasi-coherent sheaves.
\end{proof}

\begin{lem} \label{hom-tensor-flats-injectives}
\textup{(a)} For any injective quasi-coherent sheaf $\J$ over
a quasi-compact semi-separated scheme $X$, the functor\/
$\cHom_{X\qc}({-},\J)\:X\qcoh\rarrow X\qcoh$ is exact. \par
\textup{(b)} For any flat quasi-coherent sheaf $\F$ and injective
quasi-coherent sheaf $\J$ over a Noetherian scheme $X$,
the quasi-coherent sheaf $\F\ot_{\cO_X}\J$ is injective. \par
\textup{(c)} For any flat quasi-coherent sheaf $\F$ and injective
quasi-coherent sheaf $\J$ over a scheme $X$, the quasi-coherent
sheaf\/ $\cHom_{X\qc}(\F,\J)$ is injective. \par
\textup{(d)} For any injective quasi-coherent sheaves $\J'$ and $\J$
over a semi-separated Noetherian scheme $X$, the quasi-coherent
sheaf\/ $\cHom_{X\qc}(\J',\J)$ is flat.
\end{lem}

\begin{proof}
 This is~\cite[Lemma~8.7]{Mur} or~\cite[Lemma~2.5]{EP}.
\end{proof}

 Let $X$ be a semi-separated Noetherian scheme.
 For us, a \emph{dualizing complex} $\D^\bu$ on $X$ is a complex of
injective quasi-coherent sheaves, $\D^\bu\in\sC(X\qcoh_\inj)$,
satisfying the following conditions:
\begin{enumerate}
\renewcommand{\theenumi}{\roman{enumi}}
\item the complex $\D^\bu$ is homotopy equivalent to a bounded
complex of injective quasi-coherent sheaves on~$X$;
\item the cohomology sheaves of the complex $\D^\bu$ are coherent
sheaves on~$X$;
\item the natural morphism of complexes of quasi-coherent sheaves
$\cO_X\rarrow\cHom_{X\qc}(\D^\bu,\D^\bu)$ is a quasi-isomorphism
(of complexes in the abelian category $X\qcoh$).
\end{enumerate}

 This definition is equivalent to the one in~\cite[Section~V.2]{Hart},
with the only difference that a dualizing complex is viewed as
a derived category object in~\cite{Hart}, while we presume a complex
of injectives representing this derived category object to be chosen.
 The complex of injectives $\D^\bu$ does not have to be bounded, but
it must be homotopy equivalent to a bounded complex of injectives.
 In view of the semi-separatedness assumption in
Lemma~\ref{hom-tensor-flats-injectives}(a), we are imposing
the semi-separatedness assumption on the scheme $X$ in the definition
above in order to be able to use the quasi-coherent internal Hom in
the formulation of condition~(iii).
 In particular, being a dualizing complex is a local property of
a complex of quasi-coherent sheaves.

\begin{lem} \label{support-restriction-of-dualizing-complex}
 Let $i\:Z\rarrow X$ be a closed immersion of semi-separated Noetherian
schemes and $\D^\bu\in\sC(X\qcoh_\inj)$ be a dualizing complex on~$X$.
 Then $i^!\D^\bu\in\sC(Z\qcoh_\inj)$ is a dualizing complex on~$Z$.
\end{lem}

\begin{proof}
 First of all, the functor $i^!\:X\qcoh\rarrow Z\qcoh$ is right adjoint
to an exact functor~$i_*$; so $i^!$~takes injectives to injectives.
 The rest is~\cite[Proposition~V.2.4]{Hart}.
 (Cf.\ Lemmas~\ref{star-qc-Hom-shriek-lemma}
and~\ref{dualizing-homothety-quasi-isomorphism-as-flats} below.)
\end{proof}

 An ind-scheme is said to be \emph{ind-semi-separated} if it can be
represented by an inductive system of semi-separated schemes.
 It follows from Lemma~\ref{closed-immersion-two-morphisms-lemma}(a)
that any closed subscheme of an ind-semi-separated ind-scheme is
ind-semi-separated (since any locally closed subscheme of
a semi-separated scheme is semi-separated).
 Thus any ind-semi-separated ind-scheme can be represented by
an inductive system of closed immersions of semi-separated schemes.
 Moreover, any ind-semi-separated ind-Noetherian ind-scheme can be
represented by an inductive system of closed immersions of
semi-separated Noetherian schemes.

 Similarly, an ind-scheme is said to be \emph{ind-separated} if
it can be represented by an inductive system of separated schemes.
 Any closed subscheme of an ind-separated ind-scheme is ind-separated.

 Let $\X$ be an ind-semi-separated ind-Noetherian ind-scheme.
 A \emph{dualizing complex} $\rD^\bu$ on $\X$ is a complex of
quasi-coherent torsion sheaves, $\rD^\bu\in\sC(\X\tors)$, satisfying
the following condition:
\begin{enumerate}
\renewcommand{\theenumi}{\roman{enumi}}
\setcounter{enumi}{3}
\item for every closed subscheme $Z\subset\X$ with the closed immersion
morphism $i\:Z\rarrow\X$, the complex $i^!\rD^\bu\in\sC(Z\qcoh)$ is
a dualizing complex on~$Z$.
\end{enumerate}
 Here the functor $i^!\:\X\tors\rarrow Z\qcoh$ is applied to the complex
$\rD^\bu\in\sC(\X\tors)$ termwise (no derived functor is presumed
in our notation).
 In view of condition~(i) for $i^!\rD^\bu\in\sC(Z\qcoh)$ and
Proposition~\ref{torsion-injectives-characterized}(b), it follows
from condition~(iv) that $\rD^\bu\in\sC(\X\tors)$ is actually a complex
of injective quasi-coherent torsion sheaves;
so $\rD^\bu\in\sC(\X\tors_\inj)$.

 Let $\X=\ilim_{\gamma\in\Gamma}X_\gamma$ be a representation of $\X$
by an inductive system of closed immersions of semi-separated
Noetherian schemes.
 Then, in view of Lemma~\ref{support-restriction-of-dualizing-complex},
it suffices to check condition~(iv) for the closed subschemes belonging
to the inductive system $(X_\gamma)_{\gamma\in\Gamma}$; so one can
assume $Z=X_\gamma$ for some $\gamma\in\Gamma$.

\begin{rem}
 A more common point of view is to consider a dualizing complex on
a scheme $X$ as an object of the derived category $\sD(X\qcoh)$, or
more specifically, of the bounded derived category $\sD^\bb(X\qcoh)$.
 A similar point of view on dualizing complexes on ind-schemes is also
possible, but it prescribes viewing a dualizing complex $\rD^\bu$
on $\X$ as an object of the \emph{coderived category}
$\sD^\co(\X\tors)$, as defined in Section~\ref{coderived-subsecn} below.
 Indeed, the homotopy category of complexes of injective quasi-coherent
torsion sheaves $\sK(\X\tors_\inj)$ is equivalent to the coderived
category by Corollary~\ref{ind-Noetherian-coderived-cor}.

 For schemes, one does not feel the difference between the derived
and the coderived category in this context, because there is no such
difference for complexes bounded below; see the discussion in
Remark~\ref{cotensor-right-derived-remark}(1).
 The dualizing complex on an ind-scheme is often bounded above
(see Remarks~\ref{cotensor-right-derived-remark}(3\+-5)), but it is
usually not bounded below (unless the whole ind-scheme is
finite-dimensional).

 In fact, a dualizing complex $\rD^\bu$ on an ind-scheme $\X$ can well
be an \emph{acyclic} complex, and in some very simple examples it is;
see Section~\ref{tate-space-subsecn}(7).
\end{rem}

\begin{ex} \label{aleph-zero-dualizing-glueing}
 Let $\X=\ilim\,(X_0\to X_1\to X_2\to\dotsb)$ be an ind-semi-separated
ind-Noetherian $\aleph_0$\+ind-scheme represented by an inductive
system of closed immersions of (semi-separated Noetherian) schemes
indexed by the poset of nonnegative integers.
 Let $i_n\:X_n\rarrow X_{n+1}$ denote the closed immersion morphisms in
the inductive system.
 Suppose that we are given a dualizing complex $\D_n^\bu$ on the scheme
$X_n$ for every $n\ge0$ together with homotopy equivalences
$\D_n^\bu\rarrow i_n^!\D_{n+1}^\bu$ of complexes of (injective)
quasi-coherent sheaves on~$X_n$.
 Then there are the related morphisms $i_n{}_*\D_n^\bu\rarrow
\D_{n+1}^\bu$ of complexes of quasi-coherent sheaves on~$X_{n+1}$.

 Let $k_n\:X_n\rarrow\X$ be the natural closed immersions.
 Consider the inductive system $k_0{}_*\D_0^\bu\rarrow k_1{}_*\D_1^\bu
\rarrow\dotsb$ of complexes of quasi-coherent torsion sheaves on
the ind-scheme $\X$, and put $\rD^\bu=\varinjlim_{n\ge0}k_n{}_*\D_n^\bu
\in\sC(\X\tors)$.
 Then $\rD^\bu$ is a dualizing complex on the ind-scheme~$\X$.

 Indeed, for every $m\ge0$ we have $k_m^!\rD^\bu=\varinjlim_{n\ge0}
k_m^!k_n{}_*\D_n^\bu=\varinjlim_{n\ge m} i_m^!\dotsm i_{n-1}^!
\D_n^\bu\in\sC(X_m\qcoh)$.
 Since the class of all injective quasi-coherent sheaves on $X_m$ is
closed under direct limits (the scheme $X_m$ being Noetherian), it
follows that $k_m^!\rD^\bu$ is a complex of injective quasi-coherent
sheaves on~$X_m$.
 Now we claim that the natural morphism $\D_m^\bu\rarrow k_m^!\rD^\bu$
is a homotopy equivalence of complexes of (injective) quasi-coherent
sheaves on~$X_m$.
 As a complex of injective quasi-coherent sheaves homotopy equivalent
to a dualizing complex is also a dualizing complex, this suffices to
prove the desired assertion.

 The morphisms of complexes $\D_m^\bu\rarrow i_m^!\dotsm i_{n-1}^!
\D_n^\bu$ are homotopy equivalences by assumption.
 So it remains to observe that, for any sequence $\J_0^\bu\rarrow
\J_1^\bu\rarrow\J_2^\bu\rarrow\dotsb$ of homotopy equivalences of
complexes of injective quasi-coherent sheaves on a Noetherian
scheme $X$, the natural morphism $\J_0^\bu\rarrow
\varinjlim_{n\ge0}\J_n^\bu$ is a homotopy equivalence.
 Indeed, the telescope short exact sequence $0\rarrow\bigoplus_{n\ge0}
\J_n^\bu\rarrow\bigoplus_{n\ge0}\J_n^\bu\rarrow\varinjlim_{n\ge0}
\J_n^\bu\rarrow0$ is a short exact sequence of complexes of injectives,
so it is termwise split.
 Hence $\varinjlim_{n\ge0}\J_n^\bu$ is the homotopy colimit of
the sequence $\J_0^\bu\rarrow\J_1^\bu\rarrow\J_2^\bu\rarrow\dotsb$
in the homotopy category $\sK(X\qcoh)$.
 It remains to apply~\cite[Lemma~1.6.6]{Neem}.

 Proposition~\ref{aleph-zero-torsion-morphism-strictification-prop}
or Lemma~\ref{aleph-zero-torsion-morphism-strictification-lemma}
below, combined with Lemma~\ref{dualizing-no-negative-selfext}, show
that this construction of a dualizing complex $\rD^\bu$ on $\X$ does
not depend on the arbitrary choices of the morphisms of complexes
$i_n{}_*\D_n^\bu\rarrow\D_{n+1}^\bu$ in the given homotopy classes
of such morphisms.
\end{ex}

\begin{exs} \label{ind-closed-immersion}
 (1)~A morphism of ind-schemes $k\:\Z\rarrow\X$ is said to be
an \emph{ind-closed immersion} if, for every closed subscheme
$Z\subset\Z$, the composition $Z\rarrow\Z\rarrow\X$ is a closed
immersion.
 Let $k\:\Z\rarrow\X$ be an ind-closed immersion of reasonable
ind-schemes, and let $\rM$ be a quasi-coherent torsion sheaf on~$\X$.
 For every reasonable closed subscheme $Z\subset\Z$, put
$\rN_{(Z)}=k_Z^!\rM$, where $k_Z\:Z\rarrow\X$.
 Then the collection of quasi-coherent sheaves $\rN_{(Z)}\in Z\qcoh$
defines a quasi-coherent torsion sheaf $\rN$ on~$\Z$.
 Put $k^!\rM=\rN$; this rule defines a functor $k^!\:\X\tors\rarrow
\Z\tors$.
 This construction generalizes the construction of the functor~$i^!$
for a closed immersion of ind-schemes $i\:\Z\rarrow\X$ in
Section~\ref{torsion-inverse-images-subsecn}.

\smallskip
 (2)~Let $\X$ be an ind-semi-separated ind-Noetherian ind-scheme,
and let $k\:\Z\rarrow\X$ be an ind-closed immersion of schemes.
 Let $\rD^\bu$ be a dualizing complex on~$\X$.
 Then it is clear from the definitions and
Lemma~\ref{support-restriction-of-dualizing-complex} that $k^!\rD^\bu$
is a dualizing complex on~$\Z$.

\smallskip
 (3)~In particular, let $X$ be a semi-separated Noetherian scheme,
and let $\X$ be an ind-scheme endowed with an ind-closed immersion
of ind-schemes $k\:\X\rarrow X$.
 This means that $\X=\ilim_{\gamma\in\Gamma}X_\gamma$, where
$(X_\gamma)_{\gamma\in\Gamma}$ is an inductive system of closed
subschemes in~$X$.
 Let $\D^\bu$ be a dualizing complex on~$X$.
 Denoting by $k_\gamma\:X_\gamma\rarrow X$ the composition
$X_\gamma\rarrow\X\rarrow X$, put $\D_\gamma^\bu=k_\gamma^!\D^\bu\in
\sC(X_\gamma\qcoh_\inj)$.
 Then, by Lemma~\ref{support-restriction-of-dualizing-complex},
\,$\D_\gamma^\bu$ is a dualizing complex on~$X_\gamma$.
 Denoting by~$i_{\gamma\delta}$ the closed immersions $X_\gamma
\rarrow X_\delta$, we have $\D_\gamma^\bu\simeq
i_{\gamma\delta}^!\D_\delta^\bu$ for all $\gamma<\delta\in\Gamma$.
 So the collection of complexes of injective quasi-coherent sheaves
$\D_\gamma^\bu$ on $X_\gamma$ defines an complex of injective
quasi-coherent torsion sheaves $\rD^\bu$ on~$\X$
(by Proposition~\ref{torsion-injectives-characterized}(b)).
 By the definition, $\rD^\bu$ is a dualizing complex on~$\X$.
\end{exs}

\subsection{Derived categories of flat sheaves and flat pro-sheaves}
 We refer to~\cite{Neem0}, \cite[Section~10]{Bueh},
or~\cite[Section~A.7]{Partin} for the definition of the \emph{derived
category} $\sD(\sE)$ of an exact category~$\sE$.
 The bounded below, bounded above, and bounded (on both sides) versions
of the derived category are denoted, as usually, by
$\sD^+(\sE)$, $\sD^-(\sE)$, and $\sD^\bb(\sE)\subset\sD(\sE)$;
these are full triangulated subcategories of the triangulated
category~$\sD(\sE)$.

 In particular, let $X$ be a scheme.
 Then the full subcategory $X\flat\subset X\qcoh$ of flat
quasi-coherent sheaves inherits an exact category structure from
the ambient abelian category of quasi-coherent sheaves on~$X$.
 So one can consider the derived category $\sD(X\flat)$ alongside
with the derived category $\sD(X\qcoh)$.

 A complex of flat quasi-coherent sheaves which vanishes as an object
of $\sD(X\qcoh)$ need \emph{not} vanish as an object of
$\sD(X\flat)$, generally speaking.
 Rather, a complex $\F^\bu$ of flat quasi-coherent sheaves vanishes
as an object of $\sD(X\flat)$ (``is acyclic with respect to
the exact category $X\flat$\,'') if and only if $\F^\bu$ is
acyclic as a complex of quasi-coherent sheaves on $X$ \emph{and}
all the quasi-coherent sheaves of cocycles of the complex $\F^\bu$
are flat.

 Let $X$ be a quasi-compact semi-separated scheme.
 A quasi-coherent sheaf $\cP$ on $X$ is said to be \emph{cotorsion}
if $\Ext^1_{X\qcoh}(\F,\cP)=0$ for all flat quasi-coherent sheaves
$\F$ on~$X$.

\begin{lem} \label{flat-and-cotorsion-lemma}
\textup{(a)} For any quasi-coherent sheaf $\M$ and any injective
quasi-coherent sheaf $\J$ on a quasi-compact semi-separated
scheme $X$, the quasi-coherent sheaf\/ $\cHom_{X\qc}(\M,\J)$ is
cotorsion. \par
\textup{(b)} For any flat quasi-coherent sheaf $\F$ and any cotorsion
quasi-coherent sheaf\/ $\cP$ on a quasi-compact semi-separated scheme
$X$, the quasi-coherent sheaf\/ $\cHom_{X\qc}(\F,\cP)$ is cotorsion.
{\hbadness=1650\par}
\textup{(c)} For any family $(\cP_\xi)_{\xi\in\Xi}$ of flat cotorsion
quasi-coherent sheaves on a semi-separated Noetherian scheme $X$,
the quasi-coherent sheaf\/ $\prod_{\xi\in\Xi}\cP_\xi$ on $X$ is flat
cotorsion.  (Here the product is taken in the Grothendieck category
$X\qcoh$.)
\end{lem}

\begin{proof}
 Parts~(a\+-b): the key fact is that there are enough flat
quasi-coherent sheaves on any quasi-compact semi-separated scheme,
i.~e., any quasi-coherent sheaf is a quotient of a flat one
(see~\cite[Section~2.4]{Mur0} or~\cite[Lemma~A.1]{EP}).
 Consequently, a quasi-coherent sheaf $\cQ$ on $X$ is cotorsion if
and only if, for every short exact sequence of flat quasi-coherent
sheaves $0\rarrow\F'\rarrow\F\rarrow\F''\rarrow0$ on $X$, the short
sequence of abelian groups $0\rarrow\Hom_X(\F'',\cQ)\rarrow
\Hom_X(\F,\cQ)\rarrow\Hom_X(\F',\cQ)\rarrow0$ is exact.
 Here $\Hom_X$ is a shorthand for $\Hom_{X\qcoh}$.

 With this criterion, the assertions of both~(a) and~(b) follow
straightforwardly from the universal property definition of
$\cHom_{X\qc}$.

 Part~(c): let $X=\bigcup_\alpha U_\alpha$ be a finite affine open
covering of a quasi-compact semi-separated scheme~$X$.
 Denote by $j_\alpha\:U_\alpha\rarrow X$ the open immersion morphisms.
 According to~\cite[Lemma~4.1.12]{Pcosh}, the flat cotorsion
quasi-coherent sheaves on $X$ are precisely the direct summands of
finite direct sums of the form $\bigoplus_\alpha j_\alpha{}_*
\cQ_\alpha$, where $\cQ_\alpha$ are flat cotorsion quasi-coherent
sheaves on~$U_\alpha$.
 Notice that the direct image functors~$j_\alpha{}_*$ (being right
adjoint to the inverse image functors~$j_\alpha^*$) preserve
infinite products of quasi-coherent sheaves.

 This reduces the question to the particular case of an affine
scheme~$X$.
 It remains to recall that, over any (commutative) ring, the class of
cotorsion modules is closed under infinite products; and over
a Noetherian (or coherent) ring, the class of flat modules is closed
under infinite products as well.
\end{proof}

\begin{lem} \label{internal-hom-of-complexes-flat}
 For any complexes of injective quasi-coherent sheaves\/ ${}'\!\J^\bu$
and $\J^\bu$ on a semi-separated Noetherian scheme $X$,
the complex\/ $\cHom_{X\qc}({}'\!\J^\bu,\J^\bu)$ is a complex of
flat cotorsion quasi-coherent sheaves. 
\end{lem}

\begin{proof}
 The assertion follows from Lemma~\ref{hom-tensor-flats-injectives}(d)
combined with Lemma~\ref{flat-and-cotorsion-lemma}(a,c).
 Alternatively, one can observe that both the quasi-coherent sheaves
${}'\K=\bigoplus_{p\in\boZ}{}'\!\J^p$ and $\K=\prod_{q\in\boZ}\J^q$
are injective and all the terms of the complex
$\cHom_{X\qc}({}'\!\J^\bu,\J^\bu)$ are direct summands of
the quasi-coherent sheaf $\cHom_{X\qc}({}'\K,\K)$.
 Then it remains to apply Lemma~\ref{hom-tensor-flats-injectives}(d)
and Lemma~\ref{flat-and-cotorsion-lemma}(a).
\end{proof}

\begin{lem} \label{dualizing-homothety-quasi-isomorphism-as-flats}
 For any dualizing complex\/ $\D^\bu$ on a semi-separated Noetherian
scheme $X$, the natural morphism $\cO_X\rarrow\cHom_{X\qc}(\D^\bu,
\D^\bu)$ from condition~(iii) in
Section~\ref{dualizing-definition-subsecn} is a quasi-isomorphism
of complexes in the exact category $X\flat$.
\end{lem}

\begin{proof}
 First of all, $\cHom_{X\qc}(\D^\bu,\D^\bu)$ is a complex of flat
quasi-coherent sheaves by Lemma~\ref{internal-hom-of-complexes-flat}.
 Now the point is that, by the definition, any dualizing complex
$\D^\bu$ on $X$ is homotopy equivalent to a bounded dualizing
complex~${}'\D^\bu$.
 Hence the complex $\cHom_{X\qc}(\D^\bu,\D^\bu)$ is homotopy equivalent
to the complex $\cHom_{X\qc}({}'\D^\bu,{}'\D^\bu)$, which is a bounded
complex of flat quasi-coherent sheaves.
 Finally, a morphism of bounded (above) complexes of flat quasi-coherent
sheaves is a quasi-isomorphism of complexes in $X\flat$ if and only if
it is a quasi-isomorphism of complexes in $X\qcoh$ (because any bounded
above complex of flat quasi-coherent sheaves that is acyclic as
a complex of quasi-coherent sheaves has flat sheaves of cocycles).
\end{proof}

\begin{lem} \label{dualizing-no-negative-selfext}
 For any dualizing complex\/ $\D^\bu$ on a semi-separated Noetherian
scheme $X$, one has $\Hom_{\sK(X\qcoh_\inj)}(\D^\bu,\D^\bu[n])=0$
for all $n<0$.
\end{lem}

\begin{proof}
 By the adjunction property defining the quasi-coherent internal Hom,
we have
$$
 \Hom_{\sK(X\qcoh_\inj)}(\D^\bu,\D^\bu[n])=
 \Hom_{\sK(X\flat)}(\cO_X,\cHom_{X\qc}(\D^\bu,\D^\bu)[n]).
$$
 Without loss of generality, we can assume the complex $\D^\bu$ to
be a bounded complex of injective quasi-coherent sheaves.
 Set $\cP^\bu=\cHom_{X\qc}(\D^\bu,\D^\bu)$; then $\cP^\bu$ is a bounded
complex as well.
 All the quasi-coherent sheaves $\cP^n$ on $X$ are cotorsion by
Lemma~\ref{flat-and-cotorsion-lemma}(a).
 Since there are enough flat quasi-coherent sheaves on
a quasi-compact semi-separated scheme~$X$, one has
$\Ext_{X\qcoh}^m(\F,\cP)=0$ for all $\F\in X\flat$, all cotorsion
quasi-coherent sheaves $\cP$ on $X$, and all $m>0$; in particular,
in the situation at hand $\Ext_{X\qcoh}^m(\cO_X,\cP^n)=0$ for all
$m>0$ and $n<0$.
 As the complex $\cP^\bu$ in $X\qcoh$ also has vanishing
cohomology sheaves in the negative cohomological degrees, it follows
that $\Hom_{\sK(X\flat)}(\cO_X,\cHom_{X\qc}(\D^\bu,\D^\bu)[n])=0$
for $n<0$.
\end{proof}

 Let $\X=\ilim_{\gamma\in\Gamma}X_\gamma$ be an ind-scheme represented
by an inductive system of closed immersions of schemes.
 The construction of Proposition~\ref{flat-pro-sheaves-exact-category}
defines an exact category structure on the category of flat
pro-quasi-coherent pro-sheaves $\X\flat$.
 Hence the related derived category $\sD(\X\flat)$.

\begin{lem} \label{flat-pro-sheaves-complex-acyclicity-criterion}
 A complex of flat pro-quasi-coherent pro-sheaves\/ $\fF^\bu$ is
acyclic (as a complex in $\X\flat$) if and only if, for every\/
$\gamma\in\Gamma$, the complex of flat quasi-coherent sheaves\/
$\fF^\bu{}^{(X_\gamma)}$ is acyclic (as a complex in $X_\gamma\flat$).
\end{lem}

\begin{proof}
 The proof is straightforward.
\end{proof}

\subsection{Coderived category of torsion sheaves}
\label{coderived-subsecn}
 Let $\sE$ be an exact category.
 A complex in $\sE$ is said to be \emph{absolutely acyclic} if it
belongs to the minimal thick subcategory of $\sK(\sE)$ containing
the totalizations of short exact sequences of complexes in~$\sE$.
 Here a short sequence of complexes in $\sE$ is said to be \emph{exact}
if it is termwise exact (i.~e., exact in every degree), and 
``totalization'' of a short sequence of complexes means taking
the total complex of a bicomplex with three rows.
 The triangulated quotient category of $\sK(\sE)$ by the thick
subcategory of absolutely acyclic complexes is called
the \emph{absolute derived category} of an exact category $\sE$
and denoted by $\sD^\abs(\sE)$.

 Let $\sE$ be an exact category in which infinite coproducts exist
and the class of all short exact sequences is closed under coproducts
(in this case, we will say that $\sE$ has \emph{exact coproducts}).
 Then a complex in $\sE$ is said to be \emph{coacyclic} if it belongs
to the minimal triangulated subcategory of $\sK(\sE)$ containing
the totalizations of short exact sequences in $\sC(\sE)$ and
closed under coproducts.
 The \emph{coderived category} $\sD^\co(\sE)$ is defined as
the triangulated quotient category of the homotopy category
$\sK(\sE)$ by the thick subcategory of coacyclic complexes.

 The reader is referred to~\cite[Section~2.1]{Psemi},
\cite[Sections~3\+-4]{Pkoszul}, \cite[Section~1.3]{EP},
\cite[Appendix~A]{Pcosh}, \cite[Section~2]{Pfp} for a discussion
of these definitions.

 An exact category $\sE$ is said to have \emph{homological
dimension~$\le d$} (where $d\ge-1$ is an integer)
if $\Ext_\sE^{d+1}(E,F)=0$ for all $E$, $F\in\sE$.

\begin{lem} \label{three-derived-categories-lemma}
\textup{(a)} In any exact category\/ $\sE$, any absolutely acyclic
complex is acyclic; so there is a natural triangulated Verdier
quotient functor\/ $\sD^\abs(\sE)\rarrow\sD(\sE)$. \par
\textup{(b)} In any exact category\/ $\sE$ with exact coproducts,
any absolutely acyclic complex is coacyclic, and any coacyclic
complex is acyclic; so there are natural triangulated Verdier
quotient functors\/ $\sD^\abs(\sE)\rarrow\sD^\co(\sE)\rarrow
\sD(\sE)$. \par
\textup{(c)} In an exact category $\sE$ of finite homological
dimension, any acyclic complex is absolutely acyclic.
\end{lem}

\begin{proof}
 Parts~(a\+-b) are straightforward; part~(c)
is~\cite[Remark~2.1]{Psemi}.
\end{proof}

\begin{prop} \label{coderived-and-homotopy-of-injectives}
\textup{(a)} For any exact category\/ $\sE$ with exact coproducts,
the composition\/ $\sK(\sE_\inj)\rarrow\sK(\sE)\rarrow\sD^\co(\sE)$
of the inclusion functor $\sK(\sE_\inj)\rarrow\sE(\sE)$ (induced by
the inclusion $\sE_\inj\rarrow\sE$) with the Verdier quotient functor\/
$\sK(\sE)\rarrow\sD^\co(\sE)$ is a fully faithful triangulated functor\/
$\sK(\sE_\inj)\rarrow\sD^\co(\sE)$. \par
\textup{(b)} Let\/ $\sE$ be an exact category with infinite coproducts
and enough injective objects.
 Assume that the class of all injectives\/ $\sE_\inj\subset\sE$ is
closed under coproducts in\/~$\sE$.
 Then the triangulated functor\/ $\sK(\sE_\inj)\rarrow\sD^\co(\sE)$
from part~\textup{(a)} is an equivalence of triangulated categories.
 Moreover, for any complex $E^\bu$ in\/ $\sE$ there exists a complex
$J^\bu$ in\/ $\sE_\inj$ together with a morphism of complexes
$E^\bu\rarrow J^\bu$ whose cone is a coacyclic complex in\/~$\sE$.
\end{prop}

\begin{proof}
 It should be noticed that in any exact category with infinite
coproducts and enough injective objects the infinite coproduct
functors are exact.
 Part~(a) is~\cite[Section~3.5]{Pkoszul}
or~\cite[Lemma~A.1.3(a)]{Pcosh}.
 Part~(b) is~\cite[Section~3.7]{Pkoszul}; for a far-reaching
generalization, see~\cite[Proposition~2.1]{Pfp}.
\end{proof}

 A Grothendieck abelian category is said to be \emph{locally Noetherian}
if it has a set of generators consisting of Noetherian objects.
 Equivalently, a Grothendieck abelian category $\sA$ is locally
Noetherian if and only if every object of $\sA$ is the union of its
Noetherian subobjects.

\begin{lem} \label{locally-Noetherian-coproducts-injective}
 In any locally Noetherian Grothendieck category, the class of all
injective objects is closed under coproducts.
\end{lem}

\begin{proof}
 Follows from Lemma~\ref{injectives-in-Grothendieck-characterized}.
\end{proof}

\begin{prop} \label{ind-Noetherian-torsion-locally-Noetherian}
 For any ind-Noetherian ind-scheme\/ $\X$, the category\/ $\X\tors$ of
quasi-coherent torsion sheaves on\/ $\X$ is a locally Noetherian
Grothendieck category.
 The direct images of coherent sheaves from closed subschemes of\/ $\X$
are the Noetherian objects of\/ $\X\tors$.
\end{prop}

\begin{proof}
 The category $\X\tors$ is Grothendieck by
Theorem~\ref{torsion-sheaves-abelian}, and it has a set of
Noetherian generators by
Lemma~\ref{closed-subschemes-torsion-subcategories} (because
the category of quasi-coherent sheaves on a Noetherian scheme is
locally Noetherian and the coherent sheaves are its Noetherian objects).
 The description of the Noetherian objects in $\X\tors$ easily follows.
\end{proof}

\begin{cor} \label{ind-Noetherian-coderived-cor}
 For any ind-Noetherian ind-scheme\/ $\X$, the coderived category of
quasi-coherent torsion sheaves is naturally equivalent to the homotopy
category of injective quasi-coherent torsion sheaves,
$\sK(\X\tors_\inj)\simeq\sD^\co(\X\tors)$.
\end{cor}

\begin{proof}
 Combine Proposition~\ref{ind-Noetherian-torsion-locally-Noetherian},
Lemma~\ref{locally-Noetherian-coproducts-injective},
and Proposition~\ref{coderived-and-homotopy-of-injectives}(b).
\end{proof}

\begin{rem} \label{coderived-compact-generators}
 Let $\X$ be an ind-Noetherian ind-scheme.
 One can say that a quasi-coherent torsion sheaf $\rM$ on $\X$ is
\emph{coherent} if there exists a closed subscheme $Z\subset\X$ with
the closed immersion morphism $i\:Z\rarrow\X$ and a coherent sheaf
$\M$ on $Z$ such that $\rM\simeq i_*\M$.
 Then the full subcategory of coherent torsion sheaves $\X\tcoh\subset
\X\tors$ is an abelian Serre subcategory;
by Proposition~\ref{ind-Noetherian-torsion-locally-Noetherian}, it is
the full subcategory of Noetherian objects in the locally Noetherian
category $\X\tors$.
 It follows that the coderived category $\sD^\co(\X\tors)$ is
compactly generated and the bounded derived category $\sD^\bb(\X\tcoh)$
is the full subcategory of compact objects in $\sD^\co(\X\tors)$.
\end{rem}

\begin{lem} \label{complex-of-injectives-contractible}
 Let\/ $\sA$ be a Grothendieck abelian category and\/ $\sS\subset\sA$
be a class of objects, closed under quotients and containing a set of
generators of\/~$\sA$.
 Let $J^\bu\in\sK(\sA_\inj)$ be a complex of injective objects in\/
$\sA$ such that for every object $S\in\sS$ one has\/
$\Hom_{\sK(\sA)}(S,J^\bu)=0$.
 Then the complex $J^\bu$ is contractible.
\end{lem}

\begin{proof}
 Using the assumption that $\sS$ contains a set of generators for $\sA$,
one proves that the complex $J^\bu$ is acyclic in~$\sA$.
 Since $\sS$ is also closed under quotients,
Lemma~\ref{injectives-in-Grothendieck-characterized} shows that any
object $K\in\sA$ for which $\Ext^1_\sA(S,K)=0$ for all $S\in\sS$ is
injective.
 From this observation one deduces injectivity of the cocycle objects
of the complex~$J^\bu$.
\end{proof}

\begin{lem} \label{complex-of-injective-torsion-contractible}
 Let\/ $\X=\ilim_{\gamma\in\Gamma}X_\gamma$ be a reasonable ind-scheme
represented by an inductive system of closed immersions of reasonable
closed subschemes, and let $i_\gamma\:X_\gamma\rarrow\X$ be the natural
closed immersions.
 Let $\rJ^\bu\in\sK(\X\tors_\inj)$ be a complex of injective
quasi-coherent torsion sheaves on\/~$\X$.
 Assume that the complex of injective quasi-coherent sheaves
$i_\gamma^!\rJ^\bu$ on $X_\gamma$ is contractible for every\/
$\gamma\in\Gamma$.
 Then the complex of injective quasi-coherent torsion sheaves $\rJ^\bu$
on\/ $\X$ is contractible as well.
\end{lem}

\begin{proof}
 For every complex of quasi-coherent sheaves $\M^\bu$ on $X_\gamma$,
\,$\gamma\in\Gamma$, we have $\Hom_{\sK(\X\tors)}(i_\gamma{}_*\M^\bu,
\rJ^\bu)\simeq\Hom_{\sK(X_\gamma\qcoh)}(\M^\bu,i_\gamma^!\rJ^\bu)=0$.
 In view of Lemmas~\ref{complex-of-injectives-contractible}
and~\ref{closed-subschemes-torsion-subcategories}, it follows that
the complex $\rJ^\bu$ is contractible.
\end{proof}

%

\subsection{The triangulated equivalence}
\label{ind-Noetherian-triangulated-equivalence-subsecn}
 To begin with, we recall the triangulated equivalence for
a semi-separated Noetherian scheme with a dualizing complex.

\begin{thm} \label{Noetherian-scheme-triangulated-equivalence}
 Let\/ $X$ be a semi-separated Noetherian scheme with a dualizing
complex\/~$\D^\bu$.
 Then there is a natural equivalence of triangulated categories\/
$\sD^\co(X\qcoh)\simeq\sD(X\flat)$, provided by mutually inverse
triangulated functors\/ $\cHom_{X\qc}(\D^\bu,{-})\:\sK(X\qcoh_\inj)
\rarrow\sD(X\flat)$ and\/ $\D^\bu\ot_{\cO_X}{-}\,\:\sD(X\flat)\rarrow
\sK(X\qcoh_\inj)$.
\end{thm}

\begin{proof}
 Notice first of all that any Noetherian scheme with a dualizing
complex has finite Krull dimension~\cite[Corollary~V.7.2]{Hart};
hence the exact category $X\flat$ has finite homological
dimension~\cite[Corollaire~II.3.2.7]{RG}, \cite[Lemma~5.4.1]{Pcosh}
(for a direct argument showing that existence of a dualizing complex
implies finite projective dimension of flat modules,
see~\cite[Proposition~1.5]{CFH}, \cite[Corollary~B.4.2]{Pcosh},
or~\cite[Proposition~4.3]{Pfp}).
 By Lemma~\ref{three-derived-categories-lemma}, it follows that
$\sD(X\flat)=\sD^\co(X\flat)=\sD^\abs(X\flat)$.

 Furthermore, the equivalence $\sD^\co(X\qcoh)\simeq
\sK(X\qcoh_\inj)$ is a particular case of
Corollary~\ref{ind-Noetherian-coderived-cor}.

 The assertion of the theorem is a result of Murfet~\cite[Theorem~8.4
and Proposition~8.9]{Mur}; for a different argument, which is much
closer to our exposition below, see~\cite[Theorem~2.5]{EP}
(which is stated and proved in the more complicated context of
matrix factorizations).
 Cf.~\cite[Theorem~5.7.1]{Pcosh} and~\cite[Theorem~4.5]{Pfp}.
 
 The dualizing complex $\D^\bu$ on $X$ is assumed to be a bounded
complex of injectives in the above references; here we assume it to be
homotopy equivalent to bounded, which is essentially the same.
 For a generalization to Noetherian schemes of infinite Krull
dimension with pointwise dualizing complexes of infinite injective
dimension, see~\cite[Corollary~3.10]{Neem-BIMS}.
\end{proof}

 The following theorem is the main result of
Section~\ref{dualizing-on-ind-Noetherian-secn}.

\begin{thm} \label{ind-Noetherian-triangulated-equivalence-thm}
 Let\/ $\X$ be an ind-semi-separated ind-Noetherian ind-scheme with
a dualizing complex\/~$\rD^\bu$.
 Then there is a natural equivalence of triangulated categories\/
$\sD^\co(\X\tors)\simeq\sD(\X\flat)$, provided by mutually inverse
triangulated functors\/ $\fHom_{\X\qc}(\rD^\bu,{-})\:\sK(\X\tors_\inj)
\rarrow\sD(\X\flat)$ and\/ $\rD^\bu\ot_\X{-}\,\:\sD(\X\flat)\rarrow
\sK(\X\tors_\inj)$.
\end{thm}

 The notation $\fHom_{\X\qc}({-},{-})$ will be explained below, and
the proof of Theorem~\ref{ind-Noetherian-triangulated-equivalence-thm}
will be given below in this
Section~\ref{ind-Noetherian-triangulated-equivalence-subsecn}.

\begin{lem}  \label{reasonable-shriek-flat-star-commutation}
 Let $f\:Y\rarrow X$ be a flat morphism of schemes and $Z\subset X$ be
a reasonable closed subscheme with the closed immersion $i\:Z\rarrow X$.
 Consider the pullback diagram
$$
\xymatrix{
 Z\times_XY \ar[r]^-k \ar[d]^-g & Y \ar[d]^-f \\
 Z \ar[r]^-i & X
}
$$
 Put $W=Z\times_XY$.
 Then there is a natural isomorphism $g^*i^!\simeq k^!f^*$ of
functors $X\qcoh\rarrow W\qcoh$.
\end{lem}

\begin{proof}
 The assertion is local in $X$ and reduces to the case of affine
schemes, for which it means the following.
 Let $R\rarrow S$ be a homomorphism of commutative rings such that
$S$ is a flat $R$\+module, and let $R\rarrow T$ be a surjective
homomorphism of commutative rings with a finitely generated
kernel ideal.
 Let $M$ be an $R$\+module.
 Then the natural map
\begin{multline*}
 (S\ot_RT)\ot_T\Hom_R(T,M)\simeq S\ot_R\Hom_R(T,M) \\
 \lrarrow \Hom_R(T,\>S\ot_RM)\simeq\Hom_S(S\ot_RT,\>S\ot_RM)
\end{multline*}
is an isomorphism of $(S\ot_RT)$\+modules.
\end{proof}

\begin{lem} \label{star-qc-Hom-shriek-lemma}
 Let $i\:Z\rarrow X$ be a closed immersion of schemes, and let
$\M$, $\K$ be quasi-coherent sheaves on~$X$.
 Then there is a natural morphism
$$
 i^*\cHom_{X\qc}(\M,\K)\lrarrow\cHom_{Z\qc}(i^!\M,i^!\K)
$$
of quasi-coherent sheaves on~$Z$.
 This morphism is an isomorphism whenever the scheme $X$ is
(quasi-compact and) semi-separated, $i(Z)$ is a reasonable closed
subscheme in $X$, and $\K$ is an injective quasi-coherent sheaf on~$X$.
\end{lem}

\begin{proof}
 To prove the first assertion, it suffices to construct a natural
morphism $\cHom_{X\qc}(\M,\K)\rarrow i_*\cHom_{Z\qc}(i^!\M,i^!\K)$
of quasi-coherent sheaves on~$X$.
 Let $\cL$ be an arbitrary quasi-coherent sheaf on~$X$.
 Then morphisms $\cL\rarrow\cHom_{X\qc}(\M,\K)$ correspond bijectively
to morphisms $\cL\ot_{\cO_X}\M\rarrow\K$, while morphisms
$\cL\rarrow i_*\cHom_{Z\qc}(i^!\M,i^!\K)$ correspond bijectively to
morphisms $i^*\cL\ot_{\cO_Z}i^!\M\rarrow i^!\K$.
 Applying~$i^!$ to a morphism $\cL\ot_{\cO_X}\M\rarrow\K$ and
precomposing with the natural morphism from
Lemma~\ref{shriek-star-tensor}, one obtains a morphism
$i^*\cL\ot_{\cO_Z}i^!\M\rarrow i^!\K$.

 For the second assertion, notice that any injective quasi-coherent
sheaf on a quasi-compact scheme $X$ is a direct summand of a finite
direct direct sum of direct images of injective quasi-coherent sheaves
from affine open subschemes of $X$ (as there are enough injective
objects of this form in $X\qcoh$).
 Let $U\subset X$ be an affine open subscheme and $j\:U\rarrow X$ be
the open immersion morphism.
 We can assume that $\K=j_*\J$, where $\J$ is an injective
quasi-coherent sheaf on~$U$.

 Now
$$
 i^*\cHom_{X\qc}(\M,j_*\J)\simeq i^*j_*\cHom_{U\qc}(j^*\M,\J)
 \simeq g_*k^*\cHom_{U\qc}(j^*\M,\J),
$$
where $k\:W=Z\times_XU\rarrow U$, \ $g\:W\rarrow Z$, the first
isomorphism holds by Lemma~\ref{qcoh-internal-Hom-projection}, and
the second one by Lemma~\ref{affine-flat-base-change}(a).
 (The assumption that $X$ is semi-separated is used here, as we need
$j$~to be an affine morphism.)

 On the other hand, we have $i^!j_*\J\simeq g_*k^!\J$ by
Lemma~\ref{reasonable-base-change}(a) and
$g^*i^!\M\simeq k^!j^*\M$ by
Lemma~\ref{reasonable-shriek-flat-star-commutation}.
 Hence
\begin{multline*}
 \cHom_{Z\qc}(i^!\M,i^!j_*\J)\simeq\cHom_{Z\qc}(i^!\M,g_*k^!\J)
 \\ \simeq g_*\cHom_{W\qc}(g^*i^!\M,k^!\J)\simeq
 g_*\cHom_{W\qc}(k^!j^*\M,k^!\J).
\end{multline*}

 This reduces the second assertion of the lemma to the particular case
of an affine scheme $U$ with a reasonable closed subscheme
$k(W)\subset U$, with the quasi-coherent sheaf $j^*\M$ and
the injective quasi-coherent sheaf $\J$ on~$U$.
 In the affine case, the assertion means the following.
 Let $R\rarrow S$ be a surjective ring homomorphism with a finitely
generated kernel ideal, let $M$ be an $R$\+module, and let $J$ be
an injective $R$\+module.
 Then the natural morphism of $S$\+modules
$$
 S\ot_R\Hom_R(M,J)\lrarrow\Hom_R(\Hom_R(S,M),J)\simeq
 \Hom_S(\Hom_R(S,M),\Hom_R(S,J))
$$
is an isomorphism.
\end{proof}

\begin{lem} \label{star-product-lemma}
 Let $i\:Z\rarrow X$ be a closed immersion of schemes, and let
$(\cP_\xi)_{\xi\in\Xi}$ be a family of quasi-coherent sheaves on~$X$.
 Then there is a natural morphism
$$
 i^*\prod\nolimits_{\xi\in\Xi}\cP_\xi
 \lrarrow\prod\nolimits_{\xi\in\Xi}i^*\cP_\xi
$$
of quasi-coherent sheaves on~$Z$.
 This morphism is an isomorphism whenever the scheme $X$ is
(quasi-compact and) semi-separated, $i(Z)$ is a reasonable closed
subscheme in $X$, and\/ $\cP_\xi$ are flat cotorsion quasi-coherent
sheaves on~$X$.
\end{lem}

\begin{proof}
 The first assertion holds for any functor between categories with
products (in the role of~$i^*$).
 The result of~\cite[Lemma~4.1.12]{Pcosh}, as restated in
the proof of Lemma~\ref{flat-and-cotorsion-lemma}(c), reduces
the second assertion to the particular case when $\cP_\xi=j_*\cQ_\xi$
for every $\xi\in\Xi$, where $j\:U\rarrow X$ is the immersion
of an affine open subscheme and $\cQ_\xi$ are (flat cotorsion)
quasi-coherent sheaves on~$U$.

 Put $W=Z\times_XU$, and denote by $k\:W\rarrow U$ and $g\:W\rarrow Z$
the natural morphisms.
 In view of Lemma~\ref{affine-flat-base-change}, we have
$$
 i^*\prod\nolimits_{\xi\in\Xi}j_*\cQ_\xi\simeq
 i^*j_*\prod\nolimits_{\xi\in\Xi}\cQ_\xi\simeq
 g_*k^*\prod\nolimits_{\xi\in\Xi}\cQ_\xi,
$$
while
$$
 \prod\nolimits_{\xi\in\Xi}i^*j_*\cQ_\xi\simeq
 \prod\nolimits_{\xi\in\Xi}g_*k^*\cQ_\xi\simeq
 g_*\prod\nolimits_{\Xi\in\Xi}k^*\cQ_\xi.
$$

 Now we have reduced the second assertion of the lemma to the particular
case of an affine scheme $U$ with a reasonable closed subscheme
$k(W)\subset U$ and the family of (flat cotorsion) quasi-coherent
sheaves $\cQ_\xi$ on~$U$.
 In the affine case, the assertion means the following.
 Let $R\rarrow S$ be a surjective ring homomorphism with a finitely
generated kernel ideal, and let $(Q_\xi)_{\xi\in\Xi}$ be a family of
(flat cotorsion) $R$\+modules.
 Then the natural morphism of $S$\+modules
$$
 S\ot_R\prod\nolimits_{\xi\in\Xi}Q_\xi\lrarrow
 \prod\nolimits_{\xi\in\Xi}S\ot_RQ_\xi
$$
is an isomorphism.
 The assumption that the $R$\+modules $Q_\xi$ are flat cotorsion is not
needed here; it was only used in order to reduce the question to
the affine case.
\end{proof}

\begin{lem} \label{star-qc-Hom-shriek-for-complexes}
 Let $i\:Z\rarrow X$ be a closed immersion of schemes, and let
$\M^\bu$, $\K^\bu\in\sC(X\qcoh)$ be complexes of quasi-coherent
sheaves on~$X$.
 Then there is a natural morphism
$$
 i^*\cHom_{X\qc}(\M^\bu,\K^\bu)\lrarrow
 \cHom_{Z\qc}(i^!\M^\bu,i^!\K^\bu)
$$
of complexes of quasi-coherent sheaves on~$Z$.
 This morphism is an isomorphism whenever the scheme $X$ is
(quasi-compact and) semi-separated, $i(Z)$ is a reasonable closed
subscheme in $X$, and $\K^\bu$ is a complex of injective quasi-coherent
sheaves on~$X$.
\end{lem}

\begin{proof}
 For every $n\in\boZ$, the degree~$n$ component of the desired
morphism of complexes of quasi-coherent sheaves on $Z$ is
the composition
\begin{multline} \label{degree-n-qcoh-map}
 i^*\cHom_{X\qc}(\M^\bu,\K^\bu)^n=
 i^*\prod\nolimits_{q-p=n}\cHom_{X\qc}(\M^p,\K^q) \\
 \lrarrow\prod\nolimits_{q-p=n}i^*\cHom_{X\qc}(\M^p,\K^q) \\
 \lrarrow\prod\nolimits_{q-p=n}\cHom_{Z\qc}(i^!\M^p,i^!\K^q)
 =\cHom_{Z\qc}(i^!\M^\bu,i^!\K^\bu)^n
\end{multline}
of the morphisms induced by the natural morphisms from
Lemmas~\ref{star-product-lemma} and~\ref{star-qc-Hom-shriek-lemma}.

 Assuming that $X$ is semi-separated, $i(Z)$ is reasonable in $X$,
and $\K^\bu$ is a complex of injectives, the morphisms
$i^*\cHom_{X\qc}(\M^p,\K^q)\rarrow\cHom_{Z\qc}(i^!\M^p,i^!\K^q)$
are isomorphisms by Lemma~\ref{star-qc-Hom-shriek-lemma}.
 Assuming additionally that $X$ is a Noetherian scheme and $\M^\bu$
is also a complex of injectives, the quasi-coherent sheaves
$\cHom_{X\qc}(\M^p,\K^q)$ are flat cotorsion by
Lemmas~\ref{hom-tensor-flats-injectives}(d)
and~\ref{flat-and-cotorsion-lemma}(a).
 Hence the map $i^*\prod_{q-p=n}\cHom_{X\qc}(\M^p,\K^q)
\rarrow\prod_{q-p=n}i^*\cHom_{X\qc}(\M^p,\K^q)$ is an isomorphism
by Lemma~\ref{star-product-lemma}, and we are done.

 In the general case of the assumptions of the lemma,
put $\N=\bigoplus_{p\in\boZ}\M^p$ and $\cL=\prod_{q\in\boZ}\K^q$.
 Then $\N$ is a quasi-coherent sheaf and $\cL$ is an injective
quasi-coherent sheaf on~$X$.
 The functor $i^!\:X\qcoh\rarrow Z\qcoh$ preserves both the infinite
products (as a right adjoint functor) and coproducts (in fact, it
even preserves direct limits, since $Z$ is a reasonable closed
subscheme in~$X$).
 By Lemma~\ref{star-qc-Hom-shriek-lemma}, the natural morphism
$i^*\cHom_{X\qc}(\N,\cL)\rarrow\cHom_{Z\qc}(i^!\N,i^!\cL)$ is
an isomorphism of quasi-coherent sheaves on~$Z$.
 In every degree $n\in\boZ$, the morphism~\eqref{degree-n-qcoh-map}
is a direct summand of this isomorphism, hence also an isomorphism.

 Yet another approach is to notice that the proof of
Lemma~\ref{star-product-lemma} is actually applicable to any
family of quasi-coherent sheaves $\cP_\xi$ each of which is
a direct summand of a finite direct sum of direct images of
quasi-coherent sheaves from affine open subschemes in a fixed finite
affine open covering $X=\bigcup_\alpha U_\alpha$.
 Then it remains to observe that, for any quasi-coherent sheaf $\M$
and any injective quasi-coherent sheaf $\K$ on $X$, the quasi-coherent
sheaf $\cHom_{X\qc}(\M,\K)$ is a direct summand of such a direct sum
of direct images (since the quasi-coherent sheaf~$\K$ is).
\end{proof}

 Let $\X$ be a reasonable ind-semi-separated ind-scheme.
 Let $\rE^\bu\in\sC(\X\tors)$ be a complex of quasi-coherent
torsion sheaves on $\X$, and let $\rJ^\bu\in\sC(\X\tors_\inj)$ be
a complex of injective quasi-coherent torsion sheaves on~$\X$.
 Then the complex of pro-quasi-coherent pro-sheaves
$\fHom_{\X\qc}(\rE^\bu,\rJ^\bu)\in\sC(\X\pro)$ is constructed
as follows.

 For every reasonable closed subscheme $Z\subset\X$, put
$\fHom_{\X\qc}(\rE^\bu,\rJ^\bu)^{(Z)}=
\cHom_{Z\qc}(i^!\rE^\bu,i^!\rJ^\bu)=\cHom_{Z\qc}(\rE^\bu_{(Z)},
\rJ^\bu_{(Z)})$, where $i\:Z\rarrow X$ is the closed immersion morphism.
 According to Lemma~\ref{star-qc-Hom-shriek-for-complexes},
for every pair of reasonable closed subschemes $Y$, $Z\subset\X$
such that $Z\subset Y$, we have $\fHom_{\X\qc}(\rE^\bu,\rJ^\bu)^{(Z)}
\simeq i_{ZY}^*\fHom_{\X\qc}(\rE^\bu,\rJ^\bu)^{(Y)}$, as required in
the definition of a pro-quasi-coherent pro-sheaf (where
$i_{ZY}\:Z\rarrow Y$ is the closed immersion).
 This explains the meaning of the notation in
Theorem~\ref{ind-Noetherian-triangulated-equivalence-thm}.

\begin{proof}[Proof of
Theorem~\ref{ind-Noetherian-triangulated-equivalence-thm}]
 The equivalence $\sD^\co(\X\tors)\simeq\sK(\X\tors_\inj)$ is provided
by Corollary~\ref{ind-Noetherian-coderived-cor}.

 The tensor product functor $\ot_\X\:\X\flat\times\X\tors\rarrow
\X\tors$ was constructed in Sections~\ref{pro-in-torsion-action-subsecn}
and~\ref{flat-pro-sheaves-subsecn}.
 Here we switch the two arguments and write $\ot_\X\:\X\tors\times
\X\flat\rarrow\X\tors$.
 Given a complex of quasi-coherent torsion sheaves
$\rE^\bu\in\sC(\X\tors)$ and a complex of flat pro-quasi-coherent
pro-sheaves $\fF^\bu\in\sC(\X\flat)$, the complex of quasi-coherent
torsion sheaves $\rE^\bu\ot_\X\fF^\bu\in\sC(\X\tors)$ is constructed
by taking coproducts along the diagonals of the bicomplex of
quasi-coherent torsion sheaves $\rE^p\ot_\X\fF^q$, \ $p$, $q\in\boZ$.

 This construction defines a functor $\rE^\bu\ot_\X{-}\,\:
\sC(\X\flat)\rarrow\sC(\X\tors)$, which obviously descends to
a triangulated functor between the homotopy categories
$\rE^\bu\ot_\X\nobreak{-}\,\:\allowbreak\sK(\X\flat)
\rarrow\sK(\X\tors)$.
 By Proposition~\ref{flat-torsion-tensor-prop}, for any complex
$\fF^\bu\in\sC(\X\flat)$ and any closed subscheme $Z\subset\X$ with
the closed immersion morphism $i\:Z\rarrow\X$, we have
$i^!(\rE^\bu\ot_\X\fF^\bu)\simeq i^!\rE^\bu\ot_{\cO_Z}i^*\fF^\bu$.

 Now let us assume that $\rE^\bu$ is a complex of injective
quasi-coherent torsion sheaves, $\rE^\bu\in\sC(\X\tors_\inj)$.
 By Lemma~\ref{hom-tensor-flats-injectives}(b), the complex
$i^!\rE^\bu\ot_{\cO_Z}i^*\fF^\bu$ is a complex of injective
quasi-coherent sheaves on~$Z$ (recall that injectivity of
quasi-coherent sheaves on a Noetherian scheme is preserved by
coproducts).
 By Proposition~\ref{torsion-injectives-characterized}(b), it follows
that $\rE^\bu\ot_\X\fF^\bu$ is a complex of injective quasi-coherent
torsion sheaves on~$\X$.
 We have constructed a triangulated functor
\begin{equation} \label{tensor-with-complex-of-injective-torsion}
 \rE^\bu\ot_\X{-}\,\:\sK(\X\flat)\lrarrow\sK(\X\tors_\inj).
\end{equation}

 Let us check that the latter functor induces a well-defined
triangulated functor
\begin{equation} \label{derived-to-homotopy-tensoring-functor}
 \rE^\bu\ot_\X{-}\,\:\sD(\X\flat)\lrarrow\sK(\X\tors_\inj).
\end{equation}
 For this purpose, we need to show that the complex 
$\rE^\bu\ot_\X\fF^\bu\in\sK(\X\tors_\inj)$ is contractible whenever
a complex $\fF^\bu\in\sK(\X\flat)$ is acyclic with respect to
the exact category $\X\flat$.
 By Lemma~\ref{complex-of-injective-torsion-contractible}, it suffices
to check that the complex $i^!(\rE^\bu\ot_\X\fF^\bu)\simeq
i^!\rE^\bu\ot_{\cO_Z}i^*\fF^\bu\in\sK(Z\qcoh_\inj)$ is contractible.
 Here $i^!\rE^\bu$ is a complex of injective quasi-coherent sheaves
on $Z$ and $i^*\fF^\bu$ is an acyclic complex in the exact category
$Z\flat$.

 The latter assertion is essentially a part of
Theorem~\ref{Noetherian-scheme-triangulated-equivalence}.
 One can say that any acyclic complex in $Z\flat$ is coacyclic,
and it is easy to see that the tensor product of a coacyclic complex
in $Z\flat$ with any complex in $Z\qcoh$ is coacyclic; a coacyclic
complex of injectives is contractible.
 Alternatively, one can assume that $i^!\rE^\bu$ is homotopy equivalent
to a bounded complex in $Z\qcoh_\inj$; on any Noetherian scheme $Z$,
it is clear that the tensor product of an acyclic complex in $Z\flat$
with a bounded complex of injectives is contractible.
 For quite general results in this direction, see
Lemmas~\ref{qcoh-complexes-tensor-exactness}(d)
and~\ref{pro-torsion-complexes-tensor-exactness}(c) below.

 On the other hand, for any complex $\rE^\bu\in\sC(\X\tors)$,
the construction preceding this proof provides a functor
$\fHom_{\X\qc}(\rE^\bu,{-})\:\sC(\X\tors_\inj)\rarrow\sC(\X\pro)$,
which obviously descends to a triangulated functor between
the homotopy categories $\fHom_{\X\qc}(\rE^\bu,{-})\:\sK(\X\tors_\inj)
\rarrow\sK(\X\pro)$.
 By construction, for any complex $\rJ^\bu\in\sC(\X\tors_\inj)$ and
any closed subscheme $Z\subset X$ with the closed immersion morphism
$i\:Z\rarrow\X$, we have $i^*\fHom_{\X\qc}(\rE^\bu,\rJ^\bu)=
\cHom_{Z\qc}(i^!\rE^\bu,i^!\rJ^\bu)$.

 Once again, assume that $\rE^\bu$ is a complex of injective
quasi-coherent torsion sheaves on~$\X$.
 Then, by Lemma~\ref{internal-hom-of-complexes-flat},
the complex $\cHom_{Z\qc}(i^!\rE^\bu,i^!\rJ^\bu)$ is a complex of flat
quasi-coherent sheaves on~$Z$.
 Hence $\fHom_{\X\qc}(\rE^\bu,\rJ^\bu)$ is a complex of flat
pro-quasi-coherent pro-sheaves on~$\X$.
 We have constructed a triangulated functor
\begin{equation} \label{inner-hom-from-complex-of-injective-torsion}
 \fHom_{\X\qc}(\rE^\bu,{-})\:\sK(\X\tors_\inj)\lrarrow\sK(\X\flat).
\end{equation}

 Composing the latter functor with the canonical triangulated
Verdier quotient functor $\sK(\X\flat)\rarrow\sD(\X\flat)$, we obtain
a triangulated functor
\begin{equation} \label{homotopy-to-derived-inner-hom-functor}
 \fHom_{\X\qc}(\rE^\bu,{-})\:\sK(\X\tors_\inj)\lrarrow\sD(\X\flat).
\end{equation}
 It is straightforward to see that
the functor~\eqref{inner-hom-from-complex-of-injective-torsion}
is right adjoint to
the functor~\eqref{tensor-with-complex-of-injective-torsion}.
 Hence the functor~\eqref{homotopy-to-derived-inner-hom-functor}
is right adjoint to
the functor~\eqref{derived-to-homotopy-tensoring-functor}.

 It remains to show that
the functors~\eqref{derived-to-homotopy-tensoring-functor}
and~\eqref{homotopy-to-derived-inner-hom-functor} are mutually
inverse equivalences when $\rE^\bu=\rD^\bu$ is a dualizing complex
on~$\X$.
 For this purpose, we need to check that the adjunction morphisms
are isomorphisms.

 Let $\fF^\bu\in\sC(\X\flat)$ be a complex of flat pro-quasi-coherent
pro-sheaves on~$\X$.
 Then the adjunction $\fF^\bu\rarrow
\fHom_{\X\qc}(\rD^\bu,\>\rD^\bu\ot_\X\fF^\bu)$ is a natural morphism
in $\sC(\X\flat)$; we need to show that it is an isomorphism in
$\sD(\X\flat)$.
 By Lemma~\ref{flat-pro-sheaves-complex-acyclicity-criterion}, it
suffices to check that $i^*\fF^\bu\rarrow i^*\fHom_{\X\qc}(\rD^\bu,\>
\rD^\bu\ot_\X\fF^\bu)$ is an isomorphism in $\sD(Z\flat)$.
 According to the discussion above, we have isomorphisms of complexes
\begin{multline*}
 i^*\fHom_{\X\qc}(\rD^\bu,\>\rD^\bu\ot_\X\fF^\bu)\simeq
 \cHom_{Z\qc}(i^!\rD^\bu,\>i^!(\rD^\bu\ot_\X\fF^\bu)) \\ \simeq
 \cHom_{Z\qc}(i^!\rD^\bu,\>i^!\rD^\bu\ot_{\cO_Z}i^*\fF^\bu).
\end{multline*}
 So we have to show that the natural morphism
$\F^\bu\rarrow\cHom_{Z\qc}(i^!\rD^\bu,\>i^!\rD^\bu\ot_{\cO_Z}\F^\bu)$
is an isomorphism in $\sD(Z\flat)$ for every complex $\F^\bu$ of
flat quasi-coherent sheaves on~$Z$.

 By the definitions of a dualizing complex on~$\X$ and on~$Z$
(see Section~\ref{dualizing-definition-subsecn}), the complex of
injective quasi-coherent sheaves $i^!\rD^\bu$ on $Z$ is homotopy
equivalent to a bounded complex of injective quasi-coherent sheaves
$\D^\bu_Z$ (which is also a dualizing complex on~$Z$).
 So the complex of flat quasi-coherent sheaves
$\cHom_{Z\qc}(i^!\rD^\bu,\>i^!\rD^\bu\ot_{\cO_Z}\F^\bu)$ on $Z$ is
homotopy equivalent to the complex of flat quasi-coherent sheaves
$\cHom_{Z\qc}(\D_Z^\bu,\>\D_Z^\bu\ot_{\cO_Z}\F^\bu)$.
 Finally, the assertion that the natural morphism of complexes of
flat quasi-coherent sheaves $\F^\bu\rarrow
\cHom_{Z\qc}(\D_Z^\bu,\>\D_Z^\bu\ot_{\cO_Z}\F^\bu)$ is an isomorphism
in $\sD(Z\flat)$ is essentially a part of
Theorem~\ref{Noetherian-scheme-triangulated-equivalence} for
the semi-separated Noetherian scheme~$Z$.

 Similarly, let $\rJ^\bu\in\sC(\X\tors_\inj)$ be a complex of injective
quasi-coherent torsion sheaves on~$\X$.
 Then the adjunction $\rD^\bu\ot_\X\fHom_{\X\qc}(\rD^\bu,\rJ^\bu)\rarrow
\rJ^\bu$ is a natural morphism in $\sC(\X\tors_\inj)$; we have to show
that it is a homotopy equivalence.
 By Lemma~\ref{complex-of-injective-torsion-contractible}, it suffices
to check that $i^!(\rD^\bu\ot_\X\fHom_{\X\qc}(\rD^\bu,\rJ^\bu))\rarrow
i^!\rJ^\bu$ is a homotopy equivalence of complexes in $Z\qcoh_\inj$.
 According to the discussion above, we have isomorphisms of complexes
\begin{multline*}
 i^!(\rD^\bu\ot_\X\fHom_{\X\qc}(\rD^\bu,\rJ^\bu))\simeq
 i^!\rD^\bu\ot_{\cO_Z}i^*\fHom_{\X\qc}(\rD^\bu,\rJ^\bu) \\ \simeq
 i^!\rD^\bu\ot_{\cO_Z}\cHom_{Z\qc}(i^!\rD^\bu,i^!\rJ^\bu).
\end{multline*}
 So we have to show that the natural morphism $i^!\rD^\bu\ot_{\cO_Z}
\cHom_{Z\qc}(i^!\rD^\bu,\J^\bu)\rarrow\J^\bu$ is a homotopy equivalence
of complexes in $Z\qcoh_\inj$ for every complex $\J^\bu$ of injective
quasi-coherent sheaves on~$Z$.

 As above, the complex of injective quasi-coherent sheaves
$i^!\rD^\bu\ot_{\cO_Z}\cHom_{Z\qc}(i^!\rD^\bu,\allowbreak\J^\bu)$ on $Z$
is homotopy equivalent to the complex of injective quasi-coherent
sheaves $\D_Z^\bu\ot_{\cO_Z}\cHom_{Z\qc}(\D_Z^\bu,\J^\bu)$.
 Once again, the assertion that the natural morphism of complexes of
injective quasi-coherent sheaves $\D_Z^\bu\ot_{\cO_Z}
\cHom_{Z\qc}(\D_Z^\bu,\J^\bu)\rarrow\J^\bu$ is a homotopy equivalence
is essentially a part of
Theorem~\ref{Noetherian-scheme-triangulated-equivalence}.
\end{proof}

\Section{The Cotensor Product}  \label{cotensor-secn}

  The triangulated tensor structure on the coderived category
$\sD^\co(X\qcoh)$ of quasi-coherent sheaves on a Noetherian scheme
with a dualizing complex $\D^\bu$ was introduced
in~\cite[Propositions~6.2, 8.10, and~B.6]{Mur} and studied
in~\cite[Section~B.2.5]{EP}, where it was denoted by $\oc_{\D^\bu}$
and called the \emph{cotensor product} of complexes of quasi-coherent
sheaves on $X$ over the dualizing complex~$\D^\bu$.
 The aim of this section is to generalize this construction to
complexes of quasi-coherent torsion sheaves on an ind-Noetherian
ind-scheme with a  dualizing complex, and explain the connection
with the cotensor product of complexes of comodules over
a cocommutative coalgebra.

\subsection{Construction of cotensor product}
\label{cotensor-product-construction-subsecn}
 We start with a lemma about tensor products of complexes of
quasi-coherent sheaves on a scheme.

\begin{lem} \label{qcoh-complexes-tensor-exactness}
 Let $\M^\bu\in\sC(X\qcoh)$ be a complex of
quasi-coherent sheaves and $\F^\bu$, $G^\bu\in\sC(X\flat)$ be two
complexes of flat quasi-coherent sheaves on a scheme~$X$. \par
\textup{(a)} If the complex $\F^\bu$ is acyclic in $X\flat$,
then the complex $\F^\bu\ot_{\cO_X}\M^\bu$ is acyclic in $X\qcoh$. \par
\textup{(b)} If the complex $\F^\bu$ is acyclic in $X\flat$,
then the complex $\F^\bu\ot_{\cO_X}\G^\bu$ is acyclic in $X\flat$. \par
\textup{(c)} If the complex $\M^\bu$ is coacyclic in $X\qcoh$, then
the complex $\F^\bu\ot_{\cO_X}\M^\bu$ is coacyclic in $X\qcoh$. \par
\textup{(d)} If  the complex $\F^\bu$ is acyclic in $X\flat$ and
the scheme $X$ is Noetherian, then the complex $\F^\bu\ot_{\cO_X}\M^\bu$
is coacyclic in $X\qcoh$.
\end{lem}

\begin{proof}
 The assertions~(a\+-b) are essentially local and reduce to the case of
an affine scheme~$X$.
 In this context, both the assertions are explained by the observation
that over an (arbitrary associative) ring $R$, an acyclic (in $R\modl$)
complex of flat modules is homotopy flat if and only if it has flat 
modules of cocycles.

 Specifically, part~(a) is provable by representing $\M^\bu$ as
a direct limit of bounded complexes (with a silly truncation on the left
and a canonical truncation on the right).
 This reduces the question to the case of a one-term complex
$\M^\bu=\M$, which is obvious.
 To prove part~(b), it suffices to check that the complex $\F^\bu
\ot_{\cO_X}\G^\bu\ot_{\cO_X}\N$ is exact in $X\qcoh$ for every
sheaf $\N\in X\qcoh$.
 This follows from part~(a) applied to the complexes $\F^\bu$ and
$\M^\bu=\G^\bu\ot_{\cO_X}\N$.

 The assertions~(c\+-d) are essentially local, too (assuming that $X$
is either quasi-compact and semi-separated or else Noetherian).
 But this needs to be explained (we postpone this discussion to
Section~\ref{locality-of-coacyclicity-subsecn} of the appendix;
cf.~\cite[Remark~1.3]{EP}).
 Part~(c) is straightforward: it suffices to observe that the functor
$\F^\bu\ot_{\cO_X}{-}$ takes short exact sequences of complexes in
$X\qcoh$ to short exact sequences of complexes in $X\qcoh$ and preserves
coproducts of complexes in $X\qcoh$.

 Part~(d): by Corolary~\ref{ind-Noetherian-coderived-cor}, there exists
a complex of injective quasi-coherent sheaves $\J^\bu$ on $X$ endowed
with a morphism of complexes $\M^\bu\rarrow\J^\bu$ whose cone $\N^\bu$
is coacyclic in $X\qcoh$.
 By part~(c), the complex $\F^\bu\ot_{\cO_X}\N^\bu$ is coacyclic in
$X\qcoh$.
 It remains to check that the complex $\F^\bu\ot_{\cO_X}\J^\bu$ is
coacyclic in $X\qcoh$; in fact, this complex is contractible.
 This is the result of~\cite[Corollary~9.7(ii)]{Neem1}
(see~\cite[Theorem~8.6]{Neem1} for context) and~\cite[Lemma~8.2]{Mur}.

 Let us spell out some details, following~\cite{Neem1,Mur}.
 By Lemma~\ref{hom-tensor-flats-injectives}(b),
\,$\F^\bu\ot_{\cO_X}\J^\bu$ is a complex of injective quasi-coherent
sheaves on~$X$.
 A complex of injective objects is coacyclic if and only if it is
contractible, and if and only if its objects of cocycles are injective.
 Injectivity of a quasi-coherent sheaf on a Noetherian scheme is a local
property; hence the question is local and reduces to affine schemes.

 Let $F^\bu$ be a complex of flat modules over a Noetherian commutative
ring $R$ and $J^\bu$ be a complex of injective $R$\+modules.
 Assume that the complex $F^\bu$ is acyclic in $R\flat$.
 Let $L$ be a finitely generated (equivalently, finitely presentable)
$R$\+module.
 Then we have a natural isomorphism of complexes of $R$\+modules
$\Hom_R(L,\>F^\bu\ot_RJ^\bu)\simeq F^\bu\ot_R\Hom_R(L,J^\bu)$.
 By part~(a), the complex $F^\bu\ot_R\Hom_R(L,J^\bu)$ is acyclic.
 By Lemma~\ref{complex-of-injectives-contractible}, acyclicity
of the complexes $\Hom_R(L,\>F^\bu\ot_RJ^\bu)$ for all
finitely generated $R$\+modules $L$ implies contractibility of
the complex $F^\bu\ot_RJ^\bu$.
\end{proof}

 Let $\fP^\bu$ and $\fQ^\bu\in\sC(\X\pro)$ be two complexes of
pro-quasi-coherent pro-sheaves on an ind-scheme~$\X$.
 Then the complex $\fP^\bu\ot^\X\fQ^\bu\in\sC(\X\pro)$ is constructed
by totalizing the bicomplex $\fP^p\ot^\X\fQ^q$, \ $p$, $q\in\boZ$,
by taking infinite coproducts in $\X\pro$ along the diagonals of
the bicomplex.
 Since the full subcategory $\X\flat$ is closed under coproducts in
$\X\pro$ (see Section~\ref{colimits-of-pro-sheaves-subsecn}),
the tensor product of two complexes in $\X\flat$ is a complex in
$\X\flat$,
\begin{equation} \label{flat-pro-sheaves-tensor-structure-on-complexes}
 \ot^\X\:\sC(\X\flat)\times\sC(\X\flat)\lrarrow\sC(\X\flat).
\end{equation}

 Similarly, let $\fP^\bu\in\sC(\X\pro)$ be a complex of
pro-quasi-coherent pro-sheaves and $\rM^\bu\in\sC(X\tors)$ be a complex
of quasi-coherent torsion sheaves on a reasonable ind-scheme~$\X$.
 Then the complex $\fP^\bu\ot_\X\rM^\bu\in\sC(\X\tors)$ is constructed
by totalizing the bicomplex $\fP^p\ot_\X\rM^q$, \ $p$, $q\in\boZ$,
by taking infinite coproducts in $\X\tors$ along the diagonals of
the bicomplex (as in the construction of the tensor product functor in
the proof of Theorem~\ref{ind-Noetherian-triangulated-equivalence-thm}).
 In particular, we have
\begin{equation} \label{flat-torsion-module-structure-on-complexes}
 \ot_\X\:\sC(\X\flat)\times\sC(\X\tors)\lrarrow\sC(\X\tors).
\end{equation}

 As the tensor product on $\X\pro$ and the action of $\X\pro$ in
$\X\tors$ preserve coproducts in both the categories
(see Section~\ref{colimits-of-pro-sheaves-subsecn}), the above
tensor products of complexes are associative.
 So $\sC(\X\pro)$ is a tensor category, $\sC(\X\flat)\subset
\sC(\X\pro)$ is a tensor subcategory, and $\sC(\X\tors)$ is
a module category over $\sC(\X\pro)$.
 Hence, in particular, $\sC(\X\tors)$ is a module category over
$\sC(\X\flat)$.

 The tensor product
functors~(\ref{flat-pro-sheaves-tensor-structure-on-complexes}\+-%
\ref{flat-torsion-module-structure-on-complexes}) obviously descend
to the homotopy categories, providing tensor product functors
\begin{gather}
 \label{flat-pro-sheaves-tensor-structure-on-homotopy}
 \ot^\X\:\sK(\X\flat)\times\sK(\X\flat)\lrarrow\sK(\X\flat), \\
 \label{flat-torsion-module-structure-on-homotopy}
 \ot_\X\:\sK(\X\flat)\times\sK(\X\tors)\lrarrow\sK(\X\tors),
\end{gather}
making $\sK(\X\flat)$ a tensor triangulated category and
$\sK(\X\tors)$ a triangulated module category over $\sK(\X\flat)$.


\begin{lem} \label{pro-torsion-complexes-tensor-exactness}
\textup{(a)} Let\/ $\X$ be an ind-scheme and\/ $\fF^\bu$, $\fG^\bu\in
\sC(\X\flat)$ be two complexes of flat pro-quasi-coherent pro-sheaves
on\/~$\X$.
 Assume that the complex\/ $\fF^\bu$ is acyclic in\/ $\X\flat$.
 Then the complex\/ $\fF^\bu\ot^\X\fG^\bu$ is acyclic in\/ $\X\flat$.
\par
\textup{(b)} Let\/ $\X$ be a reasonable ind-scheme, $\fF^\bu\in
\sC(\X\flat)$ be a complex of flat pro-quasi-coherent pro-sheaves
on\/ $\X$, and $\rM^\bu\in\sC(\X\tors)$ be a complex of quasi-coherent
torsion sheaves on\/~$\X$.
 Assume that the complex\/ $\rM^\bu$ is coacyclic in\/ $\X\tors$.
 Then the complex\/ $\fF^\bu\ot_\X\rM^\bu$ is coacyclic in\/
$\X\tors$. \par
\textup{(c)} Let\/ $\X$ be an ind-Noetherian ind-scheme,
$\fF^\bu\in\sC(\X\flat)$ be a complex of flat pro-quasi-coherent
pro-sheaves on\/ $\X$, and $\rM^\bu\in\sC(\X\tors)$ be a complex of
quasi-coherent torsion sheaves on\/~$\X$.
 Assume that the complex\/ $\fF^\bu$ is acyclic in\/ $\X\flat$.
 Then the complex\/ $\fF^\bu\ot_\X\rM^\bu$ is coacyclic in\/
$\X\tors$.
\end{lem}

\begin{proof}
 Part~(a) follows from
Lemmas~\ref{flat-pro-sheaves-complex-acyclicity-criterion}
and~\ref{qcoh-complexes-tensor-exactness}(b).
 Part~(b) holds, because the functor $\fF^\bu\ot_\X{-}$ takes short
exact sequences of complexes in $\X\tors$ to short exact sequences
of complexes in $\X\tors$ and preserves coproducts of complexes in
$\X\tors$.
 Similarly one shows that the complex $\fF^\bu\ot_\X\rM^\bu$ is
coacyclic in $\X\tors$ whenever the complex $\fF^\bu$ is coacyclic
in $\X\flat$ and $\rM^\bu$ is an arbitrary complex in $\X\tors$.

 Part~(c): by Corollary~\ref{ind-Noetherian-coderived-cor},
there exists a complex of injective quasi-coherent torsion sheaves
$\rJ^\bu\in\sK(\X\tors_\inj)$ together with a morphism of complexes
$\rM^\bu\rarrow\rJ^\bu$ whose cone $\rN^\bu$ is coacyclic in $\X\tors$.
 By part~(b), the complex $\fF^\bu\ot_\X\rN^\bu$ is coacyclic in
$\X\tors$.
 It remains to check that the complex $\fF^\bu\ot_\X\rJ^\bu$ is
coacyclic in $\X\tors$.
 In fact, we will show that this complex is contractible.

 Indeed, according to the arguments in the beginning of the proof of
Theorem~\ref{ind-Noetherian-triangulated-equivalence-thm} in
Section~\ref{ind-Noetherian-triangulated-equivalence-subsecn},
\,$\fF^\bu\ot_\X\rJ^\bu$ is a complex of injective quasi-coherent
torsion sheaves on~$\X$.
 Furthermore, for any closed subscheme $Z\subset\X$ with the closed
immersion morphism $i\:Z\rarrow\X$,
Proposition~\ref{flat-torsion-tensor-prop} provides a natural
isomorphism of complexes of quasi-coherent sheaves
$i^!(\fF^\bu\ot_\X\rJ^\bu)\simeq i^*\fF^\bu\ot_{\cO_Z}i^!\rJ^\bu$
on~$Z$.
 In the situation at hand, the complex $i^*\fF^\bu$ is acyclic in
$Z\flat$, while $i^!\rJ^\bu$ is a complex of injective quasi-coherent
sheaves on~$Z$; hence, by (the proof of)
Lemma~\ref{qcoh-complexes-tensor-exactness}(d), the complex of injective
quasi-coherent sheaves $i^*\fF^\bu\ot_{\cO_Z}i^!\rJ^\bu$ on $Z$
is contractible.
 By Lemma~\ref{complex-of-injective-torsion-contractible}, it follows
that the complex $\fF^\bu\ot_\X\rJ^\bu$ in $\X\tors_\inj$
is contractible, too.
\end{proof}

 Let $\X$ be an ind-Noetherian ind-scheme.
 It is clear from Lemma~\ref{pro-torsion-complexes-tensor-exactness}
that the tensor product
functors~(\ref{flat-pro-sheaves-tensor-structure-on-homotopy}\+-%
\ref{flat-torsion-module-structure-on-homotopy}) descend to
the derived and coderived categories, providing tensor product functors
\begin{gather}
 \label{flat-pro-sheaves-tensor-structure-on-derived}
 \ot^\X\:\sD(\X\flat)\times\sD(\X\flat)\lrarrow\sD(\X\flat), \\
 \label{flat-torsion-module-structure-on-co-derived}
 \ot_\X\:\sD(\X\flat)\times\sD^\co(\X\tors)\lrarrow\sD^\co(\X\tors).
\end{gather}
 So $\sD(\X\flat)$ is a tensor triangulated category and
$\sD^\co(\X\tors)$ is a triangulated module category
over $\sD(\X\flat)$.

 Now let $\X$ be an ind-semi-separated ind-Noetherian ind-scheme with
a dualizing complex~$\rD^\bu$.
 Then the triangulated equivalence $\rD^\bu\ot_\X{-}\:\sD(\X\flat)
\rarrow\sD^\co(\X\tors)$ from
Theorem~\ref{ind-Noetherian-triangulated-equivalence-thm}
is an equivalence of module categories over $\sD(\X\flat)$.
 This follows from the associativity of the tensor products,
$$
 (\fF^\bu\ot^\X\fG^\bu)\ot_\X\rD^\bu\simeq
 \fF^\bu\ot_\X(\fG^\bu\ot_\X\rD^\bu)
$$
for all complexes of flat pro-quasi-coherent pro-sheaves $\fF^\bu$
and $\fG^\bu$ on~$\X$.

 Using the triangulated equivalence $\sD(\X\flat)\simeq
\sD^\co(\X\tors)$, we transfer the tensor structure of the category
$\sD(\X\flat)$ to the category $\sD^\co(\X\tors)$.
 The resulting functor
\begin{equation} \label{coderived-torsion-cotensor-product}
 \oc_{\rD^\bu}\:\sD^\co(\X\tors)\times\sD^\co(\X\tors)\lrarrow
 \sD^\co(\X\tors),
\end{equation}
defining a tensor triangulated category structure on $\sD^\co(\X\tors)$, 
is called the \emph{cotensor product} of complexes of quasi-coherent
torsion sheaves on $\X$ over the dualizing complex~$\rD^\bu$.
 Explicitly, we have
\begin{multline*}
 \rM^\bu\oc_{\rD^\bu}\rN^\bu=
 \rD^\bu\ot_\X(\fHom_{\X\qc}(\rD^\bu,\rK^\bu)
 \ot^\X\fHom_{\X\qc}(\rD^\bu,\rJ^\bu)) \\ \simeq
 \rM^\bu\ot_\X\fHom_{\X\qc}(\rD^\bu,\rJ^\bu)
\end{multline*}
in $\sD^\co(\X\tors)$ for any complexes $\rM^\bu$ and $\rN^\bu\in
\sK(\X\tors)$ endowed with morphisms with coacyclic cones
$\rM^\bu\rarrow\rK^\bu$ and $\rN^\bu\rarrow\rJ^\bu$ into complexes
$\rK^\bu$ and $\rJ^\bu\in\sK(\X\tors_\inj)$.
 The dualizing complex $\rD^\bu\in\sD^\co(\X\tors)$ is the unit object
of the tensor structure $\oc_{\rD^\bu}$ on $\sD^\co(\X\tors)$, since
$\rD^\bu$ corresponds to the unit object $\fO_\X\in\sD(\X\flat)$ under
the equivalence of categories $\sD(\X\flat)\simeq\sD^\co(\X\tors)$.

\begin{rems} \label{cotensor-right-derived-remark}
 The cotensor product functor~$\oc_{\rD^\bu}$
\,\eqref{coderived-torsion-cotensor-product} is ``similar to
a right derived functor'' in the following sense
(see Theorem~\ref{cotensor-shriek-tensor-thm} below for
a much more specific assertion).

\smallskip
 (1)~Given an additive category $\sA$, let us denote by
$\sK^{\le0}(\sA)$ and $\sK^{\ge0}(\sA)\subset\sK(\sA)$ the full
subcategories in the homotopy category consisting of complexes
concentrated in the nonpositive and nonnegative cohomological degrees,
respectively.
 For an abelian category $\sA$, the notation $\sD^{\le0}(\sA)$ and
$\sD^{\ge0}(\sA)$ is understood similarly.
 When $\sA$ is abelian, a complex $A^\bu\in\sD^+(\sA)$ is said to have
\emph{injective dimension\/~$\le0$} if $\Hom_{\sD^+(\sA)}
(C^\bu,A^\bu)=0$ for all bounded complexes $C^\bu\in\sD^{\ge0}(\sA)$.
 When $\sA$ is abelian with enough injective objects, a bounded below
complex in $\sA$ has injective dimension~$\le0$ if and only if it
is quasi-isomorphic to a bounded complex in $\sK^{\le0}(\sA_\inj)$.

 For any abelian category $\sA$ with exact functors of infinite
direct sum, the full subcategories of nonpositively and nonnegatively
situated complexes $\sD^{\co,\le0}(\sA)$ and $\sD^{\co,\ge0}(\sA)$
form a t\+structure (of the derived type) on the coderived category
$\sD^\co(\sA)$ \,\cite[Remark~4.1]{Psemi}, \cite[Proposition~5.5]{PS2}.
 Furthermore, if there are enough injective objects in $\sA$, then
for any complex $A^\bu\in\sK^{\ge0}(\sA)$ there is a complex
$J^\bu\in\sK^{\ge0}(\sA_\inj)$ together with a quasi-isomorphism
$A^\bu\rarrow J^\bu$ of complexes in~$\sA$.
 Since any bounded below acyclic complex is
coacyclic~\cite[Lemma~2.1]{Psemi} (cf.~\cite[Lemma~A.1.2(a)]{Pcosh}),
the morphism of complexes $A^\bu\rarrow J^\bu$ is an isomorphism in
$\sD^\co(\sA)$.
 So we have equivalences of categories $\sK^{\ge0}(\sA_\inj)\simeq
\sD^{\co,\ge0}(\sA)\simeq\sD^{\ge0}(\sA)$.

\smallskip
 (2)~In the context of the exposition above in this section, assume that
the dualizing complex $\rD^\bu$ is concentrated in the nonpositive
cohomological degrees, $\rD^\bu\in\sK^{\le0}(\X\tors_\inj)$.
 For any complexes $\rM^\bu$, $\rN^\bu\in\sK^{\ge0}(\X\tors)$,
one can choose a complex $\rJ^\bu\in\sK^{\ge0}(\X\tors_\inj)$
together with a morphism with (co)acyclic cone $\rN^\bu\rarrow\rJ^\bu$.
 Then one has $\fHom_{\X\qc}(\rD^\bu,\rJ^\bu)\in\sK^{\ge0}(\X\flat)$
and $\rM^\bu\ot_\X\fHom_{\X\qc}(\rD^\bu,\rJ^\bu)\in\sK^{\ge0}(\X\tors)$.
 Thus the functor~$\oc_{\rD^\bu}$
\,\eqref{coderived-torsion-cotensor-product} restricts to a functor
$$
 \oc_{\rD^\bu}\:\sD^{\co,\ge0}(\X\tors)\times\sD^{\co,\ge0}(\X\tors)
 \lrarrow\sD^{\co,\ge0}(\X\tors).
$$

\smallskip
 (3)~Let us explain why the assumption that
$\rD^\bu\in\sK^{\le0}(\X\tors_\inj)$ is mild and reasonable.
 Suppose that the ind-scheme $\X$ is of ind-finite type over
a Noetherian scheme $S$ with a dualizing complex~$\D^\bu$.
 Without loss of generality, one can assume $\D^\bu$ to be a bounded
complex of injective quasi-coherent sheaves on~$S$;
shifting if necessary, one can assume further that
$\D^\bu\in\sK^{\le0}(S\qcoh_\inj)$.
 Let $Z\subset\X$ be a closed subscheme; so we have a morphism of
finite type $Z\rarrow S$.

 The idea is to use the extraordinary inverse image functor in
order to lift $\D^\bu$ to a dualizing complex $\D^\bu_Z$ on~$Z$.
 Given a morphism of finite type between Noetherian schemes
$f\:Y\rarrow X$, the relevant functor is denoted by~$f^!$
in~\cite{Hart} and Deligne's appendix to~\cite{Hart} (where it was
first constructed, under mild assumptions on~$f$); in the terminology
and notation of~\cite{Pcosh}, it is called the \emph{extraordinary
inverse image functor in the sense of Deligne} and denoted by
$f^+\:\sD^+(X\qcoh)\rarrow\sD^+(Y\qcoh)$. 

 First of all one observes that the derived direct image functor
$\boR f_*\:\sD^+(Y\qcoh)\rarrow\sD^+(X\qcoh)$ has a right adjoint
functor; this functor is called the \emph{extraordinary inverse image
functor in the sense of Neeman} and denoted by~$f^!$ in~\cite{Pcosh}.
 Both the functor $\boR f_*$ and its right adjoint~$f^!$ are also
well-defined on the unbounded derived categories and the coderived
categories.
 In particular, for a closed immersion $i\:Z\rarrow Y$,
the triangulated functor $i^!=\boR i^!\:\sD^+(Y\qcoh)\rarrow
\sD^+(Z\qcoh)$ is simply the right derived functor of the left
exact functor $i^!\:Y\qcoh\rarrow Z\qcoh$.

 The functor~$f^+$ is essentially characterized by three properties:
for a composable pair of morphisms $f$ and~$g$, one has $(fg)^+\simeq
g^+f^+$; for an open immersion~$g$, one has $g^+=g^*$; for a proper
morphism~$f$, one has $f^+=f^!$.
 In particular, for a closed immersion~$i$ one has $i^+=\boR i^!$.
 For a smooth morphism~$f$, the functor~$f^+$ only differs from~$f^*$
by a shift and a twist~\cite[Chapter~III]{Hart}.

 The functor~$f^+$ takes dualizing complexes to dualizing
complexes~\cite[Proposition~V.2.4, Theorem~V.8.3 and Remark in
Section~V.8]{Hart}.
 More precisely, in our terminology one can say that, given a morphism
$f\:Y\rarrow X$ and a dualizing complex $\D_X^\bu$ on~$X$,
the object $f^+\D_X^\bu\in\sD^+(Y\qcoh)$ is quasi-isomorphic to
a dualizing complex (of injective quasi-coherent sheaves) $\D_Y^\bu$
on~$Y$.

 Concerning an unbounded version of the functor~$f^+$, it turns out
to be well-defined as a functor between the coderived categories
$f^+\:\sD^\co(X\qcoh)\rarrow\sD^\co(Y\qcoh)$, but \emph{not} as
a functor between the conventional unbounded derived categories
$\sD(X\qcoh)$ and $\sD(Y\qcoh)$ \,\cite{Gai}, \cite[Introduction
and Section~5.16]{Pcosh}.

 In the recent overview~\cite{Neem-rev}, the notation $f^\times$ is
used for the functor which we denote by~$f^!$, and the notation~$f^!$
is used for the functor which we denote by~$f^+$.

\smallskip
 (4)~The key observation for our purposes is that the functor~$f^+$
takes complexes of injective dimension~$\le0$ to complexes of injective
dimension~$\le0$.
 Indeed, the restriction to an open subscheme in a Noetherian scheme
preserves injective dimension, since it preserves injectivity.
 It remains to see that the right adjoint functor~$f^!$ to
the functor $\boR f_*$ preserves injective dimension.
 Indeed, if $\N^\bu\in\sD^{\ge0}(Y\qcoh)$ is a bounded complex and
$\K^\bu\in\sK^{\le0}(X\qcoh_\inj)$ is a bounded complex of injective
dimension~$\le0$, then
$$
 \Hom_{\sD^+(Y\qcoh)}(\N^\bu,f^!\K^\bu)\simeq
 \Hom_{\sD^+(X\qcoh)}(\boR f_*\N^\bu,\K^\bu)=0
$$
since $\boR f_*\N^\bu\in\sD^{\ge0}(X\qcoh)$.

\smallskip
 (5)~Returning to the situation at hand, we have a closed subscheme
$Z\subset\X$; denote by $p_Z\:Z\rarrow S$ the related morphism, and
let $\D_Z^\bu$ be a dualizing complex (of injective quasi-coherent
sheaves) on $Z$ quasi-isomorphic to $p_Z^+\D^\bu\in\sD^+(Z\qcoh)$.
 The argument above allows to have
$\D_Z^\bu\in\sK^{\le0}(Z\qcoh_\inj)$.
 For any pair of closed subschemes $Z\subset Y\subset\D^\bu$ with
the related open immersion $i\:Z\rarrow Y$, the complex $i^!\D_Y^\bu\in
\sK^{\le0}(Z\qcoh_\inj)$ is naturally homotopy equivalent to $\D_Z^\bu$,
since $p_Z=p_Yi$ and $i^+\simeq\boR i^!$.
 In this sense, the dualizing complexes $\D_Z^\bu$ on closed
subschemes $Z\subset\X$ agree with each other up to homotopy 
equivalence.

 Assuming additionally that $\X$ is an $\aleph_0$\+ind-scheme, one can
apply the construction of Example~\ref{aleph-zero-dualizing-glueing}
in order to produce a dualizing complex $\rD^\bu\in\sK^{\le0}(\X\tors)$
on $\X$ out of the dualizing complexes $\D_Z^\bu$ on the closed
subschemes $Z\subset\X$.
\end{rems}

\subsection{Ind-Artinian examples}
 The following example explains the terminology ``coderived category''
and ``cotensor product''.

\begin{exs} \label{cotensor-over-coalgebra}
 (1)~Let $\rC$ be a coassociative, counital coalgebra over
a field~$\kk$.
 Let $\rM$ be a right $\rC$\+comodule and $\rN$ be a left
$\rC$\+comodule.
 Then the \emph{cotensor product} $\rM\oc_\rC\rN$ is the $\kk$\+vector
space constructed as the kernel of the difference of the natural pair
of maps
$$
 \rM\ot_\kk\rN\,\rightrightarrows\,\rM\ot_\kk\rC\ot_\kk\rN.
$$
 Here one map $\rM\ot_\kk\rN\rarrow\rM\ot_\kk\rC\ot_\kk\rN$ is induced
by the right coaction map $\rM\rarrow\rM\ot_\kk\rC$ and the other one
by the left coaction map $\rN\rarrow\rC\ot_\kk\rN$.

\smallskip
 (2)~When the coalgebra $\rC$ is cocommutative, there is no difference
between left and right $\rC$\+comodules.
 Moreover, the cotensor product $\rM\oc_\rC\rN$ of two $\rC$\+comodules
$\rM$ and $\rN$ has a natural $\rC$\+comodule structure in this case.
 The cotensor product operation $\oc_\rC$ makes the abelian category of
$\rC$\+comodules $\rC\comodl$ an associative, commutative, and unital
tensor category with the unit object $\rC\in\rC\comodl$.

\smallskip
 (3)~For any coalgebra $\rC$ as in~(1), the categories $\rC\comodl$
and $\comodr\rC$ of left and right $\rC$\+comodules are locally
Noetherian (in fact, locally finite) Grothendieck abelian categories,
so Proposition~\ref{coderived-and-homotopy-of-injectives} with
Lemma~\ref{locally-Noetherian-coproducts-injective} are applicable.
 Hence the coderived categories of $\rC$\+comodules are equivalent to
the respective homotopy categories of (complexes of) injective
comodules, $\sD^\co(\rC\comodl)\simeq\sK(\rC\comodl_\inj)$ and
$\sD^\co(\comodr\rC)\simeq\sK(\comodrinj\rC)$.

 The left $\rC$\+comodule $\rC$ is an injective cogenerator of
$\rC\comodl$; moreover, a $\rC$\+comodule is injective if and only
if it is a direct summand of a direct sum of copies of
the $\rC$\+comodule~$\rC$.
 For a given left $\rC$\+comodule $\rJ$, the coproduct functor
${-}\oc_\rC\nobreak\rJ\:\comodr\rC\rarrow\kk\vect$ is exact
if and only if $\rJ$ is an injective left $\rC$\+comodule (and 
similarly for a right $\rC$\+comodule).
 Here $\kk\vect$ denotes the category of $\kk$\+vector spaces.

 The functor of cotensor product of complexes of comodules
$$
 \oc_\rC\:\sC(\comodr\rC)\times\sC(\rC\comodl)\lrarrow
 \sC(\kk\vect)
$$
is constructed in the obvious way (totalizing the bicomplex of cotensor
products by taking direct sums along the diagonals).
 To define the right derived functor of cotensor product
$$
 \oc_\rC^\boR\:\sD^\co(\comodr\rC)\times\sD^\co(\rC\comodl)\lrarrow
 \sD(\kk\vect),
$$
suppose that we are given a complex of right $\rC$\+comodules $\rM^\bu$
and a complex of left $\rC$\+comodules $\rN^\bu$.
 Let $\rM^\bu\rarrow\rK^\bu$ be a morphism in $\sC(\comodr\rC)$ from
$\rM^\bu$ to a complex of injective right $\rC$\+comodules $\rK^\bu$
such that the cone of $\rM^\bu\rarrow\rK^\bu$ is coacyclic in
$\comodr\rC$.
 Let $\rN^\bu\rarrow\rJ^\bu$ be a similar resolution of the complex
of left $\rC$\+comodules~$\rN^\bu$; so $\rJ^\bu\in\sC(\rC\comodl_\inj)$
and the cone is coacyclic in $\rC\comodl$.
 Then the derived cotensor product is defined as the object
$$
 \rM^\bu\oc^\boR_\rC\rN^\bu=\rK^\bu\oc_\rC\rJ^\bu\simeq
 \rM^\bu\oc_\rC\rJ^\bu\simeq\rK^\bu\oc_\rC\rN^\bu\,\in\,\sD(\kk\vect).
$$
 We refer to~\cite[Section~0.2]{Psemi} for a discussion.

\smallskip
 (4)~For a cocommutative coalgebra $\rC$ as in~(2), the functor of
cotensor product of complexes of comodules
$$
 \oc_\rC\:\sC(\rC\comodl)\times\sC(\rC\comodl)\lrarrow
 \sC(\rC\comodl)
$$
is constructed, once again, by totalizing the complex of cotensor
products by taking coproducts along the diagonals (notice that
the forgetful functor $\rC\comodl\rarrow\kk\vect$ preserves coproducts).
 The construction of the right derived functor of cotensor product
$$
 \oc_\rC^\boR\:\sD^\co(\rC\comodl)\times\sD^\co(\rC\comodl)\lrarrow
 \sD^\co(\rC\comodl)
$$
is similar to the one in~(3).
 In the same notation and the same conditions on the resolutions
$\rM^\bu\rarrow\rK^\bu$ and $\rN^\bu\rarrow\rJ^\bu$, we put
$$
 \rM^\bu\oc^\boR_\rC\rN^\bu=\rK^\bu\oc_\rC\rJ^\bu\simeq
 \rM^\bu\oc_\rC\rJ^\bu\simeq\rK^\bu\oc_\rC\rN^\bu
 \,\in\,\sD^\co(\rC\comodl).
$$
 The derived tensor product operation~$\oc_\rC^\boR$ makes
$\sD^\co(\rC\comodl)$ an associative, commutative, and unital tensor
triangulated category.

\smallskip
 (5)~For a cocommutative coalgebra $\rC$, let $\X=\Spi\rC^*$ be
the ind-Artinian ind-scheme corresponding to $\rC$ as per
the construction from Example~\ref{coalgebra-ind-scheme}(2)
and the discussion in Example~\ref{ind-Artinian-ind-schemes}(2).
 According to Section~\ref{torsion-ind-affine-subsecn}(4),
the abelian category $\X\tors$ of quasi-coherent torsion sheaves on $\X$
is equivalent to the abelian category of $\rC$\+comodules $\rC\comodl$.

 Notice that any Artinian scheme admits a dualizing injective
quasi-coherent sheaf, i.~e., a dualizing complex which is
a one-term complex of injective quasi-coherent sheaves.
 In particular, for any finite-dimensional cocommutative coalgebra $\E$
over~$\kk$, the quasi-coherent sheaf on $\Spec\E^*$ corresponding to
the injective $\E^*$\+module $\E$ is a dualizing complex on $\Spec\E^*$.
 It follows that the injective quasi-coherent torsion sheaf on
$\Spi\rC^*$ corresponding to the injective $\rC$\+comodule $\rC$ is
a dualizing complex on~$\Spi\rC^*$.

 Following the proof of
Theorem~\ref{ind-Noetherian-triangulated-equivalence-thm} specialized
to the particular case of a one-term dualizing complex of injectives
$\rD^\bu=\rC$, one can see that there is an equivalence of additive
categories of injective quasi-coherent torsion sheaves and flat
pro-quasi-coherent pro-sheaves on~$\X$, provided by the mutually
inverse functors $\fHom_{\X\qc}(\rC,{-})$ and $\rC\ot_\X{-}$,
\begin{equation} \label{spi-coalgebra-injective-flat-correspondence}
 \fHom_{\X\qc}(\rC,{-})\:\X\tors_\inj\,\simeq\,
 \X\flat\,:\!\rC\ot_\X{-}.
\end{equation}
 Moreover, this equivalence transforms short exact sequences in
$\X\flat$ into split short exact sequences in $\X\tors_\inj$; so
the exact category structure on $\X\flat$ is split.

\smallskip
 (6)~We refer to~\cite[Section~0.2.4]{Psemi} or~\cite[Section~1]{Prev}
for an introductory discussion of the category of left contramodules
$\rC\contra$ over a coassociative coalgebra~$\rC$.
  Contramodules over a coalgebra $\rC$ are the same thing as 
contramodules over the topological ring $\fR=\rC^*$ \,
\cite[Section~1.10]{Pweak}, \cite[Sections~2.1 and~2.3]{Prev}.

 The category $\rC\contra$ is abelian with enough projective objects.
 Denoting by $\rC\contra_\proj\subset\rC\contra$ the full subcategory
of projective $\rC$\+contramodules, one has a natural equivalence
of additive categories~\cite[Section~0.2.6]{Psemi},
\cite[Sections~1.2 and 3.1]{Prev}, \cite[Sections~5.1\+-5.2]{Pkoszul}
\begin{equation} \label{coalgebra-underived-co-contra-correspondence}
 \Hom_\rC(\rC,{-})\:\rC\comodl_\inj\,\simeq\,
 \rC\contra_\proj\,:\!\rC\ocn_\rC{-}.
\end{equation}
 Here $\Hom_\rC=\Hom_{\rC\comodl}$ denotes the $\kk$\+vector space of
morphisms in the abelian category $\rC\comodl$, while $\ocn_\rC$ is
the \emph{contratensor product} functor.

 For a cocommutative coalgebra $\rC$, it is not difficult to construct
a pair of adjoint functors between the additive categories $\X\pro$
and $\rC\contra$ (cf.\
Examples~\ref{countable-flat-pro-sheaves-flat-contramodules}).
 For a finite-dimensional subcoalgebra $\E\subset\rC$, let $X_\E
\subset\X$ denote the closed subscheme $\Spec\E^*\subset\Spi\rC^*$.
 For a finite-dimensional coalgebra $\E$, the category of
$\E$\+contramodules is naturally equivalent to the category
of $\E^*$\+modules.
 
 The left adjoint functor $\rC\contra\rarrow\X\pro$ assigns to
a $\rC$\+contramodule $\fG$ the pro-quasi-coherent pro-sheaf $\fP$
whose component $\fP^{(X_\E)}$ is the $\E^*$\+module produced as
the maximal quotient contramodule of $\fG$ whose $\rC$\+contramodule
structure comes from an $\E$\+contramodule structure.
 The right adjoint functor $\X\pro\rarrow\rC\contra$ assigns to
a pro-quasi-coherent pro-sheaf $\fP$ the projective limit
$\varprojlim_{\E\subset\rC}\fP^{(X_\E)}$.

 Furthermore, it is well-known that all flat modules over an Artinian
ring are projective.
 Moreover, all (contra)flat contramodules over a coalgebra over a field
are projective~\cite[Sections~0.2.9 and~A.3]{Psemi}.
 The functor $\rC\contra\rarrow\X\pro$ restricts to a functor
$\rC\contra_\proj\rarrow\X\flat$, which is obviously fully faithful.

 Comparing the two equivalences of additive
categories~\eqref{spi-coalgebra-injective-flat-correspondence}
and~\eqref{coalgebra-underived-co-contra-correspondence}
and taking into account the equivalence $\X\tors_\inj\simeq
\rC\comodl_\inj$ induced by the equivalence $\X\tors\simeq
\rC\comodl$, one can see that the functor $\rC\contra_\proj
\rarrow\X\flat$ is an equivalence of additive categories.
 Moreover, the projective limit functor $\X\pro\rarrow\rC\contra$
restricts to a functor $\X\flat\rarrow\rC\contra_\proj$, providing
the inverse equivalence.

 Similarly to
Example~\ref{countable-flat-pro-sheaves-flat-contramodules}(3),
the equivalences of categories $\X\tors\simeq\rC\comodl$ and
$\X\flat\simeq\rC\contra_\proj$ transform the tensor product functor
$\ot_\X\:\X\tors\times\X\flat\rarrow\X\tors$ into the contratensor
product functor $\ocn_\rC\:\rC\comodl\times\rC\contra\rarrow
\rC\comodl$ restricted to $\rC\contra_\proj\subset\rC\contra$.

\smallskip
 (7)~Let us explain why the equivalence of coderived categories
$\sD^\co(\X\tors)\simeq\sD^\co(\rC\comodl)$ induced by
the equivalence of abelian categories $\X\tors\simeq\rC\comodl$
transforms the cotensor product functor~$\oc_{\rD^\bu}$
\,\eqref{coderived-torsion-cotensor-product} from
Section~\ref{cotensor-product-construction-subsecn} for
the dualizing complex $\rD^\bu=\rC$ on the ind-scheme $\X=\Spi\rC^*$
into the right derived cotensor product functor $\oc_\rC^\boR$
from Example~\ref{cotensor-over-coalgebra}(4).

 Let $\E\subset\rC$ be a finite-dimensional subcoalgebra and
$i_\E\:X_\E\rarrow\X$ be the related immersion of the closed subscheme
into the ind-scheme.
 Then the functor $i_\E^!\:\X\tors\rarrow X_\E\qcoh$ corresponds,
under the equivalences of abelian categories $\X\tors\simeq
\rC\comodl$ and $X_\E\qcoh\simeq\E\comodl$, to the functor
$\rC\comodl\rarrow\E\comodl$ assigning to an $\rC$\+comodule $\rM$
its maximal subcomodule $\rM_{(\E)}\subset\rM$ whose $\rC$\+comodule
structure comes from an $\E$\+comodule structure.
 The $\E$\+comodule $\rM_{(\E)}$ can be computed as the cotensor
product
$$
 \rM_{(\E)}\simeq\E\oc_\rC\rM.
$$
 Consequently, for any two $\rC$\+comodules $\rM$ and $\rN$ one has
$$
 (\rM\oc_\rC\rN)_{(\E)}\simeq\E\oc_\rC\rM\oc_\rC\rN\simeq
 (\E\oc_\rC\rM)\oc_\E(\E\oc_\rC\rN)\simeq\rM_{(\E)}\oc_\E\rN_{(\E)}.
$$
 Together with the computations in the proof of
Theorem~\ref{ind-Noetherian-triangulated-equivalence-thm}, this
reduces the question to the case of a finite-dimensional coalgebra
$\E$ and related Artinian scheme $X_\E=\Spec\E^*$, for which
it means the following.

 Let $\M$ be an $\E$\+comodule and $\J$ be an injective $\E$\+comodule
(we recall once again that an $\E$\+comodule is the same thing as
an $\E^*$\+module).
 Then there are natural isomorphisms of $\E$\+comodules
\begin{multline*}
 \M\oc_\E\J\simeq\M\oc_\E(\E\ot_{\E^*}\Hom_{\E^*}(\E,\J)) \\
 \simeq(\M\oc_\E\E)\ot_{\E^*}\Hom_{\E^*}(\E,\J)
 \simeq\M\ot_{\E^*}\Hom_{\E^*}(\E,\J).
\end{multline*}
 Here the injective $\E$\+comodule $\E$ corresponds to the dualizing
(one-term) complex $\D^\bu=i_\E^!\rD^\bu$ on the scheme $X_\E$
under the equivalence of categories $X_\E\qcoh\simeq\E\comodl$,
where $\rD^\bu\in\X\tors$ corresponds to $\rC\in\rC\comodl$ under
$\X\tors\simeq\rC\comodl$, as per the discussion in~(5).
 Notice that the $\E^*$\+module $\Hom_{\E^*}(\E,\J)$ is projective
for any injective $\E^*$\+module~$\J$.

\smallskip
 (8)~Alternatively, one can avoid the reduction to finite-dimensional
coalgebras when establishing the comparison between $\oc_{\rD^\bu}$
and $\oc_\rC^\boR$, by using $\rC$\+contramodules and the discussion of
the equivalence $\X\flat\simeq\rC\contra_\proj$ in~(6).

 In this context, the desired comparison is expressed by the natural
isomorphisms of $\rC$\+comodules
\begin{multline*}
 \rM\oc_\rC\rJ\simeq\rM\oc_\rC(\rC\ocn_\rC\Hom_\rC(\rC,\rJ)) \\
 \simeq(\rM\oc_\rC\rC)\ocn_\rC\Hom_\rC(\rC,\rJ)
 \simeq\rM\ocn_\rC\Hom_\rC(\rC,\rJ),
\end{multline*}
which hold for any $\rC$\+comodule $\rM$ and any injective
$\rC$\+comodule~$\rJ$.
 We refer to~\cite[Proposition~5.2.1]{Psemi}
or~\cite[Proposition~3.1.1]{Prev} for a discussion on this kind of
associativity isomorphisms connecting the cotensor and contratensor
products.
\end{exs}

 Our next aim is to discuss the \emph{torsion product} functor
$\Tor^R_1({-},{-})$ for modules over a Dedekind domain~$R$.
 Let $Q$ denote the field of quotients of~$R$.
 The motivating example for us is the case of torsion abelian groups,
when $R=\boZ$ and $Q=\boQ$.
 We refer to the book~\cite[Section~V.6]{McL} or the overview~\cite{Kee}
for the explicit construction and discussion of the torsion products
of abelian groups (see also the classical paper~\cite{Nun}).
 Let us start the discussion in a more general context of a Noetherian
ring of Krull dimension~$1$ before specializing to Dedekind domains.

\begin{exs} \label{krulldim1-torsion-ind-scheme}
 (1)~Let $R$ be a Noetherian commutative ring of Krull dimension~$\le1$,
and let $S\subset R$ denote the complement to the union of all
the nonmaximal prime ideals in~$R$.
 Notice that all the nonmaximal prime ideals in $R$ are minimal,
hence there is only a finite number of nonmaximal prime ideals;
so an ideal $I\subset R$ does not intersect $S$ if and only if
it is contained in one of the nonmaximal prime ideals.
 If $I$ does intersect $S$, then $R/I$ is an Artinian ring.
 We refer to~\cite[Section~13]{Pcta} for a more detailed discussion
of this setting.

 Let $\Gamma$ denote the directed poset of all ideals in $R$ 
intersecting $S$, with respect to the reverse inclusion order.
 To every ideal $I\in\Gamma$, assign the Artinian scheme $X_I=
\Spec R/I$.
 Whenever $I''\subset I'$ are two ideals in $R$, with $I''$
intersecting $S$, there is a unique (surjective) ring homomorphism
$R/I''\rarrow R/I'$ forming a commutative triangle diagram with
the natural projections $R\rarrow R/I''$ and $R\rarrow R/I'$.
 Let $X_{I'}\rarrow X_{I''}$ be the related closed immersion of
Artinian schemes.
 The inductive system of schemes $(X_I)_{I\in\Gamma}$ represents
an ind-Artinian ind-scheme~$\X$.
 The ind-scheme $\X$ comes together with an ind-closed immersion
of ind-schemes $\X\rarrow\Spec R$ (in the sense of
Examples~\ref{ind-closed-immersion}).

\smallskip
 (2)~In the context of~(1), let $\fR$ be the topological ring
$\fR=\varprojlim_{I\in\Gamma}R/I$, with the topology of projective
limit of discrete rings~$R/I$.
 The topological ring $\fR$ can be computed as the topological
product $\fR=\prod_{m\subset R}\widehat R_m$, where $m$~ranges over
the maximal ideals of $R$ and $\widehat R_m$ is the completion of
the local ring~$R_m$.
 The complete local ring $\widehat R_m$ is endowed with
the $m$\+adic topology and the product $\prod_{m\subset R}\widehat R_m$
is endowed with the product topology.
 Then the ind-scheme $\X=\ilim_{I\in\Gamma}X_I$ from~(1) can be
described as $\X=\Spi\fR$, in the notation of
Example~\ref{topological-ring-ind-scheme}(1).

 An $R$\+module $M$ is said to be \emph{$S$\+torsion} if for every
$b\in M$ there exists $s\in S$ such that $sb=0$ in~$M$.
 Denote by $Q=S^{-1}R$ the localization of the ring $R$ at
the multiplicative subset~$S$.
 An $R$\+module $M$ is $S$\+torsion if and only if $Q\ot_RM=0$.
 Denote by $R\tors\subset R\modl$ the full subcategory of $S$\+torsion
$R$\+modules; clearly, $R\tors$ is a locally Noetherian (in fact,
locally finite) Grothendieck abelian category, which is closed under
subobjects, quotients, and extensions as a full subcategory in $R\modl$.
 The category of $S$\+torsion $R$\+modules is naturally equivalent to
the category of discrete $\fR$\+modules and to the category of
quasi-coherent torsion sheaves on $\X$, that is
$R\tors\simeq\fR\discr\simeq\X\tors$
(cf.\ Section~\ref{torsion-ind-affine-subsecn}(6)).

 The ind-scheme $\X$ is the disjoint union (coproduct) of
the ind-schemes $\Spi\widehat R_m$ over the maximal ideals $m\subset R$.
 The topological ring $\widehat R_m$ has a countable base of
neighborhoods of zero, so the category of flat pro-quasi-coherent
pro-sheaves on $\Spi\widehat R_m$ is equivalent to the category of
flat $\widehat R_m$\+contramodules by
Example~\ref{countable-flat-pro-sheaves-flat-contramodules}(2).
 Using the result of~\cite[Lemma~7.1(b)]{Pproperf} providing
an equivalence between the abelian category of $\fR$\+contramodules
$\fR\contra$ and the Cartesian product of the abelian categories
of $\widehat R_m$\+contramodules $\widehat R_m\contra$, one can
conclude that the category of flat pro-quasi-coherent pro-sheaves
on $\X$ is equivalent to the category of flat $\fR$\+contramodules,
$\X\flat\simeq\fR\flat$.
 The construction of this equivalence is similar to the one in
Example~\ref{countable-flat-pro-sheaves-flat-contramodules}(1).
 In fact, by~\cite[Corollary~8.4 or
Theorem~10.1(vi)$\Rightarrow$(iii)]{Pproperf}, all flat
$\fR$\+contramodules are projective, $\fR\flat=\fR\contra_\proj$.

\smallskip
 (3)~The following description of the category $\X\flat$ may be more
instructive.
 An $R$\+module $C$ is said to be \emph{$S$\+reduced} if it has
no submodules in which all the elements of $S$ act by invertible
operators; equivalently in our context, this means that
$\Hom_R(Q,C)=0$ \,\cite[Theorem~13.8(a)]{Pcta}.
 An $R$\+module $C$ is said to be \emph{$S$\+weakly cotorsion} if
$\Ext_R^1(Q,C)=0$; equivalently in our context, this means that $C$
is \emph{cotorsion}, that is $\Ext_R^1(F,C)=0$ for all flat
$R$\+modules~$F$ \,\cite[Theorem~13.9(b)]{Pcta}.

 Reduced cotorsion abelian groups were called ``co-torsion''
in~\cite{Harr}, and $R$\+modules $C$ satisfying $\Hom_R(Q,C)=0=
\Ext^1_R(Q,C)$ (for a domain $R$ and the multiplicative subset
$S=R\setminus\{0\}$) were called ``cotorsion'' in~\cite{Mat}.
 $R$\+modules $C$ satisfying $\Hom_R(Q,C)=0=\Ext^1_R(Q,C)$ are called
``$S$\+contramodules'' in~\cite{PMat,BP}.

 We claim that the category of flat pro-quasi-coherent pro-sheaves
$\X\flat$ is naturally equivalent to the category of flat
$S$\+reduced cotorsion $R$\+modules.
 The equivalence assigns to a flat pro-quasi-cohererent pro-sheaf
$\fF\in\X\flat$ the $R$\+module $\varprojlim_{I\in\Gamma}
\fF^{(X_I)}(X_I)$.
 Conversely, to a flat $S$\+reduced cotorsion $R$\+module $F$,
the flat pro-quasi-coherent pro-sheaf $\fF$ with the components
$\fF^{(X_I)}\in X_I\flat$ corresponding to the flat $R/I$\+modules
$F/IF$ is assigned.

 Indeed, by the result of~\cite[Corollary~13.13(b)]{Pcta}
or~\cite[Corollary~6.17]{BP}, taken together
with~\cite[Theorem~B.1.1]{Pweak} and~\cite[Lemma~7.1(b)]{Pproperf},
the category $S$\+reduced cotorsion $R$\+modules is abelian and
equivalent to the category of $\fR$\+contramodules (the equivalence
being provided by the forgetful functor $\fR\contra\rarrow R\modl$).
 It remains to notice that an $\widehat R_m$\+contramodule is flat,
or equivalently, projective (as a contramodule) if and only if it is
a flat $R$\+module~\cite[Corollary~10.3(a) or Theorem~10.5]{Pcta},
\cite[Corollary~B.8.2]{Pweak}.

 The equivalence of abelian categories $\X\tors\simeq R\tors\subset
R\modl$ and the fully faithful functor $\X\flat\simeq\fR\flat\rarrow
R\modl$ identify the tensor product functor $\ot_\X\:\X\tors\times
\X\flat \rarrow\X\flat$ with the restriction of the tensor product
functor $\ot_R\:R\tors\times R\modl\rarrow R\tors$ to the full
subcategory of flat $S$\+reduced cotorsion $R$\+modules in the second
argument.

\smallskip
 (4)~Choosing a dualizing complex $\D^\bu$ on $\Spec R$, one can use
the construction of Example~\ref{ind-closed-immersion}(3) to
obtain a dualizing complex $\rD^\bu$ on~$\X$.
 This construction allows to produce a dualizing complex which would
be a two-term complex of injective quasi-coherent torsion sheaves
on~$\X$.

 Here is how one can construct a dualizing one-term complex on~$\X$.
 Let $E\in R\modl$ be the direct sum of injective envelopes of
the simple $R$\+modules $R/m$, where $m$~ranges over the maximal
ideals of~$R$ (one copy of each).
 Then $E$ is an injective $S$\+torsion $R$\+module, $E\in R\tors_\inj$
(notice that an $S$\+torsion $R$\+module is injective in $R\modl$ if
and only if it is injective in $R\tors$, since $R$ is Noetherian).
 Let $\rE\in\X\tors_\inj$ be an injective quasi-coherent torsion
sheaf corresponding to $E$ under the equivalence $R\tors\simeq
\X\tors$.
 Then $\rD^\bu=\rE$ is a one-term dualizing complex on~$\X$.

 Indeed, let $i_I\:X_I\rarrow\X$ denote the natural closed immersion.
 Then the quasi-coherent sheaf $i_I^!\rE$ on $X_I$ corresponds to
an injective $R/I$\+module which is the direct sum of the injective
envelopes of the simple modules over the Artinian ring~$R/I$.
 Since, by Matlis duality~\cite[Theorem~3.7]{Mat0}, the injective
envelope of the residue field of a local Artinian ring is a dualizing
complex, $i_I^!\rE$ is a dualizing complex on~$X_I$.

\smallskip
 (5)~Similarly to Example~\ref{cotensor-over-coalgebra}(5), the proof
of Theorem~\ref{ind-Noetherian-triangulated-equivalence-thm} specialized
to the case of a one-term dualizing complex $\rD^\bu=\rE$ shows
that there is an equivalence of additive (split exact) categories
of injective quasi-coherent torsion sheaves and flat pro-quasi-coherent
pro-sheaves on~$\X$,
\begin{equation} \label{krulldim1-injective-flat-correspondence}
 \fHom_{\X\qc}(\rE,{-})\:\X\tors_\inj\,\simeq\,
 \X\flat\,:\!\rE\ot_\X{-}.
\end{equation}
 In view of the above interpretation of the categories $\X\tors_\inj$
and $\X\flat$ as full subcategories in $R\modl$, this means
an equivalence between the additive categories of injective $S$\+torsion
$R$\+modules and flat $S$\+reduced cotorsion $R$\+modules, provided
by the mutually inverse functors $\Hom_R(E,{-})$ and $E\ot_R{-}$.

 Let us \emph{warn} the reader that this equivalence of additive
subcategories in $R\modl$ is different from the Matlis equivalence
between the full subcategories of $S$\+divisible $S$\+torsion
$R$\+modules and $S$\+torsionfree $S$\+reduced cotorsion
$R$\+modules~\cite[Theorem~3.4]{Mat}, \cite[Corollary~5.2]{PMat},
which is given by a different pair of adjoint functors.
 The equivalence~\eqref{krulldim1-injective-flat-correspondence}
is induced by a \emph{dualizing} torsion module/complex $E$,
while the Matlis equivalence is induced by a \emph{dedualizing}
compex $R\rarrow Q$ (see~\cite[Introduction and Remark~4.10]{Pmgm}
for a discussion of dualizing and dedualizing complexes).
 For a Dedekind domain $R$, the two equivalences are the same.
\end{exs}

 Let $R$ be a Dedekind domain.
 In the context of Examples~\ref{krulldim1-torsion-ind-scheme},
we have $S=R\setminus\{0\}$; so $Q=S^{-1}R$ is the field of
quotients of~$R$.
 An $R$\+module $M$ is said to be \emph{torsion} if for every $m\in M$
there exists $r\in R$, \,$r\ne0$, such that $rm=0$ in~$M$; as above,
we denote by $R\tors\subset R\modl$ the full subcategory of torsion
$R$\+modules.

 Notice that $Q\ot_R\Tor^R_1(M,N)\simeq\Tor^R_1(Q\ot_RM,\>N)=0$ for
all $R$\+modules $M$ and $N$, so the $R$\+module $\Tor^R_1(M,N)$ is
always torsion.
 Partly following the notation in~\cite{Kee}, we put
$M\ct_{Q/R}N=\Tor^R_1(M,N)$ for all torsion $R$\+modules $M$ and~$N$;
so
$$
 \ct_{Q/R}=\Tor^R_1({-},{-})\:R\tors\times R\tors\lrarrow R\tors.
$$
 The subindex $Q/R$ in our notation $\ct_{Q/R}$ is explained by
the observation that there are natural isomorphisms $Q/R\ct_{Q/R}M
\simeq M\simeq M\ct_{Q/R}Q/R$ for all torsion $R$\+modules~$M$.

 Notice that $\Tor^R_2(M,N)=0$ for all $R$\+modules $M$ and $N$
(since $R$ is a Dedekind domain).
 It follows that $\ct_{Q/R}$ is a left exact functor.
 Furthermore, associativity of the derived tensor product functor
$\ot_R^\boL\:\sD(R\modl)\times\sD(R\modl)\rarrow\sD(R\modl)$ implies
associativity of the functor~$\ct_{Q/R}$, as
$$
 L\ct_{Q/R}(M\ct_{Q/R}N)\simeq H^{-2}(L\ot_R^\boL M\ot_R^\boL N)
 \simeq (L\ct_{Q/R}M)\ct_{Q/R}N
$$
for all torsion $R$\+modules $L$, $M$, and~$N$.
 So the category $R\tors$ with the functor $\ct_{Q/R}$ is
an associative, commutative, and unital tensor category with
the unit object~$Q/R$.

\begin{ex} \label{dedekind=coalgebra-example}
 Let $\kk$~be a field and $R=\kk[x]_{(x)}$ be the localization of
the ring of polynomials $\kk[x]$ in one variable~$x$ at the prime
ideal $(x)=x\kk[x]\subset\kk[x]$.
 So the field of rational functions $Q=\kk(x)=R[x^{-1}]$ can be obtained
by inverting the single element $x\in R$.
 Let $\rC$ be the coalgebra over~$\kk$ whose dual topological algebra
is $\rC^*=\kk[[x]]$, as mentioned in Examples~\ref{coalgebra-ind-scheme}.
 Then the category of torsion $R$\+modules is naturally equivalent to
the category of $\rC$\+comodules, $R\tors\simeq\rC\comodl$.
 This equivalence of abelian categories identifies the tensor structure
$\ct_{Q/R}$ on $R\tors$ with the tensor structure $\oc_\rC$ on
$\rC\comodl$.
\end{ex}

 Let $R$ be a Dedekind domain.
 Then the homological dimension of the abelian category $R\modl$ is
equal to~$1$, hence the homological dimension of the abelian category
$R\tors$ is equal to~$1$ as well.
 One easily concludes that the unbounded derived category $\sD(R\tors)$
is equivalent to the homotopy category of complexes of injecitves
$\sK(R\tors_\inj)$.
 Similarly, the derived category $\sD(R\modl)$ is equivalent to
the homotopy category $\sK(R\modl_\inj)$.
 It follows that $\sD(R\tors)$ is a full subcategory in $\sD(R\modl)$
(see~\cite[Theorem~6.6(a)]{PMat} for a more general result).

 The torsion $R$\+module $Q/R$ is an injective cogenerator of
$R\tors$; moreover, a torsion $R$\+module is injective if and only if
it is a direct summand of a direct sum of copies of~$Q/R$.
 In fact, the injective $R$\+module $Q/R$ is a direct sum of injective
envelopes of the simple $R$\+modules $R/m$, where $m$~ranges over
the maximal ideals of $R$ (one copy of each); so in the context of
Example~\ref{krulldim1-torsion-ind-scheme}(4) one can take $E=Q/R$
for a Dedekind domain`$R$.
 For a given torsion $R$\+module $J$, the torsion product functor
$J\ct_{Q/R}{-}\,\:R\tors\rarrow R\tors$ is exact if and only if $J$
is injective.

 The functor of torsion product of complexes of torsion modules
$\ct_{Q/R}\:\sC(R\tors)\times\sC(R\tors)\rarrow\sC(R\tors)$ is 
constructed in the obvious way (using the totalization by taking
the direct sums along the diagonals of the bicomplex).
 To define the right derived functor of torsion product
$$
 \ct_{Q/R}^\boR\:\sD(R\tors)\times\sD(R\tors)\lrarrow\sD(R\tors),
$$
suppose that we are given two complexes of torsion $R$\+modules
$M^\bu$ and $N^\bu$.
 Let $K^\bu$ and $J^\bu$ be complexes of torsion $R$\+modules endowed
with quasi-isomorphisms of complexes of torsion $R$\+modules
$M^\bu\rarrow K^\bu$ and $N^\bu\rarrow J^\bu$.
 Then the derived torsion product is defined as the object
$$
 M^\bu\ct_{Q/R}^\boR N^\bu = K^\bu\ct_{Q/R}J^\bu\simeq
 M^\bu\ct_{Q/R}J^\bu \simeq K^\bu\ct_{Q/R}N^\bu \in\sD(R\tors),
$$
similarly to Example~\ref{cotensor-over-coalgebra}(4).

 The right derived torsion product functor agrees with the left derived
tensor product: restricting the derived tensor product functor
$\ot_R^\boL\:\sD(R\modl)\times\sD(R\modl)\rarrow\sD(R\modl)$ to the full
subcategory $\sD(R\tors)\times\sD(R\tors)\subset\sD(R\modl)\times
\sD(R\modl)$, one obtains the functor $\ct_{Q/R}^\boR\:
\sD(R\tors)\times\sD(R\tors)\rarrow\sD(R\tors)$.
 This comparison holds, essentially, because one has $M\ot_RJ=0$ for
any torsion $R$\+module $M$ and any injective torsion $R$\+module~$J$.

\begin{prop}
 Let $R$ be a Dedekind domain and\/ $\X=\ilim_{I\in\Gamma}X_I
=\Spi\varprojlim_{I\in\Gamma}R/I$ be the related ind-Artinian ind-scheme
from Examples~\ref{krulldim1-torsion-ind-scheme}.
 (Here $\Gamma$ is the poset of all nonzero ideals in $R$ in the reverse
inclusion order, and $X_I=\Spec R/I$.)
 Let $\rD^\bu=\rE$ be the one-term dualizing complex of\/ $\X$ 
corresponding to the injective torsion $R$\+module $E=Q/R$.
 Then the equivalence of (co)derived categories\/ $\sD(\X\tors)\simeq
\sD(R\tors)$ induced by the equivalence of abelian categories\/ $\X\tors
\simeq R\tors$ transforms the cotensor product functor\/ $\oc_{\rD^\bu}$
\,\eqref{coderived-torsion-cotensor-product} from
Section~\ref{cotensor-product-construction-subsecn} into the right
derived torsion product functor~$\ct_{Q/R}^\boR$. \hfuzz=1.3pt
\end{prop}

\begin{proof}
 One can argue similarly to Example~\ref{cotensor-over-coalgebra}(7),
reducing the question to the case of an Artinian scheme $R/I$, but
we prefer to spell out an argument in the spirit of
Example~\ref{cotensor-over-coalgebra}(8), working with special
classes of $R$\+modules and the ind-scheme $\X$ as a whole.
 In this context, the desired comparison is expressed by the composition
of the natural isomorphisms of torsion $R$\+modules
\begin{multline*}
 M\ct_{Q/R}J \simeq M\ct_{Q/R}(Q/R\ot_R\Hom_R(Q/R,J)) \\
 \simeq (M\ct_{Q/R}Q/R)\ot_R\Hom_R(Q/R,J)\simeq
 M\ot_R\Hom_R(Q/R,J),
\end{multline*}
which hold for any torsion $R$\+module $M$ and any injective torsion
$R$\+module~$J$.

 Here the natural isomorphism $J\simeq Q/R\ot_R\Hom_R(Q/R,J)$ for
a divisible torsion $R$\+module $J$ is due to
Harrison~\cite[Proposition~2.1]{Harr} and
Matlis~\cite[Theorem~3.4]{Mat}, while the middle (associativity)
isomorphism is provided by part~(b) of the next lemma.
\end{proof}

\begin{lem}
 Let $R$ be a Dedekind domain.
 Then, for any torsion $R$\+modules $M$ and $E$, and any
$R$\+module $P$, there is a natural homomorphism of torsion $R$\+modules
\begin{equation} \label{torsion-tensor-assoc}
 (M\ct_{Q/R}E)\ot_RP\lrarrow M\ct_{Q/R}(E\ot_RP),
\end{equation}
which is an isomorphism whenever either
\textup{(a)}~$M$~is injective, or \textup{(b)}~$P$~is flat.
\end{lem}

\begin{proof}
 This result is analogous to~\cite[Proposition~3.1.1]{Prev}
(cf.\ Example~\ref{dedekind=coalgebra-example}); it is also
a particular case of~\cite[Lemma~1.7.2(a)]{Pweak}.
 Denote the left-hand side of~\eqref{torsion-tensor-assoc} by
$l(M,P)$ and the right-hand side by $r(M,P)$.

 There is an obvious isomorphism~\eqref{torsion-tensor-assoc} when
$P=F$ is a free $R$\+module with a fixed set of free generators (since
the functor~$\ct_{Q/R}$ preserves direct sums).
 It is straightforward to check that this isomorphism is functorial
with respect to arbitrary morphisms of free $R$\+modules~$F$.

 Now let $P$ be the cokernel of a morphism of free $R$\+modules
$f\:F'\rarrow F''$.
 Then there is a natural isomorphism $\coker(l(M,f))\simeq
l(M,P)$ and a natural morphism $\coker(r(M,f))\rarrow r(M,P)$.
 As we already have a natural isomorphism of morphisms $l(M,f)
\simeq r(M,f)$, the desired morphism $l(M,P)\rarrow r(M,P)$
is obtained.
 Its functoriality is again straightforward.

 Now we can prove part~(a).
 Since $M$ is injective and therefore the functor $M\ct_{Q/R}{-}$
is exact, the natural morphism $\coker(r(M,f))\rarrow r(M,P)$ is
an isomorphism.
 Hence $l(M,P)\rarrow r(M,P)$ is an isomorphism.

 Let $M$ be the kernel of a morphism of injective torsion
$R$\+modules $g\:K'\rarrow K''$.
 Then there is a natural morphism $l(M,P)\rarrow\ker(l(g,P))$
and a natural isomorphism $r(M,P)\simeq\ker(r(g,P))$
(since the functor $\ct_{Q/R}$ is left exact).

 Now we can prove part~(b).
 Since $P$ is flat, the morphism $l(M,P)\rarrow\ker(l(g,P))$
is also an isomorphism.
 As we already know from part~(a) that the natural morphism of
morphisms $l(g,P)\rarrow r(g,P)$ is an isomorphism, it follows that
the morphism of torsion modules $l(M,P)\rarrow r(M,P)$
is an isomorphism.
\end{proof}

\begin{ex} \label{general-ind-Artinian-ind-scheme}
 Quite generally, let $\X$ be an ind-Artinian ind-scheme, and let
$\fR$ be the related pro-Artinian topological commutative ring such that
$\X=\Spi\fR$, as per Example~\ref{ind-Artinian-ind-schemes}(3).
 According to Section~\ref{torsion-ind-affine-subsecn}(6), the abelian
category $\X\tors$ is equivalent to the abelian category of discrete
$\fR$\+modules $\fR\discr$.

 The abelian category $\fR\discr$ does not have a natural injective
cogenerator which would correspond to a one-term dualizing complex
on~$\X$.
 Such an injective discrete $\fR$\+module, namely, the direct sum of
injective envelopes of all the simple objects in $\fR\discr$, does
exist, but it is only defined up to a nonunique isomorphism.

 In the memoir~\cite[Section~1]{Pweak}, the category of
\emph{$\fR$\+comodules} $\fR\comodl$ is defined in such a way that
it comes endowed with a natural injective cogenerator $\C(\fR)$,
similar to the injective cogenerator $\rC$ of the category of
comodules over a cocommutative coalgebra $\rC$, as in
Examples~\ref{cotensor-over-coalgebra}.
 Accordingly, there is a naturally defined cotensor product functor
$\oc_\fR\:\fR\comodl\times\fR\comodl\rarrow\fR\comodl$, making
$\fR\comodl$ a tansor category with the unit object~$\C(\fR)$.

 The choice of an injective object $\rC\in\fR\discr$ isomorphic to
a direct sum of injective cogenerators of simple discrete modules,
induces an equivalence of categories $\fR\discr\simeq\fR\comodl$
taking $\rC\in\fR\discr$ to $\C\in\fR\comodl$.
 So $\fR\discr$ becomes a tensor category with the unit object $\rC$;
the cotensor product operation $\oc_\fR=\oc_\rC$ defining this tensor
structure is described in~\cite[Section~1.9]{Pweak}.
 The injective object $\rC\in\fR\discr$ corresponds to a one-term
dualizing complex on~$\X$.

 Similarly to Example~\ref{cotensor-over-coalgebra}(6), the category
$\X\flat$ can be naturally identified with the category of flat,
or which is the same, projective contramodules over the topological
ring $\fR$, that is $\X\flat\simeq\fR\flat=\fR\contra_\proj$
(see~\cite[Section~2]{Pproperf} for the definition of a flat
$\fR$\+contramodule and~\cite[Lemma~1.9.1(a)]{Pweak}
or~\cite[Corollary~8.4 or Theorem~10.1(vi)$\Rightarrow$(iii)]{Pproperf}
for a proof that all flat contramodules are projective over
a pro-Artinian commutative topological ring~$\fR$).
 In particular, the exact category structure on $\X\flat$ is split.

 Having chosen a one-term dualizing complex $\rC\in\X\tors_\inj$, one
obtains an equivalence of additive categories
as in Example~\ref{cotensor-over-coalgebra}(5),
$$
 \fHom_{\X\qc}(\rC,{-})\:\X\tors_\inj\,\simeq\,
 \X\flat\,:\!\rC\ot_\X{-},
$$
which corresponds to the equivalence of additive categories
$\fR\comodl_\inj\simeq\fR\contra_\proj$
\,\cite[Proposition~1.5.1]{Pweak} under the identifications
$\X\tors\simeq\fR\comodl$ and $\X\flat\simeq\fR\contra_\proj$.

 Similarly to Example~\ref{cotensor-over-coalgebra}(4), one constructs
the right derived functor of cotensor product of discrete
$\fR$\+modules
$$
 \oc_\rC^\boR\:\sD^\co(\fR\discr)\times\sD^\co(\fR\discr)\lrarrow
 \sD^\co(\fR\discr).
$$
 The cotensor product functor $\oc_{\rD^\bu}\:\sD^\co(\X\tors)\times
\sD^\co(\X\tors)\rarrow\sD^\co(\X\tors)$
\,\eqref{coderived-torsion-cotensor-product}
from Section~\ref{cotensor-product-construction-subsecn}
for $\rD^\bu=\rC$ is transformed into the functor $\oc_\rC^\boR$
by the equivalence of coderived categories $\sD^\co(\X\tors)\simeq
\sD^\co(\fR\discr)$ induced by the equivalence of abelian
categories $\X\tors\simeq\fR\discr$.
\end{ex}

\begin{ex}
 For an ind-affine ind-Noetherian $\aleph_0$\+ind-scheme with
a dualizing complex~$\rD^\bu$, the construciton of the cotensor
product functor~$\oc_{\rD^\bu}$ in
Section~\ref{cotensor-product-construction-subsecn} agrees
with the one in~\cite[Section~D.3]{Pcosh}, as one can see by
comparing the two constructions in light of the discussion in
Examples~\ref{countable-flat-pro-sheaves-flat-contramodules}.
\end{ex}

\Section{Ind-Schemes of Ind-Finite Type and the $!$-Tensor Product}
\label{ind-finite-type-secn}

 Throughout this section, $\kk$~denotes a fixed ground field.
 Given two ind-schemes $\X'$ and $\X''$ (or two schemes $X'$ and $X''$)
over~$\kk$, we denote the fibered product $\X'\times_{\Spec\kk}\X''$
(or $X'\times_{\Spec\kk}X''$) simply by $\X'\times_\kk\X''$
(or $X'\times_\kk X''$) for brevity.
 (See Sections~\ref{ind-objects-subsecn}\+-\ref{ind-schemes-subsecn}
for a discussion of fibered products of ind-schemes.)
 
 Let $\X$ be an ind-separated ind-scheme of ind-finite type
over the field~$\kk$.
 The aim of this section is to describe the cotensor product functor
$\oc_{\rD^\bu}\:\sD^\co(\X\tors)\times\sD^\co(\X\tors)\rarrow
\sD^\co(\X\tors)$, for a suitable choice of the dualizing complex
$\rD^\bu$ on $\X$, as the derived $!$\+restriction to the diagonal
$\Delta_\X\:\X\rarrow\X\times_\kk\X$ of the external tensor product
on $\X\times_\kk\X$ of the two given complexes of quasi-coherent
torsion sheaves on~$\X$.

\subsection{External tensor product of quasi-coherent sheaves}
 Let $X'$ and $X''$ be two schemes over~$\kk$.
 Consider the Cartesian product $X'\times_\kk X''$, and let
$p'\:X'\times_\kk X''\rarrow X'$ and $p''\:X'\times_\kk X''\rarrow X''$
denote the natural projections.

 Let $\M'$ be a quasi-coherent sheaf over $X'$ and $\M'$ be
a quasi-coherent sheaf over~$X''$.
 Then the \emph{external tensor product} $\M'\bt_\kk\M''$ of
the quasi-coherent sheaves $\M'$ and $\M'' $ is a quasi-coherent
sheaf on $X'\times_\kk X''$ defined by the formula
$$
 \M'\bt_\kk\M''=p'{}^*\M'\ot_{\cO_{X'\times_\kk X''}}p''{}^*\M''.
$$

\begin{lem} \label{qcoh-flat-external-product}
 Let $\F'$ be a flat quasi-coherent sheaf over $X'$ and $\F''$ be
a flat quasi-coherent sheaf over~$X''$.
 Then $\F'\bt_\kk\F''$ is a flat quasi-coherent sheaf over
$X'\times_\kk X''$.
\end{lem}

\begin{proof}
 Follows from immediately from the definition of the external tensor
product $\F'\bt_\kk\F''$ and the facts that the inverse images
$p'{}^*$, $p''{}^*$ and the tensor product
$\ot_{\cO_{X'\times_\kk X''}}$ preserve flatness
of quasi-coherent sheaves.
\end{proof}

\begin{lem} \label{qcoh-external-tensor-exact}
 The external tensor product functor
$$
 \bt_\kk\:X'\qcoh\times X''\qcoh\lrarrow (X'\times_\kk X'')\qcoh
$$
is exact and preserves coproducts (hence all colimits) in each of its
arguments.
\end{lem}

\begin{proof}
 The assertion is local in both $X'$ and $X''$, so it reduces to
the case of affine schemes, for which it means the following.
 Let $R'$ and $R''$ be two commutative $\kk$\+algebras.
 Let $M'$ be an $R'$\+module and $M''$ be an $R''$\+module.
 Then the functor assigning to $M'$ and $M''$
the $(R'\ot_\kk R'')$\+module
\begin{multline*}
 ((R'\ot_\kk R'')\ot_{R'}M')\ot_{R'\ot_\kk R''}
 ((R'\ot_\kk R'')\ot_{R''}M'') \\ \simeq
 (R''\ot_\kk M')\ot_{R'\ot_\kk R''}(R'\ot_\kk M'')
 \simeq M'\ot_\kk M''
\end{multline*}
is exact in each of the arguments.
 The preservation of coproducts is obvious.
\end{proof}

\begin{lem} \label{inverse-image-of-external-product}
 Let $f'\:Y'\rarrow X'$ and $f''\:Y''\rarrow X''$ be two morphisms
of schemes over\/~$\kk$, and let $f=f'\times_\kk f''\:
Y'\times_\kk Y''\rarrow X'\times_\kk X''$ be the induced morphism of
the Cartesian products.
 Let $\M'$ be a quasi-coherent sheaf on $X'$ and $\M''$ be
a quasi-coherent sheaf on~$X''$.
 Then there is a natural isomorphism
$$
 f^*(\M'\bt_\kk\M'')\simeq f'{}^*\M'\bt_\kk f''{}^*\M''
$$
of quasi-coherent sheaves on $Y'\times_\kk Y''$.
\end{lem}

\begin{proof}
 Let $p^{(s)}\:X'\times_\kk X''\rarrow X^{(s)}$ and
$q^{(s)}\:Y'\times_\kk Y''\rarrow\Spec Y^{(s)}$, \ $s=1$, $2$,
be the natural morphisms.
 Then one has
\begin{multline*}
 f^*(\M'\bt_\kk\M'')=
 f^*(p'{}^*\M'\ot_{\cO_{X'\times_\kk X''}}p''{}^*\M'')
 \simeq f^*p'{}^*\M'\ot_{\cO_{Y'\times_\kk Y''}} f^*p''{}^*\M'' \\
 \simeq q'{}^*f'{}^*\M'\ot_{\cO_{Y'\times_\kk Y''}}q''{}^*f''{}^*\M''
 =f'{}^*\M'\bt_\kk f''{}^*\M'',
\end{multline*}
since $p'f=f'q'$ and $p''f=f''q''$.
\end{proof}

\begin{lem} \label{direct-image-of-external-product}
 Let $f'\:Y'\rarrow X'$ and $f''\:Y''\rarrow X''$ be two affine
morphisms of schemes over\/~$\kk$, and let $f=f'\times_\kk f''\:
Y'\times_\kk Y''\rarrow X'\times_\kk X''$ be the induced morphism of
the Cartesian products.
 Let $\N'$ be a quasi-coherent sheaf on $Y'$ and $\N''$ be
a quasi-coherent sheaf on~$Y''$.
 Then there is a natural isomorphism
$$
 f_*(\N'\bt_\kk\N'')\simeq f'_*\N'\bt_\kk f''_*\N''
$$
of quasi-coherent sheaves on $X'\times_\kk X''$.
\end{lem}

\begin{proof}
 The assertion is essentially local in $X'$ and $X''$, so it
reduces to the case of affine schemes, for which it means the following
very tautological observation (cf.\ the computation in the proof of
Lemma~\ref{qcoh-external-tensor-exact}).
 Let $R'\rarrow S'$ and $R''\rarrow S''$ be two homomorphisms of
commutative $\kk$\+algebras.
 Let $N'$ be an $S'$\+module and $N''$ be an $S''$\+module; then
$N'\ot_\kk N''$ is an $(S'\ot_\kk S'')$\+module.
 Consider the underlying $R'$\+module of $N'$ and the underlying
$R''$\+module of~$N''$; then the tensor product $N'\ot_\kk N''$
acquires the structure of an $(R'\ot_\kk R'')$\+module.
 The claim is that the latter $(R'\ot_\kk R'')$\+module structure
$N'\ot_\kk N''$ underlies the former $(S'\ot_\kk S'')$\+module structure
with respect to the ring homomorphism $R'\ot_\kk R''\rarrow
S'\ot_\kk S''$ (i.~e., the two $(R'\ot_\kk R'')$\+module structures
on $N'\ot_\kk N''$ coincide).
\end{proof}

\begin{lem} \label{product-of-reasonable}
 Let $Z'$ be a reasonable closed subscheme in a scheme $X'$ over\/~$\kk$
and $Z'$ be a reasonable closed subscheme in a scheme~$Z''$
over\/~$\kk$.
 Then $Z'\times_\kk Z''$ is a reasonable closed subscheme in
the scheme $X'\times_\kk X''$.
\end{lem}

\begin{proof}
 To deduce the assertion from
Lemma~\ref{base-change-composition-reasonable}, decompose the closed
immersion $Z'\times_\kk Z''\rarrow X'\times_\kk X''$ as
$Z'\times_\kk Z''\rarrow Z'\times_\kk X''\rarrow X'\times_\kk X''$
and notice that $Z'\times_\kk Z''=(Z'\times_\kk X'')\times_{X''}Z''$.
\end{proof}

\begin{lem} \label{shriek-of-external-product}
 Let $Z'$ be a reasonable closed subscheme in a scheme $X'$ over\/~$\kk$
and $Z'$ be a reasonable closed subscheme in a scheme~$Z''$
over\/~$\kk$.
 Let $i'\:Z'\rarrow X'$ and $i''\:Z''\rarrow X''$ be the closed
immersion morphisms, and let $i=i'\times_\kk i''\: Z'\times_\kk Z''
\rarrow X'\times_\kk X''$ be the induced closed immersion of
the Cartesian products.
 Let $\M'$ be a quasi-coherent sheaf on $X'$ and $\M''$ be
a quasi-coherent sheaf on~$X''$.
 Then there is a natural isomorphism
$$
 i^!(\M'\bt_\kk\M'')\simeq i'{}^!\M'\bt_\kk i''{}^!\M''
$$
of quasi-coherent sheaves on $Z'\times_\kk Z''$.
\end{lem}

\begin{proof}
 The assertion is essentially local in $X'$ and $X''$, so it reduces
to affine schemes, for which it means the following
(cf.\ the computation in the proof of
Lemma~\ref{qcoh-external-tensor-exact}).
 Let $R'\rarrow S'$ and $R''\rarrow S''$ be two surjective homomorphisms
of commutative $\kk$\+algebras (with finitely generated kernel ideals).
 Let $M'$ be an $R'$\+module and $M''$ be an $R''$\+module.
 Then there is a natural isomorphism of $(S'\ot_\kk S'')$\+modules
$$
 \Hom_{R'\ot_\kk R''}(S'\ot_\kk S'',\>M'\ot_\kk M'')\simeq
 \Hom_{R'}(S',M')\ot_\kk\Hom_{R''}(S'',M'').
$$
 To obtain the latter isomorphism, one can notice firstly that
$\Hom_{R'\ot_\kk R''}(S'\ot_\kk\nobreak R'',\>\allowbreak M'\ot_\kk M'')
\simeq\Hom_{R'}(S',\>M'\ot_\kk M'')\simeq\Hom_{R'}(S',M')\ot_\kk M''$
and similarly $\Hom_{R'\ot_\kk R''}(R'\ot_\kk S'',\>M'\ot_\kk M'')\simeq
M'\ot_\kk\Hom_{R''}(S'',M'')$.
\end{proof}

 Let $\M'{}^\bu$ be a complex of quasi-coherent sheaves on $X'$ and
$\M''{}^\bu$ be a complex of quasi-coherent sheaves on~$X''$.
 Then the complex $\M'{}^\bu\bt_\kk\M''{}^\bu$ of quasi-coherent
sheaves on $X'\times_\kk X''$ is constructed by totalizing the bicomplex
of external tensor products using infinite coproducts along
the diagonals.

\begin{lem} \label{qcoh-complexes-external-tensor-exactness}
\textup{(a)} Let $\M'{}^\bu$ be a complex of quasi-coherent sheaves
on a scheme $X'$ and $\M''{}^\bu$ be a complex of quasi-coherent
sheaves on a scheme $X''$ over\/~$\kk$.
 Assume that the complex $\M'{}^\bu$ is acyclic (in $X'\qcoh$).
 Then the complex $\M'{}^\bu\bt_\kk\M''{}^\bu$ is acyclic (in
$(X'\times_\kk X'')\qcoh$). \par
\textup{(b)} Let $\F'{}^\bu$ be a complex of flat quasi-coherent
sheaves on a scheme $X'$ and $\F''{}^\bu$ be a complex of flat
quasi-coherent sheaves on a scheme $X''$ over\/~$\kk$.
 Assume that the complex $\F'{}^\bu$ is acyclic in $X'\flat$.
 Then the complex $\F'{}^\bu\bt_\kk\F''{}^\bu$ is acyclic in
$(X'\times_\kk X'')\flat$.
\end{lem}

\begin{proof}
 Both the assertions~(a) and~(b) are local in $X'$ and $X''$, so they
reduce to the case of affine schemes.
 Then part~(a) follows from the fact that the tensor product of
an acyclic complex of $\kk$\+vector spaces with an arbitrary complex
of $\kk$\+vector spaces is an acyclic complex.
 Part~(b) can be proved directly using the definition of the external
tensor product of quasi-coherent sheaves as the tensor product of
inverse images with respect to the projection maps: the inverse image
under any morphism of schemes is exact as a functor between the exact
categories of flat quasi-coherent sheaves, and it remains to refer to
Lemma~\ref{qcoh-complexes-tensor-exactness}(b).
\end{proof}

 Let $X$ be a scheme over~$\kk$, and let $\M'$ and $\M''$ be two
quasi-coherent sheaves on~$X$.
 Then $\M'\bt_\kk\M''$ is a quasi-coherent sheaf on $X\times_\kk X$.
 Denote by $\Delta_X\:X\rarrow X\times_\kk X$ the diagonal morphism
(defined by the property that its compositions with both the projections
$X\times_\kk X\rightrightarrows X$ are the identity morphisms).

\begin{lem} \label{qcoh-tensor-as-restricted-external}
 For any two quasi-coherent sheaves $\M'$ and $\M''$ on a scheme $X$
over\/~$\kk$, there is a natural isomorphism
$$
 \M'\ot_{\cO_X}\M''\,\simeq\,\Delta_X^*(\M'\bt_\kk\M'')
$$
of quasi-coherent sheaves on~$X$.
\end{lem}

\begin{proof}
 The assertion is essentially local in $X$, so it reduces to the case
of an affine scheme, for which it means the following.
 Let $R$ be a commutative $\kk$\+algebra, and let $M'$ and $M''$ be
two $R$\+modules.
 Then there is a natural isomorphism of $R$\+modules
$$
 M'\ot_RM''\,\simeq\,R\ot_{R\ot_\kk R}(M'\ot_\kk M''),
$$
where the diagonal ring homomorphism $R\ot_\kk R\rarrow R$ endows $R$
with an $(R\ot_\kk\nobreak R)$\+module structure.

 Alternatively, denoting by $p'$ and $p''\:X\times_\kk X\rarrow X$
the natural projections, one computes
\begin{multline*}
 \Delta_X^*(\M'\bt_\kk\M'') =
 \Delta_X^*(p'{}^*\M'\ot_{\cO_{X\times_\kk X}}p''{}^*\M'') \\
 \simeq\Delta_X^*p'{}^*\M'\ot_{\cO_X}\Delta_X^*p''{}^*\M''
 \simeq(p'\Delta_X)^*\M'\ot_{\cO_X}(p''\Delta_X)^*\M''
 \simeq\M'\ot_{\cO_X}\M''
\end{multline*}
using the fact that the tensor products of quasi-coherent sheaves
are preserved by the inverse images.
\end{proof}

\subsection{External tensor product of pro-sheaves}
\label{external-of-pro-subsecn}
 Let $\X'$ and $\X''$ be ind-schemes over~$\kk$.
 Let $\fP'$ be a pro-quasi-coherent pro-sheaf on $\X'$ and $\fP''$ be
a pro-quasi-coherent pro-sheaf on~$\X''$.
 For every pair of closed subschemes $Z'\subset\X'$ and $Z''\subset\X''$
put $\fQ^{(Z'\times_\kk Z'')}=\fP'{}^{(Z')}\bt_\kk\fP''{}^{(Z'')}
\in(Z'\times_\kk Z'')\qcoh$.
 Then it follows from Lemma~\ref{inverse-image-of-external-product}
that the collection of quasi-coherent sheaves
$\fQ^{(Z'\times_\kk Z'')}$ on the closed subschemes $Z'\times_\kk Z''
\subset\X'\times_\kk\X''$ defines a pro-quasi-coherent pro-sheaf $\fQ$
on the ind-scheme $\X'\times_\kk\X''$.

 Put $\fP'\bt_\kk\fP''=\fQ$.
 This construction defines the functor of \emph{external tensor product
of pro-quasi-coherent pro-sheaves}
\begin{equation} \label{external-product-pro-sheaves-eqn}
 \bt_\kk\:\X'\pro\times\X''\pro\lrarrow
 (\X'\times_\kk\X'')\pro.
\end{equation}
 It is clear from Lemma~\ref{qcoh-flat-external-product} that
the external tensor product of two flat pro-quasi-coherent pro-sheaves
is a flat pro-quasi-coherent pro-sheaf,
\begin{equation} \label{external-product-flat-pro-sheaves-eqn}
 \bt_\kk\:\X'\flat\times\X''\flat\lrarrow
 (\X'\times_\kk\X'')\flat.
\end{equation}
 Furthermore, it follows from
Lemma~\ref{qcoh-external-tensor-exact} that
the functor~\eqref{external-product-pro-sheaves-eqn}
preserves colimits in each of its argument, while
the functor~\eqref{external-product-flat-pro-sheaves-eqn}
is exact (as a functor between exact categories) and preserves
direct limits in each of its arguments.

 Let $\fP'{}^\bu$ be a complex of pro-quasi-coherent pro-sheaves on
$\X'$ and $\fP''{}^\bu$ be a complex of pro-quasi-coherent pro-sheaves
on~$\X''$.
 Then the complex $\fP'{}^\bu\bt_\kk\fP''{}^\bu$ of pro-quasi-coherent
pro-sheaves on $\X'\times_\kk\X''$ is constructed by totalizing
the bicomplex of external tensor products using infinite coproducts
along the diagonals.

\begin{lem} \label{flat-pro-complexes-external-tensor-exactness}
 Let\/ $\fF'{}^\bu$ be a complex of flat pro-quasi-coherent pro-sheaves
on an ind-scheme\/ $\X'$ and\/ $\fF''{}^\bu$ be a complex of flat
pro-quasi-coherent pro-sheaves on an ind-scheme\/ $\X''$ over\/~$\kk$.
 Assume that the complex\/ $\fF'{}^\bu$ is acyclic in\/ $\X'\flat$.
 Then the complex\/ $\fF'{}^\bu\bt_\kk\fF''{}^\bu$ is acyclic in
$(\X'\times_\kk\X'')\flat$.
\end{lem}

\begin{proof}
 The result of Lemma~\ref{flat-pro-sheaves-complex-acyclicity-criterion}
reduces the question to the case of schemes (rather than ind-schemes),
for which we have
Lemma~\ref{qcoh-complexes-external-tensor-exactness}(b).
\end{proof}

 It follows from
Lemma~\ref{flat-pro-complexes-external-tensor-exactness} that
the external tensor product is well-defined as a functor between
the derived categories of flat pro-quasi-coherent pro-sheaves,
\begin{equation} \label{derived-flat-external-product}
 \bt_\kk\:\sD(\X'\flat)\times\sD(\X''\flat)\lrarrow
 \sD((\X'\times_\kk\X'')\flat).
\end{equation}

 The next three lemmas do not depend on any flatness conditions.

\begin{lem} \label{pro-inverse-image-of-external-product}
 Let $f'\:\Y'\rarrow\X'$ and $f''\:\Y''\rarrow\X''$ be two morphisms
of ind-schemes over\/~$\kk$, and let
$f=f'\times_\kk f''\:\Y'\times_\kk\Y''\rarrow\X'\times_\kk\X''$ be
the induced morphism of the Cartesian products.
 Let\/ $\fP'$ be a pro-quasi-coherent pro-sheaf on\/ $\X'$ and\/
$\fP''$ be a pro-quasi-coherent pro-sheaf on\/~$\X''$.
 Then there is a natural isomorphism
$$
 f^*(\fP'\bt_\kk\fP'')\simeq f'{}^*\fP'\bt_\kk f''{}^*\fP''
$$
of pro-quasi-coherent pro-sheaves on\/ $\Y'\times_\kk\Y''$.
\end{lem}

\begin{proof}
 Follows immediately from the definition of the external tensor product
of pro-quasi-coherent pro-sheaves (above),
the definition of the inverse image of pro-quasi-coherent pro-sheaves
(see Section~\ref{pro-sheaves-inverse-direct-subsecn}),
and Lemma~\ref{inverse-image-of-external-product}.
\end{proof}

\begin{lem} \label{pro-direct-image-of-external-product}
 Let $f'\:\Y'\rarrow\X'$ and $f''\:\Y''\rarrow\X''$ be two affine
morphisms of ind-schemes over\/~$\kk$, and let
$f=f'\times_\kk f''\:\Y'\times_\kk\Y''\rarrow\X'\times_\kk\X''$ be
the induced morphism of the Cartesian products.
 Let\/ $\fQ'$ be a pro-quasi-coherent pro-sheaf on\/ $\Y'$ and\/
$\fQ''$ be a pro-quasi-coherent pro-sheaf on\/~$\Y''$.
 Then there is a natural isomorphism
$$
 f_*(\fQ'\bt_\kk\fQ'')\simeq f'_*\fQ'\bt_\kk f''_*\fQ''
$$
of pro-quasi-coherent pro-sheaves on\/ $\X'\times_\kk\X''$.
\end{lem}

\begin{proof}
 Follows from the definition of the external tensor product
of pro-quasi-coherent pro-sheaves, the definition of the direct
image of pro-quasi-coherent pro-sheaves
(see Section~\ref{pro-sheaves-inverse-direct-subsecn}),
and Lemma~\ref{direct-image-of-external-product}.
\end{proof}

 Let $\X$ be an ind-scheme over~$\kk$, and let $\fP'$ and $\fP''$ be
two pro-quasi-coherent pro-sheaves on~$\X$.
 Then $\fP'\bt_\kk\fP''$ is a pro-quasi-coherent pro-sheaf
on $\X\times_\kk\X$.
 Let $\Delta_\X\:\X\rarrow\X\times_\kk\X$ denote the diagonal morphism
of ind-schemes (defined by the property that its compositions with
both the projections $\X\times_\kk\X\rightrightarrows\X$ are
the identity morphisms).

\begin{lem} \label{pro-tensor-as-restricted-external}
 For any two pro-quasi-coherent pro-sheaves\/ $\fP'$ and\/ $\fP''$ on
an ind-scheme\/ $\X$ over\/~$\kk$, there is a natural isomorphism
$$
 \fP'\ot^\X\fP''\,\simeq\,\Delta_\X^*(\fP'\bt_\kk\fP'')
$$
of pro-quasi-coherent pro-sheaves on\/~$\X$.
\end{lem}

\begin{proof}
 Follows from Lemma~\ref{qcoh-tensor-as-restricted-external}
and the definitions of the functors~$\ot^\X$ (see
Section~\ref{pro-sheaves-subsecn}), $\Delta_\X^*$ (see
Section~\ref{pro-sheaves-inverse-direct-subsecn}), and~$\bt_\kk$.
\end{proof}

\subsection{External tensor product of torsion sheaves}
\label{external-of-torsion-subsecn}
 Let $\X'=\ilim_{\gamma'\in\Gamma'}X'_{\gamma'}$ and
$\X''=\ilim_{\gamma''\in\Gamma''}X''_{\gamma''}$ be two reasonable
ind-schemes over~$\kk$, represented by inductive systems of closed
immersions of reasonable closed subschemes.
 Then $\X'\times_\kk\X''=
\ilim_{(\gamma',\gamma'')\in\Gamma'\times\Gamma''}
X'_{\gamma'}\times_\kk X''_{\gamma''}$ is a representation of
the ind-scheme $\X'\times_\kk\X''$ by an inductive system of closed
immersions of reasonable closed subschemes
(by Lemma~\ref{product-of-reasonable}).

 Let $\rM'$ be a quasi-coherent torsion sheaf on $\X'$ and $\rM''$ be
a quasi-coherent torsion sheaf on~$\X''$.
 For every pair of reasonable closed subschemes $Z'\subset\X'$ and
$Z''\subset\X''$ put $\rL_{(Z'\times_\kk Z'')}=\rM'_{(Z')}\bt_\kk
\rM''_{(Z'')}\in (Z'\times_\kk Z'')\qcoh$.
 Then it is clear from Lemma~\ref{shriek-of-external-product} that
the collection of quasi-coherent sheaves $\rL_{(Z'\times_\kk Z'')}$
on the reasonable closed subschemes $Z'\times_\kk Z''\subset
\X'\times_\kk\X''$ defines a quasi-coherent torsion sheaf $\rL$ on
$\X'\times_\kk\X''$.

 Put $\rM'\bt_\kk\rM''=\rL$.
 This construction defines the functor of \emph{external tensor product
of quasi-coherent torsion sheaves}
\begin{equation} \label{external-product-torsion-sheaves-eqn}
 \bt_\kk\:\X'\tors\times\X''\tors\lrarrow(\X'\times_\kk\X'')\tors.
\end{equation}

\begin{lem} \label{torsion-direct-image-of-external-product}
 Let $f'\:\Y'\rarrow\X'$ and $f''\:\Y''\rarrow\X''$ be two affine
morphisms of reasonable ind-schemes over\/~$\kk$, and let
$f=f'\times_\kk f''\:\Y'\times_\kk\Y''\rarrow\X'\times_\kk\X''$ be
the induced morphism of the Cartesian products.
 Let $\rN'$ be a quasi-coherent torsion sheaf on\/ $\Y'$ and $\rN''$
be a quasi-coherent torsion sheaf on\/~$\Y''$.
 Then there is a natural isomorphism
$$
 f_*(\rN'\bt_\kk\rN'')\simeq f'_*\rN'\bt_\kk f''_*\rN''
$$
of quasi-coherent torsion sheaves on\/ $\X'\times_\kk\X''$.
\end{lem}

\begin{proof}
 Follows immediately from the definition of the external tensor product
of quasi-coherent torsion sheaves (above),
the definition of the direct image of quasi-coherent torsion sheaves
(see Section~\ref{torsion-direct-images-subsecn}),
and Lemma~\ref{direct-image-of-external-product}.
\end{proof}

 Similarly to the construction above, one defines the functor of
external tensor product of $\Gamma$\+systems
\begin{equation} \label{external-product-Gamma-systems-eqn}
 \bt_\kk\:(\X',\Gamma')\syst\times(\X',\Gamma'')\syst\lrarrow
 (\X'\times_\kk\X'',\>\Gamma'\times\Gamma'')\syst
\end{equation}
by setting $(\boM'\bt_\kk\boM'')_{(\gamma'\times\gamma'')}
=\boM'_{(\gamma')}\bt_\kk\boM''_{(\gamma'')}$ for any
$\Gamma'$\+system $\boM'$ on $\X'$, any $\Gamma''$\+system $\boM''$
on $\X''$, and any two indices $\gamma'\in\Gamma'$,
\,$\gamma''\in\Gamma''$.

\begin{lem} \label{torsion-and-Gamma-systems-external-product}
 Let\/ $\X'$ and\/ $\X''$ be reasonable ind-schemes over\/~$\kk$.
 Then \par
\textup{(a)} the functor of external tensor product\/ $\bt_\kk\:
\X'\tors\times\X''\tors\rarrow(\X'\times_\kk\X'')\tors$ preserves direct
limits (and in particular, coproducts) in each of the arguments; \par
\textup{(b)} the functors of external tensor product of\/
$\Gamma$\+systems~\eqref{external-product-Gamma-systems-eqn} and of
quasi-coherent torsion
sheaves~\eqref{external-product-torsion-sheaves-eqn} form a commutative
square diagram with the functors\/ $\boM^{(s)}\longmapsto\boM^{(s)}{}^+
\:(\X^{(s)},\Gamma^{(s)})\syst\rarrow\X^{(s)}\tors$, \ $s=1$,~$2$,
and\/ $\boL\longmapsto\boL^+\:
(\X'\times_\kk\X'',\>\allowbreak\Gamma'\times\nobreak\Gamma'')\syst
\rarrow(\X'\times_\kk\X'')\tors$.
\end{lem}

\begin{proof}
 Part~(a) follows from the description of direct limits of
quasi-coherent torsion sheaves in
Section~\ref{torsion-direct-limits-subsecn} together with
Lemma~\ref{qcoh-external-tensor-exact}.
 Part~(b) holds because the functors~$({-})^+$ are constructed
in terms of direct images (with respect to closed immersions
of reasonable closed subschemes into ind-schemes) and direct limits
of quasi-coherent torsion sheaves (see
Section~\ref{Gamma-systems-subsecn}).
 The external tensor products commute with the direct images by
Lemma~\ref{torsion-direct-image-of-external-product} and preserve
direct limits by part~(a).
\end{proof}

\begin{lem} \label{torsion-external-tensor-exact}
 For for any reasonable ind-schemes\/ $\X'$ and\/ $\X''$ over\/~$\kk$,
the functor of external tensor product\/ $\bt_\kk\:
\X'\tors\times\X''\tors\rarrow(\X'\times_\kk\X'')\tors$ is exact in
each of its arguments.
\end{lem}

\begin{proof}
 Exactness of the functor of external tensor product of
$\Gamma$\+systems~\eqref{external-product-Gamma-systems-eqn} follows
immediately from Lemma~\ref{qcoh-external-tensor-exact}.
 To deduce exactness of the functor of external tensor product of
quasi-coherent torsion
sheaves~\eqref{external-product-torsion-sheaves-eqn}, it remains to
recall that the functors~$({-})^+$ represent the abelian categories
of quasi-coherent torsion sheaves as quotient categories of
the abelian categories of $\Gamma$\+systems by some Serre subcategories
(see the proof of Proposition~\ref{Giraud-subcategory-abstract-prop})
and use Lemma~\ref{torsion-and-Gamma-systems-external-product}(b).
\end{proof}

\begin{lem} \label{torsion-inverse-image-of-external-product}
 Let\/ $\X'$ and\/ $\X''$ be reasonable ind-schemes over\/~$\kk$, and
let $f'\:\Y'\rarrow\X'$ and $f''\:\Y''\rarrow\X''$ be morphisms of
ind-schemes which are ``representable by schemes''.
 Let $f=f'\times_\kk f''\:\Y'\times_\kk\Y''\rarrow\X'\times_\kk\X''$
be the induced morphism of the Cartesian products.
 Let $\rM'$ be a quasi-coherent torsion sheaf on\/ $\X'$ and $\rM''$
be a quasi-coherent torsion sheaf on\/~$\X''$.
 Then there is a natural isomorphism
$$
 f^*(\rM'\bt_\kk\rM'')\simeq f'{}^*\rM'\bt_\kk f''{}^*\rM''
$$
of quasi-coherent torsion sheaves on\/ $\Y'\times_\kk\Y''$.
\end{lem}

\begin{proof}
 Follows from the definition of the inverse image of quasi-coherent
torsion sheaves (see Section~\ref{torsion-inverse-images-subsecn})
and Lemmas~\ref{direct-image-of-external-product}
and~\ref{torsion-and-Gamma-systems-external-product}(b).
\end{proof}

 Let us say that a closed immersion of ind-schemes $i\:\Z\rarrow\X$ is
\emph{reasonable} if, for any scheme $T$ and a morphism of schemes
$T\rarrow\X$, the morphism of schemes $\Z\times_\X T\rarrow T$ is
the closed immersion of a reasonable closed subscheme in~$T$.
 If $\X=\ilim_{\gamma\in\Gamma}X_\gamma$ is a representation of $\X$
by an inductive system of closed immersions of schemes, then
a closed immersion of ind-schemes $i\:\Z\rarrow\X$ is reasonable if
and only if, for every $\gamma\in\Gamma$, the fibered product
$\Z\times_\X X_\gamma$ is a reasonable closed subscheme in $X_\gamma$
(use Lemma~\ref{base-change-composition-reasonable}(a)).
 It is clear from Lemma~\ref{product-of-reasonable} that the Cartesian
product of reasonable closed immersions of ind-schemes over~$\kk$ is
a reasonable closed immersion.

\begin{lem} \label{torsion-shriek-of-external-product}
 Let $i'\:\Z'\rarrow\X'$ and $i''\:\Z''\rarrow\X''$ be reasonable closed
immersions of reasonable ind-schemes over\/~$\kk$, and let
$i=i'\times_\kk i''\:\Z'\times_\kk\Z''\rarrow\X'\times_\kk\X''$ be
the induced reasonable closed immersion of the Cartesian products.
 Let $\rM'$ be a quasi-coherent torsion sheaf on\/ $\X'$ and $\rM''$
be a quasi-coherent torsion sheaf on\/~$\X''$.
 Then there is a natural isomorphism
$$
 i^!(\rM'\bt_\kk\rM'')\simeq i'{}^!\rM'\bt_\kk i''{}^!\rM''
$$
of quasi-coherent torsion sheaves on\/ $\Z'\times_\kk\Z''$.
\end{lem}

\begin{proof}
 Follows immediately from the definition of the external tensor product
of quasi-coherent torsion sheaves, the definition of the functor~$i^!$
for a closed immersion of ind-schemes~$i$
(see Section~\ref{torsion-inverse-images-subsecn}), and
Lemma~\ref{shriek-of-external-product}.
\end{proof}

 Let $\rM'{}^\bu$ be a complex of quasi-coherent torsion sheaves on
$\X'$ and $\rM''{}^\bu$ be a complex of quasi-coherent torsion sheaves
on~$\X''$.
 Then the complex $\rM'{}^\bu\bt_\kk\rM''{}^\bu$ of quasi-coherent
torsion sheaves on $\X'\times_\kk\X''$ is constructed by totalizing
the bicomplex of external tensor products using infinite coproducts
along the diagonals.

\begin{lem}  \label{torsion-complexes-external-tensor-coacyclic}
 Let $\rM'{}^\bu$ be a complex of quasi-coherent torsion sheaves
on a reasonable ind-scheme\/ $\X'$ and $\rM''{}^\bu$ be a complex of
quasi-coherent torsion sheaves on a reasonable ind-scheme\/ $\X''$
over\/~$\kk$.
 Assume that the complex $\rM'{}^\bu$ is coacyclic (as a complex in\/
$\X'\tors$).
 Then the complex $\rM'{}^\bu\bt_\kk\rM''{}^\bu$ is coacyclic (as
a complex in $(\X'\times_\kk\X'')\tors$).
\end{lem}

\begin{proof}
 Follows from Lemmas~\ref{torsion-and-Gamma-systems-external-product}(a)
and~\ref{torsion-external-tensor-exact}.
\end{proof}

 It is clear from
Lemma~\ref{torsion-complexes-external-tensor-coacyclic} that
the external tensor product is well-defined as a functor between
the coderived categories of quasi-coherent torsion sheaves,
\begin{equation} \label{coderived-torsion-external-product}
 \bt_\kk\:\sD^\co(\X'\tors)\times\sD^\co(\X''\tors)\lrarrow
 \sD^\co((\X'\times_\kk\X'')\tors).
\end{equation}

\subsection{Derived restriction with supports}
\label{derived-supports-subsecn}
 Let $\X$ be an ind-Noetherian ind-scheme and $i\:\Z\rarrow\X$ be
a closed immersion of ind-schemes (then $\Z$ is also an ind-Noetherian
ind-scheme).
 The functor $i^!\:\X\tors\rarrow\Z\tors$ was defined in
Section~\ref{torsion-inverse-images-subsecn}.

 According to Corollary~\ref{ind-Noetherian-coderived-cor},
the inclusion of the full subcategory of injective quasi-coherent
torsion sheaves $\X\tors_\inj\rarrow\X\tors$ induces a triangulated
equivalence $\sK(\X\tors_\inj)\simeq\sD^\co(\X\tors)$.
 The right derived functor
$$
 \boR i^!\:\sD^\co(\X\tors)\lrarrow\sD^\co(\Z\tors)
$$
is constructed by applying the functor~$i^!$ to complexes of injective
quasi-coherent torsion sheaves on~$\X$.

 Notice that the right derived functor $\boR i^!$ preserves coproducts.
 Indeed, the underived functor $i^!\:\X\tors\rarrow\Z\tors$ preserves
coproducts for any reasonable closed immersion of reasonable ind-schemes
$i\:\Z\rarrow\X$; and coproducts of injective quasi-coherent torsion
sheaves are injective on an ind-Noetherian ind-scheme
(by Lemma~\ref{locally-Noetherian-coproducts-injective} and
Proposition~\ref{ind-Noetherian-torsion-locally-Noetherian}).

 Let $\X'$ and $\X''$ be ind-schemes of ind-finite type over~$\kk$
(then $\X'\times_\kk\X''$ is also an ind-scheme of ind-finite type).
 Let $i'\:\Z'\rarrow\X'$ and $i''\:\Z''\rarrow\X''$ be closed
immersions of ind-schemes; denote by $i=i'\times_\kk i''\:
\Z'\times_\kk\Z''\rarrow\X'\times_\kk\X''$ the induced closed immersion
of the Cartesian products.
 The aim of this Section~\ref{derived-supports-subsecn} is to prove
the following proposition (for another comparable result, see
Proposition~\ref{derived-supports-flat-pullback-prop} below).

\begin{prop} \label{derived-supports-external-product-prop}
 For any complexes of quasi-coherent torsion sheaves $\rM'{}^\bu$
on\/ $\X'$ and $\rM''{}^\bu$ on\/ $\X''$, there is a natural isomorphism
$$
 \boR i^!(\rM'{}^\bu\bt_\kk\rM''{}^\bu)\simeq
 \boR i'{}^!(\rM'{}^\bu)\bt_\kk\boR i''{}^!(\rM''{}^\bu)
$$
in the coderived category\/ $\sD^\co((\Z'\times_\kk\Z'')\tors)$.
\end{prop}

 The proof of Proposition~\ref{derived-supports-external-product-prop}
requires some work because the external tensor product of two
injective quasi-coherent (torsion) sheaves is usually \emph{not}
an injective quasi-coherent (torsion) sheaf.
 In fact, the assertion of the proposition follows almost immediately
from the next lemma (together with
Lemma~\ref{torsion-shriek-of-external-product}).

\begin{lem} \label{derived-supports-external-product-lemma}
 Let $\rJ'{}^\bu$ be a complex of injective quasi-coherent torsion
sheaves on\/ $\X'$ and $\rJ''{}^\bu$ be a complex of injective
quasi-coherent torsion sheaves on\/~$\X''$.
 Let $r\:\rJ'{}^\bu\bt_\kk\rJ''{}^\bu\rarrow\rK^\bu$ be a morphism of
complexes of quasi-coherent torsion sheaves on\/ $\X'\times_\kk\X''$
such that $\rK^\bu$ is a complex of injective quasi-coherent
torsion sheaves and the cone of\/~$r$ is a coacyclic complex of
quasi-coherent torsion sheaves on\/ $\X'\times_\kk\X''$.
 Then the induced morphism of complexes of quasi-coherent torsion
sheaves on\/ $\Z'\times_\kk\Z''$
$$
 i^!(r)\:i^!(\rJ'{}^\bu\bt_\kk\rJ''{}^\bu)\lrarrow i^!\rK^\bu
$$
has coacyclic cone.
\end{lem}

 The proof of Lemma~\ref{derived-supports-external-product-lemma} will
be given below near the end of Section~\ref{derived-supports-subsecn}.

\begin{lem} \label{Ext-of-external-tensor-product-modules}
 Let $R'$ and $R''$ be Noetherian commutative\/ $\kk$\+algebras,
$M'$ be a finitely generated $R'$\+module, $M''$ be a finitely
generated $R''$\+module, $N'$ be an $R'$\+module, and $N''$ be
an $R''$\+module.
 Then for every $n\ge0$ there is a natural isomorphism of\/
$\kk$\+vector spaces
$$
 \Ext^n_{R'\ot_\kk R''}(M'\ot_\kk M'',\>N'\ot_\kk N'')\simeq
 \bigoplus\nolimits_{p+q=n}
 \Ext^p_{R'}(M',N')\ot_\kk\Ext^q_{R''}(M'',N'').
$$
 In particular, for any injective $R'$\+module $J'$ and any
injective $R''$\+module $J''$ one has\/ $\Ext^n_{R'\ot_\kk R''}
(M'\ot_\kk M'',\>J'\ot_\kk J'')=0$ for all $n>0$.
\end{lem}

\begin{proof}
 The assumption of commutativity of the rings $R'$ and $R''$ is
actually not needed.
 One starts with the observation that, for any finitely generated
projective $R'$\+module $P'$ and any finitely generated projective
$R''$\+module $P''$ there is a natural isomorphism of Hom spaces
$$
 \Hom_{R'\ot_\kk R''}(P'\ot_\kk P'',\>N'\ot_\kk N'')\simeq
 \Hom_{R'}(P',N')\ot_\kk\Hom_{R''}(P'',N'').
$$
 Now let $P'_\bu\rarrow M'$ be a resolution of $M'$ by finitely
generated projective $R'$\+modules and $P''_\bu\rarrow M''$ be
a resolution of $M''$ by finitely generated projective $R''$\+modules.
 Then the tensor product of two complexes $P'_\bu\ot_\kk P''_\bu$ is
a resolution of $M'\ot_\kk M''$ by (finitely generated) projective
$(R'\ot_\kk R'')$\+modules.
 It remains to compute
$$
 \Hom_{R'\ot_\kk R''}(P'_\bu\ot_\kk P''_\bu,\>N'\ot_\kk N'')\simeq
 \Hom_{R'}(P'_\bu,N')\ot_\kk\Hom_{R''}(P''_\bu,N'')
$$
and recall that for any two complexes of $\kk$\+vector spaces
$C'{}^\bu$ and $C''{}^\bu$ one has $H^n(C'{}^\bu\ot_\kk C''{}^\bu)
\simeq\bigoplus_{p+q=n}H^p(C'{}^\bu)\ot_\kk H^q(C''{}^\bu)$.
\end{proof}

\begin{lem} \label{support-external-product-exactness}
 Let $X'$ and $X''$ be schemes of finite type over\/~$\kk$, and let
$i'\:Z'\rarrow X'$ and $i''\:Z''\rarrow X''$ be closed immersions
of schemes.
 Denote by $i\:Z'\times_\kk Z''\rarrow X'\times_\kk X''$ the induced
closed immersion of the Cartesian products.
 Let $\J'$ be an injective quasi-coherent sheaf on $X'$ and $\J''$ be
an injective quasi-coherent sheaf on~$X''$.
 Let $\J'\bt_\kk\J''\rarrow\K^\bu$ be an injective resolution of
the quasi-coherent sheaf $\J'\bt_\kk\J''$ on the scheme
$X'\times_\kk X''$.
 Then one has
$$
 H^0(i^!\K^\bu)\simeq i'{}^!\J'\bt_\kk i''{}^!\J''
 \quad\text{and}\quad
 H^n(i^!\K^\bu)=0 \text{ \ for\/ $n>0$}.
$$
\end{lem}

\begin{proof}
 To compute $H^0(i^!\K^\bu)$, it suffices to observe that
the functor~$i^!$ is left exact (as a right adjoint), so
$H^0(i^!\K^\bu)\simeq i^!(\J'\bt_\kk\J'')\simeq
i'{}^!\J'\bt_\kk i''{}^!\J''$ by Lemma~\ref{shriek-of-external-product}.
 The vanishing assertion in local in $X'$ and $X''$ (notice that all
the schemes involved are Noetherian, so injectivity of a quasi-coherent
sheaf is a local property), so it reduces to the case of affine
schemes, for which it means the following.

 Let $R'\rarrow S'$ and $R''\rarrow S''$ be surjective homomorphisms
of finitely generated commutative $\kk$\+algebras, and let $J'$ and
$J''$ be injective modules over $R'$ and $R''$, respectively.
 Let $K^\bu$ be an injective resolution of the $(R'\ot_\kk R'')$\+module
$J'\ot_\kk J''$.
 Then the complex $\Hom_{R'\ot_\kk R''}(S'\ot_\kk S'',\>K^\bu)$ has
vanishing cohomology in the positive cohomological degrees.
 This is a particular case of
Lemma~\ref{Ext-of-external-tensor-product-modules}.
\end{proof}

\begin{lem} \label{torsion-subscheme-support-external-product-exactness}
 Let\/ $\X'$ and\/ $\X''$ be ind-schemes of ind-finite type
over\/~$\kk$, and let $Z'\subset\X'$ and $Z''\subset\X''$ be closed
subschemes with the closed immersion morphisms $Z'\rarrow\X'$ and
$Z''\rarrow\X''$.
 Denote by $i\:Z'\times_\kk Z''\rarrow\X'\times_\kk\X''$ the induced
closed immersion of the Cartesian products.
 Let $\rJ'$ be an injective quasi-coherent torsion sheaf on\/ $\X'$ and
$\rJ''$ be an injective quasi-coherent torsion sheaf on\/~$\X''$.
 Let $\rJ'\bt_\kk\rJ''\rarrow\rK^\bu$ be an injective resolution of
the quasi-coherent torsion sheaf $\rJ'\bt_\kk\rJ''$ on the ind-scheme\/
$\X'\times_\kk\X''$.
 Then one has
$$
 H^0(i^!\rK^\bu)\simeq i'{}^!\rJ'\bt_\kk i''{}^!\rJ''
 \quad\text{and}\quad
 H^n(i^!\rK^\bu)=0 \text{ \ for\/ $n>0$}.
$$
\end{lem}

\begin{proof}
 The computation of $H^0$ is similar to the one in
Lemma~\ref{support-external-product-exactness} (use
Lemma~\ref{torsion-shriek-of-external-product}).
 To prove the higher cohomology vanishing, choose inductive systems
of closed immersions of schemes of finite type over~$\kk$ representing
the ind-schemes $\X'$ and $\X''$, and consider the related inductive
system representing the ind-scheme $\X'\times_\kk\X''$, as in
the beginning of Section~\ref{external-of-torsion-subsecn}.
 Put $\X=\X'\times_\kk\X''$ and $\Gamma=\Gamma'\times\Gamma''$.
 For any biindex $\gamma=(\gamma',\gamma'')\in\Gamma$, put
$X_\gamma=X'_{\gamma'}\times X''_{\gamma''}$.
 We can always assume that there exist $\gamma_0'\in\Gamma'$ and
$\gamma_0''\in\Gamma''$ such as $Z'=X'_{\gamma_0'}\subset\X'$ and
$Z''=X''_{\gamma_0''}\subset\X''$.

 Our aim is to show that the exact sequence of quasi-coherent torsion
sheaves $0\rarrow\rJ'\bt_\kk\rJ''\rarrow\rK^0\rarrow\rK^1\rarrow\rK^2
\rarrow\dotsb$ on $\X$ remains exact after applying the functor
$\rM\longmapsto\rM|_\Gamma\:\X\tors\rarrow(\X,\Gamma)\syst$.
 Notice that, by the definitions of the external tensor products,
we have $(\rJ'\bt_\kk\rJ'')|_\Gamma=
\rJ'|_{\Gamma'}\bt_\kk\rJ''|_{\Gamma''}$.

 Denote by $\boM^1$ the cokernel of the morphism of $\Gamma$\+systems
$(\rJ'\bt_\kk\rJ'')|_\Gamma\rarrow\rK^0|_\Gamma$.
 Let $\gamma=(\gamma',\gamma'')\le\delta=(\delta',\delta'')$ be two
biindices (where $\gamma'\le\delta'\in\Gamma'$ and
$\gamma''\le\delta''\in\Gamma'')$.
 Denote the related transition maps in the inductive systems by
$i'_{\gamma'\delta'}\:X'_{\gamma'}\rarrow X'_{\delta'}$ and
$i''_{\gamma''\delta''}\:X''_{\gamma''}\rarrow X''_{\delta''}$, and put
$i_{\gamma\delta}=i'_{\gamma'\delta'}\times_\kk i''_{\gamma''\delta''}\:
X_\gamma\rarrow X_\delta$.

 By Proposition~\ref{torsion-injectives-characterized}(a),
the quasi-coherent sheaves $\rK^n_{(X_\delta)}$ on the scheme
$X_\delta$ are injective for all $n\ge0$.
 By Lemma~\ref{support-external-product-exactness}, the short exact
sequence $0\rarrow(\rJ'\bt_\kk\nobreak\rJ'')_{(X_\delta)}\allowbreak
\rarrow\rK^0_{(X_\delta)}\rarrow\boM^1_{(\delta)}\rarrow0$ of
quasi-coherent sheaves on $X_\delta$ remains exact after applying
the functor~$i_{\gamma\delta}^!$.
 Hence the structure map $\boM^1_{(\gamma)}\rarrow
i_{\gamma\delta}^!\boM^1_{(\delta)}$ in the $\Gamma$\+system $\boM^1$
is an isomorphism (of quasi-coherent sheaves on~$X_\gamma$).

 As this holds for all $\gamma\le\delta\in\Gamma$, we can conclude
that the collection of quasi-coherent sheaves $\rM^1_{(X_\gamma)}=
\boM^1_{(\gamma)}$ defines a quasi-coherent torsion sheaf $\rM^1$
on~$\X$.
 So we have $\boM^1=\rM^1|_\Gamma$ and $\rM^1=\boM^1{}^+$.
 In other words, this means that the adjunction morphism $\boM^1\rarrow
\boM^1{}^+|_\Gamma$ is an isomorphism of $\Gamma$\+systems.
 Notice that the quasi-coherent torsion sheaf $\boM^1{}^+$ on $\X$ is,
by the definition, the cokernel of the monomorphism of quasi-coherent
torsion sheaves $\rJ'\bt_\kk\rJ''\rarrow\rK^0$.
 We have shown that the short exact sequence of quasi-coherent torsion
sheaves $0\rarrow\rJ'\bt_\kk\rJ''\rarrow\rK^0\rarrow\boM^1{}^+\rarrow0$
on $\X$ remains exact after applying the functor
$\rM\longmapsto\rM|_\Gamma$.

 Denote by $\boM^2$ the cokernel of the morphism of $\Gamma$\+systems
$\rK^0|_\Gamma\rarrow\rK^1|_\Gamma$; as we have seen, this is the same
thing as the cokernel of the morphism of $\Gamma$\+systems
$\boM^1\rarrow\rK^1|_\Gamma$.
 By Lemma~\ref{support-external-product-exactness}, the exact
sequence $0\rarrow(\rJ'\bt_\kk\nobreak\rJ'')_{(X_\delta)}\allowbreak
\rarrow\rK^0_{(X_\delta)}\rarrow\rK^1_{(X_\delta)}\rarrow
\boM^2_{(\delta)}\rarrow0$ of quasi-coherent sheaves on $X_\delta$
remains exact after applying the functor~$i_{\gamma\delta}^!$.
 Hence the structure map $\boM^2_{(\gamma)}\rarrow
i_{\gamma\delta}^!\boM^2_{(\delta)}$ in the $\Gamma$\+system $\boM^2$
is an isomorphism.
 As this holds for all $\gamma\le\delta\in\Gamma$, we can conclude
that the adjunction morphism $\boM^2\rarrow\boM^2{}^+|_\Gamma$ is
an isomorphism of $\Gamma$\+systems.

 Notice that the quasi-coherent torsion sheaf $\boM^2{}^+$ on $\X$ is,
by the definition, the cokernel of the morphism of quasi-coherent
torsion sheaves $\rK^0\rarrow\rK^1$.
 We have shown that the exact sequence of quasi-coherent torsion sheaves
$0\rarrow\rJ'\bt_\kk\nobreak\rJ''\allowbreak\rarrow\rK^0\rarrow\rK^1
\rarrow\boM^2{}^+\rarrow0$ on $\X$ remains exact after applying
the functor $\rM\longmapsto\rM|_\Gamma$.
 Proceeding in this way, we prove the desired preservation of exactness
by induction in the cohomological degree.
\end{proof}

\begin{lem} \label{torsion-support-external-product-exactness}
 Let\/ $\X'$ and\/ $\X''$ be ind-schemes of ind-finite type
over\/~$\kk$, and let $i'\:\Z'\rarrow\X'$ and $i''\:\Z''\rarrow\X''$
be closed immersions of ind-schemes.
 Denote by $i\:\Z'\times_\kk\Z''\rarrow\X'\times_\kk\X''$ the induced
closed immersion of the Cartesian products.
 Let $\rJ'$ be an injective quasi-coherent torsion sheaf on\/ $\X'$ and
$\rJ''$ be an injective quasi-coherent torsion sheaf on\/~$\X''$.
 Let $\rJ'\bt_\kk\rJ''\rarrow\rK^\bu$ be an injective resolution of
the quasi-coherent torsion sheaf $\rJ'\bt_\kk\rJ''$ on the ind-scheme\/
$\X'\times_\kk\X''$.
 Then one has
$$
 H^0(i^!\rK^\bu)\simeq i'{}^!\rJ'\bt_\kk i''{}^!\rJ''
 \quad\text{and}\quad
 H^n(i^!\rK^\bu)=0 \text{ \ for\/ $n>0$}.
$$
\end{lem}

\begin{proof}
 The computation of $H^0$ is similar to the one in
Lemmas~\ref{support-external-product-exactness}
and~\ref{torsion-subscheme-support-external-product-exactness}.
 The functor~$i^!$ is left exact as a right adjoint, and it
remains to use Lemma~\ref{torsion-shriek-of-external-product}.
 To prove the vanishing assertion, choose closed subschemes
$Z'\subset\Z'$ and $Z''\subset\Z''$ with the closed immersion
morphisms $k'\:Z'\rarrow\Z'$ and $k''\:Z''\rarrow\Z''$.
 Denote by $k=k'\times_\kk k''\:Z'\times_\kk Z''\rarrow
\Z'\times_\kk\Z''$ the induced closed immersion of
the Cartesian products.
 Then by
Lemma~\ref{torsion-subscheme-support-external-product-exactness}
we have $H^n(k^!i^!\rK^\bu)=0$ for $n>0$, and it follows that
$H^n(i^!\rK^\bu)=0$ for $n>0$ as well.
\end{proof}

\begin{lem} \label{torsion-support-sum-of-external-products-exactness}
 In the notation of
Lemma~\ref{torsion-support-external-product-exactness},
let $(\rJ'_\theta)_{\theta\in\Theta}$ be a family of injective
quasi-coherent torsion sheaves on\/ $\X'$ and
$(\rJ''_\theta)_{\theta\in\Theta}$ be a family of injective
quasi-coherent torsion sheaves on\/ $\X''$, indexed by the same
set\/~$\Theta$.
 Let\/ $\coprod_{\theta\in\Theta}\rJ'_\theta\bt_\kk\rJ''_\theta\rarrow
\rK^\bu$ be an injective resolution of the coproduct of quasi-coherent
torsion sheaves $\rJ'_\theta\bt_\kk\rJ''_\theta$ on the ind-scheme
$\X'\times_\kk\X''$.
 Then one has
$$
 H^0(i^!\rK^\bu)\simeq\coprod\nolimits_{\theta\in\Theta}
 i'{}^!\rJ'_\theta\bt_\kk i''{}^!\rJ''_\theta
 \quad\text{and}\quad
 H^n(i^!\rK^\bu)=0 \text{ \ for\/ $n>0$}.
$$
\end{lem}

\begin{proof}
 Essentially, the claim is that the right derived functor
$$
 \boR i^!\:\sD^\co((\X'\times_\kk\X'')\tors)\lrarrow
 \sD^\co((\Z'\times_\kk\Z'')\tors)
$$
preserves coproducts (as mentioned above in the beginning of
Section~\ref{derived-supports-subsecn}).
 This reduces the question to
Lemma~\ref{torsion-support-external-product-exactness}.
\end{proof}

\begin{proof}[Proof of
Lemma~\ref{derived-supports-external-product-lemma}]
 Given two complexes $\rJ'{}^\bu$ and $\rJ''{}^\bu$, a related complex
$\rK^\bu$ is defined uniquely up to a homotopy equivalence (by
Proposition~\ref{coderived-and-homotopy-of-injectives}(a)); so it
suffices to prove the assertion of the lemma for one specific choice
of the complex~$\rK^\bu$.
 We will use the complex $\rK^\bu$ provided by the construction on
which the proof of
Proposition~\ref{coderived-and-homotopy-of-injectives}(b) is based.

 Let $\rJ'{}^\bu\bt_\kk\rJ''{}^\bu\rarrow\rL^{0,\bu}$ be a monomorphism
of complexes in $(\X'\times_\kk\X'')\tors$ such that $\rL^{0,\bu}$ is
a complex of injective quasi-coherent torsion sheaves.
 Denote by $\rM^{1,\bu}$ the cokernel of this morphism of complexes,
and let $\rM^{1,\bu}\rarrow\rL^{1,\bu}$ be a monomorphism of complexes
in which $\rL^{1,\bu}$ is a complex of injectives.
 Proceeding in this way, we construct a bounded below complex of 
complexes of injective quasi-coherent torsion sheaves $\rL^{\bu,\bu}$
together with a quasi-isomorphism $\rJ'{}^\bu\bt_\kk\rJ''{}^\bu\rarrow
\rL^{\bu,\bu}$ of complexes of complexes in $(\X'\times_\kk\X'')\tors$.
 Let us emphasize that the notation $\rJ'{}^\bu\bt_\kk\rJ''{}^\bu$ here
stands for the total complex of the bicomplex of tensor products, while
$\rL^{\bu,\bu}$ is a bicomplex (not totalized yet).
 In every cohomological degree~$n$, the complex $\rL^{\bu,n}$ is
an injective resolution of the quasi-coherent torsion sheaf
$(\rJ'{}^\bu\bt_\kk\nobreak\rJ''{}^\bu)^n$.
 The complex $\rK^\bu$ is then constructed by totalizing the bicomplex
$\rL^{\bu,\bu}$ using infinite coproducts along the diagonals.

 Recall that the functor $i^!\:(\X'\times_\kk\X'')\tors\rarrow
(\Z'\times_\kk\Z'')\tors$ preserves coproducts.
 In every cohomological degree~$n$, applying~$i^!$ to the complex
$0\rarrow(\rJ'{}^\bu\bt_\kk\rJ''{}^\bu)^n\rarrow
\rL^{0,n}\rarrow\rL^{1,n}\rarrow\dotsb$ produces an acyclic complex
in $(\Z'\times_\kk\Z'')\tors$
by Lemma~\ref{torsion-support-sum-of-external-products-exactness}.
 It remains to point out that the coproduct totalization of an acyclic
bounded below complex of complexes is a coacyclic
complex~\cite[Lemma~2.1]{Psemi}.
\end{proof}

\begin{proof}[Proof of
Proposition~\ref{derived-supports-external-product-prop}]
 It is relevant that the external tensor product is well-defined as
a functor between the coderived categories (by
Lemma~\ref{torsion-complexes-external-tensor-coacyclic}).
 Choose a complex of injective quasi-coherent torsion sheaves
$\rJ'{}^\bu$ on $\X'$ and a complex of injective quasi-coherent
torsion sheaves $\rJ''{}^\bu$ on $\X''$ together with morphisms
$\rM'{}^\bu\rarrow\rJ'{}^\bu$ and $\rM''{}^\bu\rarrow\rJ''{}^\bu$
with coacyclic cones.
 Then it remains to apply
Lemma~\ref{derived-supports-external-product-lemma} and
take Lemma~\ref{torsion-shriek-of-external-product} into account.
\end{proof}

\subsection{Rigid dualizing complexes} \label{rigid-subsecn}
 The notion of a rigid dualizing complex was introduced originally
in~\cite[Definition~8.1]{VdB} and studied in~\cite{YZ,AIL} and many
other papers (see also~\cite{Yek} and~\cite{Sha}).
 Without going into details, we will formulate a simple definition
of a rigid dualizing complex on an ind-scheme of ind-finite type over
a field~$\kk$ in the form suitable for our purposes.

\begin{lem} \label{external-product-dualizing}
 Let\/ $\X'$ and\/ $\X''$ be ind-semi-separated ind-schemes of
ind-finite type over\/~$\kk$.
 Let $\rD'{}^\bu$ be a dualizing complex on\/ $\X'$ and $\rD''{}^\bu$
be a dualizing complex on\/~$\X''$.
 Let $\rE^\bu$ be a complex of injective quasi-coherent torsion
sheaves on\/ $\X'\times_\kk\X''$ endowed with a morphism of complexes
$\rD'{}^\bu\bt_\kk\rD''{}^\bu\rarrow\rE^\bu$ with coacyclic cone.
 Then $\rE^\bu$ is a dualizing complex on\/ $\X'\times_\kk\X''$.
\end{lem}

\begin{proof}
 Proposition~\ref{derived-supports-external-product-prop}
(for closed subschemes $Z'\subset\X'$ and $Z''\subset\X''$)
reduces the question to schemes, for which it is straightforward.
\end{proof}

 Let $\X$ be an ind-scheme over~$\kk$.
 Denote by $\Delta\:\X\rarrow\X\times_\kk\X$ the diagonal morphism.
 Then the ind-scheme $\X$ is ind-separated (as defined in
Section~\ref{dualizing-definition-subsecn}) if and only if
the morphism~$\Delta$ is a closed immersion of ind-schemes.

 Let $\X$ be an ind-separated ind-scheme of ind-finite type over~$\kk$.
 An object $\rE^\bu\in\sD^\co(\X\tors)$ in the coderived category of
quasi-coherent torsion sheaves on $\X$ is said to be \emph{rigid}
if it is endowed with an isomorphism $\rE^\bu\simeq
\boR\Delta^!(\rE^\bu\bt_\kk\rE^\bu)$ (called the \emph{rigidifying
isomorphism}) in the coderived category $\sD^\co(\X\tors)$.
 A dualizing complex $\rD^\bu\in\sC(\X\tors_\inj)$ on $\X$ is said
to be \emph{rigid} if it is rigid (i.~e., has been endowed with
a rigidifying isomorphism) as an object of $\sD^\co(\X\tors)$.

 Our aim is to explain how to produce a rigid dualizing complex on
an $\aleph_0$\+ind-scheme from rigid dualizing complexes on schemes,
in the spirit of Example~\ref{aleph-zero-dualizing-glueing}.
 For this purpose, we need to start with some preliminary work.

\begin{lem} \label{coderived-isomorphisms-characterized}
 Let $\X=\ilim_{\gamma\in\Gamma}X_\gamma$ be an ind-Noetherian
ind-scheme represented by an inductive system of closed immersions
of (Noetherian) schemes.
 Denote by $i_\gamma\:X_\gamma\rarrow\X$ the closed immersion
morphisms.
 Let $f\:\rM^\bu\rarrow\rN^\bu$ be a morphism in the coderived
category\/ $\sD^\co(\X\tors)$.
 Then the morphism~$f$ is an isomorphism if and only if the morphism\/
$\boR i_\gamma^!(f)\:\boR i_\gamma^!\rM^\bu\rarrow\boR i_\gamma^!
\rN^\bu$ is an isomorphism in\/ $\sD^\co(X_\gamma\qcoh)$ for every\/
$\gamma\in\Gamma$.
\end{lem}

\begin{proof}
 It suffices to show that $\boR i_\gamma^!\rL^\bu=0$ for $\rL^\bu\in
\sD^\co(\X\tors)$ and all $\gamma\in\Gamma$ implies $\rL^\bu=0$ in
$\sD^\co(\X\tors)$.
 Indeed, by Corollary~\ref{ind-Noetherian-coderived-cor}, there exists
a complex of injective quasi-coherent torsion sheaves $\rJ^\bu\in
\sK(\X\tors_\inj)$ such that $\rL^\bu\simeq\rJ^\bu$ in
$\sD^\co(\X\tors)$.
 Then it remains to apply
Lemma~\ref{complex-of-injective-torsion-contractible}.
\end{proof}

\begin{prop} \label{aleph-zero-torsion-morphism-strictification-prop}
 Let $\X=\ilim\,(X_0\to X_1\to X_2\to\dotsb)$ be an ind-Noetherian
$\aleph_0$\+ind-scheme represented by an inductive system of closed
immersions of (Noetherian) schemes indexed by the poset of nonnegative
integers.
 Let $i_n\:X_n\rarrow X_{n+1}$ and $k_n\:X_n\rarrow\X$, \,$n\ge0$,
denote the closed immersion morphisms.
 Let $\rM^\bu$ and $\rN^\bu\in\sD^\co(\X\tors)$ be two objects in
the coderived category of quasi-coherent torsion sheaves.
 Suppose that, for every $n\ge0$, we are given a morphism
$f_n\:\boR k_n^!\rM^\bu\rarrow\boR k_n^!\rN^\bu$ in\/
$\sD^\co(X_n\qcoh)$ such that, for every $n\ge0$, the morphisms~$f_n$
and\/ $\boR i_n^!(f_{n+1})$ form a commutative square diagram with
the natural isomorphisms\/ $\boR k_n^!\rM^\bu\simeq
\boR i_n\boR k_{n+1}^!\rM^\bu$ and\/ $\boR k_n^!\rN^\bu\simeq
\boR i_n\boR k_{n+1}^!\rN^\bu$.
 Then \par
\textup{(a)} there exists a morphism $f\:\rM^\bu\rarrow\rN^\bu$ in\/
$\sD^\co(\X\tors)$ such that $f_n=\boR k_n^!(f)$ for every $n\ge0$; \par
\textup{(b)} assuming additionally that\/
$\Hom_{\sD^\co(X_n\qcoh)}(\boR k_n^!\rM^\bu,\boR k_n^!\rN^\bu[-1])=0$
for all\/ $n\ge0$, the morphism~$f$ with these properties
is also unique.
\end{prop}

 The proof of the proposition is based on the next lemma.

\begin{lem} \label{aleph-zero-torsion-morphism-strictification-lemma}
 In the assumptions and notation of the proposition, let\/
$\rL^\bu\in\sK(\X\tors)$ be a complex of quasi-coherent torsion sheaves
and $\rJ^\bu\in\sK(\X\tors_\inj)$ be a complex of injective
quasi-coherent torsion sheaves on\/~$\X$.
 Suppose that, for every $n\ge0$, we are given a morphism
$g_n\:k_n^!\rL^\bu\rarrow k_n^!\rJ^\bu$ in\/ $\sK(X_n\qcoh)$ such that,
for every $n\ge0$, the morphisms~$g_n$ and\/ $i_n^!(g_{n+1})$ form
a commutative square diagram with the natural isomorphisms\/
$k_n^!\rL^\bu\simeq i_n^! k_{n+1}^!\rL^\bu$ and\/
$k_n^!\rJ^\bu\simeq i_n k_{n+1}^!\rJ^\bu$.
 Then \par
\textup{(a)} there exists a morphism $g\:\rL^\bu\rarrow\rJ^\bu$ in\/
$\sK(\X\tors)$ such that $g_n=k_n^!(g)$\/ in $\sK(X_n\qcoh)$
for every $n\ge0$; \par
\textup{(b)} assuming additionally that\/
$\Hom_{\sK(X_n\qcoh)}(k_n^!\rL^\bu,k_n^!\rJ^\bu[-1])=0$ for all\/
$n\ge0$, the morphism~$g$ with these properties is also unique.
\end{lem}

\begin{proof}
 Let $\rL^\bu\in\sC(\X\tors)$ and $\rJ^\bu\in\sC(\X\tors_\inj)$ be
our two complexes.
 The lemma claims that, given a system of morphisms $g_n\:k_n^!\rL^\bu
\rarrow k_n^!\rJ^\bu$ which is compatible up to chain homotopy,
one can replace every morphism~$g_n$ by a homotopic morphism
$f_n\:k_n^!\rL^\bu\rarrow k_n^!\rJ^\bu$ in $\sC(X_n\qcoh)$ such
that the morphisms~$f_n$ are compatible with each other in the strict
sense, i.~e., the square diagrams become commutative in the abelian
category of complexes $\sC(X_n\qcoh)$.

 Part~(a): proceeding by induction in~$n\ge0$, suppose that we have 
constructed chain homotopies $h_m\:k_m^!\rL^\bu\rarrow k_m^!\rJ^\bu[-1]$
for all $m\le n$ such that the morphisms $f_m=g_m+d(h_m)$ satisfy
$f_m=i_m^!(f_{m+1})$ in $\sC(X_m\qcoh)$ for all $m<n$.
 For $n=0$ (the induction base), it suffices to take $h_0=0$.
 Let us construct~$h_{n+1}$.
 
 By assumption, we have $f_n=i_n^!g_{n+1}$ in $\sK(X_n\qcoh)$; in
other words, there exists a chain homotopy $u_n\:k_n^!\rL^\bu\rarrow
k_n^!\rJ^\bu[-1]$ such that $f_n=i_n^!g_{n+1}+d(u_n)$
in $\sC(X_n\qcoh)$.
 Consider the map $i_n{}_*u_n\:i_n{}_*k_n^!\rL^\bu\rarrow
i_n{}_*k_n^!\rJ^\bu[-1]$ and compose it with the adjunction morphism
$i_n{}_*k_n^!\rJ^\bu[-1]=i_n{}_*i_n^!k_{n+1}^!\rJ^\bu[-1]\rarrow
k_{n+1}^!\rJ^\bu[-1]$.
 The resulting map $t_n\:i_n{}_*k_n^!\rL^\bu\rarrow
k_{n+1}^!\rJ^\bu[-1]$ is a morphism of graded objects in $X_{n+1}\qcoh$,
where the target $k_{n+1}^!\rJ^\bu[-1]$ is a graded object in
$X_{n+1}\qcoh_\inj$.

$$
\xymatrix{
 k_{n+1}^!\rL^\bu
 \ar[rr]^{g_{n+1}} \ar@{==>}@/_1.75pc/[rr]^{h_{n+1}}
 && k_{n+1}^!\rJ^\bu \\ \\
 i_n{}_*k_n^!\rL^\bu
 \ar[rr]^{i_n{}_*f_n} \ar@{=>}@/_0.7pc/[rruu]^{t_n} \ar@{^{(}->}[uu]
 && i_n{}_*k_n^!\rJ^\bu \ar@{^{(}->}[uu]
}
$$

 Now the adjunction morphism $i_n{}_*k_n^!\rL^\bu=
i_n{}_*i_n^!k_{n+1}^!\rL^\bu\rarrow k_{n+1}^!\rL^\bu$ is
a monomorphism, so it makes $i_n{}_*k_n^!\rL^\bu$ a graded subobject
(in fact, a subcomplex) in $k_{n+1}^!\rL^\bu$.
 By injectivity of $k_{n+1}^!\rJ^\bu$, the map~$t_n$ can be extended to
a morphism of graded objects $h_{n+1}\:k_{n+1}^!\rL^\bu\rarrow
k_{n+1}^!\rJ^\bu[-1]$ in $X_{n+1}\qcoh$, which provides the desired
chain homotopy for which the morphisms of complexes $f_{n+1}=g_{n+1}
+d(h_{n+1})\:\allowbreak k_{n+1}^!\rL^\bu\rarrow k_{n+1}^!\rJ^\bu$
satisfy $f_n=i_n^!f_{n+1}$ in $\sC(X_n\qcoh)$.

 The proof part~(b) is similar.
 Suppose that we are given a morphism $g\:\rL^\bu\rarrow\rJ^\bu$ in
$\sK(\X\tors)$ such that $k_n^!(g)=0$ in the homotopy category
$\sK(X_n\qcoh)$ for every $n\ge0$.
 Then there exist chain homotopies $h_n\:k_n^!\rL^\bu\rarrow
k_n^!\rJ^\bu[-1]$ such that $k_n^!(g)=d(h_n)$ in the category
of complexes $\sC(X_n\qcoh)$.
 The difference $h_n-i_n^!(h_{n+1})$ satisfies $d(h_n-i_n^!(h_{n+1}))
=d(h_n)-i_n^!(d(h_{n+1}))=k_n^!(g)-i_n^!k_{n+1}^!(g)=0$; so
$h_n-i_n^!(h_{n+1})$ is a morphism $k_n^!\rL^\bu\rarrow
k_n^!\rJ^\bu[-1]$ in the category of complexes $\sC(X_n\qcoh)$.
 By the assumption of part~(b), this morphism must be homotopic
to zero.
 So the chain homotopies~$h_n$ ``agree with each other up to a chain
homotopy of the next degree''.

 Arguing as in part~(a) and using the assumption that
$k_{n+1}^!\rJ^\bu$ are complexes of injective quasi-coherent sheaves
on $X_{n+1}$, one proceeds by induction in~$n$, constructing chain
homotopies $r_{n+1}\:k_{n+1}^!\rL^\bu\rarrow k_{n+1}^!\rJ^\bu[-2]$
such that the chain homotopies $q_{n+1}=h_{n+1}+d(r_{n+1})$ satisfy
$q_n=i_n^!q_{n+1}$ in $\sC(X_n\qcoh)$ for all $n\ge0$.
 The point is that the discrepancy between $q_n$ and~$i_n^!h_{n+1}$
can be lifted from a chain homotopy $k_n^!\rL^\bu\rarrow
k_n^!\rJ^\bu[-2]$ to a chain homotopy $k_{n+1}^!\rL^\bu\rarrow
k_{n+1}^!\rJ^\bu[-2]$.
\end{proof}

\begin{proof}[Proof of
Proposition~\ref{aleph-zero-torsion-morphism-strictification-prop}]
 Use Corollary~\ref{ind-Noetherian-coderived-cor} in order to find
two complexes of injective quasi-coherent torsion sheaves $\rL^\bu$
and $\rJ^\bu$ on $\X$ together with a morphism $\rM^\bu\rarrow\rL^\bu$
and $\rN^\bu\rarrow\rJ^\bu$ with coacyclic cones.
 Furthermore, notice that, for any scheme $X$ and any complexes
$\cL^\bu\in\sK(X\qcoh_\inj)$ and $\J^\bu\in\sK(X\qcoh_\inj)$,
the natural morphism of Hom groups $\Hom_{\sK(X\qcoh)}(\cL^\bu,\J^\bu)
\rarrow\Hom_{\sD^\co(X\qcoh)}(\cL^\bu,\J^\bu)$ is an isomorphism.
 The latter assertion holds by
Proposition~\ref{coderived-and-homotopy-of-injectives}(a)
(a slightly more precise version of
Proposition~\ref{coderived-and-homotopy-of-injectives}(a), which is
also valid in any abelian/exact category with exact coproducts,
tells that the same holds for any $\cL^\bu\in\sK(X\qcoh)$; but we
still need $\rL^\bu\in\sK(\X\tors_\inj)$ in order to have
$\boR k_n^!\rL^\bu=k_n^!\rL^\bu$).
 These observations reduce the assertions of the proposition
to those of
Lemma~\ref{aleph-zero-torsion-morphism-strictification-lemma}.
\end{proof}

\begin{exs}
 (0)~We refer to Remark~\ref{cotensor-right-derived-remark}(3) for
the discussion of the extraordinary inverse image functors and
the notation $f^+\:\sD^+(X\qcoh)\rarrow\sD^+(Y\qcoh)$ for a morphism
of finite type between Noetherian schemes $f\:Y\rarrow X$.

 The functor~$f^+$ commutes with external tensor products.
 In particular, given a morphism of finite type $p\:X\rarrow\Spec\kk$,
consider the morphism $p\times_\kk p\:X\times_\kk X\rarrow\Spec\kk$,
and denote simply by~$\kk$ the quasi-coherent sheaf on $\Spec\kk$
corresponding to the $\kk$\+vector space~$\kk$.
 Then there is a natural isomorphism $(p\times_\kk p)^+(\kk)\simeq
p^+(\kk)\bt_\kk\nobreak p^+(\kk)$ in $\sD^+((X\times_\kk X)\qcoh)$.
 In fact, $p^+(\kk)$ is a complex of quasi-coherent sheaves with
bounded cohomology sheaves, which are coherent sheaves on $X$,
and the complex $p^+(\kk)$ also has finite injective dimension.

 Assume that the scheme $X$ is separated, and choose a bounded
complex $\D^\bu\in\sK^\bb(X\qcoh_\inj)$ of injective quasi-coherent
sheaves on $X$ quasi-isomorphic to~$p^+(\kk)$.
 Then $\D^\bu$ is a dualizing complex on~$X$.
 Moreover, $\D_n^\bu$ is a \emph{rigid} dualizing complex, i.~e.,
denoting the diagonal morphism by $\Delta_X\:X\rarrow X\times_\kk X$,
there is a natural isomorphism $\D^\bu\simeq
\boR\Delta_X^!(\D^\bu\bt_\kk\D^\bu)$ in $\sD^\bb(X\qcoh)\subset
\sD^+(X\qcoh)\subset\sD^\co(X\qcoh)$.

\smallskip
 (1)~Let $\X=\ilim\,(X_0\to X_1\to X_2\to\dotsb)$ be
an ind-separated $\aleph_0$\+ind-scheme of ind-finite type
over~$\kk$, represented by an inductive system of closed immersions
of schemes indexed by the nonnegative integers.
 For every $n\ge0$, denote by $p_n\:X_n\rarrow\Spec\kk$ the structure
morphism of the scheme $X_n$ over~$\kk$, and let $\D_n^\bu\in
\sK^\bb(X_n\qcoh_\inj)$ denote a bounded complex of injective
quasi-coherent sheaves on $X_n$ quasi-isomorphic to~$p_n^+(\kk)$.
 Then $\D_n^\bu$ is a rigid dualizing complex on~$X_n$.

 In the above notation $i_n\:X_n\rarrow X_{n+1}$ for the transition
morphisms in the inductive system of schemes, there are also natural
isomorphisms $p_n^+(\kk)\simeq\boR i_n^!p_{n+1}^+(\kk)$ in
$\sD^+(X_n\qcoh)$, which lead to natural homotopy equivalences
$\D_n^\bu\simeq i_n^!\D_{n+1}^\bu$ in $\sK^\bb(X_n\qcoh_\inj)$.
 Using the construction of Example~\ref{aleph-zero-dualizing-glueing},
we produce a dualizing complex $\rD^\bu\in\sK(\X\tors_\inj)$ together
with natural homotopy equivalences $\D_n^\bu\simeq k_n^!\rD^\bu$
(where $k_n\:X_n\rarrow\X$ are the closed immersions).

\smallskip
 (2)~For every $n\ge0$, choose a bounded complex of injective
quasi-coherent sheaves $\E_n^\bu$ on $X_n\times_\kk X_n$
quasi-isomorphic to $(p_n\times_\kk p_n)^+(\kk)$.
 So $\E_n^\bu$ is a dualizing complex on $X_n\times_\kk X_n$ naturally
isomorphic to the complex $\D_n^\bu\bt_\kk\D_n^\bu$ in
the derived category $\sD^\bb((X_n\times_\kk X_n)\qcoh)$.
 We also have natural homotopy equivalences $\Delta_n^!(\E_n^\bu)
\simeq\D_n^\bu$ in $\sK^\bb(X_n\qcoh_\inj)$ (where $\Delta_n\:
X_n\rarrow X_n\times_\kk X_n$ is the diagonal morphism) and
$\E_n^\bu\simeq (i_n\times_\kk i_n)^!\E_{n+1}^\bu$ in
$\sK^\bb((X_n\times_\kk X_n)\qcoh_\inj)$.

 Using the construction of Example~\ref{aleph-zero-dualizing-glueing},
we can produce a dualizing complex $\rE^\bu\in
\sK((\X\times_\kk\X)\tors_\inj)$ together with natural homotopy
equivalences $\E_n^\bu\simeq (k_n\times_\kk k_n)^!\rE^\bu$.
 Now Lemma~\ref{aleph-zero-torsion-morphism-strictification-lemma}(a)
allows to extend these data to a morphism
$\rD^\bu\rarrow\Delta^!\rE^\bu$ in $\sK(\X\tors_\inj)$
(where $\Delta\:\X\rarrow\X\times_\kk\X$ is the diagonal).
 Using the same lemma, we also obtain a morphism
$\rD^\bu\bt_\kk\rD^\bu\rarrow\rE^\bu$ in $\sK((\X\times_\kk\X)\tors)$.
 Moreover,
Lemma~\ref{aleph-zero-torsion-morphism-strictification-lemma}(b)
combined with Lemma~\ref{dualizing-no-negative-selfext} show that
the relevant morphisms in the homotopy categories of quasi-coherent
torsion sheaves are uniquely defined.

 Alternatively, one can use
Proposition~\ref{aleph-zero-torsion-morphism-strictification-prop}(a)
together with Proposition~\ref{derived-supports-external-product-prop}.
 This allows to avoid mentioning the specific complexes $\E_n^\bu$
and~$\rE^\bu$, constructing directly a morphism
$\rD^\bu\rarrow\boR\Delta^!(\rD^\bu\bt_\kk\rD^\bu)$ in the coderived
category $\sD^\co(\X\tors)$ instead.
 This morphism in the coderived category is uniquely defined by
Proposition~\ref{aleph-zero-torsion-morphism-strictification-prop}(b)
combined with Lemma~\ref{dualizing-no-negative-selfext}.

 By Lemma~\ref{coderived-isomorphisms-characterized}, both
the morphisms $\rD^\bu\rarrow\Delta^!\rE^\bu$ and
$\rD^\bu\bt_\kk\rD^\bu\rarrow\rE^\bu$ are isomorphisms in the respective
coderived categories $\sD^\co(\X\tors)$ and
$\sD^\co((\X\times_\kk\X)\tors)$.
 Thus $\rD^\bu\in\sK(\X\tors_\inj)$ is a rigid dualizing complex
on~$\X$.
\end{exs}

\subsection{Covariant duality commutes with external tensor products}
\label{covariant-external-commute-subsecn}
 Let $\X$ be an ind-Noetherian ind-scheme, and
let $\rM^\bu\in\sC(\X\tors)$ be a complex of quasi-coherent torsion
sheaves on~$\X$.
 For any complex of flat pro-quasi-coherent pro-sheaves $\fF^\bu$
on $\X$, put
$$
 \Phi_{\rM^\bu}(\fF^\bu)=\rM^\bu\ot_\X\fF^\bu\,\in\,\sC(\X\tors).
$$
 According to
formula~\eqref{flat-torsion-module-structure-on-co-derived} from
Section~\ref{cotensor-product-construction-subsecn},
the functor $\Phi_{\rM^\bu}$ induces a well-defined triangulated
functor from the derived category of flat pro-quasi-coherent
pro-sheaves to the coderived category of quasi-coherent torsion sheaves,
$$
 \Phi_{\rM^\bu}\:\sD(\X\flat)\lrarrow\sD^\co(\X\tors).
$$
 Furthermore, any morphism $\rM^\bu\rarrow\rN^\bu$ in the coderived
category $\sD^\co(\X\tors)$ induces a morphism of functors
$\Phi_{\rM^\bu}\rarrow\Phi_{\rN^\bu}$, and any isomorphism $\rM^\bu
\simeq\rN^\bu$ in $\sD^\co(\X\tors)$ induces an isomorphism of
triangulated functors $\Phi_{\rM^\bu}\simeq\Phi_{\rN^\bu}$.

\begin{lem} \label{torsion-pro-action-external-product}
 Let\/ $\X'$ and\/ $\X''$ be reasonable ind-schemes over\/~$\kk$.
 Let $\rM'$ and $\rM''$ be quasi-coherent torsion sheaves on\/
$\X'$ and\/ $\X''$, and let\/ $\fP'$ and\/ $\fP''$ be
pro-quasi-coherent pro-sheaves on\/ $\X'$ and\/ $\X''$, respectively.
 Then there is a natural isomorphism
$$
 (\rM'\bt_\kk\rM'')\ot_{\X'\times_\kk\X''}(\fP'\bt_\kk\fP'')
 \simeq(\rM'\ot_{\X'}\fP')\bt_\kk(\rM''\ot_{\X''}\fP'')
$$
of quasi-coherent torsion sheaves on\/ $\X'\times_\kk\X''$.
\end{lem}

\begin{proof}
 Let $\X'=\ilim_{\gamma'\in\Gamma'}X'_{\gamma'}$ and
$\X''=\ilim_{\gamma''\in\Gamma''}X''_{\gamma''}$ be representations
of $\X'$ and $\X''$ by inductive systems of closed immersions of
reasonable closed subschemes.
 Then $\X'\times_\kk\X''=\ilim_{(\gamma',\gamma'')\in
\Gamma'\times\Gamma''}X'_{\gamma'}\times_\kk X''_{\gamma''}$
is a similar representation of the ind-scheme $\X\times_\kk\X''$
(see Section~\ref{external-of-torsion-subsecn}).
 To prove the lemma, one first observes that a similar isomorphism
obviously holds for $\Gamma$\+systems in place of quasi-coherent
torsion sheaves.
 Let $\boM'$ be a $\Gamma'$\+system on $\X'$ and $\boM''$ be
a $\Gamma''$\+system on~$\X''$.
 Then the external tensor product $\boM'\bt_\kk\boM''$, as defined in
Section~\ref{external-of-torsion-subsecn}, is
a $(\Gamma'\times\Gamma'')$\+system on~$\X'\times_\kk\X''$.
 One has a natural isomorphism of $(\Gamma'\times\Gamma'')$\+systems
$$
 (\fP'\bt_\kk\fP'')\ot_{\X'\times_\kk\X''}(\boM'\bt_\kk\boM'')
 \simeq(\fP'\ot_{\X'}\boM')\bt_\kk(\fP''\ot_{\X''}\boM'').
$$
 In order to deduce the desired isomorphism of quasi-coherent torsion
sheaves, it remains to recall the definition of the functor
$\ot_\X\:\X\pro\times\X\tors\rarrow\X\tors$ in
Section~\ref{pro-in-torsion-action-subsecn} and
use Lemma~\ref{torsion-and-Gamma-systems-external-product}(b).
\end{proof}

 Let $\X'$ and $\X''$ be ind-schemes of ind-finite type over~$\kk$,
and let $\rM'{}^\bu$ and $\rM''{}^\bu$ be complexes of quasi-coherent
torsion sheaves on $\X'$ and~$\X''$.
 Then it follows from Lemma~\ref{torsion-pro-action-external-product}
that, for any complexes flat pro-quasi-coherent pro-sheaves $\fP'{}^\bu$
and $\fP''{}^\bu$ on $\X'$ and $\X''$, the natural isomorphism
$$
 \Phi_{\rM'{}^\bu\subbt_\kk\rM''{}^\bu}(\fP'{}^\bu\bt_\kk\fP''{}^\bu)
 \simeq
 \Phi_{\rM'{}^\bu}(\fP'{}^\bu)\bt_\kk\Phi_{\rM''{}^\bu}(\fP''{}^\bu)
$$
holds in the category of complexes of quasi-coherent torsion sheaves
on $\X'\times_\kk\X''$.

\begin{cor} \label{covariant-duality-external-product-cor}
 Let\/ $\X'$ and\/ $\X''$ be ind-semi-separated ind-schemes of
ind-finite type over\/~$\kk$.
 Let $\rD'{}^\bu$ and $\rD''{}^\bu$ be dualizing complexes on\/ $\X'$
and\/ $\X''$, respectively, and let $\rE^\bu$ be the related dualizing
complex on\/ $\X'\times_\kk\X''$, as in
Lemma~\ref{external-product-dualizing}.
 Then the triangulated equivalences\/ $\sD(\X'\flat)\simeq
\sD^\co(\X''\tors)$, \ $\sD(\X''\flat)\simeq\sD(\X''\tors)$, and\/
$\sD((\X'\times_\kk\X'')\flat)\simeq\sD((\X'\times_\kk\X'')\tors)$
from Theorem~\ref{ind-Noetherian-triangulated-equivalence-thm},
induced by the dualizing complexes $\rD'{}^\bu$, $\rD''{}^\bu$,
and $\rE^\bu$, form a commutative square diagram with the external
tensor product functors\/~$\bt_\kk$
\,\eqref{derived-flat-external-product}
and~\eqref{coderived-torsion-external-product}
from Sections~\ref{external-of-pro-subsecn}\+-%
\ref{external-of-torsion-subsecn}.
\end{cor}

\begin{proof}
 Follows immediately from the preceding discussion.
\end{proof}

\subsection{The cotensor product as the $!$-tensor product}
 The definition of a reasonable closed immersion of ind-schemes,
which is used in the second assertion of the following lemma, was given
in Section~\ref{external-of-torsion-subsecn}.

\begin{lem} \label{ind-scheme-shriek-star-tensor}
 Let $i\:\Z\rarrow\X$ be a closed immersion of reasonable ind-schemes.
 Let\/ $\fF$ be a pro-quasi-coherent pro-sheaf on\/ $\X$ and
$\rM$ be a quasi-coherent torsion sheaf on\/~$\X$.
 Then there is a natural morphism of quasi-coherent torsion sheaves
on\/~$\Z$
$$
 i^*\fF\ot_\Z i^!\rM\lrarrow i^!(\fF\ot_\X\rM),
$$
which is an isomorphism whenever $i$~is a reasonable closed immersion
and\/ $\fF$ is a flat pro-quasi-coherent pro-sheaf on~$\X$.
\end{lem}

\begin{proof}
 This is a generalization of
Lemma~\ref{shriek-star-tensor}
and Proposition~\ref{flat-torsion-tensor-prop}.
 To construct the desired morphism, let $\rL$ be an arbitrary
quasi-coherent torsion sheaf on $\Z$, and let $\rL\rarrow
i^*\fF\ot_\Z i^!\rM$ be a morphism in $\Z\tors$.
 Applying the direct image functor~$i_*$, we produce a morphism
$i_*\rL\rarrow i_*(i^*\fF\ot_\Z i^!\rM)\simeq\fF\ot_\X i_*i^!\rM$
in $\X\tors$ (where the isomorphism holds by
Lemma~\ref{torsion-pro-projection-formula-second} below).
 Composing with the morphism $\fF\ot_\X i_*i^!\rM\rarrow\fF\ot_\X\rM$
induced by the adjunction morphism $i_*i^!\rM\rarrow\rM$, we obtain
a morphism $i_*\rL\rarrow\fF\ot_X\rM$ in $\X\tors$, which corresponds
by adjunction to a morphism $\rL\rarrow i^!(\fF\ot_\X\rM)$ in $\Z\tors$.

 To prove isomorphism assertion, we let $Z\subset\Z$ be a reasonable
closed subscheme with the closed immersion morphism $k\:Z\rarrow\Z$,
and apply the functor~$k^!$ to the morphism in question.
 Notice that $ik(Z)$ is a reasonable closed subscheme in~$\X$.
 By Proposition~\ref{flat-torsion-tensor-prop} applied to
the closed subscheme $Z\subset\Z$, we have
$$
 k^!(i^*\fF\ot_\Z i^!\rM)\simeq k^*i^*\fF\ot_{\cO_Z}k^!i^!\rM.
$$
 Using the same proposition applied to the closed subscheme $ik(Z)
\subset\X$, we compute
$$
 k^!i^!(\fF\ot_\X\rM)\simeq(ik)^!(\fF\ot_\X\rM)\simeq
 (ik)^*\fF\ot_{\cO_Z}(ik)^!\rM\simeq
 k^*i^*\fF\ot_{\cO_Z}k^!i^!\rM,
$$
so the functor~$k^!$ transforms the morphism in question into
an isomorphism.
 As this holds for every reasonable closed subscheme $Z\subset\Z$,
the assertion follows.
\end{proof}

 Let $f\:\Y\rarrow\X$ be a morphism of ind-schemes.
 Then the inverse image functor $f^*\:\X\flat\rarrow\Y\flat$ is
exact (see Section~\ref{flat-pro-sheaves-subsecn}), so it induces
a triangulated functor between the derived categories
$$
 f^*\:\sD(\X\flat)\lrarrow\sD(\Y\flat).
$$

\begin{prop} \label{covariant-duality-shriek-star-prop}
 Let\/ $\X$ be an ind-semi-separated ind-Noetherian scheme and
$i\:\Z\rarrow\X$ be a closed immersion of ind-schemes.
 Let $\rD^\bu$ be a dualizing complex on\/~$\X$; then $i^!\rD^\bu$
is a dualizing complex on\/~$\Z$ (cf.\
Example~\ref{ind-closed-immersion}(2)).
 Then the triangulated equivalences\/ $\sD(\X\flat)\simeq
\sD^\co(\X\tors)$ and\/ $\sD(\Z\flat)\simeq\sD^\co(\Z\tors)$
from Theorem~\ref{ind-Noetherian-triangulated-equivalence-thm},
induced by the dualizing complexes $\rD^\bu$ and\/ $i^!\rD^\bu$,
transform the inverse image functor
$$
 i^*\:\sD(\X\flat)\lrarrow\sD(\Z\flat)
$$
into the right derived functor
$$
 \boR i^!\:\sD^\co(\X\tors)\lrarrow\sD^\co(\Z\tors)
$$
from Section~\ref{derived-supports-subsecn}.
\end{prop}

\begin{proof}
 Follows from Lemma~\ref{ind-scheme-shriek-star-tensor} together
with the fact that $\rD^\bu\ot_\X\fF^\bu$ is a complex of injective
quasi-coherent torsion sheaves on $\X$ for every complex
$\fF^\bu\in\sC(\X\flat)$ (which was explained in the proof
of Theorem~\ref{ind-Noetherian-triangulated-equivalence-thm}).
\end{proof}

 The following theorem is the main result of
Section~\ref{ind-finite-type-secn}.

\begin{thm} \label{cotensor-shriek-tensor-thm}
 Let\/ $\X$ be an ind-separated ind-scheme of ind-finite type
over\/~$\kk$, and let $\rD^\bu$ be a rigid dualizing complex on\/~$\X$
(in the sense of Section~\ref{rigid-subsecn}).
 Let\/ $\Delta\:\X\rarrow\X\times_\kk\X$ be the diagonal morphism.
 Then for any two complexes of quasi-coherent torsion sheaves $\rM^\bu$
and $\rN^\bu$ on\/ $\X$ there is a natural isomorphism
\begin{equation} \label{cotensor-shriek-tensor-iso-eqn}
 \rM^\bu\oc_{\rD^\bu}\rN^\bu\,\simeq\,
 \boR\Delta^!(\rM^\bu\bt_\kk\rN^\bu)
\end{equation}
in the coderived category\/ $\sD^\co(\X\tors)$.
\end{thm}

\begin{proof}
 Notice the natural isomorphism of complexes of (flat)
pro-quasi-coherent pro-sheaves $\fF^\bu\ot^\X\fG^\bu\simeq
\Delta^*(\fF^\bu\bt_\kk\fG^\bu)$ for all complexes of (flat)
pro-quasi-coherent pro-sheaves $\fF^\bu$ and $\fG^\bu$ on~$\X$
(see Lemma~\ref{pro-tensor-as-restricted-external}).
 By Corollary~\ref{covariant-duality-external-product-cor}
(applied to the ind-schemes $\X'=\X=\X''$) and
Proposition~\ref{covariant-duality-shriek-star-prop}
(applied to the closed immersion $\Delta\:\X\rarrow\X\times_\kk\X$),
it follows that the triangulated equivalence $\sD(\X\flat)\simeq
\sD^\co(\X\tors)$ induced by $\rD^\bu$ transforms the tensor product
functor~\eqref{flat-pro-sheaves-tensor-structure-on-derived}
from Section~\ref{cotensor-product-construction-subsecn} into
the $!$\+tensor product in the right-hand side
of~\eqref{cotensor-shriek-tensor-iso-eqn}.
 But this coincides with the definition of the left-hand side
of~\eqref{cotensor-shriek-tensor-iso-eqn} given in
Section~\ref{cotensor-product-construction-subsecn}
(see~\eqref{coderived-torsion-cotensor-product}).
\end{proof}

 As a byproduct of Theorem~\ref{cotensor-shriek-tensor-thm},
we see that a rigid dualizing complex $\rD^\bu$ on $\X$ is unique up
to a natural isomorphism in the assumptions of the theorem (because
it can be recovered as the unit object of the tensor structure
on $\sD^\co(\X\tors)$ given by the right-hand side
of~\eqref{cotensor-shriek-tensor-iso-eqn}).

\Section{$\X$-Flat Pro-Quasi-Coherent Pro-Sheaves on~$\bY$}
\label{X-flat-on-Y-secn}

 In this section we consider a flat affine morphism of ind-schemes
$\pi\:\bY\rarrow\X$.
 Eventually we will assume that $\X$ is ind-semi-separated,
ind-Noetherian, and endowed with a dualizing complex~$\rD^\bu$.

\subsection{Semiderived category of torsion sheaves}
\label{semiderived-subsecn}
 Let $f\:\Y\rarrow\X$ be a morphism of ind-schemes which is
``representable by schemes'' in the sense of
Section~\ref{morphisms-of-ind-schemes-subsecn}.
 Assume that the ind-scheme $\X$ is reasonable; following
the discussion in Section~\ref{reasonable-subsecn}, the ind-scheme $\Y$
is then reasonable as well.

 Moreover, let $\X=\ilim_{\gamma\in\Gamma}X_\gamma$ be a representation
of $\X$ by an inductive system of closed immersions of reasonable
closed subschemes.
 Put $Y_\gamma=X_\gamma\times_\X\Y$;
then $Y_\gamma$ are reasonable closed subschemes in $\Y$, and
$\Y=\varinjlim_{\gamma\in\Gamma}Y_\gamma$.

 By Theorem~\ref{torsion-sheaves-abelian}, we have Grothendieck
abelian categories of quasi-coherent torsion sheaves $\Y\tors$
and $\X\tors$.
 Furthermore, there is the direct image functor
$f_*\:\Y\tors\rarrow\X\tors$ constructed in
Section~\ref{torsion-direct-images-subsecn}.
 According to Lemma~\ref{torsion-direct-inverse-adjunction}(b),
the functor~$f_*$ has a left adjoint functor
$f^*\:\X\tors\rarrow\Y\tors$.

 Furthermore, following the discussion in
Section~\ref{torsion-inverse-images-subsecn}, there is also a pair
of adjoint functors of direct and inverse images of $\Gamma$\+systems
on $\X$ and $\Y$, with the inverse image functor
$f^*\:(\X,\Gamma)\syst\rarrow(\Y,\Gamma)\syst$ left adjoint to
the direct image functor $f_*\:(\Y,\Gamma)\syst\rarrow(\X,\Gamma)\syst$.

\begin{lem} \label{representable-by-schemes-direct-image}
 Let $f\:\Y\rarrow\X$ be a morphism of reasonable ind-schemes which
is ``representable by schemes''.
 Then \par
\textup{(a)} the functor $f_*\:\Y\tors\rarrow\X\tors$ preserves
direct limits (and in particular, coproducts); \par
\textup{(b)} the functors of direct image of\/ $\Gamma$\+systems
$f_*\:(\Y,\Gamma)\syst\rarrow(\X,\Gamma)\syst$ and of quasi-coherent
torsion sheaves $f_*\:\Y\tors\rarrow\X\tors$ form a commutative
square diagram with the functors\/ $\boN\longmapsto\boN^+\:
(\Y,\Gamma)\syst\rarrow\Y\tors$ and\/ $\boM\longmapsto\boM^+\:
(\X,\Gamma)\syst\rarrow\X\tors$.
\end{lem}

\begin{proof}
 Part~(a) follows from the description of direct limits of
quasi-coherent torsion sheaves in
Section~\ref{torsion-direct-limits-subsecn} and the description
of direct images in
Section~\ref{torsion-direct-images-subsecn}, together with the fact
that the direct image functors $f_\gamma{}_*\:Y_\gamma\qcoh\rarrow
X_\gamma\qcoh$ for the morphisms of (concentrated) schemes
$f_\gamma\:Y_\gamma\rarrow X_\gamma$ preserve direct limits.
 To prove part~(b), one observes that the functors~$({-})^+$ are
constructed in terms of the functors of direct image (with respect
to the closed immersions $Y_\gamma\rarrow\Y$ and $X_\gamma\rarrow\X$)
and direct limit in $\Y\tors$ and $\X\tors$ (see
Section~\ref{Gamma-systems-subsecn}).
 Direct images obviously commute with direct images, and they preserve
direct limits by part~(a).
\end{proof}

\begin{lem} \label{affine-torsion-direct-image}
 For any affine morphism of reasonable ind-schemes $f\:\Y\rarrow\X$,
the direct image functor $f_*\:\Y\tors\rarrow\X\tors$ is exact and
faithful.
\end{lem}

\begin{proof}
 The faithfulness assertion follows immediately from the fact that
the direct image functors of quasi-coherent sheaves $f_\gamma{}_*\:
Y_\gamma\qcoh\rarrow X_\gamma\qcoh$ are faithful for affine morphisms
of schemes $f_\gamma\:Y_\gamma\rarrow X_\gamma$.
 To check exactness, notice that exactness of the functor of direct
image of $\Gamma$\+systems $f_*\:(\Y,\Gamma)\syst\rarrow
(\X,\Gamma)\syst$ follows immediately from exactness of
the functors~$f_\gamma{}_*$ for affine morphisms of schemes~$f_\gamma$.
 To deduce exactness of the functor $f_*\:\Y\tors\rarrow\X\tors$,
it remains to recall that the functors~$({-})^+$ represent the abelian
categories of quasi-coherent torsion sheaves as quotient categories of
the abelian categories of $\Gamma$\+systems by some Serre subcategories
(see the proof of Proposition~\ref{Giraud-subcategory-abstract-prop})
and use Lemma~\ref{representable-by-schemes-direct-image}(b).
\end{proof}

\begin{lem} \label{flat-torsion-sheaves-inverse-image}
 For any flat morphism of reasonable ind-schemes $f\:\Y\rarrow\X$,
the inverse image functor $f^*\:\X\tors\rarrow\Y\tors$ is exact.
\end{lem}

\begin{proof}
 The functor of inverse image of $\Gamma$\+systems
$f^*\:(\X,\Gamma)\syst\rarrow(\Y,\Gamma)\syst$ is exact,
because the functors of inverse image of quasi-coherent sheaves
$f_\gamma^*\:X_\gamma\qcoh\allowbreak\rarrow Y_\gamma\qcoh$
are exact for flat morphisms of schemes~$f_\gamma$.
 Exactness of the functor $f^*\:\X\tors\rarrow\Y\tors$ now follows
from the construction of this functor in
Section~\ref{torsion-inverse-images-subsecn}, similarly to
the proof of Lemma~\ref{affine-torsion-direct-image}.
\end{proof}

\begin{rem} \label{simplified-flat-torsion-inverse-image}
 The argument above is sufficient to prove
Lemma~\ref{flat-torsion-sheaves-inverse-image}, but in fact
one can say more.
 The functor $f^*\:\X\tors\rarrow\Y\tors$ was constructed in
a relatively complicated way, using $\Gamma$\+systems and
the functors~$({-})^+$, in Section~\ref{torsion-inverse-images-subsecn}
in order to include not necessarily flat morphisms~$f$.

 For a flat morphism of reasonable ind-schemes $f\:\Y\rarrow\X$ and
a quasi-coherent torsion sheaf $\rM$ on $\X$, one can simply put
$\rN_{(Y_\gamma)}=f_\gamma^*\rM_{(X_\gamma)}\in Y_\gamma\qcoh$ for all
$\gamma\in\Gamma$.
 Then it follows from
Lemma~\ref{reasonable-shriek-flat-star-commutation} that there are
natural isomorphisms of quasi-coherent sheaves
$k_{\gamma\delta}^!\rN_{(Y_\delta)}=k_{\gamma\delta}^!f_\delta^*
\rM_{(X_\delta)}\simeq f_\gamma^*i_{\gamma\delta}^!\rM_{(X_\delta)}
\simeq f_\gamma^*\rM_{(X_\gamma)}=\rN_{(Y_\gamma)}$ for
all $\gamma<\delta\in\Gamma$, where $i_{\gamma\delta}\:X_\gamma\rarrow
X_\delta$ and $k_{\gamma\delta}\:Y_\gamma\rarrow Y_\delta$ are
the closed immersions.
 So the collection of quasi-coherent sheaves
$(\rN_{(Y_\gamma)})_{\gamma\in\Gamma}$ defines a quasi-coherent
torsion sheaf $\rN$ on~$\Y$.
 It is easy to see that $\rN=f^*\rM$.
\end{rem}

\begin{lem} \label{torsion-pro-projection-formula-lemma}
 Let $f\:\Y\rarrow\X$ be an affine morphism of reasonable ind-schemes.
 Let $\rM$ be a quasi-coherent torsion sheaf on\/ $\X$ and\/ $\bfQ$
be a pro-quasi-coherent pro-sheaf on\/~$\Y$.
 Then there is a natural isomorphism 
$$
 \rM\ot_\X f_*\bfQ\simeq f_*(f^*\rM\ot_\Y\bfQ),
$$
of quasi-coherent torsion sheaves on\/~$\X$.
\end{lem}

\begin{proof}
 Similarly to Lemma~\ref{projection-formula}, the natural morphism
\begin{equation} \label{torsion-pro-projection-formula-morphism}
 \rM\ot_\X f_*\bfQ\lrarrow f_*(f^*\rM\ot_\Y\bfQ)
\end{equation}
is adjoint to the composition $f^*(\rM\ot_\X f_*\bfQ)\simeq
f^*\rM\ot_\Y f^*f_*\bfQ\rarrow f^*\rM\ot_\Y\bfQ$ of the isomorphism
$f^*(\rM\ot_\X f_*\bfQ)\simeq f^*\rM\ot_\Y f^*f_*\bfQ$
provided by Lemma~\ref{inverse-images-preserve-torsion-pro-tensor}
and the morphism $f^*\rM\ot_\Y f^*f_*\bfQ\rarrow f^*\rM\ot_\Y\bfQ$
induced by the adjunction morphism $f^*f_*\bfQ\rarrow\bfQ$.
 Similarly one constructs, for any $\Gamma$\+system $\boM$ on $\X$,
a natural morphism of $\Gamma$\+systems $\boM\ot_\X f_*\bfQ\rarrow
f_*(f^*\boM\ot_\Y\bfQ)$, which is an isomorphism essentially
by Lemma~\ref{projection-formula}.

 To show that~\eqref{torsion-pro-projection-formula-morphism} is
an isomorphism, one computes
\begin{multline*}
 \rM\ot_\X f_*\bfQ=(\rM|_\Gamma\ot_\X f^*\bfQ)^+\simeq
 (f_*(f^*(\rM|_\Gamma)\ot_\Y\bfQ))^+ \\ \simeq
 f_*((f^*(\rM|_\Gamma)\ot_\Y\bfQ)^+)\simeq
 f_*((f^*(\rM|_\Gamma))^+\ot_\Y\bfQ) \\ \simeq
 f_*(f^*((\rM|_\Gamma)^+)\ot_\Y\bfQ)\simeq
 f_*(f^*\rM\ot_\Y\bfQ)
\end{multline*}
using the definitions of the functors $\ot_\X\:\X\pro\times\X\tors
\rarrow\X\tors$ and $\ot_\Y\:\Y\pro\times\Y\tors\rarrow\Y\tors$,
and also Lemmas~\ref{torsion-inverse-image-commutes-with-plus}
and~\ref{representable-by-schemes-direct-image}(b).
 The point is that both the direct and inverse image functors,
as well as the tensor products in question, commute with
the functors~$({-})^+$.
\end{proof}

 Let $\pi\:\bY\rarrow\X$ be an affine morphism of reasonable
ind-schemes.
 The \emph{$\bY/\X$\+semiderived category} (or more precisely,
\emph{semicoderived category}) $\sD_\X^\si(\bY\tors)$ of quasi-coherent
torsion sheaves on $\bY$ is defined as the triangulated quotient
category of the homotopy category $\sK(\bY\tors)$ by the thick 
subcategory of all complexes of quasi-coherent torsion sheaves
$\bM^\bu$ on $\bY$ such that \emph{the complex of quasi-coherent
torsion sheaves\/ $\pi_*\bM^\bu$ on\/ $\X$ is coacyclic}
(in the sense of Section~\ref{coderived-subsecn}).

 Notice that the functor $\pi_*\:\bY\tors\rarrow\X\tors$ takes coacyclic
complexes to coacyclic complexes (by
Lemmas~\ref{representable-by-schemes-direct-image}(a)
and~\ref{affine-torsion-direct-image}), and a complex $\bM^\bu$
in $\bY\tors$ is acyclic if and only if the complex $\pi_*\bM^\bu$ is
acyclic in $\X\tors$ (also by Lemma~\ref{affine-torsion-direct-image}).
 Hence $\sD_\X^\si(\bY\tors)$ is an intermediate Verdier quotient
category between the derived category $\sD(\bY\tors)$ and
the coderived category $\sD^\co(\bY\tors)$, i.~e., there are natural
Verdier quotient functors
$$
 \sD^\co(\bY\tors)\,\twoheadrightarrow\,\sD_\X^\si(\bY\tors)
 \,\twoheadrightarrow\,\sD(\bY\tors).
$$

\begin{rem} \label{semiderived-compact-generators}
 A functor is called \emph{conservative} if it takes nonisomorphisms
to nonisomorphisms.
 A triangulated functor is conservative if and only if it takes nonzero
objects to nonzero objects.

 Let $\pi\:\bY\rarrow\X$ be an affine morphism of ind-schemes.
 Then it is clear from the definition of the semiderived category
that the functor $\pi_*\:\bY\tors\rarrow\X\tors$ induces
a well-defined, conservative triangulated functor
$\pi_*\:\sD_\X^\si(\bY\tors)\rarrow\sD^\co(\X\tors)$.
 Since the direct image functor~$\pi_*$ takes coproducts
in $\bY\tors$ to coproducts in $\X\tors$ by
Lemma~\ref{representable-by-schemes-direct-image}(a), one
can easily show using~\cite[Lemma~3.2.10]{Neem} that coproducts
exists in the semiderived category $\sD_\X^\si(\bY\tors)$
and the direct image functor takes coproducts in $\sD_\X^\si(\bY\tors)$
to coproducts in $\sD^\co(\X\tors)$.

 Let $\pi\:\bY\rarrow\X$ be a flat morphism of ind-schemes.
 Then the inverse image functor $\pi^*\:\X\tors\rarrow\bY\tors$ is exact
by Lemma~\ref{flat-torsion-sheaves-inverse-image}; being a left adjoint,
it also preserves coproducts.
 Hence the functor~$\pi^*$ takes coacyclic complexes to coacyclic
complexes.

 Now assume that the morphism $\pi\:\bY\rarrow\X$ is both flat
and affine.
 Then it follows that the functor~$\pi^*$ takes coacyclic complexes
to complexes whose direct image under~$\pi$ is coacyclic.
 Hence it induces a well-defined triangulated functor of inverse image
$\pi^*\:\sD^\co(\X\tors)\rarrow\sD_\X^\si(\bY\tors)$, which is left
adjoint to the triangulated functor of direct image $\pi_*\:
\sD_\X^\si(\bY\tors)\rarrow\sD^\co(\X\tors)$.

 Any triangulated functor left adjoint to a coproduct-preserving
triangulated functor takes compact objects to compact objects.
 Moreover, if $\sC$ and $\sD$ are triangulated categories with
coproducts, $G\:\sD\rarrow\sC$ is a conservative triangulated functor
preserving coproducts, $F\:\sC\rarrow\sD$ is a triangulated functor
left adjoint to $G$, and the triangulated category $\sC$ is compactly
generated with a set of compact generators $\sS\subset\sC$, then
the category $\sD$ is compactly generated, too, with the set of objects
$\{F(S)\in \sD\mid S\in\sS\}$ being a set of compact generators
for~$\sD$.

 Assume additionally that $\X$ is an ind-Noetherian ind-scheme.
 Then, by Remark~\ref{coderived-compact-generators}, the coderived
category $\sD^\co(\X\tors)$ is compactly generated, with all
the (representatives of isomorphism classes of) coherent torsion sheaves
$\rM$ on $\X$ forming a set of compact generators
for $\sD^\co(\X\tors)$.
 Therefore, the semiderived category $\sD_\X^\si(\bY\tors)$ is also
compactly generated.
 The inverse images $\pi^*(\rM)\in\bY\tors\subset\sD_\X^\si(\bY\tors)$
of coherent torsion sheaves on $\X$ form a set of compact generators
of $\sD_\X^\si(\bY\tors)$.
\end{rem}

 Let us say that a quasi-coherent torsion sheaf $\bK$ on $\bY$ is
\emph{$\X$\+injective} if the quasi-coherent torsion sheaf $\pi_*\bK$
on $\X$ is injective.
 We will denote the full subcategory of $\X$\+injective quasi-coherent
torsion sheaves on $\bY$ by $\bY\tors_{\X\dinj}\subset\bY\tors$.

 In particular, given an affine morphism of schemes $f\:\bnY\rarrow X$,
a quasi-coherent sheaf $\bcK$ on $\bnY$ is said to be
\emph{$X$\+injective} if the quasi-coherent sheaf $f_*\bcK$ on $X$ is
injective.
 The full subcategory of $X$\+injective quasi-coherent sheaves is
denoted by $\bnY\qcoh_{X\dinj}\subset\bnY\qcoh$.

\begin{lem} \label{injective-over-base}
\textup{(a)} For any affine morphism of reasonable ind-schemes\/
$\pi\:\bY\rarrow\X$, the full subcategory\/ $\bY\tors_{\X\dinj}$
is closed under extensions and cokernels of monomorphisms
in\/ $\bY\tors$. \par
\textup{(b)} If $\pi\:\bY\rarrow\X$ is a flat affine morphism of
reasonable ind-schemes, then any injective quasi-coherent torsion
sheaf on\/ $\bY$ is\/ $\X$\+injective, that is\/ $\bY\tors_\inj
\subset\bY\tors_{\X\dinj}$.
\end{lem}

\begin{proof}
 Part~(a) holds because the functor $\pi_*\:\bY\tors\rarrow\X\tors$
is exact and the full subcategory $\X\tors_\inj\subset\X\tors$ is
closed under extensions and cokernels of monos.
 Part~(b) claims that the functor~$\pi_*$ preserves injectives;
this is so because $\pi_*$~is right adjoint to the functor~$\pi^*$,
which is exact by Lemma~\ref{flat-torsion-sheaves-inverse-image}.
\end{proof}

 The assertions of Lemma~\ref{injective-over-base} can be expressed
by saying that $\bY\tors_{\X\dinj}$ is a \emph{coresolving
subcategory} in $\bY\tors$.
 In particular, being a full subcategory closed under extensions,
$\bY\tors_{\X\dinj}$ inherits an exact category structure from
the abelian category $\bY\tors$.
 So one can form the derived category $\sD(\bY\tors_{\X\dinj})$
(cf.~\cite[paragraphs preceding Theorem~5.2]{Pfp} for a discussion).

\begin{lem} \label{injective-over-base-acyclicity-criterion}
 Let\/ $\pi\:\bY\rarrow\X$ be an affine morphism of reasonable
ind-schemes.
 Then a complex $\bK^\bu\in\sC(\bY\tors_{\X\dinj})$ is acyclic in\/
$\bY\tors_{\X\dinj}$ if and only if the complex\/ $\pi_*\bK^\bu$ in
$\X\tors_\inj$ is contractible, and if and only if the complex\/
$\pi_*\bK^\bu$ is coacyclic in\/ $\X\tors$.
\end{lem}

\begin{proof}
 The first assertion holds because the functor $\pi_*\:\bY\tors\rarrow
\X\tors$ is exact and faithful, and a complex of injective objects in
$\X\tors$ is contractible if and only if its cocycle objects are
injective.
 Furthermore, a complex of injectives is contractible if and only if
it is coacyclic
(by Proposition~\ref{coderived-and-homotopy-of-injectives}(a)).
\end{proof}

 The next proposition is a generalization of
Corollary~\ref{ind-Noetherian-coderived-cor}.
 It should be also compared to~\cite[Theorems~5.1(a) and~5.2(a)]{Pfp}.

\begin{prop} \label{over-ind-Noetherian-semiderived-prop}
 Let\/ $\X$ be an ind-Noetherian ind-scheme and\/ $\pi\:\bY\rarrow\X$
be a flat affine morphism of ind-schemes.
 Then the inclusion of exact categories\/ $\bY\tors_{\X\dinj}\rarrow
\bY\tors$ induces an equivalence of triangulated categories\/
$\sD(\bY\tors_{\X\dinj})\simeq\sD_\X^\si(\bY\tors)$.
\end{prop}

\begin{proof}
 It follows from Lemma~\ref{injective-over-base-acyclicity-criterion}
that the triangulated functor $\sD(\bY\tors_{\X\dinj})
\rarrow\sD_\X^\si(\bY\tors)$ induced by the inclusion
$\bY\tors_{\X\dinj}\rarrow\bY\tors$ is well-defined.
 Moreover, by~\cite[Lemma~1.6(b)]{Pkoszul}, in order to show that
this triangulated functor is an equivalence it suffices to check
that for any complex $\bM^\bu\in\sC(\bY\tors)$ there exists a complex
$\bK^\bu\in\sC(\bY\tors_{\X\dinj})$ together with a morphism of
complexes $\bM^\bu\rarrow\bK^\bu$ whose cone becomes coacyclic after
applying~$\pi_*$.

 In fact, one can even make the cone of $\bM^\bu\rarrow\bK^\bu$
coacyclic in $\bY\tors$.
 One only needs to observe that, since the direct image functor~$\pi_*$
preserves coproducts and the class of injective objects in $\X\tors$
is closed under coproducts (as $\X$ is ind-Noetherian;
see Lemma~\ref{locally-Noetherian-coproducts-injective}
and Proposition~\ref{ind-Noetherian-torsion-locally-Noetherian}),
the full subcategory $\bY\tors_{\X\dinj}\subset\X\tors$ is closed
under coproducts.

 Besides, there are enough injective objects in a Grothendieck
category $\bY\tors$ and all of them belong to $\bY\tors_{\X\dinj}$;
so any complex $\bM^\bu\in\sC(\bY\tors)$ admits a termwise monic
morphism into a complex $\bJ^{0,\bu}\in\sC(\bY\tors_{\X\dinj})$.
 The cokernel $\bJ^{0,\bu}/\bM^\bu$, in turn, can be embedded into
a complex $\bJ^{1,\bu}\in\sC(\bY\tors_{\X\dinj})$, etc.
 Totalizing the bicomplex $\bJ^{\bu,\bu}$ by taking countable
coproducts along the diagonals, one produces the desired complex
$\bK^\bu$ together with a morphism $\bM^\bu\rarrow\bK^\bu$, whose
cone is coacyclic by~\cite[Lemma~2.1]{Psemi}.
 We refer to~\cite[proof of Theorem~3.7]{Pkoszul}
or~\cite[proof of Proposition~A.3.1(b)]{Pcosh} for further details
of this argument.
\end{proof}

\subsection{Pro-sheaves flat over the base}
\label{pro-flat-over-base-subsecn}
 Let $f\:\bnY\rarrow X$ be an affine morphism of schemes.
 A quasi-coherent sheaf $\bcG$ on $\bnY$ is said to be \emph{$X$\+flat}
(or \emph{flat over~$X$}) if the quasi-coherent sheaf $f_*\bcG$ on $X$
is flat.
 We will denote the full subcategory of $X$\+flat quasi-coherent sheaves
on $\bnY$ by $\bnY_X\flat\subset\bnY\qcoh$.

\begin{lem} \label{flat-over-scheme-tensor-product}
 Let $f\:\bnY\rarrow X$ be an affine morphism of schemes.
 Then, for any flat quasi-coherent sheaf\/ $\bcF$ on\/ $\bnY$ and 
any $X$\+flat quasi-coherent sheaf\/ $\bcG$ on\/ $\bnY$,
the quasi-coherent sheaf\/ $\bcF\ot_{\cO_\bnY}\bcG$ on\/ $\bnY$
is $X$\+flat.
\end{lem}

\begin{proof}
 The assertion is local in $X$, so it reduces to the case of affine
schemes, for which it means the following.
 Let $R\rarrow\bnS$ be a homomorphism of commutative rings, let
$F$ be a flat $\bnS$\+module, and let $G$ be an $R$\+flat
$\bnS$\+module (i.~e., an $\bnS$\+module whose underlying $R$\+module
is flat).
 Then $F\ot_\bnS G$ is an $R$\+flat $\bnS$\+module.
\end{proof}

 Clearly, the full subcategory $\bnY_X\flat$ is closed under extensions
in the abelian category $\bnY\qcoh$ (since the functor~$f_*$ is exact
and the full subcategory $X\flat\subset X\qcoh$ is closed under
extensions); so it inherits an exact category structure.
 The full subcategory $\bnY_X\flat$ is closed under direct limits in
$\bnY\qcoh$ (because the functor~$f_*$ preserves direct limits).
 When the morphism~$f$ is (affine and) flat, any flat quasi-coherent
sheaf on $\bnY$ is $X$\+flat, as the functor~$f_*$ takes flat
quasi-coherent sheaves on $\bnY$ to flat quasi-coherent sheaves on~$X$;
so $\bnY\flat\subset\bnY_X\flat\subset\bnY\qcoh$.

 For a flat affine morphism of schemes $f\:\bnY\rarrow X$,
the equivalence of categories from
Lemma~\ref{qcoh-modules-via-qcoh-algebras-described} restricts to
an equivalence between the category $\bnY_X\flat$ of $X$\+flat
quasi-coherent sheaves on $\bnY$ and the category of module objects
over the algebra object $f_*\cO_\bnY$ in the tensor category $X\flat$.
 This is an equivalence of exact categories (with the exact structure
on the category of module objects over $f_*\cO_\bnY$ in $X\flat$
coming from the exact structure on $X\flat$).

\begin{lem} \label{qcoh-flat-over-base}
 Let $f\:\bnY\rarrow X$ be an affine morphism and $h\:Z\rarrow X$ be
a morphism of schemes.
 Consider the pullback diagram
$$
\xymatrix{
 Z\times_X\bnY \ar[r]^-k \ar[d]^-g & \bnY \ar[d]^-f \\
 Z \ar[r]^-h & X
}
$$
and put\/ $\bnW=Z\times_X\bnY$.
 Then \par
\textup{(a)} for any $X$\+flat quasi-coherent sheaf\/ $\bcG$ on\/ $\bnY$,
the quasi-coherent sheaf\/ $k^*\bcG$ on\/  $\bnW$ is $Z$\+flat; \par
\textup{(b)} the functor $k^*\:\bnY_X\flat\rarrow\bnW_Z\flat$ takes
short exact sequences to short exact sequences; \par
\textup{(c)} assuming that $h$~is a flat affine morphism, for any
$Z$\+flat quasi-coherent sheaf\/ $\bcH$ on\/ $\bnW$,
the quasi-coherent sheaf\/ $k_*\bcH$ on\/ $\bnY$ is $X$\+flat.
\end{lem}

\begin{proof}
 Parts~(a\+-b): both the assertions are local in $X$ and $Z$, so they
reduce to the case of affine schemes, for which they mean the following.
 Let $R\rarrow\bnS$ and $R\rarrow T$ be two morphisms of commutative
rings.
 Let $\bnN$ be an $\bnS$\+module which is flat over~$R$;
then the $(T\ot_R\bnS)$\+module $(T\ot_R\bnS)\ot_\bnS\bnN$ is flat
over~$T$.
 Let $0\rarrow\bnL\rarrow\bnM\rarrow\bnN\rarrow0$ be a short exact
sequence of $R$\+flat $\bnS$\+modules; then $0\rarrow
(T\ot_R\bnS)\ot_\bnS\bnL\rarrow(T\ot_R\bnS)\ot_\bnS\bnM\rarrow
(T\ot_R\bnS)\ot_\bnS\bnN\rarrow0$ is a short exact sequence of
$T$\+flat $(T\ot_R\bnS)$\+modules.
 Alternatively, part~(a) follows from
Lemma~\ref{affine-flat-base-change}(a) and the fact that the functor
$h^*\:X\qcoh\rarrow Z\qcoh$ takes flat quasi-coherent sheaves on $X$
to flat quasi-coherent sheaves on~$Z$.
 Part~(c): the assertion follows from the fact that the functor
$h_*\:Z\qcoh\rarrow X\qcoh$ takes flat quasi-coherent sheaves on $Z$
to flat quasi-coherent sheaves on~$X$ (for a flat affine morphism~$h$).
\end{proof}

 Let $\pi\:\bY\rarrow\X$ be an affine morphism of ind-schemes.
 We refer to Section~\ref{pro-sheaves-inverse-direct-subsecn} for
the construction of the functors of inverse and direct image of
pro-quasi-coherent pro-sheaves, $\pi^*\:\X\pro\rarrow\bY\pro$ and
$\pi_*\:\bY\pro\rarrow\X\pro$.

 A pro-quasi-coherent pro-sheaf $\bfG$ on $\bY$ is said to be
\emph{$\X$\+flat} if the pro-quasi-coherent pro-sheaf $\pi_*\bfG$
on~$\X$ is flat in the sense of Section~\ref{flat-pro-sheaves-subsecn},
that is $\pi_*\bfG\in\X\flat\subset\X\pro$.
 Explicitly, this means that, for every closed subscheme $Z\subset\X$,
denoting by $\pi_Z$ the related morphism $\bnW=Z\times_\X\bY\rarrow Z$,
the quasi-coherent sheaf $\bfG^{(\bnW)}$ on $\bnW$ is $Z$\+flat.
 Given a representatation $\X=\ilim_{\gamma\in\Gamma}X_\gamma$ of
the ind-scheme $\X$ by an inductive system of closed immersions
$(X_\gamma)_{\gamma\in\Gamma}$, it suffices to check the latter
condition for the closed subschemes $Z=X_\gamma$ (cf.\
Lemma~\ref{qcoh-flat-over-base}(a)).
 We will denote the full subcategory of $\X$\+flat pro-quasi-coherent
pro-sheaves by $\bY_\X\flat\subset\bY\pro$.

\begin{lem} \label{flat-pro-over-ind-scheme-tensor-product}
 Let\/ $\pi\:\bY\rarrow\X$ be an affine morphism of ind-schemes.
 Then, for any flat pro-quasi-coherent pro-sheaf\/ $\bfF$ on\/ $\bY$
and any\/ $\X$\+flat pro-quasi-coherent pro-sheaf\/ $\bfG$ on\/ $\bY$,
the pro-quasi-coherent pro-sheaf\/ $\bfF\ot^\bY\bfG$ on\/ $\bY$ is\/
$\X$\+flat.
\end{lem}

\begin{proof}
 Follows from Lemma~\ref{flat-over-scheme-tensor-product}
and the construction of the tensor product functor $\ot^\bY$ in
Section~\ref{pro-sheaves-subsecn}.
\end{proof}

 The full subcategory $\bY_\X\flat$ is closed under direct limits
(in particular, coproducts) in $\bY\pro$.
 When the morphism~$\pi$ is affine and flat, the functor $\pi_*\:\bY\pro
\rarrow\X\pro$ takes flat pro-quasi-coherent pro-sheaves on $\bY$ to
flat pro-quasi-coherent pro-sheaves on $\X$ (see
Section~\ref{flat-pro-sheaves-subsecn}); so $\bY\flat\subset\bY_\X\flat
\subset\bY\pro$.

 Let $0\rarrow\bfF\rarrow\bfG\rarrow\bfH\rarrow0$ be a short sequence
of $\X$\+flat pro-quasi-coherent pro-sheaves on~$\bY$.
 We say that this is an (\emph{admissible}) short \emph{exact} sequence
in $\bY_\X\flat$ if, for every closed subscheme $Z\subset\X$ and
the related closed subscheme $\bnW=Z\times_\X\nobreak\bY\subset\bY$,
the sequence of quasi-coherent sheaves $0\rarrow\bfF^{(\bnW)}\rarrow
\bfG^{(\bnW)}\rarrow\bfH^{(\bnW)}\rarrow0$ is exact in the abelian
category $\bnW\qcoh$.
 It suffices to check this condition for the closed subschemes
$Z=X_\gamma\subset X$, \,$\gamma\in\Gamma$, belonging any fixed
representation $\X=\ilim_{\gamma\in\Gamma}X_\gamma$ of
the ind-scheme $\X$ by an inductive system of closed immersions
of schemes.

\begin{prop}
 For any affine morphism of ind-schemes\/ $\pi\:\bY\rarrow\X$,
the category\/ $\bY_\X\flat$ of\/ $\X$\+flat pro-quasi-coherent
pro-sheaves on\/ $\bY$, endowed with the class of admissible short
exact sequences as defined above, is an exact category.
\end{prop}

\begin{proof}
 The argument is similar to the proof of
Proposition~\ref{flat-pro-sheaves-exact-category} and based on
Lemma~\ref{qcoh-flat-over-base}(a\+-b).
 It is helpful to notice that a pro-quasi-coherent pro-sheaf $\bfP$
on $\bY$ can be defined as a collection of quasi-coherent sheaves
$\bfP^{(\bnW)}\in\bnW\qcoh$ on the closed subschemes $\bnW\subset\bY$
of the form $\bnW=Z\times_\X\bY$ (where $Z$ ranges over the closed
subschemes in~$X$), endowed with the (iso)morphisms and satisfying
the compatibilies listed in items~(i\+-iv) of
Section~\ref{pro-sheaves-subsecn}.
\end{proof}

 Given an affine morphism of ind-schemes $\pi\:\bY\rarrow\X$,
where the ind-scheme $\X=\ilim_{\gamma\in\Gamma}X_\gamma$ is
represented by an inductive system of closed immersions
$(X_\gamma)_{(\gamma\in\Gamma)}$, we put
$\bnY_\gamma=X_\gamma\times_X\bY$.
 Then $\bY=\ilim_{\gamma\in\Gamma}\bnY_\gamma$ is a representation
of $\bY$ by an inductive system of closed immersions.

\begin{lem} \label{flat-over-base-component-acyclicity-criterion}
 A complex of\/ $\X$\+flat pro-quasi-coherent pro-sheaves\/ $\bfG^\bu$
on\/ $\bY$ is acyclic (as a complex in\/ $\bY_\X\flat$) if and only if,
for every\/ $\gamma\in\Gamma$, the complex of $X_\gamma$\+flat
quasi-coherent sheaves\/ $\bfG^\bu{}^{(X_\gamma)}$ on $\bnY_\gamma$
is acyclic (as a complex in $(\bnY_\gamma)_{X_\gamma}\flat$).
\end{lem}

\begin{proof}
 This is a generalization of
Lemma~\ref{flat-pro-sheaves-complex-acyclicity-criterion}, provable
in the same way.
 The ``only if'' assertion is obvious.
 To prove the ``if'', one needs to observe that the functor assigning to
an acyclic complex of quasi-coherent sheaves its sheaves of cocycles
can be expressed as a kind of cokernel in the category of quasi-coherent
sheaves, and as such, commutes with inverse images (whenever
the latter preserve acyclicity of a given complex).
 Then one needs also to use Lemma~\ref{qcoh-flat-over-base}(a\+-b).
\end{proof}

 For any affine morphism of ind-schemes\/ $\pi\:\bY\rarrow\X$,
the direct image functor $\pi_*\:\bY\pro\rarrow\X\pro$ takes
the full subcategory $\bY_\X\flat\subset\bY\pro$ into the full
subcategory $\X\flat\subset\X\pro$.
 The resulting direct image functor
$$
 \pi_*\:\bY_\X\flat\lrarrow\X\flat
$$
is an exact functor between exact categories (i.~e, it takes short
exact sequences to short exact sequences).
 This follows immediately from the observation that the direct image
functor $\pi_Z{}_*\:\bnW_Z\flat\rarrow Z\flat$ for the affine morphism
$\pi_Z\:\bnW\rarrow Z$ is exact (which is true because the functor
$\pi_Z{}_*\:\bnW\qcoh\rarrow Z\qcoh$ is exact).

 For a flat affine morphism of ind-schemes\/ $\pi\:\bY\rarrow\X$,
the inclusion functor $\bY\flat\rarrow\bY_\X\flat$ is exact; in fact,
a short sequence in $\bY\flat$ is exact if and only if it is exact
in $\bY_\X\flat$.
 Furthermore, there is an exact inverse image functor
$$
 \pi^*\:\X\flat\lrarrow\bY\flat\,\subset\,\bY_\X\flat,
$$
which is left adoint to~$\pi_*$.

\begin{lem} \label{flat-over-base-direct-image-acyclicity-criterion}
 Let\/ $\pi\:\bY\rarrow\X$ be an affine morphism of ind-schemes.
 Then a complex\/ $\bfG^\bu\in\sC(\bY_\X\flat)$ is acyclic in\/
$\bY_\X\flat$ if and only if the complex\/ $\pi_*\bfG^\bu$ is acyclic
in\/ $\X\flat$.
\end{lem}

\begin{proof}
 Follows from Lemmas~\ref{flat-pro-sheaves-complex-acyclicity-criterion}
and~\ref{flat-over-base-component-acyclicity-criterion} together with
the fact that, for an affine morphism of schemes $f\:\bnY\rarrow X$,
a complex $\G^\bu$ in $\bnY_X\flat$ is acyclic if and only if
the complex $f_*\G^\bu$ is acyclic in $X\flat$.
 The latter assertion holds because the functor $f_*\:\bnY\qcoh\rarrow
X\qcoh$ is exact and faithful.
\end{proof}

 For a flat affine morphism of ind-schemes $\pi\:\bY\rarrow\X$,
the equivalence of categories from
Proposition~\ref{pro-torsion-modules-via-pro-algebras-described}(a)
restricts to an equivalence between the category $\bY_\X\flat$ of
$\X$\+flat pro-quasi-coherent pro-sheaves on $\bY$ and the category
of module objects over the algebra object $\pi_*\fO_\bY$ in
the tensor category $\X\flat$.
 This is an equivalence of exact categories (with the exact structure
on the category of module objects over $\pi_*\fO_\bY$ in $\X\flat$
coming from the exact structure on $\X\flat$).

\subsection{The triangulated equivalence}
\label{relative-triangulated-equivalence-subsecn}
 The following theorem, generalizing
Theorem~\ref{ind-Noetherian-triangulated-equivalence-thm}, is the main
result of Section~\ref{X-flat-on-Y-secn}.

\begin{thm} \label{relative-triangulated-equivalence-thm}
 Let\/ $\X$ be an ind-semi-separated ind-Noetherian ind-scheme with
a dualizing complex\/ $\rD^\bu$, and let\/ $\pi\:\bY\rarrow\X$ be
a flat affine morphism of ind-schemes.
 Then there is a natural equivalence of triangulated categories\/
$\sD_\X^\si(\bY\tors)\simeq\sD(\bY_\X\flat)$, provided by mutually
inverse triangulated functors\/ $\fHom_{\bY\qc}(\pi^*\rD^\bu,{-})\:
\allowbreak\sD(\bY\tors_{\X\dinj})\rarrow\sD(\bY_\X\flat)$ and\/
$\pi^*\rD^\bu\ot_\bY{-}\,\:\sD(\bY_\X\flat)\rarrow
\sD(\bY\tors_{\X\dinj})$. \emergencystretch=1em
\end{thm}

 The notation $\fHom_{\bY\qc}({-},{-})$ will be explained below, and
the proof of Theorem~\ref{relative-triangulated-equivalence-thm}
will be given below in this
Section~\ref{relative-triangulated-equivalence-subsecn}.
 The next two lemmas are not needed for this proof and are included here
mostly for completeness of the exposition and to help the reader
feel more comfortable (they will be useful, however, in
Section~\ref{semiderived-equivalence-and-base-change-subsecn}).
 The subsequent three lemmas play a more important role, and among
them Lemma~\ref{relative-star-qc-Hom-shriek-for-complexes} is essential.

\begin{lem} \label{relative-shriek-star-tensor}
 Let $f\:\bnY\rarrow X$ be a flat affine morphism of schemes and
$Z\subset X$ be a reasonable closed subscheme with the closed
immersion morphism $i\:Z\rarrow X$.
 Consider the pullback diagram
$$
\xymatrix{
 Z\times_X\bnY \ar[r]^-k \ar[d]^-g & \bnY \ar[d]^-f \\
 Z \ar[r]^-i & X
}
$$
and put\/ $\bnW=Z\times_X\bnY$.
 Let $\M$ be a quasi-coherent sheaf on $X$ and\/ $\bcG$ be
an $X$\+flat quasi-coherent sheaf on\/~$\bnY$; put\/
$\bcN=f^*\M\in\bnY\qcoh$.
 Then the natural morphism of quasi-coherent sheaves on\/~$\bnW$
$$
 k^*\bcG\ot_{\cO_\bnW}k^!\bcN\lrarrow k^!(\bcG\ot_{\cO_\bnY}\bcN)
$$
from Lemma~\ref{shriek-star-tensor} is an isomorphism.
\end{lem}

\begin{proof}
 The assertion is local in $X$, so it reduces to the case of affine
schemes, for which it means the following.
 Let $R\rarrow\bnS$ be a homomorphism of commutative rings such that
$\bnS$ is a flat $R$\+module and $R\rarrow T$ be a surjective
homomorphism of commutative rings with a finitely generated
kernel ideal.
 Let $M$ be an $R$\+module and $\bnG$ be an $R$\+flat $\bnS$\+module.
 Then the natural homomorphism of $(T\ot_R\bnS)$\+modules
\begin{multline*}
 ((T\ot_R\bnS)\ot_\bnS\bnG)\ot_{(T\ot_R\bnS)}
 \Hom_\bnS (T\ot_R\bnS,\>\bnS\ot_R M) \\
 = \bnG\ot_\bnS\Hom_\bnS (T\ot_R\bnS,\>\bnS\ot_R M)
 = \bnG\ot_\bnS\Hom_R(T,\>\bnS\ot_R M) \\
 \lrarrow\Hom_R(T,\>\bnG\ot_RM)
 = \Hom_\bnS(T\ot_R\bnS,\>\bnG\ot_\bnS(\bnS\ot_RM))
\end{multline*}
is an isomorphism, because both the maps
\begin{multline*}
 \bnG\ot_\bnS\Hom_R(T,\>\bnS\ot_R M) \llarrow
 \bnG\ot_\bnS(\bnS\ot_R\Hom_R(T,M)) \\
 =\bnG\ot_R\Hom_R(T,M)\lrarrow \Hom_R(T,\>\bnG\ot_RM)
\end{multline*}
are isomorphisms (cf.\
Lemma~\ref{reasonable-shriek-flat-star-commutation} for
the first arrow).
\end{proof}

\begin{lem} \label{relative-scheme-ind-scheme-shriek-star-tensor}
 Let\/ $\pi\:\bY\rarrow\X$ be a flat affine morphism of reasonable
ind-schemes and $Z\subset\X$ be a reasonable closed subscheme with
the closed immersion morphism $i\:Z\rarrow\X$.
 Put\/ $\bnW=Z\times_\X\bY$ and denote by $k\:\bnW\rarrow\bY$
the natural closed immersion.
 Let $\rM$ be a quasi-coherent torsion sheaf on\/ $\X$
and\/ $\bfG$ be an\/ $\X$\+flat pro-quasi-coherent pro-sheaf
on\/~$\bY$; put\/ $\bN=\pi^*\rM\in\bY\tors$.
 Then there is a natural isomorphism $k^!(\bfG\ot_\bY\bN)\simeq
k^*\bfG\ot_{\cO_\bnW}k^!\bN=\bfG^{(\bnW)}\ot_{\cO_\bnW}\bN_{(\bnW)}$
in\/ $\bnW\qcoh$.
\end{lem}

\begin{proof}
 The argument is similar to the proof of
Proposition~\ref{flat-torsion-tensor-prop}
and uses Lemma~\ref{relative-shriek-star-tensor}.
\end{proof}

\begin{lem} \label{relative-star-qc-Hom-shriek}
 Let $f\:\bnY\rarrow X$ be a flat affine morphism of semi-separated
schemes and $Z\subset X$ be a reasonable closed subscheme with
the closed immersion morphism $i\:Z\rarrow X$.
 Consider the pullback diagram
\begin{equation} \label{bold-pullback-diagram}
\begin{gathered}
\xymatrix{
 Z\times_X\bnY \ar[r]^-k \ar[d]^-g & \bnY \ar[d]^-f \\
 Z \ar[r]^-i & X
}
\end{gathered}
\end{equation}
and put\/ $\bnW=Z\times_X\bnY$.
 Let $\M$ be a quasi-coherent sheaf on $X$ and\/ $\bcK$ be
an $X$\+injective quasi-coherent sheaf on\/~$\bnY$; put\/
$\bcN=f^*\M\in\bnY\qcoh$.
 Then the natural morphism of quasi-coherent sheaves on\/~$\bnW$
$$
 k^*\cHom_{\bnY\qc}(\bcN,\bcK)\lrarrow\cHom_{\bnW\qc}(k^!\bcN,k^!\bcK)
$$
from Lemma~\ref{star-qc-Hom-shriek-lemma} is an isomorphism.
\end{lem}

\begin{proof}
 Since the direct image functor $g_*\:\bnW\qcoh\rarrow Z\qcoh$ is exact
and faithful, it suffices to check that the morphism in question becomes
an isomorphism after applying~$g_*$.
 We have
$$
 g_*k^*\cHom_{\bnY\qc}(f^*\M,\bcK)\simeq
 i^*f_*\cHom_{\bnY\qc}(f^*\M,\bcK)\simeq
 i^*\cHom_{X\qc}(\M,f_*\bcK)
$$
by Lemmas~\ref{affine-flat-base-change}(a)
and~\ref{qcoh-internal-Hom-projection}.
 On the other hand,
\begin{multline*}
 g_*\cHom_{\bnW\qc}(k^!f^*\M,k^!\bcK)\simeq
 g_*\cHom_{\bnW\qc}(g^*i^!\M,k^!\bcK) \\ \simeq
 \cHom_{Z\qc}(i^!\M,g_*k^!\bcK)\simeq
 \cHom_{Z\qc}(i^!\M,i^!f_*\bcK)
\end{multline*}
by Lemmas~\ref{reasonable-shriek-flat-star-commutation},
\ref{qcoh-internal-Hom-projection},
and~\ref{reasonable-base-change}(a).
 The assertion now follows from Lemma~\ref{star-qc-Hom-shriek-lemma},
as the quasi-coherent sheaf $f_*\bcK$ on $X$ is injective
by assumption.
\end{proof}

\begin{lem} \label{relative-star-product}
 Let $f\:\bnY\rarrow X$ be an affine morphism of semi-separated
schemes and $Z\subset X$ be a reasonable closed subscheme with
the closed immersion morphism $i\:Z\rarrow X$.
 Consider the pullback diagram~\eqref{bold-pullback-diagram}
with\/ $\bnW=Z\times_X\bnY$.
 Let $(\bcQ_\xi)_{\xi\in\Xi}$ be a family of quasi-coherent sheaves
on\/~$\bnY$ such that the quasi-coherent sheaves $f_*\bcQ$ on $X$
are flat cotorsion.
 Then the natural morphism of quasi-coherent sheaves on\/~$\bnW$
$$
 k^*\prod\nolimits_{\xi\in\Xi}\bcQ_\xi
 \lrarrow\prod\nolimits_{\xi\in\Xi}k^*\bcQ_\xi
$$
from Lemma~\ref{star-product-lemma} is an isomorphism.
\end{lem}

\begin{proof}
 As in the previous proof, it suffices to show that the morphism in
question becomes an isomorphism after applying~$g_*$.
 Since the direct image functors~$f_*$ and~$g_*$, being right adjoints,
preserve infinite products of quasi-coherent sheaves, and
$g_*k^*\simeq i^*f_*$ by Lemma~\ref{affine-flat-base-change}(a),
the desired assertion follows from Lemma~\ref{star-product-lemma}.
\end{proof}

\begin{lem} \label{relative-star-qc-Hom-shriek-for-complexes}
 Let $f\:\bnY\rarrow X$ be an affine morphism of semi-separated
schemes and $Z\subset X$ be a reasonable closed subscheme with
the closed immersion morphism $i\:Z\rarrow X$.
 Consider the pullback diagram~\eqref{bold-pullback-diagram}
with\/ $\bnW=Z\times_X\bnY$.
 Let $\M^\bu$ be a complex of quasi-coherent sheaves on $X$ and\/
$\bcK^\bu$ be a complex of $X$\+injective quasi-coherent sheaves
on\/~$\bnY$; put\/ $\bcN^\bu=f^*\M^\bu\in\sC(\bnY\qcoh)$.
 Then the natural morphism of complexes of quasi-coherent sheaves
on\/~$\bnW$
$$
 k^*\cHom_{\bnY\qc}(\bcN^\bu,\bcK^\bu)\lrarrow
 \cHom_{\bnW\qc}(k^!\bcN^\bu,k^!\bcK^\bu)
$$
from Lemma~\ref{star-qc-Hom-shriek-for-complexes} is an isomorphism.
\end{lem}

\begin{proof}
 This is provable similarly to
Lemma~\ref{relative-star-qc-Hom-shriek}, using
Lemma~\ref{qcoh-internal-Hom-of-complexes-projection} in
order to reduce the assertion to
Lemma~\ref{star-qc-Hom-shriek-for-complexes}.
 Alternatively, a direct proof, similar to the proof of
Lemma~\ref{star-qc-Hom-shriek-for-complexes} and based on the result
of Lemma~\ref{relative-star-qc-Hom-shriek}, is possible;
in particular, assuming additionally that $X$ is Noetherian
and $\M^\bu$ is a complex of injective quas-coherent sheaves on $X$,
one can use Lemma~\ref{relative-star-product}.
 (For this purpose, one observes that
$f_*\cHom_{\bnY\qc}(f^*\M^p,\bcK^q)\simeq\cHom_{X\qc}(\M^p,f_*\bcK^q)$
are flat cotorsion quasi-coherent sheaves on $X$
by Lemmas~\ref{hom-tensor-flats-injectives}(d) and
\ref{flat-and-cotorsion-lemma}(a).)
\end{proof}

 Let $\pi\:\bY\rarrow\X$ be a flat affine morphism of reasonable
ind-semi-separated ind-schemes.
 Let $\rE^\bu\in\sC(\X\tors)$ be a complex of quasi-coherent torsion
sheaves on $\X$, and let $\bK^\bu\in\sC(\bY\tors_{\X\dinj})$ be
a complex of $\X$\+injective quasi-coherent torsion sheaves on~$\bY$.
 Then the complex $\fHom_{\bY\qc}(\pi^*\rE^\bu,\bK^\bu)\in
\sC(\bY\pro)$ of pro-quasi-coherent pro-sheaves on $\bY$ is constructed
as follows.

 For every reasonable closed subscheme $Z\subset\X$, consider
the pullback diagram
$$
\xymatrix{
 Z\times_\X\bY \ar[r]^-k \ar[d]^-{\pi_Z} & \bY \ar[d]^-\pi \\
 Z \ar[r]^-i & \X
}
$$
so $\bnW=Z\times_\X\bY$ is a reasonable closed subscheme in~$\bY$.
 Put
\begin{multline*}
\fHom_{\bY\qc}(\pi^*\rE^\bu,\bK^\bu)^{(\bnW)}=
\cHom_{\bnW\qc}(k^!\pi^*\rE^\bu,k^!\bK^\bu) \\ \simeq
\cHom_{\bnW\qc}(\pi_Z^*i^!\rE^\bu,k^!\bK^\bu)=
\cHom_{\bnW\qc}(\pi_Z^*\rE^\bu_{(Z)},\bK^\bu_{(\bnW)}),
\end{multline*}
where the middle isomorphism holds by
Remark~\ref{simplified-flat-torsion-inverse-image}.
 According to Lemma~\ref{relative-star-qc-Hom-shriek-for-complexes},
for every pair of reasonable closed subschemes $Z'\subset Z''\subset\X$
and the related closed subschemes $\bnW'\subset\bnW''\subset\bY$, \
$\bnW^{(s)}=Z^{(s)}\times_\X\bY$ with the natural closed immersion
$k_{\bnW'\bnW''}\:\bnW'\rarrow\bnW''$, we have
$$
 \fHom_{\bY\qc}(\pi^*\rE^\bu,\bK^\bu)^{(\bnW')}\,\simeq\,
 k_{\bnW'\bnW''}^*\fHom_{\bY\qc}(\pi^*\rE^\bu,\bK^\bu)^{(\bnW'')}
$$
as required for the construction of a pro-quasi-coherent pro-sheaf.
 This explains the meaning of the notation in
Theorem~\ref{relative-triangulated-equivalence-thm}.

\begin{lem} \label{internal-Hom-of-torsion-projection-formula}
 Let\/ $\pi\:\bY\rarrow\X$ be a flat affine morphism of reasonable
ind-semi-separated ind-schemes.
 Let $\rE^\bu$ be a complex of quasi-coherent torsion sheaves on\/ $\X$,
and let $\bK^\bu$ be a complex of\/ $\X$\+injective quasi-coherent
torsion sheaves on\/~$\bY$.
 Then there is a natural isomorphism
$$
 \pi_*\fHom_{\bY\qc}(\pi^*\rE^\bu,\bK^\bu)\simeq
 \fHom_{\X\qc}(\rE^\bu,\pi_*\bK^\bu)
$$
of complexes of pro-quasi-coherent pro-sheaves on\/~$\X$ (where
the functor\/ $\fHom_{\X\qc}$ in the right-hand side was defined before
the proof of Theorem~\ref{ind-Noetherian-triangulated-equivalence-thm}
in Section~\ref{ind-Noetherian-triangulated-equivalence-subsecn}).
\end{lem}

\begin{proof}
 In the notation above, we have
\begin{multline*}
 \pi_Z{}_*\fHom_{\bY\qc}(\pi^*\rE^\bu,\bK^\bu)^{(\bnW)}\simeq
 \pi_Z{}_*\cHom_{\bnW\qc}(\pi_Z^*i^!\rE^\bu,k^!\bK^\bu)\\ \simeq
 \cHom_{Z\qc}(i^!\rE^\bu,\pi_Z{}_*k^!\bK^\bu)=
 \cHom_{Z\qc}(i^!\rE^\bu,i^!\pi_*\bK^\bu)=
 \fHom_{\X\qc}(\rE^\bu,\pi_*\bK^\bu)^{(Z)},
\end{multline*}
where the second isomorphism is provided by
Lemma~\ref{qcoh-internal-Hom-of-complexes-projection}, while
the isomorphism $\pi_Z{}_*k^!\simeq i^!\pi_*$ holds by
the construction of the functor $\pi_*\:\bY\tors\rarrow\X\tors$
in Section~\ref{torsion-direct-images-subsecn}.
 Having this computation done, it remains to recall the construction of
the functor $\pi_*\:\bY\pro\rarrow\X\pro$
from Section~\ref{pro-sheaves-inverse-direct-subsecn}.
\end{proof}

\begin{proof}[Proof of
Theorem~\ref{relative-triangulated-equivalence-thm}]
 The argument follows the idea of the proof
of~\cite[Theorem~5.6]{Pfp}, which is a module version.

 The equivalence $\sD_\X^\si(\bY\tors)\simeq\sD(\bY\tors_{\X\dinj})$
is provided by Proposition~\ref{over-ind-Noetherian-semiderived-prop}.

 The tensor product functor $\ot_\bY\:\bY\tors\times\bY\pro\rarrow
\bY\tors$ was constructed in
Section~\ref{pro-in-torsion-action-subsecn}.
 Here we restrict it to the full subcategory of $\X$\+flat
pro-quasi-coherent pro-sheaves $\bY_\X\flat\subset\bY\pro$, obtaining
a functor $\ot_\bY\:\bY\tors\times\bY_\X\flat\rarrow\bY\tors$.
 This functor is extended to complexes similarly to the construction
in (the proof of)
Theorem~\ref{ind-Noetherian-triangulated-equivalence-thm},
using the coproduct totalization of bicomplexes.
 Given a complex of quasi-coherent torsion sheaves $\brE^\bu$ on $\bY$,
we obtain the functor
$$
 \brE^\bu\ot_\bY{-}\,\:\sC(\bY_\X\flat)\lrarrow\sC(\bY\tors),
$$
which obviously descends to a triangulated functor between the homotopy
categories $\brE^\bu\ot_\bY{-}\,\:\sK(\bY_\X\flat)\rarrow\sK(\bY\tors)$.

 Now let us assume that $\brE^\bu=\pi^*\rE^\bu$, where
$\rE^\bu\in\sC(\X\tors_\inj)$ is a complex of injective quasi-coherent
torsion sheaves on~$\X$.
 Let $\bfG^\bu$ be a complex of $\X$\+flat pro-quasi-coherent
pro-sheaves on~$\bY$.
 By Lemma~\ref{torsion-pro-projection-formula-lemma}, we have
$$
 \pi_*(\pi^*\rE^\bu\ot_\bY\bfG^\bu)\simeq\rE^\bu\ot_\X\pi_*\bfG^\bu
$$
(recall that the functor~$\pi_*$ preserves coproducts
by Lemma~\ref{representable-by-schemes-direct-image}(a)).
 Following the proof of
Theorem~\ref{ind-Noetherian-triangulated-equivalence-thm},
\,$\rE^\bu\ot_\X\pi_*\bfG^\bu$ is a complex of injective quasi-coherent
torsion sheaves on $\X$ (since $\pi_*\bfG^\bu$ is a complex of flat
pro-quasi-coherent pro-sheaves on~$\X$).
 So $\pi^*\rE^\bu\ot_\bY\bfG^\bu$ is a complex of $\X$\+injective
quasi-coherent torsion sheaves on~$\bY$.
 We have constructed a triangulated functor
\begin{equation} \label{tensor-with-lifted-complex-of-injectives}
 \pi^*\rE^\bu\ot_\bY{-}\,\:\sK(\bY_\X\flat)
 \lrarrow\sK(\bY\tors_{\X\dinj}).
\end{equation}

 Let us check that the latter functor induces a well-defined
triangulated functor
\begin{equation} \label{derived-to-derived-tensoring-functor}
 \pi^*\rE^\bu\ot_\bY{-}\,\:\sD(\bY_\X\flat)
 \lrarrow\sD(\bY\tors_{\X\dinj}).
\end{equation}
 We need to show that the complex $\pi^*\rE^\bu\ot_\bY\bfG^\bu$ is
acyclic with respect to $\bY\tors_{\X\dinj}$ whenever a complex
$\bfG^\bu$ is acyclic with respect to $\bY_\X\flat$.
 By Lemma~\ref{injective-over-base-acyclicity-criterion}, it suffices
to check that the complex $\pi_*(\pi^*\rE^\bu\ot_\bY\bfG^\bu)\simeq
\rE^\bu\ot_\X\pi_*\bfG^\bu$ in $\X\tors_\inj$ is contractible.
 By Lemma~\ref{flat-over-base-direct-image-acyclicity-criterion},
the complex $\pi_*\bfG^\bu$ is acyclic in $\X\flat$.
 Now the functor~\eqref{derived-to-derived-tensoring-functor} is
well-defined because
the functor~\eqref{derived-to-homotopy-tensoring-functor} from
(the proof of) Theorem~\ref{ind-Noetherian-triangulated-equivalence-thm}
is well-defined.

 On the other hand, for any complex $\rE^\bu\in\sC(\X\tors)$,
the construction before this proof provides a functor
$$
 \fHom_{\bY\qc}(\pi^*\rE^\bu,{-})\:
 \sC(\bY\tors_{\X\dinj})\lrarrow\sC(\bY\pro),
$$
which obviously descends to a triangulated functor between the homotopy
categories $\fHom_{\bY\qc}(\pi^*\rE^\bu,{-})\:\sK(\bY\tors_{\X\dinj})
\rarrow\sK(\bY\pro)$.

 Assume that $\rE^\bu$ is a complex of injective quasi-coherent
torsion sheaves on~$\X$, and let $\bK^\bu$ be a complex of
$\X$\+injective quasi-coherent torsion sheaves on~$\bY$.
 Then, according to the proof of
Theorem~\ref{ind-Noetherian-triangulated-equivalence-thm},
\,$\fHom_{\X\qc}(\rE^\bu,\pi_*\bK^\bu)$ is a complex of flat
pro-quasi-coherent pro-sheaves on~$\X$.
 By Lemma~\ref{internal-Hom-of-torsion-projection-formula},
it follows that $\fHom_{\bY\qc}(\pi^*\rE^\bu,\bK^\bu)$ is a complex
of $\X$\+flat pro-quasi-coherent pro-sheaves on~$\bY$.
 We have constructed a triangulated functor
\begin{equation} \label{inner-hom-from-lifted-complex-of-injectives}
 \fHom_{\bY\qc}(\pi^*\rE^\bu,{-})\:\sK(\bY\tors_{\X\dinj})
 \lrarrow\sK(\bY_\X\flat).
\end{equation}

 Let us check that the latter functor induces a well-defined
triangulated functor
\begin{equation} \label{derived-to-derived-internal-Hom}
 \fHom_{\bY\qc}(\pi^*\rE^\bu,{-})\:\sD(\bY\tors_{\X\dinj})
 \lrarrow\sD(\bY_\X\flat).
\end{equation}
 We need to show that the complex $\fHom_{\bY\qc}(\pi^*\rE^\bu,\bK^\bu)$
is acyclic with respect to $\bY_\X\flat$ whenever a complex $\bK^\bu$
is acyclic with respect to $\bY\tors_{\X\dinj}$.
 By Lemma~\ref{flat-over-base-direct-image-acyclicity-criterion},
it suffices to check that the complex
$\pi_*\fHom_{\bY\qc}(\pi^*\rE^\bu,\bK^\bu)\simeq
\fHom_{\X\qc}(\rE^\bu,\pi_*\bK^\bu)$ is acyclic in $\X\flat$.
 By Lemma~\ref{injective-over-base-acyclicity-criterion}, the complex
$\pi_*\bK^\bu$ is contractible in $\X\tors_\inj$, and the assertion
follows.

 It is straightforward to see that
the functor~\eqref{inner-hom-from-lifted-complex-of-injectives}
is right adjoint to
the functor~\eqref{tensor-with-lifted-complex-of-injectives}.
 Hence the functor~\eqref{derived-to-derived-internal-Hom} is
right adjoint to
the functor~\eqref{derived-to-derived-tensoring-functor}.
 It remains to show that
the functors~\eqref{derived-to-derived-tensoring-functor}
and~\eqref{derived-to-derived-internal-Hom} are mutually inverse
equivalences when $\rE^\bu=\rD^\bu$ is a dualizing complex on~$\X$.
 For this purpose, it suffices to check that the adjunction morphisms
are isomorphisms.

 Now, similarly to the proof of~\cite[Theorem~5.6]{Pfp}, we consider
the direct image (``forgetful'') functors
\begin{align*}
 \pi_*\:\sD_\X^\si(\bY\tors)&\lrarrow\sD^\co(\X\tors), \\
 \pi_*\:\sD(\bY\tors_{\X\dinj})&\lrarrow\sK(\X\tors_\inj),
\end{align*}
and
$$
 \pi_*\:\sD(\bY_\X\flat)\lrarrow\sD(\X\flat),
$$
which are well-defined, and moreover, conservative (by the definition
of the semiderived category or)
by Lemmas~\ref{injective-over-base-acyclicity-criterion}
and~\ref{flat-over-base-direct-image-acyclicity-criterion}.

 As we have already seen in the above discussion, by
Lemmas~\ref{torsion-pro-projection-formula-lemma}
and~\ref{internal-Hom-of-torsion-projection-formula}, the direct
image functors transform
the functors~\eqref{derived-to-derived-tensoring-functor}
and~\eqref{derived-to-derived-internal-Hom}
into the functors~\eqref{derived-to-homotopy-tensoring-functor}
and~\eqref{homotopy-to-derived-inner-hom-functor} from (the proof of)
Theorem~\ref{ind-Noetherian-triangulated-equivalence-thm}.
 In other words, there are commutative diagrams of
triangulated functors
$$
\xymatrix{
 \sD(\bY_\X\flat) \ar[rr]^-{\pi^*\rD^\subbu\ot_\bY{-}}
 \ar[d]_-{\pi_*}
 && \sD(\bY\tors_{\X\dinj}) \ar@{=}[r] \ar[d]^-{\pi_*} 
 & \sD_\X^\si(\bY\tors) \ar[d]^-{\pi_*} \\
 \sD(\X\flat) \ar[rr]_-{\rD^\subbu\ot_\X{-}}
 && \sK(\X\tors_\inj) \ar@{=}[r] 
 & \sD^\co(\X\tors)
}
$$
and
$$
\xymatrix{
 \sD_\X^\si(\bY\tors) \ar@{=}[r] \ar[d]_-{\pi_*}
 & \sD(\bY\tors_{\X\dinj})
 \ar[rrr]^-{\fHom_{\bY\qc}(\pi^*\rD^\subbu,{-})}
 \ar[d]_-{\pi_*}
 &&& \sD(\bY_\X\flat) \ar[d]^-{\pi_*} \\
 \sD^\co(\X\tors) \ar@{=}[r]
 & \sK(\X\tors_\inj) \ar[rrr]_-{\fHom_{\X\qc}(\rD^\subbu,{-})}
 &&& \sD(\X\flat)
}
$$

 The direct image functors also transform the adjunction morphisms for
the pair of adjoint functors $\pi^*\rD^\bu\ot_\bY{-}$ and
$\fHom_{\bY\qc}(\pi^*\rD^\bu,{-})$ into the adjunction morphisms for
the pair of adjoint functors $\rD^\bu\ot_\X{-}$ and
$\fHom_{\X\qc}(\rD^\bu,{-})$.
 Since the latter pair of adjunction morphisms are isomorphisms in
the respective derived/homotopy categories by
Theorem~\ref{ind-Noetherian-triangulated-equivalence-thm}, and
the direct image functors are conservative, the former pair of
adjunction morphisms are isomorphisms in the derived categories of
the respective exact categories, too.
\end{proof}

\Section{The Semitensor Product}  \label{semitensor-secn}

 The aim of this section is to construct the semitensor product
operation
$$
 \os_{\pi^*\rD^\bu}\:\sD_\X^\si(\bY\tors)\times
 \sD_\X^\si(\bY\tors)\lrarrow\sD_\X^\si(\bY\tors)
$$
on the $\bY/\X$\+semiderived category of quasi-coherent torsion sheaves
on $\bY$, making $\sD_\X^\si(\bY\tors)$ a tensor triangulated
category.
 The inverse image $\pi^*\rD^\bu$ on $\bY$ of the dualizing complex
$\rD^\bu$ on $\X$ is the unit object of this tensor structure.

 We follow the approach of~\cite[Section~6]{Pfp}, with suitable
modifications.
 We also explain how to correct a small mistake in the exposition
in~\cite[Section~6]{Pfp}.

\subsection{Underived tensor products in the relative context}
\label{underived-tensor-subsecn}
 We start with a sequence of lemmas extending
Lemma~\ref{pro-torsion-complexes-tensor-exactness} to the relative
situation.

\begin{lem} \label{tensor-with-inverse-pro-lemma}
 Let\/ $\pi\:\bY\rarrow\X$ be an affine morphism of ind-schemes.
 Let\/ $\fF^\bu\in\sC(\X\flat)$ be a complex of flat pro-quasi-coherent
pro-sheaves on\/ $\X$ and\/ $\bfG^\bu$ be a complex of\/ $\X$\+flat
pro-quasi-coherent pro-sheaves on\/ $\bY$ which is acyclic as
a complex in\/ $\bY_\X\flat$.
 Then the complex\/ $\pi^*\fF^\bu\ot^\bY\bfG^\bu$ of\/ $\X$\+flat
pro-quasi-coherent pro-sheaves on\/ $\bY$ is also acyclic as
a complex in\/ $\bY_\X\flat$.
\end{lem}

\begin{proof}
 First of all, $\pi^*\fF^\bu$ is a complex of flat pro-quasi-coherent
pro-sheaves on $\bY$ according to the discussion in
Section~\ref{flat-pro-sheaves-subsecn}, hence
$\pi^*\fF^\bu\ot^\bY\bfG^\bu$ is a complex of $\X$\+flat
pro-quasi-coherent pro-sheaves on $\bY$ by
Lemma~\ref{flat-pro-over-ind-scheme-tensor-product}.
 (Notice that the full subcategory of $\X$\+flat pro-quasi-coherent
pro-sheaves is closed under direct limits, and in particular, under
coproducts in $\bY\pro$, as one can see from the discussion in
Section~\ref{colimits-of-pro-sheaves-subsecn}.)

 By Lemma~\ref{flat-over-base-direct-image-acyclicity-criterion}, in
order to show that the complex $\pi^*\fF^\bu\ot^\bY\bfG^\bu$ is
acyclic in $\bY_\X\flat$, it suffices to check that the complex
$\pi_*(\pi^*\fF^\bu\ot^\bY\bfG^\bu)$ is acyclic in $\X\flat$.
 By the projection
formula~\eqref{pro-sheaves-projection-formula-eqn} from
Section~\ref{pro-sheaves-inverse-direct-subsecn}, we have
$\pi_*(\pi^*\fF^\bu\ot^\bY\bfG^\bu)\simeq\fF^\bu\ot^\X\pi_*\bfG^\bu$.
 Again by Lemma~\ref{flat-over-base-direct-image-acyclicity-criterion},
the complex $\pi_*\bfG^\bu$ is acyclic in $\X\flat$.
 It remains to refer to
Lemma~\ref{pro-torsion-complexes-tensor-exactness}(a).
\end{proof}

 Similarly one can prove that the complex $\pi^*\fF^\bu\ot^\bY\bfG^\bu$
is acyclic in $\bY_\X\flat$ whenever a complex $\fF^\bu$ is acyclic
in $\X\flat$ and $\bfG^\bu\in\sC(\bY_\X\flat)$ is an arbitrary complex.

 The next lemma is another version of projection formula for
the action of pro-quasi-coherent pro-sheaves in quasi-coherent
torsion sheaves (cf.\
Lemma~\ref{torsion-pro-projection-formula-lemma}).

\begin{lem} \label{torsion-pro-projection-formula-second}
 Let $f\:\Y\rarrow\X$ be an affine morphism of reasonable ind-schemes.
 Let\/ $\fP$ be a pro-quasi-coherent pro-sheaf on\/ $\X$ and
$\N$ be a quasi-coherent torsion sheaf on\/~$\Y$.
 Then there is a natural isomorphism
$$
 \fP\ot_\X f_*\N\simeq f_*(f^*\fP\ot_\Y\N)
$$
of quasi-coherent torsion sheaves on\/~$\X$.
\end{lem}

\begin{proof}
 The argument is similar to (but simpler than) the proof of
Lemma~\ref{torsion-pro-projection-formula-lemma}.
 The natural morphism
\begin{equation} \label{torsion-pro-second-projection-morphism}
 \fP\ot_\X f_*\N\lrarrow f_*(f^*\fP\ot_\Y\N)
\end{equation}
is adjoint to the composition $f^*(\fP\ot_\X f_*\N)\simeq
f^*\fP\ot_\Y f^*f_*\N\rarrow f^*\fP\ot_\Y\N$ of the isomorphism
$f^*(\fP\ot_\X f_*\N)\simeq f^*\fP\ot_\Y f^*f_*\N$ provided by
Lemma~\ref{inverse-images-preserve-torsion-pro-tensor} and
the morphism $f^*\fP\ot_\Y f^*f_*\N\rarrow f^*\fP\ot_\Y\N$
induced by the adjunction morphism $f^*f_*\N\rarrow\N$.

 In the notation of Section~\ref{semiderived-subsecn}, we consider
$\Gamma$\+systems of quasi-coherent sheaves on $\X$ and on~$\Y$.
 Similarly to the above one constructs, for any $\Gamma$\+system $\boN$
on $\Y$, a natural morphism of $\Gamma$\+systems $\fP\ot_\X f_*\boN
\rarrow f_*(f^*\fP\ot_\Y\boN)$, which is an isomorphism essentially
by Lemma~\ref{projection-formula}.

 To show that~\eqref{torsion-pro-second-projection-morphism} is
an isomorphism, it remains to compute
\begin{multline*}
 \fP\ot_\X f_*\N\simeq\fP\ot_\X f_*((\N|_\Gamma)^+)\simeq
 \fP\ot_\X (f_*(\N|_\Gamma))^+ \\ \simeq (\fP\ot_\X f_*(\N|_\Gamma))^+
 \simeq (f_*(f^*\fP\ot_\Y\N|_\Gamma))^+ \simeq
 f_*((f^*\fP\ot_\Y\N|_\Gamma)^+) \\
 \simeq f_*(f^*\fP\ot_\Y(\N_\Gamma)^+)\simeq f_*(f^*\fP\ot_\Y\N),
\end{multline*}
using the definitions of the functors $\ot_\X\:\X\pro\times\X\tors
\rarrow\X\tors$ and $\ot_\Y\:\Y\pro\times\Y\tors\rarrow\Y\tors$,
and also Lemma~\ref{representable-by-schemes-direct-image}(b).
 The point is that the direct image and tensor product functors
in question commute with the functors~$({-})^+$.
\end{proof}

\begin{lem} \label{inverse-of-pro-tensor-with-torsion}
 Let\/ $\pi\:\bY\rarrow\X$ be an affine morphism of reasonable
ind-schemes.
 Let\/ $\fF^\bu\in\sC(\X\flat)$ be a complex of flat pro-quasi-coherent
pro-sheaves on\/ $\X$ and $\bN^\bu$ be a complex of quasi-coherent
torsion sheaves on\/ $\bY$ such that the complex of quasi-coherent
torsion sheaves\/ $\pi_*\bN^\bu$ on\/ $\X$ is coacyclic.
 Then the complex\/ $\pi^*\fF^\bu\ot_\bY\bN^\bu$ of quasi-coherent
torsion sheaves on\/ $\bY$ also has the property that its direct image\/
$\pi_*(\pi^*\fF^\bu\ot_\bY\bN^\bu)$ is a coacyclic complex of
quasi-coherent torsion sheaves on\/~$\X$.
\end{lem}

\begin{proof}
 By Lemma~\ref{torsion-pro-projection-formula-second}, we have
$\pi_*(\pi^*\fF^\bu\ot_\bY\bN^\bu)\simeq
\fF^\bu\ot_\bY\pi_*\bN^\bu$ (recall that the direct image functor
$\pi_*\:\bY\tors\rarrow\X\tors$ preserves coproducts by
Lemma~\ref{representable-by-schemes-direct-image}(a)).
 So it remains to refer to
Lemma~\ref{pro-torsion-complexes-tensor-exactness}(b).
\end{proof}

 Similarly one can prove, assuming that $\X$ is an ind-Noetherian
ind-scheme and using
Lemma~\ref{pro-torsion-complexes-tensor-exactness}(c),
that $\pi_*(\pi^*\fF^\bu\ot_\bY\bN^\bu)$ is a coacyclic complex
of quasi-coherent torsion sheaves on $\X$ whenever $\fF^\bu$ is
an acyclic complex in $\X\flat$ and $\bN^\bu\in\sC(\bY\tors)$ is
an arbitrary complex.

\begin{lem} \label{pro-tensor-with-inverse-of-torsion}
 Let\/ $\X$ be an ind-Noetherian ind-scheme, and let\/
$\pi\:\bY\rarrow\X$ be an affine morphism of ind-schemes.
 Let $\rM^\bu\in\sC(\X\tors)$ be a complex of quasi-coherent torsion
sheaves on\/ $\X$ and\/ $\bfG^\bu$ be a complex of\/ $\X$\+flat
pro-quasi-coherent pro-sheaves on\/ $\bY$ which is acyclic as
a complex in\/ $\bY_\X\flat$.
 Then the complex\/ $\bfG^\bu\ot_\bY\pi^*\rM^\bu$ of quasi-coherent
torsion sheaves on\/ $\bY$ has the property that its direct image\/
$\pi_*(\bfG^\bu\ot_\bY\pi^*\rM^\bu)$ is a coacyclic complex of
quasi-coherent torsion sheaves on\/~$\X$.
\end{lem}

\begin{proof}
 By Lemma~\ref{torsion-pro-projection-formula-lemma}, we have
$\pi_*(\bfG^\bu\ot_\bY\pi^*\rM^\bu)\simeq\pi_*\bfG^\bu\ot_\X\rM^\bu$,
so it remains to refer to
Lemma~\ref{pro-torsion-complexes-tensor-exactness}(c).
\end{proof}

 Similarly one can prove, using 
Lemma~\ref{pro-torsion-complexes-tensor-exactness}(b), that
$\pi_*(\bfG^\bu\ot_\bY\pi^*\rM^\bu)$ is a coacyclic complex of
quasi-coherent torsion sheaves on $\X$ whenever $\rM^\bu$ is
a coacyclic complex in $\X\tors$ and $\bfG^\bu\in\sC(\bY_\X\flat)$
is an arbitrary complex.

 Let $\pi\:\bY\rarrow\X$ be an affine morphism of ind-schemes.
 Consider the functor of tensor product composed with inverse image
$$
 \pi^*({-})\ot^\bY{-}\,\:\sC(\X\flat)\times\sC(\bY_\X\flat)\lrarrow
 \sC(\bY_\X\flat).
$$
 It follows from Lemma~\ref{tensor-with-inverse-pro-lemma} and
the subsequent discussion that this functor descends to a triangulated
functor of two arguments
\begin{equation} \label{derived-lifted-flat-flat-tensor}
 \pi^*({-})\ot^\bY{-}\,\:\sD(\X\flat)\times\sD(\bY_\X\flat)\lrarrow
 \sD(\bY_\X\flat).
\end{equation}

 Let $\X$ be an ind-Noetherian ind-scheme, and let $\bY\rarrow\X$
be an affine morphism of schemes.
 Consider the two functors of tensor product composed with inverse
image
\begin{align*}
 \pi^*({-})\ot_\bY{-}\,\:\sC(\X\flat)\times\sC(\bY\tors)
 &\lrarrow\sC(\bY\tors), \\
 {-}\ot_\bY\pi^*({-})\:\sC(\bY_\X\flat)\times\sC(\X\tors)
 &\lrarrow\sC(\bY\tors).
\end{align*}
 It follows from Lemmas~\ref{inverse-of-pro-tensor-with-torsion}\+-%
\ref{pro-tensor-with-inverse-of-torsion} and the discussion that
these functors descend to triangulated functors of two arguments
\begin{align}
 \pi^*({-})\ot_\bY{-}\,\:\sD(\X\flat)\times\sD_\X^\si(\bY\tors)
 &\lrarrow\sD_\X^\si(\bY\tors),
 \label{derived-lifted-flat-torsion-tensor}\\
 {-}\ot_\bY\pi^*({-})\:\sD(\bY_\X\flat)\times\sD^\co(\X\tors)
 &\lrarrow\sD_\X^\si(\bY\tors).
 \label{derived-flat-lifted-torsion-tensor}
\end{align}

\subsection{Relatively homotopy flat resolutions}
\label{relatively-homotopy-flat-subsecn}
 Let $\pi\:\bY\rarrow\X$ be an affine morphism of reasonable
ind-schemes.
 Recall once again that, for any flat pro-quasi-coherent pro-sheaf
$\bfF$ on $\bY$ and any $\X$\+flat pro-quasi-coherent pro-sheaf
$\bfQ$ on $\bY$, the pro-quasi-coherent pro-sheaf $\bfF\ot^\bY\bfQ$
on $\bY$ is $\X$\+flat
(see Lemma~\ref{flat-pro-over-ind-scheme-tensor-product}).

 We will say that a complex $\bfF^\bu\in\sC(\bY\flat)$ of flat
pro-quasi-coherent pro-sheaves on $\bY$ is \emph{relatively
homotopy flat} if the following two conditions hold:
\begin{enumerate}
\renewcommand{\theenumi}{\roman{enumi}}
\item for any complex $\bfQ^\bu\in\sC(\bY_\X\flat)$ which is acyclic in
$\bY_\X\flat$, the complex $\bfF^\bu\ot^\bY\bfQ^\bu$ is acyclic in
$\bY_\X\flat$;
\item for any complex $\bN^\bu\in\sC(\bY\tors)$ such that the complex
$\pi_*\bN^\bu$ is coacyclic in $\X\tors$, the complex
$\pi_*(\bfF^\bu\ot_\bY\bN^\bu)$ is coacyclic in $\X\tors$.
\end{enumerate}

\begin{lem} \label{relatively-flat-complexes-lemma}
 Let\/ $\pi\:\bY\rarrow\X$ be an affine morphism of reasonable
ind-schemes.
 Then \par
\textup{(a)} the relatively homotopy flat complexes form a full
triangulated subcategory closed under coproducts in\/ $\sK(\bY\flat)$;
\par
\textup{(b)} for any complex\/ $\fP^\bu\in\sC(\X\flat)$ of flat
pro-quasi-coherent pro-sheaves on\/ $\X$, the complex\/
$\pi^*\fP^\bu\in\sC(\bY\flat)$ of flat pro-quasi-coherent pro-sheaves
on\/ $\bY$ is relatively homotopy flat.
\end{lem}

\begin{proof}
 Part~(a): clearly, any complex in $\bY\flat$ which is homotopy
equivalent to a relatively homotopy flat complex is also relatively
homotopy flat.
 Furthermore, for any complex $\bfQ^\bu\in\sK(\bY_\X\flat)$,
the functor ${-}\ot^\bY\bfQ^\bu\:\sK(\bY\flat)\rarrow\sK(\bY_\X\flat)$
is a triangulated functor preserving coproducts (see
Section~\ref{colimits-of-pro-sheaves-subsecn}), and the class of
all short exact sequences (hence the class of all acyclic complexes)
in $\bY_\X\flat$ is closed under coproducts.
 Similarly, for any complex $\bN^\bu\in\sK(\bY\tors)$,
the functor ${-}\ot_\bY\bN^\bu\:\sK(\bY\flat)\rarrow\sK(\bY\tors)$
is a triangulated functor preserving coproducts, and the class all
complexes in $\bY\tors$ whose direct images are coacyclic in $\X\tors$
is closed under coproducts (by
Lemma~\ref{representable-by-schemes-direct-image}(a)).
 Therefore, the class of all relatively homotopy flat complexes
is closed under shifts, cones, and coproducts in $\sK(\bY\flat)$.

 Part~(b) follows immediately from
Lemmas~\ref{tensor-with-inverse-pro-lemma}
and~\ref{inverse-of-pro-tensor-with-torsion}.
\end{proof}

 The next two abstract category-theoretic lemmas are well-known.

\begin{lem} \label{bounded-above-total-complex-lemma}
 Let\/ $\sB$ be an additive category with countable coproducts, and
let\/ $\dotsb\rarrow B_2^\bu\rarrow B_1^\bu\rarrow B_0^\bu\rarrow0$ be
a bounded above complex of complexes in\/~$\sB$.
 Let $T^\bu$ be the total complex of the bicomplex $B^\bu_\bu$
constructed by taking infinite coproducts along the diagonals.
 Then the complex $T^\bu\in\sK(\sB)$ is homotopy equivalent to
a complex which can be obtained from the complexes $B_n^\bu$, \,$n\ge0$,
using the operations of shift, cone, and countable coproduct.
\end{lem}

\begin{proof}
 Denote by $C_n^\bu$ the total complex of the finite complex of
complexes $B_n^\bu\rarrow B_{n-1}^\bu\rarrow\dotsb\rarrow B_1^\bu
\rarrow B_0^\bu$.
 Then there is a natural termwise split monomorphism of complexes
$C_n^\bu\rarrow T^\bu$ for every $n\ge0$, and the complex $T^\bu$
is the direct limit of the sequence of its termwise split
subcomplexes~$C_n^\bu$.
 Hence the telescope sequence $0\rarrow\coprod_{n\ge0}C_n^\bu
\rarrow\coprod_{n\ge0}C_n^\bu\rarrow T^\bu\rarrow0$ is a termwise
split short exact sequence of complexes in~$\sB$.
 It follows that the complex $T^\bu$ is homotopy equivalent to
the cone of the morphism of complexes $\coprod_{n\ge0}C_n^\bu
\rarrow\coprod_{n\ge0}C_n^\bu$.
 It remains to observe that every complex $C_n^\bu$ can be obtained
from the complexes $B_0^\bu$,~\dots, $B_n^\bu$ using the operations
of shift and cone finitely many times.
\end{proof}

\begin{lem} \label{bar-complex-lemma}
 Let\/ $\sA$ and $\sB$ be additive categories, $F\:\sA\rarrow\sB$
be a functor, and $G\:\sB\rarrow\sA$ be a functor right adjoint to~$F$.
 Denote the adjunction morphisms by $\mu\:FG\rarrow\Id_\sB$
and $\eta\:\Id_\sA\rarrow GF$.
 Then, for any object $B\in\sB$, there is a ``bar-complex'' in\/~$\sB$
$$
 \dotsb\rarrow B_n=(FG)^{n+1}(B)\rarrow\dotsb\rarrow
 FGFG(B)\rarrow FG(B)\rarrow B\rarrow0
$$
whose differential $d_n\:B_n\rarrow B_{n-1}$ is the alternating sum
of $n+1$ morphisms
$$
 d_n=\sum\nolimits_{i=0}^n(-1)^i\,
 (FG)^i\mu(FG)^{n-i}(B)\:(FG)^{n+1}(B)\lrarrow(FG)^n(B).
$$
 Applying the functor~$G$ to the above complex produces a contractible
complex in\/ $\sA$ with the contracting homotopy $h_n=\eta(GF)^nG(B)\:
(GF)^nG(B)\rarrow(GF)^{n+1}G(B)$.
\end{lem}

\begin{proof}
 The proof is a straightforward computation using the identity
$(G\mu)\circ(\eta G)=\id_{GFG}$.
\end{proof}

\begin{prop} \label{relatively-homotopy-flat-resolution}
 Let\/ $\pi\:\bY\rarrow\X$ be a flat affine morphism of reasonable
ind-schemes.
 Then for any complex\/ $\bfP^\bu\in\sC(\bY_\X\flat)$ of\/ $\X$\+flat
pro-quasi-coherent pro-sheaves on\/ $\bY$ there exists a relatively
homotopy flat complex\/ $\bfF^\bu$ of flat pro-quasi-coherent
pro-sheaves on\/ $\bY$ together with a morphism of complexes\/
$\bfF^\bu\rarrow\bfP^\bu$ whose cone is acyclic in\/ $\bY_\X\flat$.
\end{prop}

\begin{proof}
 Notice first of all that $\bY\flat\subset\bY_\X\flat$, since
the morphism~$\pi$ assumed to be flat.
 In the context of Lemma~\ref{bar-complex-lemma}, put
$\sA=\sC(\X\flat)$, \ $\sB=\sC(\bY_\X\flat)$, \ $F=\pi^*\:\sC(\X\flat)
\rarrow\sC(\bY\flat)\subset\sC(\bY_\X\flat)$, and $G=\pi_*\:
\sC(\bY_\X\flat)\rarrow\sC(\X\flat)$.
 Applying the construction of the lemma to the object $B=\bfP^\bu
\in\sC(\bY_\X\flat)=\sB$, we obtain a bicomplex
$$
 \dotsb\rarrow(\pi^*\pi_*)^{n+1}(\bfP^\bu)\rarrow\dotsb\rarrow
 \pi^*\pi_*\pi^*\pi_*\bfP^\bu\rarrow\pi^*\pi_*\bfP^\bu\rarrow\bfP^\bu
 \rarrow0
$$
in $\bY_\X\flat$.
 The differentials are the alternating sums of the maps induced
by the adjunction morphism $\pi^*\pi_*\rarrow\Id$.

 Let $\bfF^\bu$ be the total complex of the truncated bicomplex
$\dotsb\rarrow\pi^*\pi_*\pi^*\pi_*\bfP^\bu\rarrow\pi^*\pi_*\bfP^\bu
\rarrow0$, constructed by taking infinite coproducts along
the diagonals.
 By the last assertion of Lemma~\ref{bar-complex-lemma}, the cone
$\bfH^\bu$ of the morphism $\bfF^\bu\rarrow\bfP^\bu$ is a complex
in $\bY_\X\flat$ which becomes contractible after applying
the direct image functor $\pi_*\:\bY_\X\flat\rarrow\X\flat$.
 By Lemma~\ref{flat-over-base-direct-image-acyclicity-criterion},
it follows that the complex $\bfH^\bu$ is acyclic in $\bY_\X\flat$,
as desired.

 On the other hand, by Lemma~\ref{bounded-above-total-complex-lemma},
the complex $\bfF^\bu$ can be obtained from the complexes
$(\pi^*\pi_*)^{n+1}(\bfP^\bu)$, \,$n\ge0$, using the operations of
shift, cone, countable coproduct, and the passage to a homotopy
equivalent complex.
 The complex $\pi_*(\pi^*\pi_*)^n(\bfP^\bu)$ is a complex of flat
pro-quasi-coherent pro-sheaves on~$\X$; hence, by
Lemma~\ref{relatively-flat-complexes-lemma}(b), the complex
$(\pi^*\pi_*)^{n+1}(\bfP^\bu)$ is a relatively homotopy flat complex
of flat pro-quasi-coherent pro-sheaves on~$\bY$.
 According to Lemma~\ref{relatively-flat-complexes-lemma}(a), it follows
that the complex $\bfF^\bu$ is relatively homotopy flat.
\end{proof}

 Now we have to work out the torsion sheaf side of the story.
 We will say that a complex $\brG^\bu\in\sC(\bY\tors)$ of quasi-coherent
torsion sheaves on $\bY$ is \emph{homotopy\/ $\bY/\X$\+flat} if,
for any complex $\bfP^\bu\in\sC(\bY_\X\flat)$ which is acyclic in
$\bY_\X\flat$, the complex $\bfP^\bu\ot_\bY\brG^\bu$ of quasi-coherent
torsion sheaves on $\bY$ has the property that its direct image
$\pi_*(\bfP^\bu\ot_\bY\brG^\bu)$ is coacyclic in $\X\tors$.

\begin{lem} \label{homotopy-Y/X-flat-complexes-lemma}
 Let\/ $\X$ be an ind-Noetherian ind-scheme, and let\/
$\pi\:\bY\rarrow\X$ be an affine morphism of ind-schemes.
 Then \par
\textup{(a)} the homotopy\/ $\bY/\X$\+flat complexes form a full
triangulated subcategory closed under coproducts in\/ $\sK(\bY\tors)$;
\par
\textup{(b)} for any complex $\rM^\bu\in\sC(\X\tors)$ of
quasi-coherent torsion sheaves on\/ $\X$, the complex\/
$\pi^*\rM^\bu\in\sC(\bY\tors)$ of quasi-coherent torsion sheaves
on\/ $\bY$ is homotopy\/ $\bY/\X$\+flat.
\end{lem}

\begin{proof}
 Part~(a): clearly, any complex in $\bY\tors$ which is homotopy
equivalent to a homotopy $\bY/\X$\+flat compex is also homotopy
$\bY/\X$\+flat.
 Furthermore, for any complex $\bfP^\bu\in\sK(\bY_\X\flat)$,
the functor $\bfP^\bu\ot_\bY{-}\,\:\sK(\bY\tors)\rarrow\sK(\bY\tors)$
is a triangulated functor preserving coproducts, and the class of
all complexes in $\bY\tors$ whose direct images are coacyclic in
$\X\tors$ is closed under coproducts.
 Therefore, the class of all homotopy $\bY/\X$\+flat complexes is
closed under shifts, cones, and coproducts in $\sK(\bY\tors)$.
 These arguments do not need the assumption that $\X$ is ind-Noetherian
yet.
 Part~(b) follows immediately from
Lemma~\ref{pro-tensor-with-inverse-of-torsion} (which depends on
the ind-Noetherianity assumption).
\end{proof}

\begin{prop} \label{homotopy-Y/X-flat-resolution}
 Let\/ $\X$ be an ind-Noetherian ind-scheme, and let\/
$\pi\:\bY\rarrow\X$ be an affine morphism of ind-schemes.
 Then for any complex\/ $\bN^\bu\in\sC(\bY\tors)$ of quasi-coherent
torsion sheaves on\/ $\bY$ there exists a homotopy\/ $\bY/\X$\+flat
complex\/ $\brG^\bu$ of quasi-coherent torsion sheaves on\/ $\bY$
together with a morphism of complexes\/ $\brG^\bu\rarrow\bN^\bu$ whose
cone has the property that its direct image is coacyclic in\/ $\X\tors$.
\end{prop}

\begin{proof}
 In the context of Lemma~\ref{bar-complex-lemma}, put
$\sA=\sC(\X\tors)$, \ $\sB=\sC(\bY\tors)$, \ $F=\pi^*\:\sC(\X\tors)
\rarrow\sC(\bY\tors)$, and $G=\pi_*\:\sC(\bY\tors)\rarrow\sC(\X\tors)$.
 Applying the construction of the lemma to the object $B=\bN^\bu
\in\sC(\bY\tors)=\sB$, we obtain a bicomplex
$$
 \dotsb\rarrow(\pi^*\pi_*)^{n+1}(\bN^\bu)\rarrow\dotsb\rarrow
 \pi^*\pi_*\pi^*\pi_*\bN^\bu\rarrow\pi^*\pi_*\bN^\bu\rarrow\bN^\bu
 \rarrow0
$$
in $\bY\tors$.
 The differentials are the alternating sums of the maps induced
by the adjunction morphism $\pi^*\pi_*\rarrow\Id$.

 Let $\brG^\bu$ be the total complex of the truncated bicomplex
$\dotsb\rarrow\pi^*\pi_*\pi^*\pi_*\bN^\bu\rarrow\pi^*\pi_*\bN^\bu
\rarrow0$, constructed by taking infinite coproducts along
the diagonals.
 By the last assertion of Lemma~\ref{bar-complex-lemma}, the cone
of the morphism $\brG^\bu\rarrow\bN^\bu$ is a complex in $\bY\tors$
which becomes contractible (hence coacyclic) after applying
the direct image functor $\pi_*\:\bY\tors\rarrow\X\tors$.

 On the other hand, by Lemma~\ref{bounded-above-total-complex-lemma},
the complex $\brG^\bu$ can be obtained from the complexes
$(\pi^*\pi_*)^{n+1}(\bN^\bu)$, \,$n\ge0$, using the operations of
shift, cone, countable coproduct, and the passage to a homotopy
equivalent complex.
 By Lemma~\ref{homotopy-Y/X-flat-complexes-lemma}(b), the complex
$(\pi^*\pi_*)^{n+1}(\bN^\bu)$ is a homotopy $\bY/\X$\+flat complex
of quasi-coherent torsion sheaves on~$\bY$.
 According to Lemma~\ref{homotopy-Y/X-flat-complexes-lemma}(a), it
follows that the complex $\brG^\bu$ is homotopy $\bY/\X$\+flat.
\end{proof}

\subsection{Left derived tensor products for pro-sheaves
flat over a base} \label{left-derived-tensor-products-subsecn}
 Let $\pi\:\bY\rarrow\X$ be a flat affine morphism of reasonable
ind-schemes.
 The left derived functor of tensor product of $\X$\+flat
pro-quasi-coherent pro-sheaves on~$\bY$
\begin{equation} \label{derived-tensor-of-pro-sheaves}
 \ot^{\bY,\boL}\:\sD(\bY_\X\flat)\times\sD(\bY_\X\flat)\lrarrow
 \sD(\bY_\X\flat)
\end{equation}
is constructed in the following way.

 Let $\bfP^\bu$ and $\bfQ^\bu\in\sC(\bY_\X\flat)$ be two complexes
of $\X$\+flat pro-quasi-coherent pro-sheaves on~$\bY$.
 Using Proposition~\ref{relatively-homotopy-flat-resolution}, choose
two morphisms of complexes of $\X$\+flat pro-quasi-coherent pro-sheaves
$\bfF^\bu\rarrow\bfP^\bu$ and $\bfG^\bu\rarrow\bfQ^\bu$ such that
the cones of both morphisms are acyclic in $\bY_\X\flat$, and
both the complexes $\bfF^\bu$ and $\bfG^\bu\in\sC(\bY\flat)$ are
relatively homotopy flat complexes of flat pro-quasi-coherent
pro-sheaves on~$\bY$.
 Then, by condition~(i) in the definition of a relatively homotopy
flat complex, both the induced morphisms
$$
 \bfP^\bu\ot^\bY\bfG^\bu\llarrow\bfF^\bu\ot^\bY\bfG^\bu\lrarrow
 \bfF^\bu\ot^\bY\bfQ^\bu
$$
have cones acyclic in $\bY_\X\flat$.
 So we put
$$
 \bfP^\bu\ot^{\bY,\boL}\bfQ^\bu=\bfF^\bu\ot^\bY\bfG^\bu\simeq
 \bfP^\bu\ot^\bY\bfG^\bu\simeq\bfF^\bu\ot^\bY\bfQ^\bu
 \,\in\,\sD(\bY_\X\flat).
$$

 Using the definition of a relatively homotopy flat complex again,
the complex $\bfP^\bu\ot^\bY\bfG^\bu$ is acyclic in $\bY_\X\flat$
whenever the complex $\bfP^\bu$ is acyclic in $\bY_\X\flat$, and
the complex $\bfF^\bu\ot^\bY\bfQ^\bu$ is acyclic in $\bY_\X\flat$
whenever the complex $\bfQ^\bu$ is.
 So the derived functor $\ot^{\bY,\boL}$ is well-defined.
 We refer to~\cite[Lemma~2.7]{Psemi} for an abstract formulation of
this kind of construction of balanced derived functors of two
arguments (which is applicable in a much more general context of
two-sided derived functors).

\begin{rem}
 If one is only interested in the derived functor $\ot^{\bY,\boL}$
defined above (and \emph{not} in the derived functor $\ot_\bY^\boL$,
which we will define immediately below), then one can harmlessly drop
condition~(ii) from the definition of a relatively homotopy flat
complex.
 Then the assumption that the ind-schemes $\X$ and $\bY$ are reasonable
is not needed in the above construction.
\end{rem}

\begin{rem}
 Let us emphasize that the underived tensor product $\bfP\ot^\bY\bfQ$
of two $\X$\+flat pro-quasi-coherent pro-sheaves $\bfP$ and $\bfQ$ on
$\bY$ need \emph{not} be $\X$\+flat.
 It is only the derived tensor product $\bfP^\bu\ot^\bY\bfQ^\bu$ of
two complexes of $\X$\+flat pro-quasi-coherent pro-sheaves $\bfP^\bu$
and $\bfQ^\bu$ that is well-defined as an object of the derived category
of $\X$\+flat pro-quasi-coherent pro-sheaves $\sD(\bY_\X\flat)$.
 This is the reason why we had to assume our relatively homotopy flat
complexes to be complexes of flat (and not just $\X$\+flat)
pro-quasi-coherent pro-sheaves in the definition given
in Section~\ref{relatively-homotopy-flat-subsecn}.

 This subtlety was overlooked in the exposition
in~\cite[Section~6]{Pfp}.
 In the context of~\cite{Pfp}, a commutative ring homomorphism
$A\rarrow R$ was considered, with the assumption that $R$ is a flat
$A$\+module.
 Then the tensor product of two $A$\+flat $R$\+modules, taken over~$R$,
need not be $A$\+flat.
 It is only the tensor product of an $R$\+flat $R$\+module and
an $A$\+flat $R$\+module that is always $A$\+flat.
 To correct the mistake, one needs to include the assumption of
termwise flatness over $R$ into the definition of a ``relatively
homotopy $R$\+flat complex'' in the proof
of~\cite[Proposition~6.3]{Pfp} and the formulation
of~\cite[Lemma~6.4]{Pfp}.
\end{rem}

 Let $\X$ be an ind-Noetherian ind-scheme, and let $\pi\:\bY\rarrow\X$
be a flat affine morphism of ind-schemes.
 The left derived functor of tensor product of $\X$\+flat
pro-quasi-coherent pro-sheaves and quasi-coherent torsion sheaves
on~$\bY$
\begin{equation} \label{derived-tensor-of-pro-and-torsion}
 \ot_\bY^\boL\:\sD(\bY_\X\flat)\times\sD_\X^\si(\bY\tors)\lrarrow
 \sD_\X^\si(\bY\tors)
\end{equation}
is constructed in the following way.

 Let $\bfP^\bu\in\sC(\bY_\X\flat)$ be a complex of $\X$\+flat
pro-quasi-coherent pro-sheaves and $\bN^\bu\in\sC(\bY\tors)$ be
a complex of quasi-coherent torsion sheaves on~$\bY$.
 Using Proposition~\ref{relatively-homotopy-flat-resolution},
choose a morphism of complexes of $\X$\+flat pro-quasi-coherent
pro-sheaves $\bfF^\bu\rarrow\bfP^\bu$ whose cone is acyclic in
$\bY_\X\flat$, while $\bfF^\bu\in\sC(\bY\flat)$ is a relatively
homotopy flat complex of flat pro-quasi-coherent pro-sheaves on~$\bY$.
 Using Proposition~\ref{homotopy-Y/X-flat-resolution}, choose
a morphism of complexes of quasi-coherent torsion sheaves $\brG^\bu
\rarrow\bN^\bu$ whose cone has the property that its direct image
is coacyclic in $\X\tors$, while $\brG^\bu\in\sC(\bY\tors)$ is
a homotopy $\bY/\X$\+flat complex of quasi-coherent torsion sheaves
on~$\bY$.
 Then, by condition~(ii) in the definition of a relatively homotopy
flat complex, and by the definition of a homotopy $\bY/\X$\+flat
complex, both the induced morphisms
$$
 \bfP^\bu\ot_\bY\brG^\bu\llarrow\bfF^\bu\ot_\bY\brG^\bu\lrarrow
 \bfF^\bu\ot_\bY\bN^\bu
$$
have cones whose direct images are coacyclic in $\X\tors$.
 So we put
$$
 \bfP^\bu\ot_\bY^\boL\bN^\bu=\bfF^\bu\ot_\bY\brG^\bu\simeq
 \bfP^\bu\ot_\bY\brG^\bu\simeq\bfF^\bu\ot_\bY\bN^\bu
 \,\in\,\sD_\X^\si(\bY\tors).
$$

 Using the definition of a homotopy $\bY/\X$\+flat complex again,
the complex $\pi_*(\bfP^\bu\ot_\bY\brG^\bu)$ is coacyclic in $\X\flat$
whenever the complex $\bfP^\bu$ is acyclic in $\bY_\X\flat$.
 Using condition~(ii) from the definition of a relatively homotopy flat
complex, the complex $\pi_*(\bfF^\bu\ot_\bY\bN^\bu)$ is coacyclic
in $\X\flat$ whenever the complex $\pi_*\bN^\bu$ is coacyclic in
$\X\flat$.
 Thus the derived functor~$\ot_\bY^\boL$ is well-defined.
 This construction of a derived functor of two arguments is also
a particular case of~\cite[Lemma~2.7]{Psemi}.

 The derived functor $\ot^{\bY,\boL}$
\,\eqref{derived-tensor-of-pro-sheaves} defines an (associative,
commutative, and unital) tensor triangulated category structure on
the derived category $\sD(\bY_\X\flat)$ of the exact category of
$\X$\+flat pro-quasi-coherent pro-sheaves on~$\bY$.
 The ``pro-structure pro-sheaf'' $\fO_\bY\in\bY\flat\subset
\bY_\X\flat\subset\sD(\bY_\X\flat)$ is the unit object.
 (Notice that the one-term complex $\fO_\bY$ is a relatively homotopy
flat complex of flat pro-quasi-coherent pro-sheaves on~$\bY$.)

 The derived functor~$\ot_\bY^\boL$
\,\eqref{derived-tensor-of-pro-and-torsion} defines a structure of
triangulated module category over the tensor triangulated category
$\sD(\bY_\X\flat)$ on the $\bY/\X$\+semiderived category
$\sD_\X^\si(\bY\tors)$ of quasi-coherent torsion sheaves on~$\bY$.

\subsection{Construction of semitensor product}
\label{construction-of-semitensor-subsecn}
 In this section, $\X$ is an ind-semi-separated ind-Noetherian
ind-scheme with a dualizing complex $\rD^\bu$, and $\pi\:\bY\rarrow\X$
is a flat affine morphism of ind-schemes.

 The following lemma may help the reader feel more comfortable.

\begin{lem} \label{relative-or-Y/X-homotopy-flatness-lemma}
\textup{(a)} In the assumptions above, condition~\textup{(ii)} from
the definition of a relatively homotopy flat complex of
flat pro-quasi-coherent pro-sheaves in
Section~\ref{relatively-homotopy-flat-subsecn} implies
condition~\textup{(i)}. \par
\textup{(b)} In the same assumptions, let\/ $\bfF^\bu$ be a relatively
homotopy flat complex of flat pro-quasi-coherent pro-sheaves on\/~$\bY$.
 Then the complex\/ $\pi^*\rD^\bu\ot_\bY\bfF^\bu$ of quasi-coherent
torsion sheaves on\/ $\bY$ is homotopy\/ $\bY/\X$\+flat.
\end{lem}

\begin{proof}
 Both the assertions are based on the associativity property of
the action of the tensor category $\sC(\bY\pro)$ in its module category
$\sC(\bY\tors)$.

 Part~(a): let $\bfG^\bu\in\sC(\bY\flat)$ be a complex of flat
pro-quasi-coherent pro-sheaves on $\bY$ satisfying~(ii), and let
$\bfP^\bu\in\sC(\bY_\X\flat)$ be a complex of $\X$\+flat
pro-quasi-coherent pro-sheaves on $\bY$ which is acyclic
in $\bY_\X\flat$.
 Then, according to the proof of
Theorem~\ref{relative-triangulated-equivalence-thm} in
Section~\ref{relative-triangulated-equivalence-subsecn},
the complex of quasi-coherent torsion sheaves
$\pi^*\rD^\bu\ot_\bY\bfP^\bu$ has the property that its direct
image $\pi_*(\pi^*\rD^\bu\ot_\bY\bfP^\bu)$ is coacyclic (in fact,
contractible) in $\X\tors$.
 By condition~(ii), it follows that the complex
$$
 (\pi^*\rD^\bu\ot_\bY\bfP^\bu)\ot_\bY\bfG^\bu\simeq
 \pi^*\rD^\bu\ot_\bY(\bfP^\bu\ot^\bY\bfG^\bu)
$$
has the same property, i.~e., the compex
$\pi_*(\pi^*\rD^\bu\ot_\bY\bfP^\bu\ot^\bY\bfG^\bu)$ is
coacyclic in $\X\tors$.
 Now the assertion of
Theorem~\ref{relative-triangulated-equivalence-thm} implies
that the complex $\bfP^\bu\ot^\bY\bfG^\bu$ is acyclic in
$\bY_\X\flat$, since the corresponding object vanishes in
$\sD_\X^\si(\bY\tors)$.

 Part~(b): let $\bfQ^\bu\in\sC(\bY_\X\flat)$ a complex of
$\X$\+flat pro-quasi-coherent pro-sheaves on $\bY$ which is acyclic
in $\bY_\X\flat$.
 Then we have
$$
 (\pi^*\rD^\bu\ot_\bY\bfF^\bu)\ot_\bY\bfQ^\bu\simeq
 \pi^*\rD^\bu\ot_\bY(\bfF^\bu\ot^\bY\bfQ^\bu)
$$
in $\sC(\bY\tors)$.
 By condition~(i), the complex $\bfF^\bu\ot^\bY\bfQ^\bu$ is
acyclic in $\bY_\X\flat$.
 According to the proof of
Theorem~\ref{relative-triangulated-equivalence-thm}, it follows
that the complex $\pi_*(\pi^*\rD^\bu\ot_\bY\bfF^\bu\ot^\bY\bfQ^\bu)$
is coacyclic in $\X\tors$, as desired.
\end{proof}

 The triangulated equivalence $\pi^*\rD^\bu\ot_\bY{-}\,\:
\sD(\bY_\X\flat)\rarrow\sD_\X^\si(\bY\tors)$ is an equivalence of
module categories over the tensor category $\sD(\bY_\X\flat)$.
 Indeed, let $\bfP^\bu$ and $\bfQ^\bu$ be two complexes of $\X$\+flat
pro-quasi-coherent pro-sheaves on $\bY$, and let
$\bfF^\bu\rarrow\bfP^\bu$ and $\bfG^\bu\rarrow\bfQ^\bu$ be two
morphisms of complexes with the cones acyclic in $\bY_\X\flat$
and relatively homotopy flat complexes of flat pro-quasi-coherent
pro-sheaves $\bfF^\bu$ and $\bfG^\bu$ on~$\bY$.
 Then the desired natural isomorphism
$$
 (\bfP^\bu\ot^{\bY,\boL}\bfQ^\bu)\ot_\bY\pi^*\rD^\bu\,\simeq\,
 \bfP^\bu\ot_\bY^\boL(\bfQ^\bu\ot_\bY\pi^*\rD^\bu)
$$
in the semiderived category $\sD_\X^\si(\bY\tors)$ is represented
by any one of the associativity isomorphisms
\begin{align*}
 (\bfF^\bu\ot^\bY\bfQ^\bu)\ot_\bY\pi^*\rD^\bu &\,\simeq\,
 \bfF^\bu\ot_\bY(\bfQ^\bu\ot_\bY\pi^*\rD^\bu), \\
 (\bfP^\bu\ot^\bY\bfG^\bu)\ot_\bY\pi^*\rD^\bu &\,\simeq\,
 \bfP^\bu\ot_\bY(\bfG^\bu\ot_\bY\pi^*\rD^\bu),
\end{align*}
or
$$
 (\bfF^\bu\ot^\bY\bfG^\bu)\ot_\bY\pi^*\rD^\bu\,\simeq\,
 \bfF^\bu\ot_\bY(\bfG^\bu\ot_\bY\pi^*\rD^\bu)
$$
in the category of complexes $\sC(\bY\tors)$.
 Notice that the complex $\bfG^\bu\ot_\bY\pi^*\rD^\bu$ is homotopy
$\bY/\X$\+flat by
Lemma~\ref{relative-or-Y/X-homotopy-flatness-lemma}(b).

 Using the triangulated equivalence $\sD(\bY_\X\flat)\simeq
\sD_\X^\si(\bY\tors)$, we transfer the tensor structrue of the category
$\sD(\bY_\X\flat)$ to the category $\sD_\X^\si(\X\tors)$.
 The resulting functor
\begin{equation} \label{semiderived-torsion-semitensor-product}
 \os_{\pi^*\rD^\bu}\:\sD_\X^\si(\bY\tors)\times
 \sD_\X^\si(\bY\tors)\lrarrow\sD_\X^\si(\bY\tors),
\end{equation}
defining a tensor triangulated category structure on
the semiderived category $\sD_\X^\si(\bY\tors)$, is called
the \emph{semitensor product} of complexes of quasi-coherent torsion
sheaves on $\bY$ over the inverse image $\pi^*\rD^\bu$ of
the dualizing complex~$\rD^\bu$.
 The object $\pi^*\rD^\bu\in\sD_\X^\si(\bY\tors)$ is the unit object
of the tensor structure~$\os_{\pi^*\rD^\bu}$ on $\sD_\X^\si(\bY\tors)$,
since $\pi^*\rD^\bu$ corresponds to the unit object
$\fO_\bY\in\sD(\bY_\X\flat)$ under the equivalence of categories
$\sD(\bY_\X\flat)\simeq\sD_\X^\si(\bY\tors)$.

 Explicitly, let $\bM^\bu$ and $\bN^\bu\in\sK(\bY\tors)$ be two
complexes endowed with morphisms $\bM^\bu\rarrow\bK^\bu$ and
$\bN^\bu\rarrow\bJ^\bu$ into complexes $\bK^\bu$ and $\bJ^\bu\in
\sK(\bY\tors_{\X\dinj})$ with cones whose direct images are
coacyclic in $\X\tors$.
 Then one has
\begin{multline*}
 \bM^\bu\os_{\pi^*\rD^\bu}\bN^\bu=
 \pi^*\rD^\bu\ot_\bY(\fHom_{\bY\qc}(\pi^*\rD^\bu,\bK^\bu)
 \ot^{\bY,\boL}\fHom_{\bY\qc}(\pi^*\rD^\bu,\bJ^\bu)) \\ \simeq
 \fHom_{\bY\qc}(\pi^*\rD^\bu,\bK^\bu)\ot_\bY^\boL\bN^\bu \simeq
 \fHom_{\bY\qc}(\pi^*\rD^\bu,\bJ^\bu)\ot_\bY^\boL\bM^\bu.
\end{multline*}
 Let $\brF^\bu\rarrow\bM^\bu$ and $\brG^\bu\rarrow\bN^\bu$ be two
morphisms in $\sK(\bY\tors)$ whose cones' direct images are coacyclic
in $\X\tors$, while $\brF^\bu$ and $\brG^\bu$ are homotopy
$\bY/\X$\+flat complexes of quasi-coherent torsion sheaves on~$\bY$.
 Then the semiderived category object
$\bM^\bu\os_{\pi^*\rD^\bu}\bN^\bu\in\sD_\X^\si(\bY\tors)$ is
represented by any one of the two complexes
$$
 \fHom_{\bY\qc}(\pi^*\rD^\bu,\bK^\bu)\ot_\bY\brG^\bu
 \quad\text{or}\quad
 \brF^\bu\ot_\bY\fHom_{\bY\qc}(\pi^*\rD^\bu,\bJ^\bu)
$$
of quasi-coherent torsion sheaves on~$\bY$.

\subsection{The ind-Artinian base example}
\label{ind-artinian-base-subsecn}
 In this section we discuss flat affine morphisms of ind-schemes
$\pi\:\bY\rarrow\X$, where $\X$ is an ind-Artinian ind-scheme of
ind-finite type over a field (as in
Examples~\ref{cotensor-over-coalgebra}).

\smallskip
 (1)~Let $\rC$ be a coassociative coalgebra over a field~$\kk$.
 A (\emph{semiassociative, semiunital}) \emph{semialgebra} $\bS$ over
$\rC$ is defined as an associative, unital algebra object in
the (associative, unital, noncommutative) tensor category of
$\rC$\+$\rC$\+\emph{bicomodules}.

 Let $\bS$ be a semialgebra over~$\rC$.
 A \emph{left semimodule} over $\bS$ is defined as a module object in
the module category of left $\rC$\+comodules over the algebra object
$\bS$ in the tensor category of $\rC$\+$\rC$\+bicomodules.
 \emph{Right semimodules} are defined similarly~\cite[Sections~0.3.2
and~1.3.1]{Psemi}, \cite[Section~2.6]{Prev}.

 The category of left $\bS$\+semimodules, which we will denote by
$\bS\simodl$, is abelian whenever $\bS$ is an injective
right $\rC$\+comodule.
 In this case, $\bS\simodl$ is a Grothendieck abelian category.
 The category of right $\bS$\+semimodules is denoted by $\simodr\bS$.

\smallskip
 (2)~We are interested in \emph{semicommutative} semialgebras, which
are a particular case of~(1).
 Let $\rC$ be a cocommutative coalgebra over~$\kk$.
 Then, following Example~\ref{cotensor-over-coalgebra}(2),
the category of $\rC$\+comodules $\rC\comodl$ is an associative,
commutative, and unital tensor category with respect to the cotensor
product operation~$\oc_\rC$.
 In fact, $\rC\comodl$ is a tensor subcategory in the tensor category
of $\rC$\+$\rC$\+bicomodules (consisting of those bicomodules in
which the left and right $\rC$\+coactions agree).

 A \emph{semicommutative semialgebra} $\bS$ over $\rC$ is defined as
a commutative (associative, and unital) algebra object in the tensor
category $\rC\comodl$.
 In other words, $\bS$ is an $\rC$\+comodule endowed with
a \emph{semiunit map} $\rC\rarrow\bS$ and a \emph{semimultiplication
map} $\bS\oc_\rC\bS\rarrow\bS$.
 Both the maps must be $\rC$\+comodule morphisms, and the usual
associativity, commutativity, and unitality equations should be 
satisfied.

 Over a semicommutative semialgebra $\bS$, there is no difference
between left and right semimodules.

\smallskip
 (3)~More specifically, we are interested in \emph{$\rC$\+injective}
semicommutative semialgebras $\bS$ over~$\rC$.
 So the underlying $\rC$\+comodule of $\bS$ is assumed to be injective.
 Injective $\rC$\+comodules form a tensor subcategory in $\rC\comodl$.

 For a cocommutative coalgebra $\rC$, the category of
$\rC$\+contramodules is also naturally an (associative, commutative,
and unital) tensor category with respect to the operation of
\emph{contramodule tensor product} $\ot^\rC=\ot^{\rC^*}$
\cite[Section~1.6]{Pweak}.
 The free $\rC$\+contramodule with one generator $\rC^*$ is
the unit object.
 The full subcategory of projective $\rC$\+contramodules
$\rC\contra_\proj\subset\rC\contra$ is a tensor subcategory.

 The equivalence between the additive categories of injective
$\rC$\+comodules and projective $\rC$\+contramodules
$\rC\comodl_\inj\simeq\rC\contra_\proj$
(see formula~\eqref{coalgebra-underived-co-contra-correspondence}
in Example~\ref{cotensor-over-coalgebra}(6)) is an equivalence of
tensor categories.

\smallskip
 (4)~Let $\X=\Spi\rC^*$ be the ind-Artinian ind-scheme corresponding
to the coalgebra $\rC$, as in Examples~\ref{coalgebra-ind-scheme}(2),
\ref{ind-Artinian-ind-schemes}(2), and~\ref{cotensor-over-coalgebra}(5).
 According to Example~\ref{cotensor-over-coalgebra}(6), there is
an equivalence of additive categories $\X\flat\simeq\rC\contra_\proj$.
 Moreover, similarly to
Example~\ref{countable-flat-pro-sheaves-flat-contramodules}(4)
(cf.\ Example~\ref{general-ind-Artinian-ind-scheme}),
this is an equivalence of tensor categories.

 We are interested in transformations of commutative algebra objects
under the equivalences of tensor categories above.
 Let $\bS$ be a $\rC$\+injective semicommutative semialgebra over
$\rC$, let $\bfS$ be the corresponding commutative algebra object in
the tensor category $\rC\contra_\proj$, and let $\bfA$ be
the corresponding commutative algebra object in the tensor category
$\X\flat$.
 Then the anti-equivalence of categories from
Proposition~\ref{ind-schemes-affine-morphisms-pro-qcoh-algebras}
assigns to $\bfA$ an ind-scheme $\bY=\Spi_\X\bfA$ together with a flat
affine morphism of ind-schemes $\pi\:\bY\rarrow\X$.
 Let us describe the ind-scheme $\bY$ and the morphism~$\pi$
more explicitly. 

 For any algebra object $\fS$ in the tensor category $\rC\contra$,
precomposing the multiplication morphism $\fS\ot^{\rC^*}\fS\rarrow\fS$
with the natural map $\fS\ot_{\rC^*}\fS\rarrow\fS\ot^{\rC^*}\fS$ allows
to define the underlying $\rC^*$\+algebra structure on~$\fS$.
 Moreover, any projective $\rC$\+contramodule $\fF$ has a natural
underlying structure of a complete, separated topological module
over a topological ring~$\rC^*$; for a free $\rC$\+contramodule
$\fF=\Hom_\kk(\rC,V)$ spanned by a $\kk$\+vector space $V$, this is
the usual topology on the Hom space of infinite-dimensional vector
spaces.
 For an algebra object $\fS$ in the tensor category $\rC\contra_\proj$,
this topology makes $\fS$ a complete, separated topological ring with
a base of neighborhoods of zero formed by open (two-sided) ideals.
 The unit morphism of $\fS$ provides a continuous ring homomorphism
$\rC^*\rarrow\fS$.

 In the situation at hand, $\bfS=\Hom_\rC(\rC,\bS)=
\Hom_{\rC^*}(\rC,\bS)$ is the $\kk$\+vector space of $\rC$\+comodule
(or equivalently, $\rC^*$\+module) homomorphisms $\rC\rarrow\bS$.
 The topology on $\bfS$ is the \emph{finite topology} of the Hom
module: the annihilators of finite-dimensional subspaces (equivalently,
finite-dimensional subcoalgebras) in $\rC$ form a base of neighborhoods
of zero in $\Hom_\rC(\rC,\bS)$.
 The semiunit map $\rC\rarrow\bS$ induces the unit map
$\rC^*\simeq\Hom_\rC(\rC,\rC)\rarrow\Hom_\rC(\rC,\bS)$, which is
a continuous ring homomorphism $\rC^*\rarrow\bfS$.
 Then one has $\bY=\Spi\bfS$ in the notation of
Example~\ref{topological-ring-ind-scheme}(1); the flat affine morphism
$\pi\:\bY\rarrow\X$ corresponds to the homomorphism of topological
rings $\rC^*\rarrow\bfS$.

 For any cocommutative coalgebra $\rC$ over~$\kk$ and the corresponding
ind-Artinian ind-scheme $\X=\Spi\rC^*$, we have constructed a natural
anti-equivalence between the category of $\rC$\+injective
semicommutative semalgebras $\bS$ over $\rC$ and the category of
ind-schemes $\bY$ endowed with a flat affine morphism $\bY\rarrow\X$.

\smallskip
 (5)~We keep the notation of~(4).
 According to
Proposition~\ref{pro-torsion-modules-via-pro-algebras-described}(b),
the abelian category $\bY\tors$ is equivalent to the category of
module objects in the module category $\X\tors$ over the algebra
object $\bfA$ in the tensor category $\X\flat$.
 Following Section~\ref{torsion-ind-affine-subsecn}(4), we have
an equivalence of categories $\X\tors\simeq\rC\comodl$.
 Hence the category $\bY\tors$ is equivalent to the category of
module objects in the module category $\rC\comodl$ over
the algebra object $\bfS$ in the tensor category $\rC\contra_\proj$.
 Here the action of $\rC\contra_\proj$ in $\rC\comodl$ is given
by the contratensor product functor~$\ocn_\rC$ (see the last paragraph
of Example~\ref{cotensor-over-coalgebra}(6)).

 The equivalence of tensor categories $\rC\comodl_\inj\simeq
\rC\contra_\proj$ transforms the action of $\rC\comodl_\inj$ in
$\rC\comodl$ (by the cotensor product) into the action of
$\rC\contra_\proj$ in $\rC\comodl$ (by the contratensor product).
 This is clear from the associativity isomorphism connecting
the cotensor and contratensor products~\cite[Proposition~5.2.1]{Psemi},
\cite[Proposition~3.1.1]{Prev}, which was already mentioned in
Example~\ref{cotensor-over-coalgebra}(8).
 Taken together, the constructions above combine into a natural
equivalence between the abelian category of quasi-coherent torsion
sheaves on $\bY$ and the abelian category of $\bS$\+semimodules,
$\bY\tors\simeq\bS\simodl$.

 A quasi-coherent torsion sheaf on $\bY$ is $\X$\+injective if and
only if the corresponding $\bS$\+semimodule is injective \emph{as
a $\rC$\+comodule}.
 We will denote the full subcategory of semimodules whose underlying
comodules are injective by $\bS\simodl_{\rC\dinj}\subset\bS\simodl$.
 So the equivalence of abelian categories $\bY\tors\simeq\bS\simodl$
restricts to an equivalence of full subcategories $\bY\tors_{\X\dinj}
\simeq\bS\simodl_{\rC\dinj}$.
 The latter one is obviously an equivalence of exact categories (with
the exact category structures inherited from the ambient abelian
categories).

\smallskip
 (6)~According to
Proposition~\ref{pro-torsion-modules-via-pro-algebras-described}(a)
and the discussion in Section~\ref{pro-flat-over-base-subsecn},
the exact category $\bY_\X\flat$ of $\X$\+flat pro-quasi-coherent
pro-sheaves on $\bY$ is equivalent to the exact category of module
objects over the algebra object $\bfA$ in the tensor category $\X\flat$.
 As the tensor category $\X\flat$ is equivalent to the tensor category
$\rC\contra_\proj$ by~(4), it follows that the exact category
$\bY_\X\flat$ is equivalent to the exact category of module objects
over the algebra object $\bfS$ in the tensor category $\rC\contra_\proj$.
 Here the exact structure on $\X\flat\simeq\rC\contra_\proj$ is split,
but the exact structure on the category of module objects is not;
rather, a short sequence of module objects is exact if and only if it
becomes split exact after the module structures are forgotten.

 One can see that specifying a module structure over $\bfS$ on
a given projective $\rC$\+contramodule $\fF$ is equivalent to
specifying a \emph{$\bS$\+semicontramodule} structure on $\fF$
(see~\cite[Sections~0.3.5 and~3.3.1]{Psemi}
or~\cite[Section~2.6]{Prev} for the definition).
 So the exact category $\bY_\X\flat$ is equivalent to the exact
category of $\rC$\+projective $\bS$\+semicontramodules.
 Semicontramodules over an $\rC$\+injective semialgebra $\bS$ form
an abelian category $\bS\sicntr$; we will denote the full subcategory
of $\rC$\+injective semicontramodules by
$\bS\sicntr_{\rC\dproj}\subset\bS\sicntr$.
 The exact structure on $\bS\sicntr_{\rC\dproj}$ is inherited from
the abelian category $\bS\sicntr$.
 So we have an equivalence of exact categories
$\bY_\X\flat\simeq\bS\sicntr_{\rC\dproj}$.

\smallskip
 (7)~It is explained in~\cite[Section~10.3]{PS} that the datum of
an $\bS$\+semicontramodule structure on a given vector space (or
$\rC$\+contramodule) is equivalent to the datum of a contramodule
structure over the topological ring $\bfS$, that is $\bS\sicntr\simeq
\bfS\contra$.
 It follows that the exact category $\bY_\X\flat$ is equivalent to
the exact category of $\bfS$\+contramodules which are projective
\emph{as contramodules over the topological ring\/ $\rC^*$}, i.~e.,
$\bY_\X\flat\simeq\bfS\contra_{\rC^*\dproj}$.
 The latter equivalence restricts to an equivalence between the category
of flat pro-quasi-coherent pro-sheaves on $\bY$ and the category of flat
$\bfS$\+contramodules (in the sense of~\cite[Section~2]{Pproperf}),
$\bY\flat\simeq\bfS\flat$.

 Furthermore, the datum of an $\bS$\+semimodule structure on a given
vector space (or $\rC$\+comodule) is equivalent to the datum of
a discrete $\bfS$\+module structure, $\bS\simodl\simeq\bfS\discr$
\,\cite[Remark~10.9]{PS}.
 Consequently, we have a natural equivalence of abelian categories
$\bY\tors\simeq\bfS\discr$.

 The equivalences of categories $\bY\flat\simeq\bfS\flat$ and
$\bY_\X\flat\simeq\bfS\contra_{\rC^*\dproj}$ from~(7) transform
the tensor product functor $\ot^\bY\:\bY\flat\times\bY_\X\flat\rarrow
\bY_\X\flat$ into the contramodule tensor product functor
$\ot^\bfS\:\bfS\contra\times\bfS\contra\rarrow\bfS\contra$ (as defined
in~\cite[Section~1.6]{Pweak}), restricted to the full subcategories
$\bfS\flat\subset\bfS\contra_{\rC^*\dproj}\subset\bfS\contra$.
 The equivalences of categories $\bY_\X\flat\simeq
\bfS\contra_{\rC^*\dproj}$ and $\bY\tors\simeq\bfS\discr$ from~(7)
transform the tensor product functor $\ot_\bY\:\bY_\X\flat\times
\bY\tors\rarrow\bY\tors$ into the contratensor product functor
$\ocn_\bfS\:\bfS\contra\times\bfS\discr\rarrow\bfS\discr$ restricted
to $\bfS\contra_{\rC^*\dproj}\subset\bfS\contra$ (see the discussion
and references in
Example~\ref{countable-flat-pro-sheaves-flat-contramodules}(3)).

 The equivalences of categories $\bY_\X\flat\simeq
\bS\sicntr_{\rC\dproj}$ and $\bY\tors\simeq\bS\simodl$ from~(6)
and~(5) transform the tensor product functor $\ot_\bY\:\bY_\X\flat
\times\bY\tors\rarrow\bY\tors$ into the functor of
\emph{contratensor product of semimodules and semicontramodules}
$\Ocn_\bS\:\bS\sicntr\times\bS\simodl\rarrow\bS\simodl$ (constructed
in~\cite[Sections~0.3.7 and~6.1]{Psemi}), restricted to the full
subcategory $\bS\sicntr_{\rC\dproj}\subset\bS\sicntr$.

\smallskip
 (8)~As explained in Example~\ref{cotensor-over-coalgebra}(5),
the $\rC$\+comodule $\rC$ corresponds to a one-term dualizing complex
of injective quasi-coherent torsion sheaves $\rD^\bu=\rC$ on
the ind-Artinian ind-scheme $\X=\Spi\rC^*$.

 Following the proof of
Theorem~\ref{relative-triangulated-equivalence-thm} specialized to
the particular case of a one-term dualizing complex of injectives
$\rD^\bu=\rC$ on $\X$, one can see that there is an equivalence
between the exact categories of $\X$\+injective quasi-coherent
torsion sheaves and $\X$\+flat pro-quasi-coherent pro-sheaves on $\bY$,
provided by the mutually inverse functors
$\fHom_{\bY\qc}(\pi^*\rC,{-})$ and $\pi^*\rC\ot_\bY{-}$,
\begin{equation} \label{over-base-injective-flat-correspondence}
 \fHom_{\bY\qc}(\pi^*\rC,{-})\:\bY\tors_{\X\dinj}
 \,\simeq\,\bY_\X\flat\,:\!\pi^*\rC\ot_\bY{-}.
\end{equation}

 The equivalence of exact categories $\bY\tors_{\X\dinj}\simeq
\bS\simodl_{\rC\dinj}$ and $\bY_\X\flat\simeq\bS\sicntr_{\rC\dproj}$
from items~(5) and~(6) form a commutative square diagram with
the equivalence~\eqref{over-base-injective-flat-correspondence}
and the equivalence of exact categories
\begin{equation} \label{over-semialgebra-underived-co-contra}
 \Hom_\bS(\bS,{-})\:\bS\simodl_{\rC\dinj}\,\simeq\,
 \bS\sicntr_{\rC\dproj}\,:\!\bS\Ocn_\bS{-},
\end{equation}
which was constructed in~\cite[Sections~0.3.7 and~6.2]{Psemi}
and discussed in~\cite[Section~3.5]{Prev}.
 Here $\Hom_\bS=\Hom_{\bS\simodl}$ denoted the vector space of morphisms
in the abelian category of $\bS$\+semimodules, while $\Ocn_\bS$ is
the contratensor product functor mentioned in~(7) above.
 Notice that the quasi-coherent torsion sheaf $\pi^*\rC=\pi^*\rD^\bu$
on $\bY$ corresponds to the $\bS$\+semimodule $\bS$ under
the equivalence of categories $\bY\tors\simeq\bS\simodl$
(or $\bY\tors_{\X\dinj}\simeq\bS\simodl_{\rC\dinj}$).

\smallskip
 (9)~Given a semiassociative semialgebra $\bS$ over a coassociative
coalgebra $\rC$, the \emph{semitensor product} $\bM\os_\bS\bN$ of
a right $\bS$\+semimodule $\bM$ and a left $\bS$\+semimodule $\bN$ is
the $\kk$\+vector space constructed as the cokernel of the difference
of the natural pair of maps
$$
 \bM\oc_\rC\bS\oc_\rC\bN\,\rightrightarrows\,\bM\oc_\rC\bN.
$$
 Here one map is induced by the right semiaction map $\bM\oc_\rC\bS
\rarrow\bM$ and the other one by the left semiaction map
$\bS\oc_\rC\bN\rarrow\bN$ \cite[Sections~0.3.2 and~1.4.1\+-2]{Psemi}
(cf.\ the definition of the cotensor product~$\oc_\rC$ in
Example~\ref{cotensor-over-coalgebra}(1) above).

 Let $\bS$ be a semicommutative semialgebra over a cocommutative
coalgebra~$\rC$; assume that $\bS$ is an injective $\rC$\+comodule.
 Then the semitensor product $\bM\os_\bS\bN$ of two $\bS$\+semimodules
$\bM$ and $\bN$ has a natural $\bS$\+semimodule
structure~\cite[Section~1.4.4]{Psemi}.
 The semitensor product operation $\os_\bS$ on the category $\bS\simodl$
is commutative and unital; the $\bS$\+semimodule $\bS$ is the unit 
object.

 However, one needs to impose some additional assumptions in order to
make sure that the semitensor product is associative.
 The semitensor product of any three $\rC$\+injective $\bS$\+semimodules
is associative~\cite[Proposition~1.4.4(a)]{Psemi}, but the full
subcategory $\bS\simodl_{\rC\dinj}\subset\bS\simodl$ is \emph{not} 
preserved by~$\os_\bS$, generally speaking.
 The full subcategory of so-called \emph{semiflat} $\bS$\+semimodules
(defined in~\cite[Section~1.4.2]{Psemi}) is a commutative, associative,
and unital tensor category with respect to the semitensor product
over~$\bS$.

\smallskip
 (10)~For a semiassociative semialgebra $\bS$ over a coassociative
coalgebra $\rC$, the double-sided derived functor of semitensor product
$$
 \SemiTor_\bS\:\sD^\si(\simodr\bS)\times\sD^\si(\bS\simodl)
 \lrarrow\sD(\kk\vect)
$$
is constructed in~\cite[Section~2.7]{Psemi}.
 Here $\sD^\si(\bS\simodl)=\sD^\si_\rC(\bS\simodl)$ is the semiderived
(or the ``semicoderived'') category of left $\bS$\+semimodules
\emph{relative to\/~$\rC$}, i.~e., the triangulated quotient category
of the homotopy category $\sK(\bS\simodl)$ by the thick subcategory of
complexes that are coacyclic as complexes of $\rC$\+comodules.
 The semiderived category $\sD^\si(\simodr\bS)=\sD^\si_\rC(\simodr\bS)$
is defined similarly.

 A construction of the double-sided derived functor of semitensor
product of \emph{bisemimodules}, taking values in a semiderived
category of bisemimodules, can be found in~\cite[Section~2.9]{Psemi}.
 As one can use (strongly semiflat complexes of) semiflat bisemimodules
in the construction of this derived functor, the double-sided derived
functor of semitensor product of bisemimodules is associative.

 Similarly, for a semicommutative semialgebra $\bS$ over a cocommutative
coalgebra $\rC$, one can construct the double-sided derived functor
of semitensor product
\begin{equation} \label{double-sided-derived-semitensor}
 \os_\bS^\boD\:\sD^\si(\bS\simodl)\times\sD^\si(\bS\simodl)
 \lrarrow\sD^\si(\bS\simodl),
\end{equation}
which defines a structure of associative, commutative, and unital
tensor category on the triangulated category $\sD^\si(\bS\simodl)$.
 The one-term complex of $\bS$\+semimodules $\bS$ is the unit object.

 The equivalence of abelian categories $\bY\tors\simeq\bS\simodl$ from
item~(5) forms a commutative square diagram with the equivalence of
abelian categories $\X\tors\simeq\rC\comodl$, the direct image functor
$\pi_*\:\bY\tors\rarrow\X\tors$, and the forgetful functor $\bS\simodl
\rarrow\rC\comodl$.
 Therefore, a triangulated equivalence of the semiderived categories
$\sD^\si_\X(\bY\tors)\simeq\sD^\si(\bS\simodl)$ is induced.

 The result of~\cite[Corollary~6.6(b)]{Psemi} together with
the discussion in items~(7\+-8) above shows that the triangulated
equivalence $\sD^\si_\X(\bY\tors)\simeq\sD^\si(\bS\simodl)$
transforms the semitensor product functor~$\os_{\pi^*\rD^\bu}$
\,\eqref{semiderived-torsion-semitensor-product} from
Section~\ref{construction-of-semitensor-subsecn} for the dualizing
complex $\rD^\bu=\rC$ on $\X$ into the double-sided derived
functor~$\os_\bS^\boD$ \,\eqref{double-sided-derived-semitensor}.

\Section{Flat Affine Ind-Schemes over Ind-Schemes of Ind-Finite Type}
\label{flat-affine-over-ind-finite-type-secn}

 In this section, as in Section~\ref{ind-finite-type-secn},
\,$\kk$~denotes a fixed ground field.
 Let $\X$ be an ind-separated ind-scheme of ind-finite type over~$\kk$,
and let $\pi\:\bY\rarrow\X$ be a flat affine morphism of schemes.
 Consider the diagonal morphism $\Delta_\bY\:\bY\rarrow
\bY\times_\kk\bY$; the morphism~$\Delta_\bY$ factorizes naturally into 
the composition
$$
 \bY\overset\delta\lrarrow\bY\times_\X\bY
 \overset\eta\lrarrow\bY\times_\kk\bY.
$$
 We denote the two morphisms involved by $\delta=\delta_{\bY/\X}\:
\bY\rarrow\bY\times_\X\bY$ and $\eta=\eta_{\bY/\X}\:\bY\times_\X\bY
\rarrow\bY\times_\kk\bY$.

 Let $\rD^\bu$ be a rigid dualizing complex on~$\X$
(as defined in Section~\ref{rigid-subsecn}). 
 The aim of this section is to describe the semitensor product functor
$\os_{\pi^*\rD^\bu}\:\sD_\X^\si(\bY\tors)\times\sD_\X^\si(\bY\tors)
\rarrow\sD_\X^\si(\bY\tors)$ as the composition of the left derived
$*$\+restriction and the right derived $!$\+restriction of the external
tensor product on $\bY\times_\kk\bY$ to the closed immersions
$\delta_{\bY/\X}$ and~$\eta_{\bY/\X}$, resprectively; that is
$$
 \bM^\bu\os_{\pi^*\rD^\bu}\bN^\bu \,=\,
 \boL\delta^*\,\boR\eta^!(\bM^\bu\bt_\kk\bN^\bu)
$$
for any two complexes of quasi-coherent torsion sheaves $\bM^\bu$
and $\bN^\bu$ on~$\bY$.

\subsection{Derived inverse image of pro-sheaves}
\label{pro-derived-inverse-image-subsecn}
 Suppose that we are given a commutative square diagram of morphisms
of ind-schemes
\begin{equation} \label{square-of-ind-schemes}
\begin{gathered}
\xymatrix{
 \bW \ar[r]^g \ar[d]_\rho & \bY \ar[d]^\pi \\
 \Z \ar[r]^f & \X
}
\end{gathered}
\end{equation}
 Assume that the morphisms~$\pi$ and~$\rho$ are flat and affine,
and the ind-scheme $\X$ is ind-Noetherian.
 The aim of this Section~\ref{pro-derived-inverse-image-subsecn} is to
construct the left derived functor of inverse image
\begin{equation} \label{pro-left-derived-inverse-image-eqn}
 \boL g^*\:\sD(\bY_\X\flat)\lrarrow\sD(\bW_\Z\flat)
\end{equation}
acting from the derived category of the exact category of $\X$\+flat
pro-quasi-coherent pro-sheaves on $\bY$ to the derived category of
the exact category of $\Z$\+flat pro-quasi-coherent
pro-sheaves on~$\bW$.

\begin{lem} \label{relatively-homotopy-flat-and-acyclic-lemma}
 Let\/ $\X$ be an ind-Noetherian ind-scheme, and let $\pi\:\bY\rarrow\X$
be a flat affine morphism of ind-schemes.
 Let\/ $\bfF^\bu$ be a relatively homotopy flat complex of
flat pro-quasi-coherent pro-sheaves on\/ $\bY$ (as defined in
Section~\ref{relatively-homotopy-flat-subsecn}).
 Assume that the complex\/ $\bfF^\bu$ is acyclic in the exact
category\/ $\bY_\X\flat$.
 Then, for any complex of quasi-coherent torsion sheaves $\bM^\bu$
on\/ $\bY$, the complex of quasi-coherent torsion sheaves\/
$\pi_*(\bfF^\bu\ot_\bY\bM^\bu)$ on\/ $\X$ is coacyclic.
\end{lem}

\begin{proof}
 This assertion is implicit in the construction of the derived
functor~$\ot_\bY^\boL$ \,\eqref{derived-tensor-of-pro-and-torsion}
in Section~\ref{left-derived-tensor-products-subsecn}.
 Explicitly, by Proposition~\ref{homotopy-Y/X-flat-resolution} there
exists a homotopy $\bY/\X$\+flat complex $\brG^\bu$ of quasi-coherent
torsion sheaves on $\bY$ together with a morphism of complexes
$\brG^\bu\rarrow\bM^\bu$ whose cone $\bN^\bu$ has the property that
its direct image under~$\pi$ is coacyclic in $\X\tors$.
 Then the complex $\pi_*(\bfF^\bu\ot_\bY\brG^\bu)$ is coacyclic in
$\X\tors$, since the complex $\bfF^\bu$ is acyclic in $\bY_\X\flat$
and the complex $\brG^\bu$ is homotopy $\bY/\X$\+flat.
 Furthermore, the complex $\pi_*(\bfF^\bu\ot_\bY\bN^\bu)$ is also
coacyclic in $\X\tors$, since $\bfF^\bu$ be a relatively homotopy flat
complex of flat pro-quasi-coherent pro-sheaves on $\bY$ and
the complex $\pi_*(\bN^\bu)$ is coacyclic in $\X\tors$ (see
condition~(ii) in Section~\ref{relatively-homotopy-flat-subsecn}).
 It follows that the complex $\pi_*(\bfF^\bu\ot_\bY\bM^\bu)$ is
coacyclic in $\X\tors$.
\end{proof}

\begin{prop} \label{relatively-homotopy-flat-and-acyclic-prop}
 Let\/ $\X$ be an ind-Noetherian ind-scheme, and let $\pi\:\bY\rarrow\X$
be a flat affine morphism of ind-schemes.
 Let\/ $\bfF^\bu$ be a relatively homotopy flat complex of
flat pro-quasi-coherent pro-sheaves on\/~$\bY$.
 Assume that the complex\/ $\bfF^\bu$ is acyclic in the exact
category\/ $\bY_\X\flat$ of\/ $\X$\+flat pro-quasi-coherent pro-sheaves
on\/~$\bY$.
 Then the complex\/ $\bfF^\bu$ is also acyclic in the exact category\/
$\bY\flat$ of flat pro-quasi-coherent pro-sheaves on\/~$\bY$.
\end{prop}

\begin{proof}
 Let $X\subset\X$ be a closed subscheme with the closed immersion
morphism $i\:X\rarrow\X$.
 Put $\bnY=X\times_\X\bY$, and denote by $k\:\bnY\rarrow\bY$
the natural closed immersion.
 In view of Lemma~\ref{flat-pro-sheaves-complex-acyclicity-criterion},
it suffices to show that the complex $k^*\bfF^\bu$ is
acyclic in $\bnY\flat$.
 For this purpose, we will show that the complex of quasi-coherent
sheaves $k^*\bfF^\bu\ot_{\cO_\bnY}\bcN$ on $\bnY$ is acyclic for
any quasi-coherent sheaf $\bcN$ on~$\bnY$.

 Since $\bnY$ is a scheme, we can consider $\bcN$ as a quasi-coherent
torsion sheaf on~$\bnY$; then $k^*\bfF^\bu\ot_{\cO_\bnY}\bcN=
k^*\bfF^\bu\ot_\bnY\bcN$ is viewed as a complex of quasi-coherent
torsion sheaves on~$\bnY$.
 As the functor $k_*\:\bnY\qcoh\rarrow\bY\tors$ is exact and faithful
(see Lemma~\ref{closed-subschemes-torsion-subcategories}(a)),
acyclicity of this complex is equivalent to acyclicity of the complex
$k_*(k^*\bfF^\bu\ot_\bnY\bcN)$ in $\bY\tors$.
 By Lemma~\ref{torsion-pro-projection-formula-second}, we have
an isomorphism in $\sC(\bY\tors)$
$$
 k_*(k^*\bfF^\bu\ot_\bnY\bcN)\simeq\bfF^\bu\ot_\bY k_*\bcN.
$$

 By Lemma~\ref{relatively-homotopy-flat-and-acyclic-lemma},
the complex $\pi_*(\bfF^\bu\ot_\bY k_*\bcN)$ is coacyclic, hence
acyclic, in $\X\tors$.
 As the functor $\pi_*\:\bY\tors\rarrow\X\tors$ is exact and faithful
by Lemma~\ref{affine-torsion-direct-image}, it follows that the complex
$\bfF^\bu\ot_\bY k_*\bcN$ is acyclic in $\bY\tors$.
\end{proof}

 The following corollary plays the key role.

\begin{cor} \label{relatively-homotopy-flat-and-acyclic-cor}
 In the context of diagram~\eqref{square-of-ind-schemes},
let\/ $\bfF^\bu$ be a relatively homotopy flat complex of
flat pro-quasi-coherent pro-sheaves on\/~$\bY$.
 Assume that the complex\/ $\bfF^\bu$ is acyclic in the exact
category\/ $\bY_\X\flat$.
 Then the complex $g^*\bfF^\bu$ of flat pro-quasi-coherent
pro-sheaves on\/ $\bW$ is acyclic in the exact category\/
$\bW_\Z\flat$.
\end{cor}

\begin{proof}
 By Proposition~\ref{relatively-homotopy-flat-and-acyclic-prop},
the complex $\bfF^\bu$ is acyclic in $\bY\flat$.
 As the direct image functor $g^*\:\bY\flat\rarrow\bW\flat$ is exact,
it follows immediately that the complex $g^*\bfF^\bu$ is acyclic
in $\bW\flat$, hence also in $\bW_\Z\flat$.
\end{proof}

 Using Proposition~\ref{relatively-homotopy-flat-resolution}
and Corollary~\ref{relatively-homotopy-flat-and-acyclic-cor},
the left derived functor~\eqref{pro-left-derived-inverse-image-eqn}
can be constructed following the general approach of a ``derived
functor in the sense of Deligne''~\cite[1.2.1--2]{Del},
\cite[Lemma~6.5.2]{Psemi}.

 Let $\bfP^\bu$ be a complex of $\X$\+flat pro-quasi-coherent
pro-sheaves on~$\bY$.
 By Proposition~\ref{relatively-homotopy-flat-resolution}, there exists
a relatively homotopy flat complex of flat pro-quasi-coherent
pro-sheaves $\bfF^\bu$ on $\bY$ together with a morphism of
complexes $\bfF^\bu\rarrow\bfP^\bu$ whose cone is acyclic
in $\bY_\X\flat$.
 Put
$$
 \boL g^*(\bfP^\bu)=g^*(\bfF^\bu)\,\in\,\sD(\bW_\Z\flat).
$$
 Notice that $g^*(\bfF^\bu)$ is a complex of flat pro-quasi-coherent
pro-sheaves on $\bW$ (see
Section~\ref{flat-pro-sheaves-subsecn}); hence it is also a complex
of $\Z$\+flat pro-quasi-coherent pro-sheaves on $\bW$, as
the morphism~$\rho$ is flat and affine by assumption
(as per the discussion in Section~\ref{pro-flat-over-base-subsecn}).

 Let $a\:\bfP^\bu\rarrow\bfQ^\bu$ be a morphism of complexes of
$\X$\+flat pro-quasi-coherent pro-sheaves on~$\bY$, and let
$\bfF^\bu\rarrow\bfP^\bu$ and $\bfG^\bu\rarrow\bfQ^\bu$ be two
morphisms in $\sC(\bY_\X\flat)$ with the cones acyclic in $\bY_\X\flat$
such that both the complexes $\bfF^\bu$ and $\bfG^\bu$ are relatively
homotopy flat complexes of flat pro-quasi-coherent pro-sheaves on~$\bY$.
 In order to construct the induced morphism
$$
 \boL g^*(a)\:g^*(\bfF^\bu)\lrarrow g^*(\bfG^\bu)
$$
in $\sD(\bW_\Z\flat)$, choose a complex $\bfR^\bu$ in $\bY_\X\flat$
together with morphisms $\bfR^\bu\rarrow\bfF^\bu$ and $\bfR^\bu
\rarrow\bfG^\bu$ in $\sC(\bY_\X\flat)$ such that the diagram
$\bfR^\bu\rarrow\bfF^\bu\rarrow\bfP^\bu\rarrow\bfQ^\bu$ and
$\bfR^\bu\rarrow\bfG^\bu\rarrow\bfQ^\bu$ is commutative in
$\sK(\bY_\X\flat)$ and the cone of the morphism
$\bfR^\bu\rarrow\bfF^\bu$ is acyclic in $\bY_\X\flat$.
 Using Proposition~\ref{relatively-homotopy-flat-resolution}, choose
a relatively homotopy flat complex of flat pro-quasi-coherent
pro-sheaves $\bfH^\bu$ on $\bY$ together with a morphism of complexes
$\bfH^\bu\rarrow\bfR^\bu$ whose cone is acyclic in $\bY_\X\flat$.

 Then the cone $\bfS^\bu$ of the composition $s\:\bfH^\bu\rarrow
\bfR^\bu\rarrow\bfF^\bu$ is a relatively homotopy flat complex of flat
pro-quasi-coherent pro-sheaves on $\bY$ which is acyclic
in $\bY_\X\flat$.
 By Corollary~\ref{relatively-homotopy-flat-and-acyclic-cor},
the complex $g^*(\bfS^\bu)$ is acyclic in $\bW_\Z\flat$.
 Denote by~$b$ the composition $\bfH^\bu\rarrow\bfR^\bu\rarrow\bfG^\bu$.
 Now the fraction formed by the morphism $g^*(b)\:g^*(\bfH^\bu)\rarrow
g^*(\bfG^\bu)$ and the isomorphism $g^*(s)\:g^*(\bfH^\bu)\rarrow
g^*(\bfF^\bu)$ represents the desired morphism $\boL g^*(a)\:
g^*(\bfF^\bu)\lrarrow g^*(\bfG^\bu)$ in $\sD(\bW_\Z\flat)$.

\subsection{Derived inverse image of torsion sheaves}
\label{torsion-derived-inverse-image-subsecn}
 Suppose that we are given a commutative triangle diagram of morphisms
of ind-schemes
\begin{equation} \label{triangle-of-ind-schemes}
\begin{gathered}
\xymatrix{
 \bW \ar[rr]^g \ar[rd]_\rho && \bY \ar[ld]^\pi \\
 & \X
}
\end{gathered}
\end{equation}
 Assume that the morphisms $\pi$ and~$\rho$ are flat and affine,
and the ind-scheme $\X$ is ind-Noetherian.
 The aim of Section~\ref{torsion-derived-inverse-image-subsecn} is
to construct the left derived functor of inverse image
\begin{equation} \label{torsion-left-derived-inverse-image-eqn}
 \boL g^*\:\sD_\X^\si(\bY\tors)\lrarrow\sD_\X^\si(\bW\tors)
\end{equation}
acting from the $\bY/\X$\+semiderived category of 
quasi-coherent torsion sheaves on $\bY$ to
the $\bW/\X$\+semiderived category of quasi-coherent
torsion sheaves on~$\bW$.

\begin{lem} \label{homotopy-Y/X-flat-and-semiacyclic-lemma}
 Let\/ $\X$ be an ind-Noetherian ind-scheme, and let $\pi\:\bY\rarrow\X$
be a flat affine morphism of ind-schemes.
 Let\/ $\brG^\bu$ be a homotopy\/ $\bY/\X$\+flat complex of
quasi-coherent torsion sheaves on\/ $\bY$ (as defined in
Section~\ref{relatively-homotopy-flat-subsecn}).
 Assume that the complex\/ $\pi_*(\brG^\bu)$ of quasi-coherent
torsion sheaves on\/ $\X$ is coacyclic.
 Then, for any complex of\/ $\bY/\X$\+flat pro-quasi-coherent
pro-sheaves\/ $\bfP^\bu$ on\/ $\bY$, the complex of quasi-coherent
torsion sheaves\/ $\pi_*(\bfP^\bu\ot_\bY\brG^\bu)$ on\/ $\X$
is coacyclic.
\end{lem}

\begin{proof}
 Similarly to Lemma~\ref{relatively-homotopy-flat-and-acyclic-lemma},
this assertion is implicit in the construction of the derived
functor~$\ot_\bY^\boL$ \,\eqref{derived-tensor-of-pro-and-torsion}
in Section~\ref{left-derived-tensor-products-subsecn}.
 Explicitly, by Proposition~\ref{relatively-homotopy-flat-resolution}
there exists a relatively homotopy flat complex of flat
pro-quasi-coherent pro-sheaves\/ $\bfF^\bu$ on\/ $\bY$ together with
a morphism of complexes\/ $\bfF^\bu\rarrow\bfP^\bu$ whose cone
$\bfQ^\bu$ is acyclic in $\bY_\X\flat$.
 Then the complex $\pi_*(\bfF^\bu\ot_\bY\brG^\bu)$ is coacyclic in
$\X\tors$, since $\bfF^\bu$ is a relatively homotopy flat complex of
flat pro-quasi-coherent pro-sheaves on $\bY$ and the complex
$\pi_*(\brG^\bu)$ is coacyclic in $\X\tors$.
 Furthermore, the complex $\pi_*(\bfQ^\bu\ot_\bY\brG^\bu)$ is coacyclic
in $\X\tors$, since the complex $\bfQ^\bu$ is acyclic in $\bY_\X\flat$
and $\brG^\bu$ is a homotopy $\bY/\X$\+flat complex of quasi-coherent
torsion sheaves on\/~$\bY$.
 It follows that the complex $\pi_*(\bfP^\bu\ot_\bY\brG^\bu)$ is
coacyclic in $\X\tors$.
\end{proof}

 The next proposition is the key technical assertion.

\begin{prop} \label{homotopy-Y/X-flat-and-semiacyclic-prop}
 In the context of the diagram~\eqref{triangle-of-ind-schemes}, let\/
$\brF^\bu$ be a homotopy\/ $\bY/\X$\+flat complex of quasi-coherent
torsion sheaves on\/~$\bY$.
 Assume that the complex\/ $\pi_*(\brF^\bu)$ of quasi-coherent torsion
sheaves on\/ $\X$ is coacyclic.
 Then the complex $\rho_*g^*(\brF^\bu)$ of quasi-coherent torsion
sheaves on\/ $\X$ is coacyclic, too.
\end{prop}

\begin{proof}
 Let us show that the morphism of ind-schemes $g\:\bW\rarrow\bY$ is
affine (whenever the morphisms~$\pi$ and $\rho=\pi g$ are).
 First of all, the morphism~$g$ is ``representable by schemes''
(since the morphisms~$\pi$ and $\rho=\pi g$ are).
 Indeed, let $X\subset\X$ be a closed subscheme.
 Then $\bnY=X\times_\X\bY$ is a closed subscheme in~$\bY$ (and any
closed subscheme in $\bY$ is contained in a closed subscheme of
this form).
 Therefore, $\bnW=\bnY\times_\bY\bW=X\times_\X\bW$ is a closed
subscheme in~$\bW$.

 We know that the morphisms of schemes $\bnY\rarrow X$ and
$\bnW\rarrow X$ are affine, and we have to show that the morphism
$\bnW\rarrow\bnY$ is.
 Let $X=\bigcup_\alpha U_\alpha$ be an affine open covering
of~$X$.
 Then $\bnY=\bigcup_\alpha (U_\alpha\times_X\bnY)$ is an affine open
covering of~$\bnY$; and the schemes $(U_\alpha\times_X\bnY)
\times_\bnY\bnW=U_\alpha\times_X\bnW$ are affine
(cf.~\cite[Tag~01SG]{SP}).

 Now we have $\rho_*g^*(\brF^\bu)\simeq
\pi_*g_*g^*(\brF^\bu)$, since $\rho=\pi g$.
 Furthermore, $g^*\brF^\bu\simeq\fO_\bW\ot_\bW g^*\brF^\bu$, where
$\fO_\bW\in\bW\flat$ is the ``pro-structure pro-sheaf'' on~$\bW$.
 By Lemma~\ref{torsion-pro-projection-formula-lemma},
$$
 g_*g^*(\brF^\bu)\simeq g_*(\fO_\bW\ot_\bW g^*\brF^\bu)
 \simeq g_*(\fO_\bW)\ot_\bY\brF^\bu.
$$
 Notice that the pro-quasi-coherent pro-sheaf $g_*(\fO_\bW)$ on
$\bY$ is $\X$\+flat, because $\pi_*g_*(\fO_\bW)\simeq\rho_*(\fO_\bW)$
and the morphism of ind-schemes $\rho\:\bW\rarrow\X$ is flat
by assumption, hence the direct image functor $\rho_*\:\bW\pro\rarrow
\X\pro$ takes flat pro-quasi-coherent pro-sheaves on $\bW$
to flat pro-quasi-coherent pro-sheaves on~$\X$.
 Applying Lemma~\ref{homotopy-Y/X-flat-and-semiacyclic-lemma}, we
conclude that the complex $\rho_*g^*(\brF^\bu)\simeq
\pi_*(g_*(\fO_\bW)\ot_\bY\brF^\bu)$ is coacyclic in $\X\tors$.
\end{proof}

 Similarly to Section~\ref{pro-derived-inverse-image-subsecn}, we use
Propositions~\ref{homotopy-Y/X-flat-resolution}
and~\ref{homotopy-Y/X-flat-and-semiacyclic-prop} in order to construct
the left derived functor~\eqref{torsion-left-derived-inverse-image-eqn}
following the general approach of~\cite[1.2.1--2]{Del}
and~\cite[Lemma~6.5.2]{Psemi}.

 Let $\bM^\bu$ be a complex of quasi-coherent torsion sheaves on~$\bY$.
 By Proposition~\ref{homotopy-Y/X-flat-resolution}, there exists
a homotopy $\bY/\X$\+flat complex of quasi-coherent torsion sheaves
$\brF^\bu$ on $\bY$ together with a morphism of complexes
$\brF^\bu\rarrow\bM^\bu$ whose cone has the property that its direct
image is coacyclic in $\X\tors$.
 Put
$$
 \boL g^*(\bM^\bu)=g^*(\brF^\bu)\,\in\,\sD_\X^\si(\bW\tors).
$$
 The action of the functor $\boL g^*$ on morphisms in
$\sD_\X^\si(\bY\tors)$ is constructed in the same way as in
Section~\ref{pro-derived-inverse-image-subsecn}
(we omit the obvious details).

\subsection{Derived restriction with supports in the relative context}
\label{relative-derived-supports-subsecn}
 Suppose that we are given a pullback diagram of morphisms of
ind-schemes (so $\bW=\Z\times_\X\bY$)
\begin{equation} \label{pullback-of-ind-schemes}
\begin{gathered}
\xymatrix{
 \bW \ar[r]^k \ar[d]_\rho & \bY \ar[d]^\pi \\
 \Z \ar[r]^i & \X
}
\end{gathered}
\end{equation}
 Assume that the morphism~$\pi$ (hence also the morphism~$\rho$) is
flat and affine, the morphism~$i$ (hence also the morphism~$k$) is
a closed immersion, and the ind-scheme $\X$ (hence also
the ind-scheme~$\Z$) is ind-Noetherian.
 The aim of Section~\ref{relative-derived-supports-subsecn} is to
construct the right derived functor
\begin{equation} \label{relative-derived-supports-eqn}
 \boR k^!\:\sD_\X^\si(\bY\tors)\lrarrow\sD_\Z^\si(\bW\tors)
\end{equation}
acting from the $\bY/\X$\+semiderived category of quasi-coherent torsion
sheaves on $\bY$ to the $\bW/\Z$\+semiderived category of quasi-coherent
torsion sheaves on~$\bW$.

\begin{lem} \label{torsion-shriek-base-change}
 In the context of the diagram~\eqref{pullback-of-ind-schemes},
there is a natural isomorphism $i^!\pi_*\simeq\rho_*k^!$ of functors\/
$\bY\tors\rarrow\Z\tors$.
\end{lem}

\begin{proof}
 Weaker assumptions than stated above are sufficient for the validity
of this lemma, which is an ind-scheme version of
Lemma~\ref{reasonable-base-change}(a).
 It suffices to assume that $\bW=\Z\times_\X\bY$, the morphism~$\pi$
(hence also~$\rho$) is ``representable by schemes'', and $i$
(hence also~$k$) is a reasonable closed immersion.

 Indeed, let $\bN$ be a quasi-coherent torsion sheaf on~$\bY$,
and let $Z\subset\Z$ be a reasonable closed subscheme.
 Put $\bnW=Z\times_\Z\bW$; so $\bnW$ is a reasonable closed subscheme
in~$\bW$.
 The scheme $Z$ can be also viewed as a reasonable closed subscheme
in $\X$, embedded via~$i$; and the scheme $\bnW$ can be viewed as
a reasonable closed subscheme in $\bY$, embedded via~$k$.
 Denote by $\rho_Z\:\bnW\rarrow Z$ the natural morphism.
 In order to obtain the desired isomorphism $i^!\pi_*\bN\simeq
\rho_*k^!\bN$ in $\Z\tors$, one constructs a compatible system of
isomorphisms in $Z\qcoh$
$$
 (i^!\pi_*\bN)_{(Z)}\simeq(\pi_*\bN)_{(Z)}\simeq\rho_Z{}_*(\bN_{(\bnW)})
 \simeq\rho_Z{}_*((k^!\bN)_{(\bnW)})\simeq(\rho_*k^!\bN)_{(Z)}
$$
for all the reasonable closed subschemes $Z\subset\Z$.
 (See the definition of the direct image functor for quasi-coherent
torsion sheaves in Section~\ref{torsion-direct-images-subsecn}.)
\end{proof}

 By Proposition~\ref{over-ind-Noetherian-semiderived-prop}, we have
natural equivalences of triangulated categories
$\sD_\X^\si(\bY\tors)\simeq\sD(\bY\tors_{\X\dinj})$ and
$\sD_\Z^\si(\bW\tors)\simeq\sD(\bW\tors_{\Z\dinj})$.
 We will use these triangulated equivalences in order to construct
the right derived functor~\eqref{relative-derived-supports-eqn}.

\begin{lem} \label{support-exact-between-injective-over-base}
 In the context of the diagram~\eqref{pullback-of-ind-schemes},
the functor $k^!\:\bY\tors\rarrow\bW\tors$ restricts to an exact
functor between exact categories\/ $\bY\tors_{\X\dinj}\rarrow
\bW\tors_{\Z\dinj}$.
\end{lem}

\begin{proof}
 The functor $i^!\:\X\tors\rarrow\Z\tors$, being right adjoint to
an exact functor $i_*\:\Z\tors\rarrow\X\tors$, takes injective objects
to injective objects.
 In view of Lemma~\ref{torsion-shriek-base-change}, it follows that
the functor $k^!\:\bY\tors\rarrow\bW\tors$ takes $\X$\+injective
quasi-coherent torsion sheaves to $\Z$\+injective ones.

 Let $0\rarrow\bL\rarrow\bM\rarrow\bN\rarrow0$ be a short exact sequence
of $\X$\+injective quasi-coherent torsion sheaves on~$\bY$.
 Then $0\rarrow\pi_*\bL\rarrow\pi_*\bM\rarrow\pi_*\bN\rarrow0$ is
a split short exact sequence of quasi-coherent torsion sheaves on~$\X$.
 Hence $0\rarrow i^!\pi^*\bL\rarrow i^*\pi_*\bM\rarrow i^!\pi_*\bN
\rarrow0$ is a split short exact sequence of quasi-coherent torsion
sheaves on~$\Z$.
 As $\rho_*k^!\simeq i^!\pi_*$, it follows that $0\rarrow k^!\bL\rarrow
k^!\bM\rarrow k^!\bN\rarrow0$ is a short exact sequence of
$\Z$\+injective quasi-coherent torsion sheaves on $\bW$ (because
the functor $\rho_*\:\bW\tors\rarrow\Z\tors$ is exact and faithful
by Lemma~\ref{affine-torsion-direct-image}).
\end{proof}

 In view of Lemma~\ref{support-exact-between-injective-over-base},
the functor $k^!\:\bY\tors_{\X\dinj}\rarrow\bW\tors_{\Z\dinj}$ induces
a well-defined triangulated functor between the derived categories
of the two exact categories,
$$
 k^!\:\sD(\bY\tors_{\X\dinj})\lrarrow\sD(\bW\tors_{\Z\dinj}).
$$
 Using the triangulated equivalences
$\sD_\X^\si(\bY\tors)\simeq\sD(\bY\tors_{\X\dinj})$ and
$\sD_\Z^\si(\bW\tors)\simeq\sD(\bW\tors_{\Z\dinj})$, we obtain
the desired right derived
functor~\eqref{relative-derived-supports-eqn}.
 This construction is also a particular case of
the construction of~\cite[1.2.1--2]{Del} or~\cite[Lemma~6.5.2]{Psemi},
which was used in Sections~\ref{pro-derived-inverse-image-subsecn}\+-%
\ref{torsion-derived-inverse-image-subsecn}.

\subsection{Composition of derived inverse images of pro-sheaves}
 Suppose that we are given a composable pair of commutative square
diagrams of morphisms of ind-schemes
\begin{equation} \label{two-squares-of-ind-schemes}
\begin{gathered}
\xymatrix{
 \bV \ar[r]^h \ar[d]_\sigma &
 \bW \ar[r]^g \ar[d]_\rho & \bY \ar[d]^\pi \\
 \U \ar[r]^t & \Z \ar[r]^f & \X
}
\end{gathered}
\end{equation}
 Assume that the morphisms~$\pi$, $\rho$, and~$\sigma$ are flat and
affine, and the ind-schemes $\X$ and $\Z$ are ind-Noetherian.
 Consider the composition of left derived
functors~\eqref{pro-left-derived-inverse-image-eqn}
constructed in Section~\ref{pro-derived-inverse-image-subsecn}
$$
 \sD(\bY_\X\flat)\overset{\boL g^*}\lrarrow
 \sD(\bW_\Z\flat)\overset{\boL h^*}\lrarrow
 \sD(\bV_\U\flat).
$$

\begin{prop} \label{composition-of-derived-inverse-images}
 There is a natural isomorphism\/ $\boL h^*\circ\boL g^*\simeq
\boL(gh)^*$ of triangulated functors\/ $\sD(\bY_\X\flat)\rarrow
\sD(\bV_\U\flat)$.
\end{prop}

\begin{proof}
 The related underived isomorphism is obvious: clearly, one has
$h^*g^*\bfP^\bu\simeq(gh)^*\bfP^\bu$ in $\sC(\bV\pro)$ for any
complex of pro-quasi-coherent pro-sheaves $\bfP^\bu$ on~$\bY$.

 The best possible argument for the proof of the composition of
derived functors being isomorphic to the derived functor of
the composition would be to show that the functor~$g^*$ takes
complexes adjusted to $\boL g^*$ and $\boL(gh)^*$ to complexes
adjusted to~$\boL h^*$.
 In our context, this would mean showing that, for any relatively
homotopy flat complex of flat pro-quasi-coherent pro-sheaves
$\bfF^\bu$ on $\bY$ (relative to the morphism $\pi\:\bY\rarrow\X$),
the complex of flat pro-quasi-coherent pro-sheaves $g^*(\bfF^\bu)$
on $\bW$ is also relatively homotopy flat
(relative to the morphism $\rho\:\bW\rarrow\Z$).

 However, we do \emph{not} know how to prove this preservation of
relative homotopy flatness in full generality.
 Instead, we will show that, given a complex $\bfP^\bu$ in
$\bY_\X\flat$, a relatively homotopy flat complex $\bfF^\bu$ endowed
with a morphism $\bfF^\bu\rarrow\bfP^\bu$ with the cone acyclic in
$\bY_\X\flat$ \emph{can be chosed in such a way that} the complex
of flat pro-quasi-coherent pro-sheaves $g^*(\bfF^\bu)$ on $\bW$
is a relatively homotopy flat complex.

 Let us recall the construction of the relatively homotopy flat
resolutions from Proposition~\ref{relatively-homotopy-flat-resolution}.
 All the complexes of flat pro-quasi-coherent pro-sheaves on $\bY$
which can be obtained from the inverse images of complexes of flat
pro-quasi-coherent pro-sheaves on $\X$ using the operations of cone,
infinite coproduct, and the passage to a homotopy equivalent complex,
are relatively homotopy flat
(by Lemma~\ref{relatively-flat-complexes-lemma});
and any complex $\bfP^\bu\in\sC(\bY_\X\flat)$ admits a relatively
homotopy flat resolution of this form.

 Assume that the complex $\bfF^\bu$ is homotopy equivalent to
a complex of flat pro-quasi-coherent pro-sheaves on $\bY$ obtained
from inverse images of complexes of flat pro-quasi-coherent
pro-sheaves on $\X$ using the operations of cone and infinite coproduct.
 Then the complex $g^*(\bfF^\bu)$ is homotopy equivalent to a complex
of flat pro-quasi-coherent pro-sheaves on $\bW$ similarly obtained from
the inverse images of complexes of flat pro-quasi-coherent pro-sheaves
on~$\Z$.
 This follows from the natural isomorphism $g^*\pi^*\simeq\rho^*f^*$
of functors $\X\flat\rarrow\bW\flat$ and the commutation of inverse
images of (flat) pro-quasi-coherent pro-sheaves with the coproducts.
\end{proof}

\subsection{External tensor products in the relative context}
\label{relative-external-products-subsecn}
 Let $\X'$ and $\X''$ be ind-schemes over~$\kk$, and let $\pi'\:\bY'
\rarrow\X'$ and $\pi''\:\bY''\rarrow\X''$ be affine morphisms
of schemes.
 Let $\pi'\times_\kk\pi''\:\bY'\times_\kk\bY''\rarrow\X'\times_\kk\X''$
be the induced morphism of the Cartesian products.
 Clearly, $\pi'\times_\kk\pi''$ is an affine morphism of ind-schemes.

 The functor of external tensor product of pro-quasi-coherent
pro-sheaves~\eqref{external-product-pro-sheaves-eqn}
$$
 \bt_\kk\:\bY'\pro\times\bY''\pro\lrarrow(\bY'\times_\kk\bY'')\pro
$$
was constructed in Section~\ref{external-of-pro-subsecn}.

\begin{lem} \label{direct-image-of-external-product-of-pro-sheaves}
 Let\/ $\bfQ'$ be a pro-quasi-coherent pro-sheaf on\/ $\bY'$ and\/
$\bfQ''$ be a pro-quasi-coherent pro-sheaf on\/~$\bY''$.
 Then there is a natural isomorphism
$$
 (\pi'\times_\kk\pi'')_*(\bfQ'\bt_\kk\bfQ'')\,\simeq\,
 \pi'_*\bfQ'\bt_\kk\pi''_*\bfQ''
$$
of pro-quasi-coherent pro-sheaves on\/ $\X'\times_\kk\X''$.
\end{lem}

\begin{proof}
 Follows from Lemma~\ref{direct-image-of-external-product}.
\end{proof}

\begin{lem} \label{X-flat-on-Y-external-product}
\textup{(a)} Let\/ $\bfG'$ be an\/ $\X'$\+flat pro-quasi-coherent
pro-sheaf on\/ $\bY'$ and\/ $\bfG''$ be an\/ $\X''$\+flat
pro-quasi-coherent pro-sheaf on\/~$\bY''$.
 Then\/ $\bfG'\bt_\kk\bfG''$ is an $(\X'\times_\kk\X'')$\+flat
pro-quasi-coherent pro-sheaf on\/ $\bY'\times_\kk\bY''$. \par
\textup{(b)} The external tensor product functor
\begin{equation} \label{external-product-X-flat-on-Y}
 \bt_\kk\:\bY'_{\X'}\flat\times\bY''_{\X''}\flat
 \lrarrow(\bY'\times_\kk\bY'')_{(\X'\times_\kk\X'')}\flat
\end{equation}
is exact (as a functor between exact categories) and preserves
direct limits (in particular, coproducts).
\end{lem}

\begin{proof}
 Part~(a) follows from
Lemma~\ref{direct-image-of-external-product-of-pro-sheaves}
and formula~\eqref{external-product-flat-pro-sheaves-eqn}.
 Part~(b) follows directly from the definitions of the exact
structures involved and Lemma~\ref{qcoh-external-tensor-exact}.
\end{proof}

\begin{lem} \label{flat-over-base-complexes-external-tensor-exactness}
 Let\/ $\bfG'{}^\bu$ be a complex of\/ $\X'$\+flat pro-quasi-coherent
pro-sheaves on\/ $\bY'$ and\/ $\bfG''{}^\bu$ be a complex of\/
$\X''$\+flat pro-quasi-coherent pro-sheaves on\/~$\bY''$.
 Assume that the complex\/ $\bfG'{}^\bu$ is acyclic in\/
$\bY'_{\X'}\flat$.
 Then the complex\/ $\bfG'{}^\bu\bt_\kk\bfG''{}^\bu$ is acyclic in
$(\bY'\times_\kk\bY'')_{(\X'\times_\kk\X'')}\flat$.
\end{lem}

\begin{proof}
 The results of
Lemmas~\ref{flat-over-base-direct-image-acyclicity-criterion}
and~\ref{direct-image-of-external-product-of-pro-sheaves} reduce
the question to
Lemma~\ref{flat-pro-complexes-external-tensor-exactness}.
\end{proof}

 It follows from
Lemma~\ref{flat-over-base-complexes-external-tensor-exactness} that
the external tensor product of pro-quasi-coherent pro-sheaves which
are flat over the base is well-defined as a functor between
the derived categories of the respective exact categories,
\begin{equation} \label{derived-flat-over-base-external-product}
 \bt_\kk\:\sD(\bY'_{\X'}\flat)\times\sD(\bY''_{\X''}\flat)\lrarrow
 \sD((\bY'\times_\kk\bY'')_{(\X'\times_\kk\X'')}\flat).
\end{equation}

 Now assume additionally that the ind-schemes $\X'$ and $\X''$ are
reasonable (then so are the ind-schemes $\bY'$ and~$\bY''$).
 The functor of external tensor product of quasi-coherent torsion
sheaves~\eqref{external-product-torsion-sheaves-eqn}
$$
 \bt_\kk\:\bY'\tors\times\bY''\tors\lrarrow(\bY'\times_\kk\bY'')\tors
$$
was constructed in Section~\ref{external-of-torsion-subsecn}.

\begin{lem} \label{torsion-complexes-external-tensor-semiacyclic}
 Let $\bN^{\prime\bu}$ be a complex of quasi-coherent torsion sheaves
on\/ $\bY'$ and $\bN^{\prime\prime\bu}$ be a complex of quasi-coherent
torsion sheaves on\/~$\bY''$.
 Assume that the complex\/ $\pi'_*(\bN^{\prime\bu})$ of quasi-coherent
torsion sheaves on\/ $\X'$ is coacyclic.
 Then the complex $(\pi'\times_\kk\pi'')_*
(\bN^{\prime\bu}\bt_\kk\bN^{\prime\prime\bu})$ of quasi-coherent
torsion sheaves on\/ $\X'\times_\kk\X''$ is coacyclic, too.
\end{lem}

\begin{proof}
 Follows immediately from
Lemmas~\ref{torsion-direct-image-of-external-product}
and~\ref{torsion-complexes-external-tensor-coacyclic}.
\end{proof}

 It is clear from
Lemma~\ref{torsion-complexes-external-tensor-semiacyclic} that
the external tensor product is well-defined as a functor between
the semiderived categories of quasi-coherent torsion sheaves,
\begin{equation} \label{semiderived-torsion-external-product}
 \bt_\kk\:\sD_{\X'}^\si(\bY'\tors)\times\sD_{\X''}^\si(\bY''\tors)
 \lrarrow\sD_{(\X'\times_\kk\X'')}^\si((\bY'\times_\kk\bY'')\tors).
\end{equation}

\medskip
 We have not used the assumption of flatness of morphisms~$\pi'$
and~$\pi''$ in this Section~\ref{relative-external-products-subsecn}
yet, but it is worth noticing that the morphism of ind-schemes
$\pi'\times_\kk\pi''$ is flat whenever the morphisms $\pi'$ and~$\pi''$
are.
 This follows from the similar property of morphisms of schemes
over~$\kk$.

\subsection{Derived tensor product of pro-sheaves as derived restriction
to the diagonal} \label{derived-tensor-restriction-subsecn}
 Let $\X$ be an ind-scheme of ind-finite type over~$\kk$,
and let $\pi\:\bY\rarrow\X$ be a flat affine morphism of ind-schemes.
 Then the following commutative square diagram of morphisms of
ind-schemes
$$
\xymatrix{
 \bY \ar[rr]^-{\Delta_\bY} \ar[d]_\pi
 && \bY\times_\kk\bY \ar[d]^{\pi\times_\kk\pi} \\
 \X \ar[rr]^-{\Delta_\X} && \X\times_\kk\X
}
$$
is a particular case of the diagram~\eqref{square-of-ind-schemes}
from Section~\ref{pro-derived-inverse-image-subsecn}.

 Let $\bfP^\bu$ and $\bfQ^\bu$ be two complexes of $\X$\+flat
pro-quasi-coherent pro-sheaves on~$\bY$.
 Our aim is to construct a natural isomorphism
\begin{equation}
\label{derived-tensor-as-derived-restriction-eqn}
 \bfP^\bu\ot^{\bY,\boL}\bfQ^\bu\,\simeq\,
 \boL\Delta_\bY^*(\bfP^\bu\bt_\kk\bfQ^\bu)
\end{equation}
in the derived category $\sD(\bY_\X\flat)$.
 Here the derived functor of tensor product~$\ot^{\bY,\boL}$
\,\eqref{derived-tensor-of-pro-sheaves} was constructed in
Section~\ref{left-derived-tensor-products-subsecn},
the functor of external tensor product~$\bt_\kk$ was discussed in
Section~\ref{relative-external-products-subsecn}, and the derived
functor of inverse image $\boL\Delta_\bY^*$ was defined in
Section~\ref{pro-derived-inverse-image-subsecn}.
 Recall that the isomorphism of underived functors
$\bfP^\bu\ot^\bY\bfQ^\bu\,\simeq\,\Delta_\bY^*(\bfP^\bu\bt_\kk\bfQ^\bu)$
is provided by Lemma~\ref{pro-tensor-as-restricted-external}.
 The following proposition shows that the derived functors agree.

\begin{prop}
\label{derived-tensor-as-derived-restriction-prop}
 For any two complexes\/ $\bfP^\bu$ and\/ $\bfQ^\bu\in\sD(\bY_\X\flat)$,
a natural isomorphism of left derived
functors~\eqref{derived-tensor-as-derived-restriction-eqn}
holds in\/ $\sD(\bY_\X\flat)$.
\end{prop}

\begin{proof}
 Let $\bfF^\bu$ and $\bfG^\bu$ be two relatively homotopy flat complexes
of flat pro-quasi-coherent pro-sheaves on $\bY$ endowed with two 
morphisms of complexes $\bfF^\bu\rarrow\bfP^\bu$ and $\bfG^\bu\rarrow
\bfQ^\bu$ with the cones acyclic in $\bY_\X\flat$.
 Then the external tensor product $\bfF^\bu\bt_\kk\bfG^\bu$ is a complex
of flat pro-quasi-coherent pro-sheaves on $\bY\times_\kk\bY$ by
formula~\eqref{external-product-flat-pro-sheaves-eqn} from
Section~\ref{external-of-pro-subsecn}, and the cone of the morphism
$\bfF^\bu\bt_\kk\bfG^\bu\rarrow\bfP^\bu\bt_\kk\bfQ^\bu$ is acyclic in
the exact category $(\bY\times_\kk\bY)_{(\X\times_\kk\X)}\flat$ by
Lemma~\ref{flat-over-base-complexes-external-tensor-exactness}.

 A problem similar to the one in the proof of
Proposition~\ref{composition-of-derived-inverse-images} arises here.
 The best possible argument for a proof of the desired isomorphism
would be to show that the complex $\bfF^\bu\bt_\kk\bfG^\bu$ is
a relatively homotopy flat complex of flat pro-quasi-coherent
pro-sheaves on $\bY\times_\kk\bY$ (relative to the morphism
$\pi\times_\kk\pi\:\bY\times_\kk\bY\rarrow\X\times_\kk\X$).
 Then the isomorphism of the derived functors would follow immediately
from (their constructions and) the isomorphism of the underived ones.
 However, we do \emph{not} know how to prove this relative homotopy
flatness in full generality.
 Instead, we will show that the relatively homotopy flat resolutions
$\bfF^\bu$ and $\bfG^\bu$ of any two given complexes $\bfP^\bu$ and
$\bfQ^\bu$ \emph{can be chosen in such a way that} the external tensor
product $\bfF^\bu\bt_\kk\bfG^\bu$ is a relatively homotopy flat complex.

 Once again we recall the construction of the relatively homotopy flat
resolutions from Proposition~\ref{relatively-homotopy-flat-resolution}.
 All the complexes of flat pro-quasi-coherent pro-sheaves on $\bY$
which can be obtained from the inverse images of complexes of flat
pro-quasi-coherent pro-sheaves on $\X$ using the operations of cone,
infinite coproduct, and the passage to a homotopy equivalent complex,
are relatively homotopy flat; and any complex
$\bfP^\bu\in\sC(\bY_\X\flat)$ admits a relatively homotopy flat
resolution of this form.

 Now assume that both the complexes $\bfF^\bu$ and $\bfG^\bu$ are
homotopy equivalent to complexes of flat pro-quasi-coherent pro-sheaves
on $\bY$ obtained from the inverse images of complexes of flat
pro-quasi-coherent pro-sheaves on $\X$ using the operations of cone
and infinite coproduct.
 Then the complex $\bfF^\bu\bt_\kk\bfG^\bu$ is homotopy equivalent to
a complex of flat pro-quasi-coherent pro-sheaves on $\bY\times_\kk\bY$
similarly obtained from the inverse images of complexes of flat
pro-quasi-coherent pro-sheaves on $\X\times_\kk\X$.
 This follows from the commutation of external tensor products with
the inverse images (Lemma~\ref{pro-inverse-image-of-external-product}),
the commutation of external tensor products with the coproducts
(see Section~\ref{external-of-pro-subsecn} and
Lemma~\ref{X-flat-on-Y-external-product}(b)), and the preservation
of flatness by the external tensor products of pro-quasi-coherent
pro-sheaves on~$\X$
(formula~\eqref{external-product-flat-pro-sheaves-eqn}).
\end{proof}

\subsection{Semiderived equivalence and change of fiber}
\label{semiderived-equivalence-and-change-of-fiber-subsecn}
 Let us return to the setting of the commutative triangle diagram
of morphisms of ind-schemes~\eqref{triangle-of-ind-schemes} from
Section~\ref{torsion-derived-inverse-image-subsecn}
$$
\xymatrix{
 \bW \ar[rr]^g \ar[rd]_\rho && \bY \ar[ld]^\pi \\
 & \X
}
$$
where the morphisms $\pi$ and~$\rho$ are flat and affine, and
the ind-scheme $\X$ is ind-Noetherian.
 Assume further that $\X$ is ind-semi-separated and endowed with
a dualizing complex~$\rD^\bu$.
 Then Theorem~\ref{relative-triangulated-equivalence-thm} provides
triangulated equivalences
$$
 \pi^*\rD^\bu\ot_\bY{-}\,\:
 \sD(\bY_\X\flat)\,\simeq\,\sD_\X^\si(\bY\tors)
 \,:\!\boR\fHom_{\bY\qc}(\pi^*\rD^\bu,{-})
$$
and
$$
 \rho^*\rD^\bu\ot_\bW{-}\,\:
 \sD(\bW_\X\flat)\,\simeq\,\sD_\X^\si(\bW\tors)
 \,:\!\boR\fHom_{\bW\qc}(\rho^*\rD^\bu,{-}).
$$
 Here the notation $\boR\fHom$ stands for the fact that the functors
$\fHom$ should be applied to complexes of \emph{$\X$\+injective}
quasi-coherent torsion sheaves (while the functors $\ot_\bY$
and $\ot_\bW$ are applied to arbitrary complexes of $\X$\+flat
pro-quasi-coherent pro-sheaves).

\begin{prop} \label{semiderived-equivalence-change-of-fiber-prop}
 In the context above, the triangulated equivalences\/
$\sD(\bY_\X\flat)\simeq\sD_\X^\si(\bY\tors)$ and\/
$\sD(\bW_\X\flat)\simeq\sD_\X^\si(\bW\tors)$ from
Theorem~\ref{relative-triangulated-equivalence-thm} transform
the left derived functor\/ $\boL g^*\:\sD(\bY_\X\flat)\rarrow
\sD(\bW_\X\flat)$ \,\eqref{pro-left-derived-inverse-image-eqn}
from Section~\ref{pro-derived-inverse-image-subsecn}
into the left derived functor\/ $\boL g^*\:\sD_\X^\si(\bY\tors)\rarrow
\sD_\X^\si(\bW\tors)$ \,\eqref{torsion-left-derived-inverse-image-eqn}
from Section~\ref{torsion-derived-inverse-image-subsecn}.
\end{prop}

\begin{proof}
 Let $\bfP^\bu$ be a complex of $\X$\+flat pro-quasi-coherent
pro-sheaves on~$\bY$.
 We have to construct a natural isomorphism
$$
 \rho^*\rD^\bu\ot_\bW \boL g^*(\bfP^\bu)\,\simeq\,
 \boL g^*(\pi^*\rD^\bu\ot_\bY\bfP^\bu)
$$
in the semiderived category $\sD_\X^\si(\bW\tors)$.
 Notice first of all that the related isomorphism for underived
functors
$$
 \rho^*\rD^\bu\ot_\bW g^*(\bfP^\bu)\simeq
 g^*\pi^*\rD^\bu\ot_\bW g^*(\bfP^\bu)\simeq
 g^*(\pi^*\rD^\bu\ot_\bY\bfP^\bu)
$$
holds because $\rho^*\rD^\bu\simeq g^*\pi^*\rD^\bu$ (since $\rho=\pi g$)
and by Lemma~\ref{inverse-images-preserve-torsion-pro-tensor}.
 (It was explained in the proof of
Proposition~\ref{homotopy-Y/X-flat-and-semiacyclic-prop}
that the morphism~$g$ is ``representable by schemes''---in fact,
affine.)

 To prove the desired isomorphism for derived functors, replace
the complex $\bfP^\bu$ with a relatively homotopy flat complex of
flat pro-quasi-coherent pro-sheaves $\bfF^\bu$ on $\bY$ endowed
with a morphism of complexes $\bfF^\bu\rarrow\bfP^\bu$ with
the cone acyclic in $\bY_\X\flat$.
 Then it remains to recall that the complex $\pi^*\rD^\bu\ot_\bY
\bfF^\bu$ of quasi-coherent torsion sheaves on $\bY$ is homotopy
$\bY/\X$\+flat by
Lemma~\ref{relative-or-Y/X-homotopy-flatness-lemma}(b).
 So we have
$$
 \rho^*\rD^\bu\ot_\bW \boL g^*(\bfP^\bu)=
 \rho^*\rD^\bu\ot_\bW g^*(\bfF^\bu)\simeq
 g^*(\pi^*\rD^\bu\ot_\bY\bfF^\bu)=
 \boL g^*(\pi^*\rD^\bu\ot_\bY\bfP^\bu)
$$
in $\sD_\X^\si(\bW\tors)$.
 Here the rightmost equality holds because the morphism of complexes
of quasi-coherent torsion sheaves $\pi^*\rD^\bu\ot_\bY\bfF^\bu\rarrow
\pi^*\rD^\bu\ot_\bY\bfP^\bu$ on $\bY$ has a cone whose direct image
under~$\pi$ is coacyclic in $\X\tors$.
 Indeed, the functor $\pi^*\rD^\bu\ot_\bY{-}$ takes complexes acyclic
in $\bY_\X\flat$ to complexes with the direct image coacyclic in
$\X\tors$, according to the proof of
Theorem~\ref{relative-triangulated-equivalence-thm}.
 The notation $\rho^*\rD^\bu\ot_\bW \boL g^*(\bfP^\bu)$ with
a derived category object $\boL g^*(\bfP^\bu)$ is well-defined
for the same reason.
\end{proof}

\subsection{Semiderived equivalence and base change}
\label{semiderived-equivalence-and-base-change-subsecn}
 Now we return to the setting of a pullback diagram of morphisms
of ind-schemes similar to~\eqref{pullback-of-ind-schemes} from
Section~\ref{relative-derived-supports-subsecn}
(so $\bW=\Z\times_\X\bY$).
$$
\xymatrix{
 \bW \ar[r]^k \ar[d]_\rho & \bY \ar[d]^\pi \\
 \Z \ar[r]^i & \X
}
$$

 We start with an ind-scheme version of
Lemma~\ref{reasonable-shriek-flat-star-commutation}.

\begin{lem} \label{torsion-reasonable-shriek-flat-star-commutation}
 In the diagram above, assume that $i$~is a reasonable closed
immersion of reasonable ind-schemes and $\pi$~is a flat morphism.
 Then there is a natural isomorphism $\rho^*i^!\simeq k^!\pi^*$ of
functors\/ $\X\tors\rarrow\bW\tors$.
\end{lem}

\begin{proof}
 The (essentially obvious) argument is based on
Remark~\ref{simplified-flat-torsion-inverse-image} (which, in turn,
is based on Lemma~\ref{reasonable-shriek-flat-star-commutation}).
 We use the notation similar to the proof of
Lemma~\ref{torsion-shriek-base-change}.
 Let $Z\subset\Z$ be a reasonable closed subscheme; then $Z$ can be
also viewed as a reasonable closed subscheme in $\X$, embedded via~$i$.
 Put $\bnW=Z\times_\Z\bW$; then $\bnW$ is a reasonable closed subscheme
in $\bW$, and the scheme $\bnW$ can be also viewed as a reasonable
closed subscheme in $\bY$, embedded via~$k$.
 Denote by $\rho_Z\:\bnW\rarrow Z$ the natural morphism.
 For any quasi-coherent torsion sheaf $\rM$ on $\X$, the desired 
isomorphism of quasi-coherent torsion sheaves $\rho^*i^!\rM\simeq
k^!\pi^*\rM$ on $\bW$ is provided by the compatible system of
isomorphisms
$$
 (\rho^*i^!\rM)_{(\bnW)}\simeq \rho_Z^*((i^!\rM)_{(Z)})\simeq
 \rho_Z^*(\rM_{(Z)})\simeq(\pi^*\rM)_{(\bnW)}
 \simeq(k^!\pi^*\rM)_{(\bnW)},
$$
of quasi-coherent sheaves on~$W$.
 Here the first and third isomorphisms hold by
Remark~\ref{simplified-flat-torsion-inverse-image}, while the second
and third ones are the definition of~$i^!$ and~$k^!$.
\end{proof}

 The next lemma is a generalization of
Lemma~\ref{relative-scheme-ind-scheme-shriek-star-tensor}.

\begin{lem} \label{relative-ind-scheme-shriek-star-tensor}
 In the diagram above, assume that $i$~is a reasonable closed immersion
of reasonable ind-schemes and $\pi$~is a flat affine morphism.
 Let $\rM$ be a quasi-coherent torsion sheaf on\/ $\X$ and\/ $\bfG$
be an\/ $\X$\+flat pro-quasi-coherent pro-sheaf on\/~$\bY$; put
$\bN=\pi^*\rM\in\bY\tors$.
 Then there is a natural isomorphism
$$
 k^!(\bfG\ot_\bY\bN)\simeq k^*\bfG\ot_\bW k^!\bN
$$
of quasi-coherent torsion sheaves on\/~$\bW$.
\end{lem}

\begin{proof}
 In the notation $Z\subset\Z$, \ $\bnW=Z\times_\Z\bW\subset\bW$,
and $\rho_Z\:\bnW\rarrow Z$ from the proofs of
Lemmas~\ref{torsion-shriek-base-change}
and~\ref{torsion-reasonable-shriek-flat-star-commutation},
we compute
$$
 (k^!(\bfG\ot_\bY\bN))_{(\bnW)}\simeq(\bfG\ot_\bY\bN)_{(\bnW)}
 \simeq \bfG^{(\bnW)}\ot_{\cO_\bnW}\bN_{(\bnW)}\simeq
 (k^*\bfG)^{(\bnW)}\ot_{\cO_\bnW}(k^!\bN)_{(\bnW)}
$$
using Lemma~\ref{relative-scheme-ind-scheme-shriek-star-tensor}
for the middle isomorphism.
 Hence the collection of quasi-coherent sheaves
$(k^*\bfG)^{(\bnW)}\ot_{\cO_\bnW}(k^!\bN)_{(\bnW)}$ on the reasonable
closed subschemes $\bnW\subset\bW$ defines a quasi-coherent torsion
sheaf on $\bW$, and it follows easily that this quasi-coherent
torsion sheaf is naturally isomorphic to the tensor product
$k^*\bfG\ot_\bW k^!\bN$.
\end{proof}

 Now we assume, in the diagram above, that $\pi$~is a flat affine
morphism, $i$~is a closed immersion, and the ind-scheme\/ $\X$ is
ind-semi-separated, ind-Noetherian, and endowed with a dualizing
complex~$\rD^\bu$.
 Then $i^!\rD^\bu$ is a dualizing complex on\/ $\Z$
(cf.\ Example~\ref{ind-closed-immersion}(2)).
 So Theorem~\ref{relative-triangulated-equivalence-thm} provides
triangulated equivalences
$$
 \pi^*\rD^\bu\ot_\bY{-}\,\:
 \sD(\bY_\X\flat)\,\simeq\,\sD_\X^\si(\bY\tors)
 \,:\!\boR\fHom_{\bY\qc}(\pi^*\rD^\bu,{-})
$$
and
$$
 \rho^*i^!\rD^\bu\ot_\bW{-}\,\:
 \sD(\bW_\Z\flat)\,\simeq\,\sD_\Z^\si(\bW\tors)
 \,:\!\boR\fHom_{\bW\qc}(\rho^*i^!\rD^\bu,{-}).
$$

\begin{prop} \label{semiderived-equivalence-base-change-prop}
 In the context above, the triangulated equivalences\/
$\sD(\bY_\X\flat)\simeq\sD_\X^\si(\bY\tors)$ and\/
$\sD(\bW_\Z\flat)\simeq\sD_\Z^\si(\bW\tors)$ from
Theorem~\ref{relative-triangulated-equivalence-thm} transform
the left derived functor\/ $\boL k^*\:\sD(\bY_\X\flat)\rarrow
\sD(\bW_\X\flat)$ \,\eqref{pro-left-derived-inverse-image-eqn}
from Section~\ref{pro-derived-inverse-image-subsecn}
into the right derived functor\/ $\boR k^!\:\sD_\X^\si(\bY\tors)\rarrow
\sD_\X^\si(\bW\tors)$ \,\eqref{relative-derived-supports-eqn}
from Section~\ref{relative-derived-supports-subsecn}.
\end{prop}

\begin{proof}
 Let $\bfP^\bu$ be a complex of $\X$\+flat pro-quasi-coherent
pro-sheaves on~$\bY$.
 We have to construct a natural isomorphism
$$
 \rho^*i^!\rD^\bu\ot_\bW\boL k^*(\bfP^\bu)\,\simeq\,
 \boR k^!(\pi^*\rD^\bu\ot_\bY\bfP^\bu)
$$
in the semiderived category $\sD_\Z^\si(\bW\tors)$.
 The related isomorphism for underived functors
$$
 \rho^*i^!\rD^\bu\ot_\bW k^*(\bfP^\bu)\,\simeq\,
 k^!\pi^*\rD^\bu\ot_\bW k^*(\bfP^\bu)\,\simeq\,
 k^!(\pi^*\rD^\bu\ot_\bY\bfP^\bu)
$$
holds by Lemmas~\ref{torsion-reasonable-shriek-flat-star-commutation}
and~\ref{relative-ind-scheme-shriek-star-tensor}.

 To prove the desired isomorphism for derived functors, replace
the complex $\bfP^\bu$ with a relatively homotopy flat complex
of flat pro-quasi-coherent pro-sheaves $\bfF^\bu$ endowed with
a morphism of complexes $\bfF^\bu\rarrow\bfP^\bu$ which is
an isomorphism in $\sD(\bY_\X\flat)$.
 Then it remains to recall that, according to the proof of
Theorem~\ref{relative-triangulated-equivalence-thm}, the complex
$\pi^*\rD^\bu\ot_\bY\bfF^\bu$ (as well as the complex
$\pi^*\rD^\bu\ot_\bY\bfP^\bu$) is a complex of $\X$\+injective
quasi-coherent torsion sheaves on~$\bY$.
\end{proof}

\subsection{Semiderived equivalence and external tensor product}
\label{semiderived-and-external-subsecn}
 This Section~\ref{semiderived-and-external-subsecn} is a relative
version of Section~\ref{covariant-external-commute-subsecn}.
 Let $\X$ be an ind-Noetherian ind-scheme and $\pi\:\bY\rarrow\X$
be an affine morphism of schemes.
 Let $\rM^\bu\in\sC(\X\tors)$ be a complex of quasi-coherent torsion
sheaves on~$\X$.
 For any complex of $\X$\+flat pro-quasi-coherent pro-sheaves $\bfG^\bu$
on $\bY$, put
$$
 \bPhi_{\rM^\bu}(\bfG^\bu)=\pi^*(\rM^\bu)\ot_\bY\bfG^\bu
 \,\in\,\sC(\bY\tors).
$$
 According to formula~\eqref{derived-flat-lifted-torsion-tensor}
from Section~\ref{underived-tensor-subsecn}, the functor
$\bPhi_{\rM^\bu}$ induces a well-defined triangulated triangulated
functor
$$
 \bPhi_{\rM^\bu}\:\sD(\bY_\X\flat)\lrarrow\sD_\X^\si(\bY\tors).
$$
 Furthermore, any morphism $\rM^\bu\rarrow\rN^\bu$ in the coderived
category $\sD^\co(\X\tors)$ induces a morphism of triangulated functors
$\bPhi_{\rM^\bu}\rarrow\bPhi_{\rN^\bu}$, which is an isomorphism
of functors whenever the morphism $\rM^\bu\rarrow\rN^\bu$ is
an isomorphism in $\sD^\co(\X\tors)$.

\begin{prop} \label{semiderived-equivalence-external-tensor-prop}
 Let\/ $\X'$ and\/ $\X''$ be ind-semi-separated ind-schemes of
ind-finite type over\/~$\kk$, and let\/ $\pi'\:\bY'\rarrow\X'$ and
$\pi''\:\bY''\rarrow\X''$ be flat affine morphisms of ind-schemes.
 Consider the induced morphism of the Cartesian products
$\pi=\pi'\times_\kk\nobreak\pi''\:\allowbreak
\bY'\times_\kk\bY''\rarrow\X'\times_\kk\X''$.
 Let $\rD'{}^\bu$ and $\rD''{}^\bu$ be dualizing complexes on\/
$\X'$ and\/ $\X''$, respectively, and let $\rE^\bu$ be the related
dualizing complex on\/ $\X'\times_\kk\X''$, as in
Lemma~\ref{external-product-dualizing}.
 Then the triangulated equivalences
\begin{alignat*}{2}
 \pi'{}^*\rD'{}^\bu\ot_{\bY'}{-}\, & \:
 & \sD(\bY'_{\X'}\flat) & \,\simeq\,\sD_{\X'}^\si(\bY'\tors), \\
 \pi''{}^*\rD''{}^\bu\ot_{\bY''}{-}\, & \:
 & \sD(\bY''_{\X''}\flat) & \,\simeq\,
 \sD_{\X''}^\si(\bY''\tors),
\end{alignat*}
and
$$
 \pi^*\rE^\bu\ot_{(\bY'\times_\kk\bY'')}{-}\,
 \:\sD((\bY'\times_\kk\bY'')_{(\X'\times_\kk\X'')}\flat)
 \,\simeq\,
 \sD_{(\X'\times_\kk\X'')}((\bY'\times_\kk\bY'')\tors)
$$
from Theorem~\ref{relative-triangulated-equivalence-thm}
form a commutative square diagram with the external tensor product
functors\/~$\bt_\kk$ \,\eqref{derived-flat-over-base-external-product}
and~\eqref{semiderived-torsion-external-product} from
Section~\ref{relative-external-products-subsecn}.
\end{prop}

\begin{proof}
 Let $\rM'{}^\bu$ and $\rM''{}^\bu$ be complexes of quasi-coherent
torsion sheaves on $\X'$ and~$\X''$.
 Then it follows from
Lemmas~\ref{torsion-inverse-image-of-external-product}
and~\ref{torsion-pro-action-external-product} that, for any complex
of $\X'$\+flat pro-quasi-coherent pro-sheaves $\bfQ'{}^\bu$ on $\bY'$
and any complex of $\X''$\+flat pro-quasi-coherent pro-sheaves
$\bfQ''{}^\bu$ on $\bY''$, there is a natural isomorphism
$$
 \bPhi_{\rM'{}^\bu\subbt_\kk\rM''{}^\bu}
 (\bfQ'{}^\bu\bt_\kk\bfQ''{}^\bu)\,\simeq\,
 \bPhi_{\rM'{}^\bu}(\bfQ'{}^\bu)\bt_\kk
 \bPhi_{\rM''{}^\bu}(\bfQ''{}^\bu)
$$
in the category of complexes of quasi-coherent torsion sheaves on
$\bY'\times_\kk\bY''$.
 It remains to take $\rM'{}^\bu=\rD'{}^\bu$ and $\rM''{}^\bu=
\rD''{}^\bu$, and observe that the isomorphism $\rD'{}^\bu\bt_\kk
\nobreak\rD''{}^\bu\rarrow\rE^\bu$ in the coderived category
$\sD^\co((\X'\times_\kk\X'')\tors)$ induces an isomorphism of
triangulated functors $\bPhi_{\rD'{}^\bu\subbt_\kk\rD''{}^\bu}
\rarrow\bPhi_{\rE^\bu}$, as per the discussion above.
\end{proof}

\subsection{The semitensor product computed}
 Let us recall the setting and notation from the introductory paragraphs
of Section~\ref{flat-affine-over-ind-finite-type-secn}.
 Let $\X$ be an ind-separated ind-scheme of ind-finite type over
a field~$\kk$, and let $\pi\:\bY\rarrow\X$ be a flat affine morphism
of schemes.
 Then the diagonal morphism $\Delta_\bY\:\bY\rarrow\bY\times_\kk\bY$
decomposes as
$$
 \bY\xrightarrow{\delta_{\bY/\X}}\bY\times_\X\bY
 \xrightarrow{\eta_{\bY/\X}}\bY\times_\kk\bY.
$$
 Both $\delta=\delta_{\bY/\X}$ and $\eta=\eta_{\bY/\X}$ are closed
immersions of ind-schemes (see~\cite[Tags~01S7, 01KU(1), 01KR]{SP} for
scheme versions of these assertions).

 In fact, there is a commutative square diagram
(a particular case of~\eqref{square-of-ind-schemes})
\begin{equation} \label{diagornal-square-of-ind-schemes}
\begin{gathered}
\xymatrix{
 \bY \ar[rrrr]^-{\Delta_\bY} \ar[rrd]_\pi
 &&&& \bY\times_\kk\bY \ar[d]^{\pi\times_\kk\pi} \\
 && \X \ar[rr]^-{\Delta_\X} && \X\times_\kk\X
}
\end{gathered}
\end{equation}
which is composed of a commutative triangle diagram
(a particular case of~\eqref{triangle-of-ind-schemes})
\begin{equation} \label{diagonal-triangle-of-ind-schemes}
\begin{gathered}
\xymatrix{
 \bY \ar[rr]^-{\delta_{\bY/\X}} \ar[rrd]_\pi
 && \bY\times_\X\bY \ar[d]^{\pi\times_\X\pi} \\
 && \X
}
\end{gathered}
\end{equation}
and a pullback diagram
(a particular case of~\eqref{pullback-of-ind-schemes})
\begin{equation} \label{diagonal-pullback-of-ind-schemes}
\begin{gathered}
\xymatrix{
 \bY\times_\X\bY \ar[rr]^{\eta_{\bY/\X}} \ar[d]^{\pi\times_\X\pi}
 && \bY\times_\kk\bY \ar[d]^{\pi\times_\kk\pi} \\
 \X \ar[rr]^{\Delta_\X} && \X\times_\kk\X
}
\end{gathered}
\end{equation}
 The morphism~$\Delta_\X$ is a closed immersion of ind-schemes.
 The morphisms of ind-schemes $\pi\times_\kk\pi$ and $\pi\times_\X\pi$
are flat and affine.

\begin{thm}
 In the context above, let\/ $\rD^\bu$ be a rigid dualizing complex on
the ind-scheme\/~$\X$ (in the sense of Section~\ref{rigid-subsecn}).
 Then, for any two complexes of quasi-coherent torsion sheaves
$\bM^\bu$ and $\bN^\bu\in\sD_\X^\si(\bY\tors)$, there is a natural
isomorphism
\begin{equation} \label{semitensor-product-computed-eqn}
 \bM^\bu\os_{\pi^*\rD^\bu}\bN^\bu\,\simeq\,
 \boL\delta^*\,\boR\eta^!(\bM^\bu\bt_\kk\bN^\bu)
\end{equation}
in the\/ $\bY/\X$\+semiderived category\/ $\sD_\X^\si(\bY\tors)$ of
quasi-coherent torsion sheaves on\/~$\bY$.
 Here the semitensor product functor\/~$\os_{\pi^*\rD^\bu}$ was defined
in formula~\eqref{semiderived-torsion-semitensor-product}
in Section~\ref{construction-of-semitensor-subsecn}.
 The external tensor product functor\/~$\bt_\kk$ was defined in
formula~\eqref{semiderived-torsion-external-product} in
Section~\ref{relative-external-products-subsecn},
the right derived functor\/ $\boR\eta^!$ was defined in
Section~\ref{relative-derived-supports-subsecn},
and the left derived functor\/ $\boL\delta^*$ was defined in
Section~\ref{torsion-derived-inverse-image-subsecn}.
\end{thm}

\begin{proof}
 Let $\bK^\bu$ and $\bJ^\bu$ be two complexes of $\X$\+injective
quasi-coherent torsion sheaves on $\bY$ endowed with morphisms of
complexes $\bM^\bu\rarrow\bK^\bu$ and $\bN^\bu\rarrow\bJ^\bu$ with
the cones whose direct images under~$\pi$ are coacyclic in $\X\tors$.
 Then $\bfF^\bu=\fHom_{\bY\qc}(\pi^*\rD^\bu,\bK^\bu)$ and
$\bfG^\bu=\fHom_{\bY\qc}(\pi^*\rD^\bu,\bJ^\bu)$ are two complexes
in $\bY_\X\flat$ corresponding to $\bM^\bu$ and $\bN^\bu$,
respectively, under the equivalence of categories
$\sD(\bY_\X\flat)\simeq\sD_\X^\si(\bY\tors)$ from
Theorem~\ref{relative-triangulated-equivalence-thm}; so natural
isomorphisms $\pi^*\rD^\bu\ot_\bY\bfF^\bu\rarrow\bK^\bu\larrow\bM^\bu$
and $\pi^*\rD^\bu\ot_\bY\bfG^\bu\rarrow\bJ^\bu\larrow\bN^\bu$ exist in
$\sD_\X^\si(\bY\tors)$.
 By the definition, we have
$$
 \bM^\bu\os_{\pi^*\rD^\bu}\bN^\bu=
 \pi^*\rD^\bu\ot_\bY(\bfF^\bu\ot^{\bY,\boL}\bfG^\bu).
$$

 Let $\rE^\bu$ be a (dualizing) complex of injective quasi-coherent
torsion sheaves on $\X\times_\kk\X$ endowed with a morphism of
complexes $\rD^\bu\bt_\kk\rD^\bu\rarrow\rE^\bu$ with the cone
coacyclic in $(\X\times_\kk\X)\tors$.
 Since $\rD^\bu$ is assumed to be a rigid dualizing complex, we have
a homotopy equivalence $\rD^\bu\simeq\Delta_\X^!\rE^\bu$ of
complexes in $\X\tors_\inj$.
 Now we compute
\begin{multline*}
 \pi^*\rD^\bu\ot_\bY(\bfF^\bu\ot^{\bY,\boL}\bfG^\bu)
 \,\overset{\ref{derived-tensor-as-derived-restriction-prop}}\simeq\,
 \pi^*\rD^\bu\ot_\bY\boL\Delta_\bY^*(\bfF^\bu\bt_\kk\bfG^\bu) \\
 \,\overset{\ref{composition-of-derived-inverse-images}}\simeq\,
 \pi^*\rD^\bu\ot_\bY\boL\delta^*\,\boL\eta^*(\bfF^\bu\bt_\kk\bfG^\bu)
 \,\overset{\ref{semiderived-equivalence-change-of-fiber-prop}}\simeq\,
 \boL\delta^*\bigl((\pi\times_\X\pi)^*(\rD^\bu)\ot_{(\bY\times_\X\bY)}
 \boL\eta^*(\bfF^\bu\bt_\kk\bfG^\bu)\bigr) \\
 \,\overset{\ref{semiderived-equivalence-base-change-prop}}\simeq\,
 \boL\delta^*\,\boR\eta^!\bigl((\pi\times_\kk\pi)^*(\rE^\bu)
 \ot_{(\bY\times_\kk\bY)}(\bfF^\bu\bt_\kk\bfG^\bu)\bigr) \\
 \,\overset{\ref{semiderived-equivalence-external-tensor-prop}}\simeq\,
 \boL\delta^*\,\boR\eta^!\bigl((\pi^*\rD^\bu\ot_\bY\bfF^\bu)
 \bt_\kk(\pi^*\rD^\bu\ot_\bY\bfG^\bu)\bigr)
 \,\overset{\ref{torsion-complexes-external-tensor-semiacyclic}}\simeq\,
 \boL\delta^*\,\boR\eta^!(\bM^\bu\bt_\kk\bN^\bu),
\end{multline*}
where the numbers over the isomorphism signs indicate the relevant
propositions and lemma where the natural isomorphisms are
established.
\end{proof}

\Section{Invariance under Postcomposition with a Smooth Morphism}
\label{weakly-smooth-postcomposition-secn}

 Let $\X$ be an ind-semi-separated ind-Noetherian ind-scheme, and
let $\tau\:\X'\rarrow\X$ be a smooth affine morphism of finite type.
 Let $\pi'\:\bY\rarrow\X'$ be a flat affine morphism, and let
$\pi\:\bY\rarrow\X$ denote the composition $\pi=\tau\pi'$.
 Let $\rD^\bu$ be a dualizing complex on $\X$; then $\rD'{}^\bu=
\tau^*\rD^\bu$ is a dualizing complex on~$\X'$.

 The aim of this section is to show that the constructions of
Sections~\ref{X-flat-on-Y-secn}\+-\ref{semitensor-secn}, including
the semiderived category of quasi-coherent torsion sheaves on $\bY$
and the semitensor product operation on it, are preserved by
the passage from the flat affine morphism $\pi'\:\bY\rarrow\X'$ to
the flat affine morphism $\pi\:\bY\rarrow\X$.

\subsection{Weakly smooth morphisms} \label{weakly-smooth-subsecn}
 We refer to~\cite[Tag~01V4]{SP} for a discussion of smooth morphisms
of schemes.
 For the purposes of this section, a slightly weaker condition is
suffient; we call it \emph{weak smoothness}.
 Essentially, a morphism of schemes is said to be weakly smooth if it
is flat with regular fibers of bounded Krull dimension.

 Let $X$ be a scheme and $x\in X$ be a point.
 Denote by $\kappa_X(x)$ the residue field of the point~$x$ on~$X$.
 Abusing notation, we will denote simply by~$x$ the one-point
scheme $\Spec\kappa_X(x)$.
 Then we have a natural morphism of schemes $x\rarrow X$.

 Let $f\:Y\rarrow X$ be a morphism of schemes and $x\in X$ be
a point.
 Then the scheme $Y_x=x\times_XY$ is called the \emph{fiber} of~$f$
over~$x$.
 Given an integer $d\ge0$, we will say that the morphism~$f$ is
\emph{weakly smooth of relative dimension~$\le d$} if $f$~is flat
\emph{and} for every point $x\in X$ the fiber $Y_x$ is a regular
Noetherian scheme of Krull dimension~$\le d$.

 By~\cite[Tags~01VB and~00TT]{SP}, any smooth morphism of schemes
is weakly smooth.
 According to~\cite[Tag~01V8]{SP}, any weakly smooth morphism of
finite type between Noetherian schemes over a field of
characteristic~$0$ is smooth.
 This is not true in finite characteristic because of nonseparability
issues (an inseparable finite field extension is the simplest
example of a nonsmooth flat morphism with regular fibers).

 Notice that weak smoothness is \emph{not} preserved by base change,
generally speaking (e.~g., base changes of inseparable finite
field extensions can have nilpotent elements in the fibers).
 However, some base changes do preserve it.
 Let us say that a morphism of schemes $Z\rarrow X$ \emph{does not
extend the residue fields} if for every point $z\in Z$ and its
image $x\in X$ the induced field extension $\kappa_X(x)\rarrow
\kappa_Z(z)$ is an isomorphism.
 In particular, locally closed immersions of schemes do not extend
the residue fields.
 Clearly, if a morphism $Y\rarrow X$ is weakly smooth of relative
dimension~$\le d$ and a morphism $Z\rarrow X$ does not extend
the residue fields, then the morphism $Z\times_XY\rarrow Z$ is also
weakly smooth of relative dimension~$\le d$.

 Any closed immersion of schemes is injective as a map of
the underlying sets of points.
 So any strict ind-scheme $\X$, represented by an inductive system
of closed immersions of schemes $\X=\ilim_{\gamma\in\Gamma}X_\gamma$,
gives rise to an inductive system of injective maps of the underlying
sets.
 The inductive limit of this inductive system of sets is called
the \emph{underlying set of points of\/~$\X$}.
 Given a point $x\in\X$, the residue field $\kappa=\kappa_\X(x)$
is well-defined, because the closed immersions $X_\gamma\rarrow
X_\delta$, \,$\gamma<\delta\in\Gamma$, do not extend the residue
fields (so one can take any $\gamma\in\Gamma$ such that $x\in X_\gamma
\subset\X$ and put $\kappa_\X(x)=\kappa_{X_\gamma}(x)$).
 Denoting the scheme $\Spec\kappa_\X(x)$ simply by~$x$, we have
a morphism of ind-schemes $x\rarrow\X$.

 Let $f\:\Y\rarrow\X$ be a morphism of ind-schemes with is
``representable by schemes''.
 Then, for every point $x\in X$, the fiber $x\times_\X\Y$ is
a scheme.
 As above, we will say that the morphism~$f$ is \emph{weakly smooth
of relative dimension~$\le d$} if $f$~is flat and for every point
$x\in X$ the scheme $x\times_\X\Y$ is Noetherian and regular of Krull
dimension~$\le d$.
 Clearly, the morphism of ind-schemes~$f$ is weakly smooth of
relative dimension~$\le d$ if and only if, for every
$\gamma\in\Gamma$, the morphism of schemes
$f_\gamma\:Y_\gamma=X_\gamma\times_\X\Y\rarrow X_\gamma$ is
weakly smooth of relative dimension~$\le d$.

 Let $\X$ be an ind-Noetherian ind-scheme and $f\:\Y\rarrow\X$ be
a morphism of ind-schemes.
 One says that $f$~is a \emph{morphism of finite type} if for
every Noetherian scheme $T$ and every morphism of ind-schemes
$T\rarrow\X$ the fibered product $T\times_\X\Y$ is a scheme
\emph{and} the morphism of schemes $T\times_\X\Y\rarrow T$ is
of finite type.
 It suffices to check these conditions for the closed subschemes
$T=X_\gamma$ appearing in a given representation of $\X$ by
an inductive system of closed immersions of schemes.

\subsection{Flat and injective dimension under weakly smooth
morphisms} \label{flat-injective-dimensions-weakly-smooth-subsecn}
 Let $\M$ be a quasi-coherent sheaf on a scheme~$X$, and $d\ge0$ be
an integer.
 One says that the \emph{injective dimension} of $\M$ does not
exceed~$d$ if there exists an exact sequence $0\rarrow\M\rarrow\J^0
\rarrow\J^1\dotsb\rarrow\J^d\rarrow0$ of quasi-coherent sheaves
on $X$ with injective quasi-coherent sheaves~$\J^i$.
 On a (locally) Noetherian scheme $X$, injectivity of quasi-coherent
sheaves is a local property~\cite[Proposition~II.7.17 and
Theorem~II.7.18]{Hart}, hence so is the injective dimension:
the injective dimension of $\M$ is equal to the supremum of
the injective dimensions of the $\cO_X(U_\alpha)$\+modules
$\M(U_\alpha)$, where $X=\bigcup_\alpha U_\alpha$ is any given affine
open covering of the scheme~$X$.

 Let $\M$ be a quasi-coherent sheaf on a quasi-compact, semi-separated
scheme~$X$.
 According to~\cite[Section~2.4]{Mur0} or~\cite[Lemma~A.1]{EP}, every
quasi-coherent sheaf on $X$ is a quotient sheaf of a flat
quasi-coherent sheaf.
 One says that the \emph{flat dimension} of $\M$ does not exceed~$d$
if there exists an exact sequence $0\rarrow \F_d\rarrow \F_{d-1}
\rarrow\dotsb\rarrow\F_0\rarrow\M\rarrow0$ of quasi-coherent sheaves
on $X$ with flat quasi-coherent sheaves~$\F_i$.
 Since flatness of quasi-coherent sheaves is a local property, so is
the flat dimension: in the same notation as above, the flat dimension
of $\M$ is equal to the supremum of the flat dimensions of
the $\cO_X(U_\alpha)$\+modules $\M(U_\alpha)$.

 The following relative form of Hilbert's syzygy theorem is
essentially well-known.

\begin{prop} \label{relative-hilberts-syzygy}
 Let $R$ be an associative ring and $S=R[x_1,\dotsc,x_d]$ be the ring
of polynomials in $d$~variables with the coefficients in~$R$.
 Then, for any $S$\+module $N$: \par
\textup{(a)} the flat dimension of $N$ as an $S$\+module does not
exceed $d$~plus the flat dimension of $N$ as an $R$\+module; \par
\textup{(b)} the projective dimension of $N$ as an $S$\+module does not
exceed $d$~plus the projective dimension of $N$ as an $R$\+module; \par
\textup{(c)} the injective dimension of $N$ as an $S$\+module does not
exceed $d$~plus the injective dimension of $N$ as an $R$\+module.
\end{prop}

\begin{proof}
 Parts~(b) and~(c) follow directly from~\cite[spectral sequence~(I)
in Section~2 and Theorem~6 in Section~4]{ERZ}.
 Parts~(a) and~(b) are provable by straightforward induction in~$d$
using~\cite[Proposition~7.5.2]{MR}.
\end{proof}

 The aim of this
Section~\ref{flat-injective-dimensions-weakly-smooth-subsecn} is to
prove the following partial generalization of
Proposition~\ref{relative-hilberts-syzygy}(a,c).

\begin{prop} \label{flat-injective-dimensions-direct-image}
 Let $X$ be a Noetherian scheme and $f\:Y\rarrow X$ be an affine
morphism of schemes.
 Assume that the morphism~$f$ is weakly smooth of relative
dimension~$\le d$.
 Then, for any quasi-coherent sheaf\/ $\N$ on $Y$: \par
\textup{(a)} assuming that the scheme $X$ is semi-separated,
the flat dimension of\/ $\N$ does not exceed $d$~plus
the flat dimension of the quasi-coherent sheaf $f_*\N$ on~$X$; \par
\textup{(b)} assuming that $f$~is a morphism of finite type,
the injective dimension of\/ $\N$ does not exceed $d$~plus
the injective dimension of the quasi-coherent sheaf $f_*\N$ on~$X$.
\end{prop}

 For any commutative ring $R$, an element $a\in R$, and an $R$\+module
$M$, we denote by ${}_aM=\ker(M\overset a\rarrow M)$ the submodule
of all elements annihilated by~$a$ in~$M$.
 So ${}_aM$ is a module over the ring $R/aR$.
 
\begin{lem} \label{quotient-by-element}
 Let $R$ be a commutative ring and $a\in R$ be an element.
 Then \par
\textup{(a)} for every flat $R$\+module $F$, the $R/aR$\+module $F/aF$
is flat; \par
\textup{(b)} for every injective $R$\+module $J$, the $R/aR$\+module
${}_aJ$ is injective.
\end{lem}

\begin{proof}
 For any ring homomorphism $R\rarrow T$, the functor $F\longmapsto
T\ot_RF$ takes flat $R$\+modules to flat $T$\+modules, and
the functor $J\longmapsto\Hom_R(T,J)$ takes injective $R$\+modules
to injective $T$\+modules.
 It remains to apply these observations to the ring homomorphism
$R\rarrow R/aR=T$ in order to deduce the assertions of the lemma.
\end{proof}

\begin{lem} \label{resolution-quotient-by-element}
 Let $R\rarrow S$ be a homomorphism of commutative rings such that
$S$ is a flat $R$\+module.
 Let $a\in R$ be an element.
 Then \par
\textup{(a)} for any flat resolution $F_\bu$ of
an $R$\+flat $S$\+module $G$, the complex $F_\bu/aF_\bu$ is
a flat resolution of an ($R/aR$\+flat) $S/aS$\+module $G/aG$; \par
\textup{(b)} for any injective resolution $J^\bu$ of
an $R$\+injective $S$\+module $K$, the complex ${}_aJ^\bu$ is
an injective resolution of an ($R/aR$\+injective)
$S/aS$\+module~${}_aK$.
\end{lem}

\begin{proof}
 Let us prove part~(b).
 Since $S$ is a flat $R$\+module, the underlying $R$\+module of any
injective $S$\+module is injective.
 So, viewed as a complex of $R$\+modules, $J^\bu$ is an injective
resolution of an injective $R$\+module~$K$.
 Hence the complex ${}_aJ^\bu=\Hom_R(R/aR,J^\bu)$ is exact, i.~e.,
it is a resolution of the $R$\+module~${}_aK$.
 The $R/aR$\+module ${}_aK$ is injective by
Lemma~\ref{quotient-by-element}(b), and the terms of the complex
${}_aJ^\bu$ are injective $S/aS$\+modules by the same lemma.
\end{proof}

\begin{lem} \label{distinguished-triangle-quotient-by-element}
 Let $S$ be a commutative ring and $a\in S$ be a nonzero-dividing
(regular) element. \par
\textup{(a)} Let $M$ and $N$ be two $S$\+modules for which the maps
$M\overset a\rarrow M$ and $N\overset a\rarrow N$ are injective.
 Then there is a distinguished triangle
$$
 M\ot_\S^\boL N\overset a\lrarrow M\ot_S^\boL N\lrarrow
 M/aM\ot_{S/aS}^\boL N/aN\lrarrow M\ot_S^\boL N[1]
$$
in the derived category of $S$\+modules. \par
\textup{(b)} Let $M$ and $K$ be two $S$\+modules such that
the map $M\overset a\rarrow M$ is injective and the map
$K\overset a\rarrow K$ is surjective.
 Then there is a distinguished triangle
\begin{multline*}
 \boR\Hom_S(M,K)[-1]\lrarrow\boR\Hom_{S/aS}(M/aM,\,{}_aK) \\
 \lrarrow\boR\Hom_S(M,K)\overset a\lrarrow\boR\Hom_S(M,K)
\end{multline*}
in the derived category of $S$\+modules.
\end{lem}

\begin{proof}
 Let us prove part~(b).
 Firstly, there is a distinguished triangle $M\overset a\rarrow M
\rarrow M/aM\rarrow M[1]$ in $\sD(S\modl)$.
 Applying $\boR\Hom_S({-},K)$, we obtain
\begin{multline*}
 \boR\Hom_S(M,K)[-1]\lrarrow\boR\Hom_S(M/aM,K) \\
 \lrarrow\boR\Hom_S(M,K)\overset a\lrarrow\boR\Hom_S(M,K).
\end{multline*}
 It remains to construct an isomorphism
$$
 \boR\Hom_S(M/aM,K)\simeq
 \boR\Hom_{S/aS}(M/aM,\,{}_aK)
$$
in $\sD(S\modl)$.
 For this purpose, choose an injective resolution $J^\bu$ of
the $S$\+module~$K$.
 Notice that, for any injective $S$\+module $J$, the map
$J\overset a\rarrow J$ is surjective, because it can be obtained
by applying the functor $\Hom_S({-},J)$ to an injective
$S$\+module morphism $S\overset a\rarrow S$.
 So $J^\bu$ is a resolution of an $a$\+divisible $S$\+module $K$
by $a$\+divisible $S$\+modules.
 It follows that the complex ${}_aJ^\bu$ is exact, i.~e., it is
a resolution of the $S$\+module~${}_aK$.
 By Lemma~\ref{quotient-by-element}(b), \,${}_aJ^\bu$ is an injective
resolution of the $S/aS$\+module~${}_aK$.
 Finally, we use the isomorphism of complexes of $S$\+modules
(in fact, $S/aS$\+modules) $\Hom_S(M/aM,J^\bu)\simeq
\Hom_{S/aS}(M/aM,\,{}_aJ^\bu)$.
\end{proof}

\begin{proof}[Proof of
Proposition~\ref{flat-injective-dimensions-direct-image}(a)]
 A result bearing some similarity with, but still quite different
from our assertion can be found in~\cite[Lemma~2.7]{AFH}.

 The question is local in $X$, so it reduces to affine schemes,
for which it means the following.
 Let $R\rarrow S$ be a homomorphism of commutative rings such that
the ring $R$ is Noetherian, the $R$\+module $S$ is flat, and for every
prime ideal $\p\subset R$, the fiber ring $\kappa_R(\p)\ot_RS$ is
Noetherian and regular of Krull dimension~$\le d$.
 Here $\kappa_R(\p)=R_\p/R_\p\p$ denotes the residue field of
the prime ideal~$\p$ in~$R$.
 Then, for any $S$\+module $N$, the flat dimension of the $S$\+module
$N$ does not exceed $d$~plus the flat dimension of the $R$\+module~$N$.

 Since $S$ is a flat $R$\+module, all flat $S$\+modules are also flat
as $R$\+modules.
 Let $F_\bu$ be a flat resolution of the $S$\+module~$N$; for every
$n\ge0$, denote by $\Omega^nN$ the cokernel of the $S$\+module morphism
$F_{n+1}\rarrow F_n$ (so, in particular, $\Omega^0N=N$).
 Suppose that the flat dimension~$e$ of the $R$\+module $N$ is finite.
 Then $G=\Omega^eN$ is a flat $R$\+module.
 It remains to show that the flat dimension of the $S$\+module $G$
does not exceed~$d$; then it will follow that $\Omega^dG=
\Omega^{d+e}N$ is a flat $S$\+module, so the flat dimension of
the $S$\+module $N$ does not exceed $d+e$.
 Thus, in order to prove the desired assertion, it suffices to consider
the case of an $R$\+flat $S$\+module~$G$.

 The argument proceeds by Noetherian induction in the ring~$R$.
 So we assume that the desired assertion holds of all the quotient
rings of $R$ by nonzero ideals.
 Let $G$ be an $R$\+flat $S$\+module.
 We consider two cases separately.

 Case~I. Suppose that $R$ has zero-divisors.
 Let $a$ and $b\in R$ be a pair of nonzero elements for which $ab=0$.

 By Lemma~\ref{quotient-by-element}(a), \,$G/aG$ is an $R/aR$\+flat
$S/aS$\+module and $G/bG$ is an $R/bR$\+flat $S/bS$\+module.
 Furthermore, the ring $S/aS$ is a flat $R/aR$\+module, and it is
clear that the fiber rings of the ring homomorphism $R/aR\rarrow S/aS$
are regular of Krull dimension~$\le d$ (and similarly for
$R/bR\rarrow S/bS$).
 By the assumption of Noetherian induction, the flat dimensions of
the $S/aS$\+module $G/aG$ and the $S/bS$\+module $G/bG$ do not
exceed~$d$.

 We have to show that $\Tor^S_{d+1}(M,G)=0$ for all $S$\+modules~$M$.
 Consider the short exact sequence of $S$\+modules
$0\rarrow aM\rarrow M\rarrow M/aM\rarrow 0$.
 The $R$\+module $aM$ is annihilated by~$b$ and the $R$\+module
$M/aM$ is annihilated by~$a$.
 So the problem reduces to $R$\+modules $M$ for which either $aM=0$
or $bM=0$.

 Suppose that $aM=0$.
 Let $F_\bu$ be a flat resolution of the $S$\+module~$G$;
then, by Lemma~\ref{resolution-quotient-by-element}(a), \,$F_\bu/aF_\bu$
is a flat resolution of the $S/aS$\+module~$G/aG$.
 Hence $\Tor^S_i(M,G)\simeq\Tor^{S/aS}_i(M,G/aG)=0$ for $i>d$,
as desired.

 Case~II. Suppose that $R$ is an integral domain, and denote by $Q$
the field of fractions of~$R$.
 For any $S$\+module $M$, we have $Q\ot_R\Tor^S_i(M,G)\simeq
\Tor^{Q\ot_RS}_i(Q\ot_RM,\>Q\ot_RG)$ for all $i\ge0$.
 By assumption, $Q\ot_RS$ is a regular Noetherian ring of Krull
dimension~$\le d$, so the global dimension of $Q\ot_RS$ does not
exceed~$d$.
 Hence $Q\ot_R\Tor^S_i(M,G)=0$ for all $i>d$, and it follows that
$\Tor^S_i(M,G)$ is a torsion $R$\+module.

 Let $a\in R$ be a nonzero element.
 In order to show that $\Tor^S_i(M,G)=0$ for $i>d$, it suffices to
prove that there are no nonzero elements annihilated by~$a$ in
$\Tor^S_i(M,G)$.

 Let $0\rarrow\Omega M\rarrow F\rarrow M\rarrow0$ be a short exact
sequence of $S$\+modules with a flat $S$\+module~$F$.
 If $d\ge1$, then we have $\Tor^S_{d+1}(M,G)\simeq
\Tor^S_d(\Omega M,G)$.
 When $d=0$, the $S$\+module $\Tor^S_1(M,G)$ is the kernel of
the morphism $\Omega M\ot_SG\rarrow F\ot_SG$.
 Since $F$ is a flat $S$\+module and $G$ is a flat $R$\+module,
the $R$\+module $F\ot_SG$ is flat; in particular, it contains no
nonzero elements annihilated by~$a$.
 So the submodules of elements annihilated by~$a$ in
the $S$\+modules $\Tor^S_1(M,G)$ and $\Omega M\ot_SG$ are
naturally isomorphic.
 In both cases, it remains to show that there are no nonzero elements
annihilated by~$a$ in $\Tor^S_d(\Omega M,G)$.

 Both the $R$\+modules $\Omega M$ and $G$ contain no nonzero elements
annihilated by~$a$ (since $\Omega M$ is a submodule in a flat
$R$\+module~$F$).
 The same applies to the $R$\+module~$S$.
 By Lemma~\ref{distinguished-triangle-quotient-by-element}(a),
we have a distinguished triangle in $\sD(S\modl)$
$$
 \Omega M\ot_S^\boL G\overset a\lrarrow\Omega M\ot_S^\boL G
 \lrarrow(\Omega M/a\Omega M)\ot_{S/aS}^\boL G/aG\lrarrow
 \Omega M\ot_S^\boL G[1].
$$
 It follows from the related long exact sequence of cohomology
modules that $\Tor^{S/aS}_{d+1}(\Omega M/a\Omega M,G/aG)\ne0$
whenever there are any nonzero elements annihilated by~$a$ in
$\Tor^S_d(\Omega M,G)$.

 Finally, similarly to Case~I, \,$G/aG$ is an $R/aR$\+flat
$S/aS$\+module, the ring $S/aS$ is a flat $R/aR$\+module, and
the fiber rings of the ring homomorphism $R/aR\rarrow S/aS$
are regular of Krull dimension~$\le d$.
 By the assumption of Noetherian induction, the flat dimension
of the $S/aS$\+module $G/aG$ does not exceed~$d$, hence
$\Tor^{S/aS}_{d+1}(\Omega M/a\Omega M,G/aG)=0$ and we are done.
\end{proof}

 For any element $a$ in a ring $R$, we denote by $R[a^{-1}]$
the localization of the ring $R$ at the multiplicative subset
$\{1,a,a^2,a^3,\dotsc\}\subset R$.
 For any $R$\+module $E$, we put $E[a^{-1}]=R[a^{-1}]\ot_RE$.

\begin{lem}[Grothendieck's generic freeness] \label{generic-freeness}
 Let $R$ be a Noetherian commutative integral domain, $S$ be a finitely
generated commutative $R$\+algebra, and $M$ be a finitely generated
$S$\+module.
 Then there exists a nonzero element $a\in R$ such that $M[a^{-1}]$ is
a free $R[a^{-1}]$\+module.
\end{lem}

\begin{proof}
 This is~\cite[Lemme~6.9.2]{Gr} or~\cite[Theorem~24.1]{Mats}.
\end{proof}

\begin{lem} \label{ext-colocalization-by-element-lemma}
 Let $S$ be a commutative ring and $a\in S$ be an element.
 Let $M$ and $K$ be $S$\+modules such that\/ $\Ext^1_S(S[a^{-1}],K)=0$.
 Then for every $i\ge0$ there is a natural surjective $S$\+module map
$$
 \Ext_{S[a^{-1}]}^i(M[a^{-1}],\Hom_S(S[a^{-1}],K))\lrarrow
 \Hom_S(S[a^{-1}],\Ext_S^i(M,K)).
$$
\end{lem}

\begin{proof}
 The key observation is that the projective dimension of
the $S$\+module $S[a^{-1}]$ cannot exceed~$1$
(see~\cite[proof of Lemma~2.1]{Pcta} or~\cite[Lemma~1.9]{PMat}).
 Therefore, our assumption implies that $\boR\Hom_S(S[a^{-1}],K)
=\Hom_S(S[a^{-1}],K)$.
 Furthermore, one clearly has
$$
 \boR\Hom_{S[a^{-1}]}(M[a^{-1}],\boR\Hom_S(S[a^{-1}],K))
 \simeq\boR\Hom_S(S[a^{-1}],\boR\Hom_S(M,K)).
$$
 Finally, since $\boR\Hom_S(S[a^{-1}],{-})$ is a derived functor
of homological dimension~$\le1$, for any complex of $S$\+modules
$C^\bu$ and integer $i\in\boZ$ there is a natural short exact
sequence of $S$\+modules
\begin{multline*}
 0\lrarrow\Ext_S^1(S[a^{-1}],H^{i-1}(C^\bu)) \\ \lrarrow
 H^i\boR\Hom_S(S[a^{-1}],C^\bu)\lrarrow\Hom_S(S[a^{-1}],H^i(C^\bu))
 \lrarrow0.
\end{multline*}
 Taking $C^\bu=\boR\Hom_S(M,K)$ and combining these observations,
we obtain, in the situation at hand, a natural short exact sequence
of $S$\+modules
\begin{multline*}
 0\lrarrow\Ext_S^1(S[a^{-1}],\Ext_S^{i-1}(M,K)) \lrarrow
 \Ext_{S[a^{-1}]}^i(M[a^{-1}],\Hom_S(S[a^{-1}],K)) \\ \lrarrow
 \Hom_S(S[a^{-1}],\Ext_S^i(M,K))\lrarrow0.
\end{multline*}
\end{proof}

\begin{lem} \label{colocalization-by-element-vanishing-lemma}
 Let $S$ be a commutative ring, $a\in S$ be an element, and $E$ be
an $S$\+module.
 Suppose that the map $a\:E\rarrow E$ is surjective.
 Then the map\/ $\Hom_S(S[a^{-1}],E)\rarrow E$ induced by
the localization map $S\rarrow S[a^{-1}]$ is surjective.
 In particular, if\/ $\Hom_S(S[a^{-1}],E)=0$, then $E=0$.
\end{lem}

\begin{proof}
 Given an element $e\in E$, put $e_0=e$, and for every $n\ge1$
choose an element $e_n\in E$ such that $ae_n=e_{n-1}$.
 Then the sequence of elements $e_0$, $e_1$, $e_2$,~\dots~$\in E$
defines the desired $S$\+module morphism $S[a^{-1}]\rarrow E$. 
\end{proof}

\begin{proof}[Proof of
Proposition~\ref{flat-injective-dimensions-direct-image}(b)]
 The question is local in $X$, so it reduces to affine schemes, for
which it means the following.
 Let $R\rarrow S$ be a homomorphism of Noetherian commutative rings
such that $S$ is a finitely generated $R$\+algebra, the $R$\+module
$S$ is flat, and and for every prime ideal $\p\subset R$, the fiber
ring $\kappa_R(\p)\ot_RS$ is regular of Krull dimension~$\le d$.
 Then for any $S$\+module $N$, the injective dimension of
the $S$\+module $N$ does not exceed $d$~plus the injective dimension
of the $R$\+module~$N$.

 Since $S$ is a flat $R$\+module, all injective $S$\+modules are
also injective as $R$\+modules.
 Arguing similarly to the proof of part~(a) above, one reduces
the question to the case of an $R$\+injective $S$\+module $K$,
for which one has to prove that its injective dimension as
an $S$\+module does not exceed~$d$.

 As in part~(a), the argument proceeds by Noetherian induction in~$R$
(so we assume that the desired assertion holds for all the quotient
rings of $R$ by nonzero ideals).

 Case~I. Suppose that $R$ has zero-divisors.
 Let $a$ and $b\in R$ be a pair of nonzero elements for which $ab=0$.
 By Lemma~\ref{quotient-by-element}(b), \,${}_aK$ is
an $R/aR$\+injective $S/aS$\+module and ${}_bK$ is
an $R/bR$\+injective $S/bS$\+module.
 By the assumption of Noetherian induction, the flat dimensions of
the $S/aS$\+module ${}_aK$ and the $S/bS$\+module ${}_bK$ do not
exceed~$d$.

 We have to show that $\Ext_S^{d+1}(M,K)=0$ for all $S$\+modules~$M$.
 Using the short exact sequence of $S$\+modules $0\rarrow aM\rarrow
M\rarrow M/aM\rarrow0$, the question is reduced to $R$\+modules $M$
for which either $aM=0$ or $bM=0$.

 Suppose that $aM=0$.
 Let $J^\bu$ be an injective resolution of the $S$\+module~$K$;
then, by Lemma~\ref{resolution-quotient-by-element}(b),
\,${}_aJ^\bu$ is an injective resolution of the $S/aS$\+module~${}_aK$.
 Hence $\Ext_S^i(M,K)\simeq\Ext_{S/aS}^i(M,{}_aK)=0$ for $i>d$,
as desired.

 Case~II. Suppose that $R$ is an integral domain.
 It suffices to show that $\Ext_S^{d+1}(M,K)=0$ for all finitely
generated $S$\+modules~$M$.
 By Lemma~\ref{generic-freeness}, there exists a nonzero element
$a\in R$ such that the $R[a^{-1}]$\+module $M[a^{-1}]$ is free.
 According to the assertion of part~(a), it follows that the flat
dimension of the $S[a^{-1}]$\+module $M[a^{-1}]$ does not
exceed~$d$.
 Since $S[a^{-1}]$ is a Noetherian ring and $M[a^{-1}]$ is
a finitely generated $S[a^{-1}]$\+module, the projective dimension of
the $S[a^{-1}]$\+module $M[a^{-1}]$ is equal to its flat dimension
(as all finitely presented flat modules are projective).
 Thus the projective dimension also does not exceed~$d$.

 We have shown that $\Ext_{S[a^{-1}]}^i(M[a^{-1}],L)=0$ for all
$S[a^{-1}]$\+modules $L$ and all $i>d$.
 Let us apply this observation to the $S[a^{-1}]$\+module
$L=\Hom_S(S[a^{-1}],K)$.
 Notice that $\Ext^1_S(S[a^{-1}],K)\simeq\Ext^1_R(R[a^{-1}],K)=0$,
since $K$ is an injective $R$\+module.
 Using Lemma~\ref{ext-colocalization-by-element-lemma}, we can conclude
from $\Ext_{S[a^{-1}]}^i(M[a^{-1}],\Hom_S(S[a^{-1}],K))=0$
that $\Hom_S(S[a^{-1}],\Ext_S^i(M,K))=0$ for $i>d$.

 In view of Lemma~\ref{colocalization-by-element-vanishing-lemma}, it
now suffices to show that the map $a\:\Ext_S^{d+1}(M,K)\rarrow
\Ext_S^{d+1}(M,K)$ is surjective, i.~e., the $S$\+module
$\Ext_S^{d+1}(M,K)$ is $a$\+divisible.
 From this point on, the argument again proceeds similarly (or rather,
dually) to the proof of part~(a).

 Let $0\rarrow\Omega M\rarrow P\rarrow M\rarrow0$ be a short exact
sequence of $S$\+modules with a projective $S$\+module~$P$.
 If $d\ge1$, then the we have $\Ext_S^{d+1}(M,K)\simeq
\Ext_S^d(\Omega M,K)$.
 When $d=0$, the $S$\+module $\Ext_S^1(M,K)$ is the cokernel of
the morphism $\Hom_S(P,K)\rarrow\Hom_S(\Omega M,K)$.
 Since $P$ is a projective $S$\+module and $K$ is an injective
$R$\+module, the $R$\+module $\Hom_S(P,K)$ is injective; in particular,
it is $a$\+divisible (as $a$~is a nonzero-divisor in~$R$).
 So the quotient modules of $\Ext_S^1(M,K)$ and $\Hom_S(\Omega M,K)$
by the action of~$a$ are naturally isomorphic; in particular,
the $S$\+module $\Ext_S^1(M,K)$ is $a$\+divisible if and only if
the $S$\+module $\Hom_S(\Omega M,K)$ is.
 Thus, in both cases $d\ge1$ or $d=0$, it remains to show that
the $S$\+module $\Ext_S^d(\Omega M,K)$ is $a$\+divisible.
{\hbadness=1225\par}

 Both the $R$\+modules $S$ and $\Omega M$ contain no nonzero elements
annihilated by~$a$ (since $S$ is a flat $R$\+module and $\Omega M$
is a submodule of a projective $S$\+module), while the $R$\+module $K$
is $a$\+divisible (since it is injective).
 By Lemma~\ref{distinguished-triangle-quotient-by-element}(b),
we have a distinguished triangle in $\sD(S\modl)$
\begin{multline*}
 \boR\Hom_S(\Omega M,K)[-1]\lrarrow
 \boR\Hom_{S/aS}(\Omega M/a\Omega M,\,{}_aK) \\
 \lrarrow\boR\Hom_S(\Omega M,K)\overset a\lrarrow
 \boR\Hom_S(\Omega M,K).
\end{multline*}
 It follows from the related long exact sequence of cohomology modules
that $\Ext_{S/aS}^{d+1}(\Omega M/a\Omega M,\,{}_aK)\ne0$ whenever
the map $a\:\Ext_S^d(\Omega M,K)\rarrow\Ext_S^d(\Omega M,K)$ is not
surjective.

 Finally, similarly to the proof of part~(a), the assumption of
Noetherian induction is applicable to the ring homomorphism
$R/aR\rarrow S/aS$ and the $S/aS$\+module ${}_aK$, which is
$R/aR$\+injective by Lemma~\ref{quotient-by-element}(b).
 Hence $\Ext_{S/aS}^{d+1}(\Omega M/a\Omega M,\,{}_aK)=0$ and
we are done.
\end{proof}

\subsection{Preservation of the derived category of pro-sheaves}
\label{weakly-smooth-derived-pro-sheaves-preservation-subsecn}
 Let $\sE$ be an exact category.
 Assume that the additive category $\sE$ is \emph{weakly idempotent
complete}, i.~e., it contains the kernels of its split epimorphisms,
or equivalently, the cokernels of its split monomorphisms.
 A full subcategory $\sF\subset\sE$ is said to be \emph{resolving}
if the following conditions are satisfied:
\begin{enumerate}
\renewcommand{\theenumi}{\roman{enumi}}
\item $\sF$ is closed under extensions in $\sE$, i.~e., for
any admissible short exact sequence $0\rarrow E'\rarrow E\rarrow E''
\rarrow0$ in $\sE$ with $E'$, $E''\in\sF$ one has $E\in\sF$;
\item $\sF$ is closed under the kernels of admissible epimorphisms
in $\sE$, i.~e., for any admissible short exact sequence
$0\rarrow E'\rarrow E\rarrow E''\rarrow0$ in $\sE$ with $E$,
$E''\in\sF$ one has $E'\in\sF$;
\item for any object $E\in\sE$ there exists an object $F\in\sF$
together with an admissible epimorphism $F\rarrow E$ in~$\sE$.
\end{enumerate}

 Let $\sF\subset\sE$ be a resolving subcategory and $d\ge0$ be
an integer.
 One says that the \emph{$\sF$\+resolution dimension} of an object
$E\in\sE$ does not exceed~$d$ if there exists an exact sequence
$0\rarrow F_d\rarrow F_{d-1}\rarrow\dotsb\rarrow F_0\rarrow E\rarrow0$
in $\sE$ with $F_i\in\sF$.
 According to~\cite[Proposition~2.3(1)]{Sto0}
or~\cite[Corollary~A.5.2]{Pcosh}, the $\sF$\+resolution dimension
of an object $E\in\sE$ does not depend on the choice of a resolution.

 Since a resolving subcategory $\sF\subset\sE$ is closed under
extensions by~(i), it inherits an exact category structure from
the ambient exact category~$\sF$.

\begin{prop} \label{unbounded-derived-finite-resol-dim-prop}
 Let\/ $\sE$ be a weakly idempotent-complete exact category and\/
$\sF\subset\sE$ be a resolving subcategory such that, for a certain
finite integer $d\ge0$, the\/ $\sF$\+resolution dimensions of
all the objects in\/ $\sE$ do not exceed~$d$.
 Then the triangulated functor between the derived categories\/
$\sD(\sF)\rarrow\sD(\sE)$ induced by the exact inclusion of exact
categories\/ $\sF\rarrow\sE$ is an equivalence of triangulated
categories.
\end{prop}

\begin{proof}
 This is a part of~\cite[Proposition~A.5.6]{Pcosh}.
\end{proof}

 Let $\tau\:\X'\rarrow\X$ and $\pi'\:\bY\rarrow\X'$ be flat affine
morphisms of ind-semi-separated ind-schemes.
 Put $\pi=\tau\pi'\:\bY\rarrow\X$.
 Recall the notation $\bY_\X\flat$ for the exact category of $\X$\+flat
pro-quasi-coherent pro-sheaves on~$\bY$ (as defined in
Section~\ref{pro-flat-over-base-subsecn}).

\begin{lem} \label{pro-flat-over-base-weakly-smooth-resol-dim}
\textup{(a)} The exact category\/ $\bY_{\X'}\flat$ is a resolving full
subcategory in the exact category\/ $\bY_\X\flat$.
 The exact category structure of\/ $\bY_{\X'}\flat$ is inherited from
the exact category structure of\/ $\bY_\X\flat$. \par
\textup{(b)} Assume additionally that the ind-scheme\/ $\X$ is
ind-Noetherian and the morphism\/ $\tau\:\X'\rarrow\X$ is weakly smooth
of relative dimension~$\le d$ (in the sense of
Section~\ref{weakly-smooth-subsecn}) for some finite integer~$d$.
 Then the resolution dimension of any object of the exact category\/
$\bY_\X\flat$ with respect to the resolving subcategory\/
$\bY_{\X'}\flat$ does not exceed~$d$ (in other words, the\/
$\X'$\+flat dimension of any\/ $\X$\+flat pro-quasi-coherent pro-sheaf
on\/ $\bY$ does not exceed~$d$).
\end{lem}

\begin{proof}
 Part~(a): let $\bfG$ be an $\X'$\+flat pro-quasi-coherent pro-sheaf
on~$\bY$; so $\pi'_*\bfG$ is a flat pro-quasi-coherent pro-sheaf
on~$\X'$.
 Since $\tau\:\X'\rarrow\X$ is a flat affine morphism, the functor
$\tau_*\:\X'\pro\rarrow\X\pro$ takes flat pro-quasi-coherent pro-sheaves
on $\X'$ to flat pro-quasi-coherent pro-sheaves on~$\X$.
 Hence $\pi_*\bfG=\tau_*\pi'_*\bfG$ is a flat pro-quasi-coherent
pro-sheaf on~$\X$; so the pro-quasi-coherent pro-sheaf $\bfG$ on $\bY$
is $\X$\+flat.
 This proves the inclusion $\bY_{\X'}\flat\subset\bY_\X\flat$.
 
 Let $\X=\ilim_{\gamma\in\Gamma}X_\gamma$ be a representation of $\X$
by an inductive system of closed immersions of ind-schemes.
 Put $X'_\gamma=X_\gamma\times_\X\X'$ and
$\bnY_\gamma=X_\gamma\times_\X\bY$; then $\X'=\ilim_{\gamma\in\Gamma}
X'_\gamma$ and $\bY=\ilim_{\gamma\in\Gamma}\bnY_\gamma$ are similar
representations of $\X'$ and~$\bY$.
 We have flat affine morphisms of schemes $\bnY_\gamma\rarrow X'_\gamma
\rarrow X_\gamma$.

 A pro-quasi-coherent pro-sheaf $\bfG$ on $\bY$ is $\X$\+flat (resp.,
$\X'$\+flat) if and only if the quasi-coherent sheaf
$\bfG^{(\bnY_\gamma)}$ on $\bnY_\gamma$ is $X_\gamma$\+flat (resp.,
$X'_\gamma$\+flat) for every $\gamma\in\Gamma$.
 Furthermore, a short sequence $0\rarrow\bfF\rarrow\bfG\rarrow\bfH
\rarrow0$ is exact in $\bY_\X\flat$ (resp., in
$\bY_{\X'\flat}$) if and only if the short sequence
$0\rarrow\bfF^{(\bnY_\gamma)}\rarrow\bfG^{(\bnY_\gamma)}\rarrow
\bfH^{(\bnY_\gamma)}\rarrow0$ is exact in
$(\bnY_\gamma)_{X_\gamma}\flat$ (resp., in
$(\bnY_\gamma)_{X'_\gamma}\flat$) for every $\gamma\in\Gamma$.

 The full subcategory $(\bnY_\gamma)_{X'_\gamma}\flat\subset
(\bnY_\gamma)_{X_\gamma}\flat$ is closed under extensions and
the kernels of admissible epimorphisms, because the full subcategory
$X'_\gamma\flat\subset X'_\gamma\qcoh$ is.
 It follows that the full subcategory
$\bY_{\X'}\flat\subset\bY_\X\flat$ is closed under extensions and
the kernels of admissible epimorphisms, too.
 The exact category structure on $(\bnY_\gamma)_{X'_\gamma}\flat$
is inherited from that on $(\bnY_\gamma)_{X_\gamma}\flat$ (since both
are inherited from the abelian category $\bnY\qcoh$).
 It follows that the exact category structure of $\bY_{\X'}\flat$ is
inherited from $\bY_\X\flat$.

 We still have not used the assumption that the morphism~$\pi'$
(hence also~$\pi$) is flat; now we need to use it in order to construct
an admissible epimorphism onto any $\X$\+flat pro-quasi-coherent
pro-sheaf on $\bY$ from an $\X'$\+flat one.
 Indeed, let $\bfG$ be an $\X$\+flat pro-quasi-coherent pro-sheaf
on~$\bY$.
 Then the adjunction morphism $\pi^*\pi_*\bfG\rarrow\bfG$ is
an admissible epimorphism in $\bY_\X\flat$ (cf.\
Lemma~\ref{bar-complex-lemma}), and the pro-quasi-coherent pro-sheaf
$\pi^*\pi_*\bfG$ on $\bY$ is even flat, hence $\X'$\+flat.

 Part~(b): let $\bfG$ be an $\X$\+flat pro-quasi-coherent pro-sheaf
on $\bY$, and let $0\rarrow\bfF_d\rarrow\bfF_{d-1}\rarrow\dotsb
\rarrow\bfF_0\rarrow\bfG\rarrow0$ be an exact sequence in
$\bY_\X\flat$ with $\bfF_i\in\bY_{\X'}\flat$ for all $0\le i<d$.
 We need to show that $\bfF_d\in\bY_{\X'}\flat$.
 In the notation above, it suffices to check that
$\bfF_d^{(\bnY_\gamma)}\in(\bnY_\gamma)_{X'_\gamma}\flat$.
 We know that $\bfF_i^{(\bnY_\gamma)}\in(\bnY_\gamma)_{X'_\gamma}\flat$
for $0\le i<d$ and $\bfG^{(\bnY_\gamma)}\in
(\bnY_\gamma)_{X_\gamma}\flat$.

 Introduce the notation $\tau_\gamma\:X'_\gamma\rarrow X_\gamma$,
\ $\pi'_\gamma\:\bnY_\gamma\rarrow X'_\gamma$, and
$\pi_\gamma\:\bnY_\gamma\rarrow X_\gamma$ for the relevant affine
morphisms of schemes.
 We need to show that $\pi'_\gamma{}_*\bfF_d^{(\bnY_\gamma)}\in
X'_\gamma\flat$.
 Put $\F_i=\pi'_\gamma{}_*\bfF_i^{(\bnY_\gamma)}$ and
$\G=\pi'_\gamma{}_*\bfG^{(\bnY_\gamma)}\in X'_\gamma\qcoh$.
 Then we have an exact sequence $0\rarrow\F_d\rarrow\F_{d-1}\rarrow
\dotsb\rarrow\F_0\rarrow\G\rarrow0$ of quasi-coherent sheaves
on~$X'_\gamma$.
 We know that $\F_i\in X'_\gamma\flat$ for $0\le i<d$ and
$\tau_\gamma{}_*\G=\tau_\gamma{}_*\pi'_\gamma{}_*\bfG^{(\bnY_\gamma)}
=\pi_\gamma{}_*\bfG^{(\bnY_\gamma)}\in X_\gamma\flat$.

 Applying Proposition~\ref{flat-injective-dimensions-direct-image}(a)
to the quasi-coherent sheaf $\G$ on $X'_\gamma$ and the morphism of
schemes $\tau_\gamma\:X'_\gamma\rarrow X_\gamma$, which is affine and
weakly smooth of relative dimension~$\le d$ by assumptions, while
the scheme $X_\gamma$ is Noetherian and semi-separated, we conclude that
$\F_d\in X'_\gamma\flat$.
\end{proof}

\begin{cor} \label{weakly-smooth-derived-pro-sheaves-preservation-cor}
 Let\/ $\X$ be an ind-semi-separated ind-Noetherian ind-scheme,
$\tau\:\X'\rarrow\X$ be an affine morphism which is weakly smooth
of relative dimension~$\le d$, and\/ $\pi'\:\bY\rarrow\X'$ be
a flat affine morphism of ind-schemes.
 Put\/ $\pi=\tau\pi'$.
 Then the exact inclusion of exact categories\/ $\bY_{\X'}\flat\rarrow
\bY_\X\flat$ induces a triangulated equivalence between the derived
categories\/ $\sD(\bY_{\X'}\flat)\simeq\sD(\bY_\X\flat)$.
\end{cor}

\begin{proof}
 Follows from Lemma~\ref{pro-flat-over-base-weakly-smooth-resol-dim}
and Proposition~\ref{unbounded-derived-finite-resol-dim-prop}.
\end{proof}

\subsection{Preservation of the semiderived category of
torsion sheaves} \label{weakly-smooth-semiderived-preservation-subsecn}
 Let $\X$ be an ind-Noetherian ind-scheme, and let $\tau\:\X'\rarrow\X$
be an affine morphism of ind-schemes.
 Assume that the morphism~$\tau$ is of finite type and weakly smooth of
relative dimension~$\le d$ (in the sense of
Section~\ref{weakly-smooth-subsecn}) for some finite integer~$d$.
 Consider the direct image functor $\tau_*\:\X'\tors\rarrow\X\tors$.

\begin{prop} \label{weakly-smooth-direct-image-reflects-coacyclicity}
 A complex of quasi-coherent torsion sheaves $\rM^\bu$ on\/ $\X'$ is
coacyclic if and only if the complex of quasi-coherent torsion sheaves\/
$\tau_*\rM^\bu$ on\/ $\X$ is coacyclic.
\end{prop}

\begin{proof}
 It is clear from Lemmas~\ref{representable-by-schemes-direct-image}(a)
and~\ref{affine-torsion-direct-image} that the functor~$\tau_*$
(for any affine morphism of reasonable ind-schemes~$\tau$) takes
coacyclic complexes to coacyclic complexes.

 To prove the converse, assume that the complex $\tau_*\rM^\bu$ is
coacyclic in $\X\tors$.
 The ind-scheme $\X'$ is ind-Noetherian, so
Corollary~\ref{ind-Noetherian-coderived-cor} is applicable and
there exists a complex of injective quasi-coherent torsion sheaves
$\rJ^\bu$ on $\X'$ together with a morphism of complexes $\rM^\bu
\rarrow\rJ^\bu$ with a coacyclic cone.
 Then the cone of the morphism $\tau_*\rM^\bu\rarrow\tau_*\rJ^\bu$ is
also coacyclic, as we have already seen.
 Hence the complex $\tau_*\rJ^\bu$ is coacyclic in $\X\tors$, and
it remains to show that the complex $\rJ^\bu$ is coacyclic
in $\X'\tors$.

 Furthermore, for a flat morphism of reasonable ind-schemes
$\tau\:\X'\rarrow\X$ the functor $\tau_*\:\X'\tors\rarrow\X\tors$
takes injectives to injectives, since its left adjoint functor
$\tau^*\:\X\tors\rarrow\X'\tors$ is exact
(by Lemma~\ref{flat-torsion-sheaves-inverse-image}).
 So $\tau_*\rJ^\bu$ is a coacyclic complex of injectives in $\X\tors$,
and it follows that $\tau_*\rJ^\bu$ is a contractible complex.

 Therefore, for any closed subscheme $Z\subset\X$, the complex
$(\tau_*\rJ^\bu)_{(Z)}$ of quasi-coherent sheaves on $Z$ is
a contractible complex of injectives, too.
 Put $Z'=Z\times_\X\X'$ and denote by $\tau_Z\:Z'\rarrow Z$
the natural morphism; then we have $(\tau_*\rJ^\bu)_{(Z)}
=\tau_Z{}_*\rJ^\bu_{(Z')}$.
 Now $\rJ^\bu_{(Z')}$ is a complex of injective quasi-coherent
sheaves on~$Z'$.
 Since the complex $\tau_Z{}_*\rJ^\bu_{(Z')}$ is acyclic and
the functor $\tau_Z{}_*\:Z'\qcoh\rarrow Z\qcoh$ is exact and faithful
(the morphism~$\tau_Z$ being affine), it follows that $\rJ^\bu_{(Z')}$
is an acyclic complex.

 Denote by $\rK^n\in Z'\qcoh$ the quasi-coherent sheaves of cocycles
of the acyclic complex $\rJ^\bu_{(Z')}$.
 Then $\tau_Z{}_*\rK^n\in Z\qcoh$ are the quasi-coherent sheaves of
cocycles of the contractible complex of injective quasi-coherent
sheaves $\tau_Z{}_*\rJ^\bu_{(Z')}$ on~$Z$.
 Hence the quasi-coherent sheaves $\tau_Z{}_*\rK^n$ on $Z$ are
injective.

 By assumptions, the morphism of schemes $\tau_Z\:Z'\rarrow Z$ is
affine, weakly smooth of relative dimension~$\le d$, and of finite
type, while $Z$ is a Noetherian scheme.
 Applying Proposition~\ref{flat-injective-dimensions-direct-image}(b)
to the morphism~$\tau$ and each of the quasi-coherent sheaves $\rK^n$
on $Z'$, we see that they have finite injective dimensions in $Z'\qcoh$.
 As these are the objects of cocycles of an acyclic complex of
injectives $\rJ^\bu_{(Z')}$, we can conclude that the quasi-coherent
sheaves $\rK^n$ are injective and the complex $\rJ^\bu_{(Z')}$ is
contractible.

 Finally, Lemma~\ref{complex-of-injective-torsion-contractible} tells
that the complex of injective quasi-coherent torsion sheaves $\rJ^\bu$
on $\X'$ is contractible.
\end{proof}

\begin{cor} \label{weakly-smooth-semiderived-preservation-cor}
 Let\/ $\X$ be an ind-Noetherian ind-scheme, $\tau\:\X'\rarrow\X$ be
an affine morphism of finite type which is weakly smooth of relative
dimension~$\le d$, and\/ $\pi'\:\bY\rarrow\X'$ be a flat affine
morphism of ind-schemes.
 Put\/ $\pi=\tau\pi'$.
 Then the\/ $\bY/\X'$\+semiderived category of quasi-coherent torsion
sheaves on\/ $\bY$ coincides with the\/ $\bY/\X$\+semiderived category,
$\sD_{\X'}^\si(\bY\tors)=\sD_\X^\si(\bY\tors)$.
\end{cor}

\begin{proof}
 Let $\bN^\bu$ be a complex of quasi-coherent torsion sheaves on~$\bY$.
 Then we have an isomorphism of complexes $\pi_*\bN^\bu\simeq
\tau_*\pi'_*\bN^\bu$ in $\X\tors$.
 By Proposition~\ref{weakly-smooth-direct-image-reflects-coacyclicity},
it follows that the complex $\pi'_*\bN^\bu$ is coacyclic in $\X'\tors$
if and only if the complex $\pi_*\bN^\bu$ is coacyclic in $\X\tors$.
\end{proof}

\subsection{Derived restriction with supports commutes
with flat pulback} \label{derived-supports-pullback-subsecn}
 Let $f\:\Y\rarrow\X$ be a flat morphism of ind-Noetherian ind-schemes,
and let $i\:\Z\rarrow\X$ be a closed immersion of ind-schemes.
 Consider the pullback diagram (so $\W=\Z\times_\X\Y$)
$$
\xymatrix{
 \W \ar[r]^k \ar[d]_g & \Y \ar[d]^f \\
 \Z \ar[r]^i & \X
}
$$
 Since the morphisms $i$ and~$k$ are closed immersions, the ind-schemes
$\Z$ and $\W$ are also ind-Noetherian.

 The right derived functor $\boR i^!\:\sD^\co(\X\tors)\rarrow
\sD^\co(\Z\tors)$ was defined in Section~\ref{derived-supports-subsecn},
and similarly there is the right derived functor
$\boR k^!\:\sD^\co(\Y\tors)\rarrow\sD^\co(\W\tors)$.

 The morphisms of ind-schemes~$f$ and~$g$ are flat, so the inverse
image functors $f^*\:\X\tors\rarrow\Y\tors$ and $g^*\:\Z\tors\rarrow
\W\tors$ are exact by Lemma~\ref{flat-torsion-sheaves-inverse-image}.
 The functors $f^*$ and~$g^*$ are also left adjoints by
Lemma~\ref{torsion-direct-inverse-adjunction}(b), so they preserve
coproducts.
 Hence the functors $f^*$ and~$g^*$ take coacyclic complexes to
coacyclic complexes, and consequently induce well-defined inverse
image functors between the coderived categories,
$$
 f^*\:\sD^\co(\X\tors)\lrarrow\sD^\co(\Y\tors)
$$
and similarly $g^*\:\sD^\co(\Z\tors)\rarrow\sD^\co(\Y\tors)$.

 The aim of this Section~\ref{derived-supports-pullback-subsecn}
is to prove the following proposition (which is to be compared with
Proposition~\ref{derived-supports-external-product-prop}).

\begin{prop} \label{derived-supports-flat-pullback-prop}
 There is a natural isomorphism\/ $g^*\circ\boR i^!\simeq
\boR k^!\circ f^*$ of triangulated functors\/ $\sD^\co(\X\tors)\rarrow
\sD^\co(\W\tors)$.
\end{prop}

 The underived version of the natural isomorphism from
Proposition~\ref{derived-supports-flat-pullback-prop} is provided by
Lemma~\ref{torsion-reasonable-shriek-flat-star-commutation}.
 Given the underived version, the assertion of the proposition follows
almost immediately from the next lemma.

\begin{lem} \label{derived-supports-flat-pullback-lemma}
 Let $\rJ^\bu$ be a complex of injective quasi-coherent torsion sheaves
on\/ $\X$, and let $r\:f^*\rJ^\bu\rarrow\rK^\bu$ be a morphism of
complexes of quasi-coherent torsion sheaves on\/ $\Y$ such $\rK^\bu$
is a complex of injective quasi-coherent torsion sheaves and
the cone of\/~$r$ is a coacyclic complex of quasi-coherent torsion
sheaves on\/~$\Y$.
 Then the induced morphism of complexes of quasi-coherent torsion
sheaves on\/~$\W$
$$
 k^!(r)\:k^!f^*\rJ^\bu\lrarrow k^!\rK^\bu
$$
has coacyclic cone.
\end{lem}

 The proof of Lemma~\ref{derived-supports-flat-pullback-lemma}
will be given below near the end of
Section~\ref{derived-supports-pullback-subsecn}.

\begin{lem} \label{flat-extension-of-scalars-Ext}
 Let $R$ be a Noetherian commutative ring and $R\rarrow S$ be
a homomorphism of commutative rings such that $S$ is a flat $R$\+module.
 Let $M$ be a finitely generated $R$\+module and $N$ be an $R$\+module.
 Then for every $n\ge0$ there is a natural isomorphism of $S$\+modules
$$
 \Ext_S^n(S\ot_RM,\>S\ot_RN)\simeq S\ot_R\Ext_R^n(M,N).
$$
 In particular, for any injective $R$\+module $J$ one has\/
$\Ext_S^n(S\ot_RM,\>S\ot_RJ)=0$ for all $n>0$.
\end{lem}

\begin{proof}
 The assumption of commutativity of $S$ can be actually dropped (then
one obtains $S$\+$S$\+bimodule isomorphisms).
 We observe that for any finitely generated $R$\+module $P$ there are
natural isomorphisms of $S$\+modules $\Hom_S(S\ot_RP,\>S\ot_RN)\simeq
\Hom_R(P,\>S\ot_RN)\simeq S\ot_R\Hom_R(P,N)$.
 Now let $P_\bu\rarrow M$ be a resolution of $M$ by finitely generated
projective $R$\+modules.
 The $S\ot_RP_\bu\rarrow S\ot_RM$ is a resolution of $S\ot_RM$ by
finitely generated projective $S$\+modules.
 The isomorphism of complexes of $S$\+modules $\Hom_S(S\ot_RP_\bu,\>
S\ot_RN)\simeq S\ot_R\Hom_R(P_\bu,N)$ induces the desired isomorphism
for the Ext modules.
\end{proof}

\begin{lem} \label{support-flat-pullback-exactness}
 Let $f\:Y\rarrow X$ be a flat morphism of Noetherian schemes, and
let $i\:Z\rarrow X$ be a closed immersion of schemes.
 Put $W=Z\times_XY$, and denote by $g\:W\rarrow Z$ and $k\:W\rarrow Y$
the natural morphisms.
 Let $\J$ be an injective quasi-coherent sheaf on $X$ and $f^*\J
\rarrow\K^\bu$ be an injective resolution of the quasi-coherent sheaf
$f^*\J$ on~$Y$.
 Then one has
$$
 H^0(k^!\K^\bu)\simeq g^*i^!\J
 \quad\text{and}\quad
 H^n(k^!\K^\bu)=0 \text{ \ for\/ $n>0$}.
$$
\end{lem}

\begin{proof}
 To compute $H^0(k^!\K^\bu)$, one can observe that the functor~$k^!$
is left exact (as a right adjoint), so $H^0(k^!\K^\bu)\simeq k^!f^*\J
\simeq g^*i^!\J$ by Lemma~\ref{reasonable-shriek-flat-star-commutation}.
 The cohomology vanishing assertion is local in $X$ and in $Y$ (since
injectivity of a quasi-coherent sheaf on a Noetherian scheme is a local
property), so it reduces to the case of affine schemes, for which it
means the following.

 Let $R\rarrow S$ be a homomorphism of Noetherian commutative rings
such that $S$ is a flat $R$\+module, and let $R\rarrow T$ be
a surjective homomorphism of commutative rings.
 Let $J$ be an injective $R$\+module, and let $K^\bu$ be an injective
resolution of the $S$\+module $S\ot_RJ$.
 Then the complex $\Hom_S(S\ot_RT,\>K^\bu)$ has vanishing cohomology
in the positive cohomological degrees.
 This is a particular case of Lemma~\ref{flat-extension-of-scalars-Ext}.
\end{proof}

\begin{lem} \label{torsion-subscheme-support-flat-pullback-exactness}
 Let $f\:\Y\rarrow\X$ be a flat morphism of ind-Noetherian ind-schemes,
and let $Z\subset\X$ be a closed subscheme with the closed immersion
morphism $i\:Z\rarrow\X$.
 Put $W=Z\times_\X\Y$ (so $W$ is a closed subscheme in\/~$\Y$), and
denote by $f_Z\:W\rarrow Z$ and $k\:W\rarrow\Y$ the natural morphisms.
 Let $\rJ$ be an injective quasi-coherent torsion sheaf on\/ $\X$
and $f^*\rJ\rarrow\rK^\bu$ be an injective resolution of
the quasi-coherent torsion sheaf $f^*\rJ$ on\/~$\Y$.
 Then one has
$$
 H^0(k^!\rK^\bu)\simeq f_Z^*i^!\rJ
 \quad\text{and}\quad
 H^n(k^!\rK^\bu)=0 \text{ \ for\/ $n>0$}.
$$
\end{lem}

\begin{proof}
 The computation of $H^0$ is similar to the one in
Lemma~\ref{support-flat-pullback-exactness}
(use Remark~\ref{simplified-flat-torsion-inverse-image}).
 To prove the higher cohomology vanishing, choose an inductive system
of closed immersions of schemes $(X_\gamma)_{\gamma\in\Gamma}$
representing the ind-scheme $\X$, so
$\X=\ilim_{\gamma\in\Gamma}X_\gamma$.
 We can always assume that there exists $\gamma_0\in\Gamma$ such that
$Z=X_{\gamma_0}$.
 Put $Y_\gamma=X_\gamma\times_\X\Y$; then $\Y=\ilim_{\gamma\in\Gamma}
Y_\gamma$ is a representation of $\Y$ by an inductive system of
closed immersions of schemes.
 Let $f_\gamma\:Y_\gamma\rarrow X_\gamma$ denote the natural morphism.

 Our aim is to show that the exact sequence of quasi-coherent torsion
sheaves $0\rarrow f^*\rJ\rarrow\rK^0\rarrow\rK^1\rarrow\rK^2\rarrow
\dotsb$ remains exact after applying the functor $\rN\longmapsto
\rN|_\Gamma\:\Y\tors\rarrow(\Y,\Gamma)\syst$.
 Notice that, by Remark~\ref{simplified-flat-torsion-inverse-image},
we have $(f^*\rJ)|_\Gamma\simeq f^*(\rJ|_\Gamma)$, or in other
words, $(f^*\rJ)_{(Y_\gamma)}=f_\gamma^*\rJ_{(X_\gamma)}$.

 Denote by $\boN^1\in(\Y,\Gamma)\syst$ the cokernel of the morphism
of $\Gamma$\+systems $(f^*\rJ)|_\Gamma\allowbreak\rarrow\rK^0|_\Gamma$.
 Let $\gamma<\delta\in\Gamma$ be a pair of indices.
 Denote the related transition maps in the inductive systems by
$i_{\gamma\delta}\:X_\gamma\rarrow X_\delta$ and
$k_{\gamma\delta}\:Y_\gamma\rarrow Y_\delta$.

 The quasi-coherent sheaves $\rK^n_{(Y_\delta)}$ on the scheme
$Y_\delta$ are injective for all $n\ge0$, since the quasi-coherent
torsion sheaves $\rK^n$ on the ind-scheme $\Y$ are injective.
 By Lemma~\ref{support-flat-pullback-exactness}, the short exact
sequence $0\rarrow f_\delta^*\rJ_{(X_\delta)}\rarrow\rK^0_{(Y_\delta)}
\rarrow\boN^1_{(\delta)}\rarrow0$ remains exact after applying
the functor~$k_{\gamma\delta}^!$.
 Hence the structure map $\boN^1_{(\gamma)}\rarrow k_{\gamma\delta}^!
\boN^1_{(\delta)}$ in the $\Gamma$\+system $\boN^1$ is an isomorphism
(of quasi-coherent sheaves on~$Y_\gamma$).

 As this holds for all $\gamma<\delta\in\Gamma$, we can conclude that
the collection of quasi-coherent sheaves $\rN^1_{(Y_\gamma)}=
\boN^1_{(\gamma)}$, \,$\gamma\in\Gamma$ defines a quasi-coherent
torsion sheaf $\rN^1$ on~$\Y$.
 So we have $\boN^1=\rN^1|_\Gamma$ and $\rN^1=\boN^1{}^+$; in other
words, this means that the adjunction morphism $\boN^1\rarrow
\boN^1{}^+|_\Gamma$ is an isomorphism of $\Gamma$\+systems.
 Notice that the quasi-coherent torsion sheaf $\boN^1{}^+$ on $\Y$
is, by the definition, the cokernel of the monomorphism of
quasi-coherent torsion sheaves $f^*\rJ\rarrow\rK^0$.
 We have shown that the short exact sequence of quasi-coherent
torsion sheaves $0\rarrow f^*\rJ\rarrow\rK^0\rarrow\boN^1{}^+\rarrow0$
on $\Y$ remains exact after applying the functor $\rN\longmapsto
\rN|_\Gamma$.

 The argument finishes similarly to the proof of
Lemma~\ref{torsion-subscheme-support-external-product-exactness},
proceeding step by step up the cohomological degree and using
Lemma~\ref{support-flat-pullback-exactness} on every step.
\end{proof}

\begin{lem} \label{torsion-support-flat-pullback-exactness}
 Let $f\:\Y\rarrow\X$ be a flat morphism of ind-Noetherian ind-schemes,
and let $i\:\Z\subset\X$ be a closed immersion of schemes.
 Put $\W=\Z\times_\X\Y$, and denote by $g\:\W\rarrow\Z$ and
$k\:\W\rarrow\Y$ the natural morphisms.
 Let $\rJ$ be an injective quasi-coherent torsion sheaf on\/ $\X$
and $f^*\rJ\rarrow\rK^\bu$ be an injective resolution of
the quasi-coherent torsion sheaf $f^*\rJ$ on\/~$\Y$.
 Then one has
$$
 H^0(k^!\rK^\bu)\simeq g^*i^!\rJ
 \quad\text{and}\quad
 H^n(k^!\rK^\bu)=0 \text{ \ for\/ $n>0$}.
$$
\end{lem}

\begin{proof}
 The computation of $H^0$ is similar to the one in
Lemmas~\ref{support-flat-pullback-exactness}
and~\ref{torsion-subscheme-support-flat-pullback-exactness}.
 The functor~$k^!$ is left exact as a right adjoint, and it remains
to use Lemma~\ref{torsion-reasonable-shriek-flat-star-commutation}.
 To prove the vanishing assertion, choose a closed subscheme
$Z\subset\Z$ with the closed immersion morphism $j\:Z\rarrow\Z$.
 Put $W=Z\times_\Z\W=Z\times_\X\Y$, and denote by $l\:W\rarrow\W$
the natural closed immersion.
 Then by Lemma~\ref{torsion-subscheme-support-flat-pullback-exactness}
we have $H^n(l^!k^!\rK^\bu)=0$ for $n>0$, and it follows that
$H^n(k^!\rK^\bu)=0$ for $n>0$ as well.
\end{proof}

\begin{proof}[Proof of Lemma~\ref{derived-supports-flat-pullback-lemma}]
 The argument is similar to the proof of
Lemma~\ref{derived-supports-external-product-lemma}.
 Given a complex $\rJ^\bu$ of quasi-coherent torsion sheaves on $\X$,
the related complex $\rK^\bu$ of quasi-coherent torsion sheaves on $\Y$
is defined uniquely up to a homotopy equivalence
(by Proposition~\ref{coderived-and-homotopy-of-injectives}(a)); so it
suffices to prove the assertion of the lemma for one specific choice
of the complex~$\rK^\bu$.
 We will use the complex $\rK^\bu$ provided by the construction on
which the proof of
Proposition~\ref{coderived-and-homotopy-of-injectives}(b) is based.

 Let $f^*\rJ^\bu\rarrow\rL^{0,\bu}$ be a monomorphism of complexes in
$\Y\tors$ such that $\rL^{0,\bu}$ is a complex of injective
quasi-coherent torsion sheaves.
 Denote by $\rN^{1,\bu}$ the cokernel of this morphism of complexes, and
let $\rN^{1,\bu}\rarrow\rL^{1,\bu}$ be a monomorphism of complexes in
which $\rL^{1,\bu}$ is a complex of injectives.
 Proceeding in this way, we construct a bounded below complex of
complexes of injective quasi-coherent torsion sheaves $\rL^{\bu,\bu}$
together with a quasi-isomorphism $f^*\rJ^\bu\rarrow\rL^{\bu,\bu}$ of
complexes of complexes in $\Y\tors$.
 In every cohomological degree~$n$, the complex $\rL^{\bu,n}$ is
an injective resolution of the quasi-coherent torsion sheaf $f^*\rJ^n$
on~$\Y$.
 The complex $\rK^\bu$ is then constructed by totalizing the bicomplex
$\rL^{\bu,\bu}$ using infinite coproducts along the diagonals.

 Recall that the functor $k^!\:\Y\tors\rarrow\W\tors$ preserves
coproducts.
 In every cohomological degree~$n$, applying~$k^!$ to the complex
$0\rarrow f^*\rJ^n\rarrow\rL^{0,n}\rarrow\rL^{1,n}\rarrow\dotsb$
produces an acyclic complex in $\W\tors$ by
Lemma~\ref{torsion-support-flat-pullback-exactness}.
 It remains to use the fact that the coproduct totalization of
an acyclic bounded below complex of complexes is a coacyclic
complex~\cite[Lemma~2.1]{Psemi}.
\end{proof}

\begin{proof}[Proof of
Proposition~\ref{derived-supports-flat-pullback-prop}]
 Let $\rM^\bu$ be a complex of quasi-coherent torsion sheaves on~$\X$.
 Choose a complex of injective quasi-coherent torsion sheaves $\rJ^\bu$
on $\X$ together with a morphism $\rM^\bu\rarrow\rJ^\bu$ with
a coacyclic cone.
 Then the cone of the morphism $f^*\rM^\bu\rarrow f^*\rJ^\bu$ is
a coacyclic complex of quasi-coherent torsion sheaves on~$\Y$.
 Choose a complex of injective quasi-coherent torsion sheaves $\rK^\bu$
on $\Y$ together with a morphism $f^*\rJ^\bu\rarrow\rK^\bu$ with
a coacyclic cone.
 By Lemma~\ref{derived-supports-flat-pullback-lemma}, the cone of
the morphism $k^!f^*\rJ^\bu\rarrow k^!\rK^\bu$ is a coacyclic complex
of quasi-coherent torsion sheaves on~$\W$.
 Thus the complex $k^!f^*\rJ^\bu$ represents the object
$\boR k^!\circ f^*(\rM^\bu)$ in the coderived category
$\sD^\co(\W\tors)$.

 On the other hand, the complex $i^!\rJ^\bu$ represents the object
$\boR i^!(\rM^\bu)\in\sD^\co(\Z\tors)$, hence the complex
$g^*i^!\rJ^\bu$ represents the object $g^*\circ\boR i^!(\rM^\bu)\in
\sD^\co(\W\tors)$.
 It remains to recall the isomorphism $g^*i^!\rJ^\bu\simeq
k^!f^*\rJ^\bu$ of complexes of quasi-coherent torsion sheaves on $\W$
provided by Lemma~\ref{torsion-reasonable-shriek-flat-star-commutation}.
\end{proof}

\subsection{Preservation of the semiderived equivalence}
 Let $\tau\:X'\rarrow X$ be a morphism of semi-separated Noetherian
schemes.
 Assume that the morphism~$\tau$ is of finite type and weakly smooth
of relative dimension~$\le d$.
 Consider the inverse image functor $\tau^*\:X\qcoh\rarrow X'\qcoh$.

\begin{lem} \label{scheme-weakly-smooth-pullback-of-dualizing}
 Let\/ $\D^\bu$ be a dualizing complex of quasi-coherent sheaves on~$X$.
 Choose a complex of injective quasi-coherent sheaves\/ $\D'{}^\bu$
on $X'$ together with a morphism of complexes\/ $\tau^*\D^\bu\rarrow
\D'{}^\bu$ with a coacyclic cone.
 Then\/ $\D'{}^\bu$ is a dualizing complex of quasi-coherent sheaves
on~$X'$.
\end{lem}

\begin{proof}
 It is helpful to keep in mind that a bounded below complex is coacyclic
if and only if it is acyclic, and any (unbounded) coacyclic complex of
injectives is contractible.
 Up to the homotopy equivalence, one can assume both $\D^\bu$ and
$\D'{}^\bu$ to be bounded below, and then it suffices that the cone
be acyclic.

 The assertion is local in both $X$ and $X'$, so it reduces to the case
of affine schemes, for which it consists of three independent
observations corresponding to the three conditions~(i\+-iii) in
the definition of a dualizing complex in
Section~\ref{dualizing-definition-subsecn}.
 In each case, we consider a homomorphism of Noetherian commutative
rings $R\rarrow S$ such that $S$ is a flat $R$\+module.

 (1)~Assume that $S$ is a finitely generated $R$\+algebra and
for every prime ideal $\p\subset R$, the fiber ring
$\kappa(\p)\ot_RS$ is regular of Krull dimension~$\le d$. 
 If $D^\bu$ is a bounded complex of injective $R$\+modules, then
the complex of $S$\+modules $S\ot_R D^\bu$ is quasi-isomorphic to
a bounded complex of injective $S$\+modules.

 Indeed, it suffices to show that, for every injective $R$\+module $J$,
the $S$\+module $S\ot_RJ$ has finite injective dimension.
 Notice that the $R$\+module $S\ot_RJ$ is injective (since
the $R$\+module $S$ is flat and $R$ is a Noetherian ring); hence
the assertion follows from
Proposition~\ref{flat-injective-dimensions-direct-image}(b).

 In fact, according to~\cite[Theorem~1]{Fox} (see
also~\cite[Corollary~1]{FT}), it suffices to assume the fiber rings
to be Gorenstein (of bounded Krull dimension) rather than regular,
and the assumption that $S$ is a finitely generated $R$\+algebra
can be dropped.
 The assumptions on the morphism~$\tau$ before the formulation of
the lemma can be weakened accordingly.

 (2)~If $D^\bu$ is a complex of $R$\+modules with finitely generated
cohomology $R$\+modules, then $S\ot_RD^\bu$ is a complex of $S$\+modules
with finitely generated cohomology $S$\+modules.
 This is obvious for a flat $R$\+algebra~$S$.

 (3)~If $D^\bu$ is a bounded complex of $R$\+modules with finitely
generated cohomology $R$\+modules and the homothety map
$R\rarrow\boR\Hom_R(D^\bu,D^\bu)$ is a quasi-isomorphism, then
the homothety map $S\rarrow\boR\Hom_S(S\ot_RD^\bu,\>S\ot_RD^\bu)$
is a quasi-isomorphism, too.

 More generally, if $M^\bu$ is a bounded above complex of $R$\+modules
with finitely generated cohomology modules and $N^\bu$ is a bounded
below complex of $R$\+modules, then there is a natural isomorphism
$\boR\Hom_S(S\ot_RM^\bu,\>S\ot_RN^\bu)\simeq S\ot_R
\boR\Hom_R(M^\bu,N^\bu)$ in the derived category of $S$\+modules.
 This is a straightforward generalization of
Lemma~\ref{flat-extension-of-scalars-Ext}; one just needs to
replace $M^\bu$ with a quasi-isomorphic bounded above complex of
finitely generated projective $R$\+modules.
\end{proof}

 Let $\tau\:\X'\rarrow\X$ be an morphism of ind-semi-separated
ind-Noetherian ind-schemes.
 Assume that the morphism~$\tau$ is of finite type and weakly smooth of
relative dimension~$\le d$.
 Consider the inverse image functor $\tau^*\:\X\tors\rarrow\X'\tors$.

\begin{lem} \label{weakly-smooth-pullback-of-dualizing-torsion}
 Let\/ $\rD^\bu$ be a dualizing complex of quasi-coherent torsion
sheaves on\/~$\X$.
 Choose a complex of injective quasi-coherent torsion sheaves\/
$\rD'{}^\bu$ on\/ $\X'$ together with a morphism of complexes\/
$\tau^*\rD^\bu\rarrow\rD'{}^\bu$ with a coacyclic cone.
 Then\/ $\rD'{}^\bu$ is a dualizing complex of quasi-coherent torsion
sheaves on\/~$\X'$.
\end{lem}

\begin{proof}
 Let $Z\subset\X$ be a closed subscheme; put $Z'=Z\times_\X\X'$.
 Let $i\:Z\rarrow\X$ and $k\:Z'\rarrow\X'$ denote the closed
immersion morphisms, and let $\tau_Z\:Z'\rarrow Z$ be the natural
morphism (which is of finite type and weakly smooth of relative
dimension~$\le d$).
 By Proposition~\ref{derived-supports-flat-pullback-prop}, we have
a natural isomorphism $\tau_Z^*i^!\rD^\bu=\tau_Z^*\boR i^!\rD^\bu
\simeq\boR k^!\tau^*\rD^\bu=k^!\rD'{}^\bu$ in $\sD^\co(Z'\qcoh)$.
 Since $\rD^\bu$ is a dualizing complex on $\X'$, the complex
of quasi-coherent sheaves $i^!\rD^\bu$ is a dualizing complex on~$Z$.
 Now Lemma~\ref{scheme-weakly-smooth-pullback-of-dualizing} applied
to the morphism of schemes~$\tau_Z$ tells that $k^!\rD'{}^\bu$ is
a dualizing complex on~$Z'$, which by the definition means
that $\rD'{}^\bu$ is a dualizing complex on~$\X'$ (see condition~(iv)
in Section~\ref{dualizing-definition-subsecn} and the discussion
after it).
\end{proof}

 Let $\X$ be an ind-semi-separated ind-Noetherian ind-scheme,
$\tau\:\X'\rarrow\X$ be an affine morphism of finite type which is
weakly smooth of relative dimension~$\le d$, and
$\pi'\:\bY\rarrow\X'$ be a flat affine morphism of ind-schemes.
 Put $\pi=\tau\pi'$.

 Let $\rD^\bu$ be a dualizing complex on $\X$ and
$\rD'{}^\bu$ be the related dualizing complex on $\X'$, as per
the rule of Lemma~\ref{weakly-smooth-pullback-of-dualizing-torsion}.
 Then Theorem~\ref{relative-triangulated-equivalence-thm} provides
triangulated equivalences
$$
 \pi^*\rD^\bu\ot_\bY{-}\,\:
 \sD(\bY_\X\flat)\,\simeq\,\sD_\X^\si(\bY\tors)
$$
and
$$
 \pi'^*\rD'{}^\bu\ot_\bY{-}\,\:
 \sD(\bY_{\X'}\flat)\,\simeq\,\sD_{\X'}^\si(\bY\tors).
$$
 
\begin{prop} \label{weakly-smooth-semiderived-equivalences-preserved}
 In the context above, the triangulated equivalences\/
$\sD(\bY_{\X'}\flat)\allowbreak\simeq\sD_{\X'}^\si(\bY\tors)$ and\/
$\sD(\bY_\X\flat)\simeq\sD_\X^\si(\bY\tors)$ from
Theorem~\ref{relative-triangulated-equivalence-thm} form
a commutative square diagram with the triangulated equivalences\/
$\sD(\bY_{\X'}\flat)\simeq\sD(\bY_\X\flat)$ and\/
$\sD_{\X'}^\si(\bY\tors)\simeq\sD_\X^\si(\bY\tors)$ from
Corollaries~\ref{weakly-smooth-derived-pro-sheaves-preservation-cor}
and~\ref{weakly-smooth-semiderived-preservation-cor}.
\end{prop}

\begin{proof}
 Notice the natural isomorphism $\pi^*\rD^\bu\simeq\pi'{}^*\tau^*
\rD^\bu$ of complexes of quasi-coherent torsion sheaves on~$\bY$.
 We are given a morphism $\tau^*\rD^\bu\rarrow\rD'{}^\bu$ of
complexes of quasi-coherent torsion sheaves on $\X'$, whose cone
is coacyclic in $\X'\tors$.
 According to the discussion in the beginning of
Section~\ref{semiderived-and-external-subsecn}, based on the existence
of a well-defined functor~\eqref{derived-flat-lifted-torsion-tensor}
from Section~\ref{underived-tensor-subsecn}, the triangulated
functors $\bPhi_{\tau^*\rD^\bu}=\pi'{}^*\tau^*\rD^\bu\ot_\bY{-}\,\:
\sD(\bY_{\X'}\flat)\rarrow\sD_{\X'}^\si(\bY\tors)$
and $\bPhi_{\rD'{}^\bu}=\pi'{}^*\rD'{}^\bu\ot_\bY{-}\,\:
\sD(\bY_{\X'}\flat)\rarrow\sD_{\X'}^\si(\bY\tors)$ are naturally
isomorphic. \hbadness=1075
\end{proof}

\subsection{Preservation of the semitensor product}
 Let $\X$ be an ind-semi-separated ind-Noetherian ind-scheme,
$\tau\:\X'\rarrow\X$ be an affine morphism of finite type which is
weakly smooth of relative dimension~$\le d$, and
$\pi'\:\bY\rarrow\X'$ be a flat affine morphism of ind-schemes.
 Put $\pi=\tau\pi'$.

 A construction from Section~\ref{left-derived-tensor-products-subsecn}
(see formula~\eqref{derived-tensor-of-pro-sheaves}) defines the left
derived tensor product functors
$$
 \ot^{\bY,\boL}=\ot^{\bY/\X,\boL}\:
 \sD(\bY_\X\flat)\times\sD(\bY_\X\flat)\lrarrow\sD(\bY_\X\flat)
$$
and
$$
 \ot^{\bY,\boL}=\ot^{\bY/\X',\boL}\:
 \sD(\bY_{\X'}\flat)\times\sD(\bY_{\X'}\flat)\lrarrow
 \sD(\bY_{\X'}\flat),
$$
endowing the derived categories $\sD(\bY_\X\flat)$ and
$\sD(\bY_{\X'}\flat)$ with tensor triangulated category structures.

\begin{prop}
\label{weakly-smooth-pro-sheaves-tensor-structure-preserved}
 In the context above, the triangulated equivalence\/
$\sD(\bY_{\X'}\flat)\simeq\sD(\bY_\X\flat)$ from
Corollary~\ref{weakly-smooth-derived-pro-sheaves-preservation-cor}
is an equivalence of tensor triangulated categories.
\end{prop}

\begin{proof}
 We just have to show that the two constructions of derived functors
of tensor product agree.
 The definition of a relatively homotopy flat complex of flat
pro-quasi-coherent pro-sheaves on $\bY$ was given in
Section~\ref{relatively-homotopy-flat-subsecn}.
 To distinguish the two settings in the situation at hand, let us
speak about \emph{$\pi$\+relatively homotopy flat complexes} and
\emph{$\pi'$\+relatively homotopy flat complexes}.

 The key observation is that every $\pi$\+relatively homotopy flat
complex is $\pi'$\+relatively homotopy flat.
 Indeed, we have $\bY_{\X'}\flat\subset\bY_\X\flat$.
 Furthermore, it follows from
Corollary~\ref{weakly-smooth-derived-pro-sheaves-preservation-cor}
that a complex in $\bY_{\X'}\flat$ is acyclic if and only if it is
acyclic in $\bY_\X\flat$.
 Hence, given a complex $\bfF^\bu$ in $\bY\flat$, condition~(i)
from Section~\ref{relatively-homotopy-flat-subsecn} holds
for $\bfF^\bu$ with respect to the morphism $\pi'\:\bY\rarrow\X'$
whenever it holds with respect to the morphism $\pi\:\bY\rarrow\X$.
 Concerning condition~(ii), it is clear from
Proposition~\ref{weakly-smooth-direct-image-reflects-coacyclicity}
that this condition holds for a given complex $\bfF^\bu\in\sC(\bY\flat)$
with respect to the morphism~$\pi'$ if and only if it holds with
respect to the morphism~$\pi$.

 Now let $\bfP^\bu$ and $\bfQ^\bu\in\sC(\bY_{\X'}\flat)$ be two
complexes of $\X'$\+flat (hence also $\X\flat$) pro-quasi-coherent
pro-sheaves on~$\bY$.
 Using Proposition~\ref{relatively-homotopy-flat-resolution}, choose
two morphisms of complexes of $\X$\+flat pro-quasi-coherent pro-sheaves
$\bfF^\bu\rarrow\bfP^\bu$ and $\bfG^\bu\rarrow\bfQ^\bu$ such that
the cones of both morphisms are acyclic in $\bY_\X\flat$, and
both the complexes $\bfF^\bu$ and $\bfG^\bu\in\sC(\bY\flat)$ are
$\pi$\+relatively homotopy flat complexes of flat pro-quasi-coherent
pro-sheaves on~$\bY$.
 Then both the morphisms can be also viewed as morphisms of complexes
of $\X'$\+flat pro-quasi-coherent pro-sheaves on $\bY$, the cones of
both the morphisms are also acyclic in $\bY_{\X'}\flat$, and
both the complexes $\bfF^\bu$ and $\bfG^\bu$ are also
$\pi'$\+relatively homotopy flat.
 Thus the tensor product $\bfF^\bu\ot^\bY\bfG^\bu$ computes both
the derived functors $\bfP^\bu\ot^{\bY/\X,\boL}\bfQ^\bu$ and
$\bfP^\bu\ot^{\bY/\X',\boL}\bfQ^\bu$.
\end{proof}

 Furthermore, another construction from
Section~\ref{left-derived-tensor-products-subsecn}
(see formula~\eqref{derived-tensor-of-pro-and-torsion})
defines the left derived tensor product functors
$$
 \ot_\bY^\boL=\ot_{\bY/\X}^\boL\:
 \sD(\bY_\X\flat)\times\sD_\X^\si(\bY\tors)
 \lrarrow\sD_\X^\si(\bY\tors)
$$
and
$$
 \ot_\bY^\boL=\ot_{\bY/\X'}^\boL\:
 \sD(\bY_{\X'}\flat)\times\sD_{\X'}^\si(\bY\tors)
 \lrarrow\sD_{\X'}^\si(\bY\tors)
$$
endowing the semiderived categories $\sD_\X^\si(\bY\tors)$ and
$\sD_{\X'}^\si(\bY\tors)$ with triangulated module category
structures over the tensor categories $\sD(\bY_\X\flat)$ and
$\sD(\bY_{\X'}\flat)$, respectively.

\begin{prop} \label{weakly-smooth-torsion-module-structure-preserved}
 In the context above, the triangulated equivalences\/
$\sD(\bY_{\X'}\flat)\allowbreak\simeq\sD(\bY_\X\flat)$ and\/
$\sD_{\X'}^\si(\bY\tors)\simeq\sD_\X^\si(\bY\tors)$ from
Corollaries~\ref{weakly-smooth-derived-pro-sheaves-preservation-cor}
and~\ref{weakly-smooth-semiderived-preservation-cor} preserve
the module category structures on\/ $\sD_\X^\si(\bY\tors)$ and\/
$\sD_{\X'}^\si(\bY\tors)$ over the tensor categories\/
$\sD(\bY_\X\flat)$ and\/ $\sD(\bY_{\X'}\flat)$.
\end{prop}

\begin{proof}
 The argument is similar to the proof of
Proposition~\ref{weakly-smooth-pro-sheaves-tensor-structure-preserved}.
 We just have to show that the two constructions of derived functors
of tensor product agree.
 The definition of a homotopy $\bY/\X$\+flat complex of quasi-coherent
torsion sheaves on $\bY$ was given in
Section~\ref{relatively-homotopy-flat-subsecn}; similarly one
defines the homotopy $\bY/\X'$\+flat complexes.

 The key observation is that every homotopy $\bY/\X$\+flat complex in
$\bY\tors$ is $\bY/\X'\flat$.
 This holds because $\bY_{\X'}\flat\subset\bY_\X\flat$
a complex in $\bY_{\X'}\flat$ is acyclic if and only if it is
acyclic in $\bY_\X\flat$, while a complex in $\X'\tors$ is coacyclic
if and only if its direct image is coacyclic in $\X\tors$.

 Now let $\bfP^\bu$ be a complex of $\X'$\+flat pro-quasi-coherent
pro-sheaves and $\bN^\bu$ be a complex of quasi-coherent torsion
sheaves on~$\bY$.
 Using Proposition~\ref{relatively-homotopy-flat-resolution}, choose
a morphism of complexes of $\X$\+flat pro-quasi-coherent pro-sheaves
$\bfF^\bu\rarrow\bfP^\bu$ whose cone is acyclic in $\bY_\X\flat$,
while $\bfF^\bu\in\sC(\bY\flat)$ is a $\pi$\+relatively homotopy flat
complex of flat pro-quasi-coherent pro-sheaves on~$\bY$.
 Then the same morphism can be also viewed as a morphism of complexes
of $\X'$\+flat pro-quasi-coherent pro-sheaves on $\bY$, its cone is
also acyclic in $\bY_{\X'}\flat$, and the complex $\bfF^\bu$
is also $\pi'$\+relatively homotopy flat (as pointed out in
the proof of
Proposition~\ref{weakly-smooth-pro-sheaves-tensor-structure-preserved}).

 Using Proposition~\ref{homotopy-Y/X-flat-resolution}, choose
a morphism of complexes of quasi-coherent torsion sheaves $\brG^\bu
\rarrow\bN^\bu$ whose cone has the property that its direct image
is coacyclic in $\X\tors$, while $\brG^\bu\in\sC(\bY\tors)$ is
a homotopy $\bY/\X$\+flat complex of quasi-coherent torsion sheaves
on~$\bY$.
 Then the direct image of the cone is also coacyclic in $\X'\tors$,
and $\brG^\bu$ is also a homotopy $\bY/\X'$\+flat complex.
 Thus the tensor product $\bfF^\bu\ot_\bY\brG^\bu$ computes both
the derived functors $\bfP^\bu\ot_{\bY/\X}^\boL\bN^\bu$ and
$\bfP^\bu\ot_{\bY/\X'}^\boL\bN^\bu$.
\end{proof}

\begin{thm} \label{weakly-smooth-semitensor-product-preservation}
 Let\/ $\X$ be an ind-semi-separated ind-Noetherian ind-scheme,
$\tau\:\X'\rarrow\X$ be an affine morphism of finite type which is
weakly smooth of relative dimension~$\le d$, and\/
$\pi'\:\bY\rarrow\X'$ be a flat affine morphism of ind-schemes.
 Put\/ $\pi=\tau\pi'$.
 Let\/ $\rD^\bu$ be a dualizing complex on\/ $\X$ and\/
$\rD'{}^\bu$ be the related dualizing complex on\/ $\X'$, as per
the rule of Lemma~\ref{weakly-smooth-pullback-of-dualizing-torsion}.
 Then the triangulated equivalence\/
$\sD_{\X'}^\si(\bY\tors)\simeq\sD_\X^\si(\bY\tors)$ from
Corollary~\ref{weakly-smooth-semiderived-preservation-cor}
is an equivalence of tensor triangulated categories with
the semitensor product operations
$$
 \os_{\pi'{}^*\rD'{}^\bu}\:\sD_{\X'}^\si(\bY\tors)\times
 \sD_{\X'}^\si(\bY\tors)\lrarrow\sD_{\X'}^\si(\bY\tors)
$$
and
$$
 \os_{\pi^*\rD^\bu}\:\sD_\X^\si(\bY\tors)\times
 \sD_\X^\si(\bY\tors)\lrarrow\sD_\X^\si(\bY\tors).
$$
as in formula~\eqref{semiderived-torsion-semitensor-product}
from Section~\ref{construction-of-semitensor-subsecn}.
\end{thm}

\begin{proof}
 Follows from
Proposition~\ref{weakly-smooth-semiderived-equivalences-preserved}
together with
Proposition~\ref{weakly-smooth-pro-sheaves-tensor-structure-preserved}
or~\ref{weakly-smooth-torsion-module-structure-preserved}.
\end{proof}

\Section{Some Infinite-Dimensional Geometric Examples}
\label{geometric-examples-secn}

 In this section we discuss several examples illustrating the nature
of infinite-dimensional algebro-geometric objects for which
the constructions and results of Sections~\ref{X-flat-on-Y-secn}\+-%
\ref{weakly-smooth-postcomposition-secn} are designed.
 All the examples below in this section will be those of flat affine
morphisms of ind-schemes $\pi\:\bY\rarrow\X$, where $\X$ is
an ind-separated ind-scheme of ind-finite type over a field~$\kk$.

\subsection{The Tate affine space example} \label{tate-space-subsecn}
 The following example, while geometrically very simple (a trivial
bundle), has unusual and attractive invariance properties.
 The idea (following~\cite[Example~7.11.2(ii)]{BD2}
or~\cite[Example~1.3(1)]{Rich}) is to consider vector spaces as a kind
of affine schemes or, as it may happen, ind-affine ind-schemes.

\smallskip
 (0)~Let $X$ be a scheme over a field~$\kk$.
 Then, in the algebro-geometric parlance, by a ``$\kk$\+point on~$X$''
one means a morphism $\Spec\kk\rarrow X$ of schemes over~$\kk$ (i.~e.,
a section of the structure morphism $X\rarrow\Spec\kk$).
 More generally, given a commutative $\kk$\+algebra $K$, by
a ``$K$\+point on~$X$'' one means a morphism $\Spec K\rarrow X$
of schemes over~$\kk$ (i.~e., a morphism forming a commutative
triangle diagram with the structure morphisms $\Spec K\rarrow\Spec\kk$
and $X\rarrow\Spec\kk$).

 Similarly one can speak of $\kk$\+points or $K$\+points on an ind-scheme
$\X$ over~$\kk$ (considering morphisms of ind-schemes in lieu of
the morphisms of schemes).
 An ind-scheme $\X$ over~$\kk$ is determined by its ``functor of
points'', assigning to a commutative $\kk$\+algebra $K$ the set $\X(K)$
of all $K$\+points on~$\X$ and to every homomorphism of commutative
$\kk$\+algebras $K'\rarrow K''$ the induced map $\X(K')\rarrow\X(K'')$
(cf.\ the discussion in the first paragraph of
Section~\ref{ind-schemes-subsecn}).

\smallskip
 (1)~Let $V$ be a finite-dimensional $\kk$\+vector space.
 We would like to assign a $\kk$\+scheme $X_V$ to $V$ in such a way
that $\kk$\+points on $X$ would correspond bijectively to the elements
of~$V$.
 More generally, for a commutative $\kk$\+algebra $K$, the $K$\+points
on $X$ will correspond bijectively to elements of the tensor
product $K\ot_\kk V$.

 Here is the (obvious) construction.
 Consider the dual vector space $V^*=\Hom_\kk(V,k)$, and consider
the symmetric algebra $\Sym_\kk(V^*)$.
 By the definition, $\Sym_\kk(V^*)$ is the commutative $\kk$\+algebra
freely generated by the $\kk$\+vector space~$V^*$.
 Put $X_V=\Spec\Sym_\kk(V^*)$.

\smallskip
 (2)~Let $V$ be an infinite-dimensional $\kk$\+vector space.
 Then there is \emph{no} natural construction of a $\kk$\+scheme whose
$\kk$\+points would correspond to the elements of~$V$.
 More precisely, the functor assigning to an affine scheme $\Spec K$
over~$\kk$ the set of all elements of $K\ot_\kk V$ is \emph{not}
representable by a scheme over~$K$.
 However, it is representable by an ind-affine ind-scheme~$\X_V$
of ind-finite type over~$\kk$.

 Denote by $\Gamma$ the directed poset of all finite-dimensional
vector subspaces $U\subset V$, ordered by inclusion.
 Then the ind-scheme $\X_V$ is given by the $\Gamma$\+indexed
inductive system of the affine schemes $X_U$ constructed in~(1),
that is $\X_V=\ilim_{U\in\Gamma}X_U$.

\smallskip
 (3)~Let $V$ be a \emph{linearly compact} $\kk$\+vector space, i.~e.,
a complete, separated topological $\kk$\+vector space in which open
vector subspaces of finite codimension form a base of neighborhoods
of zero.
 Denote by $V^*$ the vector space of all continuous linear maps
$V\rarrow\kk$.
 Then $V^*$ is an (infinite-dimensional) discrete vector space;
the vector space $V$ can be recovered as the dual vector space
$V=\Hom_\kk(V^*,k)$ to the discrete vector space $V^*$, with
the natural topology on such dual space.

 Consider the symmetric algebra $\Sym_\kk(V^*)$ (which can be defined
as the commutative $\kk$\+algebra freely generated by $V^*$, or as
the direct limit of the symmetric algebras of finite-dimensional
subspaces of~$V^*$).
 Put $Y_V=\Spec\Sym_\kk(V^*)$.
 This is the infinite-dimensional affine scheme corresponding to
a linearly compact $\kk$\+vector space.
 For any commutative $\kk$\+algebra $K$, the set of $K$\+points on $Y_V$
is naturally bijective to the set of all $\kk$\+linear maps
$V^*\rarrow K$.

\smallskip
 (4)~Let $V$ be a \emph{locally linearly compact} (or \emph{Tate})
$\kk$\+vector space, i.~e., a (complete, separated) topological vector
space admitting a linearly compact open subspace.
 Then the corresponding ind-affine ind-scheme $\Y_V$ (\emph{not}
of ind-finite type) over~$\kk$ is constructed as follows.

 Denote by $\Gamma$ the directed poset of all linearly compact open
subspaces $U\subset V$, ordered by inclusion.
 Then the ind-scheme $\Y_V$ is defined as $\Y_V=\ilim_{U\in\Gamma}Y_U$.
 Here, for any pair of linearly compact open subspaces $U'\subset U''
\subset V$, the inclusion map $U'\rarrow U''$ induces a surjective
map of the dual discrete vector spaces $U''{}^*\rarrow U'{}^*$.
 This linear map, in turn induces a surjective homomorphism of
commutative $\kk$\+algebras $\Sym_\kk(U''{}^*)\rarrow\Sym_\kk(U'{}^*)$,
which corresponds to the natural closed immersion of affine schemes
$Y_{U'}\rarrow Y_{U''}$ appearing in the inductive system.

\smallskip
 (5)~Now let $V$ be a locally linearly compact $\kk$\+vector space and
$W\subset V$ be a fixed linearly compact open subspace.
 Then the quotient space $V/W$ is discrete in the induced topology.
 Put $\X=\X_{V/W}$ and $\bY=\Y_V$, in the notation of~(2) and~(4).

 Then $\X$ is an ind-affine ind-scheme of ind-finite type over~$\kk$;
hence, in particular, $\X$ is ind-semi-separated and ind-Noetherian.
 The surjective linear map $V\rarrow V/W$ induces a flat affine
morphism of ind-schemes $\pi\:\bY\rarrow\X$.
 The fibers of~$\pi$ over the $\kk$\+points of $\X$ are
infinite-dimensional affine schemes isomorphic to~$Y_W$.

\smallskip
 (6)~Finally, let us construct a dualizing complex on
the ind-Noetherian ind-scheme $\X=\X_{V/W}$.
 Let $U\subset V/W$ be a finite-dimensional vector subspace.
 Then $\Sym_\kk(U)$ is a regular commutative ring of Krull dimension
$\dim_\kk U$ (in fact, the ring of polynomials in $\dim_\kk U$ variables
over~$\kk$).
 Choose a (finite, if one wishes) injective resolution $E_U^\bu$
of the free $\Sym_\kk(U)$\+module $\Sym_\kk(U)$, and let $\E_U^\bu$
be the corresponding complex of injective quasi-coherent sheaves
on~$X_U$.
 So $\E_U^\bu$ is an injective resolution of the structure sheaf
$\cO_{X_U}$ in the category $X_U\qcoh$.

 Put $\D_U^\bu=\E_U^\bu\ot_\kk\Lambda_\kk^{\dim_\kk U}(U^*)
[\dim_\kk U]$.
 Here $\Lambda_\kk^{\dim_\kk U}(U^*)[\dim_\kk U]$ is a complex of
$\kk$\+vector spaces whose only term is the one-dimensional top exterior
power of the vector space $U^*$, placed in the cohomological
degree~$-\dim_\kk U$.
 Then $\D_U^\bu$ is a dualizing complex on~$X_U$.
 Moreover, for any pair of finite-dimensional subspaces $U'\subset U''
\subset V/W$ and the related closed immersion of affine schemes
$i_{U'U''}\:X_{U'}\rarrow X_{U''}$, there is a natural homotopy
equivalence $\D_{U'}^\bu\simeq i_{U'U''}^!\D_{U''}^\bu$ of complexes
of injective quasi-coherent sheaves on $X_{U'}$.
 
 It remains to glue the system of dualizing complexes $\D_U^\bu$ on
the schemes $X_U$ into a dualizing complex (of injective quasi-coherent
torsion sheaves) $\rD^\bu$ on the ind-scheme~$\X_{V/W}$.
 For this purpose, we assume that $V/W$ is a vector space of at most
countable dimension over~$\kk$ (so $\X=\X_{V/W}$ and $\bY=\Y_V$ are
$\aleph_0$\+ind-schemes).
 Then the construction of Example~\ref{aleph-zero-dualizing-glueing}
does the job.

\smallskip
 (7)~Notice that, in the context of~(6), the complex $\D_U^\bu$ has
its only cohomology sheaf situated in the negative cohomological
degree~$-\dim_\kk U$.
 Assuming that $V/W$ is countably infinite-dimensional vector space
and looking into the construction of 
Example~\ref{aleph-zero-dualizing-glueing} keeping exactness of
the functors of direct image (under a closed immersion) and
direct limit in mind, one can see that $\rD^\bu$ is an \emph{acyclic}
complex of quasi-coherent torsion sheaves on~$\X_{V/W}$.
 The only cohomology sheaf just runs to the cohomological
degree~$-\infty$ and disappears in the direct limit as $\dim_\kk U$
grows to infinity.
 Hence $\pi^*\rD^\bu$ is an acyclic complex of quasi-coherent torsion
sheaves on~$\Y_V$.
 These acyclic complexes, representing quite nontrivial objects in
the coderived and semiderived categories, are the unit objects of
the respective tensor structures given by the cotensor and semitensor
product operations!

\smallskip
 (8)~At last, let us discuss the invariance properties of
the constructions above with respect to replacing a linearly compact
open subspace $W\subset V$ with another linearly compact open subspace.

 Let $W'$ and $W''\subset V$ be two linearly compact open subspaces.
 Then $W=W'+W''\subset V$ is a linearly compact open subspace as well.
 The quotient spaces $W/W'$ and $W/W''$ are finite-dimensional, so
the morphisms of ind-schemes of ind-finite type $\X'=\X_{V/W'}\rarrow
\X_{V/W}=\X$ and $\X''=\X_{V/W''}\rarrow\X_{V/W}=\X$ induced by
the surjective linear maps of discrete vector spaces $V/W'\rarrow V/W$
and $V/W''\rarrow V/W$ are affine and weakly smooth (in fact, smooth)
of finite relative dimension.
 Hence Corollary~\ref{weakly-smooth-semiderived-preservation-cor} tells
that the class of morphisms in $\sC(\bY\tors)$ or $\sK(\bY\tors)$
which are inverted in order to construct the semiderived category
$\sD_{\X'}^\si(\bY\tors)$ coincides with the class of morphisms of
complexes which are inverted in order to construct the semiderived
category $\sD_{\X''}^\si(\bY\tors)$.
 In other words, the semiderived category of quasi-coherent torsion
sheaves on the ind-scheme $\bY=\Y_V$ is determined by the locally
linearly compact topological vector space $V$, and \emph{does not
depend on the choice of a linearly compact open subspace $W\subset V$}.

\smallskip
 (9)~The invariance property of the semitensor product
operation~$\os_{\pi^*\rD^\bu}$ on $\sD_\X^\si(\bY\tors)$ is only
slightly more complicated.
 The \emph{relative dimension} $\dim W'/W''\in\boZ$ is defined as
the difference $\dim_\kk\widetilde W/W''-\dim_\kk\widetilde W/W'=
\dim_\kk W'/\overline W-\dim_\kk W''/\overline W$, where $\widetilde W$
and $\overline W\subset V$ are arbitrary linearly compact open
subspaces such that $\overline W\subset W'\cap W''$ and
$W'+W''\subset\widetilde W$.
 The \emph{relative determinant} $\det W'/W''\in\kk\vect$ is
a one-dimensional $\kk$\+vector space defined as
$\det V'/W''=\det_\kk(\widetilde W/W'')\ot_\kk
\det_\kk(\widetilde W/W')^*=\det_\kk(W'/\overline W)\ot_\kk
(\det_\kk W''/\overline W)^*$, where
$\det_\kk(U)=\Lambda_\kk^{\dim_\kk U}(U)$ for a finite-dimensional
$\kk$\+vector space~$U$.
 Here the equality signs mean natural isomorphisms.

 Then the functor $(\det W'/W'')[\dim W'/W'']\ot_\kk\Id\:
\sD_{\X'}^\si(\bY\tors)\rarrow\sD_{\X''}^\si(\bY\tors)$ is
a tensor triangulated equivalence between the two tensor triangulated
categories, with the tensor structures given by the constructions above.
 In particular, this functor takes the unit object to the unit object.
 Here $(\det W'/W'')[\dim W'/W'']$ is a complex of $\kk$\+vector spaces
in which the one-dimensional vector space $\det W'/W''$ sits in
the cohomological degree~$-\dim W'/W''$ (and the components in
all the other cohomological degrees vanish).
 This is the conclusion one obtains from
Theorem~\ref{weakly-smooth-semitensor-product-preservation}.

\begin{ex} \label{laurent-series-example}
 Let us spell out in coordinate notation an important particular case
of the above example.
 Let $\kk((t))$ be the $\kk$\+vector space of formal Laurent power
series in a variable~$t$ over the field~$\kk$, endowed with the usual
topology in which the vector subspaces $t^n\kk[[t]]\subset\kk((t))$,
\,$n\in\boZ$, form a base of neighborhoods of zero.
 Then $V=\kk((t))$ is a locally linearly compact topological vector
space over~$\kk$.

 A generic element of $\kk((t))$ has the form $f(t)=\sum_{n=-N}^\infty
x_nt^n$, where $N\in\boZ$ and $x_n\in\kk$ for all $n\in\boZ$.
 So $x_n\:V=\kk((t))\rarrow\kk$ are continuous linear functions; we will
consider them as coordinates on the ind-scheme~$\Y_V$.

 Let $W\subset V$ be the linearly compact open subspace $W=\kk[[t]]
\subset\kk((t))$.
 Then the $\aleph_0$\+ind-scheme $\X=\X_{V/W}$ can be described, in
terms of the anti-equivalence of categories from
Example~\ref{topological-ring-ind-scheme}(2), as
$\X_{V/W}=\Spi\fA$, where $\fA$ is the topological commutative ring
$\fA=\varprojlim_{n>0}\kk[x_{-n},\dotsc,x_{-1}]$ with the topology
of projective limit of the discrete polynomial rings
$A_n=\kk[x_{-n},\dotsc,x_{-1}]$.
 Here the transition map $A_{n+1}\rarrow A_n$ in the projective
system takes $x_{-n-1}$ to~$0$ and $x_{-i}$ to~$x_{-i}$ for $i\le n$.

 The $\aleph_0$\+ind-scheme $\bY=\Y_V$ can be similarly described, in
terms of the same anti-equivalence of categories, as
$\Y_V=\Spi\fR$, where $\fR$ is the topological commutative ring
$\fR=\varprojlim_{n>0}\kk[x_{-n},\dotsc,x_{-1},x_0,x_1,\dotsc]$ with
the topology of projective limit of the discrete ring of polynomials
in infinitely many variables
$R_n=\kk[x_{-n},\dotsc,x_{-1},x_0,x_1,\dotsc]$.
 The transition map $R_{n+1}\rarrow R_n$ in the projective system
of rings takes $x_{-n-1}$ to~$0$ and $x_i$ to~$x_i$ for $i\ge-n$.
 The flat affine morphism of ind-schemes $\pi\:\Y_V\rarrow\X_{V/W}$
corresponds to the natural injective continuous ring homomorphism
$\fA\rarrow\fR$, which can be obtained as the projective limit
of the natural subring inclusions $A_n\rarrow R_n$.

 Following the discussion in
Section~\ref{torsion-ind-affine-subsecn}(6), the Grothendieck
abelian category $\X\tors$ of quasi-coherent torsion sheaves on $\X$
is equivalent to the category $\fA\discr$ of discrete $\fA$\+modules.
 The abelian category $\fA\discr$ can be simply described in explicit
terms as the category of modules $\rM$ over the ring of polynomials in
infinitely many variables $A=\kk[\dotsc,x_{-3},x_{-2},x_{-1}]$ having
the property that for every $b\in\rM$ there exists $n>0$ such that
$x_{-i}b=0$ for all $i>n$.
 The Grothendieck abelian category $\bY\tors$ of quasi-coherent
torsion sheaves on $\bY$ is equivalent to the category $\fR\discr$
of discrete $\fR$\+modules, which explicitly means modules $\bN$ over
the ring of polynomials in doubly infinitely many variables
$R=\kk[\dotsc,x_{-2},x_{-1},x_0,x_1,\dotsc]$ with the property that for
every $b\in\bN$ there exists $n>0$ such that $x_{-i}b=0$ for all $i>n$
(while no condition is imposed on $x_ib$ for $i\ge0$).

 So the obvious forgetful functor $R\modl\rarrow A\modl$ induced
by the subring inclusion $A\rarrow R$ takes discrete $\fR$\+modules
to discrete $\fA$\+modules.
 The forgetful functor $\fR\discr\rarrow\fA\discr$ corresponds, under
the above equivalences between the sheaf and module categories, to
the direct image functor $\pi_*\:\bY\tors\rarrow\X\tors$, in terms of
which the semiderived category $\sD_\X^\si(\bY\tors)$ is defined.

 Finally, the dualizing complex $\rD^\bu$ on $\X$ as per
the construction in~(6) is a complex of injective discrete modules
over $\fA$ (concentrated, if one wishes, in the nonpositive
cohomological degrees; cf.\
Remarks~\ref{cotensor-right-derived-remark}(3\+-5))
with the following property.
 For every $n>0$, the subcomplex $\D_n^\bu$ of all elements annihilated
by \dots,~$x_{-n-3}$, $x_{-n-2}$,~$x_{-n-1}$ in $\rD^\bu$ is a complex
of injective $A_n$\+modules homotopy equivalent to an injective
resolution of the free $A_n$\+module with one generator
$\kk[x_{-n},\dotsc,x_{-1}]\,dx_{-n}\wedge\dotsb\wedge dx_{-1}$
shifted cohomologically by~$[n]$.

 As explained in~(8\+-9), the results of
Section~\ref{weakly-smooth-postcomposition-secn} imply that
the semiderived category $\sD_\X^\si(\bY\tors)$ is preserved by
all the continuous \emph{linear} coordinate changes in the topological
vector space $V=\kk((t))$.
 This means all the bijective continuous $\kk$\+linear maps $\kk((t))
\rarrow\kk((t))$ with continuous inverse maps.
 Moreover, the semitensor product operation on $\sD_\X^\si(\bY\tors)$
is also preserved by such coordinate changes, up to dimensional
cohomological shifts and determinantal twists.
\end{ex}

\begin{qst}
 In the context of Example~\ref{laurent-series-example}, one can
write $\sD_\X^\si(\bY\tors)=\sD_\fA^\si(\fR\discr)$, referring
to the equivalences of abelian categories $\bY\tors\simeq\fR\discr$
and $\X\tors\simeq\fA\discr$.
 Is the semiderived category $\sD_\fA^\si(\fR\discr)$ preserved by
arbitrary continuous \emph{polynomial} coordinate changes, i.~e., all
the automorphisms of the topological ring~$\fR$\,?
\end{qst}

\subsection{Cotangent bundle to discrete projective space}
\label{cotangent-bundle-subsecn}
 Let $V$ be an infinite-dimensional discrete $\kk$\+vector space.
 Then the \emph{projectivization} of $V$ is an ind-scheme $\fP(V)$
of ind-finite type over~$\kk$, defined informally as the space of
all one-dimensional vector subspaces in~$V$
(cf.~\cite[Example~1.3(2)]{Rich}).
 Any one-dimensional vector subspace $L\in V$ corresponds to
a $\kk$\+point $l\:\Spec\kk\rarrow\fP(V)$.

 The tangent space to $\fP(V)$ at the point~$l$ can be computed as
$T_l\fP(V)=\Hom_\kk(L,V/L)$; so it is a discrete $\kk$\+vector space
(as one would expect).
 Accordingly, the cotangent space $T_l^*\fP(V)=\Hom_\kk(V/L,L)$ is
a linearly compact $\kk$\+vector space.
 Following Section~\ref{tate-space-subsecn}(3), there is
an infinite-dimensional affine scheme corresponding to $T_l^*\fP(V)$,
described as $Y_{T_l^*\fP(V)}=\Spec\Sym_\kk(T_l\fP(V))$.
 So, denoting by $\bY$ the total space of the cotangent bundle to
$\X=\fP(V)$, one would expect the fibration $\pi\:\bY\rarrow\X$ to be
a flat affine morphism of ind-schemes.
 In order to show that this is indeed the case, let us explain how
to formalize this informal discussion.

\smallskip
 (1)~Let $U$ be a finite-dimensional $\kk$\+vector space.
 Then the projectivization $P(U)$ is defined as the projective
spectrum of the graded ring $\Sym_\kk(U^*)$, i.~e.,
$P(U)=\Proj\Sym_\kk(U^*)$ (where the grading on the commutative
$\kk$\+algebra $\Sym_\kk(U^*)$ is defined by the rule that the elements
of $U^*\subset\Sym_\kk(U^*)$ have degree~$1$).
 This means that the scheme points of $P(U)$ correspond bijectively
to homogeneous prime ideals in the graded ring $R=\Sym_\kk(U^*)$
not containing (in other words, different from) the ideal
$\bigoplus_{n>0}R_n\subset R$ of all elements of positive degree.

 For every element $r\in R_1$, the subset in $\Proj R$ consisting
of all homogeneous prime ideals not containing~$r$ is an affine
open subscheme in $\Proj R$ naturally isomorphic to
$\Spec R[r^{-1}]_0$, where $R[r^{-1}]$ is the $\boZ$\+graded ring
obtained by inverting the element $r\in R_1$ and $R[r^{-1}]_0
\subset R[r^{-1}]$ is the subring of all elements of degree~$0$.

 Let $U''$ be a finite-dimensional $\kk$\+vector space and
$U'\subset U''$ be a vector subspace.
 Then the surjective $\kk$\+linear map $U''{}^*\rarrow U'{}^*$
induces a surjective homomorphism of graded rings
$R''=\Sym_\kk(U''{}^*)\rarrow\Sym_\kk(U'{}^*)=R'$.
 Hence the induced closed immersion of projective spectra
$P(U')=\Proj R'\rarrow\Proj R''=P(U'')$.
 
\smallskip
 (2)~Let $M$ be a graded module over the graded ring $R=\Sym_\kk(U^*)$.
 Then a quasi-coherent sheaf $\widetilde M$ over $P(U)=\Proj R$ is
assigned to $M$ in the following way.
 For every element $r\in R_1$, the restriction of $\widetilde M$ to
the affine open subscheme $\Spec R[r^{-1}]_0\subset\Proj R$
is the quasi-coherent sheaf over $\Spec R[r^{-1}]_0$ corresponding to
the $R[r^{-1}]_0$\+module $M[r^{-1}]_0$.
 Here $M[r^{-1}]=R[r^{-1}]\ot_RM$ and $M[r^{-1}]_0$ is the degree~$0$
component of the $\boZ$\+graded module $M[r^{-1}]$.

 The functor $M\longmapsto\widetilde M$ from the abelian category
of graded $R$\+modules to the abelian category of quasi-coherent
sheaves on $\Proj R$ is exact and preserves coproducts; it is also
a tensor functor between the two tensor categories.
 In particular, the functor $M\longmapsto\widetilde M$ takes flat
graded $R$\+modules $F$ to flat quasi-coherent sheaves
$\widetilde F$ on $\Proj R$.
 One has $\widetilde M=0$ if and only if, for every $b\in M$,
there exists $n\ge1$ such that $R_ib=0$ in $M$ for all $i\ge n$.
 The functor $M\longmapsto\widetilde M$ also takes finitely generated
graded $R$\+modules to coherent sheaves on $P(U)=\Proj R$.

 For any graded $R$\+module $M$ and any integer $n\in\boZ$, denote
by $M(n)$ the graded $R$\+module with the components $M(n)_i=M_{n+i}$
(and the same action of $R$ as in~$M$).
 The \emph{tautological line bundle} $\cO_{P(V)}(-1)$ on $P(V)$ is
defined informally by the rule that the line $L$ is the fiber of
$\cO_{P(V)}(-1)$ over a $\kk$\+point $l\:\Spec\kk\rarrow P(V)$
corresponding to a one-dimensional $\kk$\+vector subspace $L\subset V$.
 More formally, the quasi-coherent sheaf $\cO_{P(V)}(n)$ on $P(V)$
(for any $n\in\boZ$) corresponds to the graded $R$\+module $R(n)$.
 Accordingly, one has $\widetilde{M(n)}=\widetilde M(n)$ for any
graded $R$\+module $M$, where $\M(n)=\cO_{P(V)}(n)\ot_{\cO_{P(V)}}\M$
for any quasi-coherent sheaf $\M$ on~$P(V)$.

 Let $U'\subset U''$ be a vector subspace in a finite-dimensional
$\kk$\+vector space, and let $R''\rarrow R'$ be the related surjective
morphism of graded rings, as in~(1).
 Denote by $i_{U'U''}\:P(U')\rarrow P(U'')$ the related closed
immersion of projective spaces over~$\kk$.
 Let $M''$ be a graded $R''$\+module and $\widetilde{M''}$ be
the related quasi-coherent sheaf on~$P(U'')$.
 Then the quasi-coherent sheaf $i_{U'U''}^*\widetilde{M''}$ on
$P(U')$ corresponds to the graded $R'$\+module $R'\ot_{R''}M''$.

\smallskip
 (3)~Let $V$ be a discrete $\kk$\+vector space.
 Denote by $\Gamma$ the directed poset of all finite-dimensional
vector subspaces $U\subset V$, ordered by inclusion.
 The \emph{projectivization} of $V$ is defined as the ind-scheme
$\fP(V)=\ilim_{U\in\Gamma}P(U)$, where the transition maps
$i_{U'U''}\:P(U')\rarrow P(U'')$ are the ones defined above in~(1\+-2).

 Our aim is to construct a flat pro-quasi-coherent pro-sheaf
$\fT$ on $\fP(V)$ corresponding to the tangent bundle.
 So, for any $\kk$\+point $l\:\Spec\kk\rarrow\fP(V)$ and the related
one-dimensional vector subspace $L\subset V$, the fiber of $\fT$
over~$l$ should be the discrete $\kk$\+vector space
$\Hom_\kk(L,V/L)\simeq (L^*\ot_\kk V)/\kk$.

 Let $U\subset V$ be a finite-dimensional vector subspace.
 We start with constructing the restriction of $\fT$ onto the closed
subscheme $P(U)\subset\fP(V)$.
 Denoting by $i_U\:P(U)\rarrow\fP(V)$ the closed immersion morphism,
we would like to construct the quasi-coherent sheaf $\cT_U=i_U^*\fT$
on the scheme $P(U)$.

 The infinite-dimensional vector bundle $V(1)$ on $P(U)$ corresponds
to the graded module $V\ot_\kk R(1)$ over the graded ring
$R=\Sym_\kk(U^*)$.
 The degree~$-1$ component of this graded module is the vector space
$V$, and the degree~$0$ component of the vector space $V\ot_\kk U^*$.
 The latter vector space contains a canonical element
$e\in V\ot_\kk U^*\simeq\Hom_\kk(U,V)$ corresponding to the identity
map $U\rarrow V$.

 Let $f_U\:R\rarrow V\ot_\kk R(1)$ be the graded $R$\+module map
taking the free generator $1\in R$ to the element $e\in V\ot_\kk U^*$.
 Denote by $T_U=\coker(f_U)$ the cokernel of~$f_U$ taken in the category
of graded $R$\+modules.
 By definition, the quasi-coherent sheaf $\cT_U$ on $P(U)=\Proj R$
corresponds to the graded $R$\+module~$T_U$.

\smallskip
 (4)~The graded $R$\+module $T_U$ is \emph{not} flat, but it is
``flat up to torsion $R$\+modules'', in a suitable sense;
so the quasi-coherent sheaf $\cT_U$ on $P(U)$ is flat.
 More precisely, the map $f_U\:R\rarrow V\ot_\kk R(1)$ factorizes as
the composition $R\rarrow U\ot_\kk R(1)\rarrow V\ot_\kk R(1)$, where
the split monomorphism of free graded $R$\+modules
$U\ot_\kk R(1)\rarrow V\ot_\kk R(1)$ is induced by the inclusion of
$\kk$\+vector spaces $U\rarrow V$.
 The cokernel of the morphism $R\rarrow U\ot_\kk R(1)$ is
finitely generated graded $R$\+module; the corresponding coherent
sheaf on $P(U)$ is the locally free sheaf corresponding to
the tangent bundle to $P(U)$.

 To see algebraically that the cokernel of the graded $R$\+module
morphism $R\rarrow U\ot_\kk R(1)$ corresponds to a locally free
coherent (or at least, a flat quasi-coherent) sheaf on $P(U)$, one
can consider the Koszul complex
$$
 0\rarrow R\lrarrow U\ot_\kk R(1)\rarrow
 \Lambda^2_\kk(U)\ot_\kk R(2)\rarrow\dotsb\rarrow
 \Lambda^{\dim_\kk U}(U)\ot_\kk R(\dim_\kk U),
$$
where $\Lambda^n_\kk(U)$ are the exterior powers of
the $\kk$\+vector space~$U$.
 This complex is a free graded resolution of the graded $R$\+module
$\Lambda^{\dim_\kk U}(U)(\dim_\kk U)$, which is a one-dimensional
$\kk$\+vector space viewed as a graded $R$\+module concentrated in
the single degree~$-\dim_\kk U$.
 The exact functor $M\longmapsto\widetilde M$ annihilates this graded
$R$\+module, so this functor takes the Koszul complex to an exact finite
complex of locally free coherent sheaves on~$P(U)$.
 Accordingly, all the graded $R$\+modules of cycles and boundaries of
the Koszul complex are also taken to locally free coherent sheaves
on $P(U)$ by the functor $M\longmapsto\widetilde M$.

\smallskip
 (5)~Let $U'\subset U''\subset V$ be two finite-dimensional
subspaces in our discrete $\kk$\+vector space~$V$.
 Then the functor $R'\ot_{R''}{-}$ takes the graded $R''$\+module
morphism~$f_{U''}$ to the graded $R'$\+module morphism~$f_{U'}$.
 Hence we have a natural isomorphism of graded $R'$\+modules
$T_{U'}\simeq R'\ot_{R''}T_{U''}$, and consequently a natural
isomorphism $\cT_{U'}\simeq i_{U'U''}^*\cT_{U''}$ of flat
quasi-coherent sheaves on~$P(U')$.

 Now the rule $\fT^{(P(U))}=\cT_U$ for all $U\in\Gamma$ defines
the desired flat pro-quasi-coherent pro-sheaf $\fT$ on
the ind-scheme~$\fP(V)$.

\smallskip
 (6)~Similarly to the symmetric algebra of a vector space, one can
define the symmetric algebra of a module over a commutative ring~$S$.
 Given an $S$\+module $N$, the symmetric algebra $\Sym_S(N)$ can
be defined as the commutative $S$\+algebra freely generated by
the $S$\+module~$N$.
 When $F$ is a free $S$\+module, $\Sym_S(F)$ is a free $S$\+module,
too; and the (nonadditive) functor $N\longmapsto\Sym_S(N)$ preserves
direct limits; so when $F$ is a flat $S$\+module, $\Sym_S(F)$ is
a flat $S$\+algebra.

 Furthermore, given a quasi-coherent sheaf $\N$ on a scheme $Z$,
one defines the quasi-coherent commutative algebra $\cSym_Z(\N)$ on $Z$
(in the sense of Section~\ref{pro-qcoh-algebras-subsecn}) by
the rule $\cSym_Z(\N)(W)=\Sym_{\cO(W)}(\N(W))$ for all the affine
open subschemes $W\subset Z$.
 Clearly, $\cSym_Z(\F)$ is a flat quasi-coherent commutative algebra
on $Z$ whenever $\F$ is a flat quasi-coherent sheaf on~$Z$.
 For every morphism of schemes $f\:Z'\rarrow Z''$ and a quasi-coherent
sheaf $\N''$ on $Z''$, one has a natural isomorphism
$\cSym_{Z'}(f^*\N'')\simeq f^*\cSym_{Z''}(\N'')$ of quasi-coherent
algebras on~$Z'$.

\smallskip
 (7)~In the context of~(3\+-5), we put $\cA_U=\cSym_{P(U)}\cT_U$
for every finite-dimensional vector subspace $U$ in the given
discrete $\kk$\+vector space~$V$.
 So $\cA_U$ is a flat quasi-coherent algebra on the projective
space~$P(U)$.
 Then the rule $\fA^{(P(U))}=\cA_U$ defines a flat pro-quasi-coherent 
commutative algebra $\fA$ on the ind-scheme $\X=\fP(V)$.
 The ind-scheme $\bY$ together with the flat affine morphism
of ind-schemes $\pi\:\bY\rarrow\X$ corresponds to
the pro-quasi-coherent algebra $\fA$ on $\X$ via the construction of
Proposition~\ref{ind-schemes-affine-morphisms-pro-qcoh-algebras}.
 This is the desired flat affine morphism of ind-schemes corresponding
to the cotangent bundle on the ind-projective space~$\fP(V)$.

\smallskip
 (8)~Assuming that the dimension of~$V$ is at most countable, one
can construct a dualizing complex $\rD^\bu$ on $\X=\fP(V)$ following
the approach of Remarks~\ref{cotensor-right-derived-remark}(3\+-5).
 Notice that, similarly to Section~\ref{tate-space-subsecn}(7),
the dualizing complex $\rD^\bu$ is an \emph{acyclic} complex of
quasi-coherent torsion sheaves on $\X$ whenever $V$ is a $\kk$\+vector
space of (countably) infinite dimension.

\subsection{Universal fibration of quadratic cones in linearly
compact vector space}  \label{quadratic-cones-subsecn}
 Let $W$ be an infinite-dimensional discrete $\kk$\+vector space and
$W^*$ be the dual linearly compact $\kk$\+vector space.
 Then the elements of the vector space $\Sym_\kk^2(W)$ correspond to 
continuous quadratic functions $q\:W^*\rarrow\kk$.
 It is worth noticing that any such quadratic function actually
factorizes through a finite-dimensional discrete quotient vector
space of~$W^*$.

 The zero locus $Y_q$ of any nonzero continuous quadratic function
$q\:W^*\rarrow\kk$ is an infinite dimensional affine scheme, or more
specifically a closed subscheme in the affine scheme $Y_{W^*}$
corresponding to the linearly compact topological vector space $W^*$
under the construction of Section~\ref{tate-space-subsecn}(3).
 This closed subscheme is an infinite-dimensional quadratic cone.
 Such quadratic cones $Y_q\subset Y_{W^*}$ are parametrized by nonzero
continuous quadratic functions $q\:W^*\rarrow\kk$ viewed up to
a multiplication by a scalar from~$\kk$.
 In other words, the space of parameters of the quadratic cones
in $W^*$ is the ind-scheme $\fP(V)$ from
Section~\ref{cotangent-bundle-subsecn}(3),
where the infinite-dimensional discrete $\kk$\+vector space $V$ is
constructed as $V=\Sym^2_\kk(W)$.

 This informal discussion suggests that there should be a flat affine
morphism of ind-schemes $\pi\:\bY\rarrow\X=\fP(V)$ whose fibers over
the $\kk$\+points of $\fP(V)$ are the quadratic cones~$Y_q$.
 The aim of this section is to spell out a precise construction of
the ind-scheme~$\bY$ and the morphism~$\pi$.

\smallskip
 (0)~Let us first return to the discussion of the ind-scheme $\fP(V)$
from Section~\ref{cotangent-bundle-subsecn}(3).
 The approach hinted at in Section~\ref{tate-space-subsecn}(0)
suggests to describe schemes and ind-schemes by their ``functors of
points'', i.~e., the functors they represent on the category of
affine schemes.
 For our present purposes, let us restrict ourselves to \emph{field
extensions} $\kk\rarrow K$.
 One easly observes that, for any such field extension, the set
$\X(K)$ of $K$\+points in $\X=\fP(V)$ is naturally bijective to
the set of all one-dimensional vector subspaces $L\subset K\ot_\kk V$
in the $K$\+vector space $K\ot_\kk V$.
 (The case of an arbitrary commutative $\kk$\+algebra $K$ is
considerably more complicated; see~\cite[Example~1.3(2)]{Rich}.)

\smallskip
 (1)~Let $W$ be a discrete $\kk$\+vector space.
 The discrete vector space $\Sym_\kk^2(W)$ can be defined as
the degree~$2$ component of the graded $\kk$\+algebra $\Sym_\kk(W)$,
that is, the vector subspace in the $\kk$\+algebra $\Sym_\kk(W)$
spanned by the products $w'w''$ with $w'$, $w''\in W$.
 Put $V=\Sym_\kk^2(W)$.
 
 Let $U\subset V$ be a finite-dimensional vector subspace.
 Consider the graded ring $R=\Sym_\kk(U^*)$.
 Then the tensor product $\Sym_\kk(W)\ot_\kk R$ is a bigraded
commutative $R$\+algebra.
 For the purposes of applying the functor $M\longmapsto\widetilde M$
(see Section~\ref{cotangent-bundle-subsecn}(2)) we will consider
the grading on $\Sym_\kk(W)\ot_\kk R$ induced by the grading of~$R$.
 Then the functor $M\longmapsto\widetilde M$ takes the graded
$R$\+algebra $\Sym_\kk(W)\ot_\kk R$ to the quasi-coherent algebra
$\Sym_\kk(W)\ot_\kk\cO_{P(U)}$ on the scheme~$P(U)$.
 This quasi-coherent algebra corresponds to the affine morphism of
schemes $Y_{W^*}\times_\kk P(U)\rarrow P(U)$, where $Y_{W^*}=
\Spec\Sym_\kk(W)$ (as in Section~\ref{tate-space-subsecn}(3)).

\smallskip
 (2)~Consider the free graded module $\Sym_\kk(W)\ot_\kk R(-1)$ over
the graded algebra $\Sym_\kk(W)\ot_\kk R$.
 Applying the functor $M\longmapsto\widetilde M$ to this graded
module produces the quasi-coherent module 
$\Sym_\kk(W)\ot_\kk\cO_{P(U)}(-1)$ over the quasi-coherent algebra
$\Sym_\kk(W)\ot_\kk\cO_{P(U)}$ on~$P(U)$.

 The free graded module $\Sym_\kk(W)\ot_\kk R(-1)$ over the graded
algebra $\Sym_\kk(W)\ot_\kk R$ is spanned by the element~$1$ sitting in
degree~$1$ (in the grading induced by the grading of~$R$).
 The degree~$0$ component of the graded algebra $\Sym_\kk(W)\ot_\kk R$
is the algebra $\Sym_\kk(W)$, and the degree~$1$ component is
the vector space $\Sym_\kk(W)\ot_\kk U^*\simeq\Hom_\kk(U,\Sym_\kk(W))$.
 The natural injective $\kk$\+linear map $U\rarrow V\simeq
\Sym_\kk^2(W)\rarrow\Sym_\kk(W)$ defines a canonical element
$e\in\Sym_\kk(W)\ot_\kk U^*$.

 Let $f_U\:\Sym_\kk(W)\ot_\kk R(-1)\rarrow\Sym_\kk(W)\ot_\kk R$
be the morphism of graded $R$\+modules taking the generator
$1\in\Sym_\kk(W)\ot_\kk R(-1)$ to the element
$e\in\Sym_\kk(W)\ot_\kk U^*$.
 The morphism~$f_U$ is injective because the ring $\Sym_\kk(W)\ot_\kk R
\simeq\Sym_\kk(W\oplus U^*)$ has no zero-divisors.
 Denote by $C_U$ the cokernel of~$f_U$ taken in the category of
graded $R$\+modules.
 So $C_U$ is naturally a graded $R$\+algebra (namely, the quotient
algebra of $\Sym_\kk(W)\ot_\kk R$ by the ideal generated by
the homogeneous element~$e$).
 By the definition, the quasi-coherent algebra $\C_U$ on $P(U)=
\Proj R$ is obtained by applying the functor $M\longmapsto\widetilde M$
to the graded $R$\+algebra~$C_U$.

\smallskip
 Our next aim is to show that the quasi-coherent sheaf $\C_U$ on
$P(U)$ is flat.

\begin{lem} \label{mono-of-locally-free-lemma}
 Let $f\:\G\rarrow\F$ be a monomorphism of locally free coherent
sheaves on a Noetherian scheme~$X$.
 Assume that, for every field~$K$ and any morphism of schemes
$i\:\Spec K\rarrow X$, the morphism of coherent sheaves on\/ $\Spec K$
(i.~e., of $K$\+vector spaces) $i^*f\:i^*\G\rarrow i^*\F$ is injective.
 Then the cokernel $\F/f(\G)$ of the monomorphism~$f$ is a locally
free sheaf on~$X$.
\end{lem} 

\begin{proof}
 This lemma is well-known, so we restrict ourselves to pointing out
that it remains true for flat quasi-coherent sheaves in lieu of
locally free coherent ones.
 If $f\:\G\rarrow\F$ is a monomorphism of flat quasi-coherent sheaves
on a Noetherian scheme $X$ and, for every $i\:\Spec K\rarrow X$ as in
the lemma, the morphism $i^*f\:i^*\G\rarrow i^*\F$ is injective,
then $\F/f(\G)$ is a flat quasi-coherent sheaf on~$X$.
 This is (a particular case of) the result of~\cite[Remark~2.3]{CI},
which is based on~\cite[Proposition~5.3.F]{AF}; it also follows
from~\cite[Theorem~1.1]{CI}.
\end{proof}

 (3)~One can start from observing that, for any finite-dimensional
subspace $U\subset\Sym_\kk^2(W)$, there exists a finite-dimensional
subspace $\overline W\subset W$ such that
$U\subset\Sym_\kk^2(\overline W)$.
 The map~$f_U$ is the direct limit of the similar maps~$\bar f_U$
with the vector space~$W$ replaced by its finite-dimensional
subspaces~$\overline W$ satisfying this inclusion.
 Furthermore, the morphism of free graded $R$\+modules $\bar f_U\:
\Sym_\kk(\overline W)\ot_\kk R(-1)\rarrow\Sym_\kk(\overline W)
\ot_\kk R$ is the direct sum of the morphisms of finitely generated
free graded $R$\+modules $\Sym_\kk^n(\overline W)\ot_\kk R(-1)\rarrow
\Sym_\kk^{n+2}(\overline W)\ot_\kk R$, where $n=0$, $1$, $2$,~\dots{}
 These observations make Lemma~\ref{mono-of-locally-free-lemma}
applicable as it is stated (for locally free coherent sheaves),
and one does not even need the more general version of it
suggested in the proof.

 Let $K$ be a field and $l\:\Spec K\rarrow P(U)$ be a morphism of
schemes.
 Then the composition $\Spec K\rarrow P(U)\rarrow\Spec\kk$ makes~$\kk$
a subfield in~$K$.
 As mentioned above in~(0), the morphism~$l$ corresponds to
a one-dimensional vector subspace $L\subset K\ot_\kk U\subset
K\ot_\kk\Sym^2_\kk(\overline W)\simeq\Sym^2_K(K\ot_\kk \overline W)
\subset K\ot_\kk V$.
 So any nonzero vector $q\in L$ defines a quadratic function
$q\:(K\ot_\kk W)^*\rarrow K$, which factorizes as
$(K\ot_\kk W)^*\twoheadrightarrow (K\ot_\kk\overline W)^*
\overset{\bar q}\rarrow K$.
 Applying the functor $M\longmapsto\widetilde M$ and then
the functor~$i^*$ to the morphism of graded $R$\+modules
$\bar f_U\:\Sym_\kk(\overline W)\ot_\kk R(-1)\rarrow
\Sym_\kk(\overline W)\ot_\kk R$, one obtains the morphism of
$K$\+vector spaces $\Sym_K(K\ot_\kk \overline W)\ot_KL\rarrow
\Sym_K(K\ot_\kk \overline W)$ taking a tensor $a\ot_K\bar q$ with
$a\in\Sym_K^n(K\ot_\kk\overline W)$ and $\bar q\in L\subset
\Sym^2_K(K\ot_\kk \overline W)$ to the vector $a\bar q\in
\Sym_K^{n+2}(K\ot_\kk \overline W)$ (where the multiplication is
performed in the symmetric algebra $\Sym_K(K\ot_\kk \overline W)$).
 This map of $K$\+vector spaces is injective.

\smallskip
 (4)~According to Lemma~\ref{mono-of-locally-free-lemma}, it follows
that the quasi-coherent algebra $\C_U$ on the projective space $P(U)$
is flat.
 Now let $U'\subset U''\subset V$ be two finite-dimensional subspaces
in the $\kk$\+vector space $V=\Sym_\kk^2(W)$.
 Put $R'=\Sym_\kk(U'{}^*)$ and $R''=\Sym_\kk(U''{}^*)$.
 As explained in Section~\ref{cotangent-bundle-subsecn}(1\+-2),
we have a surjective morphism of graded rings $R''\rarrow R'$
inducing a closed immersion of the projective spectra
$i_{U'U''}\:P(U')\rarrow P(U'')$.

 Similarly to Section~\ref{cotangent-bundle-subsecn}(5), the functor
$R'\ot_{R''}{-}$ takes the morphism of graded $R''$\+modules~$f_{U''}$
to the morphism of graded $R'$\+modules~$f_{U'}$.
 Hence we have a natural isomorphism of graded $R'$\+modules, and in
fact of graded commutative $R'$\+algebras,
$C_{U'}\simeq R'\ot_{R''}C_{U''}$, and consequently a natural
isomorphism $\C_{U'}\simeq i_{U'U''}^* \C_{U''}$ of flat quasi-coherent
commutative algebras on~$P(U')$.

 Finally, the rule $\fC^{(P(U))}=\C_U$ defines a flat pro-quasi-coherent
commutative algebra $\fC$ on the ind-scheme $\X=\fP(V)$.
 The ind-scheme $\bY$ together with the flat affine morphism of
ind-schemes $\pi\:\bY\rarrow\X$ corresponds to the pro-quasi-coherent
algebra $\fC$ on $\X$ via the construction of
Proposition~\ref{ind-schemes-affine-morphisms-pro-qcoh-algebras}.
 This is the desired flat affine morphism of ind-schemes corresponding
to the $\fP(V)$\+parametric family of quadratic cones in
the linearly compact topological vector space~$W^*$.

\smallskip
 (5)~To make the setting for possible application of the constructions
and results of Sections~\ref{X-flat-on-Y-secn}\+-%
\ref{flat-affine-over-ind-finite-type-secn} complete, it remains to
specify a dualizing complex on the ind-scheme $\X=\fP(V)$.
 It was mentioned in Section~\ref{cotangent-bundle-subsecn}(8) how this
can be done (assuming the dimension of the $\kk$\+vector space $W$,
hence also $V$, is at most countable).

\subsection{Loop group of an affine algebraic group}
\label{loop-group-subsecn}
 This family of examples, suggested to us by an anonymous referee,
is a far-reaching generalization of Example~\ref{tate-space-subsecn}.
 It is technically more complicated than the preceding examples in
this section, so we restrict our discussion to a brief sketch.

\smallskip
 (0)~Let us start with the particular case of the general linear group
$\GL_{n,\kk}$.
 Consider the $\kk$\+vector space $\Mat_n(\kk)$ of $n\times n$
matrices with the entries in~$\kk$ (for some integer $n\ge0$).
 As explained in Section~\ref{tate-space-subsecn}(1), one can assign
to $\Mat_n(\kk)$ an affine scheme $\Mat_{n,\kk}$ over~$\kk$ such that,
for every commutative $\kk$\+algebra $K$, the $K$\+points on
$\Mat_{n,\kk}$ correspond bijectively to elements of the tensor
product $\Mat_n(K)=K\ot_\kk\Mat_n(\kk)$, i.~e., the $n\times n$
matrices with the entries in~$K$.
 Let $\GL_n(K)\subset\Mat_n(K)$ denote the group of \emph{invertible}
$n\times n$ matrices with the entries in~$K$.
 The $\kk$\+scheme $\GL_{n,\kk}$ is an affine open subscheme in
$\Mat_{n,\kk}$ such that, for every $K$, the $K$\+points on
$\GL_{n,\kk}$ correspond bijectively to the elements of $\GL_n(K)$.
 Explicitly, the affine open subvariety $\GL_{n,\kk}\subset\Mat_{n,\kk}$
is the complement to the closed subvariety in the $n^2$\+dimensional affine space over~$\kk$ defined by the single equation $\det(A)=0$,
where $\det(A)$ is the determinant of an $n\times n$ matrix.
 The multiplication and inverse element maps $\GL_n(K)\times
\GL_n(K)\rarrow\GL_n(K)$ and $\GL_n(K)\rarrow\GL_n(K)$ are given
by polynomials in the matrix entries $a_{i,j}$ and the inverse
determinant $(\det A)^{-1}$, where $A=(a_{i,j})_{i,j=1}^n\in\GL_n(K)$;
so they are induced by morphisms of schemes $\GL_{n,\kk}\times_\kk
\GL_{n,\kk}\rarrow\GL_{n,\kk}$ and $\GL_{n,\kk}\rarrow\GL_{n,\kk}$.
 The latter morphisms define the structure of an \emph{affine
algebraic group} on the affine algebraic variety $\GL_{n,\kk}$
over~$\kk$.

\smallskip
 (1)~Let $\kk((t))$ be the field of formal Laurent power series in
a variable~$t$ over a field~$\kk$, and let $\kk[[t]]\subset\kk((t))$
be its subring of formal Taylor power series.
 Taking $K=\kk((t))$ or $K=\kk[[t]]$ in the constructions above,
one obtains the group of points $\GL_n(\kk((t)))$ and its subgroup
$\GL_n(\kk[[t]])$.
 We claim that there is an affine scheme (of infinite type) over~$\kk$,
denoted formally by $\GL_{n,\kk}[[t]]$, whose set of $\kk$\+points
is $\GL_n(\kk[[t]])$.
 Furthermore, there is a naturally defined ind-affine ind-scheme
(of ind-infinite type) over~$\kk$, denoted formally by
$\GL_{n,\kk}((t))$, whose set of $\kk$\+points is $\GL_n(\kk((t)))$.

 To be more precise, one needs to describe the sets of $K$\+points
on $\GL_{n,\kk}[[t]]$ and $\GL_{n,\kk}((t))$ for any commutative
$\kk$\+algebra~$K$.
 Let $K[[t]]=\varprojlim_{m\ge1} K[t]/t^mK[t]$ denote the ring of
formal Taylor power series with the coefficients in $K$, and let
$K((t))=\varinjlim_{m\ge0} t^{-m}K[[t]]$ be the ring of formal Laurent
power series with the coefficients in~$K$.
 Then the set of $K$\+points on $\GL_{n,\kk}[[t]]$ is $\GL_n(K[[t]])$
and the set of $K$\+points on $\GL_{n,\kk}((t))$ is $\GL_n(K((t)))$.

 Explicitly, one can describe the scheme $\GL_{n,\kk}[[t]]$ and
the ind-scheme $\GL_{n,\kk}((t))$ as follows.
 Let $A=(a_{i,j})_{i,j=1}^n\in\Mat_n(\kk((t)))$ be an $n\times n$
matrix whose entries $a_{i,j}=\sum_{l=-m}^\infty a_{i,j;l}t^l$,
$\,m\in\boZ$, are formal Laurent power series in the variable~$t$.
 This notation presumes the coefficients $a_{i,j;l}\in\kk$, but
let us actually consider $a_{i,j;l}$ as abstract variables.
 Then the determinant $\det(A)$ is a formal Laurent power series
whose coefficients are polynomials in~$a_{i,j;l}$.
 Now $\GL_{n,\kk}[[t]]$ is the spectrum of the commutative
$\kk$\+algebra freely generated by the variables $a_{i,j;l}$,
\,$1\le i,j\le n$, \,$l\ge0$, with the formally adjoined inverse
element to the coefficient at $t=0$ of the formal Taylor power series
$\det(A)$, where $A=(a_{i,j})_{i,j=1}^n$ and 
$a_{i,j}=\sum_{l=0}^\infty a_{i,j;l}t^l$.

 The ind-scheme $\GL_{n,\kk}((t))$ can be represented by
the following inductive system of infinite-dimensional
$\kk$\+schemes $\ilim_{m\ge0} Y_m$.
 For every integer $m\ge0$, the affine scheme $Y_m$ is the spectrum
of the commutative $\kk$\+algebra generated by two groups of variables
$a_{i,j;l}$ and $b_{i,j;l}$, \,$1\le i,j\le n$, \,$l\ge-m$, with
the imposed relations that the product of the two matrices of formal
Laurent power series $A$ and $B=(b_{i,j})_{i,j=1}^n$, where $b_{i,j}=
\sum_{l=-m}^\infty b_{i,j;l}t^l$, is equal to the identity matrix.

 Furthermore, both the affine scheme $\GL_{n,\kk}[[t]]$ and
the ind-affine ind-scheme $\GL_{n,\kk}((t))$ are endowed with
the group (ind-)scheme structures, with the multiplication morphisms
$\GL_{n,\kk}[[t]]\times_\kk\GL_{n,\kk}[[t]]\rarrow\GL_{n,\kk}[[t]]$
and $\GL_{n,\kk}((t))\times\GL_{n,\kk}((t))\rarrow\GL_{n,\kk}((t))$,
and the inverse element morphisms $\GL_{n,\kk}[[t]]\rarrow
\GL_{n,\kk}[[t]]$ and $\GL_{n,\kk}((t))\rarrow\GL_{n,\kk}((t))$.
 These group (ind-)scheme structures induce the group structures on
the sets of points $\GL_{n,\kk}[[t]](K)=\GL_n(K[[t]])$ and
$\GL_{n,\kk}((t))(K)=\GL_n(K((t)))$ mentioned above in this discussion.
{\hbadness=1150\par}

\smallskip
 (2)~Consider the set $\GL_n(\kk((t)))/\GL_n(\kk[[t]])$ of all left
cosets of the group $\GL_n(\kk((t)))$ with respect to its subgroup
$\GL_n(\kk[[t]])$.
 We claim that there is a naturally defined ind-projective ind-scheme
of ind-finite type $\X$ over~$\kk$ whose set of $\kk$\+points is
$\GL_n(\kk((t)))/\GL_n(\kk[[t]])$.
 Moreover, for any \emph{field} $K$ over $\kk$, the set $\X(K)$ is
naturally bijective to $\GL_n(K((t)))/\GL_n(K[[t]])$.

 The ind-scheme $\X$ can be constructed as follows  (for a reference,
see, e.~g., \cite[Section~VII.1]{Kum}; cf.~\cite[Sections~7.11.17
and~7.15.1]{BD2})
 Consider the free $\kk((t))$\+module with $n$~generators
$V=k((t))^n$.
 A \emph{$\kk[[t]]$\+lattice} in $V$ is a finitely generated
$\kk[[t]]$\+submodule $L\subset V$ such that the natural map
$\kk((t))\ot_{\kk[[t]]}L\rarrow V$ is an isomorphism.
 We observe that the set of cosets $\GL_n(\kk((t)))/\GL_n(\kk[[t]])$
is naturally bijective to the set of all $\kk[[t]]$\+lattices in~$V$.
 Indeed, $\GL_n(k((t)))$ is the group of all automorphisms of
the $\kk((t))$\+module~$V$; so the group $\GL_n(\kk((t)))$ acts
on $V$, and consequently it also acts on the set of all
$\kk[[t]]$\+lattices in~$V$.
 One can easily see that this action is transitive (because
all $\kk[[t]]$\+lattices in $V$ are isomorphic to each other as
$\kk[[t]]$\+modules, the ring $k[[t]]$ being a principal ideal domain).
 Furthermore, let $L_0=\kk[[t]]^n\subset\kk((t))^n$ be a chosen
``standard'' $\kk[[t]]$\+lattice in~$V$.
 Then the subgroup $\GL_n(\kk[[t]])\subset\GL_n(\kk((t)))$ is
the stabilizer of the element $L_0$ in the action of $\GL_n(\kk((t)))$
on the set of all $\kk[[t]]$\+lattices.
 Hence the map taking a coset $g\GL_n(\kk[[t]])\in
\GL_n(\kk((t)))/\GL_n(\kk[[t]])$ to the lattice $g(L_0)\subset V$
provides a bijection between the set of all left cosets of
$\GL_n(\kk((t)))$ modulo $\GL_n(\kk[[t]])$ and the set of all
$\kk[[t]]$\+lattices in~$V$.

 Now we notice that for any $\kk[[t]]$\+lattice $L\subset V$ there
exists an integer $m\ge0$ such that $t^m L_0\subset L\subset t^{-m}L_0$.
 For any submodule $E$ of the $\kk[t]/t^{2m}\kk[t]$\+module
$t^{-m}L_0/t^mL_0$, the full preimage of $E$ in $t^{-m}L_0\subset V$
is a $\kk[[t]]$\+lattice in~$V$.
 Consider the Grassmann variety $G_m$ of all vector subspaces in
the finite-dimensional $\kk$\+vector space $t^{-m}L_0/t^mL_0$, and
let $X_m\subset G_m$ be the subvariety defined by the equations
that the vector subspace in $t^{-m}L_0/t^mL_0$ is
a $\kk[t]/t^{2m}\kk[t]$\+submodule.
 So $G_m$ is a disconnected projective algebraic variety over~$\kk$
with smooth components, while $X_m\subset G_m$ is a closed
projective subvariety.
 Now the set of all $\kk$\+points on $X_m$ is bijective to the set of
all $\kk[[t]]$\+lattices $L\subset V$ for which $t^m L_0\subset L
\subset t^{-m}L_0$.
 The desired ind-scheme of ind-finite type $\X$ over~$\kk$ is
represented by the inductive system $\ilim_{m\ge0}X_m$.

\smallskip
 (3)~Alternatively, there is a natural way to produce
the ind-scheme $\X$ as a closed ind-subscheme in
an ind-projective space of ind-finite type.
 The free $\kk((t))$\+module $V=\kk((t))^n$ has a natural topology
which allows to view it as a locally linearly compact topological
vector space in the sense of
Section~\ref{tate-space-subsecn}(4).
 For any such vector space $V$, the infinite-dimensional discrete
vector space of \emph{semi-infinite exterior forms}
$\Lambda^{\infty/2+*}(V)$ is constructed as follows.

 Choose a linearly compact open vector subspace $W\subset V$.
 For every integer $m\in\boZ$, put $\Lambda^{\infty/2+m}_W(V)
=\varinjlim_{W'\subset W}\Lambda_\kk^{\dim_\kk W/W'+m}(V/W')\ot_\kk
\Lambda_\kk^{\dim_\kk W/W'}(W/W')^*$, where $W'$ ranges over
the directed poset of all open vector subspaces of~$W$.
 Here $\Lambda_\kk^{\dim_\kk W/W'}(W/W')^*$ is the dual
one-dimensional $\kk$\+vector space to the top exterior power
$\Lambda_\kk^{\dim_\kk W/W'}(W/W')$ of the finite-dimensional
$\kk$\+vector space $W/W'$, and the direct limit is taken over
the injective linear maps
\begin{multline*}
\Lambda_\kk^{\dim_\kk W/W'+m}(V/W')\ot_\kk
\Lambda_\kk^{\dim_\kk W/W'}(W/W')^* \\ \lrarrow
\Lambda_\kk^{\dim_\kk W/W''+m}(V/W'')\ot_\kk
\Lambda_\kk^{\dim_\kk W/W''}(W/W'')^*
\end{multline*}
induced by the natural isomorphisms
$$
 \Lambda_\kk^{\dim_\kk W/W''}(W/W'')\simeq
 \Lambda_\kk^{\dim_\kk W/W'}(W/W')\ot_\kk
 \Lambda_\kk^{\dim_\kk W'/W''}(W'/W'')
$$
and the natural injective linear maps
$$
 \Lambda_\kk^{\dim_\kk W/W'+m}(V/W')\ot_\kk
 \Lambda_\kk^{\dim_\kk W'/W''}(W'/W'')\lrarrow
 \Lambda_\kk^{\dim_\kk W/W''+m}(V/W'')
$$
defined for all open vector subspaces $W''\subset W'\subset W$.
 Finally, put $\Lambda^{\infty/2+*}_W(V)=\bigoplus_{m\in\boZ}
\Lambda^{\infty/2+m}_W(V)$.

 The vector space of semi-infinite forms $\Lambda^{\infty/2+*}(V)$
is not quite well-defined, in that it depends on the choice of
a linearly compact open subspace $W\subset V$ (reflected by
the more precise notation above).
 Replacing $W$ by a different linearly compact open vector
subspace $W'\subset V$ leads to the vector space
$\Lambda^{\infty/2+*}(V)$ getting twisted by (taking the tensor
product with) a one-dimensional vector space.
 Consequently, the \emph{projectivization} of 
$\Lambda^{\infty/2+*}(V)$, that is the set of all one-dimensional
$\kk$\+vector subspaces in $\Lambda^{\infty/2+*}(V)$, is perfectly
well-defined.

 To any discrete vector space $U$ over~$\kk$ one can naturally assign
an ind-scheme $\fP(U)$ of ind-finite type such that the $\kk$\+points
on $\fP(U)$ correspond bijectively to the one-dimensional vector
subspaces in $U$, as explained
in Section~\ref{cotangent-bundle-subsecn}(3)
and~\cite[Example~1.3(2)]{Rich}.
 The ind-scheme $\fP(U)$ is represented by the inductive system
of closed immersions of finite-dimensional projective spaces
corresponding to the finite-dimensional vector subspaces of~$U$.
 In the situation at hand, we take $U=\Lambda^{\infty/2+*}_W(V)$,
where $V=\kk((t))^n$ and $W=\kk[[t]]^n$.

 The group $\GL_V$ of all $\kk$\+linear automorphisms of $V$
acts naturally on the set of points $\fP(U)(\kk)$ of
the ind-projective space $\fP(U)$.
 In particular, this action can be resricted onto the subgroup
$\GL_n(\kk((t)))\subset\GL_V$ all $\kk((t))$\+linear automorphisms
of $V=\kk((t))^n$.
 The one-dimensional vector subspace $\Lambda^0_\kk(V/W)\ot_\kk
\Lambda_\kk^0(W/W)^*\subset U$ represents a chosen fixed point
$l_0\in\fP(U)(\kk)$.
 The subgroup $\GL_n(\kk[[t]])\subset\GL_n(\kk((t)))$ is
the stabilizer of the point~$l_0$ in the action
of $\GL_n(\kk((t)))$ on $\fP(U)(\kk)$.
 Hence the map $g\GL_n(\kk[[t]])\longmapsto g(l_0)$ embeds
the set of all cosets $\GL_n(\kk((t)))/\GL_n(\kk[[t]])$ into
the set of $\kk$\+points of the ind-scheme $\fP(U)$.
 This embedding is induced by a closed immersion of ind-schemes
$\X\rarrow\fP(U)$.

\smallskip
 (4)~Put $\bY=\GL_{n,\kk}((t))$.
 Then the natural surjective map of the sets of $\kk$\+points
$\bY(\kk)=\GL_n(\kk((t)))\rarrow\GL_n(\kk((t)))/\GL_n(\kk[[t]])=
\X(\kk)$ is induced by a naturally defined flat affine morphism of
ind-schemes $\pi\:\bY\rarrow\X$.
 Following the approach of
Remarks~\ref{cotensor-right-derived-remark}(3\+-5), one can
construct a dualizing complex $\rD^\bu$ on the ind-scheme~$\X$.
 So the constructions and results of Sections~\ref{X-flat-on-Y-secn}\+-%
\ref{flat-affine-over-ind-finite-type-secn} are applicable
to this example.

 More generally, a \emph{congruence subgroup} in the affine group
scheme $\GL_{n,\kk}[[t]]$ is a group subscheme
$H\subset\GL_{n,\kk}[[t]]$ of the following kind.
 For any integer $m\ge1$, there exists an affine algebraic group
$\GL_{n,\kk}^{(m)}$ over~$\kk$ such that, for any commutative
$\kk$\+algebra $K$, the group of $K$\+points on $\GL_{n,\kk}^{(m)}$
is naturally  isomorphic to $\GL_n(K[t]/t^mK[t])$.
 Let $\GL_{n,\kk}[[t]]_{(m)}$ denote the kernel of the natural
surjective morphism of affine group schemes $\GL_{n,\kk}[[t]]
\rarrow\GL_{n,\kk}^{(m)}$.
 Then a \emph{congruence subgroup} $H\subset\GL_{n,\kk}[[t]]$ is
a closed group subscheme intermediate between $\GL_{n,\kk}[[t]]$ and
$\GL_{n,\kk}[[t]]_{(m)}$ for some $m\ge1$, that is
$\GL_{n,\kk}[[t]]_{(m)}\subset H\subset\GL_{n,\kk}[[t]]$.

 For any congruence subgroup $H\subset\GL_{n,\kk}[[t]]$, there is
a naturally defined ind-scheme of ind-finite type $\X'$ over~$\kk$
whose set of $\kk$\+points is $\GL_n(\kk((t)))/H(\kk)$.
 The natural surjective map of the sets of $\kk$\+points
$\bY(\kk)=\GL_n(\kk((t)))\rarrow\GL_n(\kk((t)))/H(\kk)=\X'(\kk)$ is
induced by a naturally defined flat affine morphism of ind-schemes
$\pi'\:\bY\rarrow\X'$.
 The natural surjective map $\X'(\kk)=\GL_n(\kk((t)))/H(\kk)
\allowbreak\rarrow\GL_n(\kk((t)))/\GL_n(\kk[[t]])=\X(\kk)$ is induced
by a naturally defined smooth morphism of finite type between the two
ind-schemes of ind-finite type $\tau\:\X'\rarrow\X$.
 So the results of Section~\ref{weakly-smooth-postcomposition-secn}
are applicable, showing that the theories associated with
the two morphisms $\pi\:\bY\rarrow\X$ and $\pi'\:\bY\rarrow\X'$ are
essentially equivalent.
 In other words, one take the quotient of $\GL_{n,\kk}((t))$ by
one's preferred congruence subgroup $H\subset\GL_{n,\kk}[[t]]$
instead of the quotient by the whole group scheme $\GL_{n,\kk}[[t]]$,
and the constructions of Sections~\ref{X-flat-on-Y-secn}\+-%
\ref{flat-affine-over-ind-finite-type-secn} will remain
essentially unchanged.

\smallskip
 (5)~Now let $G$ be an affine algebraic group over~$\kk$.
 Without loss of generality, one can assume $G$ to be a closed
algebraic subvariety in $\GL_{n,\kk}$ containing the unit element and preserved by the multiplication and inverse element morphisms of
$\GL_{n,\kk}$ (for some integer $n\ge0$).
 For any commutative $\kk$\+algebra $K$, the set of $K$\+points
$G(K)$ on the algebraic group $G$ has a natural group structure.
 The same applies to the sets of $K$\+points on infinite-dimensional
group schemes and group ind-schemes.

 We claim that there is an affine group scheme (of infinite type)
over~$\kk$, denoted formally by $G[[t]]$, whose group of $K$\+points
is $G(K[[t]])$ for every~$K$.
 Furthermore, there is an ind-affine group ind-scheme (of ind-infinite
type) over~$\kk$, denoted formally by $G((t))$, whose group of
$K$\+points is $G(K((t)))$.

 Explicitly, one can describe the group scheme $G[[t]]$ and
the group ind-scheme $G((t))$ as follows.
 We assume that $G$ is a closed group subscheme in $\GL_{n,\kk}$.
 As in~(1), we consider matrices $A=(a_{i,j})_{i,j=1}^n$ of
formal power series $a_{i,j}=\sum_l a_{i,j;l}t^l$, where $a_{i,j;l}$
are abstract variables.
 Then $G[[t]]$ is the spectrum of the commutative $\kk$\+algebra
generated by the variables $a_{i,j;l}$, \,$1\le i,j\le n$, \,$l\ge0$,
with the formally adjoined inverse element to the coefficient at $t=0$
of the formal Taylor power series $\det(A)$, where $A$ is the matrix
with the entries $a_{i,j}=\sum_{l=0}^\infty a_{i,j;l}t^l$, and
the imposed equations telling that the matrix $A$ satisfies
the equations on the points of the closed subvariety
$G\subset\GL_{n,\kk}$.
 So substituting $A=(a_{i,j})_{i,j=1}^n$ into the equations on
the matrix entries defining $G$ inside $\GL_{n,\kk}$ produces
the equations which are to be imposed on the coefficients~$a_{i,j;l}$.

 The ind-scheme $G((t))$ can be represented by the following
inductive system of infinite-dimensional affine $\kk$\+schemes
$\ilim_{m\ge0}Y_m$.
 For every integer $m\ge0$, the $\kk$\+scheme $Y_m$ is the spectrum
of the commutative $\kk$\+algebra generated by two groups of
variables $a_{i,j;l}$ and $b_{i,j;l}$, \,$1\le i,j\le n$,
\,$l\ge-m$, with the imposed relations that the product of
the two matrixes $A$ and $B=(b_{i,j})_{i,j=1}^n$, where
$a_{i,j}=\sum_{l=-m}^\infty a_{i,j}t^l$ and
$b_{i,j}=\sum_{l=-m}^\infty b_{i,j}t^l$, is equal to the identity
matrix, and that the matrices $A$ and $B$ satisfy the equations
on the points of the closed subvariety $G\subset\GL_{n,\kk}$.
 So, once again, the equations on the coefficients $a_{i,j;l}$ and
$b_{i,j;l}$ produced by substituting $A=(a_{i,j})_{i,j=1}^n$ and
$B=(b_{i,j})_{i,j=1}^n$ into the equations on the matrix entries
defining $G$ inside $\GL_{n,\kk}$ must be satisfied.

 The matrix multiplication, expressed by polynomials in
the coefficients $a_{i,j;l}$, defines the multiplications morphisms
$G[[t]]\times_\kk G[[t]]\rarrow G[[t]]$ and $G((t))\times_\kk
G((t))\rarrow G((t))$.
 The inverse element morphism for the group ind-scheme $G((t))$ is
constructed by taking a pair of matrices $(A,B)$ to the pair $(B,A)$.

\smallskip
 (6)~As in~(3), we consider the set of left cosets
$G(k((t)))/G(k[[t]])$.
 There is a naturally defined ind-scheme of ind-finite type $\X$
over~$\kk$ whose set of $\kk$\+points is $G(k((t)))/G(k[[t]])$.
 Moreover, for any \emph{field} $K$ over~$\kk$, the set $\X(K)$
is naturally bijective to $G(K((t)))/G(K[[t]])$.
 The natural injective maps $G(K((t)))/G(K[[t]])\rarrow
\GL_n(K((t)))/\GL_n(K[[t]])$ are induced by a naturaly locally
closed immersion of the related ind-schemes of ind-finite type.

 Put $\bY=G((t))$.
 Then the natural surjective map of the sets of $\kk$\+points
$\bY(\kk)=G(\kk((t)))\rarrow G(\kk((t)))/G(\kk[[t]])=\X(\kk)$
is induced by a naturally defined flat affine morphism of
ind-schemes $\pi\:\bY\rarrow\X$.
 Following Remarks~\ref{cotensor-right-derived-remark}(3\+-5), one can
construct a dualizing complex $\rD^\bu$ on the ind-scheme~$\X$;
so the constructions and results of Sections~\ref{X-flat-on-Y-secn}\+-%
\ref{flat-affine-over-ind-finite-type-secn} are applicable
to this example as well.

 Similarly to~(4), one can consider \emph{congruence subgroups}
$G[[t]]_{(m)}\subset H\subset G[[t]]$, where $G[[t]]_{(m)}$ is
the kernel of the natural surjective morphism of affine group
schemes $G[[t]]\rarrow G^{(m)}$, and $G^{(m)}$ the affine
algebraic group over~$\kk$ such that, for any commutative
$K$\+algebra $K$, the group of $K$\+points on $G^{(m)}$ is
naturally isomorphic to $G(K[t]/t^mK[t])$.
 For any congruence subgroup $H\subset G[[t]]$, there is 
a naturally defined ind-scheme of ind-finite type $\X'$ over~$\kk$
whose set of $\kk$\+points is $G(\kk((t)))/H(\kk)$.
 The ind-scheme $\X'$ comes together with the related flat affine
morphism of ind-schemes $\pi'\:\bY\rarrow\X'$.
 Just as in~(4), there is a commutative triangle diagram of
morphisms of ind-schemes $\bY\rarrow\X'\rarrow\X$ with
a smooth morphism of finite type between the two ind-schemes
of ind-finite type $\tau\:\X'\rarrow\X$, inducing, via the passage
to the sets of $\kk$\+points, the commutative triangle diagram of
the sets of cosets $G(\kk((t)))\rarrow G(\kk((t)))/H(\kk)\rarrow
G(\kk((t)))/G(\kk[[t]])$.
 So the results of Section~\ref{weakly-smooth-postcomposition-secn}
are applicable, showing that the constructions of
Sections~\ref{X-flat-on-Y-secn}\+-%
\ref{flat-affine-over-ind-finite-type-secn} produce essentially
the same functors for the two morphisms $\pi'\:\bY\rarrow\X'$
and $\pi\:\bY\rarrow\X$.

\appendix
\bigskip
\section*{Appendix.  The Semiderived Category for a Nonaffine
Morphism}
\medskip
\setcounter{section}{1}
\setcounter{thm}{0}
\setcounter{equation}{0}

 Let $\X$ be an ind-semi-separated ind-Noetherian ind-scheme, $\bY$ be
an ind-semi-separated ind-scheme, and $\pi\:\bY\rarrow\X$ be
a \emph{nonaffine} flat morphism of ind-schemes.
 The aim of this appendix is to spell out a definition of
the $\bY/\X$\+semiderived category $\sD_\X^\si(\bY\tors)$ of
quasi-coherent torsion sheaves on $\bY$ in this context.

\subsection{Becker's coderived category}
 In this appendix, we will use a different approach to the definition
of the coderived category than in the main body of the paper
(cf.\ Section~\ref{coderived-subsecn}).
 The relevant references for Becker's coderived category
are~\cite{Kra,Neem2,Bec,Sto,PS5}.

 Let $\sE$ be an exact category with enough injective objects.
 Notice that in any such exact category the infinite coproduct functors
are exact (if the coproducts exist).
 A complex $E^\bu$ in $\sE$ is said to be \emph{Becker coacyclic}
(or ``coacyclic in the sense of Becker'') if, for any complex of
injective objects $J^\bu$ in $\sE$, the complex of morphisms
$\Hom_\sE(E^\bu,J^\bu)$ is acyclic (as a complex of abelian groups).

\begin{lem} \label{becker-coacyclic-lemma}
\textup{(a)} The totalization of any short exact sequence of complexes
in\/ $\sE$ is a Becker coacyclic complex. \par
\textup{(b)} The coproduct of any family of Becker coacyclic complexes,
if it exists in\/ $\sK(\sE)$, is a Becker coacyclic complex. \par
\textup{(c)} Consequently, if the infinite coproducts exist in\/ $\sE$,
then any coacyclic complex in the sense of
Section~\ref{coderived-subsecn} is also coacyclic
in the sense of Becker.
\end{lem}

\begin{proof}
 Parts~(a\+-b) are (a straightforward generalization of)
\cite[Lemma~9.1]{PS5}; they are closely related to
Proposition~\ref{coderived-and-homotopy-of-injectives}(a).
 Part~(c) follows immediately from~(a\+-b).
\end{proof}

\begin{lem} \label{becker-coacyclic-acyclic}
 If the category\/ $\sE$ is abelian with the abelian exact structure,
then any coacyclic complex in\/ $\sE$ is acyclic.
\end{lem}

\begin{proof}
 A complex $E^\bu$ in $\sE$ is acyclic if and only if the complex of
abelian groups $\Hom_\sE(E^\bu,J)$ is acyclic for any injective object
$J\in\sE$ (viewed as a one-term complex of injective objects).
\end{proof}

 The \emph{Becker coderived category} $\sD^\bco(\sE)$ is defined as
the triangulated quotient category of the homotopy category $\sK(\sE)$
by the thick subcategory of Becker coacyclic complexes.
 It is clear from Lemma~\ref{becker-coacyclic-lemma}(c) that Becker's
coderived category of an exact category $\sE$ with infinite coproducts
and enough injectives is (at worst) a triangulated quotient category
of the coderived category in the sense of 
Section~\ref{coderived-subsecn}.
 So there is a triangulated Verdier quotient functor\/
$\sD^\co(\sE)\rarrow\sE^\bco(\sE)$ forming a commutative triangle
diagram with the triangulated Verdier quotient functors
$\sK(\sE)\rarrow\sD^\co(\sE)$ and $\sK(\sE)\rarrow\sD^\bco(\sE)$.
 When $\sE$ is abelian, it follows from
Lemma~\ref{becker-coacyclic-acyclic} that there is also a triangulated
Verdier quotient functor $\sD^\bco(\sE)\rarrow\sD(\sE)$.
 
\begin{prop} \label{two-coderived-categories-comparison}
 Let\/ $\sE$ be an exact category with infinite coproducts and enough
injective objects such that the full subcategory of injective objects\/
$\sE_\inj$ is preserved by the infinite coproducts in\/~$\sE$ (e.~g.,
this holds for any locally Noetherian Grothendieck abelian category
with the abelian exact structure).
 Then the canonical functor\/ $\sD^\co(\sE)\rarrow\sD^\bco(\sE)$ is
a triangulated equivalence.
\end{prop}

\begin{proof}
 Follows from Proposition~\ref{coderived-and-homotopy-of-injectives}(b).
 The conditions on the exact category $\sE$ can be relaxed a bit;
see condition~($*$) in~\cite[Section~3.7]{Pkoszul} or (even more
generally) the results of~\cite[Section~A.6]{Pcosh}.
\end{proof}

 In is an \emph{open problem} whether the functor $\sD^\bco(\sE)\rarrow
\sD^\co(\sE)$ is a triangulated equivalence for every Grothendieck
abelian category $\sE$ (with the abelian exact structure), or even
for the category of modules over an arbitrary ring.
 See, e.~g., \cite[Example~2.5(3)]{Pps} for a discussion.
 The advantage of Becker's coderived category, though, is that it is
known to work well for all Grothendieck abelian categories.

\begin{thm} \label{becker-coderived-theorem}
 Let\/ $\sA$ be a Grothendieck abelian category and\/ $\sA_\inj
\subset\sA$ be its full subcategory of injective objects.
 Then the composition\/ $\sK(\sA_\inj)\rarrow\sK(\sA)\rarrow
\sD^\bco(\sA)$ of the inclusion functor\/ $\sK(\sA_\inj)\rarrow\sK(\sA)$
and the Verdier quotient functor\/ $\sK(\sA)\rarrow\sD^\bco(\sA)$ is
a triangulated equivalence\/ $\sK(\sA_\inj)\simeq\sD^\bco(\sA)$.
\end{thm}

\begin{proof}
 This result can be found in~\cite[Corollary~5.13]{Kra2},
\cite[Theorem~3.13]{Neem2}, or~\cite[Corollary~9.4]{PS5}.
\end{proof}

\begin{lem} \label{becker-coacyclicity-preserved}
 Let\/ $\sA$ and\/ $\sB$ be Grothendieck abelian categories, and let
$F\:\sA\rarrow\sB$ be an exact functor which has a right adjoint
(equivalently, $F$ is exact and preserves coproducts).
 Then the functor $F$ takes Becker coacyclic complexes in\/ $\sA$
to Becker coacyclic complexes in\/~$\sB$.
\end{lem}

\begin{rem} \label{also-true-for-pco}
 An analogue of Lemma~\ref{becker-coacyclicity-preserved} holds for
coacyclic complexes in the sense of
Section~\ref{coderived-subsecn} under weaker assumptions: any
exact functor preserving coproducts, acting between exact categories
with exact coproducts, preserves coacyclicity.
 This assertion, following immediately from the definitions, was
mentioned and used many times throughout the main body of this paper.
\end{rem}

\begin{proof}[Proof of Lemma~\ref{becker-coacyclicity-preserved}]
 It is a particular case of the Special Adjoint Functor Theorem that
a functor between cocomplete abelian categories having sets of
generators is a left adjoint if and only if it preserves colimits
(equivalently, is right exact and preserves coproducts).
 This explains the equivalent reformulation of the lemma's assumptions
in the parentheses.
 Now let $G\:\sB\rarrow\sA$ be the right adjoint functor to~$F$.
 Since the functor $F$ is exact, the functor $G$ takes injectives
to injectives.
 Let $A^\bu$ be a complex in $\sA$ and $J^\bu$ be a complex of
injective objects in~$\sB$.
 Then the isomorphism of complexes of abelian groups
$\Hom_\sB(F(A^\bu),J^\bu)\simeq\Hom_\sA(A^\bu,G(J^\bu))$ shows that
the complex $F(A^\bu)$ is coacyclic in $\sB$ whenever a complex
$A^\bu$ is coacyclic in~$\sA$.
\end{proof}

\subsection{Locality of coacyclity on schemes}
\label{locality-of-coacyclicity-subsecn}
 Let $\X$ be an reasonable ind-scheme.
 Then the category $\X\tors$ of quasi-coherent torsion sheaves on $\X$
is a Grothendieck abelian category (by
Theorem~\ref{torsion-sheaves-abelian}).
 So it makes sense to speak about Becker coacyclic complexes in
$\X\tors$ and the Becker coderived category $\sD^\bco(\X\tors)$.

\begin{lem} \label{flat-inverse-image-preserves-becker-coacyclicity}
 Let $f\:\Y\rarrow\X$ be a flat morphism of reasonable ind-schemes.
 Then the inverse image functor $f^*\:\X\tors\rarrow\Y\tors$ takes
Becker coacyclic complexes in\/ $\X\tors$ to Becker coacyclic
complexes in\/ $\Y\tors$.
\end{lem}

\begin{proof}
 The functor~$f^*$ is exact by
Lemma~\ref{flat-torsion-sheaves-inverse-image}
and has a right adjoint by
Lemma~\ref{torsion-direct-inverse-adjunction}(b),
so it remains to apply
Lemma~\ref{becker-coacyclicity-preserved}.
\hbadness=1400
\end{proof}

\begin{lem} \label{affine-direct-image-preserves-becker-coacyclicity}
 Let $f\:\Y\rarrow\X$ be an affine morphism of reasonable ind-schemes.
 Then the direct image functor $f_*\:\Y\tors\rarrow\X\tors$ takes
Becker coacyclic complexes in\/ $\Y\tors$ to Becker coacyclic
complexes in\/ $\X\tors$.
\end{lem}

\begin{proof}
 The functor~$f_*$ is exact by
Lemma~\ref{affine-torsion-direct-image} and preserves coproducts
by Lemma~\ref{representable-by-schemes-direct-image}(a), so
Lemma~\ref{becker-coacyclicity-preserved} is applicable.
\hbadness=1375
\end{proof}

 The ideas of the formulations and proofs of the next two lemmas
can be found in~\cite[Remark~1.3]{EP} (where a more complicated setting
of quasi-coherent curved DG\+modules is considered).

\begin{lem}
 Let $X=\bigcup_\alpha U_\alpha$ be an open covering of a Noetherian
scheme~$X$.
 Let $j_\alpha\:U_\alpha\rarrow X$ denote the open immersion morphisms.
 Then a complex of quasi-coherent sheaves $\M^\bu$ on $X$ is Becker
coacyclic if and only if, for every~$\alpha$, the complex of
quasi-coherent sheaves $j_\alpha^*\M^\bu$ on $U_\alpha$ is Becker
coacyclic. 
\end{lem}

\begin{proof}
 In fact, by Proposition~\ref{two-coderived-categories-comparison},
there is no difference between the Becker coacyclicity and
coacyclicity in the sense of Section~\ref{coderived-subsecn} in
the assumptions of this lemma.
 The functors~$j_\alpha^*$ preserve coacyclicity by
Lemma~\ref{flat-inverse-image-preserves-becker-coacyclicity};
so the ``only if'' assertion is clear.

 To prove the ``if'', one says that either by
Proposition~\ref{coderived-and-homotopy-of-injectives}(b) or
by Theorem~\ref{becker-coderived-theorem} there exists a complex of
injective quasi-coherent sheaves $\J^\bu$ on $X$ together with
a morphism of complexes $\M^\bu\rarrow\J^\bu$ with a coacyclic cone.
 By Lemma~\ref{flat-inverse-image-preserves-becker-coacyclicity},
the cones of the morphisms $j_\alpha^*\M^\bu\rarrow j_\alpha^*\J^\bu$
are coacyclic as well.
 If the complexes $j_\alpha^*\M^\bu$ are coacyclic, then it follows
that so are the complexes $j_\alpha^*\J^\bu$.

 On a Noetherian scheme, injectivity of quasi-coherent sheaves is
a local property; so $j_\alpha^*\J^\bu$ is a complex of injective
quasi-coherent sheaves on~$U_\alpha$.
 Any coacyclic complex of injectives is contractible; so
$j_\alpha^*\J^\bu$ is a contractible complex.
 Finally, a complex of injective objects is contractible if and only
if it is acyclic and its objects of cocycles are injective.
 As injectivity of sheaves and acyclicity of complexes are local
properties on $X$, we can conclude that the complex $\J^\bu\in
\sC(X\qcoh_\inj)$ is contractible.
 It follows that the complex $\M^\bu\in\sC(X\qcoh)$ is coacyclic.
\end{proof}

\begin{prop} \label{locality-of-coacyclicity-prop}
 Let $X=\bigcup_\alpha U_\alpha$ be an open covering of a quasi-compact
semi-separated scheme~$X$.
 Let $j_\alpha\:U_\alpha\rarrow X$ denote the open immersion morphisms.
 Then a complex of quasi-coherent sheaves $\M^\bu$ on $X$ is Becker
coacyclic if and only if, for every~$\alpha$, the complex of
quasi-coherent sheaves $j_\alpha^*\M^\bu$ on $U_\alpha$ is Becker
coacyclic. 
\end{prop}

\begin{proof}
 Similarly to the previous lemma, the functors of restriction to open
subschemes preserve coacyclicity by
Lemma~\ref{flat-inverse-image-preserves-becker-coacyclicity}.
 So the ``only if'' assertion is obvious.
 Moreover, refining the covering if necessary and using
the quasi-compactness, we can assume that $X=\bigcup_\alpha U_\alpha$
is a finite affine open covering of~$X$.
 
 Choose a linear order on the set of indices~$\alpha$, and for any
subset of indices $\alpha_1<\dotsb<\alpha_k$ denote by
$j_{\alpha_1,\dotsc,\alpha_k}\:\bigcap_{s=1}^k U_{\alpha_k}\rarrow X$
the open immersion of the intersection of the open subschemes
$U_1$,~\dots, $U_k$ in~$X$.
 Since the scheme $X$ is semi-separated, the open subscheme 
$\bigcap_{s=1}^k U_{\alpha_k}\subset X$ is affine for all $k>0$,
and the open immersion of any affine open subscheme into $X$ is
an affine morphism of schemes; so
the morphism~$j_{\alpha_1,\dotsc,\alpha_k}$ is affine for all $k\ge0$.

 For any quasi-coherent sheaf $\N$ on $X$, the \v Cech complex
$$
 0\lrarrow\N\lrarrow\bigoplus\nolimits_\alpha j_\alpha{}_*j_\alpha^*\N
 \lrarrow\bigoplus\nolimits_{\alpha<\beta}
 j_{\alpha,\beta}{}_{\,*}\,j_{\alpha,\beta}^*\N\lrarrow\dotsb\lrarrow0
$$
is an acyclic finite complex of quasi-coherent sheaves on~$X$
(to show the acyclicity, one observes that the restriction of
the \v Cech complex to each of the open subschemes $U_\alpha\subset X$
is contractible).
 Therefore, given a complex of quasi-coherent sheaves $\M^\bu$ on
$X$, we have an acyclic finite complex of complexes
\begin{equation} \label{cech-bicomplex}
 0\lrarrow\M^\bu\lrarrow
 \bigoplus\nolimits_\alpha j_\alpha{}_*j_\alpha^*\M^\bu
 \lrarrow\bigoplus\nolimits_{\alpha<\beta}
 j_{\alpha,\beta}{}_{\,*}\,j_{\alpha,\beta}^*\M^\bu
 \lrarrow\dotsb\lrarrow0.
\end{equation}

 Now, if the complex $j_\alpha^*\M^\bu$ is coacyclic for every~$\alpha$,
then the complex $j_{\alpha_1,\dotsc,\alpha_k}^*\M^\bu$ is coacyclic
for all~$\alpha_1<\dotsb<\alpha_k$ with $k>0$ (as the restriction to
an open subscheme preserves coacyclicity).
 Then, by Lemma~\ref{affine-direct-image-preserves-becker-coacyclicity},
the complex $j_{\alpha_1,\dotsc,\alpha_k}{}_*\,
j_{\alpha_1,\dotsc,\alpha_k}^*\M^\bu$ is a coacyclic complex of
quasi-coherent sheaves on~$X$.

 Finally, the total complex of the bicomplex~\eqref{cech-bicomplex}
is coacyclic by Lemma~\ref{becker-coacyclic-lemma}(a).
 Since the Becker coacyclic complexes form a full triangulated
subcategory in $\sK(X\qcoh)$, we can conclude that the complex
$\M^\bu$ is Becker coacyclic.
\end{proof}

\begin{rem} \label{locality-of-coacyclicity-also-true-for-pco}
 All the results of this Section~\ref{locality-of-coacyclicity-subsecn}
are also valid for the coacyclicity in the sense
of Section~\ref{coderived-subsecn} in lieu of the coacyclicity in
the sense of Becker (cf.\ Remark~\ref{also-true-for-pco}).
 The advantage of Becker's coderived categories for the purposes of
the present appendix is that Theorem~\ref{becker-coderived-theorem} is
available for them (specifically, in application to the categories
$\bY\tors$ for non-ind-Noetherian ind-schemes~$\bY$).
\end{rem}

\subsection{The semiderived category for a nonaffine morphism
of schemes} \label{semiderived-nonaffine-scheme-morphism-subsecn}
 In this section, unlike in the rest of the paper, we do \emph{not}
automatically assume all the schemes to be concentrated
(i.~e., quasi-compact and semi-separated).
 We start with a series of lemmas before proceeding to the key
definition.

\begin{lem} \label{countable-direct-limit-preserves-coacyclicity}
 Let\/ $\sA$ be an abelian category with countable coproducts and
enough injective objects, and let $M_0^\bu\rarrow M_1^\bu\rarrow
M_2^\bu\rarrow\dotsb$ be an inductive system of complexes in\/ $\sA$,
indexed by the poset of nonnegative integers.
 Assume that the complex $M_n^\bu$ is Becker coacyclic in\/ $\sA$
for every $n\ge0$.
 Then the complex\/ $\varinjlim_{n\ge0}M_n^\bu$ is Becker coacyclic
in\/ $\sA$ as well.
\end{lem}

\begin{proof}
 The complex $\coprod_{n\ge0}M_n^\bu$ is Becker coacyclic in $\sA$
by Lemma~\ref{becker-coacyclic-lemma}(b).
 The total complex of the short exact sequence of complexes
$0\rarrow\coprod_{n\ge0}M_n^\bu\rarrow\coprod_{n\ge0}M_n^\bu
\rarrow\varinjlim_{n\ge0}M_n^\bu$ is Becker coacyclic by
Lemma~\ref{becker-coacyclic-lemma}(a).
 Since the Becker coacyclic complexes form a full triangulated
subcategory in $\sK(\sA)$, it follows that the complex
$\varinjlim_{n\ge0}M_n^\bu$ is also Becker coacyclic.
\end{proof}

 Let $X=\Spec R$ be an affine scheme.
 The \emph{principal affine open subschemes} in $X$ are the open
subschemes $\Spec R[r^{-1}]\subset R$ corresponding to the morphisms
$R\rarrow R[r^{-1}]$ of localization by multiplicative subsets
generated by a single element in~$R$.
 The principal affine open subschemes form a base of neighborhoods
of zero in $\Spec R$, and the intersection of any two principal
affine open subschemes is a principal affine open subscheme,
$\Spec R[r_1^{-1}]\times_{\Spec R}\Spec R[r_2^{-1}]=
\Spec R[(r_1r_2)^{-1}]$.
 Furthermore, for any morphism of affine schemes $\Spec S\rarrow
\Spec R$, the full preimage of any principal affine open subscheme
in $\Spec R$ is a principal affine open subscheme in $\Spec S$.

\begin{lem} \label{inverting-element-upstairs-preserves-semiacyclicity}
 Let $R\rarrow S$ be a morphism of commutative rings and $N^\bu$
be a complex of $S$\+modules.
 Let $s\in S$ be an element.
 Assume that $N^\bu$ is Becker coacyclic as a complex of $R$\+modules.
 Then $S[s^{-1}]\ot_S N^\bu$ is also Becker coacyclic as a complex of
$R$\+modules.
\end{lem}

\begin{proof}
 The complex of $R$\+modules $S[s^{-1}]\ot_S N^\bu$ can be described
as the direct limit of the sequence of morphisms of complexes of
$R$\+modules $N^\bu\overset s\rarrow N^\bu\overset s\rarrow N^\bu
\rarrow\dotsb$, so
Lemma~\ref{countable-direct-limit-preserves-coacyclicity}
is applicable.
\end{proof}

\begin{lem} \label{affine-covered-by-principal-affines-to-affine}
 Let\/ $\bnY=\bigcup_\alpha\bnV_\alpha$ be an affine scheme covered by
principal affine open subschemes.
 Denote by $j_\alpha\:\bnV_\alpha\rarrow\bnY$ the open immersion
morphisms.
 Let $f\:\bnY\rarrow X$ be a morphism of affine schemes, and
let $\bcN^\bu$ be a complex of quasi-coherent sheaves on\/~$\bnY$.
 Then the complex $f_*\bcN^\bu$ of quasi-coherent sheaves on $X$ is
Becker coacyclic if and only if, for every~$\alpha$, the complex
$f_*j_\alpha{}_*j_\alpha^*\bcN^\bu$ of quasi-coherent sheaves on $X$
is Becker coacyclic.
\end{lem}

\begin{proof}
 If the complex $f_*\bcN^\bu$ is Becker coacyclic in $X\qcoh$, then
Lemma~\ref{inverting-element-upstairs-preserves-semiacyclicity} tells
that the complex $f_*j_\alpha{}_*j_\alpha^*\bcN^\bu$ is Becker
coacyclic in $X\qcoh$ for every~$\alpha$.

 Conversely, using quasi-compactness of the affine scheme~$\bnY$,
we can assume that the set of indices~$\alpha$ is finite.
 Choose a linear order on these indices, and for any subset of
indices $\alpha_1<\dotsb<\alpha_k$ denote by
$j_{\alpha_1,\dotsc,\alpha_k}\:\bigcap_{s=1}^k\bnV_{\alpha_k}\rarrow 
\bnY$ the open immersion.
 Assume that the complex $f_*j_\alpha{}_*j_\alpha^*\bcN^\bu$ is
Becker coacyclic in $X\qcoh$ for every~$\alpha$.
 Then Lemma~\ref{inverting-element-upstairs-preserves-semiacyclicity}
tells that the complex $f_*j_{\alpha_1,\dotsc,\alpha_k}{}_*\,
j_{\alpha_1,\dotsc,\alpha_k}^*\bcN^\bu$ is coacyclic in $X\qcoh$ for
all $\alpha_1<\dotsc<\alpha_k$ and $k>0$
(since $\bigcap_{s=1}^k\bnV_{\alpha_k}$ is a principal affine open
subscheme in~$\bnV_{\alpha_1}$).

 Consider the \v Cech complex of complexes~\eqref{cech-bicomplex}
for the complex of quasi-coherent sheaves $\bcN^\bu$ on the affine
scheme $\bnY$ with its open covering by the affine open subschemes
$\bnV_\alpha\subset \bnY$:
\begin{equation} \label{cech-bicomplex-on-Y}
 0\lrarrow\bcN^\bu\lrarrow
 \bigoplus\nolimits_\alpha j_\alpha{}_*j_\alpha^*\bcN^\bu
 \lrarrow\bigoplus\nolimits_{\alpha<\beta}
 j_{\alpha,\beta}{}_{\,*}\,j_{\alpha,\beta}^*\bcN^\bu
 \lrarrow\dotsb\lrarrow0.
\end{equation}
 The direct image functor~$f_*$ for the morphism of affine schemes
$f\:\bnY\rarrow X$ is exact, so it takes the exact complex of
complexes~\eqref{cech-bicomplex-on-Y} in $\bnY\qcoh$ to an exact
complex of complexes in $X\qcoh$:
\begin{equation} \label{direct-image-of-cech-bicomplex}
 0\rarrow f_*\bcN^\bu\rarrow
 \bigoplus\nolimits_\alpha f_*j_\alpha{}_*j_\alpha^*\bcN^\bu
 \rarrow\bigoplus\nolimits_{\alpha<\beta}
 f_*j_{\alpha,\beta}{}_{\,*}\,j_{\alpha,\beta}^*\bcN^\bu
 \rarrow\dotsb\rarrow0.
\end{equation}

 Finally, the total complex of
the bicomplex~\eqref{direct-image-of-cech-bicomplex} is Becker
coacyclic in $X\qcoh$ by Lemma~\ref{becker-coacyclic-lemma}(a),
and we have seen that all the terms of this complex of complexes
except possibly the leftmost one are Becker coacyclic.
 It follows that the leftmost term $f_*\bcN^\bu$ is Becker coacyclic
in $X\qcoh$, too.
\end{proof}

\begin{lem} \label{affine-covered-by-affines-to-affine}
 Let\/ $\bnY=\bigcup_\alpha\bnV_\alpha$ be an affine scheme covered by
(not necessarily principal) affine open subschemes.
 Denote by $j_\alpha\:\bnV_\alpha\rarrow\bnY$ the open immersion morphisms.
 Let $f\:\bnY\rarrow X$ be a morphism of affine schemes, and
let $\bcN^\bu$ be a complex of quasi-coherent sheaves on\/~$\bnY$.
 Then the complex $f_*\bcN^\bu$ of quasi-coherent sheaves on $X$ is
Becker coacyclic if and only if, for every~$\alpha$, the complex
$f_*j_\alpha{}_*j_\alpha^*\bcN^\bu$ of quasi-coherent sheaves on $X$
is Becker coacyclic.
\end{lem}

\begin{proof}
 As the principal affine open subschemes form a base of the topology
of~$\bnY$, every affine open subscheme $\bnV_\alpha\subset\bnY$ in our
covering can be represented as a (finite) union
$\bnV_\alpha=\bigcup_\theta\bnW_{\alpha,\theta}$ of some principal
affine open subschemes $\bnW_{\alpha,\theta}\subset\bnY$.
 If $\bnW_{\alpha,\theta}\subset\bnV_\alpha\subset\bnY$ are affine
open subschemes in an affine scheme and $\bnW_{\alpha,\theta}$ is
a principal affine open subscheme in $\bnY$, then $\bnW_{\alpha,\theta}$
is also a principal affine open subscheme in~$\bnV_\alpha$.

 Denote by $k_{\alpha,\theta}\:\bnW_{\alpha,\theta}\rarrow\bnV_\alpha$
the open immersion morphisms, and put $l_{\alpha,\theta}=
j_\alpha k_{\alpha,\theta}\:\bnW_{\alpha,\theta}\rarrow\bnY$.
 Now if the complex $f_*\bcN$ is Becker coacyclic in $X\qcoh$, then
the complexes $f_*l_{\alpha,\theta}{}_{\,*}\,l_{\alpha,\theta}^*\bcN^\bu$
are Becker coacyclic in $X\qcoh$ by
Lemma~\ref{affine-covered-by-principal-affines-to-affine} (as
$\bnW_{\alpha,\theta}$ are principal affine open subschemes in~$\bnY$).
 Since $\bnV_\alpha=\bigcup_\theta\bnW_{\alpha,\theta}$ is a covering
of an affine scheme by its principle affine open subschemes,
Lemma~\ref{affine-covered-by-principal-affines-to-affine} implies
that the complex $f_*j_\alpha{}_*j_\alpha^*\bcN^\bu$ is Becker
coacyclic in $X\qcoh$.

 Conversely, if the complexes $f_*j_\alpha{}_*j_\alpha^*\bcN^\bu$
are Becker coacyclic in $X\qcoh$, then so are the complexes
$f_*l_{\alpha,\theta}{}_{\,*}\,l_{\alpha,\theta}^*\bcN^\bu$
(by the same Lemma~\ref{affine-covered-by-principal-affines-to-affine},
as $\bnW_{\alpha,\theta}$ are principal affine open subschemes
in~$\bnV_\alpha$).
 Since $\bnY=\bigcup_{\alpha,\theta}\bnW_{\alpha,\theta}$ is a covering
of an affine scheme by its principal affine open subschemes, yet another
application of Lemma~\ref{affine-covered-by-principal-affines-to-affine}
allows to conclude that the complex $f_*\bcN^\bu$ is Becker coacyclic
in $X\qcoh$.
\end{proof}

\begin{lem} \label{two-coverings-by-affines-to-affine}
 Let\/ $\bnY$ be a scheme, and let\/ $\bigcup_\beta \bnV_\beta=\bnY=
\bigcup_\gamma\bnW_\gamma$ be two affine open coverings of\/~$\bnY$.
 Denote by $j_\beta\:\bnV_\beta\rarrow\bnY$ and $k_\gamma\:\bnW_\gamma
\rarrow\bnY$ the open immersion morphisms.
 Let $X$ be an affine scheme, $f\:\bnY\rarrow X$ be a morphism of
schemes, and let $\bcN^\bu$ be a complex of quasi-coherent sheaves
on\/~$\bnY$.
 Then the complexes of quasi-coherent sheaves $f_*j_\beta{}_*
j_\beta^*\bcN^\bu$ on $X$ are Becker coacyclic for all~$\beta$
if and only if the complexes of quasi-coherent sheaves
$f_*k_\gamma{}_*k_\gamma^*\bcN^\bu$ on $X$ are Becker coacycl
for all~$\gamma$.
\end{lem}

\begin{proof}
 For every pair of indices $\beta$ and~$\gamma$ choose an affine
open covering $\bnV_\beta\cap\bnW_\gamma =\bigcup_\theta
\bnU_{\beta,\gamma,\theta}$ of the open subscheme
$\bnV_\beta\cap\bnW_\gamma\subset\bnY$.
 Then $X=\bigcup_{\beta,\gamma,\theta}\bnU_{\beta,\gamma,\theta}$ is
an affine open covering of the scheme $X$, for every $\beta$,
$\gamma$, $\theta$ one has $\bnU_{\beta,\gamma,\theta}\subset
\bnV_\beta$ and $\bnU_{\beta,\gamma,\theta}\subset\bnW_\gamma$,
for every~$\beta$ one has $\bnV_\beta=\bigcup_{\gamma,\theta}
\bnU_{\beta,\gamma,\theta}$, and for every~$\gamma$ one has
$\bnW_\gamma=\bigcup_{\beta,\theta}\bnU_{\beta,\gamma,\theta}$.

 Denote by $l_{\beta,\gamma,\theta}\:\bnU_{\beta,\gamma,\theta}\rarrow
\bnY$ the open immersion morphisms.
 Assume that the complexes $f_*j_\beta{}_*j_\beta^*\bcN^\bu$ are Becker
coacyclic in $X\qcoh$ for all~$\beta$.
 Then, by Lemma~\ref{affine-covered-by-affines-to-affine} applied
to the affine open subscheme $\bnU_{\beta,\gamma,\theta}$ in
the affine scheme $\bnV_\beta$, the complexes
$f_*l_{\beta,\gamma,\theta}{}_{\,*}\,l_{\beta,\gamma,\theta}^*\bcN^\bu$
are coacyclic in $X\qcoh$ for all~$\beta$, $\gamma$,~$\theta$.
 By the same Lemma~\ref{affine-covered-by-affines-to-affine} applied
to the affine open covering $\bnW_\gamma=\bigcup_{\beta,\theta}
\bnU_{\beta,\gamma,\theta}$ of the affine scheme $\bnW_\gamma$, it follows
that the complexes $f_*k_\gamma{}_*k_\gamma^*\bcN^\bu$ are Becker
coacyclic in $X\qcoh$ for all~$\gamma$.
\end{proof}

 Let $X$ be a semi-separated scheme.
 Then, for any affine scheme $U$, any morphism of schemes $U\rarrow X$
is affine.
 For any two affine schemes $U$ and $\bnV$ and any morphisms of schemes
$U\rarrow X$ and $\bnV\rarrow X$, the scheme $U\times_X\bnV$ is affine.

\begin{lem} \label{affine-morphism-affine-covering-of-the-base}
 Let $f\:\bnY\rarrow X$ be an affine morphism of quasi-compact
semi-separated schemes.
 Let $X=\bigcup_\alpha U_\alpha$ be an affine open covering of~$X$.
 Put\/ $\bnV_\alpha=U_\alpha\times_X\bnY$; then\/ $\bnY=
\bigcup_\alpha\bnV_\alpha$ is an affine open covering of\/~$\bnY$.
 Consider the pullback diagram
$$
\xymatrix{
 \bnV_\alpha \ar[r]^-{k_\alpha} \ar[d]_-{f_\alpha} & \bnY \ar[d]^-f \\
 U_\alpha \ar[r]^-{j_\alpha} & X
}
$$
 Let $\bcN^\bu$ be a complex of quasi-coherent sheaves on\/~$\bnY$.
 Then the following conditions are equivalent:
\begin{enumerate}
\renewcommand{\theenumi}{\alph{enumi}}
\item the complex $f_*\bcN^\bu$ is Becker coacyclic in $X\qcoh$;
\item the complexes $f_*k_\alpha{}_*k_\alpha^*\bcN^\bu$ are
Becker coacyclic in $X\qcoh$ for all~$\alpha$;
\item the complexes $f_\alpha{}_*k_\alpha^*\bcN^\bu$ are Becker
coacyclic in $U_\alpha\qcoh$ for all~$\alpha$.
\end{enumerate}
\end{lem}

\begin{proof}
 (a)\,$\Longleftrightarrow$\,(c) In view of the natural isomorphism
$j_\alpha^*f_*\bcN^\bu\simeq f_\alpha{}_*k_\alpha^*\bcN^\bu$ of
complexes of quasi-coherent sheaves on~$U_\alpha$, the assertion
follows from Proposition~\ref{locality-of-coacyclicity-prop}.

 (b)\,$\Longleftrightarrow$\,(c) We have
$f_*k_\alpha{}_*k_\alpha^*\bcN^\bu\simeq
j_\alpha{}_*f_\alpha{}_*k_\alpha^*\bcN^\bu$, since $fk_\alpha=
j_\alpha f_\alpha$.
 It remains to observe that, for any affine open immersion
$j\:U\rarrow X$, a complex of quasi-coherent sheaves $\M^\bu$ on $U$
is Becker coacyclic if and only if the complex of quasi-coherent
sheaves $j_*\M^\bu$ on $X$ is Becker coacyclic.
 This follows from
Lemmas~\ref{flat-inverse-image-preserves-becker-coacyclicity}
and~\ref{affine-direct-image-preserves-becker-coacyclicity}.
\end{proof}

\begin{prop} \label{scheme-morphism-relatively-affine-covering-prop}
 Let $X$ be a quasi-compact semi-separated scheme, and let
$f\:\bnY\rarrow X$ be a morphism of schemes.
 Let $X=\bigcup_\alpha U_\alpha$ be an affine open covering of $X$,
and let\/ $\bigcup_\beta\bnV_\beta=\bnY=\bigcup_\gamma\bnW_\gamma$ be
two open coverings of\/ $\bnY$ such that the compositions\/
$\bnV_\beta\rarrow\bnY\rarrow X$ and\/ $\bnW_\gamma\rarrow\bnY\rarrow X$
are affine morphisms.

 Denote by $j_\alpha\:U_\alpha\rarrow X$, \ $k_\beta\:\bnV_\beta\rarrow
\bnY$, and $l_\gamma\:\bnW_\gamma\rarrow\bnY$ the open immersion morphisms.
 Furthermore, put\/ $\bnS_{\alpha,\beta}=U_\alpha\times_X\bnV_\beta$
and\/ $\bnT_{\alpha,\gamma}=U_\alpha\times_X\bnW_\gamma$, and
denote by $g_{\alpha,\beta}\:\bnS_{\alpha,\beta}\rarrow\bnY$ and
$h_{\alpha,\gamma}\:\bnT_{\alpha,\gamma}\rarrow\bnY$ the natural
open immersions, and by $f'_{\alpha,\beta}\:\bnS_{\alpha,\beta}\rarrow
U_\alpha$ and $f''_{\alpha,\gamma}\:\bnT_{\alpha,\gamma}\rarrow
U_\alpha$ the natural morphisms of affine schemes.

 Let $\bcN^\bu$ be a complex of quasi-coherent sheaves on\/~$\bnY$.
 Then the following conditions are equivalent:
\begin{enumerate}
\renewcommand{\theenumi}{\alph{enumi}}
\item the complexes $f_*k_\beta{}_*k_\beta^*\bcN^\bu$ are Becker coacyclic
in $X\qcoh$ for all~$\beta$;
\item the complexes $f_*g_{\alpha,\beta}{}_{\,*}\,
g_{\alpha,\beta}^*\bcN^\bu$ are Becker coacyclic in $X\qcoh$ for
all~$\alpha$ and~$\beta$;
\item the complexes $f'_{\alpha,\beta}{}_{\,*}\,
g_{\alpha,\beta}^*\bcN^\bu$ are Becker coacyclic in $U_\alpha\qcoh$ for
all~$\alpha$ and~$\beta$;
\item the complexes $f''_{\alpha,\gamma}{}_{\,*}\,
h_{\alpha,\gamma}^*\bcN^\bu$ are Becker coacyclic in $U_\alpha\qcoh$ for
all~$\alpha$ and~$\gamma$;
\item the complexes $f_*h_{\alpha,\gamma}{}_{\,*}\,
h_{\alpha,\gamma}^*\bcN^\bu$ are Becker coacyclic in $X\qcoh$ for
all~$\alpha$ and~$\gamma$;
\item the complexes $f_*l_\gamma{}_*l_\gamma^*\bcN^\bu$ are Becker
coacyclic in $X\qcoh$ for all~$\gamma$.
\end{enumerate}
\end{prop}

\begin{proof}
 Conditions~(a\+-c) are equivalent to each other by
Lemma~\ref{affine-morphism-affine-covering-of-the-base} applied
to the affine morphism of schemes $f k_\beta\:\bnV_\beta\rarrow X$ and
the complex of quasi-coherent sheaves $k_\beta^*\bcN^\bu$
on~$\bnV_\beta$.
 Similarly, conditions~(d\+-f) are equivalent to each other by
Lemma~\ref{affine-morphism-affine-covering-of-the-base} applied
to the affine morphism of schemes $f l_\gamma\:\bnW_\gamma\rarrow X$
and the complex of quasi-coherent sheaves $l_\gamma^*\bcN^\bu$
on~$\bnW_\gamma$.
 Finally, the equivalence (c)\,$\Longleftrightarrow$\,(d) is provided
by Lemma~\ref{two-coverings-by-affines-to-affine} applied to
the morphism of schemes $U_\alpha\times_X\bnY\rarrow U_\alpha$
into the affine scheme $U_\alpha$, the restriction of the complex
of quasi-coherent sheaves $\bcN^\bu$ on $\bnY$ to the open subscheme
$U_\alpha\times_X\bnY\subset\bnY$, and the two affine coverings
$\bigcup_\beta\bnS_{\alpha,\beta}=U_\alpha\times_X\bnY=
\bigcup_\gamma\bnT_{\alpha,\gamma}$ of the scheme
$U_\alpha\times_X\bnY$.
\end{proof}

 Now we can formulate the promised definition.
 Let $f\:\bnY\rarrow X$ be a morphism of schemes; assume that
the scheme $X$ is quasi-compact and semi-separated.
 Let $\bcN^\bu$ be a complex of quasi-coherent sheaves on~$\bnY$.

 Choose a covering of the scheme $\bnY$ by open subschemes
$\bnV_\beta\subset\bnY$ such that the compositions $\bnV_\beta\rarrow
\bnY\rarrow X$ are affine morphisms of schemes (for example, any
affine open covering of $\bnY$ satisfies this condition).
 Denote by $k_\beta\:\bnV_\beta\rarrow\bnY$ the open immersion
morphisms.

 We will say that the complex $\bcN^\bu$ is \emph{semiacyclic}
(or more precisely \emph{$\bnY/X$\+semi\-acyclic}) if, for every
index~$\beta$, the complex $f_*k_\beta{}_*k_\beta^*\bcN^\bu$ of
quasi-coherent sheaves on $X$ (that is, the direct image to $X$
of the restriction of $\bcN^\bu$ to~$\bnV_\beta$) is Becker
coacyclic in $X\qcoh$.
 According to
Proposition~\ref{scheme-morphism-relatively-affine-covering-prop}%
\,(a)\,$\Leftrightarrow$\,(f), this property does not depend on
the choice of an open covering $\bnY=\bigcup_\beta\bnV_\beta$.

\begin{rem} \label{semiacyclicity-of-qcoh-stronger-than-acylicity}
 The reader should be \emph{warned} that our terminology is misleading.
 The semiacyclicity of a complex in $\bnY\qcoh$ is by design
an intermediate property between the acyclicity and the Becker
coacyclicity.
 Any Becker coacyclic complex in $\bnY\qcoh$ is $\bnY/X$\+semiacyclic
(by Lemmas~\ref{flat-inverse-image-preserves-becker-coacyclicity}
and~\ref{affine-direct-image-preserves-becker-coacyclicity}).

 On the other hand, any $\bnY/X$\+semiacyclic complex $\bcN^\bu$ in
$\bnY\qcoh$ is acyclic.
 Iindeed, it suffices to check that the restriction of $\bcN^\bu$ to
$\bnV_\beta$ is acyclic for every~$\beta$.
 Notice that any Becker coacyclic complex in $X\qcoh$ is acyclic
by Lemma~\ref{becker-coacyclic-acyclic}.
 Now acyclicity of the complex $f_*k_\beta{}_*k_\beta^*\bcN^\bu$
in $X\qcoh$ implies acyclicity of the complex $k_\beta^*\bcN^\bu$ in
$\bnV_\beta\qcoh$, since the direct image functor
$f_*k_\beta{}_*=(fk_\beta)_*\:\bnV_\beta\qcoh\rarrow X\qcoh$ for
an affine morphism of schemes $fk_\beta\:\bnV_\beta\rarrow X$ is
exact and faithful.

 So the semiacyclicity is a \emph{stronger} property than
the acyclicity.
\end{rem}

\begin{rem} \label{semiderived-nonaffine-of-schemes-also-true-for-pco}
 Similarly to Section~\ref{locality-of-coacyclicity-subsecn}
(cf.\ Remark~\ref{locality-of-coacyclicity-also-true-for-pco})
all the results of this
Section~\ref{semiderived-nonaffine-scheme-morphism-subsecn} are
equally valid for the coacyclicity in the sense of
Section~\ref{coderived-subsecn} in lieu of coacyclicity in
the sense of Becker.
 Moreover, when the scheme $X$ is Noetherian, the two coacyclicity
notions involved are equivalent to each other by
Proposition~\ref{two-coderived-categories-comparison}.
\end{rem}

 Finally, we can define the \emph{semiderived category} (or
the \emph{$\bnY/X$\+semiderived category}) $\sD_X^\si(\bnY\qcoh)$ of
quasi-coherent sheaves on $\bnY$ as the triangulated quotient category
of the homotopy category $\sK(\bnY\qcoh)$ by the thick subcategory
of $\bnY/X$\+semiacyclic complexes.
 In view of the previous remark, this definition agrees with
the definition in Section~\ref{semiderived-subsecn}
\emph{assuming that $X$ is an semi-separated Noetherian scheme}.

 Indeed, the definition in Section~\ref{semiderived-subsecn}
(specialized from ind-schemes to schemes) presumes
the morphism $f\:\bnY\rarrow X$ to be affine.
 In this case, one can use the open covering of $\bnY$ consisting
of a single open subscheme $\bnV=\bnY$ for the purposes of
the definition above in this section.
 Then it is clear that the two definitions are the same.

\subsection{Direct images of restrictions of injective sheaves}
 The aim of this section is to prove the following technical
lemma.

\begin{lem} \label{direct-image-of-restriction-injective-lem}
 Let $X$ be a Noetherian scheme, $Y$ be a quasi-compact semi-separated
scheme, and $f\:Y\rarrow X$ be a flat morphism of schemes.
 Let $V\subset Y$ be an open subscheme with the open immersion
morphism $k\:V\rarrow Y$.
 Assume that the composition $fk\:V\rarrow X$ is an affine morphism
of schemes.
 Let $\J$ be an injective quasi-coherent sheaf on~$Y$.
 Then the quasi-coherent sheaf $f_*k_*k^*\J$ on $X$ is injective.
\end{lem}

 We will deduce Lemma~\ref{direct-image-of-restriction-injective-lem}
from the next proposition.

\begin{prop} \label{direct-image-of-flat-injective-tensor-prop}
 Let $X$ be a Noetherian scheme, $Y$ be a quasi-compact semi-separated
scheme, and $f\:Y\rarrow X$ be a flat morphism of schemes.
 Let $\J$ be an injective quasi-coherent sheaf on $Y$ and $\F$ be
a flat quasi-coherent sheaf on~$Y$.
 Then the quasi-coherent sheaf $f_*(\F\ot_{\cO_Y}\J)$ on $X$ is
injective.
\end{prop}

 The following particular cases of
Proposition~\ref{direct-image-of-flat-injective-tensor-prop}
are easy.
 If $\F=\cO_Y$, then the assertion of the proposition holds because
the direct image functor~$f_*$ for a flat morphism of schemes
$f\:Y\rarrow X$ preserves injectivity of quasi-coherent sheaves.
 In this case, there is no need to assume that the scheme $X$ is
Noetherian.
 If $X=Y$ and $f=\id$ is the identity morphism, then the result
reduces to Lemma~\ref{hom-tensor-flats-injectives}(b) (for which
the Noetherianity assumption is essential).

\begin{lem} \label{affine-direct-image-of-flat-injective-tensor}
 Let $f\:Y\rarrow X$ a flat morphism of affine schemes, where
the affine scheme $X$ in Noetherian.
 Let $\J$ be an injective quasi-coherent sheaf on $Y$ and $\F$ be
a flat quasi-coherent sheaf on~$Y$.
 Then the quasi-coherent sheaf $f_*(\F\ot_{\cO_Y}\J)$ on $X$ is
injective. 
\end{lem}

\begin{proof}
 In algebraic language, the assertion means the following.
 Let $R\rarrow S$ be a homomorphism of commutative rings such that
$S$ is a flat $R$\+module.
 Assume that the ring $R$ is Noetherian.
 Let $J$ be an injective $S$\+module and $F$ be a flat $S$\+module.
 Then the $R$\+module $F\ot_SJ$ is injective.

 Indeed, it suffices to observe that $F$ is a (filtered) direct limit of
finitely generated free $S$\+modules, $J$ is an injective $R$\+module
(since $S$ is a flat $R$\+module and $J$ is an injective $S$\+module),
and the class of all injective $R$\+modules is closed under direct
limits in $R\modl$ (since $R$ is a Noetherian ring).
\end{proof}

\begin{lem} \label{through-affine-direct-image-of-tensor}
 Let $X$ be a Noetherian scheme, $Y$ be an affine scheme, and
$f\:Y\rarrow X$ be a flat morphism of schemes.
 Assume that there exists an affine open subscheme $U\subset X$
such that the morphism~$f$ factorizes as $Y\rarrow U\rarrow X$.
 Let $\J$ be an injective quasi-coherent sheaf on $Y$ and $\F$ be
a flat quasi-coherent sheaf on~$Y$.
 Then the quasi-coherent sheaf $f_*(\F\ot_{\cO_Y}\J)$ on $X$ is
injective.
\end{lem}

\begin{proof}
 Denote the morphisms involved by $g\:Y\rarrow U$ and $h\:U\rarrow X$.
 Applying Lemma~\ref{affine-direct-image-of-flat-injective-tensor} to
the flat morphism of affine schemes $g\:Y\rarrow U$  with a Noetherian
scheme $U$, we see that the quasi-coherent sheaf
$g_*(\F\ot_{\cO_Y}\J)$ on $U$ is injective.
 It remains to say that the direct image functor
$h_*\:U\qcoh\rarrow X\qcoh$ preserves injectivity,
as an open immersion~$h$ is a flat morphism.
\end{proof}

\begin{proof}[Proof of
Proposition~\ref{direct-image-of-flat-injective-tensor-prop}]
 Let $Y=\bigcup_\alpha W_\alpha$ be a finite affine open covering
of the scheme~$Y$.
 Denote by $l_\alpha\:W_\alpha\rarrow Y$ the open immersions.

 The key observation is that any injective quasi-coherent sheaf $\J$
on $Y$ is a direct summand of a direct sum $\bigoplus_\alpha
l_\alpha{}_*\K_\alpha$, where $\K_\alpha$ are some injective
quasi-coherent sheaves on~$W_\alpha$.
 Indeed, one easily observes that there are enough injective
quasi-coherent sheaves of this particular form, that is, any
quasi-coherent sheaf $\M$ on $Y$ is a subobject of a quasi-coherent
sheaf of the form $\bigoplus_\alpha l_\alpha{}_*\K_\alpha$, where
$\K_\alpha\in W_\alpha\qcoh_\inj$.
 (It suffices to choose injective quasi-coherent sheaves $\K_\alpha$
in such a way that $l_\alpha^*\M$ is a subobject of~$\K_\alpha$
for every~$\alpha$.)

 As we are free to choose our finite affine open covering
$Y=\bigcup_\alpha W_\alpha$ of the scheme $Y$, we can make
the affine open subschemes $W_\alpha\subset Y$ as small as we wish.
 Specifically, we can assume that for every~$\alpha$ there exists
an affine open subscheme $U_\alpha\subset X$ such that
$f(W_\alpha)\subset U_\alpha$.

 Hence the question reduces to the following.
 We can assume that $\J=l_*\K$, where $l_*\:W\rarrow Y$ is
the immersion of an affine open subscheme and $\K\in W\qcoh_\inj$.
 Moreover, we can have $f(W)\subset U$ for some affine open subscheme
$U\subset X$.
 In this context, we have to prove that, for any flat quasi-coherent
sheaf $\F$ on $Y$, the quasi-coherent sheaf $f_*(\F\ot_{\cO_Y}l_*\K)$
on $X$ is injective.

 By Lemma~\ref{projection-formula}, we have $\F\ot_{\cO_Y}l_*\K
\simeq l_*(l^*\F\ot_{\cO_W}\K)$ in $Y\qcoh$ (as the morphism~$l$ is
affine, since the scheme $Y$ is semi-separated by assumption).
 Hence $f_*(\F\ot_{\cO_Y}l_*\K)\simeq f_*l_*(l^*\F\ot_{\cO_W}\K)$
in $X\qcoh$.

 It remains to apply Lemma~\ref{through-affine-direct-image-of-tensor}
to the flat morphism of schemes $fl\:W\rarrow X$, the injective
quasi-coherent sheaf $\K$ on $W$, and the flat quasi-coherent sheaf
$l^*\F$ on~$W$.
 Here $X$ is a Noetherian scheme, $W$ is an affine scheme, and
the morphism~$fl$ factorizes as $W\rarrow U\rarrow X$ for
an affine open subscheme $U\subset X$.
\end{proof}

\begin{proof}[Proof of
Lemma~\ref{direct-image-of-restriction-injective-lem}]
 We have the assumption that the morphism $fk\:V\rarrow X$ is affine.
 Let us deduce from this assumption that the morphism $k\:V\rarrow Y$
is affine.

 Let $W\subset Y$ be an affine open subscheme.
 We have to show that $W\times_YV$ is an affine scheme.
 Indeed, $W\times_XV$ is an affine scheme, since $W$ is an affine scheme
and $V\rarrow X$ is an affine morphism (so $W\times_XV\rarrow W$ is
an affine morphism as a base change of an affine morphism).
 As $W\times_YV=Y\times_{Y\times Y}(W\times_XV)$ and the morphism
$Y\rarrow Y\times Y=Y\times_{\Spec\boZ}Y$ is affine (the scheme
$Y$ being semi-separated by assumption), it follows that
$W\times_YV$ is an affine scheme.

 Finally, by Lemma~\ref{projection-formula} we have
$k_*k^*\J\simeq k_*(\cO_V\ot_{\cO_V}k^*\J)\simeq
k_*\cO_V\ot_{\cO_Y}\J$ in $Y\qcoh$.
 Hence $f_*k_*k^*\J\simeq f_*(k_*\cO_V\ot_{\cO_Y}\J)$ in $X\qcoh$.
 The quasi-coherent sheaf $\F=k_*\cO_V$ on $Y$ is flat,
as the direct image with respect to a flat affine morphism of schemes
takes flat quasi-coherent sheaves to flat quasi-coherent sheaves.
 Thus Proposition~\ref{direct-image-of-flat-injective-tensor-prop} is
applicable.

 Notice that we have never used the assumption that $V$ is an open
subscheme in~$Y$.
 It suffices that $k\:V\rarrow Y$ be a flat morphism.
\end{proof}

\subsection{The semiderived category for a morphism of ind-schemes}
 The definition of the semiderived category in this section resembles
the ones in~\cite[Section~4.3]{Pweak} (where the context is very
different).
 We start with a consistency lemma involving only schemes.

\begin{lem} \label{nonaffine-semiacyclicity-consistency-lemma}
 Let $i\:Z\rarrow X$ be a closed immersion of semi-separated Noetherian
schemes, and let $f\:\bnY\rarrow X$ be a flat morphism of quasi-compact
semi-separated schemes.
 Consider the pullback diagram (so\/ $\bnW=Z\times_X\bnY$)
$$
\xymatrix{
 \bnW \ar[r]^-k \ar[d]_-g & \bnY \ar[d]^-f \\
 Z \ar[r]^-i & X
}
$$
 Let $\bcJ^\bu$ be a\/ $\bnY/X$\+semiacyclic complex of injective
quasi-coherent sheaves on\/~$\bnY$.
 Then $k^!\bcJ^\bu$ is a\/ $\bnW/Z$\+semiacyclic complex of (injective)
quasi-coherent sheaves on\/~$\bnW$.
\end{lem}

\begin{proof}
 Let $\bnY=\bigcup_\beta\bnV_\beta$ be an open covering of the scheme
$\bnY$ such that the compositions $\bnV_\beta\rarrow\bnY\rarrow X$
are affine morphisms of schemes.
 Put $\bnT_\beta=\bnV_\beta\times_\bnY\bnW=\bnV_\beta\times_XZ$.
 Then $\bnW=\bigcup_\beta\bnT_\beta$ is an open covering of the scheme
$\bnW$ such that the compositions $\bnT_\beta\rarrow\bnW\rarrow Z$
are affine morphisms of schemes.
 Denote by $j_\beta\:\bnV_\beta\rarrow\bnY$ and $l_\beta\:\bnT_\beta
\rarrow\bnW$ the open immersion morphisms.

 By the definition, the condition that $\bcJ^\bu$ is
a $\bnY/X$\+semiacyclic complex of quasi-coherent sheaves on $\bnY$
means that the complex of quasi-coherent sheaves
$f_*j_\beta{}_*j_\beta^*\bcJ^\bu$ on $X$ is Becker coacyclic for
every~$\beta$.
 Since $\bcJ^\bu$ is a complex of injective quasi-coherent sheaves
on~$\bnY$, Lemma~\ref{direct-image-of-restriction-injective-lem}
implies that $f_*j_\beta{}_*j_\beta^*\bcJ^\bu$ is a complex of
injective quasi-coherent sheaves on~$X$.
 Any Becker coacyclic complex of injectives is contractible.
 We have shown that $f_*j_\beta{}_*j_\beta^*\bcJ^\bu$ is a contractible
complex of (injective) quasi-coherent sheaves on $X$ for every~$\beta$.

 Consider two pullback diagrams
$$
\xymatrix{
 \bnT_\beta \ar[r]^-{k_\beta} \ar[d]_-{l_\beta}
 & \bnV_\beta \ar[d]^-{j_\beta} \\
 \bnW \ar[r]^-k & \bnY
}
\qquad
\xymatrix{
 \bnT_\beta \ar[r]^-{k_\beta} \ar[d]_-{gl_\beta}
 & \bnV_\beta \ar[d]^-{fj_\beta} \\
 Z \ar[r]^-i & X
}
$$
 By Lemma~\ref{reasonable-shriek-flat-star-commutation} applied to
the leftmost diagram (taking into account
Lemma~\ref{base-change-composition-reasonable}(a)), we have
$l_\beta^*k^!\bcJ^\bu\simeq k_\beta^!j_\beta^*\bcJ^\bu$ in
$\sC(\bnT_\beta\qcoh)$.
 By Lemma~\ref{reasonable-base-change}(a) applied to the rightmost
diagram, we have $g_*l_\beta{}_*k_\beta^!j_\beta^*\bcJ^\bu\simeq
i^!f_*j_\beta{}_*j_\beta^*\bcJ^\bu$ in $\sC(Z\qcoh)$.
 Combining these isomorphisms together, we obtain
$g_*l_\beta{}_*l_\beta^*k^!\bcJ^\bu\simeq
g_*l_\beta{}_*k_\beta^!j_\beta^*\bcJ^\bu\simeq 
i^!f_*j_\beta{}_*j_\beta^*\bcJ^\bu$ in $\sC(Z\qcoh)$.

 Since $f_*j_\beta{}_*j_\beta^*\bcJ^\bu$ is a contractible complex of
(injective) quasi-coherent sheaves on $X$, the complex
$i^!f_*j_\beta{}_*j_\beta^*\bcJ^\bu$ is a contractible complex of
(injective) quasi-coherent sheaves on~$Z$.
 Thus the complex $g_*l_\beta{}_*l_\beta^*k^!\bcJ^\bu$ is
a contractible, hence Becker coacyclic, complex of (injective)
quasi-coherent sheaves on~$Z$.
 By the definition, this means that $k^!\bcJ^\bu$ is
a $\bnW/Z$\+semiacyclic complex of quasi-coherent sheaves on~$\bnW$.
\end{proof}

 Let $\X$ be an ind-semi-separated ind-Noetherian ind-scheme,
$\bY$ be an (ind-quasi-compact) ind-semi-separated ind-scheme,
and $\pi\:\bY\rarrow\X$ be a flat (but not necessarily affine)
morphism of ind-schemes.
 Our aim is to define the $\bY/\X$\+semiderived category
$\sD_\X^\si(\bY\tors)$ of quasi-coherent torsion sheaves on~$\bnY$.

 Firstly, let $\bJ^\bu$ be a complex of injective quasi-coherent
torsion sheaves on~$\bY$.
 We will say that $\bJ^\bu$ is a \emph{$\bY/\X$\+semiacyclic complex}
if, for every closed subscheme $Z\subset\X$ with the closed immersion
morphism $i\:Z\rarrow\X$ and the related closed subscheme
$\bnW=Z\times_\X\bY\subset\bY$ with the closed immersion morphism
$k\:\bnW\rarrow\bY$, the complex $k^!\bJ^\bu$ of injective
quasi-coherent sheaves on $\bnW$ is $\bnW/Z$\+semiacyclic in the sense
of the definition in
Section~\ref{semiderived-nonaffine-scheme-morphism-subsecn}.

 As we have seen in the proof of
Lemma~\ref{nonaffine-semiacyclicity-consistency-lemma}, this means
that, for every open subscheme $\bnV\subset\bnW$ such that
the composition $\bnV\rarrow\bnW\rarrow Z$ is an affine morphism,
the complex of (injective) quasi-coherent sheaves
$\pi_Z{}_*j_*j^*k^!\bJ^\bu$ on $Z$ must be contractible.
 Here the notation for morphisms is $\pi_Z\:\bnW\rarrow Z$ and
$j\:\bnV\rarrow\bnW$.
 It suffices to check this condition for open subschemes
$\bnV\subset\bnW$ belonging to a chosen covering of $\bnW$ by
open subschemes that are affine over~$Z$.

 Furthermore, let $\X=\ilim_{\gamma\in\Gamma}X_\gamma$ be some
chosen representation of $\X$ by an inductive system of closed
immersions of schemes.
 Then Lemma~\ref{nonaffine-semiacyclicity-consistency-lemma} tells
that it suffices to check the $\bnW/Z$\+semiacyclicity condition
above for the closed subschemes $Z=X_\gamma\subset\X$,
where $\gamma\in\Gamma$.

 Now we want to explain what it means for a complex of not necessarily
injective quasi-coherent torsion sheaves $\bN^\bu$ on $\bY$ to be
$\bY/\X$\+semiacyclic.
 Here we are going to use Theorem~\ref{becker-coderived-theorem}
(so it is important that we are working with the Becker coderived
categories).
 By Theorem~\ref{becker-coderived-theorem}, there exists a morphism
of complexes of quasi-coherent torsion sheaves $\bN^\bu\rarrow\bJ^\bu$
on $\bY$ whose cone is Becker coacyclic in $\bY\tors$, while
$\bJ^\bu$ is a complex of injective quasi-coherent torsion sheaves
on~$\bY$.
 The complex $\bN^\bu$ in $\bY\tors$ is said to be
\emph{$\bY/\X$\+semiacyclic} if the complex $\bJ^\bu$ in
$\bY\tors_\inj$ is $\bY/\X$\+semiacyclic in the sense of the previous
definition (which is applicable to the complexes of injectives only).

 In other words, the full subcategory of $\bY/\X$\+semiacyclic
complexes in $\sK(\bY\tors)$ is, by the definition, the minimal
thick subcategory in $\sK(\bY\tors)$ containing \emph{both}
the Becker coacyclic complexes and the $\bY/\X$\+semiacyclic complexes
of injective quasi-coherent torsion sheaves on~$\bY$.

\begin{rem}
 Similarly to
Remark~\ref{semiacyclicity-of-qcoh-stronger-than-acylicity},
we observe that any $\bY/\X$\+semiacyclic complex in $\bY\tors$
is acyclic.
 Indeed, all the Becker coacyclic complexes in $\bY\tors$ are acyclic
by Lemma~\ref{becker-coacyclic-acyclic}.
 Now let $\bJ^\bu$ be a $\bY/\X$\+semiacyclic complex of injective
quasi-coherent torsion sheaves on~$\bY$.
 Then, by Remark~\ref{semiacyclicity-of-qcoh-stronger-than-acylicity},
the ind-scheme $\bY$ can be represented by
an inductive system of closed immersions of schemes $\bY=
\ilim_{\gamma\in\Gamma}\bnY_\gamma$ such that, denoting by
$k_\gamma\:\bnY_\gamma\rarrow\bY$ the closed immersion morphisms,
the complexes of quasi-coherent sheaves $k_\gamma^!\bJ^\bu$ on
$\bnY_\gamma$ are ($\bnY_\gamma/X_\gamma$\+semiacyclic, hence)
acyclic for all~$\gamma$.
 As the direct image functors $k_\gamma{}_*\:\bnY_\gamma\qcoh\rarrow
\bY\tors$ are exact, it follows that the complex
$\bJ^\bu=\varinjlim_{\gamma\in\Gamma}k_\gamma{}_*k_\gamma^!\bJ^\bu$
is acyclic in $\bY\tors$.
\end{rem}

 The \emph{semiderived category} (or the $\bY/\X$\+semiderived category)
$\sD_\X^\si(\bY\tors)$ of quasi-coherent torsion sheaves on $\bY$ is
defined as the triangulated quotient category of the homotopy category
$\sK(\bY\tors)$ by the thick subcategory of $\bY/\X$\+semiacyclic
complexes.
 The following lemma tells that this definition agrees with the one
in Section~\ref{semiderived-subsecn} when both are applicable.

\begin{lem}
 Let\/ $\X$ be an ind-semi-separated ind-Noetherian ind-scheme and\/
$\pi\:\bY\rarrow\X$ be a flat affine morphism of ind-schemes.
 Then a complex\/ $\bN^\bu$ of quasi-coherent torsion sheaves on\/
$\bY$ is\/ $\bY/\X$\+semiacyclic if and only if the complex\/
$\pi_*\bN^\bu$ of quasi-coherent torsion sheaves on\/ $\X$
is coacyclic.
\end{lem}

\begin{proof}
 First of all, let us mention once again that there is no difference
between the two notions of coacyclicity for complexes of quasi-coherent
torsion sheaves on an ind-Noetherian ind-scheme $\X$ (by
Proposition~\ref{two-coderived-categories-comparison}).
 Furthermore, the direct image functor~$\pi_*$ takes Becker coacyclic
complexes in $\bY\tors$ to coacyclic complexes in $\X\tors$ by
Lemma~\ref{affine-direct-image-preserves-becker-coacyclicity}.
 Hence it suffices to consider the case of a complex of injective
quasi-coherent torsion sheaves $\bJ^\bu\in\sC(\bY\tors_\inj)$.

 Let $\X=\ilim_{\gamma\in\Gamma}X_\gamma$ be a representation of $\X$
by an inductive system of closed immersions of schemes.
 Put $\bnY_\gamma=X_\gamma\times_\X\bY$; then $\bY=\ilim_{\gamma\in
\Gamma}\bnY_\gamma$ is a representation of $\bY$ by an inductive
system of closed immersions of schemes.
 Denote by $i_\gamma\:X_\gamma\rarrow\X$ and $k_\gamma\:\bnY_\gamma
\rarrow\bY$ the closed immersion morphisms, and by
$\pi_\gamma\:\bnY_\gamma\rarrow X_\gamma$ the natural flat affine
morphisms of schemes.
 By the definition of the functor $\pi_*\:\bY\tors\rarrow\X\tors$
(see Section~\ref{torsion-direct-images-subsecn}), we have a natural 
isomorphism $i_\gamma^!\pi_*\simeq\pi_\gamma{}_*k_\gamma^!$ of functors
$\bY\tors\rarrow X_\gamma\qcoh$ for every $\gamma\in\Gamma$.

 Assume that the complex $\pi_*\bJ^\bu$ is coacyclic in $\X\tors$.
 By Lemma~\ref{injective-over-base}(b), the functor~$\pi_*$ takes
injective objects to injective objects; so $\pi_*\bJ^\bu$ is a complex
of injectives in $\X\tors$.
 Hence the complex $\pi_*\bJ^\bu\in\sC(\X\tors_\inj)$ is contractible.
 It follows that the complex $i_\gamma^!\pi_*\bJ^\bu$ is a contractible
complex of injective quasi-coherent sheaves on~$X_\gamma$.
 Therefore, so is the complex~$\pi_\gamma{}_*k_\gamma^!\bJ^\bu\in
\sC(X_\gamma\qcoh_\inj)$; in particular, the complex
$\pi_\gamma{}_*k_\gamma^!\bJ^\bu$ is coacyclic.
 Following the discussion of $\bnY/X$\+semiacyclicity in the end of
Section~\ref{semiderived-nonaffine-scheme-morphism-subsecn}, this
means that the complex $k_\gamma^!\bJ^\bu\in\sC(\bnY_\gamma\qcoh)$ is
$\bnY_\gamma/X_\gamma$\+semiacyclic.
 By the definition, we can conclude that the complex $\bJ^\bu\in
\sC(\bY/\X\tors_\inj)$ is $\bY/\X$\+semiacyclic.

 Conversely, assume that the complex $\bJ^\bu$ is $\bY/\X$\+semiacyclic.
 Then the complex $\pi_\gamma{}_*k_\gamma^!\bJ^\bu$ is a (Becker)
coacyclic complex of injective objects in $X_\gamma\qcoh$, so it is
a contractible complex of injective objects.
 Hence the complex $i_\gamma^!\pi_*\bJ^\bu$ is contractible in
$X_\gamma\qcoh_\inj$.
 We also know from the previous paragraph that $\pi_*\bJ^\bu$ is
a complex of injective objects in $\X\tors$.
 Using Lemma~\ref{complex-of-injective-torsion-contractible}, we
conclude that $\pi_*\bJ^\bu$ is a contractible (hence coacyclic)
complex in $\X\tors$.
\end{proof}

\begin{rem}
 One can think of this appendix as pointing a direction for possible
generalization of the results of Sections~\ref{X-flat-on-Y-secn}\+-%
\ref{weakly-smooth-postcomposition-secn} to nonaffine morphisms of
ind-schemes $\pi\:\bY\rarrow\X$, but this a long way.
 To begin with, it is not obvious (and needs to be checked) that
the full triangulated subcategory of $\bY/\X$\+semiacyclic complexes
in $\sK(\bY\tors)$ (in the sense of the definition in this appendix)
is closed under coproducts.
 The difficulty arises from the fact that the full subcategory of
injective objects in $\bY\tors$ is not closed under coproducts.
 Perhaps more importantly, it is not clear how to extend
the constructions of resolutions in
Section~\ref{relatively-homotopy-flat-subsecn} to nonaffine
morphisms~$\pi$, or what to replace them with in the nonaffine case.
 So the semi-infinite algebraic geometry of quasi-coherent torsion
sheaves for nonaffine morphisms of ind-schemes $\pi\:\bY\rarrow\X$
remains a challenge.
\end{rem}

\end{document}